\documentclass[a4paper,oneside,12pt]{article}
\usepackage{subfiles}

\usepackage{a4wide}
\usepackage[french,english]{babel}
\usepackage[utf8]{inputenc}
\usepackage[T1]{fontenc}
\usepackage{amsthm}
\usepackage{amssymb}
\usepackage{amsmath}
\usepackage{stmaryrd}
\usepackage{dsfont}
\usepackage{graphicx}
\usepackage{caption}
\usepackage{imakeidx} 
\theoremstyle{plain}
\usepackage[colorlinks=true,breaklinks=true,linkcolor=black]{hyperref} 

\usepackage{authblk}
\usepackage{amsfonts}
\usepackage{indentfirst}
\usepackage{hyperref} 
\hypersetup{colorlinks=true,linkcolor=blue,citecolor=red,urlcolor=blue}
\usepackage{tikz}
\usepackage{mathrsfs}
\usepackage[all]{xy}
\usepackage{geometry}
\usepackage[left,modulo]{lineno}
\usepackage{enumitem}
\usepackage{comment}
\usepackage{float}
\makeindex[title=Index of Definitions]
\makeindex[name=notation, title=Index of Symbols]
\makeindex[name=axiom, title=Index of Axioms]
\newcommand{\axiom}[1]{\index[axiom]{#1}}

\swapnumbers 
\theoremstyle{plain} 
\newtheorem{theorem}{Theorem}[section] 
\newtheorem{Prop}[theorem]{Proposition} 
\newtheorem{Lem}[theorem]{Lemma} 
\newtheorem{Thm}[theorem]{Theorem} 
\newtheorem{Fact}[theorem]{Fact} 
\newtheorem{Cor}[theorem]{Corollary} 
\newtheorem{proposition}[theorem]{Proposition}
\newtheorem{corollary}[theorem]{Corollary}
\newtheorem{lemma}[theorem]{Lemma}
\newtheorem{assumption}[theorem]{Assumption}

\theoremstyle{definition}
\newtheorem{thm*}{Theorem}
\newtheorem{Def}[theorem]{Definition}
\newtheorem{definition}[theorem]{Definition}
\newtheorem{Def,Thm}[theorem]{Theorem-Definition} 
\newtheorem{Def,Prop}[theorem]{Proposition-Definition} 
\newtheorem{definition/proposition}[theorem]{Definition/Proposition}
\newtheorem{Not}[theorem]{Notation}
\newtheorem{Hyp}[theorem]{Assumption} 

\theoremstyle{remark} 
\newtheorem{example}[theorem]{Example}
\newtheorem{remark}[theorem]{Remark}
\newtheorem{Ex}[theorem]{Example} 
\newtheorem{Rq}[theorem]{Remark} 
\newtheorem{Rqs}[theorem]{Remarks} 

\makeatletter \@addtoreset{figure}{section}\makeatother

\newcommand{\R}{\mathbb{R}}
\newcommand{\A}{\mathbb{A}}

\newcommand{\N}{{\mathbb{Z}_{\geq 0}}}
\newcommand{\Z}{\mathbb{Z}}
\newcommand{\Q}{\mathbb{Q}}

\newcommand{\Ne}{\mathbb{N}^*}

\newcommand{\I}{\mathcal{I}}

\newcommand{\Id}{\mathrm{Id}}

\newcommand{\F}{\mathbb{F}}

\newcommand{\FCC}{\mathscr{F}}

\newcommand{\supp}{\mathrm{supp}}
\newcommand{\cl}{\mathrm{cl}}

\newcommand{\GB}{\mathbf{G}}
\newcommand{\Kb}{\mathbb{K}}

\newcommand{\EC}{\mathcal{E}}
\newcommand{\FC}{\mathcal{F}}

\newcommand{\LC}{\mathcal{L}}
\newcommand{\MC}{\mathcal{M}}
\newcommand{\OC}{\mathcal{O}}

\newcommand{\SC}{\mathcal{S}}

\newcommand{\VC}{\mathcal{V}}

\newcommand{\XC}{\mathcal{X}}

\newcommand{\red}[1]{{\color{red}#1}}
\newcommand{\blue}[1]{{\color{blue}#1}}

\newcommand{\RF}{\mathfrak{R}}

\newcommand{\ACC}{\mathscr{A}}
\newcommand{\PCC}{\mathscr{P}}

\newcommand{\Rtot}{R}
\newcommand{\germ}{\operatorname{germ}}

\newcommand{\NGL}{N_{\mathrm{GL}_{\ell+1}}}
\newcommand{\NS}{N_{\mathrm{SL}_{\ell+1}}}
\newcommand{\TG}{T_{\mathrm{GL}_{\ell+1}}}
\newcommand{\TS}{T_{\mathrm{SL}_{\ell+1}}}

\newcommand{\Waff}{W^{\mathrm{aff}}}
\newcommand{\Wext}{\widetilde{W}}

\hypersetup{
pdfauthor={Hebert, Izquierdo, Loisel},
pdftitle={Lambda-buildings associated to quasi-split reductive groups over Lambda-valued fields},
}

\date{\today}

\title{$\Lambda$-buildings associated to quasi-split groups over $\Lambda$-valued fields}
\author[1]{Auguste \textsc{H{\'e}bert}}

\author[2]{Diego \textsc{Izquierdo} }

\author[3]{Benoit \textsc{Loisel}}

\affil[1]{auguste.hebert@univ-lorraine.fr}

\affil[2]{diego.izquierdo@polytechnique.edu}

\affil[3]{benoit.loisel@math.univ-poitiers.fr}

\begin{document}

\maketitle

\begin{abstract}
Let $\mathbf{G}$ be a quasi-split reductive group and $\mathbb{K}$ be a Henselian field equipped with a valuation $\omega:\mathbb{K}^{\times}\rightarrow \Lambda$, where $\Lambda$ is a non-zero totally ordered abelian group. In 1972 and 1984, Bruhat and Tits constructed a building on which the group $\mathbf{G}(\mathbb{K})$ acts provided that $\Lambda$ is a subgroup of $\mathbb{R}$. In this paper, we deal with the general case where there are no assumptions on $\Lambda$ and we construct a set on which $\mathbf{G}(\mathbb{K})$ acts. We then prove that it is a $\Lambda$-building, in the sense of Bennett. 
\end{abstract}

\tableofcontents

\section{Introduction}

Reductive groups over non-Archimedean local fields have been extensively  studied for the past sixty years. To study such a group $\mathbf{G}$ and the group of its rational points $G$, Bruhat and Tits associated in \cite{ BruhatTits1} and \cite{ BruhatTits2} a space $\I$ - called a \textbf{Bruhat-Tits building} - on which $G$ acts. The space $\I$ encapsulates significant information about the group $G$.

In the 1970's, Kato (\cite{kato}) and Parshin (\cite{parshinhigher}) introduced \textbf{higher-dimensional local fields}, a natural generalization of the usual local fields. A $0$-local field is by definition a finite field, and for $d>1$, a $d$-dimensional field is a complete discrete valuation field whose residue field is a $(d-1)$-local field. For instance, $1$-local fields coincide with usual non-Archimedean local fields. The equicharacteristic $2$-local fields are the fields of the form $k(\!(t)\!)$ with $k$ a $1$-local field, but there are other $2$-local fields that have mixed characteristic.

Higher-dimensional local fields play an important role both in algebraic geometry and in number theory. On the one hand, just as $p$-adic fields encode local information on arithmetic schemes with relative dimension $0$ such as $\mathrm{Spec}(\mathbb{Z})$, $2$-local fields encode local information on arithmetic curves such as $\mathbb{A}^1_{\mathbb{Z}}$, and $d$-local fields encode local information on arithmetic schemes with relative dimension $d-1$ over $\mathbb{Z}$. On the other hand, by Kato's work, higher-dimensional local fields also provide the good framework to generalize local class field theory, which is a crucial step in the understanding of $p$-adic fields. Taking into account that class field theory is the most basic example of Langlands correspondence, one may ask whether Langlands' ideas can be studied in a higher-dimensional setting, and then it seems natural to study reductive groups over higher-dimensional local fields. 

 In this article, we will work over a field $\mathbb{K}$ that is endowed with a valuation $\omega: \mathbb{K}^{\times} \rightarrow \Lambda$, where $\Lambda$ is any non-zero totally ordered abelian group. This covers the case of $d$-local fields for $d> 0$ since they are endowed with a $\mathbb{Z}^d$-valuation, where $\Z^d$ is equipped with the lexicographical order. It also allows us to work with a field of the form $\mathbb{C}(\!(t_1)\!)\cdots(\!(t_{\ell+1})\!)$, which is the geometric counterpart of higher-dimensional local fields, since it encodes local information in higher-dimensional complex varieties. Note that $\Lambda$ does not need to be discrete and could for instance be the group $\mathbb{R} \times \mathbb{R}$ with the lexicographical order. It could also have infinite rank.

In the case where $\mathbb{K}$ is a higher-dimensional local field, Parshin constructed in \cite{parshin1994higher},  \cite{parshin2000invitation}  a  ``higher Bruhat-Tits building'' on which $\mathrm{PGL}_n(\Kb)$ acts, for $n\in \Ne$. 
  Independently of this work, Bennett defined in \cite{bennett1994affine} a notion of $\Lambda$-building for $\Lambda$ a non-zero totally ordered abelian group and, for any field $\mathbb{K}$ equipped with a valuation $\omega:\mathbb{K}\rightarrow \Lambda \cup\{+\infty\}$, he constructed such a $\Lambda$-building on which $\mathrm{SL}_n(\Kb)$ acts. As sets (when we forget the topological structures), the buildings of Bennett and Parshin are very close.

Given any non-zero totally ordered abelian group $\Lambda$, any $\Lambda$-valued field $\mathbb{K}$ and any quasi-split reductive group $\mathbf{G}$ over $\mathbb{K}$, the goal of this article is to construct a $\Lambda$-building $\I(\Kb,\omega,\GB)$ endowed with a suitable $\GB(\Kb)$-action.
This partially answers \cite[Problem 2 p 187]{parshin1994higher}.
 
\subsection{Bennett's \texorpdfstring{$\Lambda$-buildings}{Lambda}} 

When $\Kb$ is a field equipped with a nontrivial valuation $\omega:\Kb\rightarrow \R$ and $\mathbf{G}$ is a reductive $\mathbb{K}$-group, Bruhat and Tits associated to $\GB(\Kb)$  its  \textbf{Bruhat-Tits building} $\I(\Kb,\omega,\GB)$ on which $\GB(\Kb)$ acts. When $\GB=\mathrm{SL}_2$ and $\omega$ is discrete for example, this space is a simplicial tree.

The Bruhat-Tits building $\I(\Kb,\omega,\GB)$ is covered by subsets called \textbf{apartments}, which are Euclidean spaces equipped with an arrangement of hyperplanes called \textbf{walls}.  These apartments are all obtained by translation by an element of $\GB(\Kb)$ from a standard apartment $\A(\mathbf{G},\Kb,\omega)$.
The hyperplane arrangement of $\A(\mathbf{G},\Kb,\omega)$ depends on the root system of $\GB$ and on the set of values of $\omega$. It naturally defines the notion of a face on $\A(\mathbf{G},\Kb,\omega)$ and by translation, on $\I(\Kb,\omega,\GB)$. Then we have the following important properties:
\begin{itemize}
\item (I1) for every two faces, there exists an apartment containing them;

\item (I2) for every two apartments $A$ and $B$, there exists an isomorphism of affine spaces from $A$ to $B$ fixing $A\cap B$ and preserving the hyperplane arrangement.
\end{itemize}

These properties motivate the definition of an \textbf{abstract building}: as a first approximation (see Definition~\ref{defBuildings}  for a precise definition), it is a set covered with subspaces called apartments, which satisfy (I1) and (I2) and which are isomorphic to a standard apartment $\A$ depending on a root system.

Let now $\Lambda$ be any non-zero totally ordered abelian group. Since every totally ordered abelian group can be embedded in an ordered real vector space, let's assume for simplicity that $\Lambda$ itself is a real vector space. 
 Bennett then defined in \cite{bennett1994affine} the notion of a \textbf{$\Lambda$-building}: it is a set covered by subsets called apartments, all isomorphic to a standard apartment $\A\otimes_\R\Lambda$, and satisfying axioms similar to those of a Bruhat-Tits building.

\paragraph{Examples of $\Lambda$-buildings} When $\Lambda\subset\R$ and $\Kb$ is a field equipped with a valuation $\omega: \mathbb{K} \rightarrow \Lambda \cup \{\infty\}$, the Bruhat-Tits building of $(\GB,\mathbb{K},\omega)$ is a $\Lambda$-building. In the case that $\Lambda$ cannot be embedded as an ordered abelian group in $\R$, there are four main previously known classes of examples of $\Lambda$-buildings: \begin{enumerate}

\item Let $\Kb$ be a field that is equipped with a valuation $\omega:\Kb\rightarrow \Lambda \cup\{\infty\}$.  In \cite{morgan1984valuations}, inspired by Serre's construction of the tree of $\mathrm{SL}_2(\Kb)$ when $\Lambda=\Z$, Morgan and Shalen define the notion of a $\Lambda$-tree and construct a $\Lambda$-tree on which $\mathrm{SL}_2(\Kb)$ acts. Generalizing these works, Bennett defines in \cite[Example 3.2]{bennett1994affine} a $\Lambda$-building on which $\mathrm{SL}_n(\Kb)$ acts, for $n\in \Z_{\geq 2}$. Independently, in \cite{parshin1994higher} and \cite{parshin2000invitation}, Parshin also introduces some higher buildings associated to $\mathrm{SL}_n(\Kb)$ but does not relate them to the structure of $\Lambda$-buildings.

\item Let $\Lambda,\Lambda'$ be non-zero totally ordered abelian groups and $e:\Lambda\rightarrow \Lambda'$ be a morphism of ordered groups.  Then Schwer and Struyve   construct a functor from the category of $\Lambda$-buildings to the category of $\Lambda'$-buildings, compatible  with $e$ (see \cite{schwer2012lambda}). Using this and using ultraproducts, they construct nontrivial examples of $\Lambda$-buildings, for $\Lambda\nsubseteq \R$. They in particular construct ultracones and asymptotic cones of buildings (see \cite[Section 6]{schwer2012lambda}).

\item Let $\mathcal{G}$ be a semi-simple Lie group. Let $\mathbb{K}_{\mathrm{real}}$ be a real closed nonarchimedean field and $\mathbb{O}_{\mathrm{real}}\subset \mathbb{K}_{\mathrm{real}}$ be an o-convex valuation ring. Then Kramer and Tent equip $\mathcal{G}(\Kb_{\mathrm{real}})/\mathcal{G}(\mathbb{O}_{\mathrm{real}})$ with the structure of a $\Lambda_{\mathrm{real}}$-building, where $\Lambda_{\mathrm{real}}=\Kb_{\mathrm{real}}^*/\mathbb{O}_{\mathrm{real}}^*$, see \cite[Theorem 4.3]{kramer2004asymptotic} and \cite{kramer2009ultrapowers}. They deduce that the asymptotic cone of $\mathcal{G}(\Kb_{\mathrm{real}})$ and the ultracone of $\mathcal{G}(\Kb_{\mathrm{real}})$ are $\Lambda$-buildings, for some $\Lambda$ (see \cite[Corollary 4.4]{kramer2004asymptotic}). Using these results, they give a new proof of Margulis's conjecture (see \cite[§ 5]{kramer2004asymptotic}).

\item Let $\Kb$ be a field equipped with a surjective valuation $\omega:\Kb\rightarrow \Lambda'\cup\{\infty\}$, where $\Lambda'$ is an infinite  subgroup of $\R$. Let $G^\circ$ be a semi-simple group over $\Kb$ and $G=G^\circ\left(\Kb[t,t^{-1}])\right)$ be the untwisted affine Kac-Moody group associated.  In \cite[4.1]{rousseau2006immeubles}, Rousseau associates  a $(\Z\times \Lambda')$-building to $G$ on which it acts.
\end{enumerate}

Our result yields a new class of examples of $\Lambda$-buildings.

\subsection{The building \texorpdfstring{$\I(\Kb,\omega,\GB)$}{I(K,w,G)}}

In differential geometry, when one works with some real Lie group $G$, it is often useful to study the action of $G$ on a symmetric space.
For instance, if $G = \mathrm{SL}_n(\mathbb{R})$, then one may consider the action of $\mathrm{SL}_n(\mathbb{R})$ on $X = \mathrm{SL}_n(\mathbb{R}) / \mathrm{SO}_n(\mathbb{R})$. The manifold $X$ is then the quotient of a real Lie group by a maximal compact subgroup.
According to Goldman-Iwahori (\cite{GoldmanIwahori}), the space $X$ can be identified with the space of norms on $\mathbb{R}^n$ up to homothety.

 Now, when one works over the $p$-adic field $\Q_p$ instead of $\mathbb{R}$, by analogy with Goldman-Iwahori's result, one can associate to the group  $\mathrm{SL}_n(\Q_p)$ the space $\mathcal{I}(\mathrm{SL}_n,\Q_p)$ of ultrametric norms over $\Q^n_p$ up to homothety.
This is a particular case of a more general construction given by Bruhat-Tits in \cite[10.2]{BruhatTits1}. In order to generalize the previous constructions and study in this context other classical groups of Lie type over Henselian valued fields, Tits introduced the definition of buildings in the 1960's.

When $\mathbb{K}$ is a field endowed with a valuation $\omega:  \mathbb{K}^{\times} \rightarrow \mathbb{Z}$, several approaches have been developed to construct a Bruhat-Tits building $\mathcal{I}(\mathbb{K},\omega,\mathbf{G})$ associated to a reductive $\mathbb{K}$-group $\mathbf{G}$.
The most elementary construction relies on lattices.
For instance, when $\mathbf{G}$ is split and has type $A_n$ (e.g. $\mathbf{G} = \mathrm{GL}_n$, $\mathrm{SL}_n$ or $\mathrm{PGL}_n$), one may define $\I(\Kb,\omega,\GB)$ as the set of $\mathbb{O}$-lattices contained in $\Kb^n$ up to homothety, where $\mathbb{O}$ stands for the ring of integers of $\Kb$.
The action of $\GB(\Kb)$ on $\I(\Kb,\omega,\GB)$ is then induced by the action of $\GB(\Kb)$ on the $\mathbb{O}$-lattices of $\Kb^n$ (see \cite[Chapter II]{serre1997arbres} for the case where $n=2$).
Note that this construction depends substantially on the Lie type of the group $\mathbf{G}$ and a case by case definition has to be settled.\footnote{
Note that in the remarkable case of type $A_n$, there are other similar approaches such as using maximal orders (see \cite{Vigneras-Quaternions} and \cite{serre1997arbres} for $n=2$).}

Parshin \cite{parshin1994higher} and Bennett \cite{bennett1994affine} both use this approach with lattices to construct a space analogous to $\mathcal{I}(\mathbb{K},\omega,\mathrm{SL}_n)$ when the valuation $\omega$ takes values in a totally ordered abelian group that might be different from $\mathbb{Z}$.

A more general approach due to Bruhat and Tits (\cite{ BruhatTits1}, \cite{ BruhatTits2}) mainly consists in generalizing the construction as a quotient of $\mathbf{G}(\mathbb{K})$ by a maximal compact subgroup.
Unfortunately, maximal compact subgroups need not be pairwise conjugate in general.
That is why, for an arbitrary reductive group $\GB$ over a $\Lambda$-valued field with $\Lambda\subset \R$, Bruhat and Tits' construction relies on the use of parahoric subgroups. More precisely, they start by defining the standard apartment $\A$ of their building as an affine space endowed with an action by affine transformations of the group $N$ of rational points of the normalizer $\mathbf{N}$ of a maximal $\mathbb{K}$-split torus of $\mathbf{G}$.
The root system $\Phi$ of $\mathbf{G}$ can be regarded as a set of affine maps on $\A$.
Then, for each element $x$ of $\A$, Bruhat and Tits define a parahoric subgroup $P_x\subset G$, which depends on the values $\alpha(x)$, for $\alpha\in \Phi$. They finally define $\I=\I(\Kb,\omega,\GB)$ as the set $\mathbf{G}(\mathbb{K}) \times \A /\sim$, where $\sim$ is an equivalence relation on $\mathbf{G}(\mathbb{K}) \times \A$ whose definition involves the parahoric subgroups $P_x$ for $x\in \A$.
The group $\mathbf{G}(\mathbb{K})$ acts on $\I$ by $g \cdot [g',x]=[gg',x]$, for $g,g'\in \mathbf{G}(\mathbb{K})$ and $x\in \A$, so that $P_x$ is the stabilizer of $x$ in $G$ for each $x$.
This is the approach we follow in this paper to deal with the case when $\Lambda$ is not necessarily contained in $\mathbb{R}$.

For that purpose, when $\mathbf{G}$ is assumed to be a quasi-split group, we can consider a Chevalley-Steinberg system (i.e. a parametrization of the root groups $\mathbf{U}_\alpha$ of $\mathbf{G}$ taking into account the Galois extension $\widetilde{\mathbb{K}} / \mathbb{K}$ that splits $\mathbf{G}$).
To such a pinning, we associate a space $\I$. We then need to prove that it is a $\Lambda$-building. A part of the proof consists in proving that $G$ satisfies certain decompositions, namely the Iwasawa decomposition and the Bruhat decomposition. To prove them, we generalize the proof by Bruhat and Tits to our framework. After proving these decompositions, the main issue is to prove that certain retractions are $1$-Lipschitz continuous. When $\Lambda\subset \R$, a standard proof of this property uses the fact that the line segments of $\R$ are compacts. This is no longer true in our framework and we thus need some additional work to prove this property.

\subsection{Affine structure of the standard apartment}

An important step in our construction consists in understanding the geometry of the apartments of our building. Roughly speaking, the apartment will be a tensor product $Y \otimes_\mathbb{Z} \Rtot$ where $Y$ is the finitely generated free $\mathbb{Z}$-module of cocharacters of a maximal torus of $\mathbf{G}$ and $\Rtot$ is some ordered ring that contains the valuation groups of all finite extensions of $\mathbb{K}$. For instance, in the classical case of a Henselian valued field $\mathbb{K}$ with a discrete valuation $\omega: \mathbb{K}^* \to \Lambda = \mathbb{Z}$, one usually takes  $\Rtot=\mathbb{R}$.

Now, assume that $\Lambda = \mathbb{Z}^d$ equipped with the lexicographical order and let $Y = \mathbb{Z}^n = \bigoplus_{i=1}^n \mathbb{Z} e_i$.
A natural way to proceed in this case is to consider the ring $\Rtot = \mathbb{R}[t] / (t^d)$ equipped with the lexicographical order induced by the degree of monomials.
On this example, the apartment $\A = Y \otimes_{\mathbb{Z}} \Rtot$ will be a $d n$-dimensional $\mathbb{R}$-vector space and a free $\Rtot$-module of rank $n$.
Thus, there are two viewpoints for an element $x \in \A$:
one can write either $x = \sum_{i=1}^n \lambda_i e_i$ with $\lambda_i \in \Rtot$ (structure of $\Rtot$-module) or $x= \sum_{s=0}^{d-1} x_s t^s$ with $x_s \in Y \otimes \mathbb{R}$. The first viewpoint allows to endow the apartment $\A$ with a geometric and combinatorial structure (it is an $R$-affine space together with combinatorial data such as chambers, faces, sectors...). The second viewpoint allows to endow the apartment with a fibration $\A \rightarrow Y \otimes_{\mathbb{R}} \mathbb{R}[t]/(t^{d'})$ that we will extend to the whole building for each $d' \leq d$.

In the situation of a general non-zero totally ordered abelian group $\Lambda$, there seems not to be a natural way to endow $\Rtot$ with a totally ordered ring structure. The apartment $\A$ will then be defined by replacing $\Rtot$ by a totally ordered real vector space $\RF^S$ together with an increasing embedding $\Lambda \to \RF^S$.

\subsection{Main results}

We now briefly describe the main results of this paper, see Theorem~\ref{thmMain} for a more precise statement.

Let $\Lambda$ be a non-zero totally ordered abelian group. Define the rank $S:=\mathrm{rk}(\Lambda)$ of $\Lambda$ as the (totally ordered) set of Archimedean equivalence classes of $\Lambda$. For instance, if $\Lambda = \mathbb{Z}^n$ for some $n\geq 1$, then $S=\{1,2,...,n\}$. By Hahn's embedding theorem, $\Lambda$ can then be regarded as a subgroup of: $$\RF^S:= \{(x_s)_{s\in S} \in \mathbb{R}^S \, |\, \text{the support of $(x_s)_{s\in S}$ is a well-ordered subset of $S$}\}.$$ Let now $\mathbb{K}$ be a field with a valuation $\omega: \mathbb{K} \rightarrow \Lambda \cup \{\infty\}$, fix a quasi-split (connected) reductive $\mathbb{K}$-group $\mathbf{G}$ and let $\mathbf{S}$ be a maximal split torus in $\mathbf{G}$ with cocharacter module $X_*(\mathbf{S})$. If $\mathbf{G}$ is not split, assume that $\mathbb{K}$ is Henselian. Our results can then be summarized as follows:

\begin{itemize}
\item[(i)] We construct a set $\I=\I(\Kb,\omega,\GB)$, a  $\Lambda$-distance $d:\I\times\I\rightarrow \Lambda_{\geq 0}:= \{\lambda \in \Lambda | \lambda \geq 0\}$ and we equip $(\I,d)$ with the structure of an $\RF^S$-building whose apartments are modelled on some quotient of $X_*(\mathbf{S})\otimes \RF^S$. The group $G=\mathbf{G}(\mathbb{K})$ acts isometrically on $\I$ and the induced action on the set of apartments is transitive.

\item[(ii)] Let $s \in S$, let $S_{\leq s} = \{ t\in S | t\leq s\}$ and let $\pi_{\mathfrak{R}^S,\leq s}: \mathfrak{R}^S \rightarrow \mathfrak{R}^{S_{\leq s}}$ be the natural projection. Consider the valuation $\omega_{\leq s} = \pi_{\RF^S,\leq s} \circ \omega$. We construct  an (explicit) surjective map:
$$\pi_{\leq s}: \mathcal{I}(\mathbb{K},\omega,\mathbf{G}) \rightarrow \mathcal{I}(\mathbb{K},\omega_{\leq s},\mathbf{G})$$
compatible with the $\mathbf{G}(\mathbb{K})$-action such that, for each $X \in \mathcal{I}(\mathbb{K},\omega_{\leq s},\mathbf{G})$, the fiber $\pi_{\leq s}^{-1}(X)$ is a product:
$$\mathcal{I}_X \times \tilde{V},$$
where $\tilde{V}$ is a $\ker(\pi_{\mathfrak{R}^S,\leq s})$-module and $\mathcal{I}_X$ is a
$\ker(\pi_{\mathfrak{R}^S,\leq s})$-building. 
\end{itemize}

In the particular case when $\mathbf{G}$ splits over $\mathbb{K}$, one can even replace $\mathfrak{R}^S$ by $\Lambda$ and construct a $\Lambda$-building associated to $\mathbf{G}$ (see section \ref{section_replacingRS}).

\subsection{Structure of the paper}

 In section~\ref{secLattice_building}, we describe the buliding of $\mathrm{SL}_{\ell+1}(\mathbb{K})$, where $\ell\in \Z_{\geq 1}$ and where  $\mathbb{K}$ is a field equipped with a valuation. The aim of this elementary section is to give an intuition on the building associated to a more general reductive group.  We use the lattice approach. 

In section~\ref{secAbstract_definition_buildings}, we provide all useful definitions concerning the construction of $\Lambda$-buildings.
In subsection~\ref{secDefinitionApartment}, we introduce all preliminary definitions that are necessary to define $\Lambda$-buildings.
These definitions are used all along the article.
In subsection~\ref{subBennett_definition}, we provide the definition of $\Lambda$-buildings themselves.
This allows us to state the main Theorem in subsection~\ref{subMain_theorem}.

In sections~\ref{SecVRGD} and~\ref{SecParahoricBruhat}, we follow the strategy of Bruhat and Tits in order to provide a generalization of the Iwasawa decomposition and the Bruhat decomposition.
Doing this, we introduce some abstract subgroups that generalize parahoric subgroups and get a better understanding of the action of $G$ on the $\Lambda$-building.

In section~\ref{sectionLambdaBuildingFromVRGD}, we glue up the apartments via an equivalence relation similar to the one introduced by Bruhat and Tits in order to provide a space $\mathcal{I}(\mathbb{K},\omega,\mathbf{G})$ which will be the $\Lambda$-building associated to $\mathbf{G}$. We then prove that it satisfies all the $\Lambda$-building axioms except the axiom~\ref{axiomCO}.

In section~\ref{SecQuasiSplitGroups}, we use the classical reductive group theory over a field (Chevalley-Steinberg systems, Borel-Tits theory) in order to construct data that satisfy the axioms of sections~\ref{SecVRGD} and~\ref{SecParahoricBruhat}.

In section~\ref{projectionsec}, inspired by work of Parshin (\cite{parshin1994higher}, \cite{parshin2000invitation}), we prove that a surjective morphism of totally ordered abelian groups $f:\Lambda \rightarrow \Lambda'$ induces a projection map from $\mathcal{I} (\mathbb{K},\omega,\mathbf{G}) \rightarrow \mathcal{I}(\mathbb{K},f\circ\omega,\mathbf{G})$ which is surjective and compatible with the action of $G$. We then give a detailed description of the fibers of this projection. 

Section~\ref{sectionCO} is dedicated to the proof of axiom~\ref{axiomCO} and completes the proof of the main Theorem.

In section~\ref{section_replacingRS}, we see how one can replace the totally ordered abelian group $\mathfrak{R}^S$ by a smaller ordered abelian group $\Gamma$ and get a $\Gamma$-building instead of an $\mathfrak{R}^S$-building. In particular, we prove that one can always associate a $\Lambda$-building with a \emph{split} reductive group over a $\Lambda$-valued field.

In section~\ref{secBuilding_sld}, we relate the building of $\mathrm{SL}_{\ell+1}$ constructed by using lattices to the building that we construct by using parahoric subgroups.

\paragraph{Acknowledgements} The first author would like to thank Petra Schwer for useful discussions on $\Lambda$-buildings.  The authors would also like to thank the anonymous referee whose careful comments were very useful to improve this text.

\paragraph{Funding} : The first and the third author were partially supported by the ANR grant ANR-15-CE40-0012.  (The French National Research Agency).

\section{The group \texorpdfstring{$\mathrm{SL}_{\ell +1}$}{SL(l+1)}}\label{secLattice_building}

In this section, we briefly explain how to construct the building associated to the special linear group over a valued field by using lattices and we describe its structure. We invite the reader to keep this particular case in mind, since it will be helpful to get a better understanding of the rest of the article, which is dedicated to the more general case of quasi-split reductive groups.

Let $\mathbb{K}$ be a field equipped with a surjective valuation $\omega:\mathbb{K}\rightarrow \Lambda\cup\{+\infty\}$, where $\Lambda$ is a totally ordered abelian group. Let $\mathbb{O} :=\omega^{-1}(\Lambda_{\geq 0})$ be the ring of integers in $\mathbb{K}$.

Given a positive integer $\ell$, fix an $\ell+1$-dimensional $\mathbb{K}$-vector space $V$. An \textbf{$\mathbb{O} $-lattice} in $V$ is an $\mathbb{O} $-submodule of $V$ of the form $\mathbb{O} b_0 \oplus ... \oplus \mathbb{O} b_{\ell}$ for some $K$-basis $(b_0,...,b_{\ell})$ of $V$. If $L_1$ and $L_2$ are two $\mathbb{O} $-lattices in $V$, we say that they are \textbf{homothetic} if there exists $a\in \mathbb{K}^{\times}$ such that $L_2=aL_1$. In that case, we denote $L_1 \sim L_2$. Let $\mathcal{I}^{\mathcal{L}}(\mathrm{SL}(V),\omega)$ be the set of $\mathbb{O} $-lattices in $V$ modulo the homothety relation. We say that $\mathcal{I}^{\mathcal{L}}(\mathrm{SL}(V),\omega)$ is the \textbf{lattice $\Lambda$-building} of $(\mathrm{SL}(V),\omega)$. The class of an $\mathbb{O} $-lattice $V$ in $\mathcal{I}^{\mathcal{L}}(\mathrm{SL}(V),\omega)$ is denoted $[L]$.

 Given two $\mathbb{O} $-lattices $L_1$ and $L_2$ such that $L_2\subset L_1$, we can always find $a_0,\dots,a_\ell\in \mathbb{O}$ such that $L_1/L_2 \simeq \mathbb{O} /a_0\mathbb{O}  \oplus ... \oplus \mathbb{O} /a_{\ell}\mathbb{O} $. We write $L_2\leq L_1$  when at least one of the $a_i$'s is a unit in $\mathbb{O} $. In other words, $L_2\leq L_1$ if, and only if, $L_2\subset L_1$ and $L_1/L_2 \simeq \mathbb{O} /a_1\mathbb{O}  \oplus ... \oplus \mathbb{O} /a_{\ell}\mathbb{O} $, for some $a_1,...,a_{\ell}\in \mathbb{O} $. In that case, the $\ell$-tuple $(\omega(a_1),...,\omega(a_{\ell}))$ is determined up to permutation, thus $d(L_1,L_2):=\max \{\omega(a_1),...,\omega(a_{\ell})\} \in \Lambda$ is well-defined. Since $a_1,...,a_{\ell}$ are in $\mathbb{O} $, we have $d(L_1,L_2)\geq 0$. 
 
Now, given $x_1,x_2\in \mathcal{I}^{\mathcal{L}}(\mathrm{SL}(V),\omega)$, one can easily check that there exist  $L_1\in x_1$ and $L_2\in x_2$ such that $L_2\leq L_1$. Then $d_{\mathbb{O} }(x_1,x_2):=d(L_1,L_2)\in \Lambda_{\geq 0}$ does not depend on the choices of $L_1$ and $L_2$, and the map $d:\mathcal{I}^{\mathcal{L}}(\mathrm{SL}(V),\omega)\rightarrow \Lambda$ is a distance on $\mathcal{I}^{\mathcal{L}}(\mathrm{SL}(V),\omega)$.

\subsection{The projection \texorpdfstring{$\pi$}{pi}}

Fix now a convex subgroup $\Lambda_0$ of $\Lambda$, and set $\Lambda_1 := \Lambda/\Lambda_0$. The group $\Lambda_1$ is naturally endowed with the structure of a totally ordered abelian group.

Let $\omega_1:\mathbb{K}\rightarrow \Lambda_1\cup\{\infty\}$ be the composite of the valuation $\omega$ followed by the projection $\Lambda \rightarrow \Lambda_1$, and set:
\begin{gather*}
\mathbb{M} := \omega^{-1}(\Lambda_{> 0}),\\
\mathcal{O} := \omega_1^{-1}((\Lambda_1)_{\geq 0}),\;\;\;\;\;
\mathcal{M} := \omega_1^{-1}((\Lambda_1)_{> 0}),\\
\mathcal{K}_1:= \mathcal{O}/\mathcal{M},\;\;\;\;\;
\mathcal{O}_1 := \mathbb{O} /\mathcal{M},\;\;\;\;\;
\mathcal{M}_1 := \mathbb{M}/\mathcal{M},\\
\kappa := \mathbb{O} /\mathbb{M} = \mathcal{O}_1/\mathcal{M}_1.
\end{gather*}

\begin{example}\label{exCas_de_Fq((u))((t))}
Suppose that $\mathbb{K}=\mathbb{F}_q(\!(u)\!)(\!(t)\!)$ is  equipped with $v:\mathbb{K}\rightarrow \Z^2$ defined by $v(t^au^b)=(a,b)$ for all $a,b\in \Z$. Let $\Lambda_0=\{0\}\times \Z$. Then $v_1:\mathbb{K}\rightarrow \Z$ maps $t^au^b$ to $a$ for all $a,b\in \Z$. One has $\mathbb{O} =\F_q[\![u]\!] \oplus t\F_q(\!(u)\!)[\![t]\!]$, $\mathbb{M}=u\F_q[\![u]\!] \oplus t\F_q(\!(u)\!)[\![t]\!] $, $\mathcal{O}=\F_q(\!(u)\!)[\![t]\!]$, $\mathcal{M}=t\F_q(\!(u)\!)[\![t]\!]$. 
\end{example}

\begin{lemma}\label{quotann}
The ring $\mathcal{O}_1$ is a valuation ring in $\mathcal{K}_1$ with valuation group:
$$\mathcal{K}_1^{\times}/\mathcal{O}_1^{\times} \cong \Lambda_0.$$
\end{lemma}

\begin{proof}
Let $x \in \mathcal{K}_1^{\times}$ and consider two liftings $\tilde{x}$ and $\tilde{x}'$ in $ \mathcal{O}$ of $x$. Observe that $\tilde{x}$ and $\tilde{x}'$ lie in $\mathcal{O}^{\times}$, so that $\omega(\tilde{x})$ and $\omega(\tilde{x}')$ are both in $\Lambda_0$. Moreover, since $\omega_1(\tilde{x}'-\tilde{x})>0$ and $\omega_1(\tilde{x})=0$, we have $\omega(\tilde{x}'-\tilde{x})>\omega(\tilde{x})$ and $\omega(\tilde{x})=\omega(\tilde{x}')$.  We can therefore define a map $\omega_0: \mathcal{K}_1^{\times} \rightarrow \Lambda_0$ by $\omega_0(x):=\omega(\tilde{x}) \in \Lambda_0$. It is then easy to check that $\omega_0$ is a surjective valuation on $\mathcal{K}_1$ and that the associated valuation ring is $\mathcal{O}_1$.
\end{proof}

Recall that $\mathcal{I}^{\mathcal{L}}(\mathrm{SL}(V),\omega)$ (resp. $\mathcal{I}^{\mathcal{L}}(\mathrm{SL}(V),\omega_1)$) stands for the set of $\mathbb{O} $-lattices of $V$ (resp. the set of $\OC $-lattices of $V$) up to homothety, and consider the map $\pi:\mathcal{I}^{\mathcal{L}}(\mathrm{SL}(V),\omega)\rightarrow \mathcal{I}^{\mathcal{L}}(\mathrm{SL}(V),\omega_1)$ defined by $\pi([L])=[\OC .L]$.  

\begin{theorem}\label{sll}
\begin{itemize}
\item[(i)] The map $\pi$ is surjective and $\mathrm{SL}(V)$-equivariant. 
\item[(ii)] For any $\mathcal{O}$-lattice $L$ of $V$, the stabilizer of the fiber $\pi^{-1}([L])$ is $\mathrm{SL}(L)$. Moreover, the action of $\mathrm{SL}(L)$ on $\pi^{-1}([L])$ factors through $\mathrm{SL}(L/\mathcal{M}L)$.
\item[(iii)] Let $\omega_0$ be the valuation of $\mathcal{K}_1$ given by lemma \ref{quotann}. For any $\mathcal{O}$-lattice $L$ of $V$, there is an $\mathrm{SL}(L/\mathcal{M}L)$-equivariant bijection $\mathrm{Res}_L$ between $\pi^{-1}([L]) $ and the lattice $\Lambda_0$-building $\mathcal{I}^{\mathcal{L}}(\mathrm{SL}(L/\mathcal{M}L),\omega_0)$ of $(\mathrm{SL}(L/\mathcal{M}L),\omega_0)$.
\end{itemize}
\end{theorem}

\begin{proof}
\begin{itemize}
\item[(i)] This is straightforward.

\item[(ii)] Since $\pi$ is $\mathrm{SL}(V)$-equivariant, the stabilizer of $\pi^{-1}([L])$ coincides with the fixator $\mathrm{Fix}([L])$ of $[L] \in \mathcal{I}^{\mathcal{L}}(\mathrm{SL}(V),\omega_1)$. The inclusion $\mathrm{SL}(L) \subseteq \mathrm{Fix}([L])$ is obvious. Conversely, if $g \in \mathrm{Fix}([L])$, then $gL=aL$ for some $a \in \mathbb{K}^{\times}$. Hence $a^{-1}g \in \mathrm{GL}(L)$, and $\det (a^{-1}g) = a^{-\ell-1}\in \mathcal{O}^{\times}$. This implies that $a\in \mathcal{O}^{\times}$ and that $g\in \mathrm{SL}(L)$. The stabilizer of $\pi^{-1}([L])$ is therefore $\mathrm{SL}(L)$. 

Now let $g  \in \ker \big(\mathrm{SL}(L) \rightarrow \mathrm{SL}(L/\mathcal{M}L)\big)$. We can write $g = \Id_L + u$ for some $u \in \mathrm{End}_{\mathcal{O}}(L,\mathcal{M}L)$ (i.e. $u(L)\subset \mathcal{M}L$).  Let $\tilde{L}$ be an $\OC $-lattice such that $\pi([\tilde{L}]) = [L]$.  Hence $g(\tilde{L}) \subseteq \tilde{L} + \mathcal{M}\tilde{L} = \tilde{L}$. Since $g$ has determinant 1, we deduce that $g(\tilde{L})=\tilde{L}$. Hence the action of $\mathrm{SL}(L)$ on $\pi^{-1}([L])$ factors through $\mathrm{SL}(L/\mathcal{M}L)$.
\item[(iii)] Let $\MC_{\mathcal{O}_1,L}$ be the set of $\mathcal{O}_1$-submodules of the $\mathcal{K}_1$-vector space $L/\mathcal{M}L$ quotiented by the homothety relation. We are first going to define a residue map:
$$\mathrm{Res}_L: \pi^{-1}([L]) \rightarrow \MC_{\mathcal{O}_1,L}.$$
To do so, let $\tilde{L}$ be any $\mathbb{O} $-lattice in $V$ such that $[\tilde{L}] \in \pi^{-1}([L])$. Since $[\mathcal{O}\tilde{L}] = [L]$, we can find $a \in \mathbb{K}^{\times}$ such that $a\mathcal{O}\tilde{L}=L$. The element $a$ is uniquely determined by $\tilde{L}$ up to multiplication by a unit in $\mathcal{O}$ and we have $\mathcal{M}L = a\mathcal{M}\tilde{L} \subseteq a\tilde{L}$. Hence the $\mathcal{O}_1$-module $a\tilde{L}/\mathcal{M}L$ can be seen as a submodule of the $\mathcal{K}_1$-vector space $L/\mathcal{M}L$ and its class in $\MC_{\OC_1,L}$ only depends on the class $[\tilde{L}]$ of $\tilde{L}$ in $\mathcal{I}^{\mathcal{L}}(\mathrm{SL}(V),\omega)$. We can therefore define the residue map $\mathrm{Res}_L$ by setting $\mathrm{Res}_L([\tilde{L}]):= [a\tilde{L}/\mathcal{M}L] \in \mathcal{I}^{\mathcal{L}}(\mathrm{SL}(L/\mathcal{M}L),\omega_0)$. Moreover, for any $g \in \mathrm{SL}(L)$, we have $ag(\mathcal{O}\tilde{L}) = g(L) = L$, and hence:
$$g\cdot \mathrm{Res}_L([\tilde{L}]) = g \cdot [a\tilde{L}/\mathcal{M}L] = [g(a\tilde{L})/\mathcal{M}L]= [ag(\tilde{L})/\mathcal{M}L]=\mathrm{Res}_L([g\cdot \tilde{L}]).$$
The residue map is therefore $\mathrm{SL}(L/\mathcal{M}L)$-equivariant.\\

Let's now prove that the image of the residue map is $\mathcal{I}^{\mathcal{L}}(\mathrm{SL}(L/\mathcal{M}L),\omega_0)$. To do so, fix an $\mathcal{O}$-basis $(b_0,...,b_{\ell})$ of $L$ and set $\tilde{L}_0:=\mathbb{O} b_0 \oplus ... \oplus \mathbb{O} b_{\ell}$. Given an $\mathbb{O}$-lattice $\tilde{L}$ in $V$ such that $[\tilde{L}] \in \pi^{-1}([L])$, we can find an $\mathbb{O}$-basis $(c_0,...,c_{\ell})$ of $\tilde{L}_0$ and elements $a,x_0,...,x_{\ell}$ in $\mathbb{K}^{\times}$ such that $a\mathcal{O}\tilde{L}=L$ and $\tilde{L} = \mathbb{O} x_0c_0 \oplus ... \oplus \mathbb{O} x_{\ell}c_{\ell}$. Since:
$$a(\mathcal{O}x_0c_0 \oplus ... \oplus \mathcal{O}x_{\ell}c_{\ell}) = a\mathcal{O}\tilde{L}=L=\mathcal{O}\tilde{L}_0=\mathcal{O}c_0 \oplus ... \oplus \mathcal{O}c_{\ell},$$
the elements $ax_0,...,ax_{\ell}$ all lie in $\mathcal{O}^{\times}$. Hence, if $\overline{c}_0,...,\overline{c}_{\ell}$ denote the projections of $c_0,...,c_{\ell}$ in $L/\MC L$ and if $y_0,...,y_{\ell}$ stand for the projections of $ax_0,...,ax_{\ell}$ in $\mathcal{K}_1^{\times}$, the family $(\overline{c}_0,...,\overline{c}_{\ell})$ is a basis in $L/\MC L$ and:
$$\mathrm{Res}_L([\tilde{L}]) = [\mathcal{O}_1 y_0\overline{c}_0 \oplus ... \oplus \mathcal{O}_1 y_{\ell}\overline{c}_{\ell}].$$
We deduce that $\mathrm{Res}_L([\tilde{L}]) \in \mathcal{I}^{\mathcal{L}}(\mathrm{SL}(L/\mathcal{M}L),\omega_0)$ and that the image of the residue map is contained in $\mathcal{I}^{\mathcal{L}}(\mathrm{SL}(L/\mathcal{M}L),\omega_0)$.

Conversely, given an $\mathcal{O}_1$-lattice $\overline{\mathcal{L}}$ in $L/\mathcal{M}L$, we can find $g \in \mathrm{GL}(\tilde{L}_0/\mathcal{M}\tilde{L}_0)$ such that $g(\tilde{L}_0/\mathcal{M}L) = \overline{\mathcal{L}}$. Since the projection $\mathrm{GL}(\tilde{L}_0) \rightarrow \mathrm{GL}(\tilde{L}_0/\mathcal{M}\tilde{L}_0)$ is surjective, we may lift $g$ into an element $\tilde{g} \in \mathrm{GL}(\tilde{L}_0) \subseteq \mathrm{GL}(L)$. We then have $ \mathcal{O}(\tilde{g}\cdot\tilde{L}_0) = \tilde{g}\cdot \mathcal{O}\tilde{L}_0 =\tilde{g} \cdot L = L$ and:
$$\mathrm{Res}_L([\tilde{g}(\tilde{L}_0)])= [\tilde{g}(\tilde{L}_0)/\mathcal{M}L]= [g(\tilde{L}_0/\MC L)] = [\overline{\mathcal{L}}].$$
This proves that the image of the residue map is $\mathcal{I}^{\mathcal{L}}(\mathrm{SL}(L/\mathcal{M}L),\omega_0)$.

We finish the proof of the theorem by establishing the injectivity of $\mathrm{Res}_L$. Consider two $\mathbb{O}$-lattices $\tilde{L}$ and $\tilde{L}'$ such that $[\tilde{L}]$ and $[\tilde{L}']$ are in $\pi^{-1}([L]) $ and $\mathrm{Res}_L([\tilde{L}]) = \mathrm{Res}_L([\tilde{L}'])$. We can then find $a$ and $a'$ in $\mathbb{K}^{\times}$ such that $a\mathcal{O}\tilde{L}=a'\mathcal{O}\tilde{L}'=L$ and $[(a\tilde{L})/\mathcal{M}\tilde{L}] = [(a'\tilde{L}')/\mathcal{M}\tilde{L}] \in B(\mathrm{SL}(L/\mathcal{M}L),\omega_0)$. Hence there exists $b \in \mathcal{O}^{\times}$ such that $(a\tilde{L})/\mathcal{M}\tilde{L} = (ba'\tilde{L}')/\mathcal{M}\tilde{L} \subseteq \tilde{L}/\mathcal{M}\tilde{L}$. Therefore, $a\tilde{L} = ba'\tilde{L}'$ and $[\tilde{L}]=[\tilde{L}']$.
\end{itemize}
\end{proof}

Figures \ref{figsl2a} and \ref{figsl2b} represent the projection $\pi$ for $\mathrm{SL}_2$ and $\mathrm{SL}_3$.

\begin{figure}[H]

\centering

\includegraphics[scale=0.3]{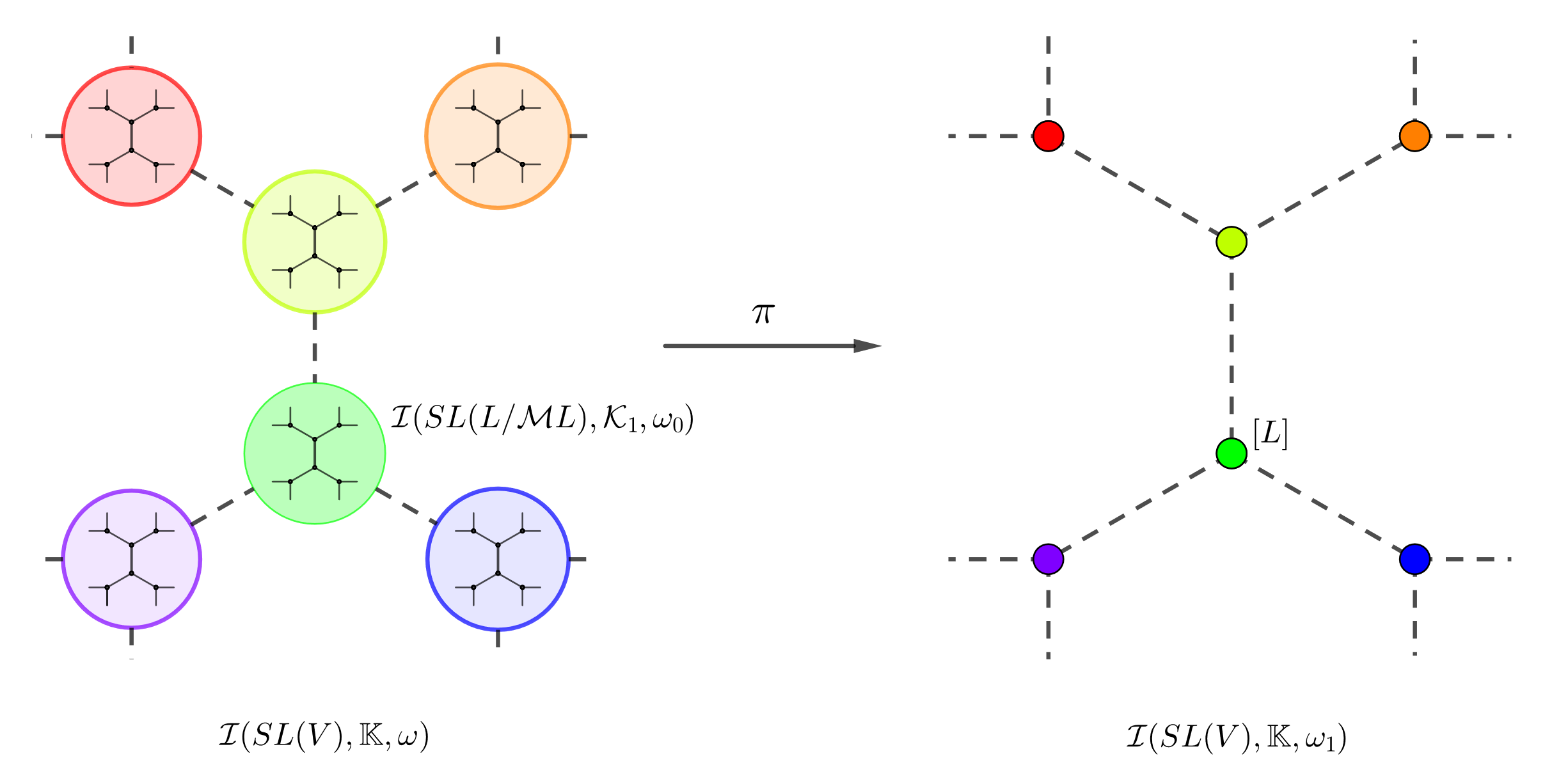}

\caption{ \label{figsl2a} The building of $\mathrm{SL}(V)$ when $V$ is a $2$-dimensional vector space over a field $\mathbb{K}$ endowed with a valuation $\omega: \mathbb{K}^{\times} \rightarrow \mathbb{Z}^2$.}

\includegraphics[scale=0.6]{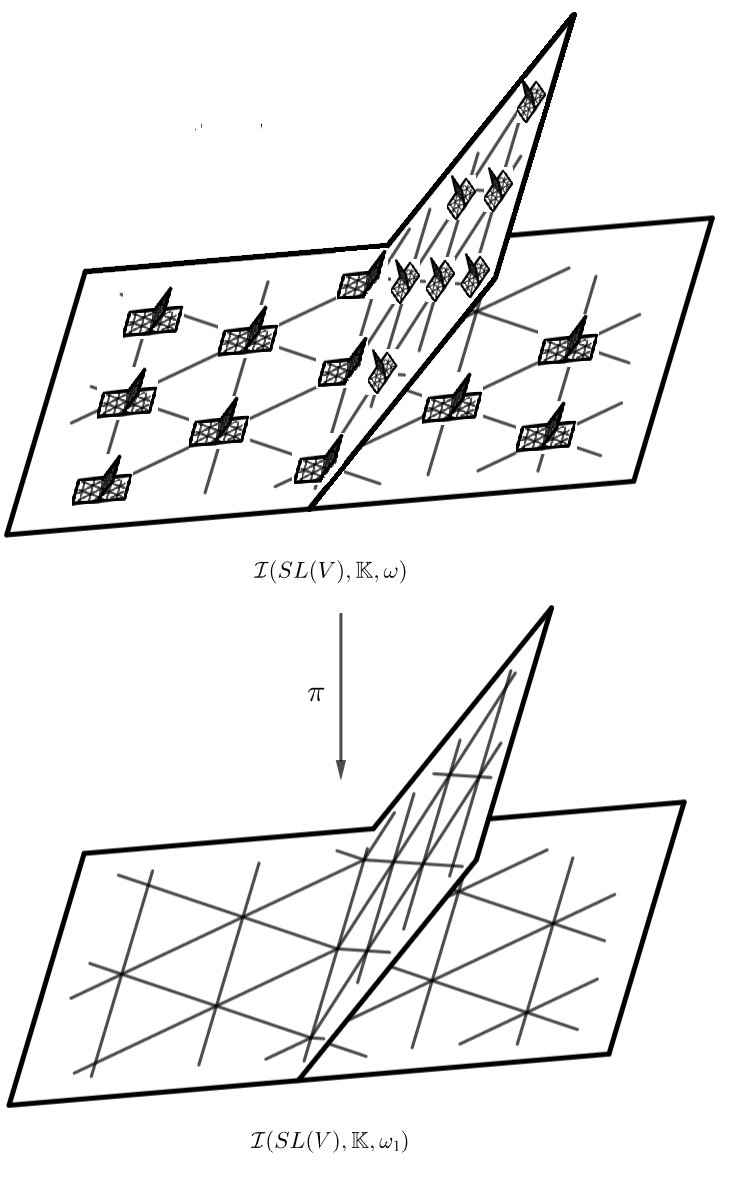}

\caption{\label{figsl2b} The building of $\mathrm{SL}(V)$ when $V$ is a $3$-dimensional vector space over a field $\mathbb{K}$ endowed with a valuation $\omega: \mathbb{K}^{\times} \rightarrow \mathbb{Z}^2$.}

\end{figure}

\subsection{The \texorpdfstring{$\mathbb{Z}^2$}{Z2}-tree of \texorpdfstring{$\mathrm{SL}_2$}{SL(2)}}

From now on, we assume that $\ell=1$ and that $\Lambda= \mathbb{Z}^2$ with the lexicographic order. The field $\mathbb{K}$ is therefore endowed with a $\mathbb{Z}^2$-valuation, and we fix two elements $t$ and $u$ of $\mathbb{K}$ with respective valuations $(1,0)$ and $(0,1)$. We are interested in the lattice $\mathbb{Z}^2$-tree of $\mathrm{SL}_2$ over $\mathbb{K}$. The ordered group $\Lambda$ has $\Lambda_0:= \{0 \} \times \mathbb{Z}$ as a convex subgroup. We can hence introduce the quotient $\Lambda_1:= \Lambda/\Lambda_0$ and adopt all the notations of the previous section. In particular, we have a valuation $\omega_1$, as well as a projection:
$$\pi: \mathcal{I}^{\mathcal{L}}(\mathrm{SL}_2(V),\omega) \rightarrow \mathcal{I}^{\mathcal{L}}(\mathrm{SL}_2(V),\omega_1).$$
By theorem \ref{sll}, all the fibers of $\pi$ are isomorphic to the usual tree of $\mathrm{SL}_2$ over the $\mathbb{Z}$-valued field $\mathcal{K}_1$.

\begin{remark}
The results presented in this section were first observed by Parshin in \cite{parshin1994higher}. We will not include the proofs: all of them can be done using elementary linear algebra or, alternatively, by applying to $\mathrm{SL}_2$ the general results that will be proved in the sequel of the article (see section \ref{secBuilding_sld}).
\end{remark}

A \textbf{lattice apartment} in $\mathcal{I}^{\mathcal{L}}(\mathrm{SL}_2(V),\omega)$ is a set of the form $\{\mathbb{O}e_1 \oplus \mathbb{O}x_{\lambda}e_2] \, | \, \lambda \in \Lambda \}$ for some basis $(e_1,e_2)$ of $V$ and some family $(x_{\lambda})_{\lambda \in \Lambda}$ in $\mathbb{K}$ such that $\omega(x_{\lambda}) = \lambda$ for each $\lambda$. The action of $G$ on $\mathcal{I}^{\mathcal{L}}(\mathrm{SL}_2(V),\omega)$ obvously preserves this apartment system.

 In order to get a better understanding of $\mathcal{I}^{\mathcal{L}}(\mathrm{SL}_2(V),\omega)$ and its apartment system, we are going to glue the different fibers of $\pi$. For $P \in \mathcal{I}^{\mathcal{L}}(\mathrm{SL}_2(V),\omega_1)$, the fiber $\pi^{-1}(P)$ will be called the \textbf{infinitesimal subtree} of $\mathcal{I}^{\mathcal{L}}(\mathrm{SL}_2(V),\omega)$ above $P$ and will be denoted $T_P$.

Let $\partial_1 \mathcal{I}^{\mathcal{L}}(\mathrm{SL}(V),\omega)$ be the set of $\mathbb{O}$-submodules of $V$ of the form $\mathbb{O} b_1\oplus \OC b_2$, where $(b_1,b_2)$ is a $\mathbb{K}$-basis of $V$, quotiented by the homothety relation. 

\begin{definition}
\begin{itemize}
\item[(i)] Let $P\in \mathcal{I}^{\mathcal{L}}(\mathrm{SL}(V),\omega_1)$. A \textbf{(group-theoretic) ray}\index{ray!group-theoretic} of $T_P$ is a sequence of the form $([\mathbb{O} b_1\oplus \mathbb{O} u^{n} b_2])_{n\in \N}\in T_P^{\N}$, for some basis $(b_1,b_2)$ of $V$.  We say that two rays $(P_n)_{n\in \N}$ and $(Q_n)_{n\in \N}$ satisfy $(P_n)\sim (Q_n)$ if there exists $k\in \Z$ such that $P_{n+k}=Q_n$ for all $n\gg 0$. A class of rays for this relation is called an \textbf{end} of $T_P$. We denote by $\EC(T_P)$ the set of ends of $T_P$ and we set $\EC(\mathcal{I}^{\mathcal{L}}(\mathrm{SL}(V),\omega))=\bigcup_{P\in \mathcal{I}^{\mathcal{L}}(\mathrm{SL}(V),\omega)}\EC(T_P)$
\item[(ii)] Let $\epsilon\in \{-,+\}$ and set $\eta(-)=1$ and $\eta(+)=0$. Let $P$ be an element in $ \partial_1 \mathcal{I}^{\mathcal{L}}(\mathrm{SL}(V),\omega)$ and take $E \in \EC(\mathcal{I}^{\mathcal{L}}(\mathrm{SL}(V),\omega))$.  We say that $E$ \textbf{converges} towards $P^\epsilon$ if there exists a $\mathbb{K}$-basis $(b_1,b_2)$ of $V$ such that the ray $\left( [\mathbb{O}b_1\oplus \mathbb{O} u^{-\epsilon n}b_2] \right)$ represents $E$ and $P=[\mathbb{O}b_1\oplus \OC t^{\eta(\epsilon)} b_2]$. This definition is inspired by the fact that $\bigcap_{n\in \N} u^{n} \mathbb{O}=t\OC$ and $\bigcup_{n\in \N}u^{-n} \mathbb{O}=\OC$.
\end{itemize}
\end{definition}

\begin{remark}
One has: 
\begin{gather*}
[\mathbb{O} e_1\oplus \mathbb{O} u^m e_2]\underset{m\rightarrow +\infty}{\rightarrow}[\mathbb{O} e_1\oplus \OC t e_2]^-,\\
[\mathbb{O} e_1\oplus \mathbb{O} u^me_2]=[\mathbb{O} u^{-m} e_1\oplus \mathbb{O} e_2]\underset{m\rightarrow +\infty}{\rightarrow} [\OC e_1\oplus \mathbb{O} e_2]^+.
\end{gather*}
\end{remark}

The following proposition shows that limits are uniquely defined:

\begin{proposition}
Let $\epsilon \in \{+,-\}$ and let $E$ be an element in $ \mathcal{E}(\mathcal{I}^{\mathcal{L}}(\mathrm{SL}(V),\omega)$.
Then there exists a unique element $P \in \partial_1 \mathcal{I}^{\mathcal{L}}(\mathrm{SL}(V),\omega)$ such that  $E \rightarrow P^{\epsilon}$.
\end{proposition}

This allows us to introduce two maps:
\begin{gather*}
{\lim}^+: \EC(\mathcal{I}^{\mathcal{L}}(\mathrm{SL}(V),\omega)) \rightarrow \partial_1 \mathcal{I}^{\mathcal{L}}(\mathrm{SL}(V),\omega), \\
{\lim}^-: \EC(\mathcal{I}^{\mathcal{L}}(\mathrm{SL}(V),\omega)) \rightarrow \partial_1 \mathcal{I}^{\mathcal{L}}(\mathrm{SL}(V),\omega),
\end{gather*}
that send each $E \in \mathcal{E}(\mathcal{I}^{\mathcal{L}}(\mathrm{SL}(V),\omega)$ to the unique element $P \in \partial_1 \mathcal{I}^{\mathcal{L}}(\mathrm{SL}(V),\omega)$ such that $E \rightarrow P^{+}$ and $E \rightarrow P^{-}$ respectively. One can then prove the following theorem:

\begin{theorem} \label{arbre}
Let  $P\in \mathcal{I}^{\mathcal{L}}(\mathrm{SL}(V),\omega_1)$ and set: $$\SC(P,1):=\{Q\in \mathcal{I}^{\mathcal{L}}(\mathrm{SL}(V),\omega_1)|\ d(P,Q)=1\}.$$ Consider the map:
\begin{align*}
\pi_1: \partial_1 \mathcal{I}^{\mathcal{L}}(\mathrm{SL}(V),\omega) & \rightarrow \mathcal{I}^{\mathcal{L}}(\mathrm{SL}_2(V),\omega_1)\\
[L] & \mapsto [\mathcal{O}L].
\end{align*}
Then:

\begin{enumerate}
\item  the map $\lim^+:\EC(T_P)\rightarrow \pi_1^{-1}(P)$ is a bijection,

\item the map $\pi\circ \lim^-: \EC(T_P)\rightarrow \SC(P,1)$ is a bijection,

\item for all $E\in \EC(T_P)$, there exists a unique $\tilde{E}\in \EC(\mathcal{I}^{\mathcal{L}}(\mathrm{SL}(V),\omega))$ such that $\lim^- E=\lim^+ \tilde{E}$. Moreover, $\lim^+ E = \lim^- \tilde{E}$ and $\tilde{E}\in \EC(T_{\pi(\lim^- E)})$.
\end{enumerate}

\end{theorem}

One can then decide to glue each end $E \in \mathcal{E}(\mathcal{I}^{\mathcal{L}}(\mathrm{SL}(V),\omega))$ with the unique end $\tilde{E} \in \mathcal{E}(\mathcal{I}^{\mathcal{L}}(\mathrm{SL}(V),\omega))$ such that $\lim^- E=\lim^+ \tilde{E}$. Once this glueing is done, apartments are then maximal paths in $\mathcal{I}^{\mathcal{L}}(\mathrm{SL}(V),\omega)$, as illustrated in figure \ref{figsl2c}.

\begin{figure}[H]

\centering

\includegraphics[scale=1.5]{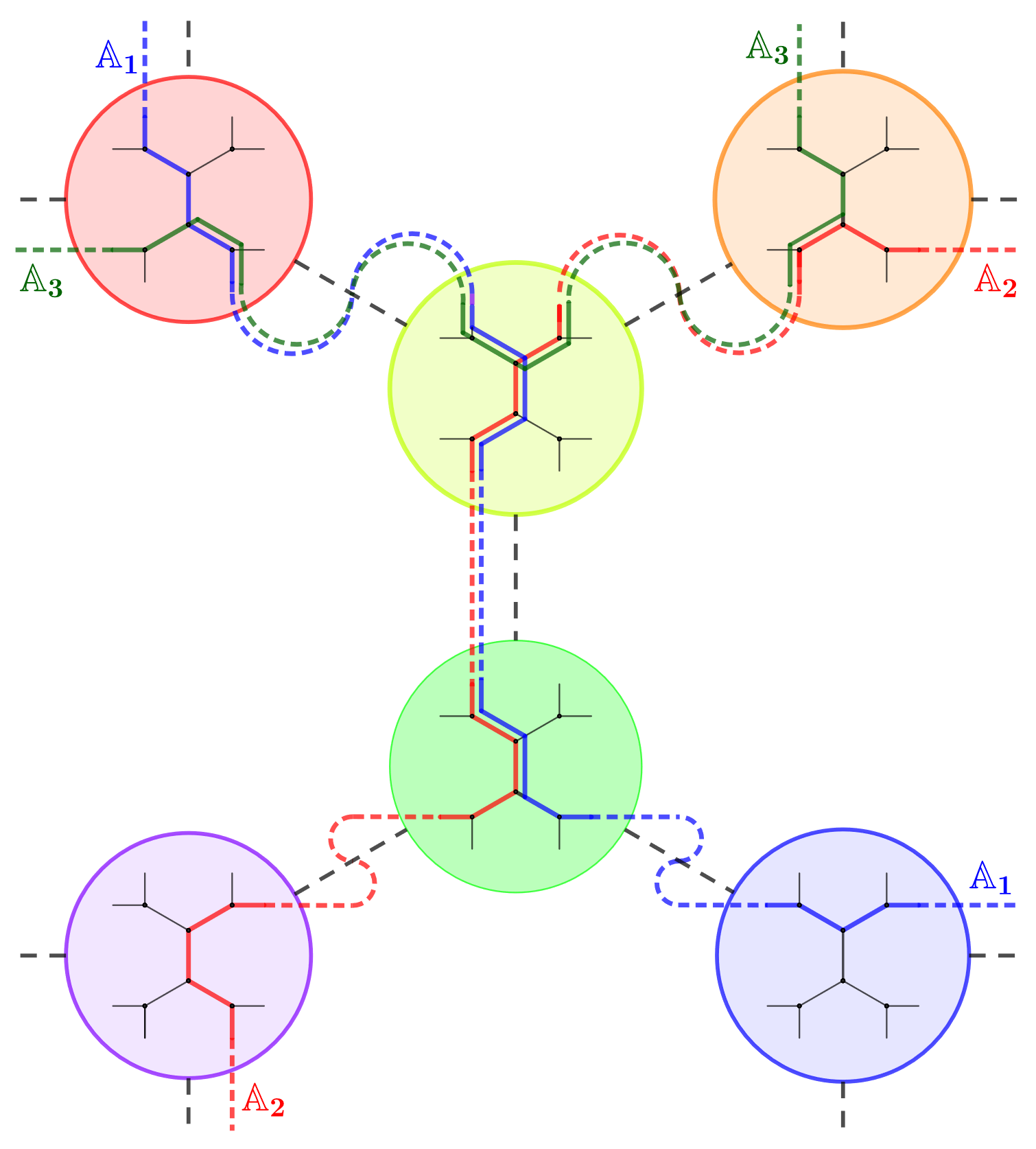}

\caption{\label{figsl2c} Three apartments in the $\mathbb{Z}^2$-building of $\mathrm{SL}_2$ over a $\mathbb{Z}^2$-valued field.}

\end{figure}

By construction, any apartment of $\mathcal{I}^{\mathcal{L}}(\mathrm{SL}(V),\mathbb{K})$ is isomorphic to $\Lambda$. If we fix a basis $(e_1,e_2)$ of $V$ and we denote by $\mathbb{A}$ the apartment associated to this basis, the stabilizer of $\mathbb{A}$ in $G$ is the group $N$ of elements of $\mathrm{SL}(V)$ whose matrices in the basis $(e_1,e_2)$ are of the form:
$$\begin{pmatrix}
\star & 0 \\
0 & \star
\end{pmatrix}
\;\;\;\;\; \mathrm{or}\;\;\;\;\;\begin{pmatrix}
0 & \star \\
\star & 0
\end{pmatrix}.$$
The group $N$ acts on  $\mathbb{A}$ through the quotient $N/T_b$ where $T_b$ is the group of elements of $\mathrm{SL}(V)$ whose matrices in the basis $(e_1,e_2)$ are of the form:
$$\begin{pmatrix}
a & 0 \\
0 & a^{-1}
\end{pmatrix}$$
with $a \in \mathbb{O}^{\times}$. The quotient $\Wext:=N/T_b$ is called the \textbf{extended affine affine Weyl group}
and it is spanned by the elements:
$$w_0=\begin{pmatrix} 0 & 1 \\
-1 & 0
\end{pmatrix}, w_1=\begin{pmatrix} 0 & u\\
-u^{-1} & 0\end{pmatrix},
w_2=\begin{pmatrix} 0 & t \\ 
-t^{-1} & 0\end{pmatrix}.$$ 
One can then check that:
\begin{itemize}
    \item the set $\{0,1\}^2$ is a fundamental domain for the action of $\Wext$ on $\A$.
    \item the set $[0,1]^2\cup (1,2)\times (0,1]=\big([0,2)\times [0,1]\big)\setminus \big((1,2)\times \{0\}\big)$ is a fundamental domain for the action of $\Wext$ on $\A\otimes_\mathbb{Z} \mathbb{R}$.
\end{itemize}
 The action of $\Wext$ on $\A$ and $\A\otimes_\mathbb{Z} \mathbb{R}$ is represented in figure \ref{figsl2d}.

\begin{figure}[H]

\centering

\includegraphics[scale=0.45]{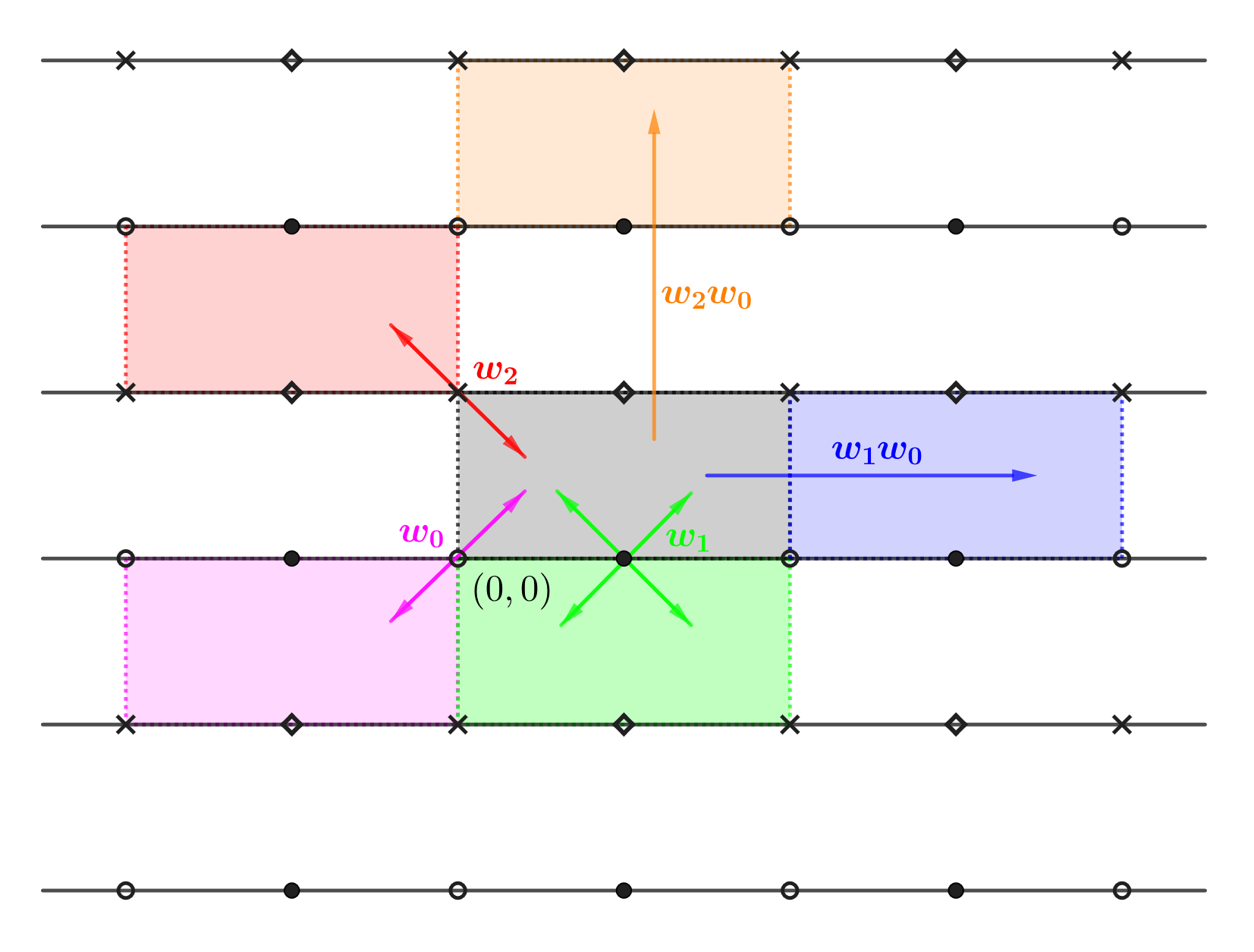}

\caption{\label{figsl2d} The action of the extended affine Weyl group on the standard apartment of the building of $\mathrm{SL}(V)$ when $V$ is a $2$-dimensional vector space over a field $\mathbb{K}$ endowed with a valuation $\omega: \mathbb{K}^{\times} \rightarrow \mathbb{Z}^2$.}

\end{figure}

\section{Abstract definition of \texorpdfstring{$\RF^S$}{RS}-buildings and statement of the main theorem}\label{secAbstract_definition_buildings}

This section is dedicated to recalling the definition $\Lambda$-buildings as defined by Bennett in \cite{bennett1994affine}. By doing so, we also introduce a certain number of new constructions that turn out to be useful to describe the linear algebra, the affine geometry and the combinatorics of the apartments of those buildings. At the end of the section, we state our main theorem. 

\subsection{Definition of the standard apartment}
\label{secDefinitionApartment}

\subsubsection{Root system, vectorial apartment over \texorpdfstring{$\R$}{R} and Weyl group}\label{subsec:rootsystems}

Let $V_\mathbb{R}$ be a finite dimensional vector space over $\R$.
Let $\Phi\subset (V_\R)^*$ be a root system\index{root system} as in the definition of \cite[6.1.1]{bourbaki1981elements}. Let $\Phi^\vee\subset V_\R$ be the dual root system of $\Phi$.\index{root system!dual} In particular, $\Phi$ (resp. $\Phi^\vee$) is a finite subset of $(V_\R^*)\setminus\{0\}$ (resp $V_\R\setminus\{0\}$) and there is a bijection ${^\vee }:\Phi\rightarrow \Phi^\vee$ such that $\alpha(\alpha^\vee)=2$ for each $\alpha\in \Phi$.

In order to deal with non-reduced root systems,\index{root system!non-reduced} let us recall some definitions and facts from \cite[VI §1.3 \& §1.4]{bourbaki1981elements}.
A root $\alpha \in \Phi$ is said to be multipliable\index{root!multipliable} if $2\alpha \in \Phi$; otherwise, it is said to be non-multipliable\index{root!non-multipliable}.
A root $\alpha \in \Phi$ is said to be divisible\index{root!divisible} if $\frac{1}{2}\alpha \in \Phi$; otherwise, it is said to be non-divisible\index{root!non-divisible}.
The set of non-divisible roots, denoted by $\Phi_{\mathrm{nd}}$\index[notation]{p@$\Phi_{\mathrm{nd}}$}, is a root system.

\begin{Not}\label{NotRootSystem}
Given two roots $\alpha,\beta \in \Phi$ with respective coroots $\alpha^\vee,\beta^\vee \in \Phi^\vee$, we denote: 
\begin{itemize}
\item $\Phi(\alpha,\beta) = \{r\alpha+s\beta \in \Phi,\ (r,s)\in \mathbb{Z}^2\}$\index[notation]{p@$\Phi(\alpha,\beta)$};
\item $(\alpha,\beta) = \{r\alpha+s\beta \in \Phi,\ (r,s)\in (\mathbb{Z}_{>0})^2\}$\index[notation]{a@$(\alpha,\beta)$};
\item $r_\alpha(\beta) = \beta - \beta(\alpha^\vee) \alpha$;
\item $r_\alpha(\beta^\vee) = \beta^\vee - \alpha(\beta^\vee) \alpha^\vee$\index[notation]{r@$r_\alpha$}.
\end{itemize}
\end{Not}

Note that $\Phi(\alpha,\beta)$ is a root system of rank $1$ or $2$ depending on the fact that $\alpha$ and $\beta$ are, or not, collinear.
The subset $(\alpha,\beta)$ is a positively closed subset of $\Phi(\alpha,\beta)$ when $\beta \not\in - \mathbb{R}_{>0} \alpha$.

The map $r_\alpha: \Phi \to \Phi$ (resp. $r_\alpha: \Phi^\vee \to \Phi^\vee$) extends linearly to a reflection $r_\alpha \in \mathrm{GL}(V_\R^*)$ (resp. $r_\alpha \in \mathrm{GL}(V_\R)$).
It satisfies $r_\alpha^2 = \operatorname{id}$ and $r_{2\alpha} = r_\alpha$ when $2\alpha \in \Phi$.
We denote by $W(\Phi)$\index[notation]{w@$W(\Phi)$} (resp. $W(\Phi^\vee)$) the subgroup of $\mathrm{GL}(V_\R^*)$ (resp. $\mathrm{GL}(V_\R)$) generated by all the $r_\alpha$.
The groups $W(\Phi)$ and $W(\Phi^\vee)$ are both finite and they are called the Weyl groups\index{Weyl group}
associated to $\Phi$ and $\Phi^\vee$ respectively.
Moreover, for any basis $\Delta$ of $\Phi$, the Weyl group $W(\Phi)$ of $\Phi$ is generated by the $r_\alpha$ for $\alpha \in \Delta$.
There is a natural isomorphism of finite groups $W(\Phi) \to W(\Phi^\vee)$ sending $r_\alpha \in W(\Phi)$ to $r_\alpha \in W(\Phi^\vee)$ for any $\alpha \in \Phi$.
By abuse of notations, we denote by $w^{-1}$ the image of an element $w \in W(\Phi)$ through this isomorphism.
The group $W^v$ defined up to this isomorphism is called the \textbf{vectorial Weyl group}\index{Weyl group!vectorial}.
Note that for any $\alpha \in \Phi$ and any $w \in W(\Phi)$, one has $w(\alpha) = \alpha \circ w^{-1}$.

A \textbf{vector chamber}\index{chamber!vector}\index{vector chamber} $C^v_\R$ is a connected component of $V_\R\setminus\bigcup_{\alpha\in\Phi}\alpha^{-1}(\{0\})$.
If  $\overline{C}^v_\R$ stands for the closure of $C^v_\R$ in $V_\R$ for the usual topology, the associated \textbf{basis}\index{chamber!basis} $\Delta_{C^v_\R}$ of $\Phi$ is the set of roots $\alpha\in \Phi$ such that $\alpha(C^v_\R)>0$ and $\alpha^{-1}(\{0\})\cap \overline{C}^v_\R$ spans a hyperplane of $V_\R$.
Let $\Phi_{C^v_\R}=\{\alpha\in \Phi\mid \alpha(C^v_\R)>0\}$ be the set of positive roots for $C^v_\R$. Then $\Phi=\Phi_{C^v_\R}\sqcup -\Phi_{C^v_\R}$ and $\Phi_{C^v_\R}\subset \bigoplus_{\alpha\in \Delta_{C^v_\R}} \N \alpha$. For every vector chamber $C^v_\R$ of $V_\R$, $(W^v,\{r_{\alpha}\mid \alpha\in \Delta_{C^v_\R}\})$ is a Coxeter system.
 
 We now fix a vector chamber $C^v_{f,\R}$ of $V_\R$, that we call the \textbf{fundamental chamber},\index{chamber!fundamental}\index{fundamental chamber} and we set $\Delta_f=\Delta_{C^v_{f,\R}}$. We denote by $\ell$ the length map on $W^v$ associated to the generating set $\{r_\alpha\mid \alpha\in\Delta_{C^v_{f,\R}}\}$.

Let $\Delta$ be a basis of $\Phi$.
For $\alpha\in \Delta$ we denote by $\varpi_\alpha^\vee$\index[notation]{p@$\varpi_\alpha^\vee$} the \textbf{fundamental coweight}\index{coweight!fundamental} in $V_\R$ associated with $\alpha$, i.e the unique element of $V_\R$ such that for all $\beta\in \Delta$, one has $\beta(\varpi_\alpha^\vee)=\delta_{\alpha,\beta}$.

\subsubsection{\texorpdfstring{$R$}{R}-aff spaces}\label{raff}

Let $\mathbf{G}$ be a quasi-split reductive group defined over a field $\mathbb{K}$ endowed with a valuation $\omega:\mathbb{K}\rightarrow \Lambda\cup\{\infty\}$ for some totally ordered abelian group $\Lambda$. In the classical Bruhat-Tits theory, when $\Lambda =  \mathbb{Z}$, the apartments of the building associated to $\mathbf{G}$ are all modeled on a real affine space. The choice to work over the real numbers in this context comes from the observation that the valuation group $\Lambda$ can be seen as an ordered subgroup of $\mathbb{R}$. \\

If now $\Lambda =\mathbb{Z}^n$ with the lexicographical order, one cannot embed $\Lambda$ into the real numbers as an ordered subgroup anymore. To tackle this difficulty, we will replace the ring $\mathbb{R}$ by the ring $\mathbb{R}[t]/(t^n)$ endowed with the lexicographical order associated to the basis $(1,t,t^2,...,t^{n-1})$. Each apartment of the "higher" building we will associate to $\mathbf{G}$ will then be an $\mathbb{R}[t]/(t^n)$-affine space, that is a set $E$ endowed with a free and transitive action of an $\mathbb{R}[t]/(t^n)$-module $\overrightarrow{E}$.\\

In the general case, where $\Lambda$ is any non-zero totally ordered abelian group, one would like to model the apartments of the "higher" building associated to $\mathbf{G}$ on an $R$-affine space for some totally ordered commutative ring $R$ containing $\Lambda$. This is morally what we will do in the sequel and one should keep this idea in mind to understand the upcoming constructions. However, this point of view is not completely correct since in some cases the base $R$ we will consider will not be endowed with a ring structure. In general, $R$ will only be a totally ordered non-zero real vector space, that is a real vector space endowed with a total order $<$ satisfying the following two properties: 

\begin{itemize}
\item[(1)] $x+y>0$ for any $x,y \in R$ such that $x>0$ and $ y>0$;
\item[(2)] $\lambda x >0$ for any $x\in R$ and $\lambda \in\mathbb{R}$ such that $x>0$ and $\lambda>0$.
\end{itemize}
In this setting, we cannot model the apartments of the "higher" building associated to $\mathbf{G}$ on an $R$-affine space since this notion makes no sense. In order to still do affine geometry over the base $R$ just as if $R$ had a ring structure, the trick will consist in replacing the category of $R$-affine spaces by the category $\mathrm{Aff}_R$ defined as follows:
\begin{itemize}
\item[-] Objects: pairs $(V,A)$ where $V$ is a free  $\mathbb{Z}$-module of finite type and $A$ is a real affine space  with underlying vector space $V_R:=V \otimes_{\mathbb{Z}} R$;
\item[-] Morphisms: if $(V,A)$ and $(W,B)$ are two objects, a morphism from $ (V,A)$ to $(W,B)$ is an affine map $f: A \rightarrow B$ whose linear part $\overrightarrow{f}$ is induced by tensorization from a $\mathbb{Z}$-linear map $\overrightarrow{f}_{\mathrm{core}}: V \rightarrow W$. Equivalently, the map $f: A \rightarrow B$ is of the form:
$$f(\,\cdot\,) = (\overrightarrow{f}_{\mathrm{core}}\otimes \mathrm{Id}_R) (\,\cdot\, - a) + b$$
for some $a \in A$ and $b\in B$.
\end{itemize}
The objects of $\mathrm{Aff}_R$ will be called \textbf{$R$-aff spaces}\index{R-aff@$R$-aff!space} and the morphisms \textbf{$R$-aff maps}\index{R-aff@$R$-aff!map}. The apartments of the "higher" building associated to $\mathbf{G}$ will then be modeled on $R$-aff spaces, and all the maps we will define on those apartments will be $R$-aff maps. By doing so, we will be able to carry out all the arguments and all the proofs in exactly the same way as if $R$ had a ring structure.

 \begin{remark}
Many of the constructions we will do in the sequel can be done in the more general situation where $R$ is not a real vector space but just an ordered abelian group or an ordered $\Q$-vector space. The notions of $R$-aff spaces and $R$-aff maps still make sense in those cases.
 \end{remark}

\subsubsection{Affine apartments over a totally ordered abelian group \texorpdfstring{$R$}{R}}\label{subsubTopology_totally_ordered_ring}

Motivated by the previous paragraph, we fix a totally ordered abelian group $R$ that will often be chosen later on to be a real vector space and we aim at giving a precise definition of an affine apartment over $R$.

\begin{Not}
We introduce a symbol $\infty$ and we extend the total ordering on $\Rtot$ to a total ordering of the set $\Rtot \cup \{\infty\}$ by setting $\lambda < \infty$ for any $\lambda \in \Rtot$.

If $\lambda\in \Rtot$ and $\top$ is a binary relation on $\Rtot$ (typically $\leqslant, <, >,\geqslant$), we denote by 
\[\Rtot_{ \top \lambda} = \{\mu \in \Rtot,\ \mu \top \lambda\}.\]

For any $\lambda < \mu$ in $\Rtot \cup \{\infty\}$, we denote the intervals:
\begin{itemize}
\item $[\lambda,\mu] = \{\nu \in \Rtot \cup\{\infty\},\ \lambda \leqslant \nu \leqslant \mu\}$;
\item $]\lambda,\mu] = \{\nu \in \Rtot \cup\{\infty\},\ \lambda < \nu \leqslant \mu\}$;
\item $[\lambda,\mu[ = \{\nu \in \Rtot,\ \lambda \leqslant \nu < \mu\}$;
\item $]\lambda,\mu[ = \{\nu \in \Rtot,\ \lambda < \nu < \mu\}$.
\end{itemize}
We equip $\Rtot$ with the order topology for which a base is given by the sets $\Rtot_{>\lambda}$ and $\Rtot_{< \lambda}$ for $\lambda \in \Rtot$  . Thus, $\Rtot$ is a completely normal Hausdorff space.
\end{Not}
\vspace{3ex}

\textbf{a) The vectorial apartment $V_R$ and its combinatorial structure}\\
 
Let us now come back to the setting described in section \ref{subsec:rootsystems} and consider a basis $\Delta^\vee$ of $\Phi^\vee_{\mathrm{nd}}$.
Denote by $V_\mathbb{Z}$ the free $\mathbb{Z}$-submodule of $V_\mathbb{R}$ spanned by the fundamental coweights, that is the lattice of \textbf{coweights}\index{coweight}, which contains $\Phi^\vee$.
Since $W^v$ acts transitively on the set of basis of $\Phi^\vee_{\mathrm{nd}}$ \cite[VI.1.5]{bourbaki1981elements}, this does not depend on the choice of $\Delta^\vee$. We then set $V_R: = V_\mathbb{Z} \otimes_{\mathbb{Z}} R$\index[notation]{v@$V_\R$}, and for $v\in V_\mathbb{Z}$ and $\lambda \in R$, we denote $\lambda v := v \otimes \lambda \in V_R$.

Fix a root  $\alpha \in \Phi$. By definition of root systems, $\alpha(\Phi^\vee) \subset \mathbb{Z}$.
Thus $\alpha$ induces a $\mathbb{Z}$-linear map $V_{\mathbb{Z}} \rightarrow \mathbb{Z}$, and by tensorization one gets an $R$-linear map $\alpha \otimes \mathrm{Id}_R: V_R \rightarrow R$. By abuse of notation this last linear map will still be denoted by $\alpha$. 

For $\lambda \in \Rtot$, set: 
$$H_{\Rtot,\alpha,\lambda} := \alpha^{-1}(\{-\lambda\}) \subset V_\Rtot, \;\;\;\;\; \mathring{D}_{\Rtot,\alpha,\lambda} := \alpha^{-1}(\Rtot_{>-\lambda}), \;\;\;\;\;D_{\Rtot,\alpha,\lambda} := \alpha^{-1}(\Rtot_{\geqslant -\lambda}) = H_{\Rtot,\alpha,\lambda} \sqcup \mathring{D}_{\Rtot,\alpha,\lambda}.$$

In the case $\lambda=0$, we also denote those sets by $H_{R,\alpha}$, $\mathring{D}_{R,\alpha}$ and $D_{R,\alpha}$ respectively. When there is no ambiguity, we might sometimes denote $H_{\alpha,\lambda}$\index[notation]{h@$H_{\alpha,\lambda}$}, $\mathring{D}_{\alpha,\lambda}$\index[notation]{d@$\mathring{D}_{\alpha,\lambda}$} and $D_{\alpha,\lambda}$\index[notation]{d@$D_{\alpha,\lambda}$} instead of $H_{\Rtot,\alpha,\lambda}$, $\mathring{D}_{\Rtot,\alpha,\lambda}$ and $D_{\Rtot,\alpha,\lambda}$, and $H_\alpha$, $\mathring{D}_\alpha$ and $D_\alpha$ instead of $H_{R,\alpha}$, $\mathring{D}_{R,\alpha}$ and $D_{R,\alpha}$. We equip $V_\Rtot$ with the topology generated by the sets  $\mathring{D}_{\alpha,\lambda}$ for $\alpha \in \Phi$ and $\lambda \in \Rtot$. The linear maps $\alpha: V_\Rtot \to \Rtot$ are then automatically continuous.

If $\Delta$ is a basis of $\Phi$, we denote by $\Phi^+_\Delta$ (resp. $\Phi^-_\Delta$) the  set $\Phi\cap\bigoplus_{\alpha\in \Delta}\Z_{\geq 0}\alpha$.\index[notation]{p@$\Phi^+_\Delta$} (resp. $-\Phi^+_\Delta$). We then have $\Phi=\Phi_\Delta^+\cup \Phi_\Delta^-$.

For any basis $\Delta$ of $\Phi$   and any subset $\Delta_P \subset \Delta$, we set:
 \[F^v_\Rtot(\Delta,\Delta_P)=\bigcap_{\alpha\in \Delta_P} H_{\Rtot,\alpha}\cap \bigcap_{\alpha\in \Delta\setminus\Delta_P} \mathring{D}_{\Rtot,\alpha} = \left\{ v \in V_\Rtot,\ \forall \alpha \in \Delta,\ \begin{array}{rl} \alpha(v) = 0 & \text{ if } \alpha \in \Delta_P\\ \alpha(v) > 0 & \text{ if } \alpha \not\in \Delta_P \end{array}\right\}\]
 \index[notation]{f@$F^v_\Rtot(\Delta,\Delta_P)$}
 and
 \[\overline{F}^v_\Rtot(\Delta,\Delta_P)=\bigcap_{\alpha\in \Delta_P} H_{\Rtot,\alpha}\cap \bigcap_{\alpha\in \Delta\setminus\Delta_P} D_{\Rtot,\alpha} = \left\{ v \in V_\Rtot,\ \forall \alpha \in \Delta,\ \begin{array}{rl} \alpha(v) = 0 & \text{ if } \alpha \in \Delta_P\\ \alpha(v) \geqslant 0 & \text{ if } \alpha \not\in \Delta_P \end{array}\right\}.\]
 
If $\Delta$ is any basis of $\Phi$ and $\Delta_P = \emptyset$, we set $C_{\Rtot,\Delta}^v = F_\Rtot^v(\Delta,\emptyset)$ and $\overline{C}^v_{\Rtot,\Delta} = \overline{F}_\Rtot^v(\Delta,\emptyset)$\index[notation]{c@$C_{\Delta}^v$}.
Note that those definition make sense even for $\Rtot = \mathbb{Z}$.
As before, if $\Rtot$ is clear in the context, we may omit to mention it in the notation.

If we fix a basis $\Delta_f$ of $\Phi$, then we denote by $F^v_\Rtot(\Delta_P)$ instead of $F^v_\Rtot(\Delta_f,\Delta_P)$ for $\Delta_P\subset \Delta_f$.
The \textbf{fundamental chamber}\index{chamber!fundamental} is the set $C^v_{f,\Rtot}=F^v_\Rtot(\emptyset)$\index[notation]{c@$C^v_{f,\Rtot}$}. A \textbf{vector face}\index{face!vector} (resp. \textbf{vector chamber})\index{chamber!vector} is a set of the form $w \cdot F^v_\Rtot(\Delta_P)$ (resp. $w \cdot C^v_{f,\Rtot}$), for some $w\in W^v$ and $\Delta_P\subset \Delta_f$.
Vector faces form a partition of $V_\Rtot$ and in Corollary~\ref{corWeylVectorial}, we get that the closure of the fundamental chamber is a fundamental domain for the action of $W^v$ on $V_\Rtot$.

\begin{Lem}\label{LemIntersectionChambers}
For any basis $\Delta$ of $\Phi$, we have $C^v_\Delta = \{v \in V_\Rtot,\ \forall \alpha \in \Phi^+_\Delta, \alpha(v) > 0\}$ and $\overline{C}^v_\Delta = \{v \in V_\Rtot,\ \forall \alpha \in \Phi^+_\Delta, \alpha(v) \geqslant 0\}$.

If $\Delta'$ is another basis of $\Phi$, we have $C^v_\Delta \cap C^v_{\Delta'} \neq \emptyset \Longleftrightarrow \Delta = \Delta'$.
\end{Lem}

\begin{proof}
Let $\alpha \in \Phi^+_\Delta$ and $v \in C^v_\Delta$ (resp. $\overline{C}^v_\Delta$).
By \cite[VI.1.6]{bourbaki1981elements}, there exist non-negative and not all zero integers $n_\beta$ such that $\alpha = \sum_{\beta \in \Delta} n_\beta \beta$.
Since $\beta(v) >0$ (resp. $\geqslant 0$) for any $\beta \in \Delta$, we get that $\sum_{\beta \in \Delta} n_\beta \beta(v) > 0$ (resp. $\geqslant 0$).
The converse is immediate since $\Delta \subset \Phi^+_\Delta$.
Hence $C^v_\Delta = \{v \in V_\Rtot,\ \forall \alpha \in \Phi^+_\Delta, \alpha(v) > 0\}$ and $\overline{C}^v_\Delta = \{v \in V_\Rtot,\ \forall \alpha \in \Phi^+_\Delta, \alpha(v) \geqslant 0\}$.

Suppose $\Delta' \neq \Delta$, then $\Phi^+_{\Delta'} \neq \Phi^+_{\Delta}$ and let $\alpha \in \Phi^+_{\Delta'} \setminus \Phi^+_{\Delta}$.
Hence $-\alpha \in \Phi^+_\Delta$.
For any $v \in C^v_{\Delta}$, we have $-\alpha(v) > 0$ so that $v \not\in C^v_{\Delta'}$.
\end{proof}

\begin{Def}\label{DefHalfLine}
For any $v \in V_\Z$, denote by $\delta_v = \{ \lambda v,\ \lambda \in \Rtot_{>0}\}$.
This is called the ``open'' half-line\index{half-line!open} in $V_\Rtot$ with  direction $v \in V_{\mathbb{Z}}$.
\end{Def}

\begin{Lem}\label{LemWellChosen}
For any basis $\Delta$ of $\Phi$ and any $\Delta_P \subset \Delta$, the $\mathbb{Z}$-vector face $ F_{\mathbb{Z}}^v(\Delta,\Delta_P)$ is non-empty.
\end{Lem}

\begin{proof}
The element $x= \sum_{\alpha \in \Delta \smallsetminus \Delta_P} \varpi_\alpha^\vee$ is in the lattive of coweights $V_\Z$.
For $\alpha \in \Delta$, we have that $\alpha(x) =0$ if $\alpha \in \Delta_P$ and $\alpha(x) =1$ otherwise.
Hence $x \in F_{\Z}^v(\Delta,\Delta_P)$ by definition.
\end{proof}

\begin{Lem}\label{LemHalfLineInSector}
For any $v \in F_\mathbb{Z}^v(\Delta,\Delta_P)$, we have $\delta_v \subset F_\Rtot^v(\Delta,\Delta_P)$.
\end{Lem}

\begin{proof}
For any $\alpha \in \Delta$, any $v \in C_{\mathbb{Z},\Delta}^v$ and any $\lambda \in \Rtot_{>0}$, we have $\alpha( \lambda v) = \alpha(v)\lambda $.
If $\alpha \in \Delta_P$, then $\alpha(v \lambda) = 0$.
If $\alpha \not\in \Delta_P$, then $\alpha(v) \in \mathbb{Z}_{> 0}$ by assumption on $v$, and since $\lambda > 0$, we deduce that $\alpha(v) \lambda > 0$.
Thus $\alpha(\lambda v) > 0$ for any $\alpha \in \Delta \setminus \Delta_P$.
Therefore $\lambda v \in F_\Rtot^v(\Delta,\Delta_P)$ for any $\lambda\in \Rtot_{>0}$.
\end{proof}

\begin{Rq}
Note that the set $\{ \lambda x,\ \lambda \in \Rtot_{>0} \}$ for $x \in C_{\Rtot,\Delta}^v$ is not contained in $C_{\Rtot,\Delta}^v$ in general, even in the case when $\Rtot$ is a ring and $x$ is $\Rtot$-torsion-free.
For instance, take $\Rtot = \mathbb{R}\llbracket t\rrbracket /(t^2)$ with the lexicographical order $a_{1} + t b_{1} < a_{2} + t b_{2} \Longleftrightarrow a_1 < a_2 \text{ or } a_1 = a_2 \text{ and } b_1 < b_2$.
Take $\Phi = \{\pm \alpha, \pm \beta, \pm (\alpha+\beta)\}$ of type $A_2$.
Take $x = 2 \alpha^\vee + (1+t) \beta^\vee$.
Then $\alpha( x ) = 4 - (1+t) = 3-t > 0$ and $\beta( x ) = 2(1+t) - 2 = 2t > 0$, so that $x \in C_{\Rtot,\{\alpha,\beta\}}^v$.
Thus for any $\lambda \in \Rtot \setminus \{0\}$, we have $\alpha(\lambda x) = (3-t) \lambda \neq 0$ since $3-t$ is invertible in $\Rtot$ so that, in particular, $x$ is torsion-free.
But $\beta(t x) =  t (2t) = 0$ so that $t x \not\in C_{\Rtot,\{\alpha,\beta\}}^v$ with $t > 0$.
\end{Rq}

The following Proposition generalizes \cite[7.3.5]{BruhatTits1}. Since the topology of the $\Rtot$-module $V_\Rtot$ is not easy to manipulate (for instance, it is not necessarily a connected space so that vector chambers cannot be defined as some connected components), we prove it in a combinatorial way instead of a topological way.

\begin{Prop}\label{PropGoodSector}
Set $\Rtot_\Q:= R \otimes \Q$ and assume that $\Rtot = \Rtot_\Q$.
Let $x \in V_\Rtot$ and $\Delta,\Delta'$ be two bases of $\Phi$.
There exists a unique $w \in W(\Phi)$ such that $x \in \overline{C}_{w(\Delta)}^v$ and $C^v_{w(\Delta)} \cap (x+C^v_{\Delta'}) \neq \emptyset$.
\end{Prop}

\begin{proof}
Since any basis of a root system is contained in $\Phi_{\mathrm{nd}}$, we can assume that $\Phi$ is reduced (i.e. $\Phi = \Phi_{\mathrm{nd}}$).

For  the  uniqueness, consider $w,w' \in W(\Phi)$ such that
$x \in \overline{C}^v_{w(\Delta)} \cap \overline{C}^v_{w'(\Delta)}$ and the intersections 
$C^v_{w(\Delta)} \cap (x + C^v_{\Delta'})$ and 
$C^v_{w'(\Delta)} \cap (x + C^v_{\Delta'}) $ are both non-empty.
Let $y \in C^v_{w(\Delta)} \cap (x + C^v_{\Delta'})$ and $z \in C^v_{w'(\Delta)} \cap (x + C^v_{\Delta'})$.
Let $\alpha \in \Phi^+_{w(\Delta)}$. Then $\alpha(y) > 0$ and $\alpha(x) \geqslant 0$ by Lemma~\ref{LemIntersectionChambers}.
By contradiction, assume that $\alpha \not\in \Phi^+_{w'(\Delta)}$. Then $\alpha(z) < 0$ and $\alpha(x) \leqslant 0$.
Thus $\alpha(x) = 0$ so that $\alpha(y-x) > 0$ and $\alpha(z-x) < 0$ by linearity of $\alpha$.
Hence the sign of $\alpha$ is non-constant on $C^v_{\Delta'}$ which is a contradiction.
We deduce that $\Phi^+_{w(\Delta)} \subseteq \Phi^+_{w'(\Delta)}$. The other inclusion can be proved in the same way, and hence $\Phi^+_{w(\Delta)} = \Phi^+_{w'(\Delta)}$. Thus $w(\Delta) = w'(\Delta)$, so that $w = w'$ according to \cite[VI.1.5 Thm.2]{bourbaki1981elements}.

For existence, we proceed as follows:
let $v \in C^v_{\mathbb{Z},\Delta'}$ so that $\delta_v \subset C^v_{\Rtot,\Delta'}$ according to Lemma~\ref{LemHalfLineInSector} and $\alpha(v) \in \mathbb{Z} \setminus \{0\}$ for any $\alpha \in \Phi$.
We denote by $\Psi$ the set of roots $\alpha \in \Phi$ such that $\alpha(x) < 0$ or $\forall \eta \in \Rtot_{>0},\ \exists \varepsilon \in ]0,\eta[,\ \alpha(x+\varepsilon v) \leqslant 0$.
We denote by $n(\Delta)$ the cardinality of $\Psi \cap \Phi_\Delta^+$.
We prove, by induction on $n$, that for any basis $\Delta$ of $\Phi$ such that $n(\Delta) \leqslant n$, there exists $w\in W(\Phi)$ such that $x \in \overline{C}^v_{w(\Delta)}$ and $C^v_{w(\Delta)} \cap (x + C^v_{\Delta'}) \neq \emptyset$.

Base case: suppose that $\Delta \cap \Psi = \emptyset$.
For any $\alpha \in \Delta$, we have $\alpha(x) \geqslant 0$ and: $$\exists \eta_\alpha >0,\ \forall \varepsilon \in ]0,\eta_\alpha[,\ \alpha(x+\varepsilon v) > 0.$$
Denote by $\eta = \min \{\eta_\alpha,\ \alpha \in \Delta\} > 0$.
Since any $\beta \in \Phi_\Delta^+$ is a linear combination with non-negative integer coefficients of the simple roots $\alpha \in \Delta$, we have $\beta(x) \geqslant 0$ and $\beta(x+\varepsilon v) > 0$ for any $\varepsilon \in ]0,\eta[$.
Hence $\Psi \cap \Phi_\Delta^+ = \emptyset$, so that $n(\Delta) = 0$, and $x \in \overline{C}^v_{\Delta}$.
Moreover, if we set $y = x + \varepsilon v$ for some $\varepsilon \in ]0,\eta[$, we have $y \in C^v_{\Delta} \cap (x + \delta_v) \subset C^v_\Delta \cap (x + C^v_{\Delta'})$ according to Lemma~\ref{LemHalfLineInSector}.
In particular, if $\Delta$ is any basis of $\Phi$ such that $n(\Delta) = 0$, we have, in particular, $\Delta \cap \Psi = \emptyset$ and we have seen that the element $w = \operatorname{id}$ provides the basis of the induction.

Induction step: Assume now that $n(\Delta) = n > 0$.
In this case, $\Delta \cap \Psi \neq \emptyset$, and hence we can choose $\alpha \in \Delta \cap \Psi$. By definition of $\Psi$, we have 
\[\alpha(x) < 0 \qquad \text{ or }\qquad \forall \eta > 0,\ \exists \varepsilon \in ]0,\eta[,\ \alpha(x+\varepsilon v) \leqslant 0.\]
Let us first prove that $-\alpha \not\in \Psi$.

First case: suppose that $\alpha(x) < 0$.
Then $-\alpha(x) > 0$.

\textbullet\ If $\alpha(v) \in \mathbb{Z}_{<0}$, let $\eta$ be any element in $\Rtot_{>0}$.
Then for any $\varepsilon \in ]0,\eta[$, we have $-\alpha(x+ \varepsilon v) > -\alpha(v) \varepsilon > 0$.

\textbullet\ If $\alpha(v) \in \mathbb{Z}_{>0}$, let $\eta = - \frac{1}{\alpha(v)} \alpha(x) > 0$.
Then for any $\varepsilon \in ]0,\eta[$, we have 
\[-\alpha(x + \varepsilon v) = - \alpha(x) - \alpha(v) \varepsilon > -\alpha(x) -\alpha(v) \eta = 0.\]
Hence, in both cases, we get $-\alpha \not\in \Psi$.

Second case: suppose that for any $\eta > 0$ there exists $\rho \in ]0,\eta[$ such that $\alpha(x+\rho v) \leqslant 0$.

\textbullet\ If $\alpha(v) \in \mathbb{Z}_{>0}$, then $-\alpha(x) \geqslant \alpha(v)\rho >0$.
Moreover, for any $\varepsilon \in ]0,\rho[$, we have $-\alpha(x+ \varepsilon v) \geqslant \alpha(v) (\rho - \varepsilon)  > 0$.
Hence $-\alpha \not\in \Psi$.

\textbullet\ If $\alpha(v) \in \mathbb{Z}_{<0}$, suppose by contradiction that $-\alpha(x) < 0$.
Then $-\alpha(x) \geqslant \alpha(v) \rho > \alpha(v) \eta$.
Thus, for $\eta = -\frac{1}{2\alpha(v)} \alpha(x) > 0$, we get $-\alpha(x) > -\frac{1}{2} \alpha(x)$ which is a contradiction.
Hence $-\alpha(x) \geqslant 0$.
Now, for any $\varepsilon \in ]0,\rho[$, we have $-\alpha(x + \varepsilon v) \geqslant -\alpha(v) \varepsilon > 0$.
Thus $-\alpha \not\in \Psi$.

As a consequence, in all cases, we get $-\alpha \not\in \Psi$.
According to \cite[VI.1.6 Cor. 1 of Prop. 17]{bourbaki1981elements}, we know that $r_\alpha$ stabilizes $\Phi_\Delta^+ \setminus \{\alpha\} = \Phi_{r_\alpha(\Delta)}^+ \setminus\{-\alpha\}$.
Thus: $$ \Phi_{r_\alpha(\Delta)}^+ \cap \Psi  = \left( \Phi_\Delta^+ \setminus \{\alpha\} \right) \cap \Psi,$$ so that $n(r_\alpha(\Delta)) = n(\Delta)-1$.
By induction, since $r_\alpha(\Delta)$ is a basis of $\Phi$, we know that there exists $w \in W(\Phi)$ such that $x \in \overline{C}_{w \circ r_\alpha(\Delta)}^v$  and $C^v_{w \circ r_\alpha(\Delta)} \cap (x + C^v_{\Delta'}) \neq \emptyset$.
\end{proof}

\begin{Cor}\label{corWeylVectorial}
Let $\Delta$ be a basis of $\Phi$. For any $\Delta_P \subset \Delta$ and any $w \in W(\Phi)$, we have $w \cdot F^v_\Rtot(\Delta,\Delta_P) =F^v_\Rtot(w(\Delta),w(\Delta_P))$.
Moreover $V_\Rtot=\bigcup_{w\in W(\Phi)} w \cdot \overline{C}^v_\Delta$.
\end{Cor}

\begin{proof}
For any $\alpha \in \Delta$ and any $x \in V_\Rtot$, we have
\[(w \cdot \alpha) \big( w \cdot x \big) =  \alpha \big(w^{-1} w \cdot x\big) = \alpha(x) \]
since $w \cdot \alpha = \alpha \circ w^{-1}$.
Thus $x \in F^v_\Rtot(\Delta,\Delta_P) \Longleftrightarrow w(x) \in w \cdot F^v_\Rtot(\Delta,\Delta_P) \Longleftrightarrow  w(x) \in F^v_\Rtot(w(\Delta),w(\Delta_P))$.

By Proposition~\ref{PropGoodSector}, for any $x \in V_\Rtot$, there exists $w \in W^v$ such that $x \in \overline{C}_{\Rtot,w(\Delta)}^v = w \cdot \overline{C}^v_\Rtot$. Hence we get the second equality.
\end{proof}

\begin{lemma}\label{lemSign_of_a_vectorial_face}
Let $\beta\in \Phi$ and $F^v$ be a vector face of $V_\Rtot$. Then either $\beta(F^v) \subset \Rtot_{>0}$ or $\beta(F^v)=\{0\}$ or $\beta(F^v)\subset \Rtot_{<0}$. 
\end{lemma}

\begin{proof}
Write $F^v=w \cdot F^v_\Rtot(\Delta_P)$, where $\Delta_P\subset \Delta_f$ and $w\in W^v$. Let $\beta'=w^{-1}.\beta$.
Then $\beta'\big(F^v_R(\Delta_P)\big)=\beta(F^v)$.
Let $\Phi^+$ be the set of positive roots defined from $\Delta_f$.
Up to replacing $\beta'$ by $-\beta'$, we may assume that $\beta'\in \Phi^+$.
Write $\beta'=\sum_{\alpha\in \Delta_f} n_\alpha\alpha$, with $n_\alpha\in \N$ for all $\alpha\in \Delta_f$.
Then if $\{\alpha\in \Delta_f\mid n_\alpha\neq 0\}\subset \Delta_P$, $\beta'(F^v)=\{0\}$ and else $\beta'(F^v) \subset \Rtot_{>0}$. The lemma follows.
\end{proof}

\begin{lemma}\label{l_neighbourhood}
Let $x\in V_\Rtot$. Let  $\mathscr{C}_x$ be the set of vector chambers containing $x$ in their closures. Then $\bigcup_{C\in \mathscr{C}_x} \overline{C}$ contains a neighbourhood of $x$.
\end{lemma}

\begin{proof}
    Let $\mathscr{C}'$ be the set of vector chambers of $V_{\Rtot}$ which do not belong to $\mathscr{C}_x$. If $C\in \mathscr{C}'$, we have $x\in V_{\Rtot}\setminus \overline{C}$ and $V_\Rtot\setminus \overline{C}$ is open. By Corollary~\ref{corWeylVectorial}, we have: \[ V_{\Rtot}\setminus \bigcup_{C'\in \mathscr{C}'} \overline{C'} = \left(\bigcup_{C\in \mathscr{C}_x\cup \mathscr{C}'} \overline{C}\right)\setminus \bigcup_{C'\in \mathscr{C}'} \overline{C'}=\bigcup_{C\in \mathscr{C}_x\cup\mathscr{C'} } \left(\overline{C}\setminus \bigcup_{C'\in \mathscr{C}'} \overline{C'}\right)=\left(\bigcup_{C\in \mathscr{C}_x} \overline{C}\right)\setminus \bigcup_{C'\in \mathscr{C}'} \overline{C'}.\] Therefore $\bigcup_{C\in \mathscr{C}_x}\overline{C}\supset V_{\Rtot}\setminus\bigcup_{C\in\mathscr{C}'} \overline{C},$ which is an open set containing $x$.
\end{proof}

Given a vector face $F^v$, we set: 
\begin{gather*}
\Phi^+_{F^v} := \{\beta \in \Phi,\ \beta(F^v) \subset \Rtot_{>0}\},\\ \Phi^-_{F^v} := \{\beta \in \Phi,\ \beta(F^v) \subset \Rtot_{<0}\},\\ \Phi^0_{F^v} := \{\beta \in \Phi,\ \beta(F^v) = \{0\}\}.
\end{gather*}
\index[notation]{p@$\Phi^+_{F}$}\index[notation]{p@$\Phi^-_{F}$}\index[notation]{p@$\Phi^0_{F}$}
\vspace{3ex}

\textbf{b) The affine apartment as an $R$-aff space}\\

Let's now move on to affine geometry and to $R$-aff spaces. In section \ref{raff}, $R$-aff spaces were defined when $R$ was a totally ordered real vector space. In order to be able to argue over other rings than the real numbers, we are going to consider here the slightly more general situation where $R$ is just a totally ordered abelian group. One can define in this setting a
category $\mathrm{Aff}_R$\index{R-aff@$R$-aff!category} of \textbf{$R$-aff spaces}\index{R-aff@$R$-aff!space} in exactly the same way as in section \ref{raff}:
\begin{itemize}
\item[-] Objects: an $R$-aff space is a pair $(V,A)$ where $V$ is a free $\mathbb{Z}$-module of finite type and $A$ is an affine space with underlying $\mathbb{Z}$-module $\overrightarrow{A}=V_R:=V \otimes_{\mathbb{Z}} R$;
\item[-] Morphisms: if $(V,A)$ and $(V',A')$ are two $R$-aff spaces, an $R$-aff map from $ (V,A)$ to $(V',A')$ is an affine map $f: A \rightarrow A'$ whose linear part $\overrightarrow{f}$ is induced by tensorization from a $\mathbb{Z}$-linear map $\overrightarrow{f}_{\mathrm{core}}: V \rightarrow V'$.
\end{itemize}
By abuse of notation, we will sometimes say that $A$ is an $R$-aff space with underlying $\mathbb{Z}$-module $V_R$ to mean that the pair $(V,A)$ is an $R$-aff space. 

Given an $R$-aff space $(V,A)$, we define the \textbf{aff-group}\index{aff!group} of $A$ by:
\[ \operatorname{Aff}_{R}(A) := \mathrm{Aut}_{\mathrm{Aff}_R}(V,A). \]
\index[notation]{a@$\operatorname{Aff}_{R}(A)$}
By introducing the action of $\mathrm{GL}(V)$ on $V_R$ induced by the embedding:
\begin{align*}
\mathrm{GL}(V) & \hookrightarrow \mathrm{GL}(V_R)\\
g & \mapsto g \otimes \mathrm{Id}_R
\end{align*}
and by fixing an origin $o$ of $A$, we can explicitly describe the aff-group as a semi-direct product:
\[ \operatorname{Aff}_R(A) \cong V_R \rtimes \mathrm{GL}(V).\]
Its action on $A$ is then given by the formula:
\[ (v,g) \cdot x = o + g(x-o) + v\qquad \forall (v,g) \in V_\Rtot \rtimes \mathrm{GL}(V),\ \forall x \in A.\]
The image of an element $f \in \operatorname{Aff}_R(A)$ in $\mathrm{GL}(V)$ is none other than $\overrightarrow{f}_{\mathrm{core}}$. By abuse of notation, we will denote it $\overrightarrow{f}$ in the sequel.

\begin{Not}
Let $\mathbb{A}_\Rtot$ be an $\Rtot$-aff space with some origin $o$ and underlying $\mathbb{Z}$-module $V_\Rtot$ defined as in~\ref{subsubTopology_totally_ordered_ring}.
Note that $V_\Rtot=0$ and $\mathbb{A}_\Rtot = \{o\}$ when $\Phi = \emptyset$ so that any action of any group on $\mathbb{A}_\Rtot$ is trivial.
In the rest of this section, we assume that $\Phi \neq \emptyset$ but any result can obviously be extended to the case of an empty root system.

Any root $\alpha\in \Phi$ induces canonically a continuous map $V_\Rtot \rightarrow \Rtot$ so that for any $\alpha \in \Phi$ and any $\lambda\in \Rtot$, one can define:
\begin{itemize}
\item an $R$-aff map\index{R-aff@$R$-aff!map} $\theta_{\alpha,\lambda}: \mathbb{A}_\Rtot \to \Rtot$ by  $\theta_{\alpha,\lambda}(x) = \alpha(x-o) +\lambda$;\index[notation]{t@$\theta_{\alpha,\lambda}$}
\item an "aff hyperplane"\index{aff!hyperplane} $H_{\alpha,\lambda} = \theta_{\alpha,\lambda}^{-1}(\{0\})$;\index[notation]{h@$H_{\alpha,\lambda}$}
\item an "open aff half-space"\index{aff!half-space} $\mathring{D}_{\alpha,\lambda} = \theta_{\alpha,\lambda}^{-1}(\Rtot_{>0})$ (resp. "closed aff half-space" $D_{\alpha,\lambda} = \theta_{\alpha,\lambda}^{-1}(\Rtot_{\geqslant 0})$).\index[notation]{d@$\mathring{D}_{\alpha,\lambda}$}
\end{itemize} 

By abuse of notation, for any root $\alpha \in \Phi$ and any point $x \in \mathbb{A}_\Rtot$, we will often denote $\alpha(x)$ instead of $\alpha(x-o)$.

We denote, by convention, $D_{\alpha,\infty} = \mathring{D}_{\alpha, \infty} = \mathbb{A}_\Rtot$ for any $\alpha \in \Phi$ so that $x \in D_{\alpha,\infty} \Longleftrightarrow \infty \geqslant -\alpha(x)$ extends the definition of the $D_{\alpha,\lambda} = \{x \in \mathbb{A}_\Rtot,\ \lambda \geqslant -\alpha(x) \}$ to any $\lambda \in \Rtot \cup \{\infty\}$.
\end{Not}

For $(\alpha,\lambda) \in \Phi \times \Rtot$, the element $\theta_{\alpha,\lambda}(x)\alpha^\vee$ belongs to the $\mathbb{Z}$-module $V_\Rtot$ for any $x \in \mathbb{A}_\Rtot$.
Thus, we define a map $r_{\alpha,\lambda}: \mathbb{A}_\Rtot \to \mathbb{A}_\Rtot$  by $r_{\alpha,\lambda}(x) = x - \theta_{\alpha,\lambda}(x) \alpha^\vee$.\index[notation]{r@$r_{\alpha,\lambda}$}
It is an $R$-aff map:

\begin{Fact}\label{FactReflectionIdentification}
For any $\alpha \in \Phi$ and any $\lambda \in \Rtot$,
the element $r_{\alpha,\lambda} \in \operatorname{Aff}_{R}(\mathbb{A}_R)$ is identified with the pair $(- \lambda \alpha^\vee,r_\alpha) \in V_\Rtot \rtimes \mathrm{GL}(V_\Z)$.
In particular, $\overrightarrow{r_{\alpha,\lambda}} = r_\alpha \in \mathrm{GL}(V_\Z)$.
\end{Fact}

\begin{proof}
For $x \in \mathbb{A}_R$, we have 
\begin{align*}
 (-\lambda \alpha^\vee,r_\alpha)(x) &= o + r_\alpha(x-o) - \lambda \alpha^\vee\\
 &= o + (x-o) - \alpha(x-o) \alpha^\vee - \lambda \alpha^\vee\\
 &= x - \theta_{\alpha,\lambda}(x) \alpha^\vee\\
 &= r_{\alpha,\lambda}(x)
\end{align*}
We get the identification by faithfulness of the action of $\mathrm{Aff}_{R}(\mathbb{A}_{R})$ on $\mathbb{A}_R$.
\end{proof}

\begin{Def}
The $R$-aff map $r_{\alpha,\lambda}$ is called an "aff-reflection" with respect to the aff-hyperplane $H_{\alpha,\lambda}$\index{aff!reflection} because it satisfies:
\begin{itemize}
\item $r_{\alpha,\lambda}^2 = \operatorname{id}_{\mathbb{A}_\Rtot}$;
\item $r_{\alpha,\lambda}(x) = x \Longleftrightarrow x \in H_{\alpha,\lambda}$;
\item $r_{\alpha,\lambda}(D_{\alpha,\lambda})  =D_{-\alpha,-\lambda}$.
\end{itemize}
\end{Def}

\begin{proof}
The second statement is immediate from the formula since $\lambda \alpha^\vee = 0 \Longleftrightarrow \lambda = 0$ for any $\lambda  \in R$.

Let $y = r_{\alpha,\lambda}(x) = x - \theta_{\alpha,\lambda}(x) \alpha^\vee$.
Then
\begin{align*}
\theta_{\alpha,\lambda}(y) =& \alpha(y) + \lambda\\
=& \alpha(x) - \alpha(\alpha^\vee) \theta_{\alpha,\lambda}(x) + \lambda\\
=& \Big(\alpha(x) + \lambda\Big) - 2 \theta_{\alpha,\lambda}(x)\\
=& - \theta_{\alpha,\lambda}(x)
\end{align*}
Thus we get the first statement:
\begin{align*}
r_{\alpha,\lambda}(y) = & y- \theta_{\alpha,\lambda}(y) \alpha^\vee\\
=&y + \theta_{\alpha,\lambda}(x)\alpha^\vee\\
=&x
\end{align*}

Finally, the third statement is given by the following Fact~\ref{FactActionOfReflection}.
\end{proof}

\begin{Fact}\label{FactActionOfReflection}
For any $(\alpha,\lambda),(\beta,\mu) \in \Phi \times \Rtot$, we have:
\[ r_{\alpha,\lambda}(\mathring{D}_{\beta,\mu}) = \mathring{D}_{\gamma,\rho}
\qquad \text{ and } \qquad
r_{\alpha,\lambda}(H_{\beta,\mu}) = H_{\gamma,\rho}
\]
for $(\gamma,\rho) \in \Phi \times \Rtot$ such that:
\begin{align*}
\gamma = r_\alpha(\beta) & = \beta - \beta(\alpha^\vee) \alpha &
\rho = \mu - \beta(\alpha^\vee) \lambda
\end{align*}
\end{Fact}

\begin{proof}
These are exactly the same results as in \cite[VI§2]{bourbaki1981elements}.
We recall that $W^v$ is generated by the $r_\alpha: \Phi \to \Phi$ for $\alpha \in \Phi$ given by $r_\alpha(\beta) = \beta - \beta(\alpha^\vee) \alpha$ so that $r_\alpha^2 = \operatorname{id}$.

We firstly prove that $r_{\alpha,\lambda}(\mathring{D}_{\beta,\mu}) \subset \mathring{D}_{\gamma,\rho}$
when $\gamma = r_\alpha(\beta)$ and $\rho = \mu - \beta(\alpha^\vee) \lambda$.
Note that $\gamma(\alpha^\vee) = \beta(\alpha^\vee) - \beta(\alpha^\vee) \alpha(\alpha^\vee) = - \beta(\alpha^\vee)$.
Hence, if we take $x \in \mathring{D}_{\beta,\mu}$ and we set $y := r_{\alpha,\lambda}(x)$, then we have:
\begin{align*}
\gamma(y) = & \gamma(x) - \gamma(\alpha^\vee) \theta_{\alpha,\lambda}(x)\\
= & \Big( \beta(x) - \beta(\alpha^\vee) \alpha(x)  \Big) + \beta(\alpha^\vee) \theta_{\alpha,\lambda}(x) \\
=& \beta(x) - \beta(\alpha^\vee) \alpha(x) + \beta(\alpha^\vee) \Big(\alpha(x) + \lambda\Big)\\
=& \beta(x) + \beta(\alpha^\vee) \lambda\\
> & -\mu + \beta(\alpha^\vee) \lambda = - \rho
\end{align*}
Thus $y \in \mathring{D}_{\gamma,\rho}$.

Conversely, choose $y \in \mathring{D}_{\gamma,\rho}$ and set $x := r_{\alpha,\lambda}(y)$.
Then $y = r_{\alpha,\lambda}(x)$ since $r_{\alpha,\lambda}^2 = \operatorname{id}$.
Thus $x \in r_{\alpha,\lambda}(\mathring{D}_{\gamma,\rho}) \subset \mathring{D}_{r_\alpha(\gamma),\rho -\gamma(\alpha^\vee)\lambda}$.
Since $r_\alpha(\gamma) =  r_\alpha^2(\beta) = \beta$ and
$\rho - \gamma(\alpha^\vee)\lambda = \rho + \beta(\alpha^\vee) \lambda= \mu$,
we get that $r_{\alpha,\lambda}(\mathring{D}_{\gamma,\rho}) \subset \mathring{D}_{\beta,\mu}$.
Thus $\mathring{D}_{\gamma,\rho} = r_{\alpha,\lambda} ( \mathring{D}_{\beta,\mu} )$.

The same equality holds for affine hyperplanes.
\end{proof}

\begin{definition}\label{defApartment}
An \textbf{affine apartment}\index{apartment!affine} over $\Rtot$ is a $5$-tuple

$\underline{\A_\Rtot}=\big(\A_\Rtot,V_\R,\Phi, (\Gamma_\alpha)_{\alpha\in\Phi}, \widehat{W}\big)$ such that:
\begin{enumerate}
\item $V_\R$ is a finite dimensional $\R$-vector space;

\item $\Phi$ is a root system over $(V_\R)^*$;

\item if $V_\Z$ denotes the lattice of coweights of $\Phi$ in $V_\R$, the pair $(V_\Z,\A_\Rtot)$ is an $R$-aff space;

\item\label{itCondition_Gamma_alpha}  $(\Gamma_\alpha)_{\alpha\in\Phi}$ is a family of unbounded subsets of $\Rtot$ containing $0$, satisfying the following property.
Let $\mathscr{H}=\{H_{\alpha,\lambda}\mid \alpha\in \Phi,\lambda\in \Gamma_\alpha\}$. For $H=H_{\alpha,\lambda}\in \mathscr{H}$ denote by $r_H$ the aff reflection $r_{\alpha,\lambda}$ of $\A_R$ fixing $H$ and whose vectorial part is $r_\alpha$.  Then $r_H$ stabilizes $\mathscr{H}$ for every $H\in \mathscr{H}$;

 \item $\widehat{W}$ is a subgroup of $W^v\ltimes \A_{\Rtot}$  containing the $r_{H}$, for $H\in \mathscr{H}$ and stabilizing $\mathscr{H}$. \index[notation]{w@$\widehat{W}$}
\end{enumerate} 
\end{definition}

When we will associate an apartment to a reductive group $G$, $\widehat{W}$ will be taken to be the extended affine Weyl group, which is  the group of the automorphisms of $\A_\Rtot$ induced by an element of $G$ stabilizing $\A_\Rtot$.

A \textbf{sector-face}\index{face!sector}\index{sector-face} $Q$ (resp. a \textbf{sector}\index{sector} $Q$) is a set of the form $x+F^v$ (resp. $x+C^v$) for some $x\in \A_\Rtot$ and some vector face $F^v$ (resp. vector chamber $C^v$) of $V_\Rtot$.
The \textbf{direction}\index{sector!direction} of $Q$ is $F^v$ (resp. $C^v$) and its \textbf{base point}\index{sector!base point} is $x$.

A set of the form $D_{\alpha,\lambda}$ (resp $\mathring{D}_{\alpha,\lambda}$, resp. $H_{\alpha,\lambda}$) for $\lambda\in \Gamma_\alpha$ and $\alpha\in \Phi$ is called a \textbf{half-apartment}\index{apartment!half} (resp. an \textbf{open half-apartment}, resp a \textbf{wall}\index{wall}) of $\A_\Rtot$.

A set of the form $D_{\alpha,\lambda}$\index[notation]{d@$D_{\alpha,\lambda}$} (resp $\mathring{D}_{\alpha,\lambda}$, resp. $H_{\alpha,\lambda}$\index[notation]{h@$H_{\alpha,\lambda}$}) for $\lambda\in \Rtot$ and $\alpha\in \Phi$ is called a \textbf{phantom half-apartment}\index{apartment!phantom half} (resp. a \textbf{phantom open half-apartment}, resp a \textbf{phantom wall}\index{wall!phantom}) of $\A_\Rtot$. This phantom part terminology is there to indicate that the objects are carried by the directions induced by the roots but not by the values prescribed by the sets $\Gamma_\alpha$. When in section~\ref{SecQuasiSplitGroups} we make a use of algebraic groups over a valued field, these phantom parts may appear after some extension of the base field.
 
The \textbf{affine weyl group}\index{Weyl group!affine}  $\Waff$\index[notation]{w@$\Waff$} of $\A_\Rtot$ is the subgroup of $\mathrm{Aff}_R(\A_\Rtot)$ generated by the $r_H$, for $H\in\mathscr{H}$.
By Fact~\ref{FactReflectionIdentification}, it is a subgroup of $V_\Rtot\rtimes W^v$.
By condition~(\ref{itCondition_Gamma_alpha}) and Fact~\ref{FactActionOfReflection}, if $\alpha\in \Phi$ and $\lambda \in \Gamma_\alpha$, then $r_\alpha(H_{\alpha,\lambda})=H_{-\alpha,\lambda}=H_{\alpha,-\lambda}$ and thus $-\Gamma_\alpha=\Gamma_\alpha$.
Since $W^v$ is generated by the $r_\alpha$, applying Fact~\ref{FactActionOfReflection}, we get that for any $\alpha\in \Phi$, $w\in W^v$ and $\lambda\in \Gamma_\alpha$, we have $w \cdot H_{\alpha,\lambda}=H_{w \cdot \alpha,\lambda}$.
Thus $\Gamma_\alpha=\Gamma_{w \cdot \alpha}$.

\subsubsection{Local combinatorial structure of the affine apartment}
\label{SubsecGerms}

In the classical Bruhat-Tits theory, when one studies a (quasi-split) reductive group $\mathbf{G}$ over a field $\mathbb{K}$ endowed with a valuation $\omega: \mathbb{K} \rightarrow \mathbb{Z} \cup \{\infty\}$, the local combinatorial structure of the apartments of the building associated to $\mathbf{G}$ is usually encoded thanks to the notion of a chamber. Things become more complicated when one works over a field $\mathbb{K}$ that is endowed with a non-discrete valuation, for instance with a surjective valuation $\omega: \mathbb{K} \rightarrow \Q \cup \{\infty\}$. In that case, the notion of a chamber makes no sense as a set anymore and the local combinatorial structure is usually encoded thanks to filters. \\

In this paragraph, we would like to define such local combinatorial structures on the affine apartment that has been introduced in the previous section. For that purpose, we fix a totally ordered abelian group $R$ as well as an affine apartment  $\underline{\A_\Rtot}=\big(\A_\Rtot,V_\R,\Phi, (\Gamma_\alpha)_{\alpha\in\Phi},\widehat{W}\big)$ over $\Rtot$.
Let $V_\Z$ be the lattice of coweights in $V_\R$ of $\Phi$ and set $V_R := V_\Z \otimes_{\mathbb{Z}} R$.

\vspace{3ex}

\textbf{a) Filters}\\

A filter\index{filter} on a set $\mathcal{E}$ is a nonempty set $\FCC$ of nonempty subsets of $\mathcal{E}$ such that, for all subsets $E$, $E'$ of $\mathcal{E}$, one has:\begin{itemize}
\item $E$, $E'\in \FCC$ implies $E\cap E'\in \FCC$ 

\item $E'\subset E$ and $E'\in \FCC$ implies $E\in \FCC$.

\end{itemize}

If $\mathcal{E}$ is a set and $\FCC,\FCC'$ are filters on $\mathcal{E}$, we define $\FCC\Cup \FCC'$\index[notation]{Z@$\Cup$} to be the filter $\{E\cup E'\mid  (E,E')\in \FCC\times \FCC'\}$.

Let $\mathcal{E},\mathcal{E}'$ be sets, $f:\mathcal{E}\rightarrow\mathcal{E}'$ be a map and $\FCC$ be a filter on $\mathcal{E}$. Then $f(\FCC):=\{E'\subset \mathcal{E}'\mid \exists E\in \FCC, E'\supset f(E)\}$ is  a filter on $\mathcal{E}'$. We say that a map fixes a  filter if it fixes at least one element of this  filter.

If $\FCC$ is a filter on a set $\mathcal{E}$, and $E$ is a subset of $\mathcal{E}$, one says that $\FCC$ contains $E$ if every element of $\FCC$ contains $E$. We denote it $\FCC\Supset E$\index[notation]{Z@$\Supset$}. If $E$ is nonempty, the \textbf{principal filter}\index{filter!principal} on $\EC$ associated with $E$ is the filter  $\FCC_{E,\EC}$\index[notation]{f@$\FCC_{E,\EC}$} of subsets of $\EC$ containing $E$. 

 A filter $\FCC$ is said to be contained in another filter $\FCC'$ (resp. in a subset $Z$ in $\mathcal{E}$) if every set in $\FCC'$  is in $\FCC$ (resp. if $Z\in \FCC$). We then denote $\FCC\Subset  \FCC'$\index[notation]{Z@$\Subset$} (resp. $\FCC\Subset Z$).
 
 These definitions of containment\index{filter!containment} are inspired by the following facts. Let $\EC$ be a set, $\FCC$ be a filter on $\EC$ and $E,E'\subset \EC$. Then :\begin{itemize}
\item $E\subset E'$ if and only if $\FCC_{E,\EC}\Subset \FCC_{E',\EC}$,

\item $E\Subset \FCC$ if and only if $\FCC_{E,\EC}\Subset \FCC$,

\item $E\Supset \FCC$ if and only if $\FCC_{E,\EC}\Supset \FCC$.

\end{itemize}

\vspace{3ex}

\textbf{b) The enclosure of a filter}\\

We now define the enclosure of a filter on $\A_\Rtot$.
The main motivation to introduce this object is the fact that in a building, the intersection of two apartments $A$ and $B$ is a finite intersection of half-apartments of $A$, see Axiom~\ref{axiomA2} in Definition~\ref{defBuildings}.

\begin{definition}\label{defEnclosure}
Let $\underline{\A_\Rtot} = \left( \mathbb{A}_\Rtot, V_\R, \Phi, \left(\Gamma_\alpha\right)_{\alpha\in \Phi},\widehat{W}\right)$ be an apartment.
Let $\VC$ be a filter on $\A_\Rtot$.
Then we define the \textbf{enclosure}\index{filter!enclosure} $\cl(\VC)$ of $\VC$ as:
\[\cl(\VC)=\{X\subset \A_\Rtot\mid \exists (\lambda_\alpha)\in \prod_{\alpha\in \Phi} (\Gamma_\alpha\cup \{\infty\})\mid \ X\supset \bigcap_{\alpha\in \Phi} D_{\alpha,\lambda_\alpha}\Supset \VC\}.\]
If $\Omega$ is a subset of $\A_\Rtot$, we set $\cl(\Omega)=\cl(\FCC_{\Omega,\A_\Rtot})$.
A subset $\Omega$ of $\A_\Rtot$ is said to be \textbf{enclosed}\index{enclosed subset} if it is an element of $\cl(\Omega)$, that is, if it is a finite intersection of half-apartments.
\end{definition}

Our definition of the enclosure is inspired by \cite[2.2.2]{gaussent2008kac}, but it differs from the enclosure $\cl_{\mathrm{BT}}$ defined in \cite[7.1.2]{ BruhatTits1}. Indeed:
\begin{itemize}
\item If $\Omega\subset \A_\Rtot$, then $\cl(\Omega)$ is a filter whereas $\cl_{\mathrm{BT}}(\Omega)$ is a set.
But even if we identify a set with the associated principal filter, the notions differ.
Indeed, suppose for example that  $\A=\R$ and that $\Gamma_\alpha=\Q$, for all $\alpha\in \Phi$. Let $x\in \R\setminus \Q$.
Then $\cl(\{x\})$ is the set of subsets of $\R$ containing a neighborhood of $x$, whereas $\cl_{\mathrm{BT}}(\{x\})=\{x\}$ and
$\FCC_{\{x\},\A_\Rtot}\neq \cl(\{x\})$.

\item In fact, one has $\cl(\Omega)\Supset \cl_{\mathrm{BT}}(\Omega)$ for every subset $\Omega$ of $\A_\Rtot$: our enclosure is therefore bigger.
\end{itemize}

\begin{Ex}
Suppose that $\A_R$ is associated with a split reductive group over a field $\mathbb{K}$ equipped with a valuation $\omega:\mathbb{K}\rightarrow \Lambda\cup\{\infty\}$. Then one has $\Gamma_\alpha=\Lambda$ for all $\alpha\in \Phi$. Suppose that $\omega$ is discrete, that is $\Lambda=\Z$ (up to renormalization) and set $R=\R$. Then if $\Omega$ is a subset of $\A_R$, $\cl(\Omega)$ is the principal filter on $\A_\R$ associated to $\cl_{\mathrm{BT}}(\Omega)$ and thus we can avoid the use of filters. When $\Lambda$ is not contained in $\R$ however the enclosure of a set is not necessarily a principal filter, even if $\Lambda$ is discrete. Indeed, set $\Lambda=\Z^2$, $V_\R=\R$, $R=\R^2$  and $\A_R=\R\otimes V_\R=\R^2$, where $R$ is equipped with the lexicographical order. Consider $\Omega=\{0\}\times \R\subset \A_R$. Then \[\cl(\Omega)=\{X\subset \R^2\mid \exists a,b\in \R,\ X\supset [(-1,a),(1,b)]_{\R^2}\},\] and this filter is not principal.
\end{Ex}

\vspace{3ex}

\textbf{c) Germs, faces and local faces}\\

Let $x$ be an element in $ \A_\Rtot$ and let $F^v$ be a vector face in $V_\Rtot$. Consider the sector-face $Q=x+F^v$.
The \textbf{germ of $Q$ at $x$}\index{germ} is the filter $\germ_x(Q)=\{\Omega\cap Q\mid \Omega\subset \A_R \text{ is a neighborhood of }x\}$.
A \textbf{local face}\index{face!local} (resp. a \textbf{local chamber}\index{chamber!local}) is a filter of the form $\germ_x(Q)$\index[notation]{g@$\germ_x(\cdot)$} for some sector-face $Q$ (resp. sector $Q$) based at $x$.
If $w_0$ is the longest element of $W^v$ and $Q_1 = x + F^v_1$ and $Q_2 = x+F_2^v$ are two sector-faces, we say that $Q_1$ and $Q_2$ have opposite directions or that the local chambers $\germ_x(Q_1)$ and $\germ_x(Q_2)$ are opposite if $F_v^2=w_0 \cdot F_v^1$.

This will be a useful notation for Bruhat decomposition~\ref{ThmBruhatDecomposition}. Note that $germ_x(Q)$ is denoted  $F^\ell(x,F^v)$ in \cite{rousseau2011masures}.

The \textbf{face}\index{face} $\FC_{x,F^v}$\index[notation]{f@$\FC_{x,F^v}$} associated to $x$ and $F^v$ is the filter on $\A_\Rtot$ generated by the sets of the form \[\XC=\bigcap_{\alpha\in \Psi} \mathring{D}_{\alpha, \lambda_\alpha}\cap \bigcap_{\alpha\in \Phi \setminus\Psi }D_{\alpha,\lambda_\alpha},\]

where $\Psi\subset  \Phi$ and  $(\lambda_\alpha)\in  (\Gamma_\alpha\cup\{\infty\})^\Phi$ such that $\XC \Supset (x+F^v)\cap \Omega$, for some neighbourhood $\Omega$ of $x$ in $\A_\Rtot$.

If $F^v = \{0\}$ and $\alpha(x) \in \Gamma_\alpha$ for any $\alpha \in \Phi$, then $\FC_{x,F^v}$ is called a \textbf{vertex}\index{vertex}. It is the principal filter associated to $\{x\}$.

The \textbf{germ of $Q$ at infinity}\index{germ!at infinity} is the filter $\germ_\infty(Q)=\{\Omega\subset  \A\mid \exists \xi\in F^v,\ \Omega\supset x+\xi+F^v\}$\index[notation]{g@$\germ_\infty(\cdot)$}.
If $Q$ is a sector, then a subset $\XC$ of $\A$ is in $\germ_\infty(Q)$ if and only if $\XC$ contains a subsector of $Q$.

\begin{lemma}\label{lemEnclosure_point_in_sector}
Let $Q$ be a sector of $\A_\Rtot$ with base point $x$ and let $y\in Q$. Write $Q=x+C^v$, where $C^v$ is a vector chamber of $\A_\Rtot$. Then $\cl(\{x,y\})\Supset \mathcal{F}_{x,C^v}\Supset \germ_x(Q)$  and $\cl(\germ_x(Q))=\cl(\FC_{x,C^v})$.
\end{lemma}

\begin{proof}
By definition, we have $\FC_{x,C^v}\Supset \germ_x(Q)$.

Let $C^v$ be the vector chamber of $\A_{\Rtot}$ such that $Q=x+C^v$. Let $\Delta$ be the basis of $\Phi$ associated to $C^v$. Let $\Omega\in \cl(\{x,y\})$. Let $(\lambda_\alpha)_{\alpha\in \Phi}\in (\Gamma_\alpha\cup \{\infty\})^\Phi$ be such that  $\Omega\supset\bigcap_{\alpha\in \Phi} D_{\alpha,\lambda_\alpha}\supset \{x,y\}$. Then for all $\alpha\in \Phi_{\Delta}^-$: $$\alpha(x)>\alpha(y)\geq -\lambda_\alpha$$ and for all $\alpha\in \Phi_{\Delta}^+$: $$\alpha(x)\geq -\lambda_\alpha.$$ Set $\XC_\Rtot=\bigcap_{\alpha\in \Phi_\Delta^-} \mathring{D}_{\alpha,\lambda_\alpha}$. Then $\XC_{\Rtot} \ni x$ and $\bigcap_{\alpha\in \Phi_{\Delta}^+}\mathring{D}_{\alpha,\lambda_\alpha}\supset Q$. Thus \[\Omega\supset \bigcap_{\alpha\in \Phi} D_{\alpha,\lambda_\alpha}\supset \bigcap_{\alpha\in\Phi_{\Delta}^-}D_{\alpha,\lambda_\alpha}\cap \bigcap_{\alpha\in \Phi_{\Delta}^+}\mathring{D}_{\alpha,\lambda_\alpha}\supset \XC_{\Rtot} \cap Q\] and thus $\Omega\in \germ_x(Q)$, which proves   that $\cl(\{x,y\})\Supset \FC_{x,C^v}$.

As $\germ_x(Q)\Subset \FC_{x,C^v}$, we have $\cl(\germ_x(Q))\supset \cl(\FC_{x,C^v})$ (as sets of subset of $\A_\Rtot$). Let now $\Omega\in \cl(\germ_x(Q))$. Then there exists an open subset $\Omega'$ of $\A_\Rtot$ containing $x$ and $(\lambda_\alpha)\in (\Gamma_\alpha\cup \{\infty\})^{\Phi}$ such that  $\Omega\supset \bigcap_{\alpha\in \Phi}D_{\alpha,\lambda_\alpha}\supset \Omega'\cap Q$. Then  $\bigcap_{\alpha\in \Phi}D_{\alpha,\lambda_\alpha}\in \cl(\FC_{x,C^v})$ and thus $\Omega\in \cl(\FC_{x,C^v})$. Therefore $\cl(\germ_x(Q))\subset \cl(\FC_{x,C^v})$. Lemma follows.
\end{proof}

\subsubsection{\texorpdfstring{$\Rtot$}{R}-distances on affine apartments}\label{subsubMetric}

As before, fix a totally ordered abelian group $R$ and consider an affine apartment $\underline{\A_\Rtot}=\big(\A_\Rtot,V_\R,\Phi, (\Gamma_\alpha)_{\alpha\in\Phi},\widehat{W}\big)$ over $R$.
Let $V_\Z$ be the lattice of coweights in $V_\R$ of $\Phi$ and set $V_R := V_\Z \otimes_{\mathbb{Z}} R$\index[notation]{v@$V_R$}. The goal of this section consists in endowing $\A_\Rtot$ with a $\Waff$-invariant distance.

First of all, for $\lambda \in \Rtot$, we define the absolute value of $\lambda$ as: 
$$|\lambda| = \begin{cases}
-\lambda \text{\;\;\;\;\; if $\lambda \in \Rtot_{<0}$}\\
\lambda \text{\;\;\;\;\; otherwise.}
\end{cases}$$
We have $|\lambda + \mu| \leqslant |\lambda| + |\mu|$  and $|n \lambda | = |n| |\lambda| $ for any $\lambda,\mu \in \Rtot$ and $n\in \mathbb{Z}$.

Let now $\Phi^+$ be any choice of a subset of positive roots and set $\| x \|_\Rtot  = \sum_{\alpha \in \Phi^+} | \alpha(x) |$ for $x\in V_R$. 
Then $2 \|x\|_\Rtot = \sum_{\alpha \in \Phi} |\alpha(x) |$.
Since $\Rtot$ is a torsion-free $\mathbb{Z}$-module, by definition of root systems, this defines a $W(\Phi^\vee)$-invariant map $\| \cdot \|_\Rtot: V_\Rtot \to \Rtot_{\geqslant 0}$ that does not depend on the choice of $\Phi^+$ and that satisfies:
\begin{gather*}
\|v+w \|_\Rtot \leq \| v \|_\Rtot + \| w \|_\Rtot,\\
\|\lambda v \|_\Rtot = \|v\|_\mathbb{Z} |\lambda|,
\end{gather*}
 for any $v,w \in V_\Z$ and any $\lambda \in \Rtot$.

Therefore the map $d_\Rtot^{\mathrm{std}}: \mathbb{A}_\Rtot \times \mathbb{A}_\Rtot \to \Rtot_{\geqslant 0}$\index[notation]{d@$d_\Rtot^{\mathrm{std}}$} defined by $d_\Rtot^{\mathrm{std}}(x,y) = \|y-x\|_\Rtot$ defines an $\Rtot$-distance\index{R-distance@$R$-distance} on $\mathbb{A}_\Rtot $, that is a map $d:\A_\Rtot \times \A_\Rtot \rightarrow \Rtot$ such that for all $x,y,z\in \A_\Rtot$, one has:
\begin{enumerate}
\item $d(x,y)\geqslant 0$, and $d(x,y)=0$ if and only if $x=y$;
\item $d(x,y)=d(y,x)$;
\item $d(x,y)\leq d(x,z)+d(z,y)$.
\end{enumerate} 
The map $d^{\mathrm{std}}_\Rtot$ is $\widehat{W}$-invariant
and coincides with the \textbf{standard $\Rtot$-distance}\index{R-distance@$R$-distance!standard} considered by Bennett in \cite{bennett1994affine}.

For $x\in \A_\Rtot$ and $\varepsilon\in \Rtot_{>0}$,
we denote by $B_\Rtot(x,\varepsilon)$ the set $\{y\in \A_\Rtot|\ d(x,y)<\varepsilon\}$.
The topology of $\mathbb{A}_\Rtot$ defined in Subsubsection~\ref{subsubTopology_totally_ordered_ring} thanks to the sets $\mathring{D}_{\alpha,\lambda}$ coincides with the topology that has the $B_\Rtot(x,\varepsilon)$ as a base.

\subsubsection{The real vector space \texorpdfstring{$R=\RF^S$}{RS}}\label{subsubRSMetrics}

In the previous sections, we have studied the notion of an affine apartment over a totally ordered abelian group $R$. In the classical Bruhat-Tits theory, when one is interested in reductive groups over a field $\mathbb{K}$ endowed with a valuation $\omega: \mathbb{K}^{\times} \rightarrow \mathbb{Z}$, one chooses $R =\mathbb{R}$, so that the valuation group $\mathbb{Z}$ is an ordered subgroup of $R$.

 In the present article, we will work over a field $\mathbb{K}$ that is endowed with a valuation $\omega: \mathbb{K}^{\times} \rightarrow \Lambda$ for some non-zero totally ordered abelian group $\Lambda$. Let $\mathrm{rk}(\Lambda)$ be the rank of $\Lambda$, that is the (totally ordered) set of Archimedean equivalence classes of $\Lambda$. Hahn's embedding theorem (see for instance \cite{Gravett}) then states that $\Lambda$ can always be embedded as an ordered subgroup into the lexicographically ordered real vector subspace of $\mathbb{R}^{\mathrm{rk}(\Lambda)}$ given by families $(x_s)_{s\in \mathrm{rk}(\Lambda)}$ with well-ordered support. It is therefore natural to introduce the definition:

\begin{definition}\label{DefRS}
Given $S$ a totally ordered set, \texorpdfstring{$\RF^S$}{RS}\index[notation]{r@$\RF^S$} is the real vector subspace of $\mathbb{R}^S$ whose elements are given by families $(x_s)_{s\in S} \in \mathbb{R}^S$ with well-ordered support. It is endowed with the lexicographical order.
\end{definition}

The totally ordered abelian group $R$ will hence be chosen to be the totally ordered real vector space $\RF^ {\mathrm{rk}(\Lambda)}$.

\begin{Ex}\label{ordinals}
In the case $\Lambda = \mathbb{Z}^n$ for some $n\geq 1$, then $\mathrm{rk}(\Lambda)=\{1,2,...,n\}$, and hence $R$ will be chosen to be $\mathbb{R}^n$. As we have already explained, in this situation, $R$ can be endowed with an $\mathbb{R}$-algebra structure by setting $R=\mathbb{R}[t]/(t^n)$, but this cannot be done in general.
\end{Ex}

We finish this section by giving some properties of affine apartments over $\RF^S$ for some fixed totally ordered set $S$. In order to simplify the notation, we denote by $V_S$, $\A_S$, $F^v_S(\Delta_P)$, etc. instead of $V_{\RF^S}$, $\A_{\RF^S}$, $F^v_{\RF^S}(\Delta_P)$, etc.

If $s\in S $ and $\top$ is a binary relation on $S$ (for example $\leq, <, >,\ldots$), we denote by $\A_{ \top s}$ the space $\A_{\{t\in S|t\top s\}}$ and we introduce the projection $\pi_{\top s}:\A_S\rightarrow \A_{\top s}$ defined by $\pi_{\top s}\big((x_t)_{t\in S}\big)=(x_t)_{t \top s}$, for $(x_s)\in \A_{S}$. We fix an origin $o$ of $\A_S$ and an origin $o_\R$ of $\A_\R$.

\begin{lemma}\label{lemDistance}
The distance $d_S^{\mathrm{std}}:\A_S\times \A_S\rightarrow \RF^S$ satisfies the following properties:
\begin{enumerate}
\item\label{itBalls_open} For all $\varepsilon\in \RF^S_{>0}$, for all $s\in S$, there exists $t\in [s,+\infty[$ and an open neighborhood $\XC_\R$ of $o_\R$ in $\A_\R$ such that $\pi_{\leq t}\big (B_S(o,\epsilon)\big)\supset \{\pi_{<t}(o) \}\times \XC_\R$. 

\item\label{itBalls_bounded} For all $s\in S$, for every open subset $\XC_\R$ of $\A_\R$ containing $\{o_\R\}$, there exists $\epsilon\in \RF^S_{>0}$ such that $B_S(o,\varepsilon)\subset \{\pi_{<s}(o)\}\times \XC_{\R}\times \A_{>s}$.

\item\label{itDistance_translation} It  is invariant under translation, that is: for all $x,y,z\in \A_S$, $d_S^{\mathrm{std}}(x,y)=d_S^{\mathrm{std}}(x+z,y+z)$.

\item\label{itWeyl_compatibility} It is Weyl-compatible\index{R-distance@$R$-distance!Weyl compatible} in the definition of  \cite[Definition 3.1]{bennett2014axiomatic}. 

\end{enumerate}
\end{lemma}

\begin{proof}
Point~(\ref{itDistance_translation}) is clearly satisfied and point~(\ref{itWeyl_compatibility}) is \cite[Lemma 10.1]{bennett2014axiomatic}). Let us prove~(\ref{itBalls_open}).
Let $\epsilon\in \RF^S_{>0}$ and $s\in S$. Let $s_0=\min \{s'\in \supp(\epsilon)\mid\epsilon_{s'}>0\}$. Suppose $s\leq s_0$. Let  $\XC_\R=\{x\in \A_{\R}\mid d_\R^{\mathrm{std}}(o_\R,x)<\frac{1}{2}\epsilon_{s_0}\}$. Then $\pi_{\leq s_0}\big (B_S(o,\epsilon)\big)\supset \{\pi_{<s_0}(o)\}\times \XC_\R$. Suppose now $s>s_0$. Then $B_S(o,\epsilon)\supset \{\pi_{<s}(o)\}\times \A_{[s,+\infty[}$ and thus $\pi_{\leq s}\big( B_S(o,\epsilon)\big)\supset \{\pi_{<s}(o)\}\times \A_\R$, which proves~(\ref{itBalls_open}).

Let us prove~(\ref{itBalls_bounded}). Let $s\in S$ and let $\XC_\R$ be an open neighborhood of $o_\R$ in $\A_\R$. Let $\epsilon_\R \in \R_{>0}$ be such that $\{x\in \A_\R \mid d(o_\R,x)<\epsilon_\R\}\subset  \XC_\R$. Let $\epsilon=( \frac{1}{2}\delta_{t,s}\epsilon_\R)_{t\in S}\in \RF^S$. Then $B_S(o,\epsilon)\subset \{\pi_{<s}(o)\}\times \XC_\R\times \A_{]s,+\infty[}$, which proves~(\ref{itBalls_bounded}).
\end{proof}

\subsection{Definition of \texorpdfstring{$\RF^S$}{Lambda}-buildings}\label{subBennett_definition}

\subsubsection{Bennett's definition of \texorpdfstring{$\RF^S$}{RS}-buildings}\label{ss_Bennetts_def_buildings}

Let $S$ be a totally ordered set and let $\underline{\A_S}=(\A_S,V_\R,\Phi,(\Gamma_\alpha)_{\alpha\in\Phi},\widehat{W})$ be an affine apartment over $\RF^S$ (see Definition~\ref{defApartment}). An \textbf{apartment of  type $\underline{\A_S}$}\index{apartment!type} is a set $A$ equipped with a nonempty set $\mathrm{Isom} (\A_S,A)$ of bijections $f:\A_S\rightarrow A$ such that if $f_0\in \mathrm{Isom}(\A_S,A)$, then $\mathrm{Isom}(\A_S,A)=\{f_0\circ  w\mid w\in \widehat{W}\}$.
An \textbf{isomorphism}\index{apartment!isomorphism} between two apartments $A,A'$ is a bijection $\phi:A\rightarrow A'$ such that there exists $f_0 \in \mathrm{Isom}(\A_S,A)$ such that $\phi\circ f_0\in \mathrm{Isom}(\A_S,A')$.
 
 Each apartment $A$ of type $\underline{\A_S}$ can be equipped with  the structure of an $R$-aff space by choosing a bijection $f:\A_S\rightarrow A$ in $\mathrm{Isom} (\A_S,A)$.

We extend all the notions that are preserved by  $\widehat{W}$ to each apartment.
In particular half-apartments, walls, enclosure, sector-faces, local germs, germs at infinity, ... are well defined in each apartment of type $\A_S$.   

We say that an apartment contains  a filter if it contains at least one element of this filter. Recall that we say that a map fixes a  filter if it fixes at least one element of this  filter.
 
\begin{definition}\label{defBuildings}
An $\RF^S$-building\index{building}\index{R-building@$\RF^S$-building} is a set $\I$ equipped with a covering $\ACC$ by subsets called \textbf{apartments}\index{apartment} such that:
\begin{enumerate}[label={(A\arabic*)}]
\item\label{axiomA1} Each $A\in \ACC$ is equipped with the structure of an apartment of type $\underline{\A_S}$.\axiom{A1@\ref{axiomA1}}

\item\label{axiomA2}  If $A, A'$ are two apartments, then $A\cap A'$ is enclosed in $A$ and there exists an isomorphism $\phi:A\rightarrow A'$ fixing $A\cap A'$.\axiom{A2@\ref{axiomA2}}

\item\label{axiomA3} For any pair of points in $\I$, there is an apartment containing both.\axiom{A3@\ref{axiomA3}}

\end{enumerate}

\medskip

Given a  $\widehat{W}$-invariant
$\RF^S$-distance $d$ on the model space $\A_S$, axioms~\ref{axiomA1}–\ref{axiomA3} imply the existence of a function $d: \I\times \I\rightarrow \RF^S$ satisfying all conditions of the definition of an $\RF^S$-distance except possibly the triangle inequality. The distance is defined as follows. Let $x,y\in \I$ and $A$ be an apartment containing both of them. Then, if we choose an element $f\in \mathrm{Isom}(\A_S,A)$, the distance between $x$ and $y$ is the distance between $f^{-1}(x)$ and $f^{-1}(y)$. This turns out not to depend on any choices.

\medskip

\begin{enumerate}[label={(A\arabic*)}]
\setcounter{enumi}{+3}
\item\label{axiomA4}  For any pair of sector-germs in $\I$, there is an apartment containing both.\axiom{A4@\ref{axiomA4}}

\item\label{axiomA5}  For any apartment $A$ and all $x\in A$, there exists a retraction $\rho_{A,x}:\I\rightarrow A$ such that $\rho_{A,x}$ does not increase distances and $\rho_{A,x}^{-1}(\{x\})=\{x\}$. \axiom{A5@\ref{axiomA5}}

\item\label{axiomA6} Let $A_1,A_2,A_3$ be apartments such that $A_1\cap A_2$, $A_2\cap A_3$ and $A_3\cap A_1$ are half-apartments. Then $A_1\cap A_2\cap A_3$ is nonempty.\axiom{A6@\ref{axiomA6}}

\end{enumerate}

\begin{remark}
Suppose that $S$ is reduced to a single element (thus $\RF^S\simeq \R$). The axioms~\ref{axiomA1} to~\ref{axiomA4} correspond to the axioms~\ref{axiomA1} to~\ref{axiomA4} of \cite[1.2]{parreau2000immeubles}. Axiom~\ref{axiomA5} corresponds to  axiom~(A5') of \cite{parreau2000immeubles} and axiom~\ref{axiomA6} corresponds to axiom~\ref{axiomA5} of \cite[1.4]{parreau2000immeubles}. Note that under this assumption,~\ref{axiomA6} is a consequence of the axioms~\ref{axiomA1} to~\ref{axiomA5}. 
\end{remark}

\end{definition}

  \subsubsection{Equivalent definition of \texorpdfstring{$\RF^S$}{RS}-buildings}

\medskip

In \cite{bennett2014axiomatic}, Bennett and Schwer introduce several other equivalent definitions of $\mathfrak{R}^S$-buildings. We now briefly recall one of them, that will be useful in the sequel.
 
 Let $\I$ be a set satisfying axioms~\ref{axiomA1},~\ref{axiomA2} and~\ref{axiomA3}, and consider the following two extra axioms:

\begin{enumerate}[label={(GG)}]
\item\label{axiomGG} Any two local chambers based at the same vertex are contained in a common apartment.\axiom{GG@\ref{axiomGG}}
\end{enumerate}
\begin{enumerate}[label={(CO)}]
\item\label{axiomCO} Say that two sectors $Q,Q'$ based at the same point $x$ are \textbf{opposite at $x$}\index{sector!opposite} if there exists an apartment $A$ containing $\germ_x(Q),\germ_x(Q')$ and such that $\germ_x(Q)$ and $\germ_x(Q')$ are opposite in $A$. If $Q$ and $Q'$ are two such sectors, then there exists a unique apartment containing $Q$ and $Q'$. \axiom{CO@\ref{axiomCO}}
\end{enumerate}

As a particular case of \cite[Theorem 3.3]{bennett2014axiomatic}, we have:

\begin{theorem}\label{thmBennett-Schwer-3.3}
Let $(\I,d)$ be a set satisfying~\ref{axiomA1},~\ref{axiomA2} and~\ref{axiomA3}, where $d$ is a Weyl-compatible $\RF^S$-distance on $\A_S$. Then $\I$ is a $\Lambda$-building if and only if $\I$ satisfies~\ref{axiomGG} and~\ref{axiomCO}. 

\end{theorem}

\subsection{Main theorem}\label{subMain_theorem}

\begin{theorem}\label{thmMain}
Let $\Lambda$ be a non-zero totally ordered abelian group with rank $S$, so that $\Lambda$ can be seen as a totally ordered subgroup of $\mathfrak{R}^S$. Let $\mathbb{K}$ be a field with a valuation $\omega: \mathbb{K} \rightarrow \Lambda \cup \{\infty\}$, let $\mathbf{G}$ be a quasi-split (connected) reductive $\mathbb{K}$-group and let $\mathbf{S}$ be a maximal split torus in $\mathbf{G}$ with cocharacter module $X_*(\mathbf{S})$. If $\mathbf{G}$ is not split, we assume that $\mathbb{K}$ is Henselian. 
\begin{itemize}
\item[(i)] The set $\mathcal{I}(G)=\mathcal{I}(\mathbb{K},\omega,\mathbf{G})$ defined in section~\ref{buildredgp} and endowed with the distance introduced in~\ref{subsubMetric} is an $\mathfrak{R}^S$-building whose apartments have type:
$$\underline{\mathbb{A}_S} = (\mathbb{A}_S,V_{\mathbb{R}},\Phi,(\Gamma_{\alpha})_{\alpha \in \Phi},\widetilde{W})$$
where $V_{\mathbb{R}}$ is the quotient of the real vector space $ X_*(\mathbf{S}) \otimes_{\mathbb{Z}} \mathbb{R}$ by the orthogonal of the roots of $\mathbf{G}$, $\mathbb{A}_S$ is an $\mathfrak{R}^S$-aff space whose underlying vector space is $V_{\mathbb{R}} \otimes_{\mathbb{R}} \mathfrak{R}^S$, $\Phi$ is the root system associated to $\mathbf{G}$ in $V_{\mathbb{R}}^*$ and, for $\alpha \in \Phi$, $\Gamma_{\alpha}$ is a subset of $\mathfrak{R}^S$ that generates a subgroup in which $\Lambda$ has finite index. The group  $\Wext$ is the extended affine Weyl group, obtained by restriction to $\A_S$ of  the action of the stabilizer of $\A_S$ in $\mathbf{G}(\mathbb{K})$. The group $\mathbf{G}(\mathbb{K})$ acts on $\mathcal{I}(G)$ by isometries whose restrictions to apartments are all real $\mathfrak{R}^S$-aff maps. The induced action on the set of apartments is transitive.
\item[(ii)] Let $s \in S$, let $S_{\leq s} = \{ t\in S | t\leq s\}$ and let $\pi_{\mathfrak{R}^S,\leq s}: \mathfrak{R}^S \rightarrow \mathfrak{R}^{S_{\leq s}}$ be the natural projection. Consider the valuation $\omega_{\leq s} = \pi_{\RF^S,\leq s} \circ \omega$. There exists an (explicit) surjective map:
$$\pi_{\leq s}: \mathcal{I}(\mathbb{K},\omega,\mathbf{G}) \rightarrow \mathcal{I}(\mathbb{K},\omega_{\leq s},\mathbf{G})$$
compatible with the $\mathbf{G}(\mathbb{K})$-action such that, for each $X \in \mathcal{I}(\mathbb{K},\omega_{\leq s},\mathbf{G})$, the fiber $\pi_{\leq s}^{-1}(X)$ is a product:
$$\mathcal{I}_X \times \left( \langle \Phi_X  \rangle^{\perp} \otimes_{\mathbb{R}} \ker(\pi_{\mathfrak{R}^S,\leq s})\right)$$
where $\Phi_X$ is a root system contained in $\Phi$, $\langle \Phi_X  \rangle^{\perp}$ is the orthogonal of $\Phi_X$ in $V_{\mathbb{R}}$, and $\mathcal{I}_X$ is a
$\ker(\pi_{\mathfrak{R}^S,\leq s})$-building. The apartments of $\mathcal{I}_X$ have type:
$$\underline{\mathbb{A}_X} = (\mathbb{A}_X,V_{\mathbb{R}}/\langle \Phi_X  \rangle^{\perp},\Phi_X,(\Gamma_{X,\alpha})_{\alpha \in \Phi},\widetilde{W}_{X})$$
where $\mathbb{A}_X$ is a $\ker(\pi_{\mathfrak{R}^S,\leq s})$-aff space whose underlying vector space is $\ker(\pi_{\mathfrak{R}^S,\leq s})$-module is $\left( V_{\mathbb{R}}/\langle \Phi_X  \rangle^{\perp}\right) \otimes_{\mathbb{R}} \ker(\pi_{\mathfrak{R}^S,\leq s})$, the group $\widetilde{W}_{X}$ is some subgroup of $\mathrm{Aff}_{\ker(\pi_{\mathfrak{R}^S,\leq s})}(\mathbb{A}_X)$, and for $\alpha \in \Phi_X$, $\Gamma_{X,\alpha}$ is some subset of $\ker(\pi_{\mathfrak{R}^S,\leq s})$. 
\end{itemize}
\end{theorem}

We now briefly explain the strategy of proof. In sections~\ref{SecVRGD} and~\ref{SecParahoricBruhat}, we introduce an abstract notion of $\mathfrak{R}^S$-valued root group data, generalizing the more classical notion of $\mathbb{R}$-valued root group data introduced by Bruhat and Tits, and in section~\ref{sectionLambdaBuildingFromVRGD}, we prove that one can always associate an $\mathfrak{R}^S$-building to such an $\mathfrak{R}^S$-valued root group data. This tool then allows us to deal in a unified way with the two parts of the theorem: each time we have to prove that some construction is an $\mathfrak{R}^S$-building, we check that it is the $\mathfrak{R}^S$-building associated to some $\mathfrak{R}^S$-valued root group data.

Let's more precisely decribe how the construction of the building $\I:=\I(\Kb,\omega,\GB)$ goes through. 

We start from the datum of a quasi-split redutive $\mathbb{K}$-group $\mathbf{G}$ and a maximal $\mathbb{K}$-split torus $\mathbf{S}$ of $\mathbf{G}$.
We let $\mathbf{T}$ and $\mathbf{N}$ be respectively the centralizer and the normalizer of $\mathbf{S}$ in $\mathbf{G}$.
The standard apartment $\A_S$ is defined as an $\RF^S$-aff space whose underlying vector space is the scalar extension to $\RF^S$ of the abelian group of $\mathbb{K}$-cocharacters of $\mathbf{T}$.
We construct an action of $N = \mathbf{N}(\mathbb{K})$ on $\A_S$.
If $\Phi$ stands for the $\mathbb{K}$-root system of $(\mathbf{G},\mathbf{S})$, we use a Chevalley-Steinberg system $(x_\alpha)_{\alpha \in \Phi}$ of $\mathbf{G}$ (i.e. a parametrization of root groups $\mathbf{U}_\alpha$ compatible with the Galois action of the splitting extension $\widetilde{\mathbb{K}} / \mathbb{K}$ of $\mathbf{G}$) in order to define the parahoric subgroups $P_{\VC}$ of $G = \mathbf{G}(\mathbb{K})$, for every filter $\VC$ on $\A_S$.
Inspired by the classical constructions of Bruhat and Tits, we then define $\I$ as $G\times \A_S/\sim$, where $\sim$ is an equivalence relation defined in section~\ref{sectionLambdaBuildingFromVRGD}. 

We then need to prove that $\I$ satisfies the axioms~\ref{axiomA1} to~\ref{axiomA6}. For this, we use Theorem~\ref{thmBennett-Schwer-3.3} and we prove that $\I$ satisfies the axioms~\ref{axiomA1},~\ref{axiomA2},~\ref{axiomA3}, \ref{axiomGG} and~\ref{axiomCO}. The fact that $\I$ satisfies~\ref{axiomA1} follows immediately from the definitions. The axiom~\ref{axiomA2} is obtained similarly as in \cite{ BruhatTits1}. In order to prove \ref{axiomGG} and~\ref{axiomA3} we generalize the Bruhat decomposition in our setting (see Theorem~\ref{ThmBruhatDecomposition}). This requires to first prove that $G$ satisfies an Iwasawa decomposition (see Theorem~\ref{thmIwasawa}). Restated in terms of buildings, this decomposition asserts that if $F$ is a face of $\I$ and $C_\infty$ is a sector-germ at infinity of $\I$, then there exists an apartment containing $F$ and $C_\infty$. 

Section~\ref{projectionsec} is dedicated to the projection map $\pi_{\leq s}$ that has been introduced in part (ii) of theorem~\ref{thmMain}. We first construct the map $\pi_{\leq s}$ itself, and we give an explicit description of its fibers. We then prove that those fibers are associated to an $\mathfrak{R}^S$-valued root group datum. By the general theory developed in sections~\ref{SecVRGD},~\ref{SecParahoricBruhat} and~\ref{sectionLambdaBuildingFromVRGD}, we deduce the decomposition of theorem~\ref{thmMain}:
$$\pi_{\leq s}^{-1}(X) = \mathcal{I}_X \times \left( \langle \Phi_X  \rangle^{\perp} \otimes_{\mathbb{R}} \mathfrak{R}^S\right)$$
for some $\mathcal{I}_X$ that satisfies axioms~\ref{axiomA1},~\ref{axiomA2},~\ref{axiomA3},~\ref{axiomA4} and \ref{axiomGG}.

In section~\ref{sectionCO}, we finish the proof of theorem~\ref{thmMain} by establishing axiom~\ref{axiomCO}. This axiom is more geometric in nature. In order to prove it we first give a sufficient condition for an `` $\R$-building'' to satisfy~\ref{axiomCO} (see Lemma~\ref{lemCO_for_R_buildings}). Using this criterion and  the projection maps defined in section~\ref{projectionsec}, we prove that our building satisfies~\ref{axiomCO}.

\section{\texorpdfstring{$\Rtot$}{R}-valued root group datum}
\label{SecVRGD}

In this section, if $G$ is a group and $X,Y$ are subsets of $G$, we denote by:
\begin{itemize}
\item $1$ the identity element of $G$;
\item $X Y = \{xy,\ x \in X,\ y \in Y\}$ the subset of $G$ obtained as  image of the map $X \times Y \to G$ given by multiplication in $G$;
\item $\langle X,Y \rangle$\index[notation]{Z@$[\cdot,\cdot]$} the subgroup of $G$ generated by $X \cup Y$; 
\item $[X,Y]$ the subgroup of $G$ generated by the set of commutators $[x,y]$ for $x \in X$ and $y \in Y$.\index{subgroup!commutator}
\end{itemize}

\subsection{Abstract groups and axioms of a root group datum}

We recall the following definition from \cite[6.1.1]{BruhatTits1}.

\begin{Def}\label{DefRGD}
Let $G$ be a group and $\Phi$ be a root system.
A root group datum\index{root group datum} of $G$ of type $\Phi$ is a system
$(T, (U_\alpha, M_\alpha)_{\alpha \in \Phi})$\index[notation]{t@$T$}\index[notation]{u@$U_\alpha$}\index[notation]{m@$M_\alpha$} satisfying the following axioms:
\begin{enumerate}[label={(RGD\arabic*)}]
\item\label{axiomRGD1} $T$ is a subgroup of $G$ and, for any root $\alpha \in \Phi$, the set $U_\alpha$ is a nontrivial subgroup of $G$, called the root group of $G$ associated to $\alpha$;\axiom{RGD1@\ref{axiomRGD1}}
\item\label{axiomRGD2} for any roots $\alpha,\beta \in \Phi$ such that $\beta \not\in \mathbb{R}_{< 0} \alpha$, the commutator subgroup $[U_\alpha, U_\beta]$ is contained in the subgroup generated by the root groups $U_{\gamma}$ for $\gamma \in (\alpha,\beta)$;\axiom{RGD2@\ref{axiomRGD2}}
\item\label{axiomRGD3} if $\alpha$ is a multipliable root, we have $U_{2\alpha} \subset U_\alpha$ and $U_{2\alpha} \neq U_\alpha$;\axiom{RGD3@\ref{axiomRGD3}}
\item\label{axiomRGD4} for any root $\alpha \in \Phi$, the set $M_\alpha$ is a right coset of $T$ in $G$ and we have $U_{-\alpha} \setminus \{ 1 \} \subset U_\alpha M_\alpha U_\alpha$;\axiom{RGD4@\ref{axiomRGD4}}
\item\label{axiomRGD5} for any roots $\alpha,\beta \in \Phi$ and any $m \in M_\alpha$, we have $m U_\beta m^{-1} = U_{r_\alpha(\beta)}$;\axiom{RGD5@\ref{axiomRGD5}}
\item\label{axiomRGD6} for any choice of positive roots $\Phi^+$ on $\Phi$, we have ${T U^+ \cap U^- = \{1\}}$ where $U^+$ (resp. $U^-$) denotes the subgroup generated by the $U_{\alpha}$ for $\alpha \in \Phi^+$ (resp. $\alpha \in \Phi^-= - \Phi^+$).\axiom{RGD6@\ref{axiomRGD6}}
\end{enumerate}

A root group datum is said generating\index{root group datum!generating} if $G$ is generated by the subgroups $T$ and the $U_\alpha$ for $\alpha\in \Phi$.
As in \cite[6.1.2(10)]{BruhatTits1}, we denote by $N$\index[notation]{n@$N$} the subgroup of $G$ generated by the $M_\alpha$ for $\alpha \in \Phi$ if $\Phi \neq \emptyset$ and by $N = T$ otherwise.

We recall that, according to \cite[6.1.2(10)]{BruhatTits1}, axiom~\ref{axiomRGD5} defines an epimorphism ${^v\!}\nu: N \to W(\Phi)$ such that ${^v\!}\nu(m) = r_\alpha$\index[notation]{n@${^v\mkern-3mu}\nu$} for any $m \in M_\alpha$, any $\alpha \in \Phi$.
Thus, for any $\alpha \in \Phi$ and any $n \in N$, we have $n U_{\alpha} n^{-1} = U_{{^v\!}\nu(\alpha)}$.
\end{Def}

\begin{Ex}
Let $\mathbb{K}$ be any field and $\mathbf{G}$ be a reductive $\mathbb{K}$-group, $\mathbf{S}$ a maximal $\mathbb{K}$-split torus of $\mathbf{G}$ and $\mathbf{Z} = \mathcal{Z}_{\mathbf{G}}(\mathbf{S})$.
According to \cite[4.1.19]{BruhatTits2}, there exist right cosets $M_\alpha$ such that $\mathbf{G}(\mathbb{K})$ admits a generating root group datum $\Big(\mathbf{Z}(\mathbb{K}),\big(\mathbf{U}_\alpha(\mathbb{K}),M_\alpha\big)_{\alpha\in\Phi}\Big)$ of type $\Phi$ which is the $\mathbb{K}$-root system of $\mathbf{G}$ with respect to $\mathbf{S}$.
In particular, for such a root group datum, one can apply any result of section \cite[6.1]{BruhatTits1}.
Moreover, $N = \mathcal{N}_{\mathbf{G}}(\mathbf{S})(\mathbb{K})$ in this example.
\end{Ex}

In Bruhat-Tits theory, it appears to be useful to consider some groups $Z$ generated by some well-chosen subgroups $X_\alpha$ of the root groups $U_\alpha$.
A first result is given by Proposition \cite[6.1.6]{BruhatTits1}, for a group generated by non-trivial subsets $X_\alpha$ indexed over a positively closed subset $\Psi \subset \Phi^+$ of roots, assuming a "condition (i)".
A second result is given by Proposition \cite[6.4.9]{BruhatTits1}, for a group generated by non-trivial subsets $X_\alpha$ indexed over the whole root system, under some specific conditions on the $X_\alpha$ that are denoted by $U_{\alpha,f}$ in \cite[§6]{BruhatTits1}.
In fact, we observe that the proof of this Proposition only relies on two axioms of ``quasi-concavity'' (QC1) and (QC2) (see \cite[6.4.7]{BruhatTits1}) that are satisfied by a quasi-concave map $f$, and that we can translate those conditions onto conditions over the groups $X_\alpha$.
Nevertheless, according to addendum \cite[E2]{BruhatTits2}, condition (QC2) is a bit too weak for some general results, so that it is useful to assume that the $X_\alpha$ also satisfy a condition (QC0).
In our definition, the condition~\ref{axiomQC2} takes into account simultaneoulsy  both conditions (QC0) and (QC2) of \cite{BruhatTits2}.
Moreover, it appears to be useful to consider an additional subgroup $Y$\index[notation]{y@$Y$} that normalizes the $X_\alpha$.
Thus, we will use the following definition:

\begin{Def}\label{DefQC}
Let $\left( T, (U_\alpha,M_\alpha)_{\alpha \in \Phi}\right)$ be a generating root group datum of a group $G$ and $N$ be the subgroup of $G$ generated by $T$ and the $M_\alpha$ for $\alpha\in \Phi$.
Let $(X_\alpha)_{\alpha\in\Phi}$ be a family of subgroups $X_\alpha \subset U_{\alpha}$ for $\alpha \in \Phi$ and $Y$ be a subgroup of $T$.

For $\alpha\in \Phi_{\mathrm{nd}}$, denote by:
\begin{itemize}
\item $X_{2\alpha}$ the trivial subgroup if $\alpha \in \Phi$ and $2\alpha \not\in\Phi$;
\item $L_\alpha$\index[notation]{l@$L_\alpha$} the subgroup generated by $X_\alpha$, $X_{2\alpha}$, $X_{-\alpha}$, $X_{-2\alpha}$ and $Y$;
\item $N_\alpha = L_\alpha \cap N$\index[notation]{n@$N_\alpha$}.
\end{itemize}

We say that the family $\left((X_\alpha)_{\alpha \in \Phi},Y\right)$ is quasi-concave\index{quasi-concave} if it satisfies the axioms:
\begin{enumerate}[label={(QC\arabic*)}]
\item\label{axiomQC1} $L_\alpha = X_\alpha X_{2\alpha} X_{-\alpha} X_{-2\alpha} N_\alpha = X_{-\alpha} X_{-2\alpha} X_\alpha X_{2\alpha} N_\alpha$ for any $\alpha \in \Phi_{\mathrm{nd}}$;\axiom{QC1@\ref{axiomQC1}}
\item\label{axiomQC2} for every $\alpha,\beta \in \Phi$ with $\beta \not\in -\mathbb{R}_{>0}\alpha$, the commutator group $[X_\alpha,X_\beta]$ is contained in the group $X_{(\alpha,\beta)}$ generated by the $X_\gamma$ for $\gamma \in (\alpha,\beta)$;\axiom{QC2@\ref{axiomQC2}}
\item\label{axiomQC3} $Y$ normalizes $X_\alpha$ for every $\alpha \in \Phi$.\axiom{QC3@\ref{axiomQC3}}
\end{enumerate}

If $Y$ is trivial, by abuse of language, the family of groups $(X_\alpha)_{\alpha \in \Phi}$ is said quasi-concave.
\end{Def}

Note that condition~\ref{axiomQC2} implies that $X_{\alpha}$ normalizes $X_{2\alpha}$ so that one can also write $L_\alpha = X_{2\alpha} X_\alpha X_{-\alpha} X_{-2\alpha} N_\alpha$ in~\ref{axiomQC1} for instance.

Because axiom~\ref{axiomQC2} does not depend on $Y$, we will say by abuse of language that the family $(X_\alpha)_{\alpha\in \Phi}$ satisfies~\ref{axiomQC2} when this condition is satisfied.

With this definition, we get the following Proposition analogous to \cite[6.4.9]{BruhatTits1}.

\begin{Prop}\label{PropDecompositionQC}
Let $\left((X_\alpha)_{\alpha\in\Phi},Y\right)$ be a quasi-concave family of groups.
Denote by $X$ the subgroup of $G$ generated by $Y$ and by the $X_\alpha$ for $\alpha \in \Phi$.
Suppose that $\Phi$ is non-empty. Then for any choice of a subset of positive roots $\Phi^+$ of $\Phi$:

\begin{enumerate}[label={(\arabic*)}]
\item\label{PropDecQC:1} $U_\alpha \cap X = X_\alpha X_{2\alpha}$ for any $\alpha \in \Phi_\mathrm{nd}$;
\item\label{PropDecQC:2} the product map $\displaystyle \prod_{\alpha \in \Phi^+_{\mathrm{nd}}} \left( X_\alpha X_{2\alpha} \right) \to X \cap U^+$ (resp.
 $\displaystyle \prod_{\alpha \in \Phi^+_{\mathrm{nd}}} \left( X_{-\alpha} X_{-2\alpha} \right) \to X \cap U^-$) induced by  multiplication in $G$ is a bijection for any ordering on the product;
\item\label{PropDecQC:3} we have $X = (X \cap U^+) (X \cap U^-) (X \cap N)$ for any choice of $\Phi^+$ in $\Phi$;
\item\label{PropDecQC:4} the group $X \cap N$ is generated by the $N_\alpha$ for $\alpha \in \Phi_{\mathrm{nd}}$.
\end{enumerate}
\end{Prop}

Note that, by definition of $N_{\alpha}$, even if $Y=1$, it may happen that $X_\alpha = X_{-\alpha} = 1$ but $N_{\alpha} \neq 1$ for a multipliable root $\alpha$.
Moreover, since $N_{2\alpha} \subset N_{\alpha}$ for any multipliable root $\alpha$, the group $X \cap N$ is also generated by the $N_{\alpha}$ for $\alpha \in \Phi$.

\begin{Lem}[{see \cite[E2]{BruhatTits2}}]\label{lemSaturationQC}
If $\left((X_\alpha)_{\alpha\in\Phi},Y\right)$ is a quasi-concave family of groups, then also is $\left((X_\alpha X_{2\alpha})_{\alpha\in\Phi},Y\right)$.
\end{Lem}

\begin{proof}
Let $\mathcal{X} = \left((X_\alpha)_{\alpha\in\Phi},Y\right)$ and $\mathcal{X}' = \left((X_\alpha X_{2\alpha})_{\alpha\in\Phi},Y\right)$.
Axiom~\ref{axiomQC3} is immediate for $\mathcal{X}'$.
Since $X_{4\alpha}$ is trivial and $X_{2\alpha} \subset U_{2\alpha}$ normalizes $X_\alpha \subset U_\alpha$ by axiom~\ref{axiomRGD2} (indeed $[U_\alpha,U_{2\alpha}] = 1$), we deduce axiom~\ref{axiomQC1} for $\mathcal{X}'$ from axiom~\ref{axiomQC1} for $\mathcal{X}$.

Now, let $\alpha, \beta \in \Phi$ with $\beta \not\in -\mathbb{R}_{>0} \alpha$.
Let $Z = \langle X_\gamma \rangle_{\gamma \in (\alpha,\beta)}$.
For any $\gamma \in (\alpha, \beta)$, the subgroup $X_\gamma$ normalizes $Z$ by axiom~\ref{axiomQC2} for $\mathcal{X}$ since $(\alpha,\beta) = (\mathbb{Z}_{>0} \alpha + \mathbb{Z}_{>0} \beta)\cap \Phi$.
Let $x \in X_\alpha$, $y \in X_{2\alpha}$, $u\in X_\beta$ and $v \in X_{2\beta}$.
It suffices to prove that $[xy,uv]$ belongs to $Z$.
Using the usual formula $[ab,c]=a[b,c][c,a^{-1}]a^{-1}$, valid for any elements $a,b,c$ of any group, and the fact  that $Z$ is normalized by $x,y,u,v$,  we have that $[xy, uv]\in Z$. Thus, we deduce axiom~\ref{axiomQC2} for $\mathcal{X}'$.
\end{proof}

\begin{proof}[Proof of Proposition~\ref{PropDecompositionQC}.]
According to Lemma~\ref{lemSaturationQC}, we may assume that $X_\alpha =X_\alpha X_{2\alpha}$ for any $\alpha \in \Phi$.

Consider an arbitrary ordering on $\Phi^+_{\mathrm{nd}}$ (resp. $\Phi^-_{\mathrm{nd}}$).
Let $f_+: \prod_{\alpha \in \Phi_{\mathrm{nd}}^+} U_\alpha \to G$ (resp. $f_-: \prod_{\alpha \in \Phi_{\mathrm{nd}}^-} U_\alpha \to G$) be the map induced by multiplication.

Denote by $X^+ = f_+\left(  \prod_{\alpha \in \Phi_{\mathrm{nd}}^+} X_\alpha \right)$ and by $X^- = f_-\left(  \prod_{\alpha \in \Phi_{\mathrm{nd}}^-} X_\alpha \right)$.
According to axiom~\ref{axiomQC2} and \cite[6.1.6]{BruhatTits1}, we know that the restriction of the map $f_+$ (resp. $f_-$) to $\prod_{\alpha \in \Phi^+_{\mathrm{nd}}} X_\alpha$ (resp. $\prod_{\alpha \in \Phi^-_{\mathrm{nd}}} X_\alpha$) induces a bijection onto $X^+$ (resp. $X^-$) and that $X^+$ (resp. $X^-$) is a subgroup of $G$.
In fact, $X^+ \subset U^+ \cap X$ and $X^- \subset U^- \cap X$.
To prove~\ref{PropDecQC:2}, it suffices to prove that these inclusions are equalities.

Denote $L_\alpha$ and $N_\alpha$ as in Definition~\ref{DefQC}.
Denote by $Z$ the group generated by the $N_\alpha$ for $\alpha \in \Phi$.
Note that $Y \subset N_\alpha$ for every $\alpha \in \Phi_{\mathrm{nd}}$, and therefore $Y \subset Z$ since $\Phi_{\mathrm{nd}}$ is non-empty.
By definition, $X^- X^+ Z$ is a subset of $X$ as product of subgroups. 

We will prove that this subset $X^- X^+ Z$ does not depend on the basis (or Weyl chamber) $\Delta$ defining the choice of positive roots $\Phi^+$ in the root system $\Phi$.
More precisely, we prove that it is the same set if we replace $\Delta$ by $r_\alpha(\Delta)$ for $\alpha \in \Delta$ a simple root and we are done since the $r_\alpha$ for $\alpha \in \Delta$ generate the Weyl group of $\Phi$.
Let $\alpha \in \Delta \subset \Phi_{\mathrm{nd}}$ and denote by $\widehat{X}_\alpha$ (resp. $\widehat{X}_{-\alpha}$) the product of the $X_\beta$
(resp. $X_{-\beta}$) for $\beta \in \Phi^+_{\mathrm{nd}} \setminus \{\alpha\}$.
According to axiom~\ref{axiomQC2}, one can apply \cite[6.1.6]{BruhatTits1} to the family of groups $X_\alpha$ and $X_\beta$ for $\beta \in \Phi^+_{\mathrm{nd}} \setminus \{\alpha\}$ to get that $\widehat{X}_\alpha$ is a subgroup of $G$ normalized by $X_\alpha$.
Similarly, using the fact that $r_\alpha(\alpha)=-\alpha$ and $r_{\alpha}(\Phi^+_{\mathrm{nd}}\setminus\{\alpha\})=\Phi^+_{\mathrm{nd}}\setminus\{\alpha\}$ we have that $X_{-\alpha}$ normalizes $\widehat{X}_\alpha$.

By the same way, $\widehat{X}_{-\alpha}$ is normalized by $X_\alpha$ and $X_{-\alpha}$.
Moreover, $\widehat{X}_\alpha$ and $\widehat{X}_{-\alpha}$ are normalized by $Y$ since $Y$ normalizes the $X_\alpha$ by~\ref{axiomQC3}.
Therefore,  as by assumption, $X_{2\alpha }\subset X_\alpha$, we deduce that $L_\alpha$ normalizes $\widehat{X}_\alpha$ and $\widehat{X}_{-\alpha}$, and so does $N_\alpha$.
Moreover $L_\alpha = X_\alpha X_{-\alpha} N_\alpha = X_{-\alpha} X_\alpha N_\alpha$ by~\ref{axiomQC1}.
Hence we have
\begin{align*}
X^- X^+ Z
=& \left(\widehat{X}_{-\alpha} X_{-\alpha}\right) \left( X_\alpha \widehat{X}_\alpha \right) \left( N_\alpha Z \right)\\
=& \widehat{X}_{-\alpha} X_{-\alpha} X_{\alpha}  N_\alpha \widehat{X}_\alpha Z\\
=& \widehat{X}_{-\alpha} L_\alpha \widehat{X}_\alpha Z\\
=& \widehat{X}_{-\alpha} X_\alpha X_{-\alpha}  N_\alpha \widehat{X}_\alpha Z\\
=& \left( \widehat{X}_{-\alpha} X_\alpha \right)\left( X_{-\alpha}  \widehat{X}_\alpha\right) N_\alpha Z\\
=& \left( \prod_{\beta \in r_\alpha(\Phi^-_{\mathrm{nd}})} X_\beta \right)
\left( \prod_{\beta \in r_\alpha(\Phi^+_{\mathrm{nd}})} X_\beta \right)
Z
\end{align*}
As a consequence, the set $X^- X^+ Z$ does not depend on the choice of $\Delta$ and, therefore, is stable by left multiplication by elements in $X_\alpha$ for any $\alpha \in \Phi_{\mathrm{nd}}$.
Moreover, it is stable by left multiplication by elements in $Y$ since $Y$ normalizes $X^-$ and $X^+$.
Thus, $X^- X^+ Z = X$.

Now, let $g \in X \cap U^-$ and write it as $g = x^- x^+ z$ with $x^- \in X^-$, $x^+ \in X^+$ and $z \in Z$.
Then $x^+ z = \left(x^-\right)^{-1} g \in U^-$.
Using a Bruhat decomposition \cite[6.1.15 (c)]{BruhatTits1}, we have $z=1$ since $N \to U^+ \backslash G / U^-$ is a bijection.
Hence $x^+ \in U^+ \cap U^- = \{1\}$ by axiom~\ref{axiomRGD6}.
Thus we get $X \cap U^- = X^-$.
This proves the surjectivity of the map $X^- \to X \cap U^-$ which is also injective by restriction of a bijection.
By symmetry, the same holds for the map $X^+ \to X \cap U^+$ and we get~\ref{PropDecQC:2}.
We deduce~\ref{PropDecQC:1} from~\ref{PropDecQC:2} by intersection with $U_\alpha$ for $\alpha \in \Phi_{\mathrm{nd}}$.

If $n \in X \cap N$, write it as $n = x^- x^+ z$ with $x^-\in U^-$, $x^+\in U^+$ and $z\in Z$.
Then $n = z$ using the Bruhat decomposition \cite[6.1.15 (c)]{BruhatTits1}.
This proves $Z = X \cap N$ which is~\ref{PropDecQC:4}
and, therefore, we deduce~\ref{PropDecQC:3}.
\end{proof}

Using a valuation of a root group datum, we will apply the proposition above to various examples of quasi-concave families of groups, see Example~\ref{ExQC} or Proposition~\ref{PropUstarQC}.

\subsection{\texorpdfstring{$\Rtot$}{R}-valuation of a root group datum}\label{subsection_Rvaluation}

In the following, we will assume that $\Rtot$ is a non-zero totally ordered abelian group.

When $\Rtot \neq \mathbb{R}$, there is no reason for $\Rtot$ to satisfy the least-upper-bound property. In particular, in this work, we avoid to introduce a notion of infimum and supremum in $\Rtot$. 
We could maybe do this by considering  the totally ordered monoid of the convex subsets of $\Rtot$ containing $\infty$ but it is not easy to manipulate.
Obviously, for $\Rtot = \mathbb{R}$, such a monoid has been introduced in \cite[6.4.1]{BruhatTits1}.
This difference firstly appears in the definition of the subsets of values $\Gamma'_\alpha$ (see Notation~\ref{not:Gamma}).

The definition of a valuation of a root group datum, given by \cite[6.2.1]{BruhatTits1}, can be naturally extended as follows:
\begin{Def}\label{DefValuation}
Let $\Phi$ be a root system and $(T,(U_\alpha, M_\alpha)_{\alpha \in \Phi})$ be a root group datum.
An $\Rtot$-valuation\index{R-valuation@$R$-valuation}\index{valuation!of a root group datum}\index{root group datum!valuation} of the root group datum is a family $(\varphi_\alpha)_{\alpha \in \Phi}$ of maps $\varphi_\alpha : U_\alpha \to \Rtot \cup \{\infty\}$\index[notation]{p@$\varphi_\alpha$} satisfying the following axioms:
\begin{enumerate}[label={(V\arabic*)}]
\setcounter{enumi}{-1}
\item\label{axiomV0} for any $\alpha \in \Phi$, the set $\varphi_\alpha(U_\alpha)$ contains at least $3$ elements;\axiom{V0@\ref{axiomV0}}
\item\label{axiomV1}  for any $\alpha \in \Phi$ and $\lambda \in \Rtot \cup \{\infty\}$, the set $U_{\alpha,\lambda} = \varphi^{-1}_\alpha( [\lambda,\infty] ) $\index[notation]{u@$U_{\alpha,\lambda}$} is a subgroup of $U_\alpha$ and $U_{\alpha,\infty} = \{1\}$;\axiom{V1@\ref{axiomV1}}
\item\label{axiomV2}  for any $\alpha \in \Phi$ and $m \in M_\alpha$, the map $U_{-\alpha} \setminus \{1\} \to \Rtot$ defined by $u \mapsto \varphi_{-\alpha}(u) - \varphi_{\alpha}(mum^{-1})$ is constant;\axiom{V2@\ref{axiomV2}}
\item\label{axiomV3}  for any $\alpha,\beta \in \Phi$ such that $\beta \not\in \mathbb{R}_{\leqslant 0} \alpha$ and any $\lambda,\mu \in \Rtot$, the commutator group $[U_{\alpha,\lambda},U_{\beta,\mu}]$ is contained in the group generated by the $U_{r\alpha + s\beta, r \lambda + s \mu}$ for $r,s \in \mathbb{Z}_{>0}$ such that $r\alpha+s\beta \in \Phi$;\axiom{V3@\ref{axiomV3}}
\item\label{axiomV4}  for any multipliable root $\alpha \in \Phi$, the map $\varphi_{2\alpha}$ is the restriction of the map $2 \varphi_\alpha$ to $U_{2\alpha}$;\axiom{V4@\ref{axiomV4}}
\item\label{axiomV5}  for any $\alpha \in \Phi$ and $u \in U_{\alpha}$, for any $u',u'' \in U_{-\alpha}$ such that $u'uu'' \in M_\alpha$, we have $\varphi_{-\alpha}(u') = - \varphi_\alpha(u)$.\axiom{V5@\ref{axiomV5}}
\end{enumerate}
\end{Def}

It is convenient to introduce notation of the trivial subgroup $U_{2\alpha,\lambda} = \{1\}$ for any $\alpha \in \Phi$ such that $2\alpha \not\in \Phi$ and $\lambda \in \Rtot \cup \{\infty\}$.

As in Bruhat-Tits theory, in section~\ref{SecQuasiSplitGroups}, we will use  Chevalley-Steinberg systems in order to provide such a valuation. Namely, $\Rtot$ will be the abelian group $\mathfrak{R}^S$ so that $\Lambda = \omega(\mathbb{K}^*)$ will be canonically identified to a subset of $\Rtot$.
For instance, if $\mathbf{G}$ is split, then the root groups $\mathbf{U}_\alpha$ are isomorphic to $\mathbb{G}_a$. Thus, the pinnings of these groups give isomorphisms $x_\alpha: \mathbb{K} \to \mathbf{U}_\alpha(\mathbb{K})$ and one can define $\varphi_\alpha: \mathbf{U}_\alpha(\mathbb{K}) \to \Lambda \subset \mathfrak{R}^S$ by $\varphi_\alpha \circ x_\alpha = \omega$.

\begin{Lem}\label{LemAxiomV1}
Axiom~\ref{axiomV1} is equivalent to the following axiom:
\begin{enumerate}[label={(V\arabic*bis)}]
\setcounter{enumi}{0}
\item\label{axiomV1bis} for any $\alpha \in \Phi$ and any $u,v \in U_{\alpha}$, we have $\varphi_{\alpha}(uv^{-1}) \geqslant \min(\varphi_\alpha(u), \varphi_\alpha(v))$ and $\varphi_\alpha^{-1}(\{\infty\}) = \{1\}$.\axiom{V1bis@\ref{axiomV1bis}}
\end{enumerate}

In particular, for any $\alpha \in \Phi$,
\begin{enumerate}[label={(\arabic*)}]
\item\label{LemAxiomV1:1} for any $u \in U_\alpha$, we have $\varphi_\alpha(u^{-1}) = \varphi_\alpha(u)$;
\item\label{LemAxiomV1:2} for any $u,v \in U_\alpha$ such that $\varphi_\alpha(v) > \varphi_\alpha(u)$, we have $\varphi_\alpha(uv) = \varphi_\alpha(u)=\varphi_\alpha(vu)$.
\end{enumerate}
\end{Lem}

\begin{proof}Consider $\alpha \in \Phi$.
By definition, $U_{\alpha,\infty} = \{1\}$ is equivalent to $\varphi_\alpha^{-1}(\{\infty\}) = \{1\}$.

Suppose axiom~\ref{axiomV1}.
Consider $u,v \in U_{\alpha}$ and let $\lambda = \min(\varphi_\alpha(u), \varphi_\alpha(v))$.
Hence $U_{\alpha,\lambda}$ is a subgroup containing $u,v$ since $\varphi_\alpha(u) \geqslant \lambda$ and  $\varphi_\alpha(v) \geqslant \lambda$.
Therefore, $uv^{-1} \in U_{\alpha,\lambda}$ gives us $\varphi_{\alpha}(uv^{-1}) \geqslant \lambda$.

Conversely, suppose axiom~\ref{axiomV1bis}.
Consider $\lambda \in \Rtot$ and pick $u,v \in U_{\alpha,\lambda}$.
Then $\varphi_\alpha(uv^{-1}) \geqslant \min(\varphi_\alpha(u),\varphi_\alpha(v)) \geqslant \lambda$.
Hence $uv^{-1} \in U_{\alpha,\lambda}$ and this proves that $U_{\alpha,\lambda}$ is a subgroup of $U_\alpha$.

\ref{LemAxiomV1:1} Hence, for any $v \in U_\alpha$, if we take $u=1$, then $\varphi_\alpha(v^{-1}) = \varphi_\alpha(uv^{-1}) \geqslant \min(\varphi_\alpha(1),\varphi_\alpha(v)) = \varphi_\alpha(v)$ and this inequality is also true for $v^{-1}$ instead of $v$.

\ref{LemAxiomV1:2} We have $\varphi_\alpha(u) = \varphi_\alpha((uv)v^{-1}) \geqslant \min(\varphi_\alpha(uv),\varphi_\alpha(v^{-1})) \geqslant \min(\varphi_\alpha(u), \varphi_\alpha(v),\varphi_\alpha(v^{-1})) = \varphi_\alpha(u)$ whenever  $\varphi_\alpha(v^{-1}) = \varphi_\alpha(v) > \varphi_\alpha(u)$.
\end{proof}

\begin{Lem}\label{LemAxiomV5}
Under the assumption of axiom~\ref{axiomV1bis}, axiom~\ref{axiomV5} is equivalent to the following axiom:
\begin{enumerate}[label={(V\arabic*bis)}]
\setcounter{enumi}{4}
\item\label{axiomV5bis} for any $\alpha \in \Phi$ and $u \in U_{\alpha}$, for any $u',u'' \in U_{-\alpha}$ such that $u'uu'' \in M_\alpha$, we have $\varphi_{-\alpha}(u'') = - \varphi_\alpha(u)$.\axiom{V5bis@\ref{axiomV5bis}}
\end{enumerate}
\end{Lem}

\begin{proof}
Let $u \in U_\alpha$ and $u',u''\in U_{-\alpha}$ such that $u'uu'' \in M_\alpha$.
Then  by \cite[6.1.2(4)]{BruhatTits1}, we know that $(u'')^{-1} u^{-1} (u')^{-1} \in M_\alpha$.
Hence by axiom~\ref{axiomV5}, we have $\varphi_{-\alpha}((u'')^{-1}) = \varphi_{\alpha}(u^{-1})$.
By Lemma~\ref{LemAxiomV1}, we have $\varphi_{-\alpha}((u'')^{-1}) = \varphi_{-\alpha}(u'')$ and $\varphi_{\alpha}((u)^{-1}) = \varphi_{\alpha}(u)$.
Hence $\varphi_{-\alpha}(u'') = \varphi_\alpha(u)$ which gives us axiom~\ref{axiomV5bis}.
We get the converse by symmetry.
\end{proof}

In all the following, we assume that a root group datum $(T,(U_\alpha, M_\alpha)_{\alpha \in \Phi})$ and an $\Rtot$-valuation $(\varphi_\alpha)_{\alpha \in \Phi}$ are given.
When $\alpha \in \Phi$ is such that $2\alpha \not\in \Phi$, we define $U_{2\alpha}  = \{1\}$.

The valuation enables us to  introduce the following sets of values:
\begin{Not}\label{not:Gamma}
For any root $\alpha \in \Phi \cup 2\Phi$, we define the following subsets, called \textbf{sets of values} associated to $\alpha$,\index{set of values} of $\Rtot$:
\begin{itemize}
\item $\Gamma_\alpha = \varphi_\alpha(U_\alpha \setminus \{1\})$;\index[notation]{g@$\Gamma_\alpha$}
\item $\displaystyle \Gamma'_\alpha = \Big\{\varphi_\alpha(u),\ u \in U_\alpha \setminus\{1\} \text{ and } U_{\alpha,\varphi_\alpha(u)} = \bigcap_{v \in U_{2\alpha}} U_{\alpha,\varphi_\alpha(u v)} \Big\} \subset \Gamma_\alpha$.\index[notation]{g@$\Gamma'_\alpha$}
\end{itemize}
\end{Not}

\begin{Rqs}\label{rqGamma_alpha}~\vspace{-1em} \

\begin{enumerate}[ref={\theproposition.\arabic*}]
\item\label{RkOppositeSetOfValue} From axiom~\ref{axiomV5} and \cite[6.1.2(2)]{BruhatTits1}, we deduce $\Gamma_{-\alpha} = - \Gamma_\alpha$.
\item\label{RkGamma_non_multipliable} By definition, if $2\alpha \not\in \Phi$, we have $\Gamma'_\alpha = \Gamma_\alpha$.
\item As in the Bruhat-Tits theory, when $\alpha$ is multipliable, it may happen that $\Gamma'_\alpha$ is empty (dense valuation); it may happen that the intersection $2\Gamma'_\alpha \cap \Gamma_{2\alpha}$ is non-empty (discrete valuation with unramified splitting extension $\widetilde{\mathbb{K}} / \mathbb{K}$).
\end{enumerate}
\end{Rqs}

\begin{Fact}\label{FactDecompositionSetOfValue} From axioms~\ref{axiomV1} and~\ref{axiomV4}, we deduce $2\Gamma_\alpha = 2\Gamma'_\alpha \cup \Gamma_{2\alpha}$.
\end{Fact}

\begin{proof}
By definition $2\Gamma'_\alpha \subset 2\Gamma_\alpha$ and, by axiom~\ref{axiomV4}, we have $\Gamma_{2\alpha} \subset 2\Gamma_\alpha$.

Conversely, let $\lambda \in 2\Gamma_\alpha \setminus \Gamma_{2\alpha}$.
Let $u \in U_\alpha$ such that $2 \varphi_\alpha(u) = \lambda$.
Then for any $v \in U_{2\alpha}$, we have $\lambda \neq \varphi_{2\alpha}(v)$ by definition.
Thus $2 \varphi_\alpha(u) = \lambda \neq 2 \varphi_{\alpha}(v)$ by axiom~\ref{axiomV4} and therefore $\varphi_\alpha(u) \neq \varphi_\alpha(v)$ since $\Rtot$ is $\mathbb{Z}$-torsion free.
If $\varphi_\alpha(v) > \varphi_\alpha(u)$, we have $\varphi_\alpha(uv) = \varphi_\alpha(u)$ by Lemma~\ref{LemAxiomV1}\ref{LemAxiomV1:2}.
Thus $U_{\alpha,\varphi_\alpha(u)} = U_{\alpha,\varphi_\alpha(uv)}$.
If $\varphi_\alpha(v) < \varphi_\alpha(u)$, we have $\varphi_\alpha(uv) = \varphi_\alpha(v) < \varphi_\alpha(u)$ by Lemma~\ref{LemAxiomV1}\ref{LemAxiomV1:1} and~\ref{LemAxiomV1:2}.
Thus $U_{\alpha,\varphi_\alpha(u)} \subset U_{\alpha,\varphi_\alpha(uv)}$.
Hence we have $U_{\alpha,\varphi_\alpha(u)} \subset \bigcap_{v \in U_{2\alpha}} U_{\alpha,\varphi_{\alpha}(uv)}$ and this is, in fact, an equality by considering $v = 1 \in U_{2\alpha}$.
Thus $\varphi_\alpha(u) \in \Gamma'_\alpha$.
Therefore $\lambda = 2 \varphi_\alpha(u) \in 2 \Gamma'_\alpha$.
\end{proof}

\begin{Not}\label{NotValuedCoset}
For $\alpha \in \Phi$ and $\lambda \in \Rtot$, we denote:
\[M_{\alpha,\lambda} = M_\alpha \cap U_{-\alpha} \varphi_{\alpha}^{-1}(\{\lambda\}) U_{-\alpha}.\]\index[notation]{m@$M_{\alpha,\lambda}$}
\end{Not}

We provide some details of \cite[6.2.2]{BruhatTits1}:

\begin{Prop}\label{PropValuedCoset}
Let $\alpha \in \Phi$ and $\lambda \in \Rtot$.
\begin{enumerate}[label={(\arabic*)}]
\item\label{PropValuedCoset:emptyness} $M_{\alpha,\lambda}$ is non-empty if, and only if, $\lambda \in \varphi_{\alpha}(U_\alpha \setminus \{1\}) = \Gamma_\alpha$;
\item\label{PropValuedCoset:reverselambda} $\varphi_\alpha^{-1}(\{\lambda\}) \subset U_{-\alpha,-\lambda} M_{\alpha,\lambda} U_{-\alpha,-\lambda}$;
\item\label{PropValuedCoset:insideLevi} $M_{\alpha,\lambda} \subset \varphi_{-\alpha}^{-1}(\{-\lambda\})\varphi_{\alpha}^{-1}(\{\lambda\})\varphi_{-\alpha}^{-1}(\{-\lambda\}) \subset U_{-\alpha,-\lambda} U_{\alpha,\lambda} U_{-\alpha,-\lambda}$;
\item\label{PropValuedCoset:opposite} $M_{-\alpha,-\lambda} = M_{\alpha,\lambda}$;
\item\label{PropValuedCoset:multipliable} $M_{2\alpha,2\lambda} \subset M_{\alpha, \lambda}$.
\end{enumerate}

\end{Prop}

\begin{proof}
\ref{PropValuedCoset:emptyness} is a consequence of \cite[6.1.2(2)]{BruhatTits1}, since $\lambda \neq \infty$.

\ref{PropValuedCoset:reverselambda} For any $u \in \varphi_{\alpha}^{-1}(\{\lambda\})$, axioms~\ref{axiomRGD1} and~\ref{axiomRGD4} provide elements $u',u'' \in U_{-\alpha}$ such that $m:=u'uu'' \in M_{\alpha,\lambda}$.
By axioms~\ref{axiomV1bis},~\ref{axiomV5} and~\ref{axiomV5bis}, we have $\varphi_{-\alpha}((u')^{-1}) = - \varphi_\alpha(u) = \varphi_{-\alpha}((u'')^{-1}) = -\lambda$. 

\ref{PropValuedCoset:insideLevi} For any $m \in M_{\alpha,\lambda}$, by definition, there exist $u',u'' \in U_{-\alpha}$ and $u \in U_{\alpha}$ such that $m = u'uu''$ and $\varphi_\alpha(u) = \lambda$.
By axioms~\ref{axiomV5} and~\ref{axiomV5bis}, we have $\varphi_{-\alpha}(u') = - \varphi_\alpha(u) = \varphi_{-\alpha}(u'') = -\lambda$.

\ref{PropValuedCoset:opposite}  For any $\alpha \in \Phi$ and $\lambda \in \Rtot$, consider $m \in M_{-\alpha,-\lambda}$.
By~\ref{PropValuedCoset:insideLevi} and axiom~\ref{axiomV1bis}, there exist $u',u'' \in U_{\alpha}$ and $u \in U_{-\alpha}$ such that $\varphi_{-\alpha}(u) = -\lambda = - \varphi_\alpha(u') = - \varphi_\alpha(u'')$ and $m = u'uu''$.
Consider $v=mu''^{-1}m^{-1}$ so that $u' = v m u^{-1}$.
By \cite[6.2.1(2)]{BruhatTits1} $v \in U_{-\alpha}$ and by \cite[6.2.1(4)]{BruhatTits1}, $m \in M_\alpha$.
Hence $m = v^{-1} u' u^{-1} \in M_\alpha \cap U_{-\alpha} \varphi_{\alpha}^{-1}(\{\lambda\}) U_{-\alpha} = M_{\alpha,\lambda}$.
Hence $M_{-\alpha,-\lambda} \subset M_{\alpha,\lambda}$ and we get the converse inclusion by exchanging $(\alpha,\lambda)$ with $(-\alpha,-\lambda)$.

\ref{PropValuedCoset:multipliable}
If $\alpha$ is multipliable, then $U_{2\alpha} \subset U_{\alpha}$ and $U_{-2\alpha} \subset U_{-\alpha}$ by axiom~\ref{axiomRGD3} and $M_{2\alpha} = M_\alpha$ by \cite[6.1.2(4)]{BruhatTits1}.
For $u \in \varphi_{2\alpha}^{-1}(\{2\lambda\})$, by axiom~\ref{axiomV4}, we have $2 \lambda = \varphi_{2\alpha}(u) = 2 \varphi_\alpha(u)$.
Since $\Rtot$ is $\mathbb{Z}$-torsion free, it gives $u \in \varphi_{\alpha}^{-1}(\{\lambda\})$.
\end{proof}

Here, we follow a different strategy than in \cite[6.2]{BruhatTits2} and we do not work with the notion of ``valuations équipollentes''.

We introduce the following useful Lemma from \cite[7.5]{Landvogt} with a different proof since we do not define integral models here:

\begin{Lem}\label{LemConjugationMaUb}
Let $\alpha \in \Phi_{\mathrm{nd}}$, $\lambda \in \Gamma_\alpha$, $\beta \in \Phi$ and $\mu \in \Rtot$.
For any $m \in M_{\alpha,\lambda}$, we have
\[ m U_{\beta,\mu} m^{-1} = U_{r_\alpha(\beta),\mu - \beta(\alpha^\vee) \lambda}.\]
In particular, we have $m U_{\alpha,\lambda} m^{-1} = U_{-\alpha,-\lambda}$.
\end{Lem}

\begin{proof}
It suffices to prove that $m U_{\beta,\mu} m^{-1} \subset U_{r_\alpha(\beta),\mu - \beta(\alpha^\vee) \lambda}$.
Indeed if we know this inclusion for all $\alpha \in \Phi_{\mathrm{nd}}$, $\lambda \in \Gamma_\alpha$, $\beta \in \Phi$ and $\mu \in \Rtot$, we can apply it to $\beta':=r_{\alpha}(\beta)$ and $\mu':=\mu-\beta(\alpha^\vee)\lambda$ instead of $\beta$ and $\mu$, and we get the reverse inclusion.

We distinguish three cases on $\beta$.

\paragraph{First case: $\beta \in \Phi_{\mathrm{nd}} \setminus \mathbb{R}\alpha $.}
Define  $t : \Phi(\alpha,\beta) \to \Rtot\cup\{\infty\}$ by $t(r\alpha+s\beta) = r\lambda+s\mu$ if $s > 0$ and $t(\gamma) = \infty$ if $s \leqslant 0$, for $r,s\in \Z$ such that $r\alpha+s\beta\in \Phi(\alpha,\beta)$.

Denote by $X_\gamma = U_{\gamma,t(\gamma)}$.
Then the family of groups $(X_\gamma)_{\gamma \in \Phi(\alpha,\beta)}$ is quasi-concave.
Indeed, for every $\gamma \in \Phi(\alpha,\beta)$ we have either $X_{2\gamma} = X_\gamma = 1$ or $X_{-2\gamma} = X_{-\gamma} = 1$ so that axiom~\ref{axiomQC1} is satisfied.
Let $r_1,s_1,r_2,s_2 \in \mathbb{Z}$ be such that $\gamma_1 = r_1 \alpha + s_1 \beta$ and $\gamma_2 = r_2 \alpha + s_2 \beta$ are in $\Phi(\alpha,\beta)$ with $\gamma_2 \not \in \mathbb{R}_{<0} \gamma_1$.
If $s_1 \leqslant 0$ or $s_2 \leqslant 0$, then we have $[X_{\gamma_1}, X_{\gamma_2}] = 1$.
Suppose $s_1,s_2 > 0$. Let $(r,s)\in (\Z_{\geq 0})^2\setminus \{0\}$ be such that $r\gamma_1+s\gamma_2\in \Phi(\alpha,\beta)$.
Then $r t(\gamma_1) + s t(\gamma_2) = t(r \gamma_1 + s \gamma_2)$ so that axiom~\ref{axiomQC2} is satisfied according to axiom~\ref{axiomV3}.

Let $Z$ be the group generated by the $X_\gamma$ for $\gamma \in \Phi(\alpha,\beta)$.
For any $\gamma \in \Phi(\alpha,\beta)$, we have $X_\gamma = 1$ if $\gamma \in \mathbb{R} \alpha$ and $[U_{\alpha,\lambda},X_\gamma] \subset Z$ by axiom~\ref{axiomV3} otherwise.
Thus $U_{\alpha,\lambda}$ normalizes $Z$.
By the same way, $U_{-\alpha,-\lambda}$ normalizes $Z$ so that $M_{\alpha,\lambda}$ normalizes $Z$ by Proposition~\ref{PropValuedCoset}~\ref{PropValuedCoset:insideLevi}.
Moreover, for every $m \in M_{\alpha,\lambda}$, we have $m U_\alpha m^{-1} = U_{r_\alpha(\beta)}$ by \cite[6.1.2(10)]{BruhatTits1}.
Thus $m U_{\beta,\mu} m^{-1} \subset U_{r_\alpha(\beta)} \cap Z = U_{r_{\alpha}(\beta),t(r_{\alpha}(\beta))} U_{2 r_\alpha(\beta), 2 t(r_\alpha(\beta))}$ by Proposition~\ref{PropDecompositionQC}\ref{PropDecQC:1}.
Finally, $U_{2 r_\alpha(\beta), t(2 r_\alpha(\beta))} \subset U_{r_{\alpha}(\beta),t(r_{\alpha}(\beta))}$ by definition of $t$ since $t(2 r_\alpha(\beta)) = 2 t(r_\alpha(\beta))$.

\paragraph{Second case: $\beta \in \mathbb{R}\alpha \cap \Phi_{\mathrm{nd}} = \{\pm \alpha\}$}
Since $M_{\alpha,\lambda} = M_{-\alpha,-\lambda}$, $r_\alpha = r_{-\alpha}$ and $(-\alpha)^\vee = -\alpha^\vee$ it suffices to do it for $\beta=-\alpha$.
Let $m \in M_{\alpha,\lambda}$ and write it as $m=u' u u''$ with $u',u'' \in \varphi_{-\alpha}^{-1}(\{-\lambda\})$ and $u \in \varphi_\alpha^{-1}(\{\lambda\})$ which is possible by Proposition~\ref{PropValuedCoset}~\ref{PropValuedCoset:insideLevi}.
Then $u' = m(u'')^{-1}m^{-1} m u^{-1}$ with $m(u'')^{-1} m^{-1} \in U_{\alpha}$.
Thus, axiom~\ref{axiomV5} applied to $(mu''m^{-1})u'u= m \in M_{-\alpha}$ gives us $\varphi_{\alpha}(mu''m^{-1}) = -\varphi_{-\alpha}(u') = \lambda$.
Let $v \in U_{-\alpha,\mu} \setminus \{1\}$.
Then axiom~\ref{axiomV2} gives us $\varphi_{-\alpha}(v) - \varphi_{\alpha}(mvm^{-1}) = \varphi_{-\alpha}(u'') - \varphi_{\alpha}(mu''m^{-1}) = - 2 \lambda$.
Hence $\varphi_{\alpha}(mvm^{-1}) = \varphi_{\alpha}(v) + 2 \lambda \geqslant \mu + 2 \lambda = \mu - \beta(\alpha^\vee) \lambda$.
Hence $m U_{\beta,\mu} m^{-1} \subset U_{r_\alpha(\beta),\mu-\beta(\alpha^\vee)\lambda}$.

\paragraph{Third case: $\beta \in \Phi \setminus \Phi_{\mathrm{nd}}$:}
By \cite[6.1.2 (10)]{BruhatTits1} and the two previous cases, we have $m U_{\beta,\mu} m^{-1} \subset m U_{\frac{1}{2}\beta,\frac{1}{2}\mu} m^{-1} \cap U_{r_\alpha(\beta)} \subset  U_{r_\alpha(\beta), \mu -  \beta(\alpha^\vee) \lambda}$.

\end{proof}

\begin{Prop}[{see \cite[6.2.7]{BruhatTits1}}]\label{PropWeylSetOfValues}
Let $\alpha \in \Phi_{\mathrm{nd}}$ and $\beta \in \Phi$.

For any $\lambda \in \Gamma_\alpha$, any $m \in M_{\alpha,\lambda}$ and any $u \in U_\beta$, we have
\[\varphi_{r_\alpha(\beta)}(mum^{-1}) = \varphi_\beta(u) - \beta(\alpha^\vee) \lambda.\]

In particular, we have
\[ \Gamma_{r_\alpha(\beta)} = \Gamma_\beta - \beta(\alpha^\vee) \Gamma_\alpha.\]
\end{Prop}

\begin{proof}
Set $\gamma = r_\alpha(\beta)$,  $\mu = \varphi_\beta(u)$ and  $v=mum^{-1}$.
By Lemma~\ref{LemConjugationMaUb}, we have $v \in U_{\gamma, \mu - \beta(\alpha^\vee) \lambda}$.
Thus $\varphi_\gamma(v) \geqslant \varphi_\beta(u) - \beta(\alpha^\vee)\lambda$.
Since $m^{-1} \in M_{\alpha,\lambda}$, exchanging the roles of $u$ and $v$ in this inequality, we get
$\varphi_{r_\alpha(\gamma)}(m^{-1}vm) = \varphi_\beta(u) \geqslant \varphi_\gamma(v) - \gamma(\alpha^\vee) \lambda$.
Since $\gamma(\alpha^\vee) = r_\alpha(\beta)(\alpha^\vee) = - \beta(\alpha^\vee)$, we get the equality.

The equality $\Gamma_{r_\alpha(\beta)} = \Gamma_\beta - \beta(\alpha^\vee) \Gamma_\alpha$ becomes clear since $\lambda, \mu = \varphi_\beta(u), \nu = \varphi_{r_\alpha(\beta)}(mum^{-1})$ run through all the values of $\Gamma_\alpha, \Gamma_\beta, \Gamma_{r_\alpha(\beta)}$ respectively.
\end{proof}

\begin{Not}
For any $\alpha \in \Phi$, we denote by $\widetilde{\Gamma}_\alpha = \langle \Gamma_\alpha - \Gamma_\alpha \rangle$\index[notation]{g@$\widetilde{\Gamma}_\alpha$} the subgroup of $\Rtot$ spanned by $\{x-y,\ x,y \in \Gamma_\alpha\}$.
Obviously, we have that $\Gamma_\alpha \subset \widetilde{\Gamma}_\alpha $ whenever $0 \in \Gamma_\alpha$.
\end{Not}

\begin{Cor}\label{CorWeylSetOfValues}
Let $\alpha \in \Phi_{\mathrm{nd}}$ and $\beta \in \Phi$.
Then $\Gamma_\beta = \Gamma_{\beta} - \beta(\alpha^\vee) \widetilde{\Gamma}_\alpha$.

If $\Gamma_\alpha \subset \widetilde{\Gamma}_\alpha$, then $\Gamma_{r_\alpha(\beta)} =\Gamma_\beta$.
\end{Cor}

\begin{proof}
Applying Proposition~\ref{PropWeylSetOfValues} twice, we get that
\[ \Gamma_{r_\alpha(r_\alpha(\beta))} = \Gamma_{r_\alpha(\beta)} - r_\alpha(\beta)(\alpha^\vee) \Gamma_\alpha = \Gamma_\beta - \beta(\alpha^\vee) \Gamma_\alpha + \beta(\alpha^\vee) \Gamma_\alpha\]
so that $\Gamma_\beta = \Gamma_\beta - \beta(\alpha^\vee) \left(\Gamma_\alpha - \Gamma_\alpha\right)$.
We conclude by iterating this equality.

If $\Gamma_\alpha \subset \widetilde{\Gamma}_\alpha$, then $\Gamma_{r_\alpha(\beta)} = \Gamma_\beta - \beta(\alpha^\vee) \Gamma_\alpha \subset \Gamma_\beta - \beta(\alpha^\vee) \widetilde{\Gamma}_\alpha = \Gamma_\beta$ and we conclude exchanging the roles of $\beta=r_\alpha(r_\alpha(\beta))$ and $r_\alpha(\beta)$.
\end{proof}

\begin{Cor}[{see \cite[6.2.16]{BruhatTits1}}]\label{cor_Symmetry_properties_Gamma_alpha}\label{CorGamma-alpha}
For any $\alpha \in \Phi$, we have $\Gamma_{\alpha} = - \Gamma_{-\alpha}$ and, if $\Gamma_\alpha \subset \widetilde{\Gamma}_\alpha$, then $\Gamma_{\alpha} = \Gamma_{-\alpha} = - \Gamma_\alpha$.
\end{Cor}

\begin{proof}
If $\alpha$ is divisible, set $\beta = \frac{1}{2}\alpha$; otherwise, set $\beta = \alpha$ so that $\beta$ is non-divisible.
Then $r_\beta(\alpha) = -\alpha$ and $\alpha(\beta^\vee) \Gamma_\beta \supset 2 \Gamma_\alpha$ according to Fact~\ref{FactDecompositionSetOfValue}.
Hence, applying Proposition~\ref{PropWeylSetOfValues}, we get that
\begin{equation}\label{eqnInclusionGamma2}
\Gamma_{-\alpha} = \Gamma_{r_\beta(\alpha)} = \Gamma_{\alpha} - \alpha(\beta^\vee) \Gamma_\beta \supset \Gamma_\alpha - 2 \Gamma_\alpha \supset - \Gamma_\alpha
\end{equation}
Thus $- \Gamma_\alpha \subset \Gamma_{-\alpha}$ and we deduce an equality from this inclusion by exchanging the roles of $\alpha$ and $-\alpha$.
In particular, $\widetilde{\Gamma}_\alpha=\widetilde{\Gamma}_{-\alpha}$.
Applying (\ref{eqnInclusionGamma2}) twice, we get
\[\Gamma_{-\alpha} \supset \Gamma_{\alpha} - 2 \Gamma_{\alpha} \supset (\Gamma_{-\alpha} - 2 \Gamma_{-\alpha}) - 2 \Gamma_{\alpha}=\Gamma_{-\alpha}-2(\Gamma_{-\alpha}+\Gamma_{\alpha}) = \Gamma_{-\alpha} - 2 (\Gamma_{-\alpha} - \Gamma_{-\alpha}).\]
We deduce that $\Gamma_{-\alpha}=\Gamma_{-\alpha}+2\tilde{\Gamma}_{-\alpha}$.
If $\Gamma_\alpha \subset \widetilde{\Gamma}_\alpha$, then
$\Gamma_\alpha \subset -\Gamma_\alpha + 2 \Gamma_\alpha \subset \Gamma_{-\alpha} + 2 \widetilde{\Gamma}_{-\alpha} = \Gamma_{-\alpha}$.
We deduce the equality by exchanging the roles of $\alpha$ and $-\alpha$.
\end{proof}

\begin{Def}
Let $(\varphi_\alpha)_{\alpha \in \Phi}$ be an $\Rtot$-valuation of a root group datum.
It is called \textbf{special}\index{R-valuation@$R$-valuation!special}\index{valuation!special}\index{special valuation} if $0 \in \Gamma_\alpha$ for every $\alpha \in \Phi_{\mathrm{nd}}$.
One says that it is \textbf{semi-special}\index{R-valuation@$R$-valuation!semi-special}\index{valuation!semi-special}\index{semi-special valuation} if for any root $\alpha \in \Phi_{\mathrm{nd}}$ we have $\Gamma_\alpha \subset \widetilde{\Gamma}_\alpha$.
\end{Def}

Note that, in this definition, there are no conditions for divisible roots.
The definition of special valuation is due to \cite[6.2.13]{BruhatTits1}.
As we will see later, there are some natural valuations that are not special but only semi-special, as shown in Example~\ref{ExSUhGalois}.
A special valuation is semi-special but the converse is not true in general.

\begin{Ex}
If $\Gamma_\alpha = 1 + 2 \mathbb{Z}$, then $\widetilde{\Gamma}_\alpha = 2 \mathbb{Z}$ does not contain $\Gamma_\alpha$.
Hence $1 + 2 \mathbb{Z}$ cannot be the set of values of $\alpha$ for a semi-special valuation.

If $\Gamma_\alpha = \bigg( \Big( \frac{1}{2} + \mathbb{Z} \Big) \times \mathbb{Z} \bigg) \cup \bigg(\mathbb{Z} \times \Big(1 + 2 \mathbb{Z}\Big)\bigg)$, then $\widetilde{\Gamma}_\alpha = \frac{1}{2} \mathbb{Z} \times \mathbb{Z} \supset \Gamma_\alpha$ while $0 \not\in \Gamma_\alpha$.
Hence if there exists a valuation for which $\Gamma_\alpha$ is a set of values for $\alpha$, then this valuation is necessarily semi-special but non-special.
The existence of such a valuation is provided by Example~\ref{ExGaloisRamified} applied to the group $\mathrm{SU}(h)$ described in section~\ref{subsecParametrization} (see Example~\ref{ExSUhGalois}).
\end{Ex}

\begin{Rq}[{see \cite[6.2.17]{BruhatTits1}}]
Let $\widetilde{\Gamma}$ be a torsion-free abelian group and consider the quotient $\pi : \widetilde{\Gamma} \to \widetilde{\Gamma} / 2 \widetilde{\Gamma} =: \widetilde{X}$. 
Then $\widetilde{X}$ is a $\mathbb{F}_2$-vector space.
Let $X$ be a non-empty generating set of $\widetilde{X}$ and define $\Gamma=\pi^{-1}(X)$.
Then $\Gamma = \Gamma + 2 \widetilde{\Gamma}$.
Moreover,  $\langle \Gamma-\Gamma \rangle = 2 \widetilde{\Gamma}$ if $X$ contains at most on point and $\langle \Gamma- \Gamma \rangle = \widetilde{\Gamma}$ otherwise.

Conversely, let $\Gamma$ be any nonzero subset of a torsion-free abelian group such that $\Gamma = \Gamma + 2 \widetilde{\Gamma}$ with $\widetilde{\Gamma} = \langle \Gamma-\Gamma \rangle$. Then $\Gamma $ is built in this way with $X = \pi(\Gamma) \subset \widetilde{X} = \langle \Gamma \rangle / 2 \widetilde{\Gamma}$.

By this way, we have described all the possible sets of a semi-special valuation and we observe that discrete semi-special valuations with values in $\mathbb{R}$ are always special.
\end{Rq}

\begin{Prop}[{see \cite[6.2.14]{BruhatTits1}}]\label{PropWeylInvarianceSetOfValues}
Let $w \in W(\Phi)$ and $\alpha \in \Phi$.
If the valuation $(\varphi_\alpha)_{\alpha \in \Phi}$ is semi-special, then $\Gamma_{w(\alpha)} = \Gamma_{\alpha}$.
If the valuation is special, then $\Gamma'_{w(\alpha)} = \Gamma'_{\alpha}$.
\end{Prop}

\begin{proof}
Since $W(\Phi)$ is generated by the $r_\alpha$ for $\alpha \in \Delta$ for a basis $\Delta \subset \Phi$, it suffices to prove it for $w = r_\alpha$ with $\alpha \in \Delta$.

Applying Corollary~\ref{CorWeylSetOfValues}, we have $\Gamma_\beta = \Gamma_{r_\alpha(\beta)}$ for any $\alpha \in \Delta \subset \Phi_{\mathrm{nd}}$ and any $\beta \in \Phi$.
Thus $\Gamma_{w(\beta)} = \Gamma_\beta$ for any $w \in W(\Phi)$ and any $\beta \in \Phi$.

For a multipliable root $\beta$, it remains to compare $\Gamma'_\beta$ and $\Gamma'_{r_\alpha(\beta)}$.
If the valuation is special, there exists $m \in M_{\alpha,0}$.
For any $u \in U_\beta$ and any $u' \in U_{2\beta}$, we have 
\[\varphi_\beta(uu') = \varphi_{r_\alpha(\beta)}((mum^{-1})(mu'm^{-1}))
\qquad \text{ and }\qquad
\varphi_\beta (u)=\varphi_{r_\alpha(\beta)}(mum^{-1})\]
by Proposition~\ref{PropWeylSetOfValues}.
Thus $m U_{\beta, \varphi_{\beta}(uu')} m^{-1} = U_{r_\alpha(\beta),\varphi_{r_\alpha(\beta)}(vv')}$ where $v= mum^{-1} \in U_{r_\alpha(\beta)}$ and $v' = mu'm^{-1} \in U_{2r_\alpha(\beta)}$.
Thus
\begin{align*}
U_{\beta,\varphi_\beta(u)} = \bigcap_{u' \in U_{2\beta}} U_{\beta,\varphi_\beta(uu')}
& \Longleftrightarrow mU_{\beta,\varphi_\beta(u)}m^{-1}  = \bigcap_{u' \in U_{2\beta}} m U_{\beta,\varphi_\beta(uu')} m^{-1}\\
& \Longleftrightarrow U_{r_\alpha(\beta),\varphi_{r_\alpha(\beta)}(v)}= \bigcap_{v' \in U_{2r_\alpha(\beta})} U_{r_\alpha(\beta),\varphi_{r_\alpha(\beta)}(vv')}.
\end{align*}
Hence $\Gamma'_\beta = \Gamma'_{r_\alpha(\beta)}$ and we deduce that $\Gamma'_{w(\beta)} = \Gamma'_\beta$ for any $w \in W^v$.
\end{proof}

\begin{Cor}
Let $\Delta \subset \Phi$ be any basis.
A valuation is special (resp. semi-special) if, and only if, $\forall \alpha \in \Delta,\ 0\in \Gamma_\alpha$ (resp. $\Gamma_\alpha \subset \widetilde{\Gamma}_\alpha$).
\end{Cor}

\begin{proof}
Assume that, for any $\alpha \in \Delta$, we have that $\Gamma_\alpha \subset \widetilde{\Gamma}_\alpha$.
Let $\beta \in \Phi$. We know that there is $\alpha \in \Delta$ and $w\in W(\Phi)$ such that $w(\alpha)=\beta$ by~\cite[VI.1.5 Prop. 15]{bourbaki1981elements}.
According to~\cite[VI.1.5, Thm.2 (vii)]{bourbaki1981elements}, we can write $w = r_{\alpha_\ell} \circ \cdots \circ r_{\alpha_1}$ with $\alpha_i\in \Delta$.
For $0 \leqslant i \leqslant \ell$, let $\beta_i = r_{\alpha_i} \circ \cdots \circ r_{\alpha_1}(\alpha)$ so that $\beta_0 = \alpha$ and $\beta_\ell = \beta$.
Applying Corollary~\ref{CorWeylSetOfValues}, we deduce by induction that $\Gamma_{\beta_i} = \Gamma_{\alpha}$ for any $0 \leqslant i \leqslant n$.
Hence $\Gamma_\beta = \Gamma_\alpha \subset \widetilde{\Gamma}_\alpha = \widetilde{\Gamma}_\beta$.

Hence, a valuation is semi-special if, and only if, $\forall \alpha\in \Delta,\ \Gamma_\alpha \subset \widetilde{\Gamma}_\alpha$.

By applying Proposition~\ref{PropWeylInvarianceSetOfValues}, we deduce the case for special valuations.
\end{proof}

\subsection{Action of \texorpdfstring{$N$}{N} on an \texorpdfstring{$\Rtot$}{R}-aff space}

We use the notations introduced in section~\ref{subsubTopology_totally_ordered_ring}.
In the rest of this section, we assume that $\Phi \neq \emptyset$ but any result can obviously be extended to the case of an empty root system.
We consider an $R$-aff space $\mathbb{A}_\Rtot = (V,A)$ and we fix an origin $o \in \mathbb{A}_\Rtot$.

For $\alpha \in \Phi$ and $u \in U_\alpha \setminus \{1\}$, we denote by $m(u)$\index[notation]{m@$m(u)$} the unique element in $M_\alpha$ given by \cite[6.1.2(2)]{BruhatTits1}.

\begin{Def}\label{DefCompatibleAction}
Let $\nu : N \to \operatorname{Aff}_{R}(\mathbb{A}_R)$\index[notation]{n@$\nu$} be an action of $N$ onto $\mathbb{A}$ by $R$-aff endomorphisms.
We say that the action of $N$ onto $\mathbb{A}_R$ is \textbf{compatible}\index{compatible action} with the valuation $\left( \varphi_\alpha \right)_{\alpha\in \Phi}$ if:
\begin{enumerate}[label={(CA\arabic*)}]
\item\label{axiomCA1} the linear part of this action is equal to ${^v\!}\nu:N \to W(\Phi)$ defined in \cite[6.1.2(10)]{BruhatTits1};\axiom{CA1@\ref{axiomCA1}}
\item\label{axiomCA2} for any $\alpha \in \Phi$ and any $u \in U_{\alpha} \setminus \{1\}$, we have $2 \varphi_\alpha(u) + \alpha\Big( \nu\big(m(u)\big)(o) - o\Big) = 0$.\axiom{CA2@\ref{axiomCA2}}
\item\label{axiomCA3} for any $\alpha \in \Phi$ and $m \in M_\alpha$, the element $\nu(m)$ has order $2$.\axiom{CA3@\ref{axiomCA3}}
\end{enumerate}
\end{Def}

In the rest of this section, we assume that an action $\nu: N \to \operatorname{Aff}_{R}(\mathbb{A}_R)$ 
satisfying~\ref{axiomCA1}, \ref{axiomCA2} and~\ref{axiomCA3} is given.

\begin{Lem}\label{LemMuNu}
For any $\alpha \in \Phi$ and any $u \in U_{\alpha} \setminus \{1\}$, we have $\nu(m(u)) = r_{\alpha,\varphi_\alpha(u)}$.

In particular, for any $m \in M_{\alpha,\lambda}$, we have $\nu(m) = r_{\alpha,\lambda}$.
\end{Lem}

\begin{proof}
Since $m(u) \in M_\alpha$, we have ${^v\!}\nu(m(u)) = r_\alpha$ by axiom~\ref{axiomCA1}.
Let $v \in V_\Rtot$ be such that $\nu\big(m(u)\big) = (r_\alpha,v) \in \mathrm{GL}(V_\Z) \ltimes V_\Rtot$.
By axiom~\ref{axiomCA3}, we have that $\nu\big(m(u)\big)^2 = (\operatorname{id},0) = \big(\operatorname{id},v+r_\alpha(v)\big)$.
Hence $2v - \alpha(v) \alpha^\vee = 0$.
We have $\alpha(\nu(m(u))(o) - o) = \alpha(v) = -2 \varphi_\alpha(u)$ by axiom~\ref{axiomCA2}.
Hence $2 v + 2 \varphi_\alpha(u) \alpha^\vee = 0$.
Since $\Rtot$ is torsion-free, we deduce that $v = -\varphi_\alpha(u) \alpha^\vee$.
Whence $\nu\big(m(u)\big) = \big(r_\alpha, -\varphi_\alpha(u) \alpha^\vee \big) = r_{\alpha,\varphi_\alpha(u)}$.

Let $m \in M_{\alpha,\lambda}$. By definition (see Notation~\ref{NotValuedCoset}), there exist $u \in \varphi_\alpha^{-1}(\{\lambda\}) \subset U_{\alpha,\lambda}$, $u',u'' \in U_{-\alpha}$ such that $m = u'uu''$ . Thus $m = m(u)$ and $\nu(m) = \nu(m(u)) = r_{\alpha,\lambda}$.
\end{proof}

We want to understand how $N$ acts by conjugation on the set of groups $U_{\alpha,\lambda}$ for $\alpha \in \Phi$ and $\lambda \in \Rtot$.
To do this, we introduce a set of affine maps $\Theta$ containing the $\theta_{\alpha,\lambda}$ and an action of $N$ onto this set.

\begin{Not}
Consider $\Theta = \{\theta_{\alpha,\lambda},\ \alpha \in \Phi,\ \lambda\in \Rtot\}$.\index[notation]{t@$\Theta$}
We endow $\Theta$ with the natural partial ordering given by:
\[ \theta \geqslant \theta' \Longleftrightarrow \forall x \in \mathbb{A}_R,\ \theta(x) \geqslant \theta'(x).\]
\end{Not}

\begin{Lem}\label{LemActionTheta}
The following formula defines an action of $N$ on $\Theta$:
\[\forall \alpha \in\Phi,\ \forall \lambda \in \Rtot,\
n \cdot \theta_{\alpha,\lambda} := \theta_{\alpha,\lambda} \circ \nu(n^{-1}) = \theta_{{^v\!}\nu(n)(\alpha), \lambda + \alpha\left( \nu(n^{-1}) (o) - o\right)}.\]
In particular, for any $\alpha,\beta \in \Phi$ and any $\lambda,\mu \in \Rtot$, we have
\[\theta_{\alpha,\lambda} \circ r_{\beta,\mu} = \theta_{r_\beta(\alpha),\lambda - \alpha(\beta^\vee) \mu}.\]

\end{Lem}

\begin{proof}
Let $\theta_{\alpha,\lambda} \in \Theta$.
It suffices to prove that the map $\theta_{\alpha,\lambda} \circ \nu(n^{-1})$ belongs to $\Theta$.
Since $N$ is generated by the $M_\beta$ for $\beta \in \Phi$, it suffices to prove it for any $\beta \in \Phi$ and any $m \in M_\beta$.
By Lemma~\ref{LemMuNu},
there is a constant $\mu$ depending on $m$ such that $\nu(m) = r_{\beta,\mu}$.
Then $\theta_{\alpha,\lambda}(r_{\beta,\mu}(x)) = (\alpha - \alpha(\beta^\vee) \beta)(x) + \lambda - \alpha(\beta^\vee) \mu = r_\beta(\alpha)(x) + \lambda - \alpha(\beta^\vee) \mu$.
Thus $\theta \circ \nu(m) \in \Theta$.
Hence the formula $n \cdot \theta = \theta \circ \nu(n^{-1})$ defines an action of $N$ on $\Theta$.

Let $n \in N$. Since the map ${^v\!}\nu : N \to \mathrm{GL}(V^*)$ induces an action of $N$ on $\Phi$, for any $\alpha \in \Phi$ and any $\lambda \in \Rtot$, there is a value $\mu \in \Rtot$ such that $\theta_{\alpha,\lambda} \circ \nu(n^{-1}) = \theta_{{^v\!}\nu(n)(\alpha),\mu}$.
Evaluating this map in the origin $o$, we get $\mu = \theta_{{^v\!}\nu(n)(\alpha),\mu}(o) = \theta_{\alpha,\lambda} \circ \nu(n^{-1})(o) = \alpha\left( \nu(n^{-1})(o) - o \right) + \lambda$.

\end{proof}

\begin{Not}
For $\theta = \theta_{\alpha,\lambda} \in \Theta$, define $U_\theta = \{u \in U_{\alpha},\ \theta \circ \nu(m(u)) \leqslant -\theta\}$.\index[notation]{u@$U_\theta$}
\end{Not}

We recall the following result from \cite[7.3]{Landvogt}:

\begin{Lem}\label{LemActionUtheta}
For any $\theta \in \Theta$, any $n \in N$, we have $n U_\theta n^{-1} = U_{n \cdot \theta}$.
\end{Lem}

\begin{proof}
Let $u \in U_\theta \subset U_\alpha$ for some $\alpha \in \Phi$.
Consider $n \in N$ and denote $\beta = {^v\!}\nu(n)(\alpha) \in \Phi$.

Suppose that $u \neq 1$ and consider the element $m(u) = u'uu'' \in U_{-\alpha}U_{\alpha}U_{-\alpha} \cap M_\alpha$, with $u',u'' \in U_{-\alpha}$ uniquely determined by $u$.
For any $n \in N$, we have $nm(u)n^{-1} = (nu'n^{-1})(nun^{-1})(nu''n^{-1})$.
By \cite[6.1.2(10)]{BruhatTits1}, $(nu'n^{-1}), (nu''n^{-1}) \in U_{-\beta}$ and $nun^{-1} \in U_{\beta}$.
Thus, by  uniqueness in \cite[6.1.2(2)]{BruhatTits1}, we have $m(nun^{-1}) = n m(u) n^{-1} \in M_\beta$.

For any $x \in \mathbb{A}_\Rtot$, we have 
\begin{align*}
\theta \circ \nu(n^{-1})\left( \nu(m(nun^{-1}))(x) \right)
& = \theta \circ \nu(n^{-1}) \circ \left( \nu(n) \circ \nu(m(u)) \circ \nu(n^{-1}) \right) (x)\\
& = \theta \circ \nu(m(u)) \left( \nu(n^{-1})(x) \right)\\
& \leqslant - \theta \left( \nu(n^{-1})(x) \right)
\end{align*}
Thus $\theta \circ \nu(n^{-1}) \circ \nu(m(nun^{-1})) \leqslant - \theta \circ \nu(n^{-1})$ which means that $nun^{-1} \in U_{\theta \circ \nu(n^{-1})} = U_{n \cdot \theta}$.

Hence $n U_\theta n^{-1} \subset U_{n \cdot \theta}$ for any $n \in N$ and any $\theta \in \Theta$.
Thus, by applying it to $(n^{-1}, n \cdot \theta)$ we get $n^{-1} U_{n \cdot \theta} n \subset U_{\theta}$ which gives the equality $n U_\theta n^{-1} = U_{n \cdot \theta}$.
\end{proof}

We get the following consequence as in \cite[7.7]{Landvogt}:

\begin{Prop}\label{PropUthetaEgalUa}
For any $\alpha \in \Phi$ and any $\lambda \in \Rtot$, we have $U_{\alpha,\lambda} = U_{\theta_{\alpha,\lambda}}$.
\end{Prop}

\begin{proof}
Let $u \in U_\alpha$ and $\mu = \varphi_\alpha(u)$.
By Lemma~\ref{LemMuNu}, we have $\nu(m(u)) = r_{\alpha,\mu}$.
For any $x \in \mathbb{A}_\Rtot$, we have
\begin{align*}
\theta_{\alpha,\lambda}\left( \nu(m(u))(x) \right) = & 
\theta_{\alpha,\lambda}\left( x - \left( \alpha(x-o) + \mu \right) \alpha^\vee \right)\\
= & -\alpha(x-o) - 2 \mu + \lambda\\
= & -\theta_{\alpha,\lambda}(x) + 2 (\lambda-\mu)
\end{align*}
Thus $u \in U_{\theta_{\alpha,\lambda}} \Longleftrightarrow \theta_{\alpha,\lambda} \circ \nu(m(u)) \leqslant - \theta_{\alpha,\lambda} \Longleftrightarrow \mu \geqslant \lambda \Longleftrightarrow u \in U_{\alpha,\lambda}$.
\end{proof}

\begin{Cor}\label{CorActionNUa}
For any $\alpha \in \Phi$, any $\lambda \in \Rtot$ and any $n \in N$, we have
\[n U_{\alpha,\lambda} n^{-1} = U_{{^v\!}\nu(n)(\alpha), \lambda + \alpha( \nu(n^{-1})(o) - o)}.\]
In particular, $\ker \nu$ normalizes $U_{\alpha,\lambda}$.

\end{Cor}

\begin{proof}
This is a consequence of Proposition~\ref{PropUthetaEgalUa} and Lemma~\ref{LemActionTheta}.
\end{proof}

\begin{Cor}[{see \cite[6.2.10]{BruhatTits1}}]\label{CorAffineRootsInvariant}
The subset $\Sigma := \{\theta_{\alpha,\lambda} \in \Theta,\ \alpha \in \Phi, \lambda \in \Gamma'_\alpha\}$ is invariant under the action of $N$.
\end{Cor}

\begin{proof}
Consider $\alpha \in \Phi$ and $\lambda \in \Gamma'_\alpha$ and $n \in N$.
By definition, there exists $u \in U_{\alpha,\lambda}$ such that $\lambda = \varphi_\alpha(u)$ and $U_{\alpha,\lambda} = \bigcap_{v \in U_{2\alpha}} U_{\alpha,\varphi_\alpha(uv)}$.
Set $\beta = {^v\!}\nu(n)(\alpha)$.
We prove that 
\begin{equation}
n U_{\alpha,\varphi_\alpha(u)} n^{-1} = U_{\beta,\varphi_\beta(nun^{-1})}.\label{eqnActionN}
\end{equation}

Indeed, let $\mu \in \Rtot$ be such that $\theta_{\beta,\mu} = n \cdot \theta_{\alpha,\lambda}$.
Then $\mu = \lambda + \alpha\big(\nu(n^{-1})(o) - o \big)$ and $\lambda = \mu + \beta\big(\nu(n)(o) - o \big)$ by Lemma~\ref{LemActionTheta}.
Let $u' = nun^{-1}$ and $\mu' = \varphi_\beta(u)$.
Then $u' \in nU_{\alpha,\lambda}n^{-1} = U_{\beta,\mu}$ by Corollary~\ref{CorActionNUa}, so that $\mu' \geqslant \mu$.
Applying Corollary~\ref{CorActionNUa} again, we have $u=n^{-1}u'n \in U_{\alpha,\mu'+\beta\left(\nu(n)(o)-o\right)}$.
Thus $\lambda = \varphi_\alpha(u) \geqslant \mu'+\beta\big(\nu(n)(o)-o\big) = \mu' - \mu + \lambda$.
Hence $\mu = \mu'$ and we have shown equality~(\ref{eqnActionN}).

Using the equality $U_{\alpha,\lambda} = \bigcap_{v \in U_{2\alpha}} U_{\alpha,\varphi_\alpha(uv)}$, we have:
\[\bigcap_{v' \in U_{2\beta}} U_{\beta,\varphi_\beta(u'v')} = \bigcap_{v \in U_{2\alpha}} U_{\beta,\varphi_\beta(nuvn^{-1})} = n \bigcap_{v \in U_{2\alpha}} U_{\alpha,\varphi_\alpha(uv)}n^{-1} = n U_{\alpha,\varphi_\alpha(u)} n^{-1} = U_{\beta,\varphi_\beta(u')}.\]
Hence $\mu' = \varphi_\beta(u') \in \Gamma'_\beta$.
This proves that $\theta_{\beta,\mu} = n \cdot \theta_{\alpha,\lambda} \in \Sigma$ for any $n\in N$ and any $\theta_{\alpha,\lambda} \in \Sigma$.
\end{proof}

In section~\ref{sectionLambdaBuildingFromVRGD}, we will see that the datum of the valuation of a root group datum, together with a compatible action allows us to define a geometric space that automatically satisfies most of axioms of $\Lambda$-buildings.
Because it may be simpler to provide an explicit affine action of $N$ onto $\mathbb{A}_{\RF^S}$, we will make the following assumption:

\begin{Hyp}\label{HypCAVRGD}
It is given a group $G$ and a non-empty root system $\Phi$.
It is given a generated root group datum $\left( T, (U_\alpha,M_\alpha)_{\alpha \in \Phi} \right)$ of $G$.
It is given an $\Rtot$-valuation $\left(\varphi_{\alpha}\right)_{\alpha\in \Phi}$ of the root group datum.
Let $N$ be the group generated by the $M_\alpha$ for $\alpha \in \Phi$.
It is given an action $\nu: N \to \operatorname{Aff}_{R}(\mathbb{A}_\Rtot)$ compatible with the valuation.
\end{Hyp}

\begin{Not}\label{NotWeylAffine}
Under the above assumption, we denote by $T_b := \ker \nu$.\index[notation]{t@$T_b$}
According to \cite[6.1.11(ii)]{BruhatTits1}, $T_b$ is a subgroup of $T$ since $T=\ker {^v\!}\nu$.

We denote by
$\Wext := \nu(N) \subset \operatorname{Aff}_{R}(\mathbb{A}_\Rtot)$
and we call it the \textbf{extended affine Weyl group}.\index{Weyl group!extended affine}\index{affine Weyl group!extended}\index[notation]{w@$\widetilde{W}$}
We denote by $W^v := {^v\!}\nu(N)$ and we call it the \textbf{spherical Weyl group}\index{Weyl group!spherical}\index{spherical Weyl group}\index[notation]{w@$W^v$}
and by $\overline{W}$ the subgroup, isomorphic to  $W^v$, of $\operatorname{Aff}_{R}(\mathbb{A}_\Rtot)$ generated by the $r_{\alpha,0}$ for $\alpha \in \Phi$.

We denote by $\Waff$\index[notation]{w@$\Waff$} the subgroup of $\operatorname{Aff}_{R}(\mathbb{A}_\Rtot)$ generated by the $r_{\alpha,\lambda}$ for $\alpha \in \Phi$ and $\lambda \in \Gamma'_\alpha$
and we call it the \textbf{affine Weyl group}.\index{Weyl group!affine}\index{affine Weyl group}

\end{Not}

\begin{Rq}
The group $T_b$ is denoted by $H$ in \cite[6.2.11]{BruhatTits1}.
This notation emphasizes the fact that when $T = \mathbf{T}(\mathbb{K})$ denotes the rational points a maximal torus of a (quasi) split semi-simple group, then $T_b$ denotes the bounded elements of this torus.
In that case, it can be called a ``maximal bounded torus''.
Be careful that $T_b$ may not be commutative, nor bounded in general.

By definition
$\Wext \simeq N / T_b$
and $W^v \simeq N / T$.
\end{Rq}

Let $\Delta$ be a basis of $\Phi$.
For $\alpha\in \Delta$
we recall that $\varpi_\alpha^\vee$ denotes
the fundamental coweight in $V_\R$ associated with $\alpha$, i.e the unique element of $V_\R$ such that for all $\beta\in \Delta$, one has $\beta(\varpi_\alpha^\vee)=\delta_{\alpha,\beta}$.

\begin{Prop}[{see \cite[6.2.20]{BruhatTits1}}]\label{PropAffineWeylGroup}
The group $\Waff$
is a normal subgroup of
$\Wext$
and for any $\alpha \in \Phi_{\mathrm{nd}}$, the group of translations $\widetilde{\Gamma}_\alpha \alpha^\vee$ is a subgroup of
$\Waff$.

Assume that the valuation $(\varphi_\alpha)_{\alpha \in \Phi}$ is semi-special.
Let $\Delta$ be a basis of $\Phi$.
Then
\begin{enumerate}[label={(\arabic*)}]
\item\label{PropWeylAffine1}
$\overline{W}$ is a subgroup of
$\Waff$.
In particular,
$\Waff = \left( V_\Rtot \cap 
\Waff\right) \rtimes W^v$ and
$\Wext = \left( V_\Rtot \cap \Wext \right) \rtimes W^v$.
\item\label{PropWeylAffine2} $V_\Rtot \cap \Waff = \bigoplus_{\alpha \in \Delta} \widetilde{\Gamma}_\alpha \alpha^\vee$ for any basis $\Delta \subset \Phi$.
\item\label{PropWeylAffine3} $V_\Rtot \cap \Wext \subset \bigoplus_{\alpha \in \Delta} \widetilde{\Gamma}_\alpha \varpi_\alpha^\vee$.
\end{enumerate}
\end{Prop}

\begin{proof}
We fix a basis $\Delta$ of $\Phi$.
For $\alpha \in \Phi$ and $\lambda \in \Gamma_\alpha$, we have either $\lambda \in \Gamma'_\alpha$ or $2\lambda \in \Gamma'_{2\alpha}$, by Fact~\ref{FactDecompositionSetOfValue} and Remark~\ref{rqGamma_alpha}. Moreover if $2\alpha\in \Phi$ and $2\lambda\in \Gamma'_{2\alpha}$, then $r_{\alpha,\lambda}=r_{2\alpha,2\lambda}$.
Thus $r_{\alpha,\lambda}$ belongs to $\Waff$.
Hence $\Waff$ is also the group generated by the $r_{\alpha,\lambda}$ for $\alpha \in \Phi$ and $\lambda \in \Gamma_{\alpha}$.

From Proposition~\ref{PropValuedCoset}~\ref{PropValuedCoset:emptyness} and Lemma~\ref{LemMuNu}, we deduce that $\Waff$ is a subgroup of $\Wext$.
Using Corollary~\ref{CorAffineRootsInvariant}, we know that for any $n \in N$, $\alpha \in \Phi$ and $\lambda \in \Gamma'_\alpha$, there are $\beta \in \Phi$ and $\mu \in \Gamma'_\beta$ such that $\nu(n) r_{\alpha,\lambda} \nu(n^{-1}) = r_{\beta,\mu}$.
Hence $\Waff$ is a normal subgroup of $\Wext$.

Let $\mu, \nu \in \Gamma_\alpha$.
There exist $m,n \in N$ such that $m \in M_{\alpha,\mu}$ and $n \in M_{\alpha,\nu}$ so that $r_{\alpha,\mu} = \nu(m)$ and $r_{\alpha,\nu} = \nu(n)$.
Then $\nu(mn)=r_{\alpha,\mu} \circ r_{\alpha,\nu} = (\nu - \mu ) \alpha^\vee \in \Waff$.
Hence $\widetilde{\Gamma}_\alpha \alpha^\vee = \langle \Gamma_\alpha - \Gamma_\alpha \rangle \alpha^\vee \subset \Waff$.
As a consequence, the group $\widetilde{V}= \bigoplus_{\alpha \in \Delta} \widetilde{\Gamma}_\alpha \alpha^\vee$ is a subgroup of $V_\Rtot \cap \Waff$. 

From now on, we assume that the valuation is semi-special. 

\ref{PropWeylAffine1} Let $\alpha\in \Phi$ and $\lambda \in \Gamma_\alpha$.
By Corollary~\ref{cor_Symmetry_properties_Gamma_alpha}, $-\lambda\alpha^\vee\in \widetilde{\Gamma}_{\alpha}\alpha^\vee\subset \Waff$.
Thus $r_{\alpha,\lambda}\circ (-\lambda\alpha^\vee)=r_{\alpha,0}\in  \Waff$ and thus $\overline{W}$ is a subgroup of $\Waff$.

\ref{PropWeylAffine2}
Recall that, according to \cite[VI.1.1 Lem.2]{bourbaki1981elements}, for any $\alpha \in \Phi$ and any $w \in W^v$, we have that $w(\alpha^\vee) = \big( w(\alpha) \big)^\vee$.

We firstly prove that $\widetilde{V}$ is $W^v$-invariant.
It suffices to show that it is invariant by applying $r_\alpha$ for $\alpha \in \Delta$.
Let $\alpha,\beta \in \Delta$, then $r_\alpha(\widetilde{\Gamma}_\beta \beta^\vee) = \widetilde{\Gamma}_\beta \big(r_\alpha(\beta)\big)^\vee$.
Since the valuation is semi-special, we deduce from Corollary~\ref{CorWeylSetOfValues} that we have $\widetilde{\Gamma}_\beta = \widetilde{\Gamma}_{r_\alpha(\beta)}$.
Hence $r_\alpha(\widetilde{\Gamma}_\beta \beta^\vee) \subset \widetilde{V}$.
Thus $\widetilde{V}$ is $r_\alpha$-invariant for any $\alpha \in \Delta$ and therefore $W^v$-invariant.

We have shown that $\Waff \supset \widetilde{V} \rtimes W^v$.
Conversely, consider any $\beta \in \Phi$ and any $\mu \in \Gamma'_\beta$.
Assume firstly that $\beta$ is non-divisible.
By \cite[VI.1.5 Prop.15]{bourbaki1981elements}, there exists $w \in W^v$ and $\alpha \in \Delta$ such that $\alpha = w(\beta)$.
Hence $w(-\mu \beta^\vee) = -\mu \big(w(\beta)\big)^\vee \in \widetilde{\Gamma}_\beta \alpha^\vee = \widetilde{\Gamma}_{w^{-1}(\alpha)} \alpha^\vee = \widetilde{\Gamma}_{\alpha} \alpha^\vee$ according to Proposition~\ref{PropWeylInvarianceSetOfValues}.
Hence $-\mu \beta^\vee \in w^{-1}(\widetilde{V}) = \widetilde{V}$, because $\widetilde{V}$ is $W^v$-invariant.
Assume that $\beta$ is divisible and write $\beta = 2 \gamma$ with $\gamma$ non-divisible.
Then $\widetilde{\Gamma}_\beta \beta^\vee = \widetilde{\Gamma}_{2\gamma} \frac{1}{2} \gamma^\vee \subset \widetilde{\Gamma}_\gamma \gamma^\vee$ since $\Gamma_\gamma = \Gamma'_\gamma \cup \frac{1}{2} \Gamma'_{2\gamma}$ by Fact~\ref{FactDecompositionSetOfValue}.
Hence $-\mu \beta^\vee \in \widetilde{V}$ for any $\beta \in \Phi$ and any $\mu \in \Gamma'_\beta$.
Therefore $(-\mu \beta^\vee, r_\beta) =  r_{\beta,\mu} \in \widetilde{V} \rtimes W^v$.
Since these elements generate $\Waff$, we deduce that
$\Waff = \widetilde{V} \rtimes W^v$
and it follows that $V_\Rtot \cap  \Waff= \widetilde{V}$.

\ref{PropWeylAffine3}
Let $v \in \Wext \cap V_\Rtot$ and $\alpha \in \Delta$.
Then $[v,r_{\alpha,0}] \in  \Waff \cap V_\Rtot = \widetilde{V}$ since $\Waff$ is normalized by $\Wext$.
But $[v,r_{\alpha,0}] = v - r_\alpha(v) = \alpha(v) \alpha^\vee$.
Hence $\alpha(v) \in \widetilde{\Gamma}_\alpha$ for every $\alpha \in \Delta$ and therefore $v \in \bigoplus_{\alpha \in \Delta} \widetilde{\Gamma}_\alpha \varpi_\alpha^\vee$ by definition of the fundamental weights.
\end{proof}

\subsection{Rank-one Levi subgroups}

In this section, we work under data and notations of assumption~\ref{HypCAVRGD}.

\begin{Not}\label{NotLeviRankOne}
For any $\alpha \in \Phi$, any $\lambda \in \Rtot$ and any $\varepsilon \in \Rtot_{\geqslant 0}$, we define the subgroups
\begin{itemize}
\item $U'_{\alpha,\lambda} = \varphi_\alpha^{-1}(]\lambda,+\infty]) = \bigcup_{\mu > \lambda} U_{\alpha,\mu}$;\index[notation]{u@$U'_{\alpha,\lambda}$}
\item $L^\varepsilon_{\alpha,\lambda}$ (resp. $L'_{\alpha,\lambda}$) the subgroup generated by $U_{\alpha,\lambda}$ and $U_{-\alpha,-\lambda+\varepsilon}$ (resp. $U_{\alpha,\lambda}$ and $U'_{-\alpha,-\lambda}$);\index[notation]{l@$L^\varepsilon_{\alpha,\lambda}$}\index[notation]{l@$L'_{\alpha,\lambda}$}
\item $N^\varepsilon_{\alpha,\lambda} = N \cap L^\varepsilon_{\alpha,\lambda}$ and $N'_{\alpha,\lambda} = N \cap L'_{\alpha,\lambda}$;\index[notation]{n@$N^\varepsilon_{\alpha,\lambda}$}\index[notation]{n@$N'_{\alpha,\lambda}$}
\item $T^\varepsilon_{\alpha,\lambda} = T \cap L^\varepsilon_{\alpha,\lambda}$ and $T'_{\alpha,\lambda} = T \cap L'_{\alpha,\lambda}$.\index[notation]{t@$T^\varepsilon_{\alpha,\lambda}$}\index[notation]{t@$T'_{\alpha,\lambda}$}
\end{itemize}

If $\varepsilon = 0$, we will denote $L_{\alpha,\lambda}, N_{\alpha,\lambda}, T_{\alpha,\lambda}$ instead of $L^0_{\alpha,\lambda}, N^0_{\alpha,\lambda}, T^0_{\alpha,\lambda}$.\index[notation]{l@$L_{\alpha,\lambda}$}\index[notation]{n@$N_{\alpha,\lambda}$}\index[notation]{t@$T_{\alpha,\lambda}$}
\end{Not}

\begin{Rq}\label{RkRankOneRootGroup}

We have $U_{\alpha,\lambda} \supseteq U'_{\alpha,\lambda}$ with equality if, and only if, $\lambda \not\in \Gamma_\alpha$.

Indeed, if $\exists u \in U_{\alpha,\lambda} \setminus U'_{\alpha,\lambda}$, then $\varphi_\alpha(u)\in [\lambda,+\infty]\setminus ]\lambda,+\infty]=\{\lambda\}$.
Thus $\lambda = \varphi_\alpha(u) \in \Gamma_\alpha$.
Conversely, if $\lambda \in \Gamma_\alpha$, there exists $u \in U_\alpha$ such that $\varphi_\alpha(u) = \lambda$ and then $u\in U_{\alpha,\lambda}\setminus U'_{\alpha,\lambda}$.

Note that, $L_{\alpha,\lambda} = L_{-\alpha,-\lambda}$ by definition, but $L'_{-\alpha,-\lambda}$ may differ from $L'_{\alpha,\lambda}$ since the groups $U_{\alpha,\lambda}$ and $U'_{\alpha,\lambda}$ are distinct for $\lambda \in \Gamma_\alpha$.

\end{Rq}

\begin{Rq}\label{RkVarChangeinT}
Since $L_{\alpha,\lambda}^\varepsilon = \langle U_{\alpha,\lambda}, U_{-\alpha,-\lambda+\varepsilon} \rangle$, taking $\beta = -\alpha$ and $\mu = -\lambda+\varepsilon$, we have $L_{\alpha,\lambda}^\varepsilon = \langle U_{\beta,\mu},U_{-\beta,-\mu+\varepsilon} \rangle$.
Hence $L_{\alpha,\lambda}^\varepsilon = L_{-\alpha,-\lambda+\varepsilon}^\varepsilon$.
By intersection with $T$, we deduce that $T_{\alpha,\lambda}^\varepsilon = T_{-\alpha,-\lambda+\varepsilon}^\varepsilon$.

For $\varepsilon' \geqslant 0$, we have $U_{\alpha,\lambda+\varepsilon'} \subset U_{\alpha,\lambda}$ and $U_{-\alpha,-\lambda+\varepsilon} = U_{-\alpha,-(\lambda+\varepsilon')+(\varepsilon+\varepsilon')}$.
Hence we have $L_{\alpha,\lambda+\varepsilon'}^{\varepsilon+\varepsilon'} \subset L_{\alpha,\lambda}^\varepsilon$ and thus $T_{\alpha,\lambda+\varepsilon'}^{\varepsilon+\varepsilon'} \subset T_{\alpha,\lambda}^\varepsilon$. By the same way $L_{\alpha,\lambda}^{\varepsilon +\varepsilon'} \subset L_{\alpha,\lambda}^\varepsilon$ and $T_{\alpha,\lambda}^{\varepsilon +\varepsilon'} \subset T_{\alpha,\lambda}^\varepsilon$.
\end{Rq}

The following lemma can be proven exactly as in \cite[6.3.1]{BruhatTits1}.

\begin{Lem}\label{LemTripleLeviCommutation}
Let $\alpha \in \Phi$.
\begin{enumerate}[label={(\arabic*)}]
\item\label{LemTripleLeviCommutation:1} For any $u \in U_\alpha$ and $v \in U_{-\alpha}$ such that $\varphi_\alpha(u) + \varphi_{-\alpha}(v) > 0$, there is a unique triple $(u',t,v') \in U_\alpha \times T \times U_{-\alpha}$ such that $vu = u'tv'$.
\item\label{LemTripleLeviCommutation:2} Moreover, we have $t \in T_b$, $\varphi_{\alpha}(u') = \varphi_{\alpha}(u)$ and $\varphi_{-\alpha}(v') = \varphi_{-\alpha}(v)$.
\item\label{LemTripleLeviGroup} Let $X_\alpha$, $X_{-\alpha}$, $H$ be respectively subsets of $U_\alpha$, $U_{-\alpha}$, $T_b$. Set $X = X_\alpha H X_{-\alpha}$. Then $X \cap N = X \cap T = X \cap T_b = H$.
\end{enumerate}
\end{Lem}

\begin{proof}
\ref{LemTripleLeviCommutation:1}
Denote by $L_{-\alpha}$ the subgroup of $G$ generated by $U_\alpha$, $U_{-\alpha}$ and $T$.
By \cite[6.1.2 (4) \& (7)]{BruhatTits1}, we know that $L_{-\alpha} = M_{\alpha} U_{-\alpha} \cup U_\alpha T U_{-\alpha}$.
Suppose $vu \in M_\alpha U_{-\alpha}$ and choose $u'' \in U_{-\alpha}$ such that $vuu'' \in M_\alpha$.
By axiom~\ref{axiomV5}, we have $\varphi_{-\alpha}(v)=- \varphi_\alpha(u)$ which contradicts the assumption.
Hence $vu \in U_\alpha T U_{-\alpha}$ and we get the existence.
The uniqueness is an immediate consequence of \cite[6.1.2(3)]{BruhatTits1} and axiom ~\ref{axiomRGD6}.

\ref{LemTripleLeviCommutation:2} By  uniqueness, it is obvious if $u=1$ or $v=1$.
Assume that $u \neq 1$ and $v \neq 1$.
Then, applying axiom~\ref{axiomRGD6}, we have $u' \neq 1$.
By axiom~\ref{axiomRGD4}, there exist $w,w' \in U_{-\alpha}$ and $m \in M_{-\alpha} = M_\alpha$ such that $u' = w' m w''$.
Moreover $u = v^{-1} u' t v' = (v^{-1} w') (mt) (t^{-1} w'' t v')$ where $(v^{-1} w')\in U_{-\alpha}$, $(mt) \in M_\alpha$ and $(t^{-1} w'' t v') \in U_{-\alpha}$.
Hence by axioms~\ref{axiomV5},~\ref{axiomV5bis} and Lemma~\ref{LemAxiomV1}\ref{LemAxiomV1:1} we have
\begin{equation}\label{Triple:eqn1}
\varphi_{-\alpha}(w') = - \varphi_\alpha(u') = \varphi_{-\alpha}(w'')
\end{equation}
and
\begin{equation}\label{Triple:eqn2}
\varphi_{-\alpha}(v^{-1} w') = - \varphi_\alpha(u) = \varphi_{-\alpha}(t^{-1} w'' t v').
\end{equation}
The assumption on $u$ and $v$ and equation (\ref{Triple:eqn2}) give us
\begin{equation}\label{Triple:eqn3}
\varphi_{-\alpha}(v) > - \varphi_\alpha(u) = \varphi_{-\alpha}(v^{-1}w').
\end{equation}
Hence, equations (\ref{Triple:eqn1}), (\ref{Triple:eqn2}), inequation (\ref{Triple:eqn3}) and Lemma~\ref{LemAxiomV1}\ref{LemAxiomV1:2} give
\begin{equation}\label{Triple:eqn4}
-\varphi_\alpha(u') = \varphi_{-\alpha}(w') = \varphi_{-\alpha}(vv^{-1}w') = \varphi_{-\alpha}(v^{-1}w') = -\varphi_\alpha(u).
\end{equation}
Since $(vu)^{-1} = u^{-1} v^{-1} = (v')^{-1} t^{-1} (u')^{-1}$, we can apply the result we just proved with $-\alpha$, $v^{-1}$ and $(v')^{-1}$ instead of $\alpha$, $u$ and $v^{-1}$
in order to have $\varphi_{-\alpha}(v^{-1}) = \varphi_{-\alpha}((v')^{-1})$.
Applying Lemma~\ref{LemAxiomV1}\ref{LemAxiomV1:1}, we get $\varphi_{-\alpha}(v) = \varphi_{-\alpha}(v')$.

Denote $\lambda = \varphi_\alpha(u) = \varphi_\alpha(u')$.
We know that \[m = (w')^{-1} u' (w'')^{-1} \in M_\alpha \cap U_{-\alpha} \varphi_\alpha^{-1}(\{\lambda\}) U_{-\alpha} = M_{\alpha,\lambda}.\]
In the same way, we have $tm \in M_{\alpha,\lambda}$.
Hence, by Lemma~\ref{LemMuNu}, $\nu(t) = \nu(tm) \nu(m)^{-1} =r_{\alpha,\lambda} r_{\alpha,\lambda}^{-1}= 1$.

\ref{LemTripleLeviGroup}
We have $H \subset X \cap T_b \subset X \cap T \subset X \cap N$ by definition.
Using a Bruhat decomposition \cite[6.1.15(c)]{BruhatTits1} with $U_\alpha \subset U^+$ and $U_{-\alpha} \subset U^-$, we deduce that $X \cap N \subset H$.
\end{proof}

The following Proposition is similar to \cite[6.3.2, 6.3.3]{BruhatTits1} and \cite[8.1--8.6]{Landvogt}, including considerations on the groups $L^\varepsilon_{\alpha,\lambda}$ we have introduced.

\begin{Prop}\label{PropRankOneLevi}
Consider any $\alpha\in \Phi$, any $\lambda \in \Rtot$ and any $\varepsilon \in \Rtot_{>0}$. Under the notations~\ref{NotLeviRankOne}, we have:
\begin{enumerate}[label={(\arabic*)}]
\item\label{PropRankOneLevi:epsilon} We have $L'_{\alpha,\lambda} = U_{\alpha,\lambda} U'_{-\alpha,-\lambda} T'_{\alpha,\lambda}$ and $L^\varepsilon_{\alpha,\lambda} = U_{\alpha,\lambda}  U_{-\alpha,-\lambda+\varepsilon}T^\varepsilon_{\alpha,\lambda}$ for any ordering on the factors.
\item\label{PropRankOneLevi:normalizers} Moreover $N'_{\alpha,\lambda} = T'_{\alpha,\lambda}$ and $N^\varepsilon_{\alpha,\lambda} = T^\varepsilon_{\alpha,\lambda}$.
\item\label{PropRankOneLevi:inGamma} If $\lambda \in \Gamma_\alpha$, then for any $m \in M_{\alpha,\lambda}$, we have $L_{\alpha,\lambda} = \left( U_{\alpha,\lambda} T_{\alpha,\lambda} U'_{-\alpha,-\lambda}\right) \sqcup \left( U_{\alpha,\lambda} m T_{\alpha,\lambda} U_{\alpha,\lambda}\right)$ and $N_{\alpha,\lambda} = T_{\alpha,\lambda} \{1,m\}$.
\item\label{PropRankOneLevi:notinGamma} If $\lambda \not\in \Gamma_\alpha$, we have $L_{\alpha,\lambda} = U_{\alpha,\lambda} T_{\alpha,\lambda} U_{-\alpha,-\lambda}$ and $N_{\alpha,\lambda} = T_{\alpha,\lambda} = T'_{\alpha,\lambda}$.
\item\label{PropRankOneLevi:tori} The groups $T_{\alpha,\lambda}$, $T'_{\alpha,\lambda}$ and $T^\varepsilon_{\alpha,\lambda}$ are subgroups of $T_b$.
\end{enumerate}
\end{Prop}

\begin{proof}
\ref{PropRankOneLevi:epsilon} Consider $\varepsilon \in \Rtot_{>0}$ and $H_\varepsilon = L^\varepsilon_{\alpha,\lambda} \cap T_b$. 
Then $H_\varepsilon$ normalizes $U_{\alpha,\lambda}$ and $U_{-\alpha,-\lambda + \varepsilon}$ since so does $T_b$ by Corollary~\ref{CorActionNUa}.
Consider $X_\varepsilon = U_{\alpha,\lambda} U_{-\alpha,-\lambda+\varepsilon} H_\varepsilon$.
Then $X_\varepsilon$ is a subset of $L^\varepsilon_{\alpha,\lambda}$ stable by multiplication on the left by elements in $H_\varepsilon$ and $U_{\alpha,\lambda}$.

Moreover, $U_{-\alpha,-\lambda+\varepsilon} U_{\alpha,\lambda} \subset U_{\alpha,\lambda} \left( T_b \cap L^\varepsilon_{\alpha,\lambda} \right) U_{-\alpha,-\lambda+\varepsilon}$ by Lemma~\ref{LemTripleLeviCommutation}.
Hence $X_\varepsilon$ is stable by multiplication on the left by elements in $U_{-\alpha,-\lambda+\varepsilon}$ since $H_\varepsilon$ normalizes the subgroup $U_{-\alpha,-\lambda+\varepsilon}$.
Hence $X_\varepsilon = L^\varepsilon_{\alpha,\lambda}$.
By Lemma~\ref{LemTripleLeviCommutation}~\ref{LemTripleLeviGroup},
we get $T^\varepsilon_{\alpha,\lambda} = H_\varepsilon$ and therefore $T^\varepsilon_{\alpha,\lambda} \subset T_b$.

Since $U'_{-\alpha,-\lambda} = \bigcup_{\varepsilon > 0} U_{-\alpha,-\lambda+\varepsilon}$ is an increasing union, we have $L'_{\alpha,\lambda} = \bigcup_{\varepsilon > 0} L^\varepsilon_{\alpha,\lambda}$ as increasing union so that the equality $L'_{\alpha,\lambda} = \bigcup_{\varepsilon > 0} U_{\alpha,\lambda} T^\varepsilon_{\alpha,\lambda} U_{-\alpha,-\lambda+\varepsilon} = U_{\alpha,\lambda} T'_{\alpha,\lambda} U'_{-\alpha,-\lambda}$ holds.

Finally, since $U_{\alpha,\lambda}$ and $U_{-\alpha,-\lambda+\varepsilon}$ (resp. $U'_{-\alpha,-\lambda}$) are subgroups normalized by $T^\varepsilon_{\alpha,\lambda}$ (resp. $T'_{\alpha,\lambda}$), we get the equality for any ordering by applying the inverse map.

\ref{PropRankOneLevi:normalizers} If $n \in N^\varepsilon_{\alpha,\lambda}$, then $n \in U^+ T U^-$ and using a spherical Bruhat decomposition \cite[6.1.15 (c)]{BruhatTits1}, we get $n \in T$.
Thus $N^\varepsilon_{\alpha,\lambda} \subset T^\varepsilon_{\alpha,\lambda}$.
The same holds in $L'_{\alpha,\lambda} \subset U^+ T U^-$.

\ref{PropRankOneLevi:inGamma} We know that $\emptyset \neq M_{\alpha,\lambda} \subset U_{\alpha,\lambda} U_{-\alpha,-\lambda} U_{\alpha,\lambda} \subset L_{\alpha,\lambda}$ by Proposition~\ref{PropValuedCoset}\ref{PropValuedCoset:emptyness} and~\ref{PropValuedCoset:insideLevi} and definitions.
Consider $H = L_{\alpha,\lambda} \cap T_b \subset T_{\alpha,\lambda}$ and $m \in M_{\alpha,\lambda}$.
Define \[X = \left( U_{\alpha,\lambda} H U'_{-\alpha,-\lambda}\right) \cup \left( U_{\alpha,\lambda} m H U_{\alpha,\lambda}\right) \subset L_{\alpha,\lambda}.\]
By Lemma~\ref{LemMuNu}, we know that $\nu(m) = r_{\alpha,\lambda}$.
Hence, we have $\nu(m^2) = r_{\alpha,\lambda}^2 = \operatorname{id}$.
Thus $m^2 \in T_b \cap L_{\alpha,\lambda} = H$ and $m^{-1} \in Hm = mH$.
Since $U_{-\alpha,-\lambda}  = m^{-1} U_{\alpha,\lambda} m$ by Lemma~\ref{LemConjugationMaUb}, we deduce that $L_{\alpha,\lambda}$ is generated by $U_{\alpha,\lambda}$ and $m$.
Thus, it suffices to prove that $X$ is stable by right multiplication by $m$ and elements in the group $U_{\alpha,\lambda}$.
It is convenient to firstly prove that $X$ is stable by right multiplication by elements in the group $H$.

Since $H \subset T_b$ normalizes $U'_{-\alpha,-\lambda}$ and $U_{\alpha,\lambda}$ and $T'_{\alpha,\lambda} = L'_{\alpha,\lambda} \cap T_b \subset H$, we deduce from~\ref{PropRankOneLevi:epsilon} that $L'_{\alpha,\lambda} H = U_{\alpha,\lambda} H U'_{-\alpha,-\lambda}$ is a group.
Hence $X H = X U_{\alpha,\lambda} = X$.

On the one hand, we have $U_{\alpha,\lambda} H U'_{-\alpha,-\lambda} m \subset U_{\alpha,\lambda} H m U_{\alpha,\lambda}$.
On the other hand, we have $U_{\alpha,\lambda} H m U_{\alpha,\lambda} m \subset U_{\alpha,\lambda} H U_{-\alpha,-\lambda} m^2 = U_{\alpha,\lambda} H U'_{-\alpha,-\lambda} \cup U_{\alpha,\lambda} H \varphi_{-\alpha}^{-1}(\{-\lambda\})$.
Let $u \in \varphi_{-\alpha}^{-1}(\{-\lambda\})$.
By Proposition~\ref{PropValuedCoset}\ref{PropValuedCoset:reverselambda} and~\ref{PropValuedCoset:opposite}, there exist $u',u'' \in U_{\alpha,\lambda}$ and $m' \in M_{-\alpha,-\lambda} =M_{\alpha,\lambda}$ such that $m' = u'uu''$, so that $m' \in L_{\alpha,\lambda}$.
But $\nu(mm') = \nu(m)\nu(m') = r^2_{\alpha,\lambda} = \operatorname{id}$ by Lemma~\ref{LemMuNu} since $m,m' \in M_{\alpha,\lambda}$.
Thus $mm' \in T_b \cap L_{\alpha,\lambda} = H$.
Hence $u \in U_{\alpha,\lambda} H m U_{\alpha,\lambda}$.
As a consequence, $U_{\alpha,\lambda} H \varphi_{-\alpha}^{-1}(\{-\lambda\}) \subset U_{\alpha,\lambda} H m U_{\alpha,\lambda}$ since $H \subset T_b$ normalizes $U_{\alpha,\lambda}$.
This proves that $X m \subset X$ and therefore $L_{\alpha,\lambda} = X$.
Moreover, $\left(U_{\alpha,\lambda} H U'_{-\alpha,-\lambda} \right) \cap \left( U_{\alpha,\lambda} m H U_{\alpha,\lambda} \right) = \emptyset$.
Indeed, by contradiction, we would have $m \in U_{\alpha,\lambda} H U'_{-\alpha,-\lambda} H U_{\alpha,\lambda} = H L'_{\alpha,\lambda}$ since $H\subset T_b$ normalizes $U'_{-\alpha,-\lambda}$ and $U_{\alpha,\lambda}$.
Thus $m \in N \cap H L'_{\alpha,\lambda} = H N'_{\alpha,\lambda} \subset T_b$ by~\ref{PropRankOneLevi:normalizers} which is a contradiction with $m \in M_{\alpha,\lambda}$.

It remains to show that $H = T_{\alpha,\lambda}$.
Let $t \in T_{\alpha,\lambda} = T \cap L_{\alpha,\lambda}$. By contradiction, suppose that $t \in U_{\alpha,\lambda} m H U_{\alpha,\lambda}$.
Write $m = u v u'$ with $u,u' \in U_{\alpha,\lambda}$ and $v \in \varphi_{-\alpha}^{-1}(\{-\lambda\})$.
This is possible by Proposition~\ref{PropValuedCoset}~\ref{PropValuedCoset:opposite} and~\ref{PropValuedCoset:insideLevi}.
Then $v \in T U_\alpha \cap U_{-\alpha} = \{1\}$ by axiom~\ref{axiomRGD6} which is a contradiction with $\varphi_{-\alpha}(v) = -\lambda$.
Thus $t \in U_{\alpha,\lambda} H U'_{-\alpha,-\lambda}$.

By Lemma~\ref{LemTripleLeviCommutation}~\ref{LemTripleLeviGroup}, 
we deduce $H = T_{\alpha,\lambda} \subset T_b$.

Let $n \in N_{\alpha,\lambda}$.
If $n \in  U_{\alpha,\lambda} T_{\alpha,\lambda} U'_{-\alpha,-\lambda} \subset U^- T U^+$, then $n \in T$ using a Bruhat decomposition \cite[6.1.15(c)]{BruhatTits1}.
Hence $n \in T \cap L_{\alpha,\lambda} = T_{\alpha,\lambda}$.
Otherwise, $n \in U_{\alpha,\lambda} m T_{\alpha,\lambda} U_{\alpha,\lambda}$.
Hence $nm \in U_{\alpha,\lambda} T_{\alpha,\lambda} U_{-\alpha,-\lambda}$ because $m^2 \in H \subset T_{\alpha,\lambda}$ and by Lemma~\ref{LemConjugationMaUb}.
Thus $nm \in T$ and therefore $n \in T_{\alpha,\lambda} m$.
Hence $N_{\alpha,\lambda} = T_{\alpha,\lambda} \{1,m\}$.

\ref{PropRankOneLevi:notinGamma} If $\lambda \not\in \Gamma_\alpha$, we know that $-\lambda \not\in - \Gamma_\alpha = \Gamma_{-\alpha}$ by Remark~\ref{rqGamma_alpha}.
Hence we have $U_{-\alpha,-\lambda}=U'_{-\alpha,-\lambda}$ by Remark~\ref{RkRankOneRootGroup} and therefore $L'_{\alpha,\lambda} = L_{\alpha,\lambda}$ by definition.
Hence $N_{\alpha,\lambda} = N'_{\alpha,\lambda} = T'_{\alpha,\lambda}$.
Since $L'_{\alpha,\lambda} = U_{\alpha,\lambda} T'_{\alpha,\lambda} U_{-\alpha,-\lambda}$, we deduce that $T_{\alpha,\lambda} = T'_{\alpha,\lambda}$
by Lemma~\ref{LemTripleLeviCommutation}~\ref{LemTripleLeviGroup}.

\ref{PropRankOneLevi:tori} has been shown among the proof.
\end{proof}

\begin{Cor}\label{CorLeviRankOne}
For any $\varepsilon \geqslant 0$, any $\alpha\in \Phi$ and any $\lambda \in \Rtot$, we have \[L^\varepsilon_{\alpha,\lambda} 
= U_{\alpha,\lambda} U_{-\alpha,-\lambda + \varepsilon} N^\varepsilon_{\alpha,\lambda}
= U_{-\alpha,-\lambda + \varepsilon} U_{\alpha,\lambda} N^\varepsilon_{\alpha,\lambda} .\]
\end{Cor}

\begin{proof}
If $\varepsilon > 0$, it is a consequence of~\ref{PropRankOneLevi:epsilon},~\ref{PropRankOneLevi:normalizers} and~\ref{PropRankOneLevi:tori} since $T_b$ normalizes $U_{\alpha,\lambda}$ and $U_{-\alpha,-\lambda+\varepsilon}$.

If $\varepsilon = 0$, the first equality is a consequence of~\ref{PropRankOneLevi:inGamma},~\ref{PropRankOneLevi:notinGamma} and ~\ref{PropRankOneLevi:tori} since $T_b$ normalizes $U_{-\alpha,-\lambda}$, $U'_{-\alpha,-\lambda}$ and $U_{\alpha,\lambda}$ and since we have $mU_{\alpha,\lambda}m ^{-1} = U_{-\alpha,-\lambda}$ by Lemma~\ref{LemConjugationMaUb}.
The last equality is obtained in the same way by exchanging $(\alpha, \lambda)$ with $(-\alpha,-\lambda)$ since $L_{\alpha,\lambda} = L_{-\alpha,-\lambda}$ by definition and, therefore, $N_{\alpha,\lambda} = N_{-\alpha,-\lambda}$.
\end{proof}

\subsubsection*{Technical lemmas of computation of some commutators}

We want to estimate some commutators in terms of the valuation of root groups.

The following Lemma is \cite[6.3.5]{BruhatTits1} with $\varepsilon$ denoting the $r+s$ of Bruhat-Tits.

\begin{Lem}\label{LemCommutation-a2a}
Let $\alpha \in \Phi$.
Let $\lambda\in \Rtot$ and $\varepsilon > 0$.
For any $u \in U_{2\alpha,2\lambda}$ and any $v \in U_{-\alpha,-\lambda + \varepsilon}$, we have
\[ [u,v] \in U_{\alpha,\lambda+\varepsilon} T_b U_{-\alpha,-\lambda+2\varepsilon}.\]
\end{Lem}

\begin{proof}
We can assume that $u \neq 1$ and $v \neq 1$.

Let $\mu = \varphi_{\alpha}(u) = \frac{1}{2} \varphi_{2\alpha}(u) \in \frac{1}{2} \Gamma_{2\alpha} \subset \Gamma_\alpha$ and $\rho = \varphi_{-\alpha}(v) \in \Gamma_{-\alpha}$.
Consider $m \in M_{2\alpha,2\mu}$ and $n \in M_{-\alpha,\rho}$ which is possible by Proposition~\ref{PropValuedCoset}~\ref{PropValuedCoset:emptyness}.
Let $v',v'' \in U_{-2\alpha,-2\mu}$ be such that $u = v'mv''$ which is possible by Proposition~\ref{PropValuedCoset}~\ref{PropValuedCoset:reverselambda}.
Let $u',u'' \in U_{\alpha,-\rho}$ be such that $v = u'nu''$.

On the one hand, one can write
\begin{align}
 [u,v] =& u ( u'nu'') u^{-1} (u'nu'')^{-1}  \nonumber \\
 =& (u  u')nu'' u^{-1} (u'')^{-1} u u^{-1} n^{-1} (u')^{-1}  \nonumber \\
 =& (uu') ( n [u'',u^{-1} ] n^{-1} )( n u^{-1} n^{-1} ) (u')^{-1} \label{eqnCommutator1}
\end{align}
By axiom~\ref{axiomRGD2}, we have $[u'',u^{-1}] = 1$.
By Lemma~\ref{LemConjugationMaUb}, we have $n u^{-1} n^{-1} \in n U_{\alpha,\mu} n^{-1} = U_{-\alpha,\mu+2\rho}$.
Since $\mu + \rho \geqslant \varepsilon >0$,
by Lemma~\ref{LemTripleLeviCommutation}, we have 
$U_{-\alpha,\mu+2\rho} U_{\alpha,-\rho} \subset U_{\alpha,-\rho} T_b U_{-\alpha,\mu+2\rho}$.
Moreover, $\mu + 2 \rho \geqslant - \lambda + 2 \varepsilon$,
whence $[u,v] \in U_{\alpha} T_b U_{-\alpha,-\lambda+2\varepsilon}$.

On the other hand, an analogous writing gives
\begin{equation}
 [u,v] = v' (m [v'', v] m^{-1}) (m v m^{-1}) (v')^{-1} v^{-1}\label{eqnCommutator2}
\end{equation}
By axiom~\ref{axiomRGD2}, we have $[v'',v] = 1$.

By Lemma~\ref{LemConjugationMaUb}, we have $m v m^{-1} \in m U_{-\alpha,\rho} m^{-1} = U_{\alpha,\rho+2\mu}$.
By Lemma~\ref{LemTripleLeviCommutation}, we have
$U_{-\alpha,-\mu}U_{\alpha,\rho+2\mu} \subset U_{\alpha, \rho+2\mu}T_bU_{-\alpha,-\mu}$.
Moreover $\rho+2 \mu \geqslant \lambda + \varepsilon$,
whence $[u,v] \subset U_{\alpha, \lambda+ \varepsilon} T_b U_{-\alpha}$.

By uniqueness of the writing in $U^+TU^-$ (axiom~\ref{axiomRGD6}), equations (\ref{eqnCommutator1}) and (\ref{eqnCommutator2}) give
\[ [u,v] \in U_{\alpha} T_b U_{-\alpha,-\lambda+2\varepsilon} \cap U_{\alpha,\lambda+\varepsilon} T_b U_{-\alpha} \subset U_{\alpha,\lambda+\varepsilon} T_b U_{-\alpha,-\lambda+2\varepsilon}.\]
\end{proof}

The following Lemma is \cite[6.3.6]{BruhatTits1} with $\varepsilon$ denoting the $k+\ell$ of Bruhat-Tits.

\begin{Lem}\label{LemRankOneLeviWith2a}
Let $\alpha \in \Phi$.
Let $\lambda \in \Rtot$ and $\varepsilon > 0$.
Then the product
\[ U_{2\alpha,2\lambda-\varepsilon} U_{\alpha,\lambda} T_b U_{-\alpha,-\lambda+\varepsilon} \]
is a group.
\end{Lem}

\begin{proof}
By Lemma~\ref{LemTripleLeviCommutation}, this subset is stable by right multiplication by elements in $U_{\alpha,\lambda}$, $T_b$ and $U_{-\alpha,-\lambda+\epsilon}$.
If $u \in U_{2\alpha,2\lambda-\varepsilon}$.
Then for any $v \in U_{-\alpha,-\lambda+\varepsilon}$, we have
\[ vu = u [u^{-1}, v] v \in u U_{\alpha,\lambda}T_bU_{-\alpha,-\lambda+\frac{3}{2}\varepsilon}v\]
by Lemma~\ref{LemCommutation-a2a}.
Thus $U_{2\alpha,2\lambda-\varepsilon} U_{\alpha,\lambda} T_b vu \subset U_{2\alpha,2\lambda-\varepsilon} U_{\alpha,\lambda} T_b U_{2\alpha,2\lambda-\varepsilon} U_{\alpha,\lambda}T_bU_{-\alpha,-\lambda+\frac{3}{2}\varepsilon}U_{-\alpha,-\lambda+\varepsilon}$.

Since $[U_{2\alpha},U_\alpha]=\{1\}$ by axiom~\ref{axiomRGD2}, the group $U_{2\alpha,2\lambda-\varepsilon}$ normalizes $U_{\alpha,\lambda}$ and $T_b$.
Since $U_{-\alpha,-\lambda + \frac{3}{2} \varepsilon}$ is a subgroup of $U_{-\alpha,-\lambda+\varepsilon}$, we are done.

\end{proof}

The following Lemma is \cite[6.3.7]{BruhatTits1} with $\varepsilon$ denoting the $k+\ell$ and $\lambda$ denoting the $k$ of Bruhat-Tits. 

\begin{Lem}\label{LemCommutatorRankOne}
Let $\alpha \in \Phi$.
Let $\lambda \in \Rtot$ and $\varepsilon >0$.
For any $u \in U_{\alpha,\lambda}$ and $v \in U_{-\alpha,-\lambda+\varepsilon}$, we have
\begin{equation}\label{eqnCommutator3}
[u,v] \in U_{2\alpha,2\lambda+\varepsilon} U_{\alpha,\lambda+\varepsilon} T_b U_{-\alpha,-\lambda+2\varepsilon} U_{-2\alpha,-2\lambda + 3\varepsilon}.
\end{equation}

If, moreover, $u \in U_{2\alpha}$, that is $u \in U_{2\alpha,2\lambda}$, then
\begin{equation}\label{eqnCommutator4}
[u,v] \in U_{2\alpha,2\lambda+\varepsilon} U_{\alpha,\lambda+\varepsilon} T_b U_{-\alpha,-\lambda+3\varepsilon} U_{-2\alpha,-2\lambda + 4\varepsilon}.
\end{equation}

\end{Lem}

Note that $U_{2\alpha,2\lambda+2\varepsilon} U_{\alpha,\lambda+\varepsilon} = U_{\alpha,\lambda+\varepsilon}$.

\begin{proof}
We can assume that $u \neq 1$ and $v \neq 1$.
We keep the same notations as in the proof of Lemma~\ref{LemCommutation-a2a}:
$\lambda \leqslant \mu = \varphi_{\alpha}(u) \in \Gamma_\alpha$ and $-\lambda+\varepsilon \leqslant \rho = \varphi_{-\alpha}(v) \in \Gamma_{-\alpha}$;
$m \in M_{\alpha,\mu}$ and $n  \in M_{-\alpha,\rho}$;
$v',v'' \in U_{-\alpha,-\mu}$ such that $u = v'mv''$;
$u',u'' \in U_{\alpha,-\rho}$ such that $v = u'nu''$.

By axiom~\ref{axiomV3}, we have $[u'',u^{-1}] \in U_{2\alpha,\mu-\rho} $.
Hence, by Lemma~\ref{LemConjugationMaUb}, we have 
\[ n U_{2\alpha,\mu-\rho} n^{-1} \subset U_{-2\alpha} \cap n U_{\alpha,\frac{\mu-\rho}{2}} n^{-1} = U_{-2\alpha} \cap U_{-\alpha,\frac{\mu-\rho}{2} + 2 \rho} = U_{-2\alpha,\mu+3\rho}.\]

Moreover, $nu^{-1}n^{-1} \in n U_{\alpha,\mu} n^{-1} = U_{-\alpha, \mu+2 \rho}$ by Lemma~\ref{LemConjugationMaUb}.
Thus, formula (\ref{eqnCommutator1}) gives

\[ [u,v] \in U_\alpha U_{-2\alpha,\mu+3\rho} U_{-\alpha,\mu+2\rho} U_{\alpha,-\rho}.\]
By Lemma~\ref{LemRankOneLeviWith2a}
applied to $(-\alpha,\mu+2\rho, \mu + \rho)$ in the role of $(\alpha, \lambda,\varepsilon)$, 
the product $U_{-2\alpha,\mu+3\rho} U_{-\alpha,\mu+2\rho} T_b U_{\alpha,-\rho}$
is a group which is equal to

$U_{\alpha,-\rho} T_b U_{-\alpha, \mu+2\rho} U_{-2\alpha,\mu +3\rho}$
by applying inverse map.

Since $\mu + 2 \rho \geqslant - \lambda + 2 \varepsilon$ and $\mu + 3 \rho \geqslant - 2 \lambda + 3 \varepsilon$
we get
\[ [u,v] \in U_\alpha T_b U_{-\alpha,-\lambda+2\varepsilon} U_{-2\alpha,-2\lambda+3\varepsilon}.\]

By axiom~\ref{axiomV3}, we have $[v'',v] \in U_{-2\alpha,\rho-\mu} $.
Hence, by Lemma~\ref{LemConjugationMaUb}, we have 
\[ m U_{-2\alpha,\rho-\mu} m^{-1} \subset U_{2\alpha} \cap m U_{-\alpha,\frac{\rho-\mu}{2}} m^{-1} = U_{2\alpha} \cap U_{\alpha,\frac{\rho-\mu}{2} + 2 \mu} = U_{2\alpha,3\mu+\rho}.\]
Moreover, $mvm^{-1} \in m U_{-\alpha,\rho} m^{-1} \subset U_{\alpha, \rho + 2 \mu}$ by Lemma~\ref{LemConjugationMaUb}.
Thus, formula (\ref{eqnCommutator2}) gives
\[ [u,v] \in U_{-\alpha,-\mu} U_{2\alpha,3\mu+\rho} U_{\alpha,\rho + 2 \mu} U_{-\alpha}.\]
Applying Lemma~\ref{LemRankOneLeviWith2a} to

$(\alpha, \rho+2\mu, \rho + \mu)$ in  the role of $(\alpha, \lambda, \varepsilon)$,
since $\rho + 2 \mu \geqslant \lambda+ \varepsilon$ and $\rho + 3 \mu \geqslant 2 \lambda + \varepsilon$

we get
\[ [u,v] \in U_{2\alpha,2\lambda+\varepsilon} U_{\alpha,\lambda+\varepsilon} T_b U_{-\alpha}.\]

Uniqueness of the writting in $U^+TU^-$ (axiom~\ref{axiomRGD6}) gives (\ref{eqnCommutator3}).

If, moreover, $u \in U_{2\alpha}$, then we have $[u'',u^{-1}] = 1$ by axiom~\ref{axiomRGD2}.

Moreover, $nu^{-1}n^{-1} \in U_{-2\alpha} \cap U_{-\alpha,\mu + 2 \rho} = U_{-2\alpha, 2 \mu + 4 \rho}$.
Thus
\[ [u,v] \in U_\alpha U_{-2\alpha,2\mu + 4 \rho} U_{\alpha,-\rho}.\]
Applying Lemma~\ref{LemRankOneLeviWith2a} to $(-\alpha,2\mu+3\rho,2(\mu+\rho))$, we get 
\[ [u,v] \in U_\alpha T_b U_{-\alpha,2\mu + 3 \rho} U_{-2\alpha,2\mu + 4 \rho}.\]
Since $2 \mu + 3 \rho \geqslant - \lambda + 3 \varepsilon$ and $2 \mu + 4 \rho \geqslant - 2\lambda + 4 \varepsilon$, we obtain
\[ [u,v] \in U_\alpha T_b U_{-\alpha,-\lambda + 3 \varepsilon} U_{-2\alpha,-2 \lambda + 4 \varepsilon}.\]

Hence, by Lemma~\ref{LemCommutation-a2a} and uniqueness of the writting in $U^+TU^-$ (axiom~\ref{axiomRGD6}), we get (\ref{eqnCommutator4}).
\end{proof}

\begin{Not}
Let $\lambda \in \Rtot$, $\varepsilon > 0$ and $\alpha \in \Phi$.
Consider $u \in U_{\alpha,\lambda}$ and $v \in U_{-\alpha,-\lambda + \varepsilon}$.
Then, according to Lemma~\ref{LemTripleLeviCommutation} and axiom~\ref{axiomRGD6}, there is a unique $t(u,v)\in T_b$ such that $[u,v] \in U_{\alpha,\lambda} t(u,v) U_{-\alpha,-\lambda+\varepsilon}$.
The element $t(u,v)$ is then called the $T$-component of $[u,v]$.
\end{Not}

The following Proposition details \cite[6.3.9]{BruhatTits1}.
This proposition is really important to prove that the fibers of the $\RF^S$-building for the projection maps also are buildings, see Proposition~\ref{PropTstar1InTb}.

\begin{Prop}\label{PropTcomponent}
Let $\lambda \in \Rtot$, $\varepsilon > 0$ and $\alpha \in \Phi$. Then 
$T^\varepsilon_{\alpha,\lambda}$ is the group generated by the $T$-components $t(u,v)$ of the commutators $[u,v]$ for $u \in U_{\alpha,\lambda}$ and $v \in U_{-\alpha,-\lambda+\varepsilon}$.

In particular, $T'_{\alpha,\lambda}$ is the group generated by the $T$-components $t(u,v)$ of commutators $[u,v]$ for $u \in U_{\alpha,\lambda}$ and $v \in U'_{-\alpha,-\lambda}$.
\end{Prop}

\begin{proof}
Denote by $X$ the subgroup of $T_b$ generated by $\{ t(u,v),\ u \in U_{\alpha,\lambda},\ v\in U_{-\alpha,-\lambda+\varepsilon}\}$.
Since $[u,v] \in L_{\alpha,\lambda}^\varepsilon$, we have $X \subset L_{\alpha,\lambda}^\varepsilon \cap T_b$.
We prove that $Y = U_{\alpha,\lambda} X U_{-\alpha,-\lambda+\varepsilon}$ is a group.
It is stable by left multiplication by elements in $U_{\alpha,\lambda}$ and $X$ since the subgroup $X \subset T_b$ normalizes $U_{\alpha,\lambda}$.
For $v \in U_{-\alpha,-\lambda+\varepsilon}$ and $uxw \in Y$, we have
$vuxw = u [u^{-1},v]vxw$.
Let $u' \in U_{\alpha,\lambda}$ and $v' \in U_{-\alpha,-\lambda+\varepsilon}$ be such that $[u^{-1},v] = u' t(u^{-1},v) v'$.
Then $vuxw = (u u') (t(u^{-1},v) x) (x^{-1}v' vx) w$.
We have $u u' \in U_{\alpha,\lambda}$ and $t(u^{-1},v) x \in X$.
Moreover, since $x \in X$ normalizes $U_{-\alpha,-\lambda+\varepsilon}$, we have $(x^{-1}v' vx) w \in U_{-\alpha,-\lambda+\varepsilon}$.
This proves that $Y$ is a subgroup of $L^\varepsilon_{\alpha,\lambda}$ containing $U_{\alpha,\lambda}$ and $U_{-\alpha,-\lambda+\varepsilon}$.
Hence $Y = L^\varepsilon_{\alpha,\lambda}$.
By Proposition~\ref{PropRankOneLevi}~\ref{PropRankOneLevi:epsilon} and  uniqueness in axiom~\ref{axiomRGD6}, we get $X = T^\varepsilon_{\alpha,\lambda}$.

Since $U'_{-\alpha,-\lambda}$ is the union $\bigcup_{\varepsilon > 0} U_{-\alpha,-\lambda+\varepsilon}$, we get that $T'_{\alpha,\lambda}= \bigcup_{\varepsilon > 0} T^\varepsilon_{\alpha,\lambda}$ by intersection of $L'_{\alpha,\lambda} = \bigcup_{\varepsilon > 0} U_{\alpha,\lambda} U_{-\alpha,-\lambda+\varepsilon} T^\varepsilon_{\alpha,\lambda}$ with $T$.
This gives the second assertion of the Proposition.
\end{proof}

We obtain the following Corollary, corresponding to \cite[6.4.25 (ii) and (iii)]{BruhatTits1} which does not use infimum and supremum.

\begin{Cor}\label{CorCommutationTaUa}
Let $\lambda \in \Rtot$, $\varepsilon \geqslant 0$ and $\alpha \in \Phi$.
Then the following commutator subgroups satisfy:
\begin{equation}
\left[ T_{\alpha,\lambda}^\varepsilon, U_{\alpha,\lambda} \right] \subset U_{\alpha,\lambda + \varepsilon} U_{2\alpha,2\lambda+\varepsilon} \label{eqnCommutator5}
\end{equation} 
and
\begin{equation}
\left[ T_{\alpha,\lambda}^\varepsilon, U_{-\alpha,-\lambda} \right] \subset U_{-\alpha,-\lambda + \varepsilon} U_{-2\alpha,-2\lambda+\varepsilon}.\label{eqnCommutator6} \end{equation}
Moreover, 
\begin{equation}
\left[ T_{\alpha,\lambda}', U_{\alpha,\lambda} \right] \subset U'_{\alpha,\lambda} \label{eqnCommutator7}
\end{equation}
and
\begin{equation}
\left[ T_{\alpha,\lambda}', U_{-\alpha,-\lambda} \right] \subset U'_{-\alpha,-\lambda} \label{eqnCommutator8}
\end{equation}
\end{Cor}

\begin{proof}
If $\varepsilon = 0$, the inclusions are immediate since $T_{\alpha,\lambda} \subset T_b$ normalizes $U_{\alpha,\lambda}$ and $U_{2\alpha,2\lambda} \subset U_{\alpha,\lambda}$.
Assume that $\varepsilon >0$ and consider the product
\[Z = U_{2\alpha,2\lambda+\varepsilon} U_{\alpha,\lambda+\varepsilon} T_b U_{-\alpha,-\lambda+\varepsilon} U_{-2\alpha,-2\lambda+\varepsilon}.\]

To prove that $Z$ is a group, it suffices to prove that $Z$ is stable by right multiplication by elements in the subgroups $U_{2\alpha,2\lambda+\varepsilon}$, $ U_{\alpha,\lambda+\varepsilon}$, $T_b$, $U_{-\alpha,-\lambda+\varepsilon}$ and $U_{-2\alpha, -2\lambda+\varepsilon}$.
It is obvious for the three last ones.
Let $u \in U_{2\alpha,2\lambda+\varepsilon} U_{\alpha,\lambda+\varepsilon}$ and $v \in U_{-\alpha,-\lambda+\varepsilon} U_{-2\alpha,-2\lambda+\varepsilon}$.
Set $\mu = \lambda + \frac{\varepsilon}{2}$.
Then $u \in U_{\alpha,\mu}$ and $v \in U_{-\alpha,-\mu+\varepsilon}$.
Thus, applying~(\ref{eqnCommutator3}) of Lemma~\ref{LemCommutatorRankOne},
\[[u^{-1},v] \in U_{2\alpha,2\mu+\varepsilon} U_{\alpha,\mu+\varepsilon} T_b U_{-\alpha,-\mu+2\varepsilon} U_{-2\alpha,-2\mu + 3\varepsilon}
=U_{2\alpha,2\lambda+2\varepsilon} U_{\alpha,\lambda+\frac{3}{2}\varepsilon} T_b U_{-\alpha,-\lambda+\frac{3}{2}\varepsilon} U_{-2\alpha,-2\lambda + 2\varepsilon}.\]
Hence $[u^{-1},v] \in Z$ and therefore
\[
vu=u[u^{-1},v]v \in uZv = Z.
\]
Hence $Z U_{2\alpha,2\lambda+\varepsilon} U_{\alpha,\lambda+\varepsilon} = Z$ and therefore $Z$ is a group.

We prove (\ref{eqnCommutator5}) by showing that the following set of generators satisfy:
\[\{[t(u,v),x],\ x,u \in U_{\alpha,\lambda},\ v \in U_{-\alpha,-\lambda+\varepsilon}\}\subset U_{\alpha,\lambda+\varepsilon}U_{2\alpha,2\lambda+\varepsilon}.\]
Let $u,x\in U_{\alpha,\lambda}$ and $v\in U_{-\alpha,-\lambda+\varepsilon}$.
We will show that $[t(u,v),x] \in Z$.
By Lemma~\ref{LemCommutatorRankOne}, we have $[u,v] \in Z$ and we can write it as $[u,v] = u_{2} u_{1} t(u,v) v_{1} v_{2}$.
By Lemma~\ref{LemCommutatorRankOne} and axiom~\ref{axiomV3}, all the commutators $[v^{-1},xu]$, $[v^{-1},x]$, $[x,u_{1}]$, $[x,u_{2}]$, $[x,v_{1}]$ and $[x,v_{2}]$ are in $Z$.
We have that $x[u,v]x^{-1} \in Z$.
Indeed, $xuv = v[v^{-1},xu]xu \in Z xu$ and $xvu = v[v^{-1},x]xu \in Z xu$ since $v \in Z$.
Moreover $xu_ix^{-1} = [x,u_i]u_i \in Z$ and $xv_ix^{-1} = [x,v_i]v_i \in Z$ for $i \in \{1,2\}$ since $u_i,v_i \in Z$.
Hence $xt(u,v)x^{-1} \in Z$ and we get $[t(u,v),x] \in Z$ since $t(u,v) \in T_b \subset Z$.
Since $T_b$ normalizes $U_{\alpha,\lambda}$, we have $[t(u,v),x] \in U_\alpha$.
Thus $[t(u,v),x] \in Z \cap U_\alpha = U_{2\alpha,2\lambda+\varepsilon} U_{\alpha,\lambda+\varepsilon}$.

We prove inclusion (\ref{eqnCommutator6}) on the set of generators $[t(u,v),x]$ for $x \in U_{-\alpha,-\lambda}$, $u \in U_{\alpha,\lambda}$ and $v \in U_{-\alpha,-\lambda+\varepsilon}$.
Let $\mu = \varphi_{-\alpha}(x) \in \Gamma_{-\alpha}$.
If $\mu \geqslant -\lambda+\varepsilon$, then $x \in Z$ and therefore $[t(u,v),x] \in Z$.
Otherwise, denote $\varepsilon' = \lambda+\mu$ so that $0 \leqslant \varepsilon' < \varepsilon$.
By Lemma~\ref{LemRankOneLeviWith2a} applied to the root $-\alpha$ and the parameters $-\lambda + \varepsilon \in \Rtot$ and $\varepsilon-\varepsilon' > 0$, 
the set
\[Z' = U_{\alpha,-(-\lambda+\varepsilon)+(\varepsilon-\varepsilon')} T_b U_{-\alpha,-\lambda+\varepsilon} U_{-2\alpha,2(-\lambda+\varepsilon)-(\varepsilon-\varepsilon')} = U_{\alpha,-\mu} T_b U_{-\alpha,-\lambda+\varepsilon} U_{-2\alpha,-2\lambda+\varepsilon+\varepsilon'}\]
is a group.
Using Proposition~\ref{PropValuedCoset}~\ref{PropValuedCoset:reverselambda}, we write $x=ymz$, with $m \in M_{-\alpha,\mu} = M_{\alpha,-\mu}$ and $y,z \in U_{\alpha,-\mu}$.

Denote $t=t(u,v)$.
Then \[[t,x] = tymzt^{-1}z^{-1}m^{-1}y^{-1} = [t,y] y [t,m] m [t,z] m^{-1} y^{-1}.\]
We have $[t,m] \in T_b \subset Z'$ since $\nu([t,m]) = [1,r_{-\alpha,\mu}] = \operatorname{id}$.

As in Remark~\ref{RkVarChangeinT}, we have
\[
L_{\alpha,\lambda}^\varepsilon
= \langle U_{\alpha,\lambda}, U_{-\alpha,-\lambda+\varepsilon} \rangle
= \langle U_{\alpha,-\mu+\varepsilon'}, U_{-\alpha,\mu+\varepsilon-\varepsilon'}\rangle
 \subset \langle U_{\alpha,-\mu}, U_{-\alpha,\mu+\varepsilon-\varepsilon'}\rangle
= L_{\alpha,-\mu}^{\varepsilon-\varepsilon'}
\]
By inclusion (\ref{eqnCommutator5}), we have $[t,y],[t,z] \in U_{\alpha,-\mu+\varepsilon-\varepsilon'} U_{2\alpha,-2\mu+\varepsilon-\varepsilon'} \subset U_{\alpha,-\mu} \subset Z'$.

By Lemma~\ref{LemConjugationMaUb}, we have \[m U_{\alpha,-\mu+\varepsilon} U_{2\alpha,-2\mu+\varepsilon} m^{-1} = U_{-\alpha,-\mu+\varepsilon+2\mu} U_{-2\alpha,-2\mu+\varepsilon+4\mu}=U_{-\alpha,\mu+\varepsilon}U_{-2\alpha,2\mu+\varepsilon}.\]
Since $\mu+\varepsilon \geqslant -\lambda+\varepsilon$ and $2\mu+\varepsilon \geqslant -2\lambda+\varepsilon+\varepsilon'$, we have $m [t,z] m^{-1} \in U_{-\alpha,-\lambda+\varepsilon} U_{-2\alpha,-2\lambda+\varepsilon+\varepsilon'} \subset Z'$.
Since $y \in Z'$, we get that $[t,x] \in Z'$ as a product of such elements.
Moreover, $[t,x] \in U_{-\alpha}$ since $T$  normalizes $U_{-\alpha}$.
Hence, we have either $[t,x] \in Z \cap U_{-\alpha}$ or $[t,x] \in Z' \cap U_{-\alpha}$.
By uniqueness in axiom~\ref{axiomRGD6}, we get $U_{-\alpha,-\lambda+\varepsilon} U_{-2\alpha,-2\lambda+\varepsilon+\varepsilon'} = Z' \cap U_{-\alpha} \subset Z \cap U_{-\alpha} = U_{-\alpha,-\lambda+\varepsilon} U_{-2\alpha,-2\lambda+\varepsilon}$ and we are done. 

Since $L'_{\alpha,\lambda}$ is the increasing union $L'_{\alpha,\lambda}= \bigcup_{\varepsilon > 0} L_{\alpha,\lambda}^\varepsilon$,
and so is $T'_{\alpha,\lambda}= \bigcup_{\varepsilon > 0} T_{\alpha,\lambda}^\varepsilon$,
we get that the commutator subgroup $[T'_{\alpha,\lambda},U_{\alpha,\lambda}]$ (resp. $[T'_{\alpha,\lambda},U_{-\alpha,-\lambda}]$) is the increasing union of the commutator subgroups
$\bigcup_{\varepsilon > 0} [T^\varepsilon_{\alpha,\lambda},U_{\alpha,\lambda}] \subset \bigcup_{\varepsilon > 0} U_{\alpha,\lambda+\varepsilon} U_{2\alpha,2\lambda+\varepsilon}$

The latter union is equal to $U'_{\alpha,\lambda}$.
Indeed it contains $U'_{\alpha,\lambda}$
by definition.
If $\varepsilon >0$ satisfies $2\lambda + \varepsilon \in \Gamma_{2\alpha}$, then by Fact~\ref{FactDecompositionSetOfValue}, there exists $\mu \in \Gamma_\alpha$ such that $2\mu = 2\lambda + \varepsilon$.
Denote $\varepsilon' = \mu - \lambda$. Then $\varepsilon' > 0$ since $2\varepsilon' = \varepsilon > 0$.
Thus $U_{2\alpha,2\lambda + \varepsilon} \subset U_{\alpha,\lambda+\varepsilon'} \subset U'_{\alpha,\lambda}$.
Otherwise, for any $u \in U_{2\alpha,2\lambda+\varepsilon}$, we have $\varphi_{2\alpha}(u) > 2\lambda+\varepsilon$.
Hence, there is $\varepsilon' > 0$ such that $u \in U_{2\alpha,2\lambda + \varepsilon + \varepsilon'} \subset U_{\alpha,\lambda + \min(\varepsilon,\varepsilon')} \subset U'_{\alpha,\lambda}$.
Thus we get $U_{2\alpha,2\lambda+\varepsilon} \subset U'_{\alpha,\lambda}$ for every $\varepsilon > 0$ so that we get the desired equality.
As a consequence, we get inclusion~(\ref{eqnCommutator7}).
By symmetry, we obtain inclusion~(\ref{eqnCommutator8}).

\end{proof}

\subsection{Subgroups generated by unipotent elements}

In this section, we work under data and notations of~\ref{HypCAVRGD}.
Moreover, we fix a non-empty subset $\Omega \subset \mathbb{A}_\Rtot$\index[notation]{o@$\Omega$}.
For simplicity, we will forget the chosen origin $o$ of $\mathbb{A}_\Rtot$ in this section, i.e. we will denote by $\alpha(x)$ the quantity $\alpha(x-o)$ by abuse of notation.

\begin{Rq}
In Bruhat-Tits theory, some results on groups are expressed in terms of concave functions $f$ \cite[§6.4]{BruhatTits1} associated to abstract groups endowed with the structure of a valued root group datum.
In fact, only some well-chosen concave functions are considered for reductive groups over local fields. For instance, the optimized function $f'$ associated to a concave function $f$ is defined in \cite[4.5.2]{BruhatTits2} as:
\[ f'(\alpha) = \inf \{ \lambda \in \Gamma'_\alpha,\ \lambda \geqslant f(\alpha) \text{ or, when } \frac{\alpha}{2} \in \Phi,\ \frac{\lambda}{2} \geqslant f(\frac{\alpha}{2}) \}.\]
In  Bruhat-Tits theory,  since the valuation group $\Lambda$ is contained in $\R$, $f'(\alpha)$ is well defined and is an element of $\Gamma'_\alpha$ so that it is relevant to work with the group $U_{\alpha,f'(\alpha)}$.
More specifically, for a subset $\Omega \in \mathbb{A}_{\mathbb{R}}$, some concave functions $f_\Omega$ can be considered in order to define parahoric subgroups $\widehat{P}_\Omega\subset G$.
Once the building $\I$ has been constructed, this subgroup $\widehat{P}_\Omega$ turns out to be the pointwise stabilizer of $\Omega$.

Recall that the quasi-concave maps were defined by:
\[f_\Omega(\alpha) = \inf\{\lambda \in \Rtot,\ \forall x \in \Omega,\ \alpha(x) + \lambda \geqslant 0\} = \inf \bigcap_{x \in \Omega} [-\alpha(x),+\infty[ = \sup \{ -\alpha(x),\ x \in \Omega\}\]
Here, we do not use infimum and supremum, so that we replace this definition by some intersections or unions of groups.

Note that in the Bruhat-Tits theory, such a function $f_\Omega$ satisfies $f_\Omega(2\alpha) = 2 f_\Omega(\alpha)$ which is the case of equality of the concave inequality $2 f(\alpha) \geqslant f(2\alpha)$.
This equality induces an equality of groups $U_{\alpha,f} = U_{\alpha,f} U_{2\alpha,f}$ in \cite[§6.4]{BruhatTits1}.
\end{Rq}

\begin{Not}
We denote by:
\begin{align*}
U_{\alpha,\Omega} &= \bigcap_{x \in \Omega} U_{\alpha,-\alpha(x)}
&U'_{\alpha,\Omega} &= \bigcap_{x \in \Omega} U'_{\alpha,-\alpha(x)}
\end{align*}\index[notation]{u@$U_{\alpha,\Omega}$}\index[notation]{u@$U'_{\alpha,\Omega}$}

We denote by $U_\Omega$ (resp. $U'_\Omega$) the subgroup of $G$ generated by the union of subgroups $U_{\alpha,\Omega}$ (resp. $U'_{\alpha,\Omega}$) for $\alpha \in \Phi$.\index[notation]{u@$U_{\Omega}$}\index[notation]{u@$U'_{\Omega}$}

It is convenient to introduce the notation $U_{2\alpha,\Omega}$ (resp. $U'_{2\alpha,\Omega}$) as being the trivial group when $\alpha \in \Phi$ and $2\alpha \not\in \Phi$.

If it is given a choice of positive roots $\Phi^+$, we denote $U^+$ and $U^-$ as in~\ref{axiomRGD6}. We denote by:
\begin{align*}
U^+_{\Omega} &= U_{\Omega} \cap U^+
&U'^+_{\Omega} &= U'_{\Omega} \cap U^+\\
U^-_{\Omega} &= U_{\Omega} \cap U^-
&U'^-_{\Omega} &= U'_{\Omega} \cap U^-\\
T_{\Omega} &= U_{\Omega} \cap T
&T'_{\Omega} &= U'_{\Omega} \cap T\\
\widetilde{N}_{\Omega} &= U_{\Omega} \cap N
&\widetilde{N}'_{\Omega} &= U'_{\Omega} \cap N\\
\end{align*}
\index[notation]{u@$U^+_{\Omega}$}
\index[notation]{u@$U^-_{\Omega}$}
\index[notation]{u@$U'^+_{\Omega}$}
\index[notation]{u@$U'^-_{\Omega}$}
\index[notation]{t@$T_{\Omega}$}
\index[notation]{t@$T'_{\Omega}$}
\index[notation]{n@$\widetilde{N}_{\Omega}$}
\index[notation]{n@$\widetilde{N}'_{\Omega}$}

If $\Omega = \{x\}$, we denote $U_{\alpha,x}$, $U_x$, etc. instead of $U_{\alpha,\{x\}}$, $U_{\{x\}}$, etc.
\end{Not}

Note that the notation $\widetilde{N}_\Omega$ was not introduced in \cite{BruhatTits1}.
In the context of algebraic groups, under favourable assumptions (typically a simply-connectedness assumption), the rational points of a group are generated by the unipotent elements.
That is why we denote with a $\widetilde{N}_\Omega$ the $N$-component of the group $U_\Omega$ generated by some unipotent elements.

\begin{Fact}
By definition, we have the following equalities:
\[ U_{\alpha,\Omega} = \bigcap_{x \in \Omega} U_{\alpha,x} = \varphi_\alpha^{-1}\left( \bigcap_{x \in \Omega} [-\alpha(x),\infty] \right)\]
and
\[ U'_{\alpha,\Omega} = \bigcap_{x \in \Omega} U'_{\alpha,x} = \varphi_\alpha^{-1}\left( \bigcap_{x \in \Omega} ]-\alpha(x),\infty] \right).\]
\end{Fact}

Note that the intersections $\bigcap_{x \in \Omega} [-\alpha(x),\infty]$ and $\bigcap_{x \in \Omega} ]-\alpha(x),\infty]$ are convex subsets of $\Rtot$ that may not be intervals of $\Rtot$ when $\Rtot \neq \mathbb{R}$.

\begin{Fact}
The group $T_b$ normalizes $U_{\alpha,\Omega}$ and $U'_{\alpha,\Omega}$ for any $\alpha \in \Phi$.
\end{Fact}

\begin{proof}
For any $x \in \Omega$ and any $\varepsilon \geqslant 0$, we know by Corollary\ref{CorActionNUa} that $T_b$ normalizes $U_{\alpha,-\alpha(x)+\varepsilon}$.
It remains true by taking increasing unions and intersections of these groups.
\end{proof}

\begin{Lem}\label{LemU2ainUa}
For any $\alpha \in \Phi_{\mathrm{nd}}$, we have $U_{2\alpha,\Omega} \subset U_{\alpha,\Omega}$.
\end{Lem}

\begin{proof}
By axiom~\ref{axiomV4}, we have $U_{2\alpha,2\lambda} \subset U_{\alpha,\lambda}$ for any $\lambda \in \Rtot$.
Hence
\[U_{2\alpha,\Omega} = \bigcap_{x \in \Omega} U_{2\alpha,x} = \bigcap_{x \in \Omega} U_{2\alpha,-2\alpha(x)} \subset \bigcap_{x \in \Omega} U_{\alpha,-\alpha(x)} = \bigcap_{x \in \Omega} U_{\alpha,x} = U_{\alpha,\Omega}.\]
\end{proof}

\begin{Not}\label{NotLevialphaOmega}
Let $\alpha \in \Phi$.
We denote by $L_{\alpha,\Omega}$ (resp. $L'_{\alpha,\Omega}$) the subgroup of $G$ generated by $U_{\alpha,\Omega}$ and $U_{-\alpha,\Omega}$ (resp. $U_{\alpha,\Omega}$ and $U'_{-\alpha,\Omega}$).
\index[notation]{l@$L_{\alpha,\Omega}$}\index[notation]{l@$L'_{\alpha,\Omega}$}
We denote 
\begin{align*}
T_{\alpha,\Omega} &= L_{\alpha,\Omega} \cap T &
T'_{\alpha,\Omega} &= L'_{\alpha,\Omega} \cap T \\
\widetilde{N}_{\alpha,\Omega} &= L_{\alpha,\Omega} \cap N &
\widetilde{N}'_{\alpha,\Omega} &= L'_{\alpha,\Omega} \cap N.
\end{align*}
\index[notation]{t@$T_{\alpha,\Omega}$}\index[notation]{t@$T'_{\alpha,\Omega}$}
\index[notation]{n@$\widetilde{N}_{\alpha,\Omega}$}\index[notation]{n@$\widetilde{N}'_{\alpha,\Omega}$}
\end{Not}

Note that since $\Omega$ is not empty, there exists some point $x \in \Omega$ so that we have $L_{\alpha,x} \supset L_{\alpha,\Omega}$ since $U_{\alpha,x} \supset U_{\alpha,\Omega}$ and $U_{-\alpha,x} \supset U_{-\alpha,\Omega}$.
Therefore $T_{\alpha,\Omega}$ and $T'_{\alpha,\Omega}$ are subgroups of $T_b$ according to Proposition~\ref{PropRankOneLevi}\ref{PropRankOneLevi:tori} applied to $L_{\alpha,x} = L_{\alpha,-\alpha(x)}$.

Note that, $L_{\alpha,\Omega} = L_{-\alpha,\Omega}$ by definition, but $L'_{-\alpha,\Omega}$ may differ from $L'_{\alpha,\Omega}$.

The following Lemma corresponds to \cite[6.4.7 (QC1)]{BruhatTits1} with a different proof since we do not use concave maps.

\begin{Lem}\label{LemQC1}
For any $\alpha \in \Phi$, we have:
\begin{equation}
\tag{QC1}
 L_{\alpha,\Omega}
= U_{-\alpha,\Omega} U_{\alpha,\Omega} \widetilde{N}_{\alpha,\Omega}
= U_{\alpha,\Omega} U_{-\alpha,\Omega} \widetilde{N}_{\alpha,\Omega}\label{eqQC1}
\end{equation}\axiom{QC1@\ref{axiomQC1}}
\begin{equation}
\tag{QC1'}
 \widetilde{N}'_{\alpha,\Omega} = T'_{\alpha,\Omega} \subset T_b
\qquad \text{ and } \qquad
L'_{\alpha,\Omega} = U_{\alpha,\Omega} T'_{\alpha,\Omega} U'_{-\alpha,\Omega} = U'_{-\alpha,\Omega} T'_{\alpha,\Omega} U_{\alpha,\Omega}.\label{eqQC1'}
\end{equation}\axiom{QC1'@(\ref{eqQC1'})}
\end{Lem}

Note that the group $U_\alpha$ may not be commutative, when $\alpha$ is multipliable for instance.

\begin{proof}
Obviously, the set $ U_{\alpha,\Omega} U_{-\alpha,\Omega} \widetilde{N}_{\alpha,\Omega}$ is contained in $L_{\alpha,\Omega}$ and the set $U_{\alpha,\Omega} U'_{-\alpha,\Omega} \widetilde{N}'_{\alpha,\Omega}$ is contained in $L'_{\alpha,\Omega}$ by definition. 
Let $g \in L_{\alpha,\Omega}$ (resp. $L'_{\alpha,\Omega}$) and write it as a product $g = \prod_{i=1}^m h_i$ with $h_{i} \in U_{\alpha,\Omega} \cup U_{-\alpha,\Omega}$ (resp. $h_{i} \in U_{\alpha,\Omega} \cup U'_{-\alpha,\Omega}$).
For $1 \leqslant i \leqslant m$, denote
\begin{align*}
\lambda_i =& \left\{
 \begin{array}{cl}
  \varphi_\alpha(h_i) & \text{ if } h_i \in U_{\alpha,\Omega},\\
  \infty & \text{ if } h_i \in U_{-\alpha,\Omega},
 \end{array}\right.&
\mu_i =& \left\{
 \begin{array}{cl}
  \infty & \text{ if } h_i \in U_{\alpha,\Omega},\\
  \varphi_{-\alpha}(h_i) & \text{ if } h_i \in U_{-\alpha,\Omega}.
 \end{array}\right.&
\end{align*}
Hence for any $x \in \Omega$, we have $\lambda_i \geqslant - \alpha(x)$ and $\mu_i \geqslant -(-\alpha)(x) = \alpha(x)$ (resp. $\mu_i > \alpha(x)$).
Let $\lambda = \min_{1 \leqslant i \leqslant m} \lambda_i$ and $\mu = \min_{1 \leqslant i \leqslant m} \mu_i$.
Then $\lambda = \lambda_i$ for some $i$ so that $U_{\alpha,\lambda} =U_{\alpha,\lambda_i} \subset U_{\alpha,x}$
and $\mu = \mu_j$ for some $j$ so that $U_{-\alpha, \mu} = U_{-\alpha,\mu_j} \subset U_{-\alpha,x}$.
Moreover $h_i \in U_{\alpha,\lambda} \cup U_{-\alpha,\mu}$ for any $1 \leqslant i \leqslant m$ by definition of $\lambda$ and $\mu$.
Let $\varepsilon = \lambda + \mu \geqslant -\alpha(x) + \alpha(x) = 0$ (resp. $\varepsilon > 0$).
Thus $g$ is in the group generated by $U_{\alpha,\lambda}$ and $U_{-\alpha,\mu} = U_{-\alpha,-\lambda+\varepsilon}$ which is $L^\varepsilon_{\alpha,\lambda} = U_{\alpha,\lambda} U_{-\alpha,-\lambda+\varepsilon} N^\varepsilon_{\alpha,\lambda} $ by Corollary~\ref{CorLeviRankOne}.
Since $U_{\alpha,\lambda} \subset U_{\alpha,x}$ for every $x \in \Omega$, we have $U_{\alpha,\lambda} \subset U_{\alpha,\Omega}$.
By the same way, $U_{-\alpha,-\lambda+\varepsilon} \subset U_{-\alpha,\Omega}$.
Hence $L^\varepsilon_{\alpha,\lambda} \subset L_{\alpha,\Omega}$.
Thus $N^\varepsilon_{\alpha,\lambda} = N \cap L^\varepsilon_{\alpha,\lambda} \subset N \cap L_{\alpha,\Omega} = \widetilde{N}_{\alpha,\Omega}$.
Moreover, $N^\varepsilon_{\alpha,\lambda} = T^\varepsilon_{\alpha,\lambda} \subset T'_{\alpha,\Omega}$ when $\varepsilon > 0$ by Proposition~\ref{PropRankOneLevi}\ref{PropRankOneLevi:normalizers}.
This gives $g \in U_{\alpha,\Omega} U_{-\alpha,\Omega} \widetilde{N}_{\alpha,\Omega}$ for any $g \in L_{\alpha,\Omega}$ (resp. $g \in U_{\alpha,\Omega} U'_{-\alpha,\Omega} T'_{\alpha,\Omega}$ for any $g \in L'_{\alpha,\Omega}$).
\

Moreover, using a Bruhat decomposition \cite[6.1.15(c)]{BruhatTits1}, it gives $L'_{\alpha,\Omega} \cap N = \widetilde{N}'_{\alpha,\Omega} \subset T$, which gives the first part of (\ref{eqQC1'}).
Since $\widetilde{N}'_{\alpha,\Omega} = T'_{\alpha,\Omega} \subset T_b$ normalizes $U_{\alpha,\Omega}$, we get the first equality of the second part of (\ref{eqQC1'}).

We get the second equality by applying inverse map.

\end{proof}

\begin{Lem}\label{LemCommutationUaU'a}
Let $\alpha \in \Phi$. The following commutator subgroups satisfy 
\begin{align*}
[U_{\alpha,\Omega},U'_{\alpha,\Omega}] &\subset U'_{2\alpha,\Omega} &
[U_{\alpha,\Omega},U'_{-\alpha,\Omega}] &\subset U'_{\alpha,\Omega} U'_{-\alpha,\Omega} T'_{\alpha,\Omega}
\end{align*}
\end{Lem}

\begin{proof}
We check it on a set of generators.
Consider $u \in U_{\alpha,\Omega}$ and $u' \in U'_{\alpha,\Omega}$.
Pick  any $x \in \Omega$ and set $\lambda=-\alpha(x) \in \Rtot$.
Let $\mu = \varphi_\alpha(u) = \lambda +\varepsilon$ and $\mu' = \varphi_\alpha(u') = \lambda+\varepsilon'$ with $\varepsilon \geqslant 0$ and $\varepsilon' > 0$.
Then, by axiom~\ref{axiomV3}, we have $[u,u'] \in U_{2\alpha,2\lambda+\varepsilon+\varepsilon'} \subset U'_{2\alpha,2\lambda} = U'_{2\alpha,x}$.
Hence $[u,u'] \in \bigcap_{x \in \Omega} U'_{2\alpha,x} = U'_{2\alpha,\Omega}$.

Consider $u \in U_{\alpha,\Omega}$ and $v \in U'_{-\alpha,\Omega}$.
For any $x \in \Omega$, let $\lambda = -\alpha(x) \in \Rtot$.
Then $u \in U_{\alpha,\lambda}$ and there exists $\varepsilon > 0$ such that $v \in U_{-\alpha,-\lambda+\varepsilon}$.
By Lemma~\ref{LemCommutatorRankOne}, we have $[u,v] \in U_{2\alpha,2\lambda+\varepsilon} U_{\alpha,\lambda+\varepsilon} t(u,v) U_{-\alpha,-\lambda+2\varepsilon} U_{-2\alpha,-2\lambda+3\varepsilon}$ where $t(u,v)$ is the $T$-component of $[u,v]$.
Hence, by Lemma~\ref{LemTripleLeviCommutation}, there are $u' \in U_\alpha$, $v' \in U_{-\alpha}$ uniquely determined such that $[u,v] = u't(u,v)v'$.
Moreover, $2\varphi_{\alpha}(u') \geqslant 2\lambda+\varepsilon > -2\alpha(x)$ and $2\varphi_{-\alpha}(v') \geqslant -2\lambda+3\varepsilon > 2\alpha(x)$.

Hence $u' \in U'_{\alpha,x}$ and $v' \in U'_{-\alpha,x}$ for any $x \in \Omega$.
Hence $u' \in U'_{\alpha,\Omega}$ and $v' \in U'_{-\alpha,\Omega}$.
Thus by Lemma~\ref{LemTripleLeviCommutation}~\ref{LemTripleLeviCommutation:2},
$t(u,v) \in L'_{\alpha,\Omega} \cap T = T'_{\alpha,\Omega}$.
Hence $[u,v] \in U'_{\alpha,\Omega} T'_{\alpha,\Omega} U'_{-\alpha,\Omega} = U'_{\alpha,\Omega} U'_{-\alpha,\Omega} T'_{\alpha,\Omega}$ since $T'_{\alpha,\Omega} \subset T_b$ normalizes $U'_{-\alpha,\Omega}$.
\end{proof}

The following Lemma corresponds to \cite[6.4.7 (QC2)]{BruhatTits1}. It is an immediate consequence of  axiom~\ref{axiomV3} and of the definitions.

\begin{Lem}\label{LemQC2}
Let $\alpha,\beta \in \Phi$ be such that $\beta \not\in - \mathbb{R}_{>0} \alpha$.
Let $U_{(\alpha,\beta),\Omega}$ (resp. $U'_{(\alpha,\beta),\Omega}$) be the subgroup of $G$ generated by the $U_{\gamma,\Omega}$ (resp. $U'_{\gamma,\Omega}$) for $\gamma \in (\alpha,\beta)$ (see Notation~\ref{NotRootSystem}).
Then the commutator subgroups satisfy:
\begin{equation}
\tag{QC2}
[U_{\alpha,\Omega},U_{\beta,\Omega} ] \subset U_{(\alpha,\beta),\Omega}
\label{eqQC2}
\end{equation}
\axiom{QC2@\ref{axiomQC2}}
\begin{equation}
\tag{QC2'}
[U_{\alpha,\Omega},U'_{\beta,\Omega} ] \subset U'_{(\alpha,\beta),\Omega}
\label{eqQC2'}
\end{equation}\axiom{QC2'@(\ref{eqQC2'})}
\end{Lem}

\begin{proof}
Let $u \in U_{\alpha,\Omega}$, $v \in U_{\beta,\Omega}$ and $v' \in U'_{\beta,\Omega}$.
Let $\lambda = \varphi_\alpha(u)$, $\mu=\varphi_{\beta}(v)$ and $\mu' = \varphi_\beta(v')$.
Let $r,s\in \mathbb{Z}_{>0}$ be such that $\gamma= r\alpha+s\beta \in \Phi$.
Then for any $x \in \Omega$, we have $\lambda \geqslant -\alpha(x)$, $\mu \geqslant - \beta(x)$ and $\mu' > -\beta(x)$.
Hence $r\lambda + s\mu \geqslant -r \alpha(x) - s \beta (x) = -\gamma(x)$ and $r\lambda + s\mu' > -r \alpha(x) - s \beta (x) = -\gamma(x)$.
Thus $U_{\gamma,r\lambda + s\mu} \subset U_{\gamma,x}$ and $U_{\gamma,r\lambda + s\mu'} \subset U'_{\gamma,x}$.
Because this inclusion holds for any $x\in\Omega$, we deduce that $U_{\gamma,r\lambda + s\mu} \subset U_{\gamma,\Omega}$ and $U_{\gamma,r\lambda + s\mu'} \subset U'_{\gamma,\Omega}$.
By axiom~\ref{axiomV3}, the commutator $[u,v]$ (resp. $[u,v']$) is contained in the group generated by the $U_{r\alpha+s\beta,r\lambda + s \mu} \subset U_{r\alpha+s\beta,\Omega}$ (resp. $U_{r\alpha+s\beta,r\lambda + s \mu'} \subset U'_{r\alpha+s\beta,\Omega}$).
\end{proof}

\begin{Ex}\label{ExQC}
According to Lemmas~\ref{LemQC1}(\ref{eqQC1}) and~\ref{LemQC2}(\ref{eqQC2}), the family $\left((U_{\alpha,\Omega})_{\alpha\in\Phi},Y\right)$ is quasi-concave for every subgroup $Y$ of $T_b$ since $T_b$ normalizes each $U_{\alpha,\Omega}$ for $\alpha \in \Phi$ and $U_{\alpha,\Omega} = U_{\alpha,\Omega} U_{2\alpha,\Omega}$ by Lemma~\ref{LemU2ainUa}.
In particular, for $Y=1$, we get from Proposition~\ref{PropDecompositionQC}:
\begin{enumerate}[label={(\arabic*)}]
\item\label{ExDecUOmega:1} $U_\alpha \cap U_\Omega = U_{\alpha,\Omega}$ and $U_{2\alpha} \cap U_{\Omega} = U_{2\alpha} \cap U_{\alpha,\Omega} = U_{2\alpha,\Omega}$ for any $\alpha \in \Phi_{\mathrm{nd}}$.
\item\label{ExDecUOmega:2} The product map $\displaystyle \prod_{\alpha \in \Phi^+_{\mathrm{nd}}} U_{\alpha,\Omega} \to U^+_\Omega = U_\Omega \cap U^+$ (resp.
 $\displaystyle \prod_{\alpha \in \Phi^+_{\mathrm{nd}}} U_{-\alpha,\Omega} \to  U^-_\Omega = U_\Omega \cap U^-$) is a bijection for any ordering on the product.
\item\label{ExDecUOmega:3} We have $U_\Omega = U^+_\Omega U^-_\Omega \widetilde{N}_\Omega = U^{-}_\Omega U^+_{\Omega} \widetilde{N}_\Omega$.
\item\label{ExDecUOmega:4} The group $\widetilde{N}_\Omega$ is generated by the $\widetilde{N}_{\alpha,\Omega}$ for $\alpha \in \Phi_{\mathrm{nd}}$.
\end{enumerate}

Since for every $\alpha \in \Phi$, the group $T'_{\alpha,\Omega} \subset T_b$ normalizes $U'_{\alpha,\Omega}$ and $U'_{-\alpha,\Omega}$, we deduce from Lemmas~\ref{LemCommutationUaU'a} and~\ref{LemQC1}(\ref{eqQC1'}) that $U'_{\alpha,\Omega} U'_{-\alpha,\Omega} (T \cap U'_{\Omega})$ is a group for any ordering of the product.
Since $U'_{\alpha,\Omega}$ is a subgroup of $U_{\alpha,\Omega}$, we deduce from Lemma~\ref{LemQC2}(\ref{eqQC2'}) that the family $\left((U'_{\alpha,\Omega})_{\alpha\in\Phi},Y\right)$ is quasi-concave for any subgroup $Y$ of $T_b$ since $T_b$ normalizes the $U'_{\alpha,\Omega}$ for $\alpha \in \Phi$.
\end{Ex}

We have, as in \cite[6.4.25(i)]{BruhatTits1}, the following:

\begin{Prop}\label{PropCommutationTbUa}
Let $\alpha,\beta \in \Phi$ be two roots such that $\beta \not\in \mathbb{R} \alpha$.
Let $\lambda,\mu \in \Rtot$ and $\varepsilon \geqslant 0$.
Then the following commutator subgroup satisfies
\[ \left[ T_{\beta,\mu}^\varepsilon, U_{\alpha,\lambda} \right] \subset U_{\alpha,\lambda + \varepsilon} U_{2\alpha,2\lambda+\varepsilon}. \]
Moreover, 
\[ \left[ T_{\beta,\mu}', U_{\alpha,\lambda} \right] \subset U'_{\alpha,\lambda}. \]
\end{Prop}

\begin{proof}
There is nothing to prove for $\varepsilon =0$,
since $T^0_{\beta,\mu} \subset T_b$ (Proposition~\ref{PropRankOneLevi}~\ref{PropRankOneLevi:tori}) normalizes $U_{\alpha,\lambda} = U_{\alpha,\lambda} U_{2\alpha,2\lambda}$ (Corollary~\ref{CorActionNUa}).
Assume $\varepsilon > 0$ and consider the group $X$ generated by the subsets $U_{\beta,\mu}$, $U_{-\beta,-\mu+\varepsilon}$ and $U_{-2\beta,-2\mu+\varepsilon}$.
Then $X$ contains $T^\varepsilon_{\beta,\mu}$ by definition.

For $s \in \mathbb{Z}$, denote $\varepsilon(s) = \left\{\begin{array}{rl} 0 & \text{ if } s > 0\\ \varepsilon & \text{ if } s \leqslant 0\end{array}\right.$.
Consider the group $Y$ generated by the $U_{r\alpha+s\beta,r\lambda+s\mu+\varepsilon(s)}$ for $(r,s) \in \mathbb{Z}_{>0} \times \mathbb{Z}$ such that $r\alpha+s\beta \in \Phi$.

By axiom~\ref{axiomV3}, we observe that the commutator $[x,y]$ for $x \in U_{\beta,\mu} \cup U_{-\beta,-\mu+\varepsilon} \cup U_{-2\beta,-2\mu+\varepsilon}$ and $y\in \bigcup_{(r,s) \in \mathbb{Z}_{>0} \times \mathbb{Z}} U_{r\alpha+s\beta,r\lambda+s\mu+\varepsilon(s)}$ belongs to $Y$.
Thus $X$ normalizes $Y$ and $[X,U_{\alpha,\lambda}] \subset Y$.
Hence $[T^\varepsilon_{\beta,\mu},U_{\alpha,\lambda}] \subset Y \cap U_\alpha$ since $T^\varepsilon_{\beta,\mu} \subset X \cap T$ normalizes $U_\alpha$.

Consider the root system $\Phi(\alpha,\beta) = \Phi \cap \left( \mathbb{Z} \alpha + \mathbb{Z} \beta \right)$ and the family of groups defined by $X_{r\alpha+s\beta} = U_{r\alpha+s\beta,r\lambda + s \mu + \varepsilon(s)}$ for $r > 0$ and $r\alpha+s\beta \in \Phi$ and by $X_{r\alpha+s\beta} = 1$ otherwise.
Then the family of groups $ \left(X_\gamma \right)_{\gamma \in \Phi(\alpha,\beta)}$
is quasi-concave (axiom~\ref{axiomQC1} is obvious for this set since either $X_{-2\gamma} = X_{-\gamma} = 1$ or $X_{2\gamma} = X_{\gamma}=1$ and axiom~\ref{axiomQC2} is a consequence of axiom~\ref{axiomV3}).
Thus by Proposition~\ref{PropDecompositionQC}, we get that $Y \cap U_\alpha = X_\alpha X_{2\alpha} = U_{\alpha,\lambda+\varepsilon} U_{2\alpha,2\lambda+\varepsilon}$.

Since $T'_{\beta,\mu} = \bigcup_{\varepsilon > 0} T^\varepsilon_{\beta,\mu}$, we get that any element in the commutator subgroup $[T'_{\beta,\mu},U_{\alpha,\lambda}]$ belongs to the commutator subgroup $[T^\varepsilon_{\beta,\mu},U_{\alpha,\lambda}]$ for some $\varepsilon > 0$.
Thus $[T'_{\beta,\mu},U_{\alpha,\lambda}] \subset \bigcup_{\varepsilon > 0} [T^\varepsilon_{\beta,\mu},U_{\alpha,\lambda}] = \bigcup_{\varepsilon > 0} U_{\alpha,\lambda+\varepsilon} U_{2\alpha,2\lambda+\varepsilon} = U'_{\alpha,\lambda}$.
\end{proof}

\subsection{Local root systems}

In this section, we work under data and notations of~\ref{HypCAVRGD}.
Moreover, we fix a non-empty subset $\Omega \subset \mathbb{A}_\Rtot$.
For simplicity, we will forget the chosen origin $o$ of $\mathbb{A}_\Rtot$ in this section, i.e. we will denote by $\alpha(x)$ the quantity $\alpha(x-o)$ by abuse of notation.
The goal of this section is to define a root system $\Phi^*_\Omega$ depending on the local geometry of $\Omega$ inside $\mathbb{A}_\Rtot$ with respect to some hyperplanes $H_{\alpha,\lambda}$ for $\alpha \in \Phi$ and $\lambda \in \Gamma'_\alpha$ that will be the walls of $\mathbb{A}_\Rtot$.

\begin{Not}
We define the following subset of roots:
\[\Phi^*_{\Omega} = \{\alpha \in \Phi,\ \exists \lambda_\alpha \in \Rtot,\ \forall x \in \Omega,\ -\alpha(x) = \lambda_\alpha\}\]
\[\Phi_\Omega = \{\alpha \in \Phi,\ \exists \lambda_\alpha \in \Gamma'_\alpha,\ \forall x \in \Omega,\ -\alpha(x) = \lambda_\alpha \}.\]
Note that, by definition, we have $\Phi_\Omega \subset \Phi^*_\Omega \subset \Phi$.
\index[notation]{p@$\Phi^*_{\Omega}$}\index[notation]{p@$\Phi_{\Omega}$}
We will observe (see Proposition~\ref{PropRootSystemPhiOmega}) that $\Phi_\Omega$ is a root system. We call it the \textbf{local root system} associated to $\Omega$.\index{root system!local}\index{local root system}
\end{Not}

\begin{Lem}\label{LemEquivalenceDefPhiOmega}
Let $\alpha \in \Phi$.
The following are equivalent:
\begin{enumerate}[label={(\roman*)}]
\item\label{LemPhiOmega-1} $\alpha \in \Phi_\Omega$;
\item\label{LemPhiOmega-2} $\exists u \in U_\alpha$ such that $\forall x \in \Omega,\ \varphi_\alpha(u) = -\alpha(x)$ and $\forall v \in U_{2\alpha},\ \varphi_\alpha(u) \geqslant \varphi_\alpha(uv)$;
\item\label{LemPhiOmega-3} $\alpha \in \Phi^*_\Omega$ and $\exists u \in U_{\alpha,\Omega}$ such that $\forall v \in U_{2\alpha},\ uv \not\in U'_{\alpha,\Omega}$.
\end{enumerate}
\end{Lem}

\begin{proof}~

\paragraph{\ref{LemPhiOmega-1} $ \Rightarrow $~\ref{LemPhiOmega-2}:}
If $\alpha \in \Phi_\Omega$, let $\lambda_\alpha \in \Gamma'_\alpha$ be such that $\forall x \in \Omega,\ -\alpha(x) = \lambda_\alpha$. By definition, there is $u \in U_\alpha$ such that $-\alpha(x) = \varphi_\alpha(u)$ and $U_{\alpha,\varphi_\alpha(u)} = \bigcap_{v \in U_{2\alpha}} U_{\alpha,\varphi_\alpha(uv)}$.
Thus $\forall v \in U_{2\alpha}$, we have $\varphi_\alpha^{-1}([\varphi_\alpha(u),+\infty] ) \subset \varphi_\alpha^{-1}([\varphi_\alpha(uv),+\infty])$ and therefore $\varphi_\alpha(u) \geqslant \varphi_\alpha(uv)$.

\paragraph{\ref{LemPhiOmega-2} $\Rightarrow$~\ref{LemPhiOmega-3}:}
Let $u \in U_\alpha$ be such that $\forall x \in \Omega,\ \varphi_\alpha(u) = -\alpha(x)$ and $\forall v \in U_{2\alpha},\ \varphi_\alpha(u) \geqslant \varphi_\alpha(uv)$.
For any $x,y \in \Omega$, we have $-\alpha(x) = \varphi_\alpha(u) = -\alpha(y)$ so that $ u \in \bigcap_{x \in \Omega} U_{\alpha,x} = U_{\alpha,\Omega}$.
Suppose by contradiction that there is a $v \in U_{2\alpha}$ such that $uv \in U'_{\alpha,\Omega}$.
Then for any $x \in \Omega$, we have $uv \in U'_{\alpha,x}$ and therefore $\varphi_\alpha(uv) > -\alpha(x) = \varphi_\alpha(u)$ which contradicts the assumption.

\paragraph{\ref{LemPhiOmega-3} $\Rightarrow$~\ref{LemPhiOmega-1}:}
Assume that there exists a $u \in U_{\alpha,\Omega}$ such that $\forall v \in U_{2\alpha},\ uv \not\in U'_{\alpha,\Omega}$.
Let $\lambda_\alpha = \varphi_\alpha(u)$.
For $x \in \Omega$, since $u \in U_{\alpha,\Omega} \subset U_{\alpha,x}$, we have $\lambda_\alpha = \varphi_\alpha(u) \geqslant -\alpha(x)$.
Let $v \in U_{2\alpha}$. Then $uv \not\in U'_{\alpha,\Omega}$ and there is $x_v \in \Omega$ such that $uv \not\in U'_{\alpha,x_v}$. Thus $\varphi_{\alpha}(uv) \leqslant -\alpha(x_v) \leqslant \varphi_\alpha(u)$.
Hence $U_{\alpha,\varphi_\alpha(u)} \subset U_{\alpha,\varphi_\alpha(uv)}$ for every $v \in U_{2\alpha}$ and we get $U_{\alpha,\varphi_\alpha(u)} = \bigcap_{v \in U_{2\alpha}} U_{\alpha,\varphi_\alpha(uv)}$ by taking $v=1$, so that $\lambda_\alpha = \varphi_\alpha(u) \in \Gamma'_\alpha$.
If, moreover, $\alpha \in \Phi^*_\Omega$, then $\forall y \in \Omega$, we have $-\alpha(y) = -\alpha(x) = \lambda_\alpha$.
\end{proof}

Without infimum and supremum, we provide a different proof and a slightly different statement of \cite[6.4.11]{BruhatTits1}:

\begin{Lem}\label{LemRootInNaOmega}~
For any $\alpha \in \Phi_{\mathrm{nd}}$, we have ${^v\!}\nu(\widetilde{N}_{\alpha,\Omega}) \subset \{1,r_\alpha\}$ with equality if, and only if, either $\alpha$ or $2\alpha$ belongs to $\Phi_{\Omega}$.
\end{Lem}

\begin{proof}
According to \cite[6.1.2(7) \& (10)]{BruhatTits1}  and \cite[6.1.15 (c)]{BruhatTits1}, since $\widetilde{N}_{\alpha,\Omega} \subset N \cap \langle U_\alpha, U_{-\alpha}, T \rangle = T \sqcup M_\alpha$, we have ${^v\!}\nu(\widetilde{N}_{\alpha,\Omega}) \subset \{1,r_\alpha\}$.

Suppose that $\alpha \in \Phi_\Omega$.
Let $u \in U_\alpha$ be given by Lemma~\ref{LemEquivalenceDefPhiOmega}\ref{LemPhiOmega-2} and $\lambda = \varphi_\alpha(u) \in \Gamma'_\alpha \subset \Gamma_\alpha$.
Then, by Lemma~\ref{LemMuNu}, we have ${^v\!}\nu(M_{\alpha,\lambda}) = \{r_\alpha\}$.
Moreover, by Proposition~\ref{PropValuedCoset}\ref{PropValuedCoset:insideLevi}, the set $M_{\alpha,\lambda}$ is contained in $L_{\alpha,\lambda}$.
Moreover, for any $x \in \Omega$, we have $-\alpha(x) = \lambda$ which gives $U_{\alpha,x} = U_{\alpha,\Omega}$ and $U_{-\alpha,x} = U_{-\alpha,\Omega}$ and therefore $M_{\alpha,\lambda} \subset \widetilde{N}_{\alpha,\Omega}$.
Thus $r_\alpha \in {^v\!}\nu(\widetilde{N}_{\alpha,\Omega})$.

Suppose that $2\alpha \in \Phi_\Omega$, then $r_\alpha \in {^v\!}\nu(\widetilde{N}_{\alpha,\Omega})$ since $\widetilde{N}_{2\alpha,\Omega} \subset \widetilde{N}_{\alpha,\Omega}$ by definition.

Conversely, suppose that $r_\alpha \in {^v\!}\nu(\widetilde{N}_{\alpha,\Omega})$.
Consider any $n \in {^v\!}\nu^{-1}(\{r_\alpha\}) \cap \widetilde{N}_{\alpha,\Omega}$.
Note that $n \not\in T$ since ${^v\!}\nu(T) = \operatorname{id}$ by \cite[6.1.11(ii)]{BruhatTits1}.
Since $U_{\alpha,\Omega} \cup U_{-\alpha,\Omega}$ is a generating set of the group $L_{\alpha,\Omega}$, we can write $n$ as a product $n = \prod_{i =1}^r u_i$ where $r \in \mathbb{Z}_{> 0}$ and $u_i \in U_{\alpha,\Omega} \cup U_{-\alpha,\Omega}$ with $u_i \neq 1$.

For any $i \in \llbracket 1,r \rrbracket$, denote
\begin{align*}
\lambda_i =& \left\{\begin{array}{rl}
 \varphi_\alpha(u_i) & \text{ if } u_i \in U_\alpha,\\
 \infty & \text{ if } u_i \in U_{-\alpha},
\end{array}\right.&
\mu_i =& \left\{\begin{array}{rl}
 \varphi_{-\alpha}(u_i) & \text{ if } u_i \in U_{-\alpha},\\
 \infty & \text{ if } u_i \in U_{\alpha},
\end{array}\right.
\end{align*}
and consider $\lambda = \min \{ \lambda_i,\ i \in \llbracket 1,r \rrbracket \} \in \Rtot \cup \{\infty\}$ and $\mu = \min \{ \mu_i,\ i \in \llbracket 1,r \rrbracket \} \in \Rtot \cup \{\infty\}$.
For any $i \in \llbracket 1,r\rrbracket$, we have either $\lambda_i = \infty$ or $u_i \in U_{\alpha,\Omega} = \bigcap_{x \in \Omega} U_{\alpha,x}$ which gives $\varphi_\alpha(u_i) = \lambda_i \geqslant -\alpha(x)$ for every $x \in \Omega$. Thus $U_{\alpha,\lambda_i} \subset U_{\alpha,\Omega}$.
By the same way, $U_{-\alpha,\mu_i} \subset U_{-\alpha,\Omega}$.
Thus $U_{\alpha,\lambda} \subset U_{\alpha,\Omega}$ and $U_{-\alpha,\mu} \subset U_{-\alpha,\Omega}$.
Hence, for $x \in \Omega$, we have $\lambda \geqslant -\alpha(x)$ and $\mu \geqslant \alpha(x)$.
Denote $\varepsilon = \lambda+\mu \geqslant 0$.
Then $n \in L^\varepsilon_{\alpha,\lambda}$ and thus $n \in N^\varepsilon_{\alpha,\Omega} \setminus T$ which gives $\varepsilon = 0$ and $\lambda \in \Gamma_\alpha$ by Proposition~\ref{PropRankOneLevi}.
Hence $\forall x \in \Omega, \lambda \geqslant -\alpha(x)$ and $\mu = -\lambda \geqslant \alpha(x)$.
Thus $\forall x \in \Omega,\ -\alpha(x) = \lambda \in \Gamma_\alpha$.
If $\lambda \in \Gamma'_\alpha$, then we have $\alpha \in \Phi_\Omega$.
Otherwise, by Fact~\ref{FactDecompositionSetOfValue},  $2\lambda \in \Gamma'_{2\alpha}$ and thus $\forall x \in \Omega,\ -(2\alpha)(x) = 2 \lambda \in \Gamma'_{2\alpha}$ which means $2\alpha \in \Phi_\Omega$.

\end{proof}

\begin{Prop}\label{PropNOmegaNormalizesUprime}
The group $\widetilde{N}_\Omega$ normalizes $U'_{\Omega}$.
\end{Prop}

\begin{proof}
Let $\alpha \in \Phi_{\mathrm{nd}}$ and $n \in \widetilde{N}_{\alpha,\Omega}$.
Let $\beta \in \Phi$ and $u \in U'_{\beta,\Omega}$.

If ${^v\!}\nu(n) = \operatorname{id}$, then $n \in T$ by \cite[6.1.11(ii)]{BruhatTits1} and therefore $n \in T_{\alpha,\Omega} \subset T_b$.
Thus $nun^{-1} \in U'_{\beta,\Omega}$.

Otherwise $\alpha$ or $2\alpha \in \Phi_\Omega$ by Lemma~\ref{LemRootInNaOmega}, so that for any $x \in \Omega$, we have $-\alpha(x) = \lambda \in \Gamma_\alpha$.
Denote by $\gamma = {^v\!}\nu(n)(\beta)$.
Then $n \in L_{\alpha,x}$ since $U_{\pm \alpha,\Omega} \subset U_{\pm \alpha,x}$ and therefore $n \in T_b m$ for any $m \in M_{\alpha,\lambda}$ by Proposition~\ref{PropRankOneLevi}.

Let $x \in \Omega$ and $\varepsilon = \varphi_\beta(u) + \beta(x) > 0$.
Then $m u m^{-1} \in U_{\gamma,-\beta(x)+\varepsilon + \beta(\alpha^\vee) \lambda} = U_{\gamma,-\gamma(x) +\varepsilon}$ with $\gamma = r_\alpha(\beta) = \beta - \beta(\alpha^\vee) \alpha$ according to Lemma~\ref{LemConjugationMaUb}.
Thus $n u n^{-1} \in U'_{\gamma,x}$ for every $x \in \Omega$ and therefore $n u n^{-1} \in U'_{\Omega}$.
Hence $n$ normalizes $U'_{\Omega}$ since it is generated by the $U'_{\beta,\Omega}$.

Thus $\widetilde{N}_\Omega$ normalizes $U'_{\Omega}$ since it is generated by those elements $n$ according to Example~\ref{ExQC}\ref{ExDecUOmega:4}.
\end{proof}

In the following Proposition, we detail the proof of \cite[6.4.10]{BruhatTits1} with some changes since we do not use infimum here.

\begin{Prop}\label{PropRootSystemPhiOmega}
The subset $\Phi_\Omega$ is a sub-root system of $\Phi$, i.e. a root system in the $\mathbb{R}$-subspace $V^*_\Omega$ generated by $\Phi_\Omega$.

Moreover, the group homomorphism $ {^v}\!\nu : N \to \mathrm{GL}(V^*)$ induces a group homomorphism $\widetilde{N}_\Omega \to \mathrm{GL}(V^*_\Omega)$ sending $n$ onto the restriction of ${^v}\!\nu(n)$ to $V^*_\Omega$ with image the Weyl group of $\Phi_\Omega$ and kernel $T_\Omega$.
\end{Prop}

\begin{proof}
Denote by $W_\Omega = {^v}\!\nu(\widetilde{N}_\Omega)$ which is a subgroup of the Weyl group of $\Phi$.

Firstly we prove that $\Phi_\Omega$ is stable by $W_\Omega$.
Let $n \in \widetilde{N}_\Omega$ and $w = {^v\!}\nu(n) \in W_\Omega \subset W(\Phi)$.
Consider any $\alpha \in \Phi_\Omega$ and denote $\beta = w(\alpha)$.
Consider an element $u \in U_\alpha$ such that $\forall x \in \Omega,\ \varphi_\alpha(u) = -\alpha(x)$ and $\forall v \in U_{2\alpha},\ \varphi_\alpha(u) \geqslant \varphi_\alpha(uv)$ given by~\ref{LemEquivalenceDefPhiOmega}\ref{LemPhiOmega-2}.
Let $\widetilde{u} = n u n^{-1}$.
By \cite[6.1.2 (10)]{BruhatTits1}, we have $\widetilde{u} \in U_\beta$.
Since $n \in \widetilde{N}_\Omega \subset U_{\Omega}$, we have $\widetilde{u} \in U_\beta \cap U_\Omega = U_{\beta,\Omega}$ by Example~\ref{ExQC}\ref{ExDecUOmega:1}.
Let $\widetilde{v} \in U_{2\beta}$ and denote $v = n^{-1} \widetilde{v} n \in U_{w^{-1}(2\beta)} = U_{2\alpha}$. Suppose that $\widetilde{u} \widetilde{v} \in U'_{\beta,\Omega}$.
Since $\widetilde{N}_\Omega$ normalizes $U'_\Omega$ by Proposition~\ref{PropNOmegaNormalizesUprime}, 
we have $uv = n^{-1} \widetilde{u} \widetilde{v} n \in U'_\Omega \cap U_\alpha = U'_{\alpha,\Omega}$ by Proposition~\ref{PropDecompositionQC}\ref{PropDecQC:1} applied to $U'_{\Omega}$ (this is possible according to Example~\ref{ExQC}).
This contradicts the assumption on $u$ since for any $x \in \Omega$, we have $\varphi_\alpha(uv) > -\alpha(x) = \varphi_\alpha(u)$.
Thus $\widetilde{u} \widetilde{v} \not\in U'_{\beta,\Omega}$.

Finally, for $x\in \Omega$, we have $U'_{\beta,x} = U'_x \cap U_\beta = n (U'_x \cap U_\alpha) n^{-1} = n U'_{\alpha,x} n^{-1}$ since $n \in \widetilde{N}_\Omega \subset N_x$ normalizes $U'_x$ by Proposition~\ref{PropNOmegaNormalizesUprime}.
Thus for $x,y \in \Omega$, we have $U'_{\beta,x} = n U'_{\alpha,x} n^{-1} =  n U'_{\alpha,y} n^{-1} = U'_{\beta,y}$.
Thus, for $\widetilde{v} = 1$, we have $\widetilde{u} \not\in U'_{\beta,\Omega} = U'_{\beta,x}$ for any $x \in \Omega$.
Hence, for any $x \in \Omega$, we get $\varphi_\beta(\widetilde{u}) \leqslant -\beta(x)$.
But we also have $\varphi_\beta(\widetilde{u}) \geqslant -\beta(x)$ since $\widetilde{u} \in U_{\beta,x}$.
Hence $\beta\in \Phi_\Omega^*$.
Thus, by  the characterisation~\ref{LemPhiOmega-3} in Lemma~\ref{LemEquivalenceDefPhiOmega}, we have $\beta \in \Phi_\Omega$.

Secondly, since $\widetilde{N}_\Omega$ is generated by the $\widetilde{N}_{\alpha,\Omega}$ according to Example~\ref{ExQC}\ref{ExDecUOmega:4}, we get from Lemma~\ref{LemRootInNaOmega} that $W_\Omega$ is the group generated by the reflections $r_\alpha$ for $\alpha \in \Phi_\Omega$.
Hence $\Phi_\Omega$ is a root system inside $V^*_\Omega$ by definition and the restriction of elements in $W_\Omega$ to $V^*_\Omega$ is the Weyl group of $\Phi_\Omega$.

Finally, write $V^* = V_\Omega^* \oplus (\Phi_\Omega^\vee)^\perp$.
Since any $r_\alpha$ pointwise stabilizes $(\Phi_\Omega^\vee)^\perp$ for $\alpha \in \Phi_\Omega$ and these $r_\alpha$ generate $W_\Omega$, we know that $W_\Omega$ pointwise stabilizes $(\Phi_\Omega^\vee)^\perp$.
Hence an element $w \in W_\Omega$ pointwise stabilizes $V^*$ if, and only if, it pointwise stabilizes $V_\Omega^*$.
Thus,
the kernel of $\widetilde{N}_\Omega \to \mathrm{GL}(V^*_\Omega)$ is $\ker {^v\!}\nu \cap \widetilde{N}_\Omega = T \cap (N \cap U_\Omega) = T_\Omega$ by definition.
\end{proof}

We say that a point $x \in \mathbb{A}$ is a \textbf{special vertex}\index{special vertex}\index{vertex!special} if $W(\Phi_x) = W(\Phi)$.

\begin{Ex}\label{ExCalculationInSL2}
Consider $G = \mathrm{SL}_2(\mathbb{K})$, let $T$ be the subgroup of diagonal matrices and $U_\alpha$ (resp. $U_{-\alpha}$) the subgroup of upper (resp. lower) unitriangular matrices. We get a generating root group datum of $G$ if we take $m = \begin{pmatrix} 0 & 1\\-1 & 0\end{pmatrix}$ and $M_\alpha = M_{-\alpha} = m T$.
There are parametrizations $\widehat{\alpha} : \mathbb{K}^* \to T$, $u_\alpha : \mathbb{K} \to U_\alpha$ and $u_{-\alpha} : \mathbb{K} \to U_{-\alpha}$ of these groups given by $\widehat{\alpha}(z) = \begin{pmatrix} \frac{1}{z} & 0\\0& z\end{pmatrix}$, $u_\alpha(x) = \begin{pmatrix} 1 & x\\0&1\end{pmatrix}$ and $u_{-\alpha}(y) = \begin{pmatrix} 1 & 0 \\-y & 1\end{pmatrix}$.
We define a valuation of the root group datum by setting $\varphi_{\pm \alpha}(u_{\pm \alpha}(x)) = \omega(x)$.
For $x,y \in \mathbb{K}$ with $\omega(x)\geqslant 0$ and $\omega(y) > 0$, 
define $u = u_\alpha(x) \in U_{\alpha,0}$ and $v = u_{-\alpha}(y) \in U'_{-\alpha,0}$.
Then
\[
[u,v] =
  \begin{pmatrix} 1 & x\\0&1\end{pmatrix}
  \begin{pmatrix} 1 & 0\\-y&1\end{pmatrix}
  \begin{pmatrix} 1 & -x\\0&1\end{pmatrix}
  \begin{pmatrix} 1 & 0\\y&1\end{pmatrix}
  =\begin{pmatrix} 1-xy+x^2y^2 & x^2y\\xy^2 & 1+xy\end{pmatrix}
\]
We set $z = 1+xy \in \mathbb{K}^*$ since $\omega(xy) > 0$.
Denote by $t = \widehat{\alpha}(z) \in T$, by $u' = u_\alpha\left( \frac{x^2y}{z} \right)$ and by $v' = u_{-\alpha}\left( - \frac{x y^2}{z} \right)$.
One can easily check that 
$[u,v] = u' t v'$
so that $t = \widehat{\alpha}(1+xy) = t(u,v) \in T'_{\alpha,0}$ for every $x,y$.

Assume, for instance, that $\omega : \mathbb{K} \to \mathbb{Z}\cup \{\infty\}$ is a discrete valuation of rank $1$.
Thus $T'_{\alpha,0} = \{ \widehat{\alpha}(z),\ \omega(z-1) \geqslant 1\}$.
By the same way, we get $T'_{-\alpha,0} = T'_{\alpha,0}$ and $T'_0 = T \cap \langle U'_{\alpha,0} ,U'_{-\alpha,0} \rangle = T \cap \langle U_{\alpha,1} ,U_{-\alpha,1} \rangle = \{ \widehat{\alpha}(z),\ \omega(z-1) \geqslant 2\}$.
In this case, we have $T'_0 \varsubsetneq \langle T'_{\alpha,0}, T'_{-\alpha,0} \rangle = T'_{\alpha,0}$.
\end{Ex}

This example and the second inclusion of Lemma~\ref{LemCommutationUaU'a} suggest to introduce the following subgroup of $T$ (which is denoted by $H_{f,f^*}$ in \cite[6.4]{BruhatTits1}).

\begin{Not}\label{T*not}
Let $T^*_\Omega$\index[notation]{t@$T^*_\Omega$} be the subgroup of $T_b$ (see Proposition~\ref{PropRankOneLevi}~\ref{PropRankOneLevi:tori} and Notation~\ref{NotLevialphaOmega}) generated by the $T'_{\alpha,\Omega}$ for $\alpha \in \Phi$. 
\end{Not}

\begin{Lem}\label{LemInclusionTOmega}
We have $T^*_\Omega \subset T_\Omega \subset T_b$ and all these groups are normal in $T_b$.
In particular, all those inclusions are inclusions of normal subgroups.
\end{Lem}

\begin{Rq}
The second assertion of this lemma becomes obvious when $T = \mathcal{Z}_\mathbf{G}(\mathbf{S})(\mathbb{K})$ is the group of rational points of the centralizer of a maximal split torus $\mathbf{S}$ of a quasi-split reductive group $\mathbf{G}$.
Indeed, in this case $\mathcal{Z}_\mathbf{G}(\mathbf{S})$ is a maximal torus of $\mathbf{G}$ by definition and thus $T$ is commutative.
But note that $\mathcal{Z}_\mathbf{G}(\mathbf{S})(\mathbb{K})$ may  not be commutative if $\mathbf{G}$ is not quasi-split.
\end{Rq}

\begin{proof}
Since $T_b$ normalizes the $U_{ \alpha,\lambda}$ and the $U'_{ \alpha,\lambda}$ for any $\lambda \in \Rtot$ and any $\alpha\in \Phi$, it normalizes $U_{\alpha,\Omega}$ and $U'_{\alpha,\Omega}$ as intersection of such groups. 
Therefore $T_b$ normalizes $U_\Omega$, $U'_\Omega$ and $L'_{\alpha,\Omega}$ for any $\alpha \in \Phi$.
Since $T_b$ is a subgroup of $T$, it normalizes the intersections $T_\Omega = T \cap U_{\Omega}$, $T'_\Omega = T \cap U'_{\Omega}$ and $T'_{\alpha,\Omega} = L'_{\alpha,\Omega} \cap T$ for any $\alpha \in \Phi$.
In particular, $T_b$ normalizes $T^*_\Omega$ being generated by the $T'_{\alpha,\Omega}$ for $\alpha\in \Phi$. Thus it remains to prove the inclusions.

For $\alpha \in \Phi$, we have $T'_{\alpha,\Omega} = L'_{\alpha,\Omega} \cap T \subset U_{\Omega} \cap T$. Hence a generating set of $T^*_\Omega$ is contained in $T_\Omega$ and therefore $T^*_\Omega \subset T_\Omega$.

If $x \in \Omega$, then $U_{\alpha,\Omega} \subset U_{\alpha,x}$ for any $\alpha \in \Phi$.
Hence $U_{\Omega} \subset U_{x}$ and thus $T_{\Omega} = T \cap U_\Omega \subset T_{\alpha,x} = T_{\alpha,-\alpha(x)} \subset T_b$ by Proposition~\ref{PropRankOneLevi}\ref{PropRankOneLevi:tori}.
\end{proof}

We have the analogous to \cite[6.4.27]{BruhatTits1}:

\begin{Lem}\label{LemCommutationTstarUaOmega}
For any $\alpha \in \Phi$, the following commutator subgroup satisfies
\[ [T^*_\Omega,U_{\alpha,\Omega}] \subset U'_{\alpha,\Omega}. \]
\end{Lem}

\begin{proof}
Since $T^*_\Omega \subset T_b$ normalizes $U'_{\alpha,\Omega}$ and, by Lemma~\ref{LemCommutationUaU'a}, $U_{\alpha,\Omega}$ normalizes $U'_{\alpha,\Omega}$, it suffices to prove that for a generating set $X$ of $T^*_\Omega$, for every $u \in U_{\alpha,\Omega}$ and every $x \in X$, we have $[x,u] \subset U'_{\alpha,\Omega}$. 
We consider $X = \bigcup_{\beta\in \Phi} T'_{\beta,\Omega}$.
Thus, let $\beta \in \Phi$, $u \in U_{\alpha,\Omega}$ and $t \in T'_{\beta,\Omega}$.
Consider any $x \in \Omega$.

Firstly we  assume that $\alpha$ is non-divisible.
Then $u  \in U_{\alpha,x} = U_{\alpha,-\alpha(x)}$
and $t \in T'_{\beta,x} = T'_{\beta,-\beta(x)}$.
We have the following three possibilities:

\paragraph{Case 1: $\beta \not\in \mathbb{R}\alpha$}
Then, by Proposition~\ref{PropCommutationTbUa}, we get that $[t,u] \in U'_{\alpha,-\alpha(x)} = U'_{\alpha,x}$.

\paragraph{Case 2: $\beta \in \mathbb{R}_{>0} \alpha$}
Then $\beta \in \{\alpha,2\alpha\}$ since $\alpha$ is non-divisible and $t \in T'_{\alpha,x}$ since $T'_{2\alpha,x} \subset T'_{\alpha,x}$.
Thus, by Corollary~\ref{CorCommutationTaUa}(\ref{eqnCommutator7}), we get that $[t,u] \in U'_{\alpha,-\alpha(x)} = U'_{\alpha,x}$.

\paragraph{Case 3: $\beta \in \mathbb{R}_{<0} \alpha$}
Then $\beta \in \{-\alpha,-2\alpha\}$ since $\alpha$ is non-divisible and $t \in T'_{-\alpha,x}$ since $T'_{-2\alpha,x} \subset T'_{-\alpha,x}$.
Thus, by Corollary~\ref{CorCommutationTaUa}(\ref{eqnCommutator8}) applied to $-\alpha$, we get that $[t,u] \in U'_{\alpha,-\alpha(x)} = U'_{\alpha,x} $.

\medskip

Finally, if $\alpha$ is divisible, we have that $u \in U_{\frac{\alpha}{2},x}$ and we have already shown that $[t,u] \in U'_{\frac{\alpha}{2},x}$.
But since $T$ normalizes $U_{\alpha}$, we have $[t,u] \in U_\alpha \cap U'_{\frac{\alpha}{2},x} = U'_{\alpha,x}$.

Since the inclusion $[t,u] \in U'_{\alpha,x}$ holds for every $x \in \Omega$, we get $[t,u] \in \bigcap_{x \in \Omega} U'_{\alpha,x} = U'_{\alpha,\Omega}$.
\end{proof}

\begin{Not}
For $\alpha \in \Phi$, we denote by
\[U^*_{\alpha,\Omega} = \left\{ \begin{array}{rl}
 U_{\alpha,\Omega} & \text{ if } \alpha \not\in \Phi^*_\Omega\\
 U'_{\alpha,\Omega} & \text{ if } \alpha \in \Phi^*_\Omega
\end{array}\right. .\]
\index[notation]{u@$U^*_{\alpha,\Omega}$}
We denote by $U^*_\Omega$\index[notation]{u@$U^*_\Omega$} the subgroup generated by $T^*_\Omega$ and the $U^*_{\alpha,\Omega}$ for $\alpha \in \Phi$.
\end{Not}

\begin{Rq}\label{RqPhiStarNonCollinear}
If $\alpha \in \Phi^*_\Omega$, then $\mathbb{R} \alpha \cap \Phi \subset \Phi^*_\Omega$.

Indeed, let $\beta \in \mathbb{R}\alpha$ and write it $\beta = r \alpha$ with $2r \in \{\pm 1, \pm 2, \pm 4\}$.
By definition, there is $\lambda_\alpha \in \Rtot$ such that $\forall x \in \Omega,\ -\alpha(x) = \lambda_\alpha$.
Thus, with $2 \lambda_\beta = 2 r \lambda_\alpha$ (that belongs to $\Rtot$ since it is a $\mathbb{Z}$-module), we have $\forall x \in \Omega,\ -2 \beta(x) = -2 r \alpha(x) = 2 \lambda_\beta$ so that $\beta \in \Phi^*_\Omega$ since $\Rtot$ is $\mathbb{Z}$-torsion free and $-\beta(x) \in \Rtot$ for any $x \in \mathbb{A}_\Rtot$. 
\end{Rq}

The difference between $\Phi_\Omega$ and $\Phi^*_\Omega$ is that $\Phi_\Omega$ takes the group structure into account whereas the set $\Phi^*_\Omega$ only considers the structure of $\Omega$.
For instance, any point $x \in \mathbb{A}_\Rtot$ satisfies $\Phi^*_x = \Phi$ whereas we will have $W(\Phi_x) = W(\Phi)$ if, and only if, $x$ is a special vertex.

We recall that, in Bruhat-Tits theory, there exist situations in which we never have $\Phi_x = \Phi$ (for instance a dense valuation with $\Gamma'_\alpha = \emptyset$ for a multipliable root $\alpha \in \Phi$, or a discrete valuation with a totally ramified extension $\widetilde{\mathbb{K}} / \mathbb{K}$ splitting $\mathbf{G}$).

The following Proposition corresponds to the beginning of the statement \cite[6.4.23]{BruhatTits1} with a different proof since we do not define quasi-concave maps.

\begin{Prop}\label{PropUstarQC}
The family of groups $\left((U^*_{\alpha,\Omega})_{\alpha\in\Phi},T^*_\Omega\right)$ is quasi-concave and the group $U^*_\Omega$ is a normal subgroup of $U_\Omega$.
\end{Prop}

\begin{proof}
Since $U^*_{\alpha,\Omega} \subset U_\Omega$ for any $\alpha \in \Phi$ and, by Lemma~\ref{LemInclusionTOmega}, $T^*_\Omega \subset T_\Omega$, we know that $U^*_\Omega$ is a subgroup of $U_\Omega$.
Moreover, $T^*_\Omega$ normalizes the $U_{\alpha,\Omega}$ for $\alpha \in \Phi$ since $T^*_\Omega$ is a subgroup of $T_b$ by Lemma~\ref{LemInclusionTOmega}, which gives~\ref{axiomQC3}.
Let $\alpha \in \Phi$ be a non-divisible root.
By Lemma~\ref{LemCommutationTstarUaOmega}, we have $[U_{\alpha,\Omega},T^*_\Omega] \subset U'_{\alpha,\Omega} \subset U^*_{\alpha,\Omega}$.
In particular, $T^*_\Omega$ normalizes $U^*_{\alpha,\Omega}$.
By Example~\ref{ExQC}, since the families $(U_{\alpha,\Omega})_{\alpha \in \Phi}$ and $(U'_{\alpha,\Omega})_{\alpha \in \Phi}$ are quasi-concave and since $\Phi_\Omega^*=-\Phi_\Omega^*$, we deduce that the family $\left((U^*_{\alpha,\Omega})_{\alpha \in \Phi},T^*_\Omega\right)$ satisfies axiom~\ref{axiomQC1}.

Now, we distinguish cases in order to prove axiom~\ref{axiomQC2} on the one hand, and that $[U_{\alpha,\Omega},U^*_{\beta,\Omega}] \subset U^*_\Omega$ for any $\alpha,\beta \in \Phi$ on the other hand.
Consider any $\alpha,\beta \in \Phi$.

\paragraph{Case $\beta\in \Phi^*_\Omega$:}
If $\beta \in -\mathbb{R}_{>0} \alpha$, by Lemma~\ref{LemCommutationUaU'a}, we get that $[U_{\alpha,\Omega},U^*_{\beta,\Omega}] = [U_{\alpha,\Omega},U'_{\beta,\Omega}] \subset U'_{\alpha,\Omega} U'_{-\alpha,\Omega} T'_{\alpha,\Omega} \subset U^*_\Omega$ (since $[U_{\alpha,\Omega},U'_{- 2\alpha,\Omega}] \subset [U_{\alpha,\Omega},U'_{- \alpha,\Omega}]$ and $\alpha$ is non-divisible).
Otherwise, by Lemma~\ref{LemQC2}(\ref{eqQC2'}), we get that $[U_{\alpha,\Omega},U^*_{\beta,\Omega}] = [U_{\alpha,\Omega},U'_{\beta,\Omega}] \subset \prod_{\gamma \in (\alpha,\beta)} U'_{\gamma,\Omega} \subset \prod_{\gamma \in (\alpha,\beta)} U^*_{\gamma,\Omega}$.
In particular, we get axiom~\ref{axiomQC2} for any $\alpha \in \Phi$ and any $\beta \in \Phi^*_\Omega$ satisfying $\beta \not\in - \mathbb{R}_{>0} \alpha$.

\paragraph{Case $\beta \in \Phi \setminus \Phi^*_{\Omega}$:}

Suppose that $\beta \not\in -\mathbb{R}_{>0}\alpha$.
Let $u_\alpha \in U^*_{\alpha,\Omega} \subset U_{\alpha,\Omega}$ and $u_\beta \in U^*_{\beta,\Omega} = U_{\beta,\Omega}$.
By Proposition~\ref{PropDecompositionQC}~\ref{PropDecQC:2}, since $(U_{\alpha,\Omega})_{\alpha\in\Phi}$ is quasi-concave (see Example~\ref{ExQC}), there is a unique family of elements $u_\gamma \in U_{\gamma,\Omega}$ for $\gamma \in (\alpha,\beta)$ such that $[u_\alpha,u_\beta] = \prod_{\gamma \in (\alpha,\beta)} u_\gamma$ (up to the choice of an ordering on the factors).
Let $\gamma \in (\alpha,\beta)$ and write it $\gamma = r \alpha + s \beta$ with $r,s \in \mathbb{Z}_{>0}$.
If $\gamma \not\in \Phi^*_\Omega$, then $u_\gamma \in U_{\gamma,\Omega} = U^*_{\gamma,\Omega}$.
Otherwise, we have $\gamma \in \Phi^*_\Omega$.
Let $\lambda_\gamma$ be such that $\lambda_\gamma = -\gamma(x)$ for every $x \in \Omega$.
Since $\beta \not\in \Phi^*_\Omega$, there are points $y,z \in \Omega$ such that $-\beta(y) < -\beta(z)$.
Thus $\varphi_\beta(u_\beta) \geqslant -\beta(z) > -\beta(y)$ and $\varphi_\alpha(u_\alpha) \geqslant -\alpha(y)$.
Since $\varphi_\gamma(u_\gamma) \geqslant r \varphi_\alpha(u_\alpha) + s \varphi_\beta(u_\beta)$ by axiom~\ref{axiomV3} 
we get that $\varphi_\gamma(u_\gamma) > -r\alpha(y) - s \beta(y) = - \gamma(y) = \lambda_\gamma$.
Thus $u_\gamma \in U'_{\gamma,\lambda_\gamma} = U'_{\gamma,x}$ for every $x \in \Omega$ since $\lambda_\gamma = - \gamma(x)$.
Hence $u_\gamma \in U'_{\gamma,\Omega} = U^*_{\gamma,\Omega}$ for every $\gamma \in (\alpha,\beta)$.
This ends the proof of axiom~\ref{axiomQC2}.

Now, consider any $\beta \not\in \Phi^*_\Omega$.
If $\alpha \not\in \Phi^*_\Omega$,
then $[U_{\alpha,\Omega},U^*_{\beta,\Omega}] = [U^*_{\alpha,\Omega},U^*_{\beta,\Omega}] \subset U^*_\Omega$ by definition.

Otherwise, $\alpha \in \Phi^*_\Omega$.
By Remark~\ref{RqPhiStarNonCollinear}, we know that $\beta \not\in \mathbb{R}\alpha$.
Since $U^*_{\beta,\Omega} = U_{\beta,\Omega}$, we have the inclusion $[U_{\alpha,\Omega},U^*_{\beta,\Omega}] \subset U_{(\alpha,\beta),\Omega}$ by Lemma~\ref{LemQC2}.
Let $\gamma=r\alpha+s\beta \in (\alpha,\beta)$ with $r,s \in \mathbb{Z}_{>0}$.
By definition, the linear form $\alpha$ is constant on $\Omega$ but $\beta$ is not. Therefore $\gamma$ is non-constant so that $\gamma \not\in \Phi^*_\Omega$.
Thus for every $\gamma \in (\alpha,\beta)$, we have $U_{\gamma,\Omega} = U^*_{\gamma,\Omega}$.

Thus, we have shown that for any $\alpha \in \Phi$, we have $[U_{\alpha,\Omega},T^*_\Omega] \subset U^*_\Omega$ and $[U_{\alpha,\Omega},U^*_{\beta,\Omega}] \subset U^*_\Omega$ for every $\beta \in \Phi$.
Hence $U_{\alpha,\Omega}$ normalizes $U^*_\Omega$ since $U^*_\Omega$ is generated by $T^*_\Omega$ and the $U^*_{\beta,\Omega}$ for $\beta \in \Phi$.

Also does $U_{2\alpha,\Omega}$ since it is a subgroup of $U_{\alpha,\Omega}$.
Since $U_{\alpha,\Omega}$ normalizes $U^*_\Omega$ for every $\alpha \in \Phi$, the group $U_\Omega$ generated by the $U_{\alpha,\Omega}$ also normalizes $U^*_\Omega$. 
\end{proof}

\begin{Not}
We denote by $\overline{G_\Omega}$ the quotient group $U_\Omega / U^*_\Omega$.

For $\alpha \in \Phi$, we denote by $\overline{U_{\alpha,\Omega}}$ the canonical image of $U_{\alpha,\Omega}$ in $\overline{G_\Omega}$.

We denote by $\overline{T_\Omega}$ the canonical image of $T_\Omega$ in $\overline{G_\Omega}$.

For $\alpha \in \Phi^*_\Omega$, we denote by $\overline{M_{\alpha,\Omega}}$ the canonical image of $M_{\alpha,\lambda_\alpha}$ with $\lambda_\alpha = -\alpha(x)$ for any $x \in \Omega$ (it does not depend on the choice of $x \in \Omega$ by definition of $\Phi^*_\Omega$).
\end{Not}

The following Lemma corresponds to part of \cite[6.4.23]{BruhatTits1} with a completely different proof.

\begin{Lem}\label{LemReducedRootSystem}
For any $\alpha \in \Phi$, we have $\overline{U_{\alpha,\Omega}} \subset \overline{U_{2\alpha,\Omega}} \Longleftrightarrow \alpha \not\in \Phi_\Omega$.
\end{Lem}

\begin{proof}
Assume that $\overline{U_{\alpha,\Omega}} \not\subset \overline{U_{2\alpha,\Omega}}$.
Suppose, by contradiction that $U_{\alpha,\Omega} \subset U_{2\alpha} U'_{\alpha,\Omega}$.
Let $u \in U_{\alpha,\Omega}$ and write it $u = v w$ with $v \in U'_{\alpha,\Omega}$, $w \in U_{2\alpha}$.
Then $w = v^{-1} u \in U_{\alpha,\Omega} \cap U_{2\alpha} = U_{2\alpha,\Omega}$. 
Since $v \in U'_{\alpha,\Omega} \subset U^*_\Omega$, it contradicts the assumption.
Hence there exists $u \in U_{\alpha,\Omega} \setminus U_{2\alpha} U'_{\alpha,\Omega}$.
Let $\lambda_\alpha = \varphi_\alpha(u) \in \Gamma_\alpha$.
We prove that $\lambda_\alpha \in \Gamma'_\alpha$.
Indeed, let $v \in U_{2\alpha}$.
If, by contradiction, we have $\varphi_{\alpha}(uv) > \varphi_\alpha(u)$, then $uv \in U'_{\alpha,\Omega}$ and therefore $u \in U'_{\alpha,\Omega} U_{2\alpha}$ which contradicts the definition of $u$.
Thus $\varphi_{\alpha}(uv) \leqslant \varphi_\alpha(u)$ for every $v \in U_{2\alpha}$ and therefore $U_{\alpha,\varphi_\alpha(u)} = \bigcap_{v \in U_{2\alpha}} U_{\alpha,\varphi_\alpha(uv)}$ which means $\varphi_\alpha(u) = \lambda_\alpha \in \Gamma'_\alpha$.
Now, for any $x \in \Omega$, we have $u \in U_{\alpha,x}$ and therefore $-\alpha(x) \leqslant \varphi_\alpha(u)$.
But if $-\alpha(x) < \varphi_\alpha(u)$ for every $x \in \Omega$, then $u \in \bigcap_{x\in \Omega} U'_{\alpha,x} = U'_{\alpha,\Omega}$ which contradicts the definition of $u$.
Thus, there exists an element $y \in \Omega$ such that $-\alpha(y) = \varphi_\alpha(u) = \lambda_\alpha \in \Gamma'_\alpha$.
Finally, $\alpha \in \Phi^*_\Omega$ since $\overline{U_{\alpha,\Omega}}$ is not trivial.
Thus for every $x \in \Omega$, we have $-\alpha(x) = -\alpha(y) = \lambda_\alpha \in \Gamma'_\alpha$ which proves $\alpha \in \Phi_\Omega$.

Conversely, let $\alpha \in \Phi_\Omega \subset \Phi^*_\Omega$.
Then $U^*_{\alpha,\Omega} = U'_{\alpha,\Omega}$ by definition.
Let $u \in U_{\alpha}$ be such that $\forall x \in \Omega,\ \varphi_\alpha(u) = -\alpha(x)$ and $\forall v \in U_{2\alpha},\ \varphi_\alpha(uv) \leqslant \varphi_\alpha(u)$ (the existence of such a $u$ is provided by Lemma~\ref{LemEquivalenceDefPhiOmega}\ref{LemPhiOmega-2}).
By contradiction, suppose that $u \in U^*_{\alpha,\Omega} U_{2\alpha,\Omega}$.
Then we would have some $v \in U_{2\alpha,\Omega}$ such that $uv \in U^*_{\alpha,\Omega} = U'_{\alpha,\Omega} = U'_{\alpha,x}$ for any $x \in \Omega$.
Thus $\varphi_\alpha(uv) > -\alpha(x) = \varphi_\alpha(u)$ which contradicts the assumption on $u$.
Hence $U_{\alpha,\Omega} \not\subset U^*_{\alpha,\Omega} U_{2\alpha,\Omega}$.
This gives $\overline{U_{\alpha,\Omega}} \not\subset \overline{U_{2\alpha,\Omega}}$ because if we had $\overline{U_{\alpha,\Omega}} \subset \overline{U_{2\alpha,\Omega}}$, then we would have $U_{\alpha,\Omega} \subset U_{2\alpha,\Omega} U^*_{\Omega} \cap U_\alpha$.
But $U^*_\Omega \cap U_\alpha = U^*_{\alpha,\Omega}$ according to Proposition~\ref{PropDecompositionQC}\ref{PropDecQC:1} since the family $\left((U^*_{\alpha,\Omega})_{\alpha \in \Phi},T^*_\Omega\right)$ is quasi-concave. 
\end{proof}

The following theorem summarizes the work of this section and of the previous one. It  corresponds to the last statement of \cite[6.4.23]{BruhatTits1}, applied to $X = T_\Omega$ and $X^* = T^*_{\Omega}$.

\begin{Thm}\label{ThmLocalRootSystem}
The system $\left( \overline{T_\Omega}, \left( \overline{U_{\alpha,\Omega}}, \overline{M_{\alpha,\Omega}} \right)_{\alpha \in \Phi_{\Omega}} \right)$ is a generating root group datum of type $\Phi_\Omega$ in $\overline{G_\Omega}$.
\end{Thm}

\begin{proof}
Axiom~\ref{axiomRGD1} is satisfied because if $\overline{U_{\alpha,\Omega}} = 1$, then $U_{\alpha,\Omega} \subset U^*_\Omega \cap U_\alpha = U'_{\alpha,\Omega}$ would contradict the assumption $\lambda_\alpha \in \Gamma'_\alpha$ in the definition of $\alpha \in \Phi_\Omega$.

Axiom~\ref{axiomRGD2} is satisfied because the family $(U_{\alpha,\Omega})_{\alpha \in \Phi}$ is quasi-concave by Example~\ref{ExQC}.

Axiom~\ref{axiomRGD3} is a consequence of Lemma~\ref{LemReducedRootSystem}.

Axiom~\ref{axiomRGD4} is a consequence of Proposition~\ref{PropValuedCoset}.

Axiom~\ref{axiomRGD5} is a consequence of Proposition~\ref{PropRootSystemPhiOmega}.

For Axiom~\ref{axiomRGD6}, denote by $\pi: U_\Omega \to \overline{G_\Omega}$ the projection homomorphism.
Denote by $\overline{U^\pm_\Omega} = \langle \overline{U_{\alpha,\Omega}}, \alpha \in \Phi_\Omega^\pm \rangle = \pi(U_\Omega^\pm)$.
Let $\overline{g} \in \overline{T_\Omega} \overline{U^+_\Omega} \cap \overline{U^-_\Omega}$ and let $tu_+ \in T_\Omega U_\Omega^+ \subset U_\Omega$ be a lifting of $\overline{g}$ with $t \in T_\Omega$ and $u_+ \in U^+_\Omega$.
There exist $u_- \in U^-_\Omega$ and $v \in U^*_\Omega$ such that $tu_+ =u_-v$.
If we write (using Proposition~\ref{PropUstarQC}) $v = v_-v_+ n$ for $v_\pm \in U^*_\Omega \cap U^\pm$ and $n \in N \cap U^*_\Omega$, we have
\[v_-^{-1} u_-^{-1} t v_+ = n (n^{-1} v_+ n)\]
By Proposition~\ref{PropDecompositionQC} applied to the quasi-concave family of groups $(U_{\alpha, \Omega}^*)$, the group $N_\Omega^*:=U_\Omega^* \cap N$ is generated by the subgroups $N_{\alpha,\Omega}^* = N \cap \langle U_{\alpha, \Omega}^* \cup U_{-\alpha,\Omega}^* \rangle$.

\textbf{Claim:} $N_{\alpha,\Omega}^*$ is a subgroup of $T_b$.

Indeed, if $\alpha \in \Phi_\Omega^*$, then so is $-\alpha$ and the result follows from Lemma~\ref{LemQC1} since $U_{-\alpha, \Omega}^* = U_{-\alpha, \Omega}'$.
If $\alpha \not\in \Phi_{\Omega}^*$, consider $x, y \in \Omega$ such that $-\alpha(y) < -\alpha(x)$ and denote by $\lambda = -\alpha(x)$ and $\mu = -\alpha(y)$.
Then $U_{\alpha,\Omega}^* = U_{\alpha, \Omega} \subset U_{\alpha,x} = U_{\alpha,\lambda}$ and $U_{-\alpha, \Omega}^* \subset U_{-\alpha, y} = U_{-\alpha,-\mu}$.
Since $\lambda - \mu = \varepsilon >0$, we have by Proposition~\ref{PropRankOneLevi} that $N_{\alpha,\Omega}^* \subset N_{\alpha, \lambda}^\varepsilon \subset T_b$.

Hence $N_\Omega^*$ is a subgroup of $T_b$, whence it normalizes $U^+$.
Thus $n^{-1} v_+ n \in U^+$.
Hence $t=n$ and $v_-^{-1}u_-^{-1} =1$ by~\cite[6.1.15(c)]{BruhatTits1}.
Therefore $u_- \in U_\Omega^*$, whence $\overline{g} = \pi(u_-) = 1$.

It is a generating root group datum since the $U_{\alpha,\Omega}$ generate $U_\Omega$ by definition.
\end{proof}

\section{Parahoric subgroups and Bruhat decomposition}
\label{SecParahoricBruhat}

In this chapter, we work under data and notations of~\ref{HypCAVRGD}.
Note that all statements can be trivially generalized for the empty root system ($G=N=T$ and $\mathbb{A}_\Rtot = \{o\}$ when $\Phi$ is empty).

\subsection{Parahoric subgroups}

In this section, we will consider various subsets $\Omega$ of the apartment $\mathbb{A}_\Rtot$.

\begin{Not}
For a non-empty subset $\Omega \subset \mathbb{A}_\Rtot$, we let $N_\Omega$\index[notation]{n@$N_\Omega$} be the subgroup of $G$ generated by $\widetilde{N}_\Omega$ and $T_b$ (as defined in \cite[7.1.3]{BruhatTits1}) and we let $P_\Omega$ (resp. $P'_\Omega$)\index[notation]{p@$P_\Omega$}\index[notation]{p@$P'_\Omega$} be the subgroup of $G$ generated by $T_b$ and $U_\Omega$ (resp. $T_b$ and $U'_{\Omega}$).
\end{Not}

\begin{Not}
For any basis $\Delta$ of $\Phi$, we denote by $U_\Delta^+$\index[notation]{u@$U_\Delta^+$} (resp. $U_\Delta^-$\index[notation]{u@$U_\Delta^-$}) the subgroup of $G$ generated by the $U_\alpha$ (resp. $U_{-\alpha}$) for $\alpha \in \Phi_\Delta^+$.
\end{Not}

\begin{Lem}\label{LemPOmegaQC}
For any non-empty subset $\Omega \subset \mathbb{A}_\Rtot$ and any basis $\Delta$ of $\Phi$, we have:
\begin{itemize}
\item $P_\Omega \cap N = N_\Omega = T_b \widetilde{N}_\Omega = \widetilde{N}_\Omega T_b$;
\item $P_\Omega \cap U_\Delta^+ = U_\Omega \cap U_\Delta^+$;
\item $P_\Omega \cap U_\Delta^- = U_\Omega \cap U_\Delta^-$;
\item $P_\Omega = (P_\Omega \cap U_\Delta^+) (P_\Omega \cap U_\Delta^-)(P_\Omega \cap N) = T_b U_\Omega$.
\end{itemize}
\end{Lem}

\begin{proof}
Since $T_b \subset N$ normalizes $U_\Omega$, it normalizes $\widetilde{N}_\Omega= N \cap U_\Omega$, so that $N_\Omega = T_b \widetilde{N}_\Omega = \widetilde{N}_\Omega T_b$ and $P_\Omega = U_\Omega T_b = (U_\Omega \cap U_\Delta^+) (U_\Omega \cap U_\Delta^-) \widetilde{N}_\Omega T_b$ according to example~\ref{ExQC}(3).

Take $p \in P_\Omega$ and write $p$ as $p = u v n$ with $n \in \widetilde{N}_\Omega T_b \subset N$, $u \in U_\Omega \cap U_\Delta^+$ and $v \in U_\Omega \cap U_\Delta^-$.
If $p \in N$, then $pn^{-1} = uv$ gives $p = n$ by spherical Bruhat decomposition \cite[6.1.15(c)]{BruhatTits1}.
Thus $N \cap P_\Omega = N_\Omega$.
If $p \in U_\Delta^+$, then $n = v^{-1} (u^{-1}p) \in U_\Delta^- U_{\Delta}^+$.
Thus $n = 1$ by \cite[6.1.15(c)]{BruhatTits1} again and $u^{-1} p = v \in U_\Delta^- \cap U_\Delta^+ = \{1\}$ by~\ref{axiomRGD6}.
Thus $p = u$ and therefore $P_\Omega \cap U_\Delta^+ = U_\Omega \cap U_\Delta^+$.
By an analogous method, we get that $P_\Omega \cap U_\Delta^- = U_\Omega \cap U_\Delta^-$.
\end{proof}

\begin{Not}\label{notParahoric_subgroup}
We denote by:
\[\widehat{N}_\Omega = \{n \in N,\ \forall x \in \Omega,\ \nu(n)(x) =x\}\] \index[notation]{n@$\widehat{N}_\Omega$} the pointwise stabilizer of $\Omega$ in $N$ and by $\widehat{P}_\Omega$ (resp. $\widehat{P}'_\Omega$)\index[notation]{p@$\widehat{P}_\Omega$}\index[notation]{p@$\widehat{P}'_\Omega$} the subgroup of $G$ generated by $U_\Omega$ (resp. $U'_{\Omega}$) and $\widehat{N}_\Omega$.
The group $\widehat{P}_\Omega$ is called the \textbf{parahoric subgroup}\index{p@parahoric subgroup}\index{subgroup!parahoric} of $\Omega$ in $G$.
We will see later in Lemma~\ref{lemParahoric_fixator} that $\widehat{P}_\Omega$ is prescribed to be the pointwise stabilizer of $\Omega$ in $G$.
\end{Not}

\begin{Fact}
If $\Omega' \subset \Omega \subset \mathbb{A}_\Rtot$ are two non-empty subsets, then $U_{\Omega}$ is a subgroup of $U_{\Omega'}$.
\end{Fact}

\begin{proof}
For $\alpha \in \Phi$, we have $U_{\alpha,\Omega} = \bigcap_{x \in \Omega} U_{\alpha,x} \subset \bigcap_{x \in \Omega'} U_{\alpha,x} = U_{\alpha,\Omega'}$.
Since $U_\Omega$ and $ U_{\Omega'}$ are respectively generated by the $U_{\alpha,\Omega}$ and the $U_{\alpha,\Omega'}$ for $\alpha \in \Phi$, we deduce that $U_\Omega \subset U_{\Omega'}$.
\end{proof}

\begin{Rq}
As a consequence, the same inclusions hold for the groups $\widetilde{N}_\Omega$, $T_\Omega$, $P_\Omega$, $N_\Omega$, $\widehat{N}_\Omega$ being the pointwise stabilizer of $\Omega$ in $N$, and finally $\widehat{P}_\Omega = \widehat{N}_\Omega U_\Omega$.
\end{Rq}

\subsection{Action of \texorpdfstring{$N$}{N} on parahoric subgroups}

The action of $N$ on $\mathbb{A}_\Rtot$ can be compared with the action of $N$ on $V^*$ as follows:

\begin{Lem}\label{LemContravariantAction}
For any $\alpha \in \Phi$, any $n \in N$ and any $x \in \Omega$, we have:
\[ {^v\!}{\nu}(n)(\alpha)(x-o) = \alpha( \nu(n^{-1})(x)-\nu(n^{-1})(o) ).\]
\end{Lem}

\begin{proof}
We know that $\nu(N)$ is contained in the subgroup of $\operatorname{Aff}_{\Rtot}(\mathbb{A}_\Rtot)$ generated by the $r_{\beta,\mu}$ for $\beta \in \Phi$ and $\mu \in \Rtot$ and that ${^v\!}\nu(M_\beta) = \{r_\beta\}$ by~\ref{axiomCA1} and \cite[6.1.2(10)]{BruhatTits1}.
Moreover, $N$ is generated by the $M_\beta$ for $\beta \in \Phi$.
Thus, it suffices to prove that for any $x \in \mathbb{A}_\Rtot$, any $\beta \in \Phi$ and any $\mu \in \Rtot$, we have:
\[ r_\beta(\alpha)(x-o) = \alpha( r_{\beta,\mu}^{-1}(x)-r_{\beta,\mu}^{-1}(o) ).\]
Note that $r_{\beta,\mu}^{-1} = r_{\beta,\mu}$. For any $x \in \mathbb{A}_\Rtot$, we have:
\begin{align*}
r_{\beta,\mu}(x) - r_{\beta,\mu}(o) =& \left( x- (\beta(x-o) + \mu) \beta^\vee \right) -\left( o- (\beta(o-o) + \mu) \beta^\vee \right)\\
=&x-o - \beta(x-o) \beta^\vee
\end{align*}
Hence
\[ \alpha\left( r_{\beta,\mu}(x) - r_{\beta,\mu}(o) \right) =
\alpha(x-o) -\alpha(\beta^\vee) \beta(x-o) = r_\beta(\alpha)(x-o).\]
\end{proof}

\begin{Lem}\label{LemActionNUaOmega}
For any $n \in N$ and any $\alpha \in \Phi$, we have $n U_{\alpha,\Omega} n^{-1} = U_{{^v\!}\nu(n)(\alpha), \nu(n)(\Omega)}$ and $n U'_{\alpha,\Omega} n^{-1} = U'_{{^v\!}\nu(n)(\alpha), \nu(n)(\Omega)}$.
\end{Lem}

\begin{proof}
Let $x \in \Omega$.
We have $n U_{\alpha,x} n^{-1} = U_{{^v\!}\nu(\alpha),-\alpha(x-o) + \alpha(\nu(n^{-1})(o) - o)}$ according to Corollary~\ref{CorActionNUa}.
But:
$$
-\alpha(x-o) + \alpha(\nu(n^{-1})(o) - o) = \alpha(\nu(n^{-1})(o) - \nu(n^{-1}) \circ \nu(n)(x)))=- {^v\!}\nu(n)(\alpha) (\nu(n)(x) - o)$$ by Lemma~\ref{LemContravariantAction}.
Thus $n U_{\alpha,x} n^{-1} = U_{{^v\!}\nu(n)(\alpha),\nu(n)(x)}$ by definition.
Hence: $$n U_{\alpha,\Omega} n^{-1} = \bigcap_{x\in\Omega} n U_{\alpha,x} n^{-1} = \bigcap_{x\in\Omega} U_{{^v\!}\nu(n)(\alpha),\nu(n)(x)} = U_{{^v\!}\nu(n)(\alpha),\nu(n)(\Omega)}.$$
We proceed in the same way for $U'_{\alpha,\Omega} = \bigcap_{x \in \Omega} \bigcup_{\varepsilon > 0} U_{\alpha,-\alpha(x) + \varepsilon}$.
\end{proof}

\begin{Prop}\label{PropActionNUOmega}
For any $n \in N$, we have $n U_\Omega n^{-1} = U_{\nu(n)(\Omega)}$ and $n U'_\Omega n^{-1} = U'_{\nu(n)(\Omega)}$.
\end{Prop}

\begin{proof}
By Lemma~\ref{LemActionNUaOmega}, we have $n U_{\alpha,\Omega} n^{-1} = U_{{^v\!}\nu(n)(\alpha), \nu(n)(\Omega)}$.
Since $U_{\nu(n)(\Omega)}$ is generated by the $U_{\beta,\nu(n)(\Omega)}$ for $\beta \in \Phi$ and ${^v\!}\nu(\Phi) = \Phi$ by definition of root systems, we are done.
We proceed in the same way for $U'_\Omega$.
\end{proof}

\begin{Cor}\label{CorNchapeauNormaliseUOmega}
The group $\widehat{N}_\Omega$ normalizes $U_\Omega$ and $U'_{\Omega}$.
\end{Cor}

\begin{proof}
Since $\nu(n)(\Omega) = \Omega$ for any $n \in \widehat{N}_\Omega$, we are done by Proposition~\ref{PropActionNUOmega}.
\end{proof}

\begin{Cor}\label{CorConjugationParahorics}
For any $n \in N$, we have $n \widehat{P}_\Omega n^{-1} = \widehat{P}_{\nu(n)(\Omega)}$ and $n \widehat{P}'_\Omega n^{-1} = \widehat{P}'_{\nu(n)(\Omega)}$.
\end{Cor}

\begin{proof}
By definition, $\widehat{N}_\Omega$ is the pointwise stabilizer of $\Omega$ in $N$ and hence $n \widehat{N}_\Omega n^{-1} = \widehat{N}_{\nu(n)(\Omega)}$.
We have $\widehat{P}_\Omega = \widehat{N}_\Omega U_\Omega$ since $\widehat{N}_\Omega$ normalizes $U_\Omega$ by Corollary~\ref{CorNchapeauNormaliseUOmega}.
Thus, by applying Proposition~\ref{PropActionNUOmega}, we get that $n \widehat{P}_\Omega n^{-1} = \widehat{N}_{\nu(n)(\Omega)} U_{\nu(n)(\Omega)}$.
The same holds for $\widehat{P}'_\Omega$.
\end{proof}

\begin{Cor}\label{CorPchapeauOmegaQC}
For any non-empty subset $\Omega \subset \mathbb{A}_\Rtot$ and any basis $\Delta$ of $\Phi$, we have:
\begin{itemize}
\item $\widehat{P}_\Omega \cap N = \widehat{N}_\Omega$;
\item $\widehat{P}_\Omega \cap U_\Delta^+ = U_\Omega \cap U_\Delta^+$;
\item $\widehat{P}_\Omega \cap U_\Delta^- = U_\Omega \cap U_\Delta^-$;
\item $\widehat{P}_\Omega = (\widehat{P}_\Omega \cap U_\Delta^+) (\widehat{P}_\Omega \cap U_\Delta^-)(\widehat{P}_\Omega \cap N) = U_\Omega \widehat{N}_\Omega$.
\end{itemize}
\end{Cor}

\begin{proof}
By Corollary~\ref{CorNchapeauNormaliseUOmega}, we have $\widehat{P}_\Omega = U_\Omega \widehat{N}_\Omega$.
Thus, one can conclude as in Lemma~\ref{LemPOmegaQC} using \cite[6.1.15(c)]{BruhatTits1}.
\end{proof}

\begin{Cor}\label{CorNtildeFixOmega}
We have $\widetilde{N}_\Omega \subset N_\Omega \subset \widehat{N}_\Omega$.
In particular, the groups $\widetilde{N}_\Omega$ and $N_\Omega$ fix $\Omega$.
\end{Cor}

\begin{proof}
By definition, we have $U_\Omega \subset P_\Omega \subset \widehat{P}_\Omega$.
Thus, by intersecting with $N$ and by using Lemma~\ref{LemPOmegaQC} and Corollary~\ref{CorPchapeauOmegaQC}, we deduce $\widetilde{N}_\Omega \subset N_\Omega \subset \widehat{N}_\Omega$.
\end{proof}

\begin{Lem}\label{LemNOmegaNormaliseTstarOmega}
The group $\widehat{N}_\Omega$ normalizes $T^*_\Omega$.
\end{Lem}

\begin{proof}
For any $\alpha \in \Phi$ and any $n \in \widehat{N}_\Omega$, we get that $n U_{\alpha,\Omega} n^{-1} = U_{{^v\!}\nu(n)(\alpha),\Omega}$ and $n U'_{-\alpha,\Omega} n^{-1} = U'_{- {^v\!}\nu(n)(\alpha), \Omega}$ thanks to Lemma~\ref{LemActionNUaOmega} and the observation that $\nu(n)(\Omega) = \Omega$.
Thus $n L'_{\alpha,\Omega} n^{-1} = L'_{{^v\!}\nu(n)(\alpha),\Omega}$.
Since $N$ normalizes $T$, by intersecting the previous equality with $T$, we get $n T'_{\alpha,\Omega} n^{-1} = T'_{{^v\!}\nu(n)(\alpha),\Omega}$.
Thus $n$ normalizes $T^*_\Omega$ since $T^*_\Omega$ is generated by the $T'_{\alpha,\Omega}$ for $\alpha \in \Phi$ and ${^v\!}\nu(n)(\Phi)=\Phi$.
Hence $\widehat{N}_\Omega$ normalizes $T^*_\Omega$.
\end{proof}

\begin{Prop}\label{PropPchapeauDistingueUprimeTstar}
The subset $U'_\Omega T^*_\Omega$ is a normal subgroup of $\widehat{P}_\Omega$.
\end{Prop}

\begin{proof}
Let us recall that $T^*_\Omega \subset T_\Omega \subset T_b$ by Lemma~\ref{LemInclusionTOmega}.
Thus $T^*_\Omega$ normalizes $U'_\Omega$ so that $U'_\Omega T^*_\Omega$ is a subgroup of $U_\Omega = U_\Omega T_\Omega$ and hence of $\widehat{P}_\Omega$.

Let $\alpha \in \Phi$ and $u \in U_{\alpha,\Omega}$.
By Lemma~\ref{LemCommutationTstarUaOmega}, we have $[U_{\alpha,\Omega},T^*_\Omega] \subset U'_{\alpha,\Omega}$. Thus $u T^*_\Omega u^{-1} \subset U'_{\Omega} T^*_\Omega$.
By Lemmas~\ref{LemCommutationUaU'a} and~\ref{LemQC2}(\ref{eqQC2'}), we get that $[U_{\alpha,\Omega},U'_{\beta,\Omega}] \subset U'_\Omega T^*_{\Omega}$.
Thus $u U'_{\beta,\Omega} u^{-1} \subset U'_\Omega T^*_{\Omega}$.
Hence $U_{\alpha,\Omega}$ normalizes $U'_\Omega T^*_{\Omega}$ for any $\alpha \in \Phi$.
Therefore $U_\Omega$ normalizes $U'_\Omega T^*_{\Omega}$.

Moreover, according to Lemma~\ref{LemNOmegaNormaliseTstarOmega} and Corollary~\ref{CorNchapeauNormaliseUOmega}, we deduce that $\widehat{N}_\Omega$ and thus $\widehat{P}_\Omega$ normalizes $U'_\Omega T^*_{\Omega}$.
\end{proof}

\begin{Prop}\label{PropNchapeauNormaliseUstarOmega}
The group $\widehat{N}_\Omega$ normalizes $U^*_\Omega$
\end{Prop}

\begin{proof}
We firstly prove that the action of $\widehat{N}_\Omega$ on $\Phi$ via ${^v\!}\nu$ stabilizes $\Phi^*_\Omega$.
Let $\alpha \in \Phi^*_\Omega$ and $n \in \widehat{N}_\Omega$.
By definition, there is a constant $\lambda_\alpha \in \Rtot$ such that $\forall x \in \Omega,\ -\alpha(x) = \lambda_\alpha$.
For any $x \in \mathbb{A}_\Rtot$, we have ${^v\!}\nu(n)(\alpha)(x-o) = \alpha(\nu(n^{-1})(x) - \nu(n^{-1})(o))$ by Lemma~\ref{LemContravariantAction}.
If $x \in \Omega$, then $\nu(n^{-1})(x) =x$ so that $-{^v\!}\nu(n)(\alpha)(x-o) -\alpha(x-\nu(n^{-1})(o)) = -\alpha(x-o) + \alpha(\nu(n^{-1})(o)-o) = \lambda_\alpha + \alpha(\nu(n^{-1})(o)-o)$.
Hence, ${^v\!}\nu(n)(\alpha) \in \Phi^*_\Omega$ and therefore $\widehat{N}_\Omega$ stabilizes $\Phi^*_\Omega$.

Let $n \in \widehat{N}_\Omega$ and $\alpha \in \Phi$.
If $\alpha \in \Phi^*_\Omega$, we have $U^*_{\alpha,\Omega} = U'_{\alpha,\Omega}$ by definition.
Thus, by Lemma~\ref{LemActionNUaOmega}, we have $n U^*_{\alpha,\Omega} n^{-1} = U'_{{^v\!}\nu(n)(\alpha),\Omega} = U^*_{{^v\!}\nu(n)(\alpha),\Omega}$ since ${^v\!}\nu(n)(\alpha) \in \Phi^*_\Omega$ and $\nu(n)(\Omega) = \Omega$.
If $\alpha \not\in \Phi^*_\Omega$, then we have $U^*_{\alpha,\Omega} = U_{\alpha,\Omega}$ by definition.
Thus, by Lemma~\ref{LemActionNUaOmega}, we have $n U^*_{\alpha,\Omega} n^{-1} = U_{{^v\!}\nu(n)(\alpha),\Omega} = U^*_{{^v\!}\nu(n)(\alpha),\Omega}$ since ${^v\!}\nu(n)(\alpha) \not\in \Phi^*_\Omega$ and $\nu(n)(\Omega) = \Omega$.

Since $U^*_\Omega$ is generated by the $U^*_{\alpha,\Omega}$ for $\alpha \in \Phi$, we get that $n$ normalizes $U^*_\Omega$ for any $n \in \widehat{N}_\Omega$.
\end{proof}

\begin{Cor}\label{CorPchapeauNormaliseUstarOmega}
The group $U^*_\Omega$ is a normal subgroup of $\widehat{P}_\Omega$.
\end{Cor}

\begin{proof}
It is an immediate consequence of Proposition~\ref{PropNchapeauNormaliseUstarOmega} and Proposition~\ref{PropUstarQC} since $\widehat{P}_\Omega$ is generated by $\widehat{N}_\Omega$ and $U_\Omega$.
\end{proof}

\subsection{Parahoric subgroups as intersections over their fixed points}

We get statements analogous to those of \cite[7.1.4]{BruhatTits1} as follows:

\begin{Lem}\label{LemUsector}
Let $\Delta$ be a basis of $\Phi$ with order $\Phi_\Delta^+$ in $\Phi$.
Let $C_{\Rtot,\Delta}^v$ be the vector chamber over $\Delta$.
Then, for any $x \in \mathbb{A}_\Rtot$, we have $U_{x + C_{\Rtot,\Delta}^v} \subset U_{\Delta}^+$.
In particular, $\widetilde{N}_{x+C_{\Rtot,\Delta}^v} = \{1\}$ and $N_{x+C_{\Rtot,\Delta}^v} = T_b$.
\end{Lem}

\begin{proof}
Let $\Omega = x + C_{\Rtot,\Delta}^v$.
Let $v \in C_{\mathbb{Z},\Delta}^v$ so that $x + \delta_v \subset x + C_{\Rtot,\Delta}^v$ according to Lemma~\ref{LemHalfLineInSector}.
Let $\alpha \in \Phi_\Delta^+$.
Thus $\alpha(v) \in \mathbb{Z}_{>0}$ by definition of $v$.

Let $\varepsilon \in \Rtot_{>0}$.
Then $y \in x + \varepsilon v \in x + \delta_v \subset \Omega$.
Moreover, we have $\alpha(y) = \alpha(x) + \alpha(\varepsilon v) = \alpha(x) + \alpha(v) \varepsilon$.
Hence $U_{-\alpha,\Omega} \subset U_{-\alpha,y} = \varphi_{-\alpha}^{-1}([\alpha(x) + \alpha(v) \varepsilon, \infty])$.
Since this is true for any $\varepsilon > 0$ and $\bigcap_{\varepsilon > 0} [\alpha(x) +  \alpha(v) \varepsilon,\infty] = \bigcap_{\varepsilon > 0} [\alpha(x) + \varepsilon,\infty]= \{\infty\}$ (because $\alpha(v) \in \mathbb{Z}_{>0}$ and $\Rtot$ is $\mathbb{Z}$-torsion free), we get that $U_{-\alpha,\Omega} = U_{-\alpha,\infty} = \{1\}$.
Thus, $U_\Omega$ is generated by $U_{\alpha,\Omega}$ for $\alpha \in \Phi_\Delta^+$, and is therefore a subgroup of $U_{\Delta}^+$. We conclude by applying Lemma~\ref{LemPOmegaQC}.
\end{proof}

\begin{Lem}\label{LemUaAlcove}
For any basis $\Delta$ of $\Phi$, any positive root $\alpha \in \Phi_\Delta^+$ and any non-empty subset $\Omega \subset \mathbb{A}_\Rtot$, we have
\[ U_{\alpha,\Omega + C_{\Rtot,\Delta}^v} = U_{\alpha,\Omega + \overline{C}_{\Rtot,\Delta}^v} = U_{\alpha,\Omega}.\]
\end{Lem}

\begin{proof}
Let $v \in C_{\mathbb{Z},\Delta}^v$ so that $\alpha(v) \in \mathbb{Z}_{>0}$.
Since $ \Omega + \overline{C}_{\Rtot,\Delta}^v \supset \Omega+ C_{\Rtot,\Delta}^v \supset \Omega + \delta_v$, we have $U_{\alpha,\Omega + \overline{C}_{\Rtot,\Delta}^v} \subset U_{\alpha,\Omega +C_{\Rtot,\Delta}^v} \subset U_{\alpha,\Omega + \delta_v}$.
Thus, it suffices to prove that $U_{\alpha,\Omega + \delta_v} \subset U_{\alpha,\Omega} \subset U_{\alpha,\Omega + \overline{C}_{\Rtot,\Delta}^v}$.

Let $u \in U_{\alpha,\Omega + \delta_v}$.
Then for any $y \in \Omega$ and any $\varepsilon \in \alpha(v) \cdot \Rtot_{>0}$,
if $\lambda \in \Rtot$ is such that $\alpha(v) \lambda = \varepsilon$,
then we have $y + \lambda v \in \Omega + \delta_v$ and hence:
\[\varphi_\alpha(u) \geqslant -\alpha(y + \lambda v) = -\alpha(x) - \varepsilon.\]
Since this inequality holds for every $\varepsilon \in \alpha(v) \cdot \Rtot_{>0}$, we get that $\varphi_\alpha(u) \geqslant -\alpha(y)$.
Thus $u \in U_{\alpha,\Omega}$ by definition.
For any $x \in \Omega + \overline{C}_{\Rtot,\Delta}^v$, write $x = y +z $ with $y \in \Omega$ and $z \in \overline{C}_{\Rtot,\Delta}^v$.
Then $-\alpha(z) \leqslant 0$ by definition so that $-\alpha(x) \leqslant -\alpha(y) \leqslant \varphi_\alpha(u)$.
Thus, $U_{\alpha,\Omega} \subset U_{\alpha,\Omega + \overline{C}_{\Rtot,\Delta}^v}$.
\end{proof}

\begin{Prop}\label{PropPintersectionSector}
For any non-empty subset $\Omega \subset \mathbb{A}_\Rtot$, and any basis $\Delta$ of $\Phi$, we have
\begin{align*}
P_\Omega \cap U^+_\Delta &= U_{\Omega + C_{\Rtot,\Delta}^v} = U_{\Omega + \overline{C}_{\Rtot,\Delta}^v}&
P_\Omega \cap U^-_\Delta &= U_{\Omega - \overline{C}_{\Rtot,\Delta}^v}= U_{\Omega - C_{\Rtot,\Delta}^v}
\end{align*}
\end{Prop}

\begin{proof}
On the one hand, we have $\Omega + C_{\Rtot,\Delta}^v \subset \Omega + \overline{C}_{\Rtot,\Delta}^v$ so that $U_{\Omega + \overline{C}_{\Rtot,\Delta}^v} \subset U_{\Omega + C_{\Rtot,\Delta}^v} \subset U_\Delta^+$ by Lemma~\ref{LemUsector}.
Hence, $U_{\Omega + \overline{C}_{\Rtot,\Delta}^v}$ and $U_{\Omega + C_{\Rtot,\Delta}^v}$ are generated by the same subgroups $U_{\alpha,\Omega + \overline{C}_{\Rtot,\Delta}^v} = U_{\alpha,\Omega + C_{\Rtot,\Delta}^v} = U_{\alpha,\Omega}$ for $\alpha \in (\Phi_\Delta^+)_{\mathrm{nd}}$ according to Example~\ref{ExQC} and Lemma~\ref{LemUaAlcove}.
Thus $U_{\Omega + \overline{C}_{\Rtot,\Delta}^v} = U_{\Omega + C_{\Rtot,\Delta}^v} \subset P_\Omega \cap U_\Delta^+$.

On the other hand, we have $P_\Omega \cap U_\Delta^+ = U_\Omega \cap U_\Delta^+$ by Lemma~\ref{LemPOmegaQC} and this group is generated by the $U_{\alpha,\Omega}$ for $\alpha \in (\Phi_\Delta^+)_{\mathrm{nd}}$ according to Example~\ref{ExQC}.
Hence we get the first equalities.

Since $U_\Delta^- = U_{-\Delta}^+$ and $-\overline{C}_{\Rtot,\Delta}^v = \overline{C}_{\Rtot,-\Delta}^v$, we get the second equalities by applying the first ones with $-\Delta$ instead of $\Delta$.
\end{proof}

\begin{Cor}\label{CorPOmegaDecQC}
For any non-empty subset $\Omega \subset \mathbb{A}_\Rtot$ and any basis $\Delta$ of $\Phi$, we have $P_\Omega = \widetilde{N}_\Omega U_{\Omega + \overline{C}_{\Rtot,\Delta}^v}U_{\Omega - \overline{C}_{\Rtot,\Delta}^v} = \widetilde{N}_\Omega U_{\Omega + C_{\Rtot,\Delta}^v}U_{\Omega - C_{\Rtot,\Delta}^v}$.
\end{Cor}

\begin{proof}
This is a combination of Proposition~\ref{PropPintersectionSector} with Lemma~\ref{LemPOmegaQC}.
\end{proof}

Thus we get, as in \cite[7.1.8]{BruhatTits1}:

\begin{Cor}\label{CorPchapeauOmegaDecQC}
For any non-empty subset $\Omega \subset \mathbb{A}_\Rtot$ and any basis $\Delta$ of $\Phi$, we have $\widehat{P}_\Omega = \widehat{N}_\Omega U_{\Omega + \overline{C}_{\Rtot,\Delta}^v}U_{\Omega - \overline{C}_{\Rtot,\Delta}^v} = \widehat{N}_\Omega U_{\Omega + C_{\Rtot,\Delta}^v}U_{\Omega - C_{\Rtot,\Delta}^v}$.
Moreover
\begin{align*}
\widehat{P}_\Omega \cap N =& \widehat{N}_\Omega &
\widehat{P}_\Omega \cap U_\Delta^+ &= U_{\Omega + \overline{C}_{\Rtot,\Delta}^v} = U_{\Omega + C_{\Rtot,\Delta}^v}&
\widehat{P}_\Omega \cap U_\Delta^- &= U_{\Omega - \overline{C}_{\Rtot,\Delta}^v} = U_{\Omega - C_{\Rtot,\Delta}^v}
\end{align*}
Moreover, $P_\Omega$ and $U_\Omega$ are normal subgroups of $\widehat{P}_\Omega$.
\end{Cor}

\begin{proof}
This is a combination of Proposition~\ref{PropPintersectionSector}, Corollary~\ref{CorNchapeauNormaliseUOmega} and spherical Bruhat decomposition \cite[7.1.15(c)]{BruhatTits1}.
\end{proof}

The previous Proposition~\ref{PropPintersectionSector} and Corollary~\ref{CorPchapeauOmegaDecQC} enable us to state the following proposition (see~\cite[7.1.5]{BruhatTits1}).

\begin{Prop}\label{PropBizarre}
Let $\Omega$ and $\Omega'$ be two subsets of $\mathbb{A}_\Rtot$.
Let $\Delta$ be any basis of $\Phi$.
\begin{enumerate}[label={(\arabic*)}]
\item\label{PropBizarre1} If $\Omega' \subset \Omega + \overline{C}_{\Rtot,\Delta}^v$, then $P_\Omega P_{\Omega'} \subset \widetilde{N}_\Omega U_\Delta^- U_{\Omega' + \overline{C}_{\Rtot,\Delta}^v} \widetilde{N}_{\Omega'}$.
\item\label{PropBizarre2} If $\Omega' \subset \Omega + \overline{C}_{\Rtot,\Delta}^v$ and $\Omega \subset \Omega' - \overline{C}_{\Rtot,\Delta}^v$, then $P_\Omega P_{\Omega'} = \widetilde{N}_\Omega U_{\Omega - \overline{C}_{\Rtot,\Delta}^v} U_{\Omega' + \overline{C}_{\Rtot,\Delta}^v} \widetilde{N}_{\Omega'}$.
\end{enumerate}
\end{Prop}

\begin{proof}
According to Corollary~\ref{CorPOmegaDecQC}, one can write 
\[ P_\Omega P_{\Omega'} = \widetilde{N}_\Omega
U_{\Omega - \overline{C}_{\Rtot,\Delta}^v}
U_{\Omega + \overline{C}_{\Rtot,\Delta}^v}
U_{\Omega' + \overline{C}_{\Rtot,\Delta}^v}
U_{\Omega' - \overline{C}_{\Rtot,\Delta}^v}
\widetilde{N}_{\Omega'}. \]

Assume that $\Omega' \subset \Omega + \overline{C}_{\Rtot,\Delta}^v$.
Thus $\Omega' + \overline{C}_{\Rtot,\Delta}^v \subset \Omega + \overline{C}_{\Rtot,\Delta}^v$ so that $U_{\Omega + \overline{C}_{\Rtot,\Delta}^v} \subset U_{\Omega' + \overline{C}_{\Rtot,\Delta}^v}$.
Hence $P_\Omega P_{\Omega'} = \widetilde{N}_\Omega U_{\Omega - \overline{C}_{\Rtot,\Delta}^v} P_{\Omega'}$.
Now, write $P_{\Omega'} = 
U_{\Omega' - \overline{C}_{\Rtot,\Delta}^v}
U_{\Omega' + \overline{C}_{\Rtot,\Delta}^v}
\widetilde{N}_{\Omega'}$.
Then
\[ P_\Omega P_{\Omega'} =\widetilde{N}_\Omega
U_{\Omega - \overline{C}_{\Rtot,\Delta}^v}
U_{\Omega' - \overline{C}_{\Rtot,\Delta}^v}
U_{\Omega' + \overline{C}_{\Rtot,\Delta}^v}
\widetilde{N}_{\Omega'}\]
and, since $U_{\Omega - \overline{C}_{\Rtot,\Delta}^v} U_{\Omega' - \overline{C}_{\Rtot,\Delta}^v} \subset U_\Delta^-$ according to Proposition~\ref{PropPintersectionSector}, we get~\ref{PropBizarre1}.

Assume, furthermore, that $\Omega \subset \Omega' - \overline{C}_{\Rtot,\Delta}^v$.
Then $\Omega - \overline{C}_{\Rtot,\Delta}^v \subset \Omega' - \overline{C}_{\Rtot,\Delta}^v$ so that $U_{\Omega' - \overline{C}_{\Rtot,\Delta}^v} \subset U_{\Omega - \overline{C}_{\Rtot,\Delta}^v}$.
Thus
\[ P_\Omega P_{\Omega'} =\widetilde{N}_\Omega
U_{\Omega - \overline{C}_{\Rtot,\Delta}^v}
U_{\Omega' + \overline{C}_{\Rtot,\Delta}^v}
\widetilde{N}_{\Omega'}\]
\end{proof}

Thus, we deduce from Proposition~\ref{PropGoodSector}, as in \cite[7.1.6 \& 7.1.7]{BruhatTits1}:

\begin{Cor}\label{corBT7.1.6}
Suppose that $\Rtot = \Rtot_\Q$.
For any non-empty subset $\Omega \subset \mathbb{A}_\Rtot$ and any point $x \in \mathbb{A}_\Rtot$, there is a basis $\Delta$ of $\Phi$ such that 
\[ P_\Omega P_x \subset \widetilde{N}_\Omega U_\Delta^- U_{x + \overline{C}_{\Rtot,\Delta}^v} \widetilde{N}_x.\]
\end{Cor}

\begin{proof}
Let $y \in \Omega$ be any point.
By Proposition~\ref{PropGoodSector}, there exists a basis $\Delta$ of $\Phi$ such that $x-y \in \overline{C}_{\Rtot,\Delta}^v$.
Thus $x \in y + \overline{C}_{\Rtot,\Delta}^v \subset \Omega + \overline{C}_{\Rtot,\Delta}^v$.
Thus, by Proposition~\ref{PropBizarre}\ref{PropBizarre1} applied to $\Omega' = \{x\}$, we have the desired inclusion.
\end{proof}

\begin{Cor}\label{corBT7.1.7}
Suppose that $\Rtot = \Rtot_\Q$.
Let $x,y \in \mathbb{A}_\Rtot$. There exists a basis $\Delta$ of $\Phi$ such that
\[P_y P_x = \widetilde{N}_y U_{y - \overline{C}_{\Rtot,\Delta}^v}U_{x +  \overline{C}_{\Rtot,\Delta}^v} \widetilde{N}_x.\]
\end{Cor}

\begin{proof}
By Proposition~\ref{PropGoodSector}, there exists a basis $\Delta$ of $\Phi$ such that $x-y \in \overline{C}_{\Rtot,\Delta}^v$.
Then $x \in y + \overline{C}_{\Rtot,\Delta}^v \subset \Omega + \overline{C}_{\Rtot,\Delta}^v$ and 
 $y \in x - \overline{C}_{\Rtot,\Delta}^v \subset \Omega + \overline{C}_{\Rtot,\Delta}^v$.
Thus, by Proposition~\ref{PropBizarre}\ref{PropBizarre2} applied to $\Omega' = \{x\}$ and $\Omega = \{y\}$, we have the desired equality.
\end{proof}

Finally, we deduce that $\widehat{P}_\Omega$ can be written as an intersection, as in \cite[7.1.11]{BruhatTits1}:

\begin{proposition}\label{propBT7.1.11}
Suppose that $\Rtot = \Rtot_\Q$.
Let $\Omega\subset \A$, $\Omega\neq \emptyset$. Then $\widehat{P}_{\Omega}=\bigcap_{x\in \Omega} \widehat{P}_x$.
\end{proposition}

\begin{proof}
The inclusion $\bigcap_{x\in \Omega}\widehat{P}_x\supset \widehat{P}_\Omega$ is a consequence of the fact $\Omega \mapsto \widehat{P}_\Omega$ is decreasing.  

Let us prove that $\bigcap_{x\in \Omega}\widehat{P}_x\subset \widehat{P}_\Omega$.

 Let $\Omega_0  \subset \A_\Rtot$.
 We begin by proving that if $x\in \A_\Rtot$, one has $\widehat{P}_{\Omega_0}\cap \widehat{P}_x=\widehat{P}_{\Omega_0\cup \{x\}}$. Let $y\in \Omega_0$ and $C^v$ be a  vector chamber such that $y-x\in \overline{C}^v$. Then $U_{\Omega_0+C^v}\subset U_{y+C^v}\subset U_{x+C^v}$. Let $g\in \widehat{P}_{\Omega_0}\cap \widehat{P}_{x}$.  Write $g=nvu$, with $n\in \widehat{N}_{\Omega_0}$, $v\in U_{\Omega_0-C^v}$ and $u\in U_{\Omega_0+C^v}$, which is possible by Corollary~\ref{CorPchapeauOmegaDecQC}.
 By Lemmas~\ref{lemInclusion_point_local_face} and  \ref{lemUV=UclV}, that will be proved in the next section, one has $U_{x+C^v}\subset \widehat{P}_x$.
 Therefore $u\in \widehat{P}_x$ and $g^{-1}u\in \widehat{P}_x$. Thus $nv=n'u'v'$, with $n'\in \widehat{N}_x$, $u'\in U_{x+C^v}$ and $v'\in U_{x-C^v}$. 
 Therefore, $n'^{-1}n=u'(v'v^{-1})\in U_{\Delta_{C^v}}^+.U_{\Delta_{C^v}}^-$ (where $\Delta_{C^v}$ denotes the basis of $\Phi$ associated to $C^v$). By \cite[6.1.15 c)]{ BruhatTits1},  $n'=n$ and by axiom~\ref{axiomRGD6}, $v=v'$. 
 Therefore, $n \in \widehat{N}_x\cap \widehat{N}_{\Omega_0}\subset \widehat{N}_{\Omega_0\cup \{x\}}$, $v\in U_{\Omega_0-C^v}\cap U_{x-C^v}$ and $u\in U_{\Omega_0+C^v}\cap \widehat{P}_x$. 
 Moreover, by Corollary~\ref{CorPchapeauOmegaDecQC} and by Proposition~\ref{PropDecompositionQC}, one has:
  \[U_{\Omega_0-C^v}\cap \widehat{P}_x\subset U_{\Omega_0-C^v}\cap \big(U_{\Delta_{C^v}}^-\cap \widehat{P}_x\big)=U_{\Omega_0-C^v}\cap U_{x-C^v}\subset U_{(\Omega_0\cup\{x\})-C^v}\] and symmetrically, $U_{\Omega_0+C^v}\cap \widehat{P}_x \subset U_{(\Omega_0\cup \{x\})+C^v}$.  Therefore $\widehat{P}_{\Omega_0}\cap \widehat{P}_x=\widehat{P}_{\Omega_0\cup \{x\}}$ (by Corollary~\ref{CorPchapeauOmegaDecQC}). \\
 By induction, we deduce that for each finite subset $\Omega'$ of $\Omega$, one has: \[\widehat{P}_{\Omega'}=\bigcap_{x\in \Omega'}\widehat{P}_x.\]

Let $x_0\in \Omega$. Let $\mathrm{Fin}(\Omega,x_0)$ be the set of finite subsets   of $\Omega$ containing $x_0$. Let $g\in \bigcap_{x\in \Omega} \widehat{P}_x$. Then for all $\Omega'\in \mathrm{Fin}(\Omega,x_0)$, one has $g\in \bigcap_{x\in \Omega'} \widehat{P}_x=\widehat{P}_{\Omega'}$ and thus $g\in \bigcap_{\Omega'\in \mathrm{Fin}(\Omega,x_0)} \widehat{P}_{\Omega'}$. 
Let us prove that $\bigcap_{\Omega'\in \mathrm{Fin}(\Omega,x_0)} \widehat{P}_{\Omega'}\subset \widehat{P}_{\Omega}$.

Let $g\in \bigcap_{\Omega'\in \mathrm{Fin}(\Omega,x_0)} \widehat{P}_{\Omega'}$.
Since $\widehat{N}_{x_0}$ is, by definition the stabilizer of $x_0$ and $T_b$ is the kernel of the action $\nu: N \to \operatorname{Aff}(\mathbb{A}_\Rtot)$, the quotient group $\widehat{N}_{x_0} / T_b$ can be identified with a subgroup of $W^v$ which is finite.
We write the cosets $n_1 T_b,\ldots,n_k T_b$, with $k\in \N$ and $n_i\in \widehat{N}_{x_0}$, for all $i\in \llbracket 1,k\rrbracket$.
Choose a vector chamber $C^v$ (for example, $C^v=C^v_f$).
For $\Omega'\in \mathrm{Fin}(\Omega,x_0)$, one can write  $g=n_{\Omega'}u_{\Omega'}v_{\Omega'}$, with $n_{\Omega'}\in \widehat{N}_{\Omega'}$, $u_{\Omega'}\in U_{\Omega'+C^v}$ and $v_{\Omega'}\in U_{\Omega'-C^v}$. Let $J$ be the set of element $j\in \llbracket 1,k\rrbracket$ such that there exists $\Omega'\in \mathrm{Fin}(\Omega,x_0)$ satisfying the following property: \[\forall \tilde{\Omega}\in \mathrm{Fin}(\Omega,x_0)|\ \tilde{\Omega}\supset \Omega', n_{\tilde{\Omega}}\notin n_j T_b.\]

For $j\in J$, we pick $\Omega_j\in \mathrm{Fin}(\Omega,x_0)$ such that for every $\Omega'\in \mathrm{Fin}(\Omega,x_0)$ satisfying $\Omega'\supset \Omega_j$ the element $n_{\Omega'}$ is not in $ n_j T_b$.
Let $\tilde{\Omega}=\bigcup_{j\in J} \Omega_j$ and $\ell\in \llbracket 1,k\rrbracket$ be such that $n_{\tilde{\Omega}}\in n_{\ell} T_b$. Then $\ell\in \llbracket 1,k\rrbracket \setminus J$.

Let \[\mathcal{F}=\{\Omega'\in \mathrm{Fin}(\Omega,x_0)|n_{\Omega'}\in n_\ell T_b\}.\]  Then for every $\Omega'\in \mathrm{Fin}(\Omega,x_0)$, there exists $\Omega''\in \mathcal{F}$ such that $\Omega'\subset \Omega''$ and in particular, $\Omega=\bigcup_{\Omega'\in \mathcal{F}}\Omega'$. 

 Let $\Omega'\in \mathrm{Fin}(\Omega,x_0)$. Let  $\Omega''\in \mathrm{Fin}(\Omega,x_0)$ be such that $\Omega''\in \mathcal{F}$. As $T_b\subset \widehat{N}_{\Omega''}$, we deduce that $n_\ell\in \widehat{N}_{\Omega''}\subset \widehat{N}_{\Omega'}$. Consequently, $n_\ell\in \bigcap_{\Omega'\in  \mathrm{Fin}(\Omega,x_0)} \widehat{N}_{\Omega'}\subset \widehat{N}_{\Omega}$.

For $\Omega'\in\mathcal{F}$, write $n_{\Omega'}=n_\ell h_{\Omega'}$, with $h_{\Omega'}\in H$.  
Let $\Omega_1,\Omega_2\in \mathcal{F}$.  Then $h_{\Omega_1}u_{\Omega_1}v_{\Omega_1}=h_{\Omega_2}u_{\Omega_2}v_{\Omega_2}$ and by axiom~\ref{axiomRGD6}, we deduce that $h_{\Omega_1}=h_{\Omega_2}:=h$, $u_{\Omega_1}=u_{\Omega_2}:=u$ and $v_{\Omega_1}=v_{\Omega_2}:=v$. Then $g=n_\ell huv$.
Moreover $u\in \bigcap_{\Omega'\in \mathcal{F}} U_{\Omega'+C^v}\subset U_{\Omega+C^v}$ and $v\in \bigcap_{\Omega'\in \mathcal{F}} U_{\Omega'-C^v} \subset U_{\Omega-C^v}$. 
Therefore, $g=n_\ell h u v \in \widehat{N}_{\Omega}U_{\Omega+C^v}U_{\Omega-C^v}=\widehat{P}_{\Omega}$ (by Corollary~\ref{CorPOmegaDecQC}) and hence $g\in \widehat{P}_{\Omega}$.
Consequently, \[\bigcap_{x\in \Omega} \widehat{P}_x\subset\bigcap_{\Omega'\in \mathrm{Fin}(\Omega,x_0)} \widehat{P}_{\Omega'}\subset \widehat{P}_\Omega\subset \bigcap_{x\in \Omega} \widehat{P}_x,\]  which proves the proposition.
\end{proof}

\subsection{Subgroups associated to a filter}

If $\EC$ is a set, we denote by $\PCC(\EC)$ the set of subsets of $\EC$. Let 
\begin{align*}
X:\PCC(\A_\Rtot)&\rightarrow \PCC(G) \\ \Omega & \mapsto X_\Omega
\end{align*} be a decreasing map (for example, $X=U,N,\widehat{P},\ldots$). If $\VC$ is a filter on $\A_\Rtot$, we set $X_\VC=\bigcup_{\Omega\in \VC} X_{\Omega}$.
If $X_\Omega$ is a subgroup of $G$ for every subset $\Omega$ of $\A_\Rtot$, then $X_\VC$ is a subgroup of $G$.\index{subgroup!associated to a filter}\index{filter!subgroup associated to}
 
\begin{lemma}\label{lemUV=UclV}
Let $\VC$ be  a filter on $\A_\Rtot$. Then $U_{\VC}=U_{\cl(\VC)}$, $\widetilde{N}_\VC = \widetilde{N}_{\cl(\VC)}$ and $T_\VC = T_{\cl(\VC)}$.
Moreover $N_\VC = N_{\cl(\VC)}$, $P_\VC = P_{\cl(\VC)}$ and $U_{\alpha,\VC} = U_{\alpha,\cl(\VC)}$ for every $\alpha \in \Phi$.
\end{lemma}

\begin{proof}
As $\VC\Subset \cl(\VC)$, one has $U_{\cl(\VC)}\subset U_{\VC}$.
Let us prove the reverse inclusion.
We first assume that $\VC=\Omega$ is a set.
Let $u\in U_\Omega$.
Then by definition of $U_\Omega$, there exists $k\in \mathbb{Z}_{\geqslant 0}$, roots $\alpha_1,\ldots,\alpha_k\in \Phi$ and elements $u_i \in U_{\alpha_i,\Omega}$ such that $u=\prod_{i=1}^k u_i$.
For $\alpha\in \Phi$, set $\lambda_\alpha=\min \{\varphi_{\alpha_i}(u_i),\ i\in \llbracket 1,k\rrbracket\mathrm{\ and\ }\alpha_i=\alpha\}$ (one may have $\lambda_\alpha=\infty$).
Set $\Omega'=\bigcap_{\alpha\in \Phi} D_{\alpha,\lambda_\alpha}$.
For $i \in \llbracket 1,k \rrbracket$ and $x \in D_{\alpha_i,\lambda_{\alpha_i}}$, we have $\varphi_{\alpha_i}(u_i) \geqslant \lambda_{\alpha_i} \geqslant - \alpha_i(x)$ so that $u_i \in U_{\alpha_i,D_{\alpha_i,\lambda_{\alpha_i}}} \subset U_{\alpha_i,\Omega'}$.
Hence $u\in U_{\Omega'}$ and $\Omega'\in \cl(\Omega)$.
Thus $u\in U_{\cl(\Omega)}$ and therefore, $U_\Omega=U_{\cl(\Omega)}$.

We no longer assume that $\VC$ is a set. Let $u\in U_\VC$. Then there exists $\Omega\in \VC$ such that $u\in U_{\Omega}$. Therefore, $u\in U_{\cl(\Omega)}$ and thus there exists $\Omega'\in \cl(\Omega)$ such that $u\in U_{\Omega'}$. As $\Omega'\in \cl(\Omega)$, $\Omega'\in \cl(\VC)$ and thus $u\in U_{\cl(\VC)}$. Hence $U_{\VC}\subset U_{\cl(\VC)}$, which proves that $U_\VC=U_{\cl(\VC)}$.

For any filter $\VC$, by definition, we have $\widetilde{N}_\VC = \bigcup_{\Omega \in \VC} \widetilde{N}_\Omega = \bigcup_{\Omega \in \VC} N \cap U_\Omega = N \cap U_\VC$.
Thus $\widetilde{N}_\VC=N\cap U_\VC=N\cap U_{\cl(\VC)}=\widetilde{N}_{\cl(\VC)}$.
By the same way, we have $T_\VC = T_{\cl(\VC)}$ and $U_{\alpha,\VC} = U_{\alpha,\cl(\VC)}$ for any $\alpha \in \Phi$, since $U_{\Omega} \cap U_{\alpha} = U_{\alpha,\Omega}$ according to Example~\ref{ExQC}.

For any filter $\VC$, by definition and Lemma~\ref{LemPOmegaQC}, we have $P_\VC = \bigcup_{\Omega \in \VC} T_b U_\Omega = T_b U_\VC$.
Thus $P_\VC = P_{\cl(\VC)}$ and, by intersecting with $N$, we have $N_\VC = N_{\cl(\VC)}$.
\end{proof}

\begin{Rq}
As $\VC\Subset \cl(\VC)$, one has $\widehat{N}_{\cl(\VC)}\subset \widehat{N}_{\VC}$ but this inclusion is strict in general.
For instance, if $\mathbf{G} = \mathrm{PGL}(2)$ and $\mathbb{K}$ is a local field, if $x$ is the center of an edge of the Bruhat-Tits tree, there is an element of $N \subset G = \mathbf{G}(\mathbb{K})$ permuting the two vertices of the edge and, therefore, fixing $x$.
The enclosure $\cl(\VC)$ of $\{x\}$ is the edge and the pointwise stabilizer of $\cl(\VC)$ cannot exchange the vertices of the edge so that $\widehat{N}_{\cl(\VC)} \neq \widehat{N}_\VC$.
Thus $\widehat{P}_{\cl(\VC)} \neq \widehat{P}_\VC$ in general.
\end{Rq}

\begin{Rq}\label{RkDecQCfiltres}
Using Example~\ref{ExQC}(3) and Lemma~\ref{LemPOmegaQC} we deduce the following decompositions, when $\VC$ is a filter on $\A_{\Rtot}$: \[ U_\VC = \left( U_\VC \cap U_\Delta^+ \right) \left( U_\VC \cap U_\Delta^- \right) \left( U_\VC \cap N \right) \text{ with } U_\VC \cap N = \widetilde{N}_\VC,\]
\[ P_\VC = \left( P_\VC \cap U_\Delta^+ \right) \left( P_\VC \cap U_\Delta^- \right) \left( P_\VC \cap N \right) \text{ with } P_\VC \cap N = N_\VC.\] Indeed, let $u\in U_{\VC}$. Let $\Omega\in \VC$ be such that $u\in U_\Omega$. Then one can write $u=u^-u^+n$, with $u^+\in U_{\Omega}\cap U_{\Delta}^+$, $u^-\in U_\Omega\cap U_{\Delta}^-$ and $n\in U_\Omega\cap N$ and thus $u\in (U_{\VC}\cap U_{\Delta}^+)(U_\VC\cap U_{\Delta}^-)(U_\VC\cap N)$. Thus $U_\VC\subset (U_{\VC}\cap U_{\Delta}^+)(U_\VC\cap U_{\Delta}^-)(U_\VC\cap N)\subset U_{\VC}$ and we get similarly the statement with $P$.
\end{Rq}

Recall that given a vector face $F^v$ and a point $x \in \mathbb{A}_\Rtot$, we denote by $\mathcal{F}_{x,F^v}$ the face of direction $F^v$ at $x$ (see definition in section~\ref{SubsecGerms}). 

\begin{Prop}\label{PropPchapeauGermSectorDecQC}
Let $C^v$ be a vector chamber of $V_\Rtot$ and $x \in \mathbb{A}_\Rtot$.
Denote by $\mathcal{F} = \mathcal{F}_{x,C^v}$ the face of direction $C^v$ at $x$.
Then, for any basis $\Delta$ of $\Phi$, we have:
\begin{itemize}
\item $\widehat{P}_\mathcal{F} \cap U_\alpha = U_{\alpha,\mathcal{F}}$ for any root $\alpha\in \Phi$;
\item $\widehat{P}_\mathcal{F} \cap N = \widehat{N}_\mathcal{F} = T_b$;
\item $\widehat{P}_\mathcal{F} \cap U_\Delta^+ = U_\mathcal{F} \cap U_\Delta^+$;
\item $\widehat{P}_\mathcal{F} \cap U_\Delta^- = U_\mathcal{F} \cap U_\Delta^-$;
\item $\widehat{P}_\mathcal{F} = (\widehat{P}_\mathcal{F} \cap U_\Delta^+) (\widehat{P}_\mathcal{F} \cap U_\Delta^-)(\widehat{P}_\mathcal{F} \cap N)$.
\end{itemize}
\end{Prop}

\begin{proof}
Let $\alpha \in \Phi^+_\Delta$ and $\Omega\in \mathcal{F}$.
According to Corollary~\ref{CorPchapeauOmegaQC} and Example~\ref{ExQC}, we have $\widehat{P}_\Omega \cap U_\alpha = U_\Omega \cap U_\alpha = U_{\alpha,\Omega}$ since $U_\alpha$ is a subgroup of $U^+_\Delta$ by definition.
Thus
\[ \widehat{P}_\mathcal{F} \cap U_\alpha = \left( \bigcup_{\Omega \in \mathcal{F}} \widehat{P}_\Omega \right) \cap U_\alpha = \bigcup_{\Omega \in \mathcal{F}} \widehat{P}_\Omega \cap U_\alpha = \bigcup_{\Omega \in \mathcal{F}} U_{\alpha,\Omega} = U_{\alpha,\mathcal{F}}\]
and we proceed in the same way for $\alpha \in \Phi^-_\Delta$.

According to Corollary~\ref{CorPchapeauOmegaQC}, we have
\[\widehat{P}_\mathcal{F} \cap U^+_\Delta = \bigcup_{\Omega \in \mathcal{F}} \widehat{P}_\Omega \cap U^+_\Delta = \bigcup_{\Omega \in \mathcal{F}} U_{\Omega} \cap U^+_\Delta = U_{\mathcal{F}} \cap U^+_\Delta\]
and the same holds for $U^-_\Delta$.
Let $\Omega \in \mathcal{F}$.
There exists $\varepsilon > 0$ such that $\Omega \supset B(0,\varepsilon) \cap \left( x + C^v\right)$.
Since $\widehat{N}_\Omega$ is the pointwise stabilizer of $\Omega$ and, for any $\alpha \in \Phi$, the affine map $\alpha$ is non-constant on $\Omega$, there exists no $\lambda \in \Rtot$ such that $r_{\alpha,\lambda}$ fixes $\Omega$ pointwise.
Thus $\nu(\widehat{N}_\Omega)$ is trivial and, therefore, $\widehat{N}_\Omega =T_b$.
Hence
\[\widehat{P}_\mathcal{F} \cap N = \bigcup_{\Omega \in \mathcal{F}} \widehat{P}_\Omega \cap N = \bigcup_{\Omega \in \mathcal{F}} \widehat{N}_\Omega = T_b = \widehat{N}_\mathcal{F}.\]
Finally, if $p \in \widehat{P}_\mathcal{F}$, there exists $\Omega \in \mathcal{F}$ such that $p \in \widehat{P}_\Omega$.
By Corollary~\ref{CorPchapeauOmegaQC}, 
\[p \in \left(\widehat{P}_\Omega \cap U^+_\Delta \right) \left(\widehat{P}_\Omega \cap U^-_\Delta \right)\left(\widehat{P}_\Omega \cap N \right) \subset \left(\widehat{P}_\mathcal{F} \cap U^+_\Delta \right) \left(\widehat{P}_\mathcal{F} \cap U^-_\Delta \right)\left(\widehat{P}_\mathcal{F} \cap N \right).\]
Hence $\widehat{P}_\mathcal{F} = \left(\widehat{P}_\mathcal{F} \cap U^+_\Delta \right) \left(\widehat{P}_\mathcal{F} \cap U^-_\Delta \right)\left(\widehat{P}_\mathcal{F} \cap N \right)$.
\end{proof}

\begin{Cor}\label{corP=Phat}
Let $x \in \mathbb{A}_\Rtot$ and $C^v$ be a vector chamber of $V_\Rtot$.
Consider the sector $Q = x+C^v$ and let the filter $\mathcal{F} = \mathcal{F}_{x,C^v}$ be the chamber of direction $C^v$ at $x$.
Then $\widehat{P}_Q = P_Q$ and $\widehat{P}_\mathcal{F} = P_\mathcal{F}$.
\end{Cor}

The following three statements are intended to prove analogous results to \cite[7.2.6]{BruhatTits1}.

\begin{lemma}\label{lemInclusion_point_local_face}
Suppose that $\Rtot = \Rtot_\Q$.
Let $x\in \A_\Rtot$ and $F^v\subset V_\Rtot$ be a vector face.
Then $\{x\}\Subset \cl\left(\mathcal{F}_{x,F^v}\right)$.

More generally, write $F^v=F^v_\Rtot(\Delta,\Delta_P)$, where $\Delta$ is a basis of $\Phi$ and $\Delta_P\subset \Delta$. Let $\tilde{\Delta}_P$ be a set such that  $\Delta\supset \tilde{\Delta}_P\supset \Delta_P$. Set $\tilde{F}^v=F^v(\Delta,\tilde{\Delta}_P)$. Then $\cl\left(\mathcal{F}_{x,F^v}\right)=\cl(\germ_x(x+F^v))\Supset \cl(\mathcal{F}_{x,\tilde{F}^v})$. 
\end{lemma}

\begin{proof}
The same proof as in the proof of Lemma~\ref{lemEnclosure_point_in_sector}, with $F^v$ instead of $C^v$, shows  that  $\cl(\FC_{x,F^v})=\cl(\germ_x(x+F^v))$.

Let $\Omega \in \cl\left( \mathcal{F}_{x,F^v} \right)$. By definition, there exists $\varepsilon > 0$ and values $\lambda_\alpha \in \Gamma_\alpha \cup \{\infty \}$ such that $\Omega \supset \bigcap_{\alpha \in \Phi} D_{\alpha, \lambda_\alpha} \supset B(x,\varepsilon) \cap (x + F^v)$.

Let $v\in F^v_\Z(\Delta,\Delta_P)$. Let $\lambda\in \Rtot_{>0}$. Then $\lambda v\in \delta_v$. We write $\|\cdot\|$ instead of $\|\cdot\|_\Z$. Then $\|\lambda v\|=\lambda\|v\|<\epsilon$ for all $\lambda\in ]0,\frac{\epsilon}{\|v\|}[$ (with $\frac{\epsilon}{\|v\|}=\infty$ if $\|v\|=0$). Thus if $\epsilon'=\frac{\epsilon}{\|v\|}$, we have $x+\lambda v\in \Omega$ for every $\lambda\in ]0,\epsilon'[$ (by Lemma~\ref{LemHalfLineInSector}). 

Let $\alpha\in \Phi$ and $\lambda\in ]0,\epsilon'[$. Then we have  $D_{\alpha,\lambda_\alpha}\ni x+\lambda v$.  Therefore \begin{equation}\label{e_inequality_lambda}\lambda_\alpha+\alpha(x)\geq -\lambda \alpha(v), \forall \lambda\in ]0,\epsilon'[.\end{equation}

Let $\alpha\in \Phi$. If $\alpha$ is such that $\alpha(v)<0$, then \eqref{e_inequality_lambda} implies $\lambda_\alpha+\alpha(x)>0$ and in particular, \begin{equation}\label{e_inclusion_mathringD}
\mathring{D}_{\alpha,\lambda_\alpha}\ni x,\  \forall \alpha\in \Phi\mid \alpha(v)<0.
\end{equation}

Let $\alpha\in \Phi$ and assume that $\alpha(v)\geq 0$. Suppose $\mu_\alpha:=\lambda_\alpha+\alpha(x)<0$.   By setting $\lambda=\min(\epsilon',-\frac{1}{2}\frac{\mu_\alpha}{\alpha(v)})$, we have $\lambda_\alpha+\alpha(x)=\mu_\alpha\geq -\lambda \alpha(v)\geq \frac{\mu_\alpha}{2}$: a contradiction. Therefore \begin{equation}\label{e_ineq_2}
    \lambda_\alpha+\alpha(x)\geq 0, \forall \alpha\in \Phi\mid \alpha(v)\geq 0. 
\end{equation}

Let $\alpha\in \Phi$. If $\alpha(v)<0$, then by \eqref{e_inclusion_mathringD}, $D_{\alpha,\lambda_\alpha}\supset \mathring{D}_{\alpha,\lambda_\alpha}\ni x$ and thus $D_{\alpha,\lambda_\alpha}\Supset \FC_{x,\tilde{F}^v}$. Assume now $\alpha(v)\geq 0$. Let $\eta\in \{-,+\}$ be such that $\alpha\in \Phi^\eta_\Delta$. Write $\alpha=\sum_{\beta\in \Delta} n_\beta \beta$, where $(n_\beta)\in (\eta \Z_{\geq 0})^{\Delta}$. If $\alpha(v)=0$, then $n_\beta=0$ for all $\beta\in \Delta\setminus \Delta_P$. As $\tilde{\Delta}\supset \Delta_P$, we deduce $\alpha(y)=0$ for all $y\in \tilde{F}^v$, thus by \eqref{e_ineq_2}, $D_{\alpha,\lambda_\alpha}\supset x+\tilde{F}^v$ and $D_{\alpha,\lambda_\alpha}\Supset \cl(\FC_{x,\tilde{F}^v})$. Assume now $\alpha(v)>0$. Then $\eta=+$ and $\alpha(y)\geq 0$ for all $y\in \tilde{F}^v$. Therefore by \eqref{e_ineq_2}, $D_{\alpha,\lambda_\alpha}\supset x+\tilde{F}^v$ and $D_{\alpha,\lambda}\Supset \cl(\FC_{x,\tilde{F^v}})$. Thus we proved : \[D_{\alpha,\lambda_\alpha}\Supset \cl(\FC_{x,\tilde{F^v}}), \forall \alpha\in \Phi,\] which proves  that $\cl(\FC_{x,F^v})\Supset \cl(\FC_{x,\tilde{F}^v})$. 

\end{proof}

\begin{Prop}\label{PropUaGermofFacet}
Suppose that $\Rtot = \Rtot_\Q$.
Let $\alpha\in \Phi$, $x\in \A_\Rtot$ and $F^v$ be a vector face in $V_\Rtot$.
Then \[ U_{\alpha,\mathcal{F}_{x,F^v}} \subseteq \left\{\begin{aligned} &   U_{\alpha,x}\mathrm{\ if\ }\alpha\in \Phi_{F^v}^+ \sqcup \Phi_{F^v}^0\\  
&   U'_{\alpha,x}\mathrm{\ if\ }\alpha\in \Phi_{F^v}^-.\end{aligned}\right.\] 
\end{Prop}

\begin{proof}
By Lemma~\ref{lemInclusion_point_local_face},  $\{x\}\Subset \cl\left(\mathcal{F}_{x,F^v}\right)$, and thus  $U_{\alpha,\cl\left(\mathcal{F}_{x,F^v}\right)}=U_{\alpha,\mathcal{F}_{x,F^v}}\subset U_{\alpha,x}$ according to Lemma~\ref{lemUV=UclV}.

By definition, there exists a basis $\Delta$ of $\Phi$ and a subset $\Delta_P$ of $\Delta$ such that $F^v = F^v_{\Rtot}(\Delta,\Delta_P)$.
If $\Delta_P = \Delta$, then $F^v = 0$ and we have $\Phi_{F^v}^- = \emptyset$: there is nothing to prove.
Suppose now that $\Delta_P \neq \Delta$ and consider $\alpha\in \Phi_{F^v}^+$.
Let $u\in U_{-\alpha, \mathcal{F}_{x,F^v}}$.
Let $\Omega\in \mathcal{F}_{x,F^v}$ be such that $u\in U_{-\alpha,\Omega}$.
Then there exists $\varepsilon\in \Rtot_{>0}$ such that $\Omega\supset B(x,\varepsilon)\cap (x+F^v)$.
Let $v \in F^v_{\mathbb{Z}}(\Delta,\Delta_P)$.
Then $\|v\|_{\Z} \in \mathbb{Z}_{>0}$ since $0 \not\in F^v_{\mathbb{Z}}(\Delta,\Delta_P)$ by assumption on $\Delta_P$.

Let $\lambda=\frac{\varepsilon}{2\|v\|}$ and $y= x + \lambda v$. Then  $y\in B(x,\varepsilon) \cap \left( x + \delta_v\right) \subset B(x,\varepsilon) \cap \left( x + F^v\right) \subset \Omega $ by Lemma~\ref{LemHalfLineInSector} and thus  $\varphi_{-\alpha}(u)\geqslant -(-\alpha)(y)$. Moreover $\alpha(\lambda v)>0$, thus $\varphi_{-\alpha}(u)>-(-\alpha)(x)$, so that $u\in U'_{-\alpha,x}$. This proves the Lemma since $\Phi_{F^v}^- = - \Phi_{F^v}^+$.
\end{proof}

\begin{Cor}\label{CorUaGermofChamber}
Suppose that $\Rtot = \Rtot_\Q$.
Let $\alpha\in \Phi$, $x\in \A_\Rtot$ and $C^v$ be a vector chamber in $V_\Rtot$.
Then \[ U_{\alpha,\mathcal{F}_{x,C^v}} = \left\{\begin{aligned} &   U_{\alpha,x}\mathrm{\ if\ }\alpha\in \Phi_{C^v}^+\\  
&   U'_{\alpha,x}\mathrm{\ if\ }\alpha\in \Phi_{C^v}^-.\end{aligned}\right.\]
\end{Cor}

\begin{proof}
Let $\alpha \in \Phi$.
If $\alpha \in \Phi^+_{C^v}$, then we have $U_{\alpha,x} = U_{\alpha,x + C^v}=U_{\alpha,\cl(x+C^v)}$ according to Lemmas~\ref{LemUaAlcove}  and \ref{lemUV=UclV}.
As $\cl(x+C^v)\Supset \FC_{x,C^v}$, we have $U_{\alpha,\cl(x + C^v)}\subseteq U_{\alpha,\mathcal{F}_{x,C^v}}$ and $U_{\alpha,\mathcal{F}_{x,C^v}} \subseteq U_{\alpha,x}$ by Proposition~\ref{PropUaGermofFacet}.
Thus $U_{\alpha,x} =U_{\alpha,x + C^v} = U_{\alpha,\mathcal{F}_{x,C^v}}$. 

If $\alpha \in \Phi^-_{C^v}$, let $u \in U'_{\alpha,x}$. Then $\varphi_\alpha(u)+\alpha(x)>0$ and thus $\mathring{D}_{\alpha,\varphi_\alpha(u)}\ni x$. Consequently, $\mathring{D}_{\alpha,\varphi_\alpha(u)}\Supset \FC_{x,C^v}$ and $u\in U_{\alpha,\mathring{D}_{\alpha,\varphi_\alpha(u)}}\subset U_{\alpha,\FC_{x,C^v}}$. Therefore $U'_{\alpha,x}\subset U_{\alpha,\FC_{x,C^v}}$ and the converse inclusion holds by Proposition~\ref{PropUaGermofFacet}.\end{proof}

\begin{Cor}\label{CorDecompositionIwahori}
Suppose that $\Rtot = \Rtot_\Q$.
Let $x\in \A_\Rtot$ and $C$  be a vector chamber in $V_\Rtot$.
Let $\Delta = \Delta_{C}$.
For any ordering of $\left(\Phi^+_\Delta \right)_{\mathrm{nd}}$ and of $\left(\Phi^-_\Delta \right)_{\mathrm{nd}}$, we have
\[ \widehat{P}_{\mathcal{F}_{x,C}} =
\left(\prod_{\alpha \in \left(\Phi^+_\Delta \right)_{\mathrm{nd}}} U_{\alpha,x}\right)
\left( \prod_{\alpha \in \left(\Phi^-_\Delta \right)_{\mathrm{nd}}} U_{\alpha,x}'\right)
T_b
=
\left( \prod_{\alpha \in \left(\Phi^-_\Delta \right)_{\mathrm{nd}}} U_{\alpha,x}' \right)
\left(\prod_{\alpha \in \left(\Phi^+_\Delta \right)_{\mathrm{nd}}} U_{\alpha,x}\right)
T_b\]
\end{Cor}

\begin{proof}
It is an immediate consequence of Proposition~\ref{PropPchapeauGermSectorDecQC} and Corollary~\ref{CorUaGermofChamber}.
\end{proof}

\subsection{Iwasawa decomposition}

Thus, we obtain the Iwasava decomposition whose proof can be conducted in the same way as in \cite[7.3.1]{BruhatTits1}:

\begin{Thm}[Iwasawa decomposition]\label{thmIwasawa}\index{Iwasawa decomposition}\index{decomposition!Iwasawa}
Suppose that the totally ordered abelian group $\Rtot$ is equipped with a $\Q$-module structure. 
Let $C$, $C'$ be two vector chambers of $V_\Rtot$ and $x \in \mathbb{A}_\Rtot$.
Let $\mathcal{F} = \mathcal{F}_{x,C}$ be the face of $C$ at $x$. 
Then 
\[ G = U_{C'}  N \widehat{P}_{\mathcal{F}} \]
and there is a natural one-to-one correspondence
\[ \Wext = N / T_b \to U_{C'} \backslash G / \widehat{P}_{\mathcal{F}}. \]
\end{Thm}

\begin{proof}
In this proof, we denote by $U^+ = U_{C'}$, by $B = \widehat{P}_\mathcal{F}$ and by $Z = U^+ N B$.
Set $\Delta = \Delta_C$ and $\Delta' = \Delta_{C'}$, and let $\Phi^+_{\Delta}$ and $\Phi^+_{\Delta'}$ be the associated subsets of positive roots.

\paragraph{First step: rank-one Levi subroups are contained in $Z$.} Let $\alpha \in \Phi$ be any root.
Let $L_\alpha = \langle U_{-\alpha}, U_\alpha, T \rangle$ and $m_\alpha = m(u)$ for some $u \in U_{\alpha} \setminus \{1\}$.
By \cite[6.1.2(5)]{BruhatTits1}, we recall that $M_\alpha' = T \cdot \{1,m_\alpha\}$ is a group.
By \cite[6.1.2(7)]{BruhatTits1}, we know that
\[L_\alpha = U_\alpha T \cup U_\alpha T m_\alpha U_\alpha.\]
Let $B_\alpha = \langle U_{\alpha,\mathcal{F}}, U_{-\alpha,\mathcal{F}}, T_b \rangle$.
Consider the set $Z_\alpha = U_\alpha M_\alpha' B_\alpha$.
We want to prove that $L_\alpha \subset Z_\alpha$.
The inclusion $U_\alpha T \subset Z_\alpha$ is obvious.
Hence, it suffices to prove that $m_\alpha u \in Z_\alpha$ for any $u \in U_\alpha$ since $U_\alpha T$ is a group (indeed, the subgroup $T$ normalizes $U_\alpha$ by definition of $U_\alpha$).

Suppose that $\alpha \in \Phi_\Delta^+$ (resp. $\alpha \in \Phi_\Delta^-$).
If $\varphi_\alpha(u) \geqslant -\alpha(x)$ (resp. $\varphi_\alpha(u) > -\alpha(x)$), then $u \in U_{\alpha,-\alpha(x)} = U_{\alpha,\mathcal{F}}$ (resp. $u \in U'_{\alpha,-\alpha(x)} = U_{\alpha,\mathcal{F}}$) by Corollary~\ref{CorUaGermofChamber}.
Thus $u \in B_\alpha \subset Z_\alpha$.
Otherwise, $\varphi_\alpha(u) < -\alpha(x)$ (resp. $\varphi_\alpha(u) \leqslant -\alpha(x)$) . Write $u = v'mv''$ for some $m \in M_\alpha$ and $v',v'' \in U_{-\alpha}$.
Write $m = t m_\alpha$ for $t \in T$.
By axiom~\ref{axiomV5bis}, we know that $\varphi_{-\alpha}(v'') = - \varphi_\alpha(u)$ so that $\varphi_{-\alpha}(v'') > - (-\alpha(x))$ (resp. $\varphi_{-\alpha}(v'') \geqslant - (-\alpha(x))$).
Thus $v'' \in U'_{-\alpha, x} = U_{-\alpha,\mathcal{F}}$ (resp. $v'' \in U_{-\alpha,x} = U_{-\alpha,\mathcal{F}}$) by Corollary~\ref{CorUaGermofChamber} since $-\alpha \in - \Phi_\Delta^+ = \Phi_\Delta^-$ (resp. $-\alpha \in \Phi_\Delta^+ = - \Phi_\Delta^-$).
Hence $m_\alpha u = m_\alpha v' t m_\alpha^{-1} v'' = \underbrace{(m_\alpha v' m_\alpha^{-1})}_{\in U_\alpha} \underbrace{(m_\alpha t m_\alpha^{-1})}_{\in T} v'' \in Z_\alpha$.

\paragraph{Second step: $Z$ is stable by left multiplication by root groups of simple roots with respect to $-C'$.}
Let $\alpha \in \Delta'$ and $V_\alpha = \langle U_\beta\ | \ \beta \in \Phi_{\Delta'}^+ \setminus \{\alpha\} \rangle$.
We recall that $V_\alpha$ is normalized by $U_\alpha$ and $U_{-\alpha}$ by axiom~\ref{axiomRGD2} and that $U^+_{\Delta'} = V_\alpha U_\alpha = U_\alpha V_\alpha$ by \cite[6.1.6]{BruhatTits1}.
Hence
\[ U_{-\alpha} Z = U_{-\alpha} V_\alpha U_{\alpha} N B = V_\alpha U_{-\alpha} U_{\alpha} N B \subset V_\alpha L_\alpha N B.\]
But since $L_\alpha \subset Z_\alpha$, we have 
\[ U_{-\alpha} Z \subset V_\alpha U_\alpha T \{1,m_\alpha\} B_\alpha N B.\]
On the one hand, since $B_\alpha \subset U_\alpha U_{-\alpha} N$ by Remark~\ref{RkDecQCfiltres}, we get
\begin{align*} 
V_\alpha U_\alpha T \{1\} B_\alpha N B \subset &
 U^+ T U_\alpha U_{-\alpha} N B\\
= & U^+ T U_{-\alpha} N B
\end{align*}
On the other hand, since $B_\alpha \subset U_{-\alpha} U_{\alpha} N$ by Remark~\ref{RkDecQCfiltres}, we get
\begin{align*}
 V_\alpha U_\alpha T \{m_\alpha\} B_\alpha N B \subset
 & U^+ T m_\alpha U_{-\alpha} U_{\alpha} N B \\
 = & U^+ T m_\alpha U_{-\alpha} m_\alpha^{-1} m_\alpha U_{\alpha} m_\alpha^{-1} N B\\
 = & U^+ T U_{\alpha} U_{-\alpha}N B\\
 = & U^+ T U_{-\alpha} N B
\end{align*}
Hence $U_{-\alpha} Z \subset U^+ T U_{-\alpha} N B$.

To conclude, we want to show that $U_{-\alpha} N \subset Z$.
Consider any $u \in U_{-\alpha}$ and any $n \in N$.
Let $v = n^{-1}un$.
Then $v \in U_\beta$ with $\beta = {^v\!}\nu(n^{-1})(-\alpha)$.
>From the first step applied to $\beta$, we get that since $v \in L_{-\beta}$:
\begin{align*}
un = nv \in & nU_{-\beta} T \{1,m_\beta\} B_{-\beta} \\
& \subset n U_{-\beta} n^{-1} n N B\\
& = U_{\alpha} N B \\
& \subset Z
\end{align*} 

Hence $U_{-\alpha} Z \subset U^+ T Z B = U^+ T U^+ N B B = U^+ N B = Z$.

\paragraph{Third step: $Z$ contains a generating subset of $G$.}
The set $Z$ is stable by left multiplication by $U_{-\alpha}$ for $\alpha \in \Delta_{C'}$, by $U_\alpha$ for $\alpha \in \Phi_{\Delta'}^+$ and by $T$.
Since these groups generate $G$, we deduce that $G = Z$.
Indeed, let $H$ the group generated by the $U_{\alpha}$ for $\alpha \in \Phi^+_{\Delta'}$, by $U_{-\alpha}$ for $\alpha \in \Delta'$ and by $T$.
For any $\alpha \in \Delta'$, the element $m_\alpha$ belongs to $L_\alpha$ and since $W^v$ is a Coxeter group generated by the ${^v}\nu(m_\alpha)$ for $\alpha \in \Delta'$, the group $N$ is contained in $H$.
Thus for any $\beta \in \Phi^-_{\Delta'}$, there exist a root $\alpha \in \Phi^+_{\Delta'}$ and an element $n \in N$ such that the root group $U_\beta = n U_\alpha n^{-1}$ is contained in $H$.
Since the root group datum is generating, we deduce that $H = G$.

\paragraph{Fourth step: determination of double cosets}
From the equality $G = Z = U_{\Delta'}^+ N \widehat{P}_\mathcal{F}$, we deduce a natural surjective map:
\[
\begin{array}{cccc}
\psi: &N & \to & U_{\Delta'}^+ \backslash G / \widehat{P}_{\mathcal{F}}\\
&n & \mapsto & U_{\Delta'}^+ n \widehat{P}_{\mathcal{F}} 
\end{array}
\]
Let $n,n' \in N$ such that $n'\in U_{\Delta'}^+ n \widehat{P}_{\mathcal{F}}$.
Denote $C'' = n^{-1} \cdot C'$ another vector chamber.
Then 
\begin{align*}
(n')^{-1} n \in & \widehat{P}_{\mathcal{F}} n^{-1} U_{\Delta'}^+ n & \\
& = \widehat{P}_{\mathcal{F}} U_{n^{-1} \cdot C'}^+ & \text{ by \cite[6.1.2(10)]{BruhatTits1}}\\
& = \widehat{P}_{\mathcal{F}} U_{\Delta''}^+ & \\
& = T_b \left( U_{\Delta''}^- \cap \widehat{P}_{\mathcal{F}} \right) U_{\Delta''}^+ & \text{ by Proposition~\ref{PropPchapeauGermSectorDecQC}}\\
& = \left( U_{\Delta''}^- \cap \widehat{P}_{\mathcal{F}} \right) T_b U_{\Delta''}^+
\end{align*}
Hence, by \cite[6.1.15(c)]{BruhatTits1}, the element $(n')^{-1}n$ corresponds to a double coset in $U_{\Delta''}^- \backslash G / U_{\Delta''}^+$ arising from some element in $T_b$ and therefore $(n')^{-1} n \in T_b$.
Hence, $\psi$ induces a bijection $N / T_b \to U_{\Delta'}^+ \backslash G / \widehat{P}_{\mathcal{F}}$.
\end{proof}

\subsection{Bruhat decomposition}

In the following theorem, we generalize  the affine Bruhat decomposition \cite[7.3.4]{BruhatTits1}. In the particular case where $G=\mathbf{G}(\mathbb{K})$ for a split reductive group $\mathbf{G}$ over a $2$-local field $\mathbb{K}$, we recover a result of Kapranov (see \cite[Proposition (1.2.3)]{kapranov2001double}).

\begin{Thm}\label{ThmBruhatDecomposition}\index{Bruhat decomposition}\index{decomposition!Bruhat}
Suppose that the totally ordered abelian group $\Rtot$ is equipped with a $\Q$-module structure. 
Let $C,C'$ be two vector chambers of $V_\Rtot$ and $x,x' \in \mathbb{A}_\Rtot$ be two points.
Then
\[ G = \widehat{P}_{\mathcal{F}_{x,C}} N \widehat{P}_{\mathcal{F}_{x',C'}} \]
and there is a natural one-to-one correspondence
\[ N / T_b \to \widehat{P}_{\mathcal{F}_{x,C}} \backslash G / \widehat{P}_{\mathcal{F}_{x',C'}} \]
\end{Thm}

\begin{Not}\label{NotGoodSector}
Let $C,C'$ be two vector chambers of $V_\Rtot$ and $x,x' \in \mathbb{A}_\Rtot$ be two points.
We denote by $Q = x+C$ and $Q' = x'+C'$ the corresponding sectors in $\mathbb{A}_\Rtot$.
Denote by $\Delta = \Delta_C$ and $\Delta' = \Delta_{C'}$ respectively the bases of $\Phi$ defining $C = C^v_{\Rtot,\Delta}$ and $C' = C^v_{\Rtot,\Delta'}$.

According to Proposition~\ref{PropGoodSector}, there is a unique $w \in W(\Phi)$ such that $x'-x \in \overline{C}^v_{w(\Delta)}$ and $C^v_{w(\Delta)} \cap (x'-x + C') \neq \emptyset$.
We denote it by $w(x,C,x',C') = w$.
We also denote by $\ell(x,C,x',C')$ the length of $w$ in the Coxeter system $\left( W(\Phi), (r_\alpha)_{\alpha \in \Delta} \right)$.

We denote by $\Delta'' = w(\Delta)$, by $C'' = C^v_{w(\Delta)}$ and by $Q'' = x + C''$ so that $x' \in \overline{Q''}$ and $Q' \cap Q'' \neq \emptyset$.

Note that the uniqueness of $w \in W(\Phi)$ and the simple transitivity of the action of $W(\Phi)$ imply that $\Delta''$, and therefore $C''$, is uniquely determined by $x,x',C'$.
We denote it by $C(x,x',C') = C''$.
\end{Not}

\begin{lemma}\label{l_neighbourhood_chambers}
    Let $x,x'\in \A_\Rtot$ and $C,C'$ be two vector chambers of $\A_\Rtot$. Let $C''=C(x,x',C')$. Then there exists a neighbourhood $\Omega$ of $x'$ in $\A_\Rtot$ such that $\Omega\cap (x'+C')\subset x+C''$. In other words, $x+C''\in \germ_{x'}(x'+C')$. \end{lemma}
    
\begin{proof}
   Let $\mathscr{C}_{x'-x}$ be the set of vector chamber dominating $x'-x$ (i.e containing $x'-x$ in their closures). Let $\tilde{C}\in \mathscr{C}_{x'-x}$. Let us prove that there exists a neighbourhood $\Omega_{\tilde{C}}$ of $x'-x$ in $V_\Rtot$ such that $\overline{\tilde{C}}\cap \Omega_{\tilde {C}}\cap ((x'-x)+C')=\emptyset$, unless $\tilde{C}=C''$. 

 Set $\Delta''=\Delta_{C''}$ and set $\Delta''_{x'-x}=\{\alpha\in \Delta''\mid \alpha(x'-x)=0\}$. Let $\tilde{\Delta}=\Delta_{\tilde{C}}$  and $\tilde{\Delta}_{x'-x}=\{\alpha\in \tilde{\Delta}\mid \alpha(x'-x)=0\}$. If $\tilde{\Delta}_{x'-x}$ is non-empty, take $\beta\in \tilde{\Delta}_{x'-x}$. We can write $\beta=\sum_{\alpha\in \Delta''} n_\alpha \alpha$, where $(n_\alpha)\in ( \epsilon\Z_{\geq 0})^{\Delta''}$ for some $\epsilon\in \{-1,1\}$. Then $\beta(x'-x)=\sum_{\alpha\in \Delta''\setminus \Delta''_{x'-x}} n_\alpha \alpha(x'-x)=0$ and thus $n_\alpha=0$ for all $\alpha\in \Delta''\setminus \Delta''_{x'-x}$. Consequently, \begin{equation}\label{e_inclusions_Delta_x'x}
  \tilde{\Delta}_{x'-x}\subset \bigoplus_{\alpha\in \Delta''_{x'-x}}\epsilon \Z_{\geq 0} \alpha.
 \end{equation}
    
    If $\Delta''_{x'-x}$ is non-empty, take $\alpha\in \Delta''_{x'-x}$. By definition of $C''$, there exists $z\in (x+C'')\cap (x'+C')$. Then $\alpha(z)>\alpha(x)=\alpha(x')$ and as the sign of $\alpha$ is constant on $C'$ (by Lemma~\ref{lemSign_of_a_vectorial_face}), we have $\alpha(C')>0.$ Therefore \begin{equation}\label{e_sign_C'}\alpha(C')>0,\  \forall\alpha\in \Delta''_{x'-x}.\end{equation}

 For $\alpha\in \tilde{\Delta}\setminus \tilde{\Delta}_{x'-x}$, we have $\alpha(x'-x)> 0$.   Let \[\Omega_{\tilde{C}}=\{y\in \A_\Rtot\mid \alpha(y)>0,\forall \alpha\in \tilde{\Delta}\setminus \tilde{\Delta}_{x'-x}\}.\] This is an open subset of $\A_\Rtot$ containing $x'-x$. Assume the existence of $y\in (x+\Omega_{\tilde{C}})\cap (x'+C')\cap  (x +\overline{\tilde{C}})$. Then $y-x\in \Omega_{\tilde{C}}\cap (x'-x+C')\cap \overline{\tilde{C}}$. Let $\alpha\in \tilde{\Delta}$. If $\alpha\notin \tilde{\Delta}_{x'-x}$, we have $\alpha(y-x)>0$, by  choice of $\Omega_{\tilde{C}}$. Assume now $\alpha\in \tilde{\Delta}_{x'-x}$. Then by \eqref{e_inclusions_Delta_x'x}, there exists $\epsilon\in \{-1,1\}$ and $(n_\beta)\in (\epsilon\Z_{\geq 0})^{\Delta''_{x'-x}}$ such that $\alpha=\sum_{\beta\in \Delta''_{x'-x}} n_\beta \beta$. We thus have 
$$\alpha(y-x)=\alpha(y-x')=\sum_{\beta\in \Delta''_{x'-x}} n_\beta \beta(y-x')\in \epsilon \Rtot_{>0}$$
since $\beta(C')>0$ for $\beta \in \Delta''_{x'-x}$.
 
 Moreover, $\alpha(\tilde{C})\geq 0$, and thus $\alpha(y-x)\in \Rtot_{>0}$. Therefore we proved:  \[\alpha(y-x)>0, \forall \alpha\in \tilde{\Delta}.\]  Therefore $y-x\in \tilde{C}$ and $y\in (x'+C')\cap (x+\tilde{C})$. By definition of $C''$, we necessarily have $\tilde{C}=C''$. Consequently:\begin{equation}\label{e_intersection_neighbourhoods}(x+\Omega_{\tilde{C}})\cap (x'+C')\cap  (x +\overline{\tilde{C}})=\emptyset,\forall \tilde{C}\in \mathscr{C}_{x'-x}\setminus \{C''\} 
 \end{equation}

and \begin{equation}\label{e_inclusion_Omega}
(x+\Omega_{C''})\cap (x'+C')\cap (x+\overline{C''})\subset x+C''.
\end{equation}
   
   Let now $\Omega''=x+\bigcap_{\tilde{C}\in \mathscr{C}_{x'-x}}\Omega_{\tilde{C}}$.  Let $\Omega'=x+\bigcup_{\tilde{C}\in \mathscr{C}_{x'-x}} \overline{\tilde{C}}$.  By Lemma~\ref{l_neighbourhood}, $\Omega'$ is a neighbourhood of $x'$ in $\A_\Rtot$. Set $\Omega=\Omega'\cap \Omega''$. 
   
   Let $y\in \Omega\cap (x'+C')$. Let $\tilde{C}\in \mathscr{C}_{x'-x}$ be such that $y\in x+\overline{\tilde{C}}$.  Then by \eqref{e_intersection_neighbourhoods} and \eqref{e_inclusion_Omega},  $\tilde{C}=C''$ and =$y\in (x+\Omega_{C''})\cap (x'+C')\cap (x+C'').$  Therefore $\Omega\cap (x'+C')\subset x+C''$.\end{proof}

\begin{Lem}\label{LemActionWeylIwahoriFilters}
Let $x \in \mathbb{A}_\Rtot$ and let $C$ be a vector chamber of $V_\Rtot$.
Then for any $n \in N$, we have $n \widehat{P}_{\mathcal{F}_{x,C}} n^{-1} = \widehat{P}_{\mathcal{F}_{\nu(n)(x), {^v\!}\nu( n )(C) }}$.
\end{Lem}

\begin{proof}
We have $n \widehat{P}_{\mathcal{F}_{x,C}} n^{-1} = \bigcup_{\Omega \in \mathcal{F}_{x,C}} n \widehat{P}_\Omega n^{-1}$.
By~\ref{CorConjugationParahorics}, we have $n \widehat{P}_\Omega n^{-1} = \widehat{P}_{\nu(n)(\Omega)}$.
Moreover, by definition:  $$\left\{ \nu(n)(\Omega),\ \Omega \in \mathcal{F}_{x,C} \right\} = \mathcal{F}_{\nu(n)(x),{^v\!}\nu(n)(C)}$$ since $\nu(n)(x+C) = \nu(n)(x) + {^v\!}\nu(n)(C)$.
Thus: $$n \widehat{P}_{\mathcal{F}_{x,C}} n^{-1} = \bigcup_{\Omega \in \mathcal{F}_{\nu(n)(x),{^v\!}\nu(n)(C)}} \widehat{P}_\Omega = \widehat{P}_{\mathcal{F}_{\nu(n)(x), {^v\!}\nu( n )(C) }}.$$
\end{proof}

We follow the proof given by Bruhat and Tits. It relies on the following technical Lemma from \cite[7.3.6]{BruhatTits1}. In our context, the notion of a half-line is more delicate but it has been introduced in Definition~\ref{DefHalfLine} so that for a well-chosen $v \in V_\Z$, given by Lemma~\ref{LemWellChosen}, one can follow word by word the proof given by Bruhat and Tits.

\begin{Lem}\label{LemTechnicalBruhat}
Suppose that $\Rtot$ is equipped with a $\Q$-module structure.
Let $C,C'$ be two vector chambers of $V_\Rtot$ and $x,x' \in \mathbb{A}_\Rtot$ be two points.
Let $g \in G$ and $n \in N$ be such that $g \in \widehat{P}_{\mathcal{F}_{x,C}} n \widehat{P}_{\mathcal{F}_{x',C'}}$.
Let $\Delta = \Delta_C$ the basis of the vector chamber $C$.
Let $v \in C^v_{\mathbb{Z},\Delta}$.

\begin{enumerate}[label=(\arabic*)]
\item Then we have either:\label{technicalBruhatStep1}
\begin{enumerate}[label=(\alph*)]
\item $g \in \widehat{P}_{\mathcal{F}_{z,C}} n \widehat{P}_{\mathcal{F}_{x',C'}}$ for any $z \in x + \delta_v$ or\label{technicalBruhatCond1}
\item there exist a value $\lambda \in \Rtot_{>0}$ and an element $n' \in N$ such that:\label{technicalBruhatCond2}
\begin{enumerate}
\item $g \in \widehat{P}_{\mathcal{F}_{z,C}} n \widehat{P}_{\mathcal{F}_{x',C'}}$ for any $z = x + \mu v$ with $\mu \in ]0,\lambda[$,\label{technicalBruhatEq1}
\item $g \in \widehat{P}_{\mathcal{F}_{y,C}} n' \widehat{P}_{\mathcal{F}_{x',C'}}$ with $y = x + \lambda v$,\label{technicalBruhatEq2}
\item $\ell(y,C,n'\cdot x',n' \cdot C') > \ell(x,C,n\cdot x',n\cdot C')$.\label{technicalBruhatEq3}
\end{enumerate}
\end{enumerate}
\item Moreover, we have $g \in \widehat{P}_{\mathcal{F}_{z,C}} N \widehat{P}_{\mathcal{F}_{x',C'}}$ for any $z \in x +\delta_v$.\label{technicalBruhatStep2}
\end{enumerate}
\end{Lem}

\begin{proof}[Proof of~\ref{technicalBruhatStep1}]
Recall that $\Delta = \Delta_C$ and by $\Delta' = \Delta_{C'}$. 

\paragraph{Reduction step:} If $g \in \widehat{P}_{\mathcal{F}_{x,C}} n \widehat{P}_{\mathcal{F}_{x',C'}}$, then $gn^{-1} \in \widehat{P}_{\mathcal{F}_{x,C}} \widehat{P}_{\mathcal{F}_{n \cdot x',n \cdot C'}}$ by Lemma~\ref{LemActionWeylIwahoriFilters}.
Thus, up to replace $g$, $x'$ and $C'$ by $gn^{-1}$, $n\cdot x'$ and $ n\cdot C'$ respectively, we can assume that $n=1$.

Let $C'' = C(x,x',C')$, $\Delta'' = \Delta_{C''}$ and $Q,Q',Q''$ as in Notation~\ref{NotGoodSector}.
Denote $\Psi^- = \left(\Phi_{\Delta}^-\right)_{\mathrm{nd}} \cap \Phi_{\Delta''}^-$ and $\Psi^+ = \left(\Phi_{\Delta}^-\right)_{\mathrm{nd}} \cap \Phi_{\Delta''}^+$.
Let $b \in \widehat{P}_{\mathcal{F}_{x,C}}$.
By Corollary~\ref{CorDecompositionIwahori}, we can write $b$ as a product:
\[
b = \left( \prod_{\alpha\in \left(\Phi_{\Delta}^+\right)_{\mathrm{nd}}} u_\alpha \right)\left( \prod_{\alpha \in \Psi^-} u_\alpha \right)\left( \prod_{\alpha \in \Psi^+} u_\alpha \right) t\]
where $u_\alpha \in U_{\alpha,x}$ for every $\alpha \in \left(\Phi_{\Delta}^+\right)_{\mathrm{nd}}$ and $u_\alpha \in U'_{\alpha,x}$ for every $\alpha \in \left(\Phi_{\Delta}^-\right)_{\mathrm{nd}}$ and $t \in T_b$.

For $\alpha \in \Phi_{\Delta}^+$ and $z \in x + \delta_v$, write $z = x + \lambda v$ with $\lambda \in \Rtot_{>0}$.
Then, we have $\alpha(z) = \alpha(x+\lambda v) = \alpha(x) + \lambda \alpha(v) > \alpha(x)$ where the last inequality follows from the assumptions on $v$.
Thus
\begin{align*} u_\alpha \in& U_{\alpha,\mathcal{F}_{x,C}} = U_{\alpha,x} & \text{ by Corollary }~\ref{CorUaGermofChamber} \text{ since } \alpha \in \Phi_{\Delta}^+\\
& = \varphi_\alpha^{-1}([-\alpha(x),+\infty]) \subset \varphi_\alpha^{-1}([-\alpha(z),+\infty]) & \text{ since } -\alpha(z) < -\alpha(x)\\
& = U_{\alpha,z} = U_{\alpha,\mathcal{F}_{z,C}}& \text{ by Corollary }~\ref{CorUaGermofChamber} \text{ since } \alpha \in \Phi_{\Delta}^+\\
& \subset \widehat{P}_{\mathcal{F}_{z,C}}&
\end{align*}
On the other side, we have $u_\alpha \in U_{\alpha,\mathcal{F}_{x,C}} \subset U_{\alpha,x}$ since $U_{\alpha,\mathcal{F}_{x,C}}$ is either equal to $U_{\alpha,x}$ or $U'_{\alpha,x}$ by Corollary~\ref{CorUaGermofChamber} and $U'_{\alpha,x} \subset U_{\alpha,x}$.
By definition of $C''$, we know that $x' \in x + \overline{C''} = \overline{Q''}$ and there is some element $y \in Q' \cap Q''$.
Write $y=x' + v' \in x + C''$ with $v' \in C'$.
For any $\alpha \in \Phi_{\Delta''}^+$, we have $\alpha(x') - \alpha(x)\geqslant 0$ since $x'-x \in \overline{C''}$ and $\alpha(y-x) = \alpha(x') + \alpha(v') - \alpha(x) > 0$ since $(x'-x) + v' \in C''$.
Hence, if $\alpha \in \Phi_{\Delta'}^+$, then 
\[u_\alpha \in U_{\alpha,x} = \varphi_\alpha^{-1}([-\alpha(x),+\infty]) \subset \varphi_\alpha^{-1}([-\alpha(x'),+\infty]) = U_{\alpha,x'} = U_{\alpha,\mathcal{F}_{x',C'}}\]
because $U_{\alpha,\mathcal{F}_{x',C'}} = U_{\alpha,x'}$ for $\alpha \in \Phi_{\Delta'}^+$ by Corollary~\ref{CorUaGermofChamber}.
Otherwise, $\alpha \in \Phi_{\Delta'}^-$ and $\alpha(x') - \alpha(x) > -\alpha(v') > 0$.
Hence
\[u_\alpha \in U_{\alpha,x} = \varphi_\alpha^{-1}([-\alpha(x),+\infty]) \subset \varphi_\alpha^{-1}(]-\alpha(x'),+\infty]) = U'_{\alpha,x'} = U_{\alpha,\mathcal{F}_{x',C'}}\]
because $U_{\alpha,\mathcal{F}_{x',C'}} = U'_{\alpha,x'}$ for $\alpha \in \Phi_{\Delta'}^-$ by Corollary~\ref{CorUaGermofChamber}.
Thus $u_\alpha \in \widehat{P}_{\mathcal{F}_{x',C'}}$ for every $\alpha \in \Phi_{\Delta''}^+$ and, in particular, for every $\alpha \in \Psi^+$.
Therefore, we have \[u_+ = \left( \prod_{\alpha\in \left(\Phi_{\Delta}^+\right)_{\mathrm{nd}}} u_\alpha \right) \in \bigcap_{z \in x + \delta_v} \widehat{P}_{\mathcal{F}_{z,C}}
\qquad \text{ and } \qquad
\left( \prod_{\alpha \in \Psi^+} u_\alpha \right)t \in \widehat{P}_{\mathcal{F}_{x',C'}}.\]

Hence, up to replace $g$ by $u_+^{-1}g$, we can assume that $g \in u \widehat{P}_{\mathcal{F}_{x',C'}}$ for some $u = \prod_{\alpha \in \Psi^-} u_\alpha$ with $u_\alpha \in U_{\alpha,\mathcal{F}_{x,C}} = U'_{\alpha,x}$ for $\alpha \in \Psi^-$.

{\bf{} Decomposition step:} 
If $u = 1$, then the condition~\ref{technicalBruhatCond1} is satisfied.
In particular, if $C'' = -C$, then $\ell(x,C,x',C')$ is maximal by definition of the length in $W(\Phi)$ with respect to $\Delta_D$ and, in this case, $u=1$ since $\Psi^- = \left(\Phi_{\Delta}^-\right)_{\mathrm{nd}} \cap \Phi_{\Delta''}^- = \left(\Phi_{\Delta}^-\right)_{\mathrm{nd}} \cap \Phi_{\Delta}^+ = \emptyset$.

Suppose that $u \neq 1$ (and therefore $C'' \neq - C$).
By assumption on $v$ and $\Rtot$,
for $\alpha \in \Psi^-$, we can define $\lambda_\alpha = \frac{- \varphi_\alpha(u_\alpha) - \alpha(x)}{\alpha(v)}$ since $\alpha(v) \in \mathbb{Z}_{<0}$.
Because $u_\alpha \in U'_{\alpha,x} = \varphi_\alpha^{-1}(]-\alpha(x),+\infty])$, we have that $- \varphi_\alpha(u_\alpha) - \alpha(x) < 0$.
Hence $\lambda_\alpha \in \Rtot$ is positive.
By finiteness of $\Psi^-$, we can find $\beta \in \Psi^-$ such that $\lambda := \lambda_\beta \in \Rtot_{>0}$ is maximal.
If we set $y = x + \lambda v \in x + \delta_v$, we have that for any $\alpha \in \Psi^-$:
\[ \alpha(y) = \alpha(x) + \lambda \alpha(v) \geqslant \alpha(x) + \lambda_\alpha \alpha(v) = -\varphi_\alpha(u_\alpha)\]
with equality at least for $\alpha=\beta$.

We firstly prove~\ref{technicalBruhatEq2}.
We have shown that $u_\beta \in U_{\beta,y} \setminus U_{\beta,y}'$ and that $u_\alpha \in U_{\alpha,y}$ for any $\alpha \in \Psi^-$.
Thus $u \in U_y$.
Since $U'_{y} \cap U_\beta = U'_{\beta,y}$ according to Example~\ref{ExQC}, we have $u \in U_{y} \setminus U'_{y}$.
Since $U_{\beta,y} \neq U'_{\beta,y}$, we know by Lemma~\ref{LemReducedRootSystem} that either $\beta \in \Phi_y$ or $2\beta \in \Phi_y$ so that, in particular, the image $\overline{u_\beta}$ of $u_\beta$ in $\overline{G}_y$ is non-trivial.
Let $C''' = C(y,x',C')$ and $\Delta''' = \Delta_{C'''}$.
Thus, by Theorem~\ref{ThmLocalRootSystem} and by the spherical Bruhat decomposition \cite[6.1.15]{BruhatTits1}, we know that there are $\overline{u'} \in \prod_{\alpha \in \Delta} \overline{U_\alpha}$, $\overline{u''} \in \prod_{\alpha \in \Delta'''} \overline{U_\alpha}$ and $\overline{n'} \in \overline{N_y}$ such that $\overline{u} = \overline{u'} \overline{n'} \overline{u''}$.
According to Proposition~\ref{PropPintersectionSector}, one can lift those elements in elements $u' \in U_{y+C} = \prod_{\alpha \in \Delta} U_{\alpha,y}$, $u'' \in U_{y+C'''} = \prod_{\alpha \in \Delta'''} U_{\alpha,y}$ and $n' \in N_y$ such that there exists an element $v \in U^*_{y} = U'_{y}$ (indeed, $y$ is a point so that $\Phi^*_y = \Phi$ by definition) satisfying $u = v u' n' u''$.
Since $v \in U'_{y} \subset \widehat{P}_{\mathcal{F}_{y,C}}$, $u' \in U_{y+C} \subset \widehat{P}_{\mathcal{F}_{y,C}}$ and $u''\in U_{y+C'''} \subset \widehat{P}_{\mathcal{F}_{x',C'}}$ since $\mathcal{F}_{x',C'} \cap (y + C''') \neq \emptyset$ by construction of $C''' = C(y,x',C')$, we get that $u \in \widehat{P}_{\mathcal{F}_{y,C}} n' \widehat{P}_{\mathcal{F}_{x',C'}}$.

We secondly prove~\ref{technicalBruhatEq1}.
Consider any value $\mu \in ]0,\lambda[$ and denote by $z = x + \mu v$.
For any $\alpha \in \Psi^-$, since $\alpha(v) \in \mathbb{Z}_{<0}$, we have that $\alpha(z) = \alpha(x) + \mu \alpha(v) > \alpha(x) + \lambda \alpha(v) \geqslant -\varphi_\alpha(u_\alpha)$.
Hence $u_\alpha \in U'_{\alpha,z} \subset U_{\mathcal{F}_{z,C}}$.
Hence $g \in \widehat{P}_{\mathcal{F}_{z,C}} \widehat{P}_{\mathcal{F}_{x',C'}}$.

It remains to prove~\ref{technicalBruhatEq3}
with the $n' \in N_y$ and the $\lambda \in \Rtot$ introduced in the proof of~\ref{technicalBruhatEq2}.
We will prove successively the inequalities $\ell(y,C,n' \cdot x', n' \cdot C') \geqslant \ell(y,C,x',C') \geqslant \ell(x,C,x',C)$ and we will check that one of them is strict, depending on whether $n' \in T_b$ or not.

We firstly prove that $\ell(y,C, x',C') \geqslant \ell(x,C,x',C')$.
Let $\mathcal{H}$ be the set of hyperplanes in $V_\Rtot$ that are kernels of elements in $\Phi$.
We know that the length $\ell(w)$ of an element $w \in W(\Phi)$ is equal to the number of hyperplanes in $\mathcal{H}$ separating $C^v_{\mathbb{R},\Delta}$ and $C^v_{\mathbb{R},w(\Delta)}$ \cite[chap. VI § 1, no. 6]{bourbaki1981elements}
and that two vector chambers of $V_\mathbb{R}$ intersect if and only the corresponding vector chamber in $V_\Rtot$ intersect, according to Lemma~\ref{LemIntersectionChambers}.
Hence, for $z \in \mathbb{A}_\Rtot$, the number $\ell(z,C,x',C')$ is the cardinality of the set $\mathcal{H}_z$ of elements $H \in \mathcal{H}$ such that $z+H$ separates $z + C$ and a neighbourhood of $x$ in $x+\overline{C'}$.
But, for $H \in \mathcal{H}$, the set of elements $z \in \mathbb{A}_\Rtot$ such that $H \in \mathcal{H}_z$ is either the open half-space $x'+H+C$ or the closed half-space $x'+H+\overline{C}$. Its intersection with the open half-line $x+\delta_v$ contains an open half-line.
Hence $\mathcal{H}_y \supset \mathcal{H}_x$ so that $\ell(y,C,x',C') \geqslant \ell(x,C,x',C')$.
More precisely, $w'=w(y,C,x',C')$ is longer than $w=w(x,C,x',C')$ in the sense that there exists $w'' \in W(\Phi)$ such that $w' = ww''$ with $\ell(w') = \ell(w) + \ell(w'')$.

We now prove that $\ell(y,C, n' \cdot x',n' \cdot C') =\ell(y,C,x',C') > \ell(x,C,x',C')$ whenever $n' \in T_b$.
Indeed, suppose by contradiction that it is an equality.
Then $w' = w$ and therefore $C''' = C''$.
But since the image $\overline{u}\in \overline{G}_y$ of $u$ is in the subgroup $ \overline{U_\Delta^-} \cap \overline{U_{\Delta''}^-}$ generated by the root groups $\overline{U_\alpha}$ for $\alpha \in \Phi_y \cap \Phi_\Delta^- \cap \Phi_{\Delta''}^-$, whereas the image $\overline{u'} \in \overline{G}_y$ (resp. $\overline{u''} \in \overline{G}_y$) of $u'$ (resp. $u''$) is contained in the group $\overline{U_\Delta^+}$ (resp. $\overline{U_{\Delta'''}^+}$) generated by the root groups $\overline{U_\alpha}$ for $\alpha \in \Phi_y \cap \Phi_\Delta^+$ (resp. $\Phi_y \cap \Phi_{\Delta'''}^+$).
If $n' \in T_b$ and $\Delta'' = \Delta'''$, then we have $\overline{u} \in \overline{U_\Delta^-} \cap \overline{U_{\Delta''}^-} \cap \overline{U_\Delta^+} \cdot \overline{T} \cdot \overline{U_{\Delta''}^+}$.
In particular, $\overline{u} = 1$ because of axiom~\ref{axiomRGD6} of spherical root groups data and \cite[6.1.6]{BruhatTits1}.
In particular $u \in U'_{y}$ which is a contradiction.
As a consequence, when $n' \in T_b$, we have $\ell(y,C, n' \cdot x',n' \cdot C') > \ell(x,C,x',C')$.

We prove that $\ell(y,C,n'\cdot x',n' \cdot C') \geqslant \ell(y,C,x',C')$ with equality if, and only if $n' \in T_b$.
Since $n' \in N_y$, we get that $w(y,C,n' \cdot x, n' \cdot C') = \nu^v(n') w(y,C,x',C')$. Indeed, $x'-y \in w(y,C,x',C') \cdot \overline{C^v_\Delta}$ and $n' \cdot x'-y \in w(y,C,n' \cdot x',n' \cdot C') \cdot \overline{C^v_\Delta}$ by definition. But $n' \cdot( x' - y ) = n' \cdot x' - y \in \nu^v(n') w(y,C,x',C') \cdot \overline{C^v_\Delta}$.
Moreover, we have $w(y,C,x',C') \cdot C^v_{\Delta} \cap (x'-y + C') \neq \emptyset$ whence $n' \cdot \big( w(y,C,x',C') \cdot C^v_{\Delta} \cap (x'-y + C') \big) = \big( \nu^v(n') w(y,C,x',C') \cdot C^v_{\Delta} \big) \cap (n' \cdot x' - y + n' \cdot C') \neq \emptyset$.
Hence the equality is a consequence of the uniqueness in $W(\Phi)$ of the element $w(y,C,n' \cdot x',n' \cdot C')$ satisfying this property.
Set $t = \nu^v(n')$ and $w'=w(y,C,x',C')$.
Then we want to prove that $\ell(tw') \geqslant \ell(w')$ with equality if, and only if, $t=1$.
Let $W_y = W(\Phi_y)$ the Weyl group of the root system $\Phi_y$ and identify it with a subgroup of $W := W(\Phi)$.
Any vector chamber of $\Phi_y$, which is a simplicial cone in some quotient of $V_\Rtot$, can be identified with its inverse image in $V_\Rtot$.
Let $\widetilde{C}$ be the vector chamber of $\Phi_y$ containing $C$.
Let $R$ (resp. $R_y$) be the generating system of $W$ (resp. $W_y$) of reflections with respect to the walls of $C$ (resp. $\widetilde{C}$).
Let $w_0$ (resp. $w'_0$) be the longest element with respect to $R$ (resp. $R_y$) of $W$ (resp. of $W_y$).
We have $w_0(C) = -C$ and $w'_0(\widetilde{C}) = -\widetilde{C}$ so that $w'_0(\widetilde{C}) \supset w_0(C)$.

In the quotient group $\overline{G}_y$, consider the minimal parabolic subgroup $\overline{B}$ associated to the vector chamber $\widetilde{C}$.
Write $\overline{B} = \overline{T} \cdot \prod_{\alpha \in \Phi_y \cap \Phi_\Delta^+} \overline{U_\alpha}$.
Since $w'(C) = C'''$, we have $\overline{U_{y+C'''}} \subset w' \overline{B} {w'}^{-1}$.
Since $u \in U_{y-C}$, we have $\overline{u} \in w'_0 \overline{B} {w'}_{0}^{-1}$.
Since $u \in U'_{y} U_{y+C} n' U_{y+C'''}$, we have $\overline{u} \in \overline{B} t w' \overline{B} {w'}^{-1}$.
Let $(r_k, \dots, r_1)$ be a reduced decomposition of ${w'}^{-1} w_0 \in W$ with respect to $R$.
Then $w_0 = w' r_k \cdots r_1$ and we know that $\ell(w') = \ell(w_0) - k$ \cite[chap.VI,§1, no.6 cor.3 of prop.17]{bourbaki1981elements}.
For $1 \leqslant i \leqslant k-1$, write $w_i = w'r_k \cdots r_{i+1}$ and, $w_k = w'$.
Then $w_i = w_{i+1} r_{i+1}$ so that the vector chambers $w_i(C)$ and $w_{i+1}(C)$ have a wall $H_i$ in common and the reflection $s_i$ with respect to $H_i$ is $s_i = w_{i+1}w_i^{-1} = w_{i} r_{i+1} w_{i}^{-1}  $.

On the other hand, let $w'_i$ be the unique element in $W_y$ such that the chamber $w'_i(\widetilde{C})$ of $\Phi_y$ contains the chamber $w_i(C)$ of $\Phi$.
Note that for $i=0$, the element $w'_0$ is equal to the $w'_0$ previously defined.
We have $w'_k = w_k = w'$ since $w' \in W_y$.
Finally, the two chambers $w'_i(\widetilde{C})$ and $w'_{i+1}(\widetilde{C})$ are adjoining (i.e. the bases $\Delta_{w'_i(\widetilde{C})}$ and $\Delta_{w'_{i+1}(\widetilde{C})}$ differ by at most one root) or equal according to the fact that $s_i$ belongs or not to $W_y$.
When $s_i \in W_y$, we also have that $s_i = w'_{i+1} {w'}_i^{-1}$.
By \cite[6.1.15(a)]{BruhatTits1}, for any $0 \leqslant i \leqslant k$, there exists a unique $\overline{w'_i} \in W_y$ such that $\overline{B} \overline{u} w'_i \overline{B} = \overline{B} \overline{w'_i} \overline{B}$.
Since $\overline{u} \in w'_0 \overline{B} {w'}_0^{-1}$, we have $\overline{w'_0} = w'_0$.
Since $\overline{u} \in \overline{B} t w' \overline{B} {w'}^{-1}$ and $w'_k = w'$, we have $\overline{w'_k} = t w'$.

Let $I = \llbracket 0, k-1 \rrbracket$ and let $I_1 = \{ i \in I,\ w'_{i+1} = w'_i \}$.
For $i \in I_1$, we have $w'_{i+1} = w'_{i}$ and $\overline{w'_{i+1}} = \overline{w'_i}$.
Moreover, we have seen that if $i \not\in I_1$, then we have $w'_{i+1} {w'}_{i}^{-1} = w_{i+1} w_i^{-1} = s_i$.
If $i \not\in I_1$, we observe that $r'_i = {w'}_{i+1}^{-1} w'_i = {w'}_i^{-1} s_i w'_i$ is a reflection with respect to a wall of $\widetilde{C}$: in other words, we have $r'_i \in R_y$.
From \cite[6.1.15]{BruhatTits1} and axiom (T3) of Tits systems, we therefore deduce that for $i \not\in I_1$:
\[\overline{B} \overline{w'_{i+1}} \overline{B} = \overline{B} \overline{u} w'_{i+1} \overline{B} = \overline{B} \overline{u} w'_i r'_i \overline{B} \subset \overline{B} \overline{w'_i} \overline{B} r'_i \overline{B} \subset \overline{B} \overline{w'_i} \overline{B} \cup \overline{B} \overline{w'_i} r'_i \overline{B}.\]
Hence, we get a partition of $I \setminus I_1$ in two (possibly empty) subsets:
\[I_2 = \{i \in I \setminus I_1,\ \overline{w'_{i+1}} = \overline{w'_i}\}
\qquad \text{ and } \qquad
I_3 = \{i \in I \setminus I_1,\ \overline{w'_{i+1}} = \overline{w'_i} r'_i\}\]
so that $I = I_1 \sqcup I_2 \sqcup I_3$.
Finally, we denote $\overline{w_i} = \overline{w'_i} {w'}_i^{-1} w_i$ for $0 \leqslant i \leqslant k$.
Note that we have $\overline{w_0} = w_0$ and $\overline{w_k} = t w' {w'}^{-1} w' = tw'$.

For $i \in I$, define $d_i = \overline{w_{i+1}}^{-1} \overline{w_i} = w_{i+1}^{-1} w'_{i+1} \overline{w'_{i+1}}^{-1} \overline{w'_{i}} {w'}_i^{-1} w_i$.

If $i \in I_1$, since $w'_{i+1} = w'_i$ and $\overline{w'_{i+1}} = \overline{w'_i}$, we have $d_i = w_{i+1}^{-1} w'_{i} \overline{w'_{i}}^{-1} \overline{w'_{i}} {w'}_i^{-1} w_i = w_{i+1}^{-1} w_i = r_{i+1}$.

If $i \in I_2$, since $w'_{i+1} = s_i w'_i$ and $w_{i+1} = s_i w_i$, we have $d_i =w_{i+1}^{-1} w'_{i+1} {w'}_i^{-1} w_i= w_{i}^{-1} s_i s_i w'_{i} {w'}_i^{-1} w_i = 1$.

If $i \in I_3$, then $d_i = w_{i+1}^{-1} w'_{i+1} \overline{w'_{i+1}}^{-1} \overline{w'_{i}} {w'}_i^{-1} w_i = w_{i+1}^{-1} w'_{i+1} r'_i {w'}_i^{-1} w_i = w_{i+1}^{-1} w_i = r_{i+1}$.

As a consequence, we have $w_0 = \overline{w_0} = \overline{w_k} \left( \overline{w_k}^{-1} \overline{w_{k-1}} \right) \cdots \left(\overline{w_1}^{-1} \overline{w_0} \right) = \overline{w_k} d_{k-1} \cdots d_0 = t w' \prod_{i \in I_1 \sqcup I_3} r_{i+1}$.
Hence $\ell(tw') \geqslant \ell(w_0) - \operatorname{Card}(I_1 \sqcup I_3) \geqslant \ell(w_0) - k = \ell(w')$ and, if this is an equality, then we necessarily have that $I_2 = \emptyset$.
In that case, $w_0 = t w_0$ and therefore $t = 1$.

To conclude, if we have $\ell(y,C,n'\cdot x',n' \cdot C') > \ell(y,C,x',C')$, then $\ell(y,C,n'\cdot x',n' \cdot C') > \ell(x,C,x',C')$ since $\ell(y,C,x',C') \geqslant \ell(x,C,x',C')$.
Otherwise, we have $\ell(y,C,n'\cdot x',n' \cdot C') = \ell(y,C,x',C')$ and $n' \in T_b$.
We have seen that in this case $\ell(y,C,x',C') > \ell(x,C,x',C')$ and therefore we also have that $\ell(y,C,n'\cdot x',n' \cdot C') > \ell(x,C,x',C')$.
\end{proof}

\begin{proof}[Proof of~\ref{technicalBruhatStep2}] If $g$ satisfies condition~\ref{technicalBruhatCond1}, there is nothing to prove.
Otherwise, we define a strictly increasing sequence $(\lambda_i)_i$ of values in $\Rtot$ and, by induction, a sequence of elements $(n_i)_i$ in $N$ as follows:
\begin{itemize}
\item $\lambda_0=0$ and $n_0 = n$;
\item for $i \geqslant 1$, we apply step~\ref{technicalBruhatStep1} to $x + \lambda_{i-1} v$, $C$, $x'$, $C'$ and $n=n_{i-1} \in N$. If we are in case~\ref{technicalBruhatCond1}, we are done.
Otherwise, case~\ref{technicalBruhatCond2} provides $\lambda \in R_{>0}$ and $n' \in N$.
We set $\lambda_i = \lambda_{i-1} + \lambda$ and $n_i := n'$.
\end{itemize}
At some point $k\in \mathbb{N}$, the element $g$ will be in case~\ref{technicalBruhatCond1} for $y_k = x + \lambda_k v$ and $n_k$ since the length in the spherical (hence finite) Weyl group is bounded so that this process stops.
Thus $g \in \widehat{P}_{\mathcal{F}_{z,C}} n_k \widehat{P}_{\mathcal{F}_{x',C'}} \subset \widehat{P}_{\mathcal{F}_{z,C}} N \widehat{P}_{\mathcal{F}_{x',C'}}$ for any $z \in y_k + \delta_v$.
For any $i \in \llbracket 0, k-1 \rrbracket$ we have:
\begin{itemize}
\item $g \in \widehat{P}_{\mathcal{F}_{z,C}} n_{i} \widehat{P}_{\mathcal{F}_{x',C'}} \subset \widehat{P}_{\mathcal{F}_{z,C}} N \widehat{P}_{\mathcal{F}_{x',C'}}$ for any $z = x + \mu v$ with $\mu \in ]\lambda_i,\lambda_{i+1}[$;
\item $g \in \widehat{P}_{\mathcal{F}_{y_{i+1},C}} n_{i+1} \widehat{P}_{\mathcal{F}_{x',C'}} \subset \widehat{P}_{\mathcal{F}_{y_{i+1},C}} N \widehat{P}_{\mathcal{F}_{x',C'}}$ for $y_{i+1} = x + \lambda_{i+1} v$.
\end{itemize}
We deduce the result from the decomposition $\displaystyle\Rtot_{>0} = \left( \bigsqcup_{i=0}^k ]\lambda_i,\lambda_{i+1} [ \sqcup \{\lambda_{i+1}\} \right) \sqcup ]\lambda_k,\infty [$.
\end{proof}

Since the topology of $\mathbb{A}_\Rtot$ is less usual than the topology of $\mathbb{A}_\mathbb{R}$, we detail how to generalize the proof of the affine Bruhat decomposition.

\begin{proof}[Proof of Theorem~\ref{ThmBruhatDecomposition}]
Let $g \in G$.
By Iwasawa Decomposition~\ref{thmIwasawa} applied to $-C$ and $\mathcal{F}_{x',C'}$, there exist $u \in U_{\Delta}^-$, $n \in N$ and $u' \in \widehat{P}_{\mathcal{F}_{x',C'}}$ such that $g= unu'$.
Write $u = \prod_{\beta \in \left(\Phi_{\Delta}^+\right)_{\mathrm{nd}}} u_\beta$ with $u_\beta \in U_{-\beta}$ for $\beta \in \left(\Phi_{\Delta}^+\right)_{\mathrm{nd}}$.

Let $v \in C^v_{\mathbb{Z},\Delta}$ so that $\beta(v) \in \mathbb{Z}_{>0}$ for any $\beta \in \Phi^+_\Delta$.
For $\beta \in \Phi^+_\Delta$, define $\lambda_\beta = 0$ if $\varphi_{-\beta}(u_\beta) = \infty$ and $\lambda_\beta = \frac{1}{\beta(v)} \Big( \beta(x) - \varphi_{-\beta}(u_\beta) \Big)$ otherwise.
Let $\lambda \in \Rtot_{>0}$ be such that $\lambda > \max \{ \lambda_\beta,\ \beta \in \Phi^+_\Delta\}$ and consider $y = x - \lambda v \in x - \delta_v$.
For any $\beta \in \Phi^+_\Delta$ such that $u_\beta \neq 1$, we have $\beta(v) \lambda > \beta(v) \lambda_\beta = \beta(x) - \varphi_{-\beta}(u_\beta)$.
Thus $-\beta(x-\lambda v) + \varphi_{-\beta}(u_\beta) > 0$.
Hence we get that $u_\beta \in U'_{-\beta,y}$ for any $\beta \in \Phi^+_\Delta$.

By Proposition~\ref{PropPchapeauGermSectorDecQC} and Corollary~\ref{CorUaGermofChamber}, we know that $U'_{\beta,y} \subset \widehat{P}_{\mathcal{F}_{y,C}}$ for any $\beta \in \Phi$.
Therefore $u \in \widehat{P}_{\mathcal{F}_{y,C}}$.
Hence $g = u n u' \in \widehat{P}_{\mathcal{F}_{y,C}} N \widehat{P}_{\mathcal{F}_{x',C'}}$.
Since $x = y + \lambda v \in y + \delta_v \subset y + C$ according to Lemma~\ref{LemHalfLineInSector}, Lemma~\ref{LemTechnicalBruhat}(2) gives that $g \in \widehat{P}_{\mathcal{F}_{x,C}} N \widehat{P}_{\mathcal{F}_{x',C'}}$.
Hence, $G = \widehat{P}_{\mathcal{F}_{x,C}} N \widehat{P}_{\mathcal{F}_{x',C'}}$.

Now, let $n,n' \in N$ be such that $n' \in \widehat{P}_{\mathcal{F}_{x,C}} n \widehat{P}_{\mathcal{F}_{x',C'}}$.
Let $x'' = \nu(n')(x')$ and $C'' = {^v\!}\nu(n')(C')$, so that $n' \widehat{P}_{\mathcal{F}_{x',C'}} (n')^{-1} = \widehat{P}_{\mathcal{F}_{x'',C''}}$ according to Lemma~\ref{LemActionWeylIwahoriFilters}.
Let $n'' = n (n')^{-1}$.
Then $1 \in \widehat{P}_{\mathcal{F}_{x,C}} n'' \widehat{P}_{\mathcal{F}_{x'',C''}}$ which gives $n'' \in \widehat{P}_{\mathcal{F}_{x,C}} \widehat{P}_{\mathcal{F}_{x'',C''}}$.
By Corollary~\ref{corP=Phat}, we have $\widehat{P}_{\mathcal{F}_{x,C}} = P_{\mathcal{F}_{x,C}}$ and $\widehat{P}_{\mathcal{F}_{x'',C''}} = P_{\mathcal{F}_{x'',C''}}$.

Since $\cl\big(\mathcal{F}_{x'',C''}\big) = \cl\big( \germ_{x''}(x''+C'')\big)$ by Lemma~\ref{lemEnclosure_point_in_sector}, ,
we have  $P_{\mathcal{F}_{x'',C''}} = P_{\germ_{x''}(x''+C'')}$ by Lemma~\ref{lemUV=UclV}.

Let $C^{(3)}=C(x,x'',C'')$ (see Notation~\ref{NotGoodSector}), so that $(x +C^{(3)}) \cap (x''+C'') \neq \emptyset$ and $x'' \in x+\overline{C^{(3)}}$. 

Let $\XC=\{\Omega''\in \germ_{x''}(x''+C'')\ |\ \Omega''\subset x+C^{(3)}\}$.
Since $\XC\subset \germ_{x''}(x''+C'')$ (as sets of subsets of $\A_{\Rtot}$), we have that $P_{\germ_{x''}(x''+C'')} \supset \bigcup_{\Omega''\in \XC} P_{\Omega''}$.

Conversely, let $\widetilde{\Omega}'' \in \germ_{x''}(x''+C'')$.
Then $\Omega''=\widetilde{\Omega}'' \cap (x+C^{(3)}) \in \XC$ by Lemma~\ref{l_neighbourhood_chambers}.
Thus \[P_{\germ_{x''}(x''+C'')}=\bigcup_{\widetilde{\Omega}'' \in \germ_{x''}(x''+C'')} P_{\widetilde{\Omega}''} \subset \bigcup_{\Omega''\in \XC} P_{\Omega''}.\]

Consequently,  \[
P_{\mathcal{F}_{x'',C''}}=P_{\germ_{x''}(x''+C'')}= \bigcup_{\Omega''\in \XC} P_{\Omega''}.
\]

Let $\Omega\in \cl(\mathcal{F}_{x,C})$ and $\Omega''\in \XC$. Set $\Omega'=\Omega\cap (\Omega''-\overline{C^{(3)}}$). By Lemma~\ref{lemInclusion_point_local_face}, $x\in \Omega$, $x''\in \Omega''$ and thus $x\in \Omega'$ (since $x''\in x+\overline{C^{(3)}}$). Therefore \[\Omega'\subset \Omega''-C^{(3)}\text{ and } \Omega''\subset x+C^{(3)}\subset \Omega'+C^{(3)}.\]

Using Proposition~\ref{PropBizarre}\ref{PropBizarre2}, we get that $P_{\Omega'} P_{\Omega''} \subset N_{\Omega'} U_{\Omega' - \overline{C^{(3)}}} U_{\Omega'' + \overline{C^{(3)}}} N_{\Omega''}$.
By Proposition~\ref{PropPchapeauGermSectorDecQC}, we have $N_{\Omega'} = N_{\Omega''} = T_b$.
By Proposition~\ref{PropPintersectionSector}, we have $U_{\Omega' - \overline{C^{(3)}}} \subset U^-_{\Delta^{(3)}}$
 and
$U_{\Omega'' + \overline{C^{(3)}}} \subset U^+_{\Delta^{(3)}}$, where $\Delta^{(3)}$ is the basis associated to $C^{(3)}$.
Thus $P_{\Omega} P_{\Omega''} \subset P_{\Omega'} P_{\Omega''} \subset U^-_{\Delta^{(3)}} T_b U^+_{\Delta^{(3)}}$ for every $\Omega \in \mathcal{F}_{x,C}$ and every $\Omega'' \in \XC$.
Hence $n'' \in \widehat{P}_{\mathcal{F}_{x,C}} \widehat{P}_{\mathcal{F}_{x'',C''}} = P_{\cl(\mathcal{F}_{x,C})} \bigcup_{\Omega''\in \XC} P_{\Omega''} \subset U^-_{\Delta^{(3)}} T_b U^+_{\Delta^{(3)}}$.
By \cite[6.1.15]{BruhatTits1}, we get that $n'' \in T_b$.
Hence $n' \in n T_b$.
This provides the correspondence between the quotient group $N/T_b$ and the double cosets $\widehat{P}_{\mathcal{F}_{x,C}} \backslash G / \widehat{P}_{\mathcal{F}_{x',C'}}$.
\end{proof}

\section{Building associated to a valued root group datum}
\label{sectionLambdaBuildingFromVRGD}

Let $V_{\mathbb{R}}$ be a finite dimensional $\mathbb{R}$-vector space and let $\Phi \subseteq V_{\mathbb{R}}^*$ be a root system such that $V_{\mathbb{R}}^*=\left\langle \Phi \right\rangle_{\mathbb{R}} $.
Let $V_\Z$ be the lattice of coweights in $V_\R$ of $\Phi$.
Let $\Rtot$ be a nonzero totally ordered abelian group such that $\Rtot_\Q = \Rtot$.
Let $\mathbb{A}$ be an $\Rtot$-aff space with some origin $o$ and underlying $\Rtot$-module $V_\Rtot:=V_\Z \otimes_{\Z} \Rtot$.
Let $G$ be a group and let $(T,(U_{\alpha})_{\alpha\in\Phi},(M_{\alpha})_{\alpha\in \Phi},(\varphi_{\alpha})_{\alpha\in\Phi})$ be a generating $\Rtot$-valued root group datum of $G$ of type $\Phi$.
Denote by $N$ the subgroup of $G$ that is generated by the $M_{\alpha}$ for $\alpha \in \Phi$, and assume that we are given a compatible action $\nu: N \rightarrow \operatorname{Aff}_\Rtot(\mathbb{A})$ of $N $ on $\mathbb{A}$.  We adopt all the notations that have been introduced in sections~\ref{SecVRGD} and~\ref{SecParahoricBruhat} and that are associated to the data we are given. \\

Under these assumptions, we can define the following relation on $G \times \mathbb{A}$:
$$(g,x) \sim (h,y) \Leftrightarrow \exists n \in N, \begin{cases}
y=\nu(n)(x),\\
g^{-1}hn \in U_x.
\end{cases}$$
This relation is reflexive and symmetric. Moreover, for any $(g_1,g_2,g_3) \in G$ and $(x_1,x_2,x_3) \in \mathbb{A}$, if we can find $n_{12}, n_{23} \in N$ such that 
$$\begin{cases}
x_2=\nu(n_{12})(x_1),\\
x_3=\nu(n_{23})(x_2),\\
g_1^{-1}g_2n_{12} \in U_{x_1},\\
g_2^{-1}g_3n_{23} \in U_{x_2},
\end{cases}$$
then:
$$\begin{cases}
x_3=\nu(n_{23}n_{12})(x_1),\\
g_1^{-1}g_3n_{23}n_{12} \in g_1^{-1}g_2 U_{x_2}n_{12} \subseteq U_{x_1}n_{12}^{-1} U_{x_2}n_{12} = U_{x_1}
\end{cases}
$$
where the last equality is given by Proposition~\ref{PropActionNUOmega}.
Hence $\sim$ is an equivalence relation on $G \times \mathbb{A}$.

\begin{definition}\label{DefLambdaBuildingFromDatum}
The \textbf{$\Rtot$-building associated to the datum:}\index{building!associated to a valued root group datum} \[\left(G,T,(U_{\alpha})_{\alpha\in\Phi},(M_{\alpha})_{\alpha\in \Phi},(\varphi_{\alpha})_{\alpha\in\Phi},\nu\right)\]
is the quotient:
\[\mathcal{I}\left(G,T,(U_{\alpha})_{\alpha\in\Phi},(M_{\alpha})_{\alpha\in \Phi},(\varphi_{\alpha})_{\alpha\in\Phi},\nu\right):=(G \times \mathbb{A})/\sim.\]
To simplify notations, we will denote it $\mathcal{I}(G)$\index[notation]{i@$\mathcal{I}(G)$}\index[notation]{i@$\mathcal{I}\left(G,T,(U_{\alpha})_{\alpha\in\Phi},(M_{\alpha})_{\alpha\in \Phi},(\varphi_{\alpha})_{\alpha\in\Phi},\nu\right)$} in the rest of this section.

We denote by $[g,x]$ the class in $\mathcal{I}(G)$ of $(g,x) \in G \times \mathbb{A}$.

The group $G$ then acts on $\mathcal{I}(G)$ by:
$$g \cdot [h,x]:= [gh,x].$$
\end{definition}

\begin{lemma}
The map:
\begin{align*}
i: \mathbb{A} & \rightarrow \mathcal{I}(G)\\
x & \mapsto [1,x]
\end{align*}
is injective.
\end{lemma}

\begin{proof}
Let $x,y \in \mathbb{A}$ such that $i(x)=i(y)$. We can then find $n\in N \cap U_x = \widetilde{N}_x$
such that $y=\nu(n)(x)$. By Corollary~\ref{CorNtildeFixOmega}, we deduce that $x=y$.
\end{proof}

\begin{Fact}
For any $n \in N$ and any $x \in \mathbb{A}$, we have $n \cdot [1,x] = [n,x] = [1,\nu(n)(x)]$.
In particular, the subgroup $N$ stabilizes $i(\mathbb{A})$.
\end{Fact}

In particular, we identify $\mathbb{A}$ with the subset $i(\mathbb{A})$ of $\mathcal{I}(G)$.
More generally, we identify any subset $\Omega$ of $\mathbb{A}$ with the subset $i(\Omega)$ of $\mathcal{I}(G)$.

\begin{definition}\label{DefCombinatorialStructureBuilding}
  An \textbf{apartment}\index{apartment} of $\mathcal{I}(G)$ is a subset of $\mathcal{I}(G)$ of the form: $$A=g \cdot \mathbb{A} = \{[g,x],\ x\in \mathbb{A}\}$$ for some $g\in G$, endowed with the set $\mathrm{Isom}(\mathbb{A},A)$ of bijections $\iota: \mathbb{A} \rightarrow A$ given by $\iota: x \mapsto [g,\nu(n)(x)]$ for some $n\in N$.
  
  A \textbf{local face}\index{face!local} (resp. \textbf{local chamber}\index{chamber!local}) of $\mathcal{I}(G)$ is a filter on $\mathcal{I}(G)$ of the form: $$\mathcal{F}=g \cdot \germ_x(x+F^v) = \{g \cdot \Omega,\ \Omega \in \germ_x(x+F^v)\}$$
  for some $g\in G$, some $x \in \mathbb{A}$ and some vector face (resp. vector chamber) $F^v$ in $V_\Rtot$.
\end{definition}

\begin{Lem}\label{lemParahoric_fixator}[{see \cite[7.4.4]{BruhatTits1}}]
Let $\Omega$ be a non-empty subset of $\mathbb{A}$.
The group $\widehat{P}_\Omega$ is the pointwise stabilizer of $\Omega$ in $G$.

In particular, for any filter $\mathcal{V}$ on $\mathbb{A}$, the group $\widehat{P}_{\mathcal{V}}$ fixes $\mathcal{V}$.
\end{Lem}

\begin{proof}
If $\Omega =\{x\}$ is a single point and $g \in G$, then
\[ [1,x] = g \cdot [1,x] \Longleftrightarrow
\exists n \in N, \begin{cases}
x=\nu(n^{-1})(x)\\
gn^{-1} \in U_x
\end{cases}
\Longleftrightarrow
\exists n \in \widehat{N}_x,\ g \in U_x n\]
since $\widehat{N}_x$ is, by definition, the stabilizer of $x$ in $N$.
Thus, the stabilizer of $x$ in $G$ is $U_x \widehat{N}_x$, which is $\widehat{P}_x$ by Corollary~\ref{CorPchapeauOmegaQC}.
Hence, Proposition~\ref{propBT7.1.11} gives that $\widehat{P}_\Omega$ is the pointwise stabilizer of $\Omega$.
\end{proof}

\begin{proposition}\label{propAxiom(A2)}(see \cite[7.4.8]{ BruhatTits1}) The set $\I(G)$ satisfies~\ref{axiomA2}. More precisely, let $g\in G$. Then:\begin{enumerate}

\item there exists $n\in N$ such that $g^{-1} \cdot x=n \cdot x$ for all $x\in  \A \cap g\cdot \A$. 

\item $\A \cap g \cdot \A$ is enclosed.
\end{enumerate}

\end{proposition}

\begin{proof}

We may assume that $\Omega:=\A\cap g.\A$ is nonempty. Let $\XC$ be the set of subsets $\widetilde{\Omega}\subset \Omega$ such that $g. N\cap \widehat{P}_{\widetilde{\Omega}}\neq \emptyset$. By definition of $\I(G)$, $\XC$ contains $\{x\}$, for all $x\in \Omega$. Let $\Omega_1\in \XC$, $x_2\in \Omega$ and $n_1,n_2\in N$ be such that $gn_1\in \widehat{P}_{\Omega_1}$ and $gn_2\in \widehat{P}_{\{x_2\}}$. Let us prove that $\Omega_1\cup\{x_2\}\in \XC$. 

By Corollary~\ref{corBT7.1.6}, there exists a vector chamber $C^v\subset V_\Rtot$ such that $n_1^{-1}n_2\in \widehat{P}_{\Omega_1}.\widehat{P}_{x_2}\subset \widehat{N}_{\Omega_1}.U_{C^v}^-.U_{C^v}^+.\widehat{N}_{x_2}$ (we used the relations $\widehat{P}_{\Omega_1}=\widehat{N}_{\Omega_1}P_{\Omega_1}$ and $\widehat{P}_{x_2}=P_{x_2}\widehat{N}_{x_2}$ from Lemma~\ref{LemPOmegaQC} and Corollary~\ref{CorPchapeauOmegaQC}). Therefore, there exists $n_1'\in \widehat{N}_{\Omega_1}$ and $n_2'\in \widehat{N}_{x_2}$ such that $n_1'^{-1}n_1^{-1}n_2n'_2\in N\cap U_{C^v}^-U_{C^v}^+=\{1\}$ by \cite[6.1.15(c)]{BruhatTits1}. Set $n=n_1n'_1=n_2n'_2$. Then $gn\in \widehat{P}_{\Omega_1}\cap \widehat{P}_{x_2}=\widehat{P}_{\Omega_1\cup\{x_2\}}$ (by Proposition~\ref{propBT7.1.11}). Consequently, $\Omega_1\cup\{x_2\}\in \XC$ and by induction, every nonempty finite subset of $\Omega$ is in $\XC$.

The group $N/T_b$ is finite. Indeed,  $\widehat{N}_{x_0}$ is  by definition the stabilizer of $x_0$ and $T_b$ is the kernel of the action $\nu: N \to \operatorname{Aff}_{\Rtot}(\mathbb{A})$. Thus the quotient group $\widehat{N}_{x_0} / T_b$ can be identified with a subgroup of $W^v$ which is finite.
Write $N/T_b=\{n_1T_b,\ldots,n_kT_b\}$, with $k\in \N$ and $n_1,\ldots,n_k\in N$.
Choose $x_0\in \Omega$. 
Let $\mathrm{Fin}(\Omega,x_0)$ be the set of finite subsets $\widetilde{\Omega}$ of $\Omega$ such that
 $x_0\in \widetilde{\Omega}$. Let $J$ be the set of elements of $j\in \llbracket 1,k\rrbracket$ such that  there exists $\Omega_j\in \mathrm{Fin}(\Omega,x_0)$ such that $gn_j\notin \widehat{P}_{\Omega_j}$. Let $\tilde{\Omega}=\bigcup_{j\in J} \Omega_j$. Then $\widetilde{\Omega}\in \mathrm{Fin}(\Omega,x_0)$. 
 Moreover, if $j\in J$, $gn_j\notin \widehat{P}_{\Omega_i}\supset \widehat{P}_{\tilde{\Omega}}$. 
 Let $i\in \llbracket 1,k\rrbracket$ be such that $gn_i\in \widehat{P}_{\widetilde{\Omega}}$. Then $i\notin J$ and thus for all $\Omega'\in \mathrm{Fin}(\Omega,x_0)$, $gn_i\in \widehat{P}_{\Omega'}$. In particular, for all $x\in \Omega$, $gn_i\in \widehat{P}_x$. Consequently, $gn_i\in \bigcap_{x\in \Omega} \widehat{P}_{x}=\widehat{P}_{\Omega}$ (by Proposition~\ref{propBT7.1.11}) and thus for all $x\in \Omega$, $g^{-1}.x=n_i.x$. 

It remains to prove that $\Omega$ is enclosed. Let $\tilde{g}=gn_i$. Then $\tilde{g} \cdot \A \cap \A=g\cdot \A \cap \A=\Omega$.
Moreover, $\tilde{g}\in \widehat{P}_{\Omega}$ and thus  there exists $\tilde{n}\in \widehat{N}_\Omega$ such that $p:=\tilde{n}\tilde{g}\in P_{\Omega}$.
By Lemma~\ref{lemUV=UclV}, $P_{\Omega}=P_{\cl(\Omega)}$: there exists $\Omega'\in \cl(\Omega)$ such that $p\in P_{\Omega'}$.
Then for $x\in \Omega'$, one has $\tilde{g}.x=\tilde{n}^{-1}p.x=\tilde{n}^{-1}.x\in \A$ and thus $\tilde{g}.x\in \Omega$ for all $x\in \Omega'$.
Let $y\in \Omega'$.
Then $\tilde{g}.y\in \Omega$ and thus $\tilde{g}\tilde{g}.y=\tilde{g}.y$.
Consequently, $\tilde{g}.y=y$ and thus $y\in \Omega$.
Therefore, $\Omega'\subset\Omega$ and thus $\Omega=\Omega'\in\cl(\Omega)$: $\Omega$ is enclosed, which proves the proposition.

\end{proof}

\begin{remark}\label{rmkDef_sector_point}

\item Let $x\in \I(G)$ and $Q_\infty$ be a sector-germ at infinity. Then by Lemma~\ref{lemDecompositions_axioms}, there exists an apartment $A$ containing $x$ and $Q_\infty$. Let $Q\subset A$ be a sector whose germ at infinity is $Q_\infty$. Then we denote by $x+Q_\infty$ the translate of $Q$ at $x$. This does not depend on the choice of $A$ by~\ref{axiomA2} (Proposition~\ref{propAxiom(A2)}).

\end{remark}

\begin{Cor}[{see \cite[9.7(i)]{Landvogt}}]\label{corLandvogt9.7}
Let $\Omega$ be a non-empty subset of $\mathbb{A}$.
The group $U_\Omega$ acts transitively on the set of apartments of $\mathcal{I}(G)$ containing $\Omega$.
\end{Cor}

\begin{proof}
Let $A'$ an apartment containing $\Omega$ and $g \in G$ such that $A' = g \cdot A$.
Let $n \in N$ such that $g^{-1} \cdot x = n \cdot x$ for all $x \in \mathbb{A} \cap g \cdot \mathbb{A} \supset \Omega$ as in Proposition~\ref{propAxiom(A2)}.
Hence $gn \in \widehat{P}_\Omega$ by Lemma~\ref{lemParahoric_fixator}.
By Corollary~\ref{CorPchapeauOmegaQC}, there exist $u \in U_\Omega$ and $n' \in \widehat{N}_\Omega \subset N$ such that $gn = un'$.
Hence $A' = g \cdot A = gn \cdot A = un' \cdot A = u \cdot A$.
This proves the transitivity.
\end{proof}

\begin{Cor}[{see \cite[7.4.10]{BruhatTits1}}]\label{CorStabilizerA}
The group $N$ is the stabilizer of $\A$ in $G$.

The group $T_b$ is the pointwise stabilizer of $\A$ in $G$.
\end{Cor}

\begin{proof}
We firstly prove that $U_\mathbb{A} = \{1\}$.
Let $\alpha \in \Phi$.
Then $U_{\alpha,\mathbb{A}} = \bigcap_{x \in \mathbb{A}} U_{\alpha, -\alpha(x-o)}$.
Considering the elements $x = o+ \lambda \alpha^\vee \in \mathbb{A}$ for $\lambda \in \Rtot$, we get that
$$U_{\alpha,\mathbb{A}} \subset \bigcap_{\lambda \in \Rtot} U_{\alpha,-\alpha(\lambda \alpha^\vee)} = \varphi_\alpha^{-1}\left( \bigcap_{\lambda \in \Rtot} [-2\lambda,\infty] \right).$$
But the last intersection is reduced to $\infty$ since for any $\varepsilon \in \Rtot_{>0}$ and any $\mu \in \Rtot$, we have $\mu \not\in [-2(-\mu - \varepsilon),\infty]$.
Thus $U_{\alpha,\mathbb{A}} = \varphi_\alpha^{-1}(\{\infty\}) = \{1\}$ for any $\alpha \in \Phi$.
Hence $U_{\mathbb{A}} = \{1\}$.

Let $g \in G$ be such that $g \cdot \mathbb{A} = \mathbb{A}$.
By Proposition~\ref{propAxiom(A2)}, there is $n \in N$ such that $\forall x\in \mathbb{A},\ g^{-1} \cdot x = n \cdot x$.
Hence $gn \in \widehat{P}_{\mathbb{A}}$ by Lemma~\ref{lemParahoric_fixator}.
Since $\widehat{P}_\mathbb{A} = \widehat{N}_\mathbb{A}$ by Corollary~\ref{CorPchapeauOmegaQC}, we have that $gn \in N$. Thus $g \in N$.
Moreover, since the action of $N$ on $\mathbb{A}$ is induced by that of $\nu$ via $n \cdot [1,x] = [1,\nu(n)(x)]$, we deduce the result.
\end{proof}

\begin{lemma}\label{lemDecompositions_axioms}
\begin{enumerate}[label={(\arabic*)}]
\item\label{itemAxiomA3GG} Any two local faces are contained in a single apartment of $\mathcal{I}(G)$. In particular, $\mathcal{I}(G)$ satisfies axioms~\ref{axiomA3} and \ref{axiomGG}.
  \item\label{itemAxiomA4} Any two sector-germs are contained in a single apartment of $\mathcal{I}(G)$. 
  In other words, $\mathcal{I}(G)$ satisfies axiom~\ref{axiomA4}.
  
\item\label{itemVisualBoundary}  If $\mathcal{F}$ is a local face and $Q_\infty$ is the germ at infinity of a sector, there exists an apartment containing $\mathcal{F}$ and $Q_\infty$.

  \end{enumerate}
  \end{lemma}

  \begin{proof}

\ref{itemAxiomA3GG} This follows the proof of \cite[Théorème 7.4.18]{BruhatTits1}. By Proposition~\ref{propAxiom(A2)}, we know that $\mathcal{I}(G)$ satisfies axiom~\ref{axiomA2}. Hence it suffices to prove that any two local chambers $C,C'$ are contained in a single apartment. Write $C=\germ_{x}(Q)$ and $C'=\germ_{x'}(Q')$, where $x,x'\in \I(G)$ and $Q,Q'$ are sectors of $\I(G)$, based at $x$ and $x'$ respectively. As $G$ acts transitively on the set of apartments, we may assume that $Q\subset \A$. Let $g\in G$ be such that $g^{-1}.Q'\subset \A$. By the Bruhat decomposition~\ref{ThmBruhatDecomposition}, we can write $g=bnb'$ with $b\in \widehat{P}_{C}$, $b'\in \widehat{P}_{g^{-1}.C'}$ and $n\in N$. Then $b.C=C\subset b.\A$ and \[C'=g.g^{-1}.C'=bnb'.\germ_{g^{-1}.x'}(g^{-1}.Q')=bn.\germ_{g^{-1}.x'}(g^{-1}.Q')\subset b.\A.\]

\ref{itemAxiomA4} Let $Q_{\infty}$ and $Q'_{\infty}$ be two sector-germs at infinity in $\mathcal{I}(G)$. Since $G$ acts transitively on the set of apartments in $\mathcal{I}(G)$, we may assume that $Q_{\infty} \Subset \mathbb{A}$. Consider an element $g \in G$ such that
$Q'_{\infty}= g\cdot Q_{\infty}$. Let $\Phi_{Q_{\infty}}^+$ be the set of roots in $\Phi$ that are positive on $Q_{\infty}$. By \cite[6.1.15(a)]{BruhatTits1}, we can write $g=u_+nv_+$ with
$u_+,v_+ \in U_{\Phi_{Q_{\infty}}^+}$ and $n\in N$. Since $u_+$ and $v_+$ fix $Q_{\infty}$, we then have $Q_{\infty}\Subset u_+n\cdot \mathbb{A}$ and:
$$Q'_{\infty}=g\cdot Q_{\infty} =u_+n\cdot Q_{\infty} \Subset u_+n\cdot \mathbb{A}.$$

\ref{itemVisualBoundary} We obtain it similarly as~(i), by replacing the Bruhat decomposition by the Iwasawa decomposition (\ref{thmIwasawa}).
  \end{proof}

\begin{remark}
Let $d:\A\times \A\rightarrow \Rtot$ be invariant under the action of $N$. Then $d$ extends uniquely to a $G$-invariant   map $d:\I(G)\times \I(G)\rightarrow \Rtot$. Indeed, suppose that such a $d$ exists. Let $x,y\in \I(G)$. Then by Lemma~\ref{lemDecompositions_axioms}, there exists $g\in G$ such that $x,y\in A:=g.\A$. One necessarily have $d(x,y)=d(g^{-1}.x,g^{-1}.y)$. It remains to prove that the last formula is well defined, that is  that it does not depend on the choice of $g\in G$. Let $g'\in G$ be such that $x,y\in g'.\A$. Let $A'=g'.\A$.  Let $h\in G$ be such that $h.A=A'$ and such that $h$ fixes pointwise $A\cap A'$. Then $g'^{-1}hg\in N=\mathrm{Stab}(\A)$ (by Corollary~\ref{CorStabilizerA}). Thus one has $d(g^{-1}.x,g^{-1}.y)=d(g'^{-1}hg.g^{-1}.x,g'^{-1}hg.g^{-1}.y)=d(g'^{-1}.x,g'^{-1}.y)$, which proves our assertion. 
\end{remark}

\section{Valuation for quasi-split reductive groups}
\label{SecQuasiSplitGroups}

In the following, given a separable field extension $\mathbb{L}/\mathbb{K}$ and an affine $\mathbb{L}$-scheme $\mathbf{X}$, we denote by $R_{\mathbb{L}/\mathbb{K}}(\mathbf{X})$\index[notation]{r@$R_{\mathbb{L}/\mathbb{K}}(\mathbf{\cdot})$} the Weil restriction\index{Weil restriction} of $\mathbf{X}$ to $\mathbb{K}$.
For more general considerations on Weil restrictions, see \cite[7.6]{BoschLutkebohmertRaynaud-NeronModels} or \cite[A5]{ConradGabberPrasad}.

\subsection{Notations and recalls for quasi-split reductive groups}

Let $\mathbb{K}$\index[notation]{k@$\mathbb{K}$} be any field and $\mathbf{G}$\index[notation]{g@$\mathbf{G}$} be any reductive $\mathbb{K}$-group.
Recall that $\mathbf{G}$ splits over a finite Galois extension of $\mathbb{K}$ and denote by $\widetilde{\mathbb{K}}/\mathbb{K}$\index[notation]{k@$\widetilde{\mathbb{K}}$} a minimal one \cite[4.1.2]{BruhatTits2}.
Denote by $\widetilde{\mathbf{G}} = \mathbf{G}_{\widetilde{\mathbb{K}}}$\index[notation]{g@$\widetilde{\mathbf{G}}$} the $\widetilde{\mathbb{K}}$-group obtained by a base change from $\mathbb{K}$ to $\widetilde{\mathbb{K}}$.

When $\mathbb{K}$ is algebraically closed, the theory of structure of reductive groups enables us to consider Borel subgroups.
In general, over an arbitrary field, a reductive group does not admit any Borel subgroup defined over the ground field and we therefore need to consider the minimal parabolic subgroups. 
The intermediate situation is that of quasi-split reductive groups:

\begin{Def,Prop}\label{def:quasi:split} \cite[4.1.1]{BruhatTits2}
One says that a reductive $\mathbb{K}$-group $\mathbf{G}$ is {\em quasi-split}\index{quasi-split} if it satisfies the following equivalent conditions:
\begin{enumerate}
\item[(i)] $\mathbf{G}$ contains a Borel subgroup defined over $\mathbb{K}$;
\item[(ii)] $\mathbf{G}$ contains a maximal $\mathbb{K}$-split torus $\mathbf{S}$ such that its centralizer $\mathcal{Z}_\mathbf{G}(\mathbf{S})$ is a torus; 
\item[(iii)] for any maximal $\mathbb{K}$-split torus $\mathbf{S}$ of $\mathbf{G}$, its centralizer $\mathcal{Z}_\mathbf{G}(\mathbf{S})$ is a torus.
\end{enumerate}
\end{Def,Prop}

We now assume that $\mathbf{G}$ is a quasi-split reductive $\mathbb{K}$-group.
We provide a choice of a maximal $\mathbb{K}$-split torus $\mathbf{S}$\index[notation]{s@$\mathbf{S}$} and a Borel subgroup $\mathbf{B}$\index[notation]{b@$\mathbf{B}$} such that $\mathbf{T} = \mathcal{Z}_\mathbf{G}(\mathbf{S})$\index[notation]{t@$\mathbf{T}$} is a maximal torus of $\mathbf{G}$ contained in $\mathbf{B}$. This is always possible \cite[20.5, 20.6 (iii)]{Borel}.
Thus $\widetilde{\mathbf{T}} = \mathbf{T}_{\widetilde{\mathbb{K}}}$ is a maximal $\widetilde{\mathbb{K}}$-torus of $\widetilde{\mathbf{G}}$ containing $\widetilde{\mathbf{S}} = \mathbf{S}_{\widetilde{\mathbb{K}}}$.

We denote by $\Phi = \Phi(\mathbf{G},\mathbf{S})$\index[notation]{p@$\Phi$} the root system of $\mathbf{G}$ with respect to $\mathbf{S}$ and we call it the \textbf{ relative root system}\index{root system!relative}.
We denote by $\widetilde{\Phi} = \Phi(\widetilde{\mathbf{G}},\widetilde{\mathbf{T}})$\index[notation]{p@$\widetilde{\Phi}$} the root system of the split group $\widetilde{\mathbf{G}}$ with respect to $\widetilde{\mathbf{T}}$ and we call it the \textbf{absolute root system}.\index{root system!absolute}

\begin{Ex}
Let $\mathbb{L} / \mathbb{K}$ be a nontrivial separable field extension.
If $\widetilde{\mathbf{G}}$ is a split reductive $\mathbb{L}$-group, then the Weil restriction $\mathbf{G} = R_{\mathbb{L}/\mathbb{K}}(\widetilde{\mathbf{G}})$ of $\widetilde{\mathbf{G}}$ is a quasi-split but non-split reductive $\mathbb{K}$-group.
If $\mathbf{S}$ is a maximal $\mathbb{K}$-split torus of $\mathbf{G}$, then $\mathcal{Z}_{\mathbf{G}}(\mathbf{S}) = R_{\mathbb{L}/\mathbb{K}} (\mathbf{S}_\mathbb{L})$ is a maximal $\mathbb{K}$-torus of $\mathbf{G}$, isomorphic to $R_{\mathbb{L}/\mathbb{K}}(\mathbb{G}_{m,\mathbb{L}})^n \not\simeq \mathbb{G}_{m,\mathbb{K}}^n$.
Thus $\mathbf{G}$ is quasi-split but non-split.
Typically, in $\mathbf{G} = R_{\mathbb{L}/\mathbb{K}}(\mathrm{GL}_{n,\mathbb{L}})$ that is the general linear group over $\mathbb{L}$ seen as a $\mathbb{K}$-group, one can take $\mathbf{S}$ the maximal split torus of diagonal matrices with entries in $\mathbb{K}$ and its centralizer in $\mathbf{G}$ is the subgroup of diagonal matrices with entries in $\mathbb{L}$. 

There are also quasi-split but non-split groups that do not come from Weil restriction.
For instance, if $\mathbb{L}/\mathbb{K}$ is a quadratic Galois extension and $h$ is the standard Hermitian form, then the group $\mathrm{SU}(h)$ is a quasi-split but non-split $\mathbb{K}$-group.
We provide an example of such a group in Example~\ref{ExSUhGalois}.
\end{Ex}

\subsection{Recalls on root groups and their parametrizations}
\label{section root groups}

\subsubsection{Definition of root groups}

Given a basis $\Delta$ of $\Phi$ (resp. $\widetilde{\Delta}$ of $\widetilde{\Phi}$), we denote $\mathrm{Dyn}(\Delta)$ (resp. $\mathrm{Dyn}(\widetilde{\Delta})$) its Dynkin diagram.
The edges represent orthogonal defects of the basis which will translate defects of commutativity between root groups.
Multiple edges appear between two simple roots of different lengths and are oriented from the long root to the short root.

Given a reductive $\mathbb{K}$-group $\mathbf{G}$ and a maximal $\mathbb{K}$-split torus $\mathbf{S}$, the choice of a minimal $\mathbb{K}$-parabolic subgroup of $\mathbf{G}$ containing $\mathcal{Z}_\mathbf{G}(\mathbf{S})$ is equivalent to the choice of a basis $\Delta$ of the relative root system \cite[4.15]{BorelTits-Reductifs}.
In particular, if $\mathbf{G}$ is quasi-split, the choice of $\mathbf{S} \subset \mathbf{T} \subset \mathbf{B}$ as before naturally determines a basis $\widetilde{\Delta} = \Delta(\widetilde{\mathbf{G}},\widetilde{\mathbf{T}},\widetilde{\mathbf{B}})$ of $\widetilde{\Phi} = \Phi(\widetilde{\mathbf{G}},\widetilde{\mathbf{T}})$ and a basis $\Delta = \Delta(\mathbf{G},\mathbf{S},\mathbf{B})$ of $\Phi = \Phi(\mathbf{G},\mathbf{S})$.

Recall that the root groups of $\mathbf{G}$ over $\mathbb{K}$ are defined by the following proposition:
\begin{Def,Prop}[{\cite[14.5 \& 21.9]{Borel}}]\label{definition root groups}
For any root $\alpha \in \Phi$, there exists a unique $\mathbb{K}$-subgroup of $\mathbf{G}$, denoted by $\mathbf{U}_\alpha$,\index[notation]{u@$\mathbf{U}_\alpha$} which is closed, connected, unipotent, normalized by $\mathcal{Z}_\mathbf{G}(\mathbf{S})$ and whose Lie algebra is $\mathfrak{g}_\alpha + \mathfrak{g}_{2\alpha}$.
It is called the {\em root group}\index{root group} of $\mathbf{G}$ with respect to $\alpha$.

If $\Psi$ is a positively closed subset of $\Phi$, then there exists a unique $\mathbb{K}$-subgroup of $\mathbf{G}$, denoted by $\mathbf{U}_\Psi$, which is closed, connected, unipotent, normalized by $\mathcal{Z}_\mathbf{G}(\mathbf{S})$ and whose Lie algebra is $\displaystyle \sum_{\alpha\in\Psi} \mathfrak{g}_\alpha$.
\end{Def,Prop}

Note that the definition depends on $\Phi$ and, therefore, on the choice of the maximal split torus $\mathbf{S}$ defining $\mathbf{T} = \mathcal{Z}_\mathbf{G}(\mathbf{S})$.

Moreover, these root groups satisfy the following proposition:

\begin{Prop}[{\cite[21.9]{Borel}}]
\label{prop:closed:subset:root:group}
For any ordering on a positively closed subset $\Psi$ of $\Phi$, the product map $\prod_{\alpha \in \Psi_{nd}} \mathbf{U}_\alpha \rightarrow \mathbf{U}_\Psi$ is an isomorphism of $\mathbb{K}$-varieties.

For any pair of non-collinear roots $\alpha,\beta \in \Phi$, the subset $(\alpha,\beta)=\{r\alpha+s\beta \in \Phi,\ r,s \in \mathbb{Z}_{> 0} \}$ is positively closed and $[\mathbf{U}_\alpha,\mathbf{U}_\beta] \subset \mathbf{U}_{(\alpha,\beta)}$.
\end{Prop}

We denote by $\widetilde{\mathbf{U}}_{\widetilde{\alpha}}$ for $\widetilde{\alpha} \in \widetilde{\Phi}$ the root groups of $\widetilde{\mathbf{G}}$ with respect to $\widetilde{\mathbf{T}}$.

\subsubsection{The Galois action on the absolute root system}

Even if we can define a $*$-action on the Dynkin diagram for an arbitrary reductive $\mathbb{K}$-group, we assume for simplicity that $\mathbf{G}$ is a quasi-split reductive $\mathbb{K}$-group.

We consider the canonical action of the absolute Galois group $\Sigma = \operatorname{Gal}(\mathbb{K}_s / \mathbb{K})$ on the abstract group $\mathbf{G}(\mathbb{K}_s)$.
Since $\mathbf{G}$ is quasi-split, we can choose a maximal $\mathbb{K}$-split torus $\mathbf{S}$ and we get a maximal torus $\mathbf{T} =\mathcal{Z}_\mathbf{G}(\mathbf{S})$ of $\mathbf{G}$ defined over $\mathbb{K}$.
Thus, we define an action of $\Sigma$ on $X^*(\mathbf{T}_{\mathbb{K}_s})$ by:
\[
\forall \sigma \in \Sigma,\ \forall \chi \in X^*(\mathbf{T}_{\mathbb{K}_s}),\ \sigma \cdot \chi = t \mapsto \sigma\Big( \chi\big( \sigma^{-1}(t)\big)\Big)
\]

\begin{Not}[The Galois action on the absolute root system]
\label{NotStarAction}
This is a summary of \cite[§6]{BorelTits-Reductifs} for a quasi-split reductive $\mathbb{K}$-group $\mathbf{G}$.
Denote by $\widetilde{\Delta}$ a set of absolute simple roots and by $\operatorname{Dyn}(\widetilde{\Delta})$ its associated Dynkin diagram.
There exists an action of the Galois group $\Sigma = \operatorname{Gal}(\widetilde{\mathbb{K}} / \mathbb{K})$ on $\operatorname{Dyn}(\widetilde{\Delta})$ which preserves the diagram structure.
This action can be extended, by linearity, to an action of $\Sigma$ on $\widetilde{V}^* = X^*(\mathbf{T}_{\widetilde{\mathbb{K}}}) \otimes_\mathbb{Z} \mathbb{R}$, and on $\widetilde{\Phi}$.
The restriction morphism $j = \iota^* : X^*(\mathbf{T}) \rightarrow X^*(\mathbf{S})$, where $\iota : \mathbf{S} \subset \mathbf{T}$ is the inclusion morphism, can be extended to an endomorphism $\rho : \widetilde{V}^* \rightarrow\ \widetilde{V}^*$  of the Euclidean space $\widetilde{V}^*$.
This morphism $\rho$ is the orthogonal projection onto the subspace $V^*$ of fixed points by the action of $\Sigma$ on $\widetilde{V}^*$.
 The inclusion of $\widetilde{\Phi}$ in the Euclidean space $\widetilde{V}^*$ provides a geometric realization of the  absolute roots from which we deduce a geometric realization of  of $\Phi=\rho(\widetilde{\Phi})$ in $V^*$.
The orbits of the action of $\Sigma$ on $\widetilde{\Phi}$ are the fibers of the map $\rho: \widetilde{\Phi} \rightarrow \Phi$.
\end{Not}

\begin{Def}
\label{DefSplittingField}
Let $\widetilde{\alpha} \in \widetilde{\Phi}$ be an absolute root.
Denote by $\Sigma_{\widetilde{\alpha}}$ be the stabilizer of $\widetilde{\alpha}$ for the canonical Galois action.
The \textbf{field of definition}\index{field!of definition} of the root $\widetilde{\alpha}$ is the subfield of $\widetilde{\mathbb{K}}$ fixed by $\Sigma_{\widetilde{\alpha}}$, denoted by $\mathbb{L}_{\widetilde{\alpha}} = \widetilde{\mathbb{K}}^{\Sigma_{\widetilde{\alpha}}}$.\index[notation]{l@$\mathbb{L}_{\widetilde{\alpha}}$}
\end{Def}

This is determined, up to isomorphism by the relative root $\alpha = \widetilde{\alpha}|_\mathbf{S}$.
Indeed, $\alpha$ is an orbit of absolute roots, which means that if $\widetilde{\beta}|_\mathbf{S} = \widetilde{\alpha}|_\mathbf{S} = \alpha$, then $\widetilde{\beta} = \sigma \cdot \widetilde{\alpha}$ and $\mathbb{L}_{\widetilde{\beta}} = \sigma(\mathbb{L}_{\widetilde{\alpha}})$.
For a root $\alpha \in \Phi$, we denote by $\mathbb{L}_\alpha$\index[notation]{l@$\mathbb{L}_\alpha$} the class of $\mathbb{L}_{\widetilde{\alpha}}$ for $\widetilde{\alpha}|_{\mathbf{S}} = \alpha$.
We call it the \textbf{splitting field}\index{field!splitting} of $\alpha$.

\begin{Rq}
\label{rq:splitting:field:multipliable:root}
If $\alpha \in \Phi$ is a multipliable root,
then there exists $\widetilde{\alpha}, \widetilde{\alpha}' \in \rho^{-1}( \alpha )$ such that $\widetilde{\alpha} + \widetilde{\alpha}' \in \widetilde{\Phi}$ \cite[4.1.4 Cas II]{BruhatTits2}.
Because $\alpha$ is an orbit, we can write $\widetilde{\alpha}' = \sigma(\widetilde{\alpha})$ where $\sigma \in \Sigma$ is of order $2$.
As a consequence, the extension of fields $\mathbb{L}_{\widetilde{\alpha}} / \mathbb{L}_{\widetilde{\alpha}+\widetilde{\alpha}'}$ is quadratic.
By abuse of notation, we denote this extension, determined up to isomorphism, by $\mathbb{L}_\alpha / \mathbb{L}_{2\alpha}$.
\end{Rq}

\subsubsection{Parametrization of root groups}
\label{subsecParametrization}

In order to valuate the root groups thanks to the $\Lambda$-valuation of the field, we have to define a parametrization of each root group.
Moreover, these valuations have to be compatible.
That is why we furthermore have to get relations between the parametrizations.

A {\em Chevalley-Steinberg system}\index{Chevalley-Steinberg system} of $(\mathbf{G},\widetilde{\mathbb{K}},\mathbb{K})$ is the datum of morphisms: $\widetilde{x}_{\widetilde{\alpha}} : \mathbb{G}_{a,\widetilde{\mathbb{K}}} \rightarrow \widetilde{\mathbf{U}}_{\widetilde{\alpha}}$\index[notation]{x@$\widetilde{x}_{\widetilde{\alpha}}$} parametrizing the various root groups of $\widetilde{\mathbf{G}}$, and satisfying some axioms of compatibility, given in \cite[4.1.3]{BruhatTits2}.
These axioms take into account the commutation relations of absolute root groups and the $\operatorname{Gal}(\widetilde{\mathbb{K}} / \mathbb{K})$-action on root groups.
Note that despite the fact that the morphisms parametrize the root groups of $\widetilde{\mathbf{G}}$, a Chevalley-Steinberg system also depends on the quasi-split group $\mathbf{G}$ because of the relations between the $\widetilde{x}_{\widetilde{\alpha}}$ where $\widetilde{\alpha} \in \widetilde{\Phi}$.
According to \cite[4.1.3]{BruhatTits2}, a quasi-split reductive $\mathbb {K}$-group always admits a Chevalley-Steinberg system $(\widetilde{x}_{\widetilde{\alpha}})_{\widetilde{\alpha} \in \widetilde{\Phi}}$.

\begin{Not}\label{NotMalphaInSTsystems}
Let us recall  that there are elements in $\mathcal{N}_\mathbf{G}(\mathbf{S})(\widetilde{\mathbb{K}})$ defined by (see \cite[3.2.1]{BruhatTits2}):
\[m_{\widetilde{\alpha}} = \widetilde{x}_{\widetilde{\alpha}}(1) \widetilde{x}_{-\widetilde{\alpha}}(1) \widetilde{x}_{\widetilde{\alpha}}(1)\]
\index[notation]{m@$m_{\widetilde{\alpha}}$}
for $\widetilde{\alpha} \in \widetilde{\Phi}$ such that for any $\widetilde{\beta} \in \widetilde{\Phi}$ and any $u \in \widetilde{\mathbb{K}}$, we have:
\[ m_{\widetilde{\alpha}} \widetilde{x}_{\widetilde{\beta}}(u) m_{-\widetilde{\alpha}} \in \{ \widetilde{x}_{r_{\widetilde{\alpha}}(\widetilde{\beta})}( \pm u )\}\]
according to the second axiom defining Chevalley systems.
Moreover, one can observe that $m_{-{\widetilde{\alpha}}} = m_{\widetilde{\alpha}}$ from the matrix realization in $\mathrm{SL}_2$.
\end{Not}

Let $\alpha\in \Phi$ be a relative root.
Let $\pi : \mathbf{G}^\alpha \rightarrow \langle \mathbf{U}_{-\alpha} , \mathbf{U}_\alpha \rangle$ be the universal covering of the quasi-split semi-simple $\mathbb{K}$-subgroup of relative rank $1$ generated by $\mathbf{U}_\alpha$ and $\mathbf{U}_{-\alpha}$.
The group $\mathbf{G}^\alpha$ splits over $\mathbb{L}_\alpha$ (this explains the terminology of splitting field of a root).
A parametrization of the simply-connected group $\mathbf{G}^\alpha$ is given by \cite[4.1.1 to 4.1.9]{BruhatTits2}.
We now recall it to fix the notation.

\paragraph{The non-multipliable case}

Let $\alpha\in\Phi_{\mathrm{nd}}$ be a relative root such that $2\alpha \not\in \Phi$ and choose $\widetilde{\alpha} \in \alpha$.
By \cite[4.1.4]{BruhatTits2}, the rank-$1$ group $\mathbf{G}^\alpha$ is isomorphic to $R_{\mathbb{L}_{\widetilde{\alpha}} / \mathbb{K}}(\mathrm{SL}_{2,\mathbb{L}_{\widetilde{\alpha}}})$. 
Inside the classical group $\mathrm{SL}_{2,\mathbb{L}_{\widetilde{\alpha}}}$, a maximal $\mathbb{L}_{\widetilde{\alpha}}$-split torus of $\mathrm{SL}_{2,\mathbb{L}_{\widetilde{\alpha}}}$ can be parametrized by the following homomorphism:
\[
\begin{array}{cccc}z :&\mathbb{G}_{m,\mathbb{L}_{\widetilde{\alpha}}} & \rightarrow & \mathrm{SL}_{2,\mathbb{L}_{\widetilde{\alpha}}}\\&t&\mapsto&\begin{pmatrix}t&0\\0&t^{-1}\end{pmatrix}\end{array}
\]
The corresponding root groups can be parametrized by the following homomorphisms:
\begin{align*}
\begin{array}{cccc}y_- :&\mathbb{G}_{a,\mathbb{L}_{\widetilde{\alpha}}} & \rightarrow & \mathrm{SL}_{2,\mathbb{L}_{\widetilde{\alpha}}}\\&v&\mapsto&\begin{pmatrix}1&0\\-v&1\end{pmatrix}\end{array}&&&
\begin{array}{cccc}y_+ :&\mathbb{G}_{a,\mathbb{L}_{\widetilde{\alpha}}} & \rightarrow & \mathrm{SL}_{2,\mathbb{L}_{\widetilde{\alpha}}}\\&u&\mapsto&\begin{pmatrix}1&u\\0&1\end{pmatrix}\end{array}
\end{align*}
According to \cite[4.1.5]{BruhatTits2}, there exists a unique $\mathbb{L}_{\widetilde{\alpha}}$-group isomorphism $\xi_{\widetilde{\alpha}} : \mathrm{SL}_{2,\mathbb{L}_{\widetilde{\alpha}}} \rightarrow \widetilde{\mathbf{G}}^{\widetilde{\alpha}}$ satisfying $\widetilde{x}_{\pm \widetilde{\alpha}} = \pi \circ \xi_{\widetilde{\alpha}} \circ y_{\pm}$, where $\widetilde{\mathbf{G}}^{\widetilde{\alpha}}$ is the simple factor of $\mathbf{G}^\alpha_{\widetilde{\mathbb{K}}}$ of index $\widetilde{\alpha}$. 

\begin{Not}\label{NotParamNonMult}
Thus, we define $\mathbb{K}$-homomorphisms
\begin{align*}
 x_{\alpha} =& \pi \circ R_{\mathbb{L}_{\widetilde{\alpha}} / \mathbb{K}}(\xi_{\widetilde{\alpha}} \circ y_{+})&
 x_{-\alpha} =& \pi \circ R_{\mathbb{L}_{\widetilde{\alpha}} / \mathbb{K}}(\xi_{\widetilde{\alpha}} \circ y_{-})
\end{align*}
\index[notation]{x@$x_\alpha$}
which are $\mathbb{K}$-group isomorphisms between $R_{\mathbb{L}_{\widetilde{\alpha}}/\mathbb{K}}(\mathbb{G}_{a,\mathbb{L}_{\widetilde{\alpha}}})$ and respectively $\mathbf{U}_{\alpha}$ and $\mathbf{U}_{-\alpha}$.

We also define the following $\mathbb{K}$-group homomorphism:
$$\widehat{\alpha} = \pi \circ R_{\mathbb{L}_{\widetilde{\alpha}} / \mathbb{K}}(\xi_{\widetilde{\alpha}} \circ z) :
 R_{\mathbb{L}_{\widetilde{\alpha}} / \mathbb{K}}(\mathbb{G}_{m,\mathbb{L}_{\widetilde{\alpha}}}) \rightarrow \mathbf{T}^\alpha$$
 \index[notation]{a@$\widehat{\alpha}$}
 where $\mathbf{T}^\alpha = \mathbf{T} \cap \langle \mathbf{U}_{-\alpha}, \mathbf{U}_\alpha \rangle$.
\end{Not}

\begin{Fact}\label{FactCommutationTXaNonMult}
For any $\widetilde{\alpha} \in \widetilde{\Phi}$, any $t \in \mathbf{T}(\mathbb{L}_{\widetilde{\alpha}})$ and any $u \in \mathbb{L}_{\widetilde{\alpha}}$, we have $t \widetilde{x}_{\widetilde{\alpha}}(u) t^{-1} = \widetilde{x}_{\widetilde{\alpha}}(\widetilde{\alpha}(t)u)$ by definition of root groups.
Thus, for any $\widetilde{\alpha} \in \alpha$, we have that:
\[t x_\alpha(u) t^{-1} = x_\alpha(\widetilde{\alpha}(t) u)\]
by definition of the Weil restriction.
Thus, since a matrix calculation gives $\widehat{\alpha}(z) x_\alpha(u) \widehat{\alpha}(z^{-1}) = x_\alpha(z^2 u)$, we get that for any $\widetilde{\alpha} \in \alpha$, any $z \in \mathbb{L}_{\widetilde{\alpha}}^*$, we have:
\[\widetilde{\alpha}(\widehat{\alpha}(z)) = z^2.\]
\end{Fact}

\begin{Not}\label{NotMaNonMult}
We define maps $m_\alpha : \mathbb{G}_{m,\mathbb{L}_\alpha} \to \mathcal{N}_\mathbf{G}(\mathbf{S})$ and $m_{-\alpha} : \mathbb{G}_{m,\mathbb{L}_\alpha} \to \mathcal{N}_\mathbf{G}(\mathbf{S})$ \cite[4.1.5]{BruhatTits2} by:
\begin{align*}
 m_\alpha(u) =& x_\alpha(u) x_{-\alpha}(u^{-1}) x_\alpha(u)&
 m_{-\alpha}(u) =& x_{-\alpha}(u) x_{\alpha}(u^{-1}) x_{-\alpha}(u)
\end{align*}
whose matrix realizations in $\mathrm{SL}_{2,\mathbb{L}_\alpha}$ are respectively $\begin{pmatrix} 0 & u\\-u^{-1} & 0\end{pmatrix}$ and $\begin{pmatrix} 0 & u^{-1}\\-u & 0\end{pmatrix}$.

The unique elements $m(x_\alpha(u))$ and $m(x_{-\alpha}(v))$ defined in \cite[6.1.2(2)]{BruhatTits1} are then:
\begin{align*}
m(x_\alpha(u)) &= m_{-\alpha}(u^{-1}) &
m(x_{-\alpha}(v)) &= m_\alpha(v^{-1})
\end{align*}

We define an element $m_{\alpha} = m_\alpha(1) = m_{-\alpha}(1) = m_{-\alpha} \in \mathcal{N}_\mathbf{G}(\mathbf{S})(\mathbb{K})$.\index[notation]{m@$m_\alpha$}
\end{Not}

\begin{Fact}\label{FactMaNonMult}
We observe that $m_\alpha = m_{\widetilde{\alpha}}$ for any $\widetilde{\alpha} \in \alpha$ by definition of $x_\alpha$ as Weil restriction.

From the matrix realization in $\mathrm{SL}_2$ we can easily check that:
\begin{itemize}
\item $\forall u \in \mathbb{L}_\alpha,\ x_{-\alpha}(u) = m_\alpha x_\alpha(u) m_\alpha^{-1}$;
\item $\forall u \in \mathbb{L}^*_\alpha,\ m_\alpha(u) = \widehat{\alpha}(u) m_\alpha = m_\alpha \widehat{\alpha}(u^{-1})$;
\item $m_\alpha^2 = \widehat{\alpha}(-1)$;
\item $m_\alpha^4 = \operatorname{id}$.
\end{itemize}
\end{Fact}

\paragraph{The multipliable case:} \label{paragraph:parametrization:multipliable}

Let $\alpha\in\Phi_{\mathrm{nd}}$ be a relative root such that $2\alpha \in \Phi$.
Let ${\widetilde{\alpha}} \in \alpha$ be an absolute root from which $\alpha$ arises, and let $\tau \in \Sigma$ be an element of the Galois group such that ${\widetilde{\alpha}} + \tau({\widetilde{\alpha}})$ is again an absolute root.
To simplify notations, we let (up to compatible isomorphisms in $\Sigma$) $\mathbb{L} = \mathbb{L}_{\widetilde{\alpha}}$ and $\mathbb{L}_2 = \mathbb{L}_{{\widetilde{\alpha}} + \tau({\widetilde{\alpha}})}$ in this paragraph.
For any $x \in \mathbb{L}$, we denote ${^\tau\!}x$ instead of $\tau(x)$.
By \cite[4.1.4]{BruhatTits2}, the $\mathbb{K}$-group $\mathbf{G}^\alpha$ is isomorphic to $R_{\mathbb{L}_2 / \mathbb{K}}(\mathrm{SU}(h))$, where $h$ denotes the hermitian form on $\mathbb{L}\times \mathbb{L} \times \mathbb{L}$ given by the formula:
$$h : (x_{-1},x_0,x_1) \mapsto \sum_{i=-1}^{1} x_i {^\tau\!}x_{-i}.$$

The group $\mathbf{G}^\alpha_{\mathbb{L}_2}$ can be written as $\displaystyle \mathbf{G}^\alpha_{\mathbb{L}_2} = \prod_{\sigma \in \mathrm{Gal}(\mathbb{L}_2 / \mathbb{K})} \widetilde{\mathbf{G}}^{\sigma({\widetilde{\alpha}}),\sigma(\tau({\widetilde{\alpha}}))}$ where each $\widetilde{\mathbf{G}}^{\sigma({\widetilde{\alpha}}),\sigma(\tau({\widetilde{\alpha}}))}$ denotes a simple factor isomorphic to $\mathrm{SU}(h)$, so that $\mathrm{SU}(h)_{\mathbb{L}} \simeq \mathrm{SL}_{3,\mathbb{L}}$.

We define a connected unipotent $\mathbb{L}_{2}$-group scheme by providing the $\mathbb{L}_{2}$-subvariety of $R_{\mathbb{L}/\mathbb{L}_2}(\mathbb{A}^2_{\mathbb{L}})$:
\[ H_0(\mathbb{L},\mathbb{L}_{2}) = \left\{ (u,v),\ u {^\tau\!}u = v + {^\tau\!}v \right\}\]

with the following group law:
\[(u,v),(u',v') \mapsto (u+u',v+v'+{^\tau\!}u u').\]
Then, we let $H(\mathbb{L},\mathbb{L}_{2}) = R_{\mathbb{L}_{2}/\mathbb{K}}(H_0(\mathbb{L},\mathbb{L}_{2}))$.
For the rational points, we get
\[H(\mathbb{L},\mathbb{L}_{2})(\mathbb{K}) = \{ (u,v) \in \mathbb{L} \times \mathbb{L},\ u {^\tau}\!u = v + {^\tau}\!v \}.\]

We parametrize a maximal torus of $\mathrm{SU}(h)$ by the isomorphism
$$\begin{array}{cccc}z :&R_{\mathbb{L}/\mathbb{L}_2}(\mathbb{G}_{m,\mathbb{L}}) & \rightarrow & \mathrm{SU}(h)\\&t&\mapsto&\begin{pmatrix}t&0&0\\0&t^{-1}{^\tau}\!t&0\\0&0&{^\tau}\!t^{-1}\end{pmatrix}\end{array}$$

We parametrize the corresponding root groups of $\mathrm{SU}(h)$ by the homomorphisms:
\begin{align*}
\begin{array}{cccc}y_- : &H_0(\mathbb{L},\mathbb{L}_2) & \rightarrow & \mathrm{SU}(h)\\&(u,v)&\mapsto&\begin{pmatrix}1&0&0\\u&1&0\\-v&-{^\tau}\!u&1\end{pmatrix}\end{array}
&&&
\begin{array}{cccc}y_+ :&H_0(\mathbb{L},\mathbb{L}_2) & \rightarrow & \mathrm{SU}(h)\\&(u,v)&\mapsto&\begin{pmatrix}1&-{^\tau}\!u&-v\\0&1&u\\0&0&1\end{pmatrix}\end{array}
\end{align*}
By \cite[4.1.9]{BruhatTits2}, there exists a unique $\mathbb{L}_2$-group isomorphism $\xi_{\widetilde{\alpha}} : \mathrm{SU}(h) \rightarrow \widetilde{\mathbf{G}}^{{\widetilde{\alpha}},\tau({\widetilde{\alpha}})}$ satisfying:
\begin{align*}
\pi(\xi_{\widetilde{\alpha}}(y_+(u,v))) =& \widetilde{x}_{{\widetilde{\alpha}}}(u) \widetilde{x}_{{\widetilde{\alpha}} + {^\tau\!}{\widetilde{\alpha}}}(-v) \widetilde{x}_{{^\tau\!}{\widetilde{\alpha}}}({^\tau\!}u),&
\pi(\xi_{\widetilde{\alpha}}(y_-(u,v))) =& \widetilde{x}_{-{^\tau\!}{\widetilde{\alpha}}}(u) \widetilde{x}_{-{\widetilde{\alpha}} - {^\tau\!}{\widetilde{\alpha}}}(v) \widetilde{x}_{-{\widetilde{\alpha}}}({^\tau\!}u)
\end{align*}

\begin{Not}\label{NotParamMult}
From this, we define $\mathbb{K}$-homomorphisms
\begin{align*}
x_\alpha =& \pi \circ R_{\mathbb{L}_2 / \mathbb{K}}(\xi_{\widetilde{\alpha}} \circ y_+)&
x_{-\alpha} =& \pi \circ R_{\mathbb{L}_2 / \mathbb{K}}(\xi_{\widetilde{\alpha}} \circ y_-)
\end{align*}\index[notation]{x@$x_\alpha$}
which are $\mathbb{K}$-group isomorphisms between the $\mathbb{K}$-group $H(\mathbb{L},\mathbb{L}_{2})$ and the   root groups $\mathbf{U}_\alpha$ and $\mathbf{U}_{-\alpha}$ respectively. The group law is given by $x_{\pm \alpha}(u,v) x_{\pm \alpha}(u',v') = x_{\pm \alpha}(u+u',v+v'+{^\tau}\!u u')$.

We also define the following $\mathbb{K}$-group homomorphism:
$$\widehat{\alpha} = \pi \circ R_{\mathbb{L}_2 / \mathbb{K}}(\xi_{\widetilde{\alpha}} \circ z) :
 R_{\mathbb{L}_{\widetilde{\alpha}} / \mathbb{K}}(\mathbb{G}_{m,\mathbb{L}_{\widetilde{\alpha}}}) \rightarrow \mathbf{T}^\alpha$$\index[notation]{a@$\widehat{\alpha}$}
 where $\mathbf{T}^\alpha = \mathbf{T} \cap \langle \mathbf{U}_{-\alpha}, \mathbf{U}_\alpha \rangle$.
 \end{Not}

\begin{Fact}\label{FactCommutationTXaMult}
Let ${\widetilde{\alpha}}$ and $\tau$ be as before.
For any $t \in \mathbf{T}(\mathbb{L}_{\widetilde{\alpha}})$ and any $(u,v) \in H(\mathbb{L},\mathbb{L}_2)$, we have
\[t \widetilde{x}_{{\widetilde{\alpha}}}(u) \widetilde{x}_{{\widetilde{\alpha}} + {^\tau\!}{\widetilde{\alpha}}}(-v) \widetilde{x}_{{^\tau\!}{\widetilde{\alpha}}}({^\tau\!}u) t^{-1} =
\widetilde{x}_{{\widetilde{\alpha}}}({\widetilde{\alpha}}(t)u) \widetilde{x}_{{\widetilde{\alpha}} + {^\tau\!}{\widetilde{\alpha}}}(-({\widetilde{\alpha}} + {^\tau\!}{\widetilde{\alpha}})(t)v) \widetilde{x}_{{^\tau\!}{\widetilde{\alpha}}}({^\tau\!}{\widetilde{\alpha}}(t) {^\tau\!}u)\]
by definition of the root groups.
Thus:
\[t x_\alpha(u,v) t^{-1} = x_\alpha({\widetilde{\alpha}}(t) u, {\widetilde{\alpha}}(t) {^\tau\!}{\widetilde{\alpha}}(t) v)\]
by definition of the Weil restriction.
Thus, since a matrix calculation gives $\widehat{\alpha}(z) x_\alpha(u,v) \widehat{\alpha}(z^{-1}) = x_\alpha({^\tau\!}z^2 z^{-1} u, z {^\tau\!}z v)$, we get that for any $z \in \mathbb{L}_{\widetilde{\alpha}}^*$, we have:
\[{\widetilde{\alpha}}(\widehat{\alpha}(z)) = ({^\tau\!}z)^2 z^{-1}.\]
\end{Fact}
 
\begin{Not}\label{NotMaMult}
We define maps $m_\alpha : H(\mathbb{L},\mathbb{L}_2) \to \mathcal{N}_\mathbf{G}(\mathbf{S})$ and $m_{-\alpha} : H(\mathbb{L},\mathbb{L}_2) \to \mathcal{N}_\mathbf{G}(\mathbf{S})$ by \cite[4.1.11]{BruhatTits2}:
\begin{align*}
 m_\alpha(u,v) =& x_\alpha(uv^{-1},({^\tau\!}v)^{-1}) x_{-\alpha}(u,v) x_\alpha(u({^\tau\!}v)^{-1},({^\tau\!}v)^{-1})\\
 m_{-\alpha}(u,v) =& x_{-\alpha}(uv^{-1},({^\tau\!}v)^{-1}) x_{\alpha}(u,v) x_{-\alpha}(u({^\tau\!}v)^{-1},({^\tau\!}v)^{-1})
\end{align*}
Their matrix realizations in $\mathrm{SU}(h)$ are respectively 
\begin{align*}
\begin{pmatrix} 0 & 0 & -({^\tau\!}v)^{-1}\\0 & -({^\tau\!}v)v^{-1} & 0\\ -v& 0 & 0\end{pmatrix}
&&&
\begin{pmatrix} 0 & 0 & -v\\0 & -({^\tau\!}v)v^{-1} & 0\\ -({^\tau\!}v)^{-1}& 0 & 0\end{pmatrix}
\end{align*}

The unique elements $m(x_\alpha(u,v))$ and $m(x_{-\alpha}(u,v))$ defined in \cite[6.1.2(2)]{BruhatTits1} are then:
\begin{align*}
m(x_\alpha(u,v)) &= m_{-\alpha}(u,v) &
m(x_{-\alpha}(u,v)) &= m_\alpha(u,v)
\end{align*}
Even if $(0,1) \not\in H(\mathbb{L},\mathbb{L}_2)(\mathbb{K})$ in general, one can define an element $m_\alpha = m_\alpha(0,1) = m_{-\alpha}(0,1) = m_{-\alpha}$, that in fact belongs to $\mathcal{N}_\mathbf{G}(\mathbf{S})(\mathbb{K})$.

By convention, we set $m_{2\alpha} = m_\alpha$.\index[notation]{m@$m_\alpha$}
\end{Not}

\begin{Fact}\label{FactMaMult}
We observe that $ \widetilde{x}_{\widetilde{\alpha}}(1) \widetilde{x}_{-{\widetilde{\alpha}}}(1) \widetilde{x}_{\widetilde{\alpha}}(1) = m_{\widetilde{\alpha}}$ (resp. $m_{{^\tau\!}{\widetilde{\alpha}}}$) has the following matrix realization:
\begin{align*} \begin{pmatrix} 1 & 0 & 0\\ 0& 1& 1\\ 0 & 0 & 1\end{pmatrix}
\begin{pmatrix} 1 & 0 & 0\\ 0& 1& 0\\ 0 & -1 & 1\end{pmatrix}
\begin{pmatrix} 1 & 0 & 0\\ 0& 1& 1\\ 0 & 0 & 1\end{pmatrix}
&= \begin{pmatrix} 1 & 0 & 0\\ 0& 0& 1\\ 0 & -1 & 0\end{pmatrix}&
\text{resp. } & \begin{pmatrix} 0 & 1 & 0\\ -1& 0& 0\\ 0 & 0 & 1\end{pmatrix}
\end{align*}
so that we have:
\[ m_\alpha = m_{{\widetilde{\alpha}}} m_{{^\tau\!}{\widetilde{\alpha}}}^{-1} m_{{\widetilde{\alpha}}} = m_{{^\tau\!}{\widetilde{\alpha}}} m_{{\widetilde{\alpha}}}^{-1} m_{{^\tau\!}{\widetilde{\alpha}}}.\]
Moreover
\begin{itemize}
\item $\forall (u,v) \in H(\mathbb{L},\mathbb{L}_2),\ x_{-\alpha}(u,v) = m_\alpha x_\alpha(u,v) m_\alpha^{-1}$;
\item $\forall (u,v) \in H(\mathbb{L},\mathbb{L}_2) \setminus \{(0,0)\},\ m_\alpha(u,v) = \widehat{\alpha}({^\tau\!}v^{-1}) m_\alpha = m_\alpha \widehat{\alpha}(v)$;
\item $m_\alpha^2 = \operatorname{id}$.
\item $\forall (u,v) \in H(\mathbb{L},\mathbb{L}_2) \setminus \{(0,0)\}, m_\alpha(u,v)^2 = \widehat{\alpha}({^\tau\!}v^{-1} v)$.
\end{itemize}
\end{Fact}

\subsection{\texorpdfstring{$\mathfrak{R}^S$}{RS}-valuation of a root groups datum}\label{subsecRvaluation}

As before, let $\mathbf{G}$\index[notation]{g@$\mathbf{G}$} be a quasi-split reductive $\mathbb{K}$-group with a choice of a maximal split torus $\mathbf{S}$\index[notation]{s@$\mathbf{S}$} contained in the maximal $\mathbb{K}$-torus $\mathbf{T} =\mathcal{Z}_\mathbf{G}(\mathbf{S})$\index[notation]{t@$\mathbf{T}$} contained in a Borel subgroup $\mathbf{B}$,\index[notation]{b@$\mathbf{B}$} together with a parametrization of the root groups $\left(x_\alpha\right)_{\alpha \in \Phi}$\index[notation]{x@$x_\alpha$} deduced from a Chevalley-Steinberg system, defined in Notations~\ref{NotParamNonMult} and~\ref{NotParamMult}.

Denote $G = \mathbf{G}(\mathbb{K})$,\index[notation]{g@$G$} $T = \mathbf{T}(\mathbb{K})$\index[notation]{t@$T$} and $N = \mathbf{N}(\mathbb{K})$.\index[notation]{n@$N$}
For any relative root $\alpha \in \Phi$, denote $U_\alpha = \mathbf{U}_\alpha(\mathbb{K})$\index[notation]{u@$U_\alpha$} and $M_\alpha = T m_\alpha$\index[notation]{m@$M_\alpha$} where the element $m_\alpha \in G$ is defined as in Notations~\ref{NotMaNonMult} and~\ref{NotMaMult}.
Then, by \cite[4.1.19(ii)]{BruhatTits2}, we know that $\Big(T,\big(U_\alpha,M_\alpha\big)_{\alpha\in\Phi}\Big)$ is a generating root group datum of $G$ of type $\Phi$.

From now on, considering a valuation $\omega: \mathbb{K} \to \Lambda \cup \{\infty\}$, we assume that the extension $\widetilde{\mathbb{K}} / \mathbb{K}$ is \textbf{univalent} (generalizing the definition of \cite[1.6.1]{BruhatTits2}).\index{univalent} This means that the $\Lambda$-valued ground field $\mathbb{K}$ satisfies the following:
\begin{assumption}\label{assumpfield}
There is a unique surjective valuation $\omega': \widetilde{\mathbb{K}}^\times \to \Lambda'$ such that:
\begin{itemize}
\item $\Lambda'$ is a non-zero totally ordered abelian group;
\item there is a strictly increasing map $\omega(\mathbb{K}^\times) \to \Lambda'$ that identifies $\omega(\mathbb{K}^\times)$ with a finite index subgroup of $\Lambda'$;
\item for all $x \in \mathbb{K}^\times$, we have $\omega'(x) = \omega(x)$.
\end{itemize}
\end{assumption}

Thus $\Lambda'$ identifies with a subgroup of $\Lambda \otimes_\mathbb{Z} \Q \subset \RF^S$.
Note that for any sub-extension $\widetilde{\mathbb{K}}/\mathbb{L}/\mathbb{K}$ and any $\sigma \in \operatorname{Aut}_{\mathbb{K}}(\mathbb{L})$, we have $\omega' \circ \sigma = \omega'$.
We still denote, by abuse of notation, the valuation $\omega: \mathbb{L}^\times \to \RF^S$ for any sub-extension $\widetilde{\mathbb{K}}/\mathbb{L}/\mathbb{K}$.

\begin{Ex}
According to Corollary 3.2.3 and section 4.1 of \cite{engler}, the assumptions~\ref{assumpfield} are all satisfied if $\mathbb{K}$ is Henselian. 
\end{Ex}

\begin{Not}\label{NotRSvaluation}
For each root $\alpha \in \Phi$, we use the parametrization $x_\alpha$ of the root group $\mathbf{U}_\alpha$, given by the choice of a Chevalley-Steinberg system and the choice of an absolute root $\widetilde{\alpha}$ in the orbit $\alpha$, to define a map $\varphi_\alpha : U_\alpha \to \RF^S \cup \{\infty\}$\index[notation]{p@$\varphi_\alpha$}\index{root group datum!valuation} as follows:
\begin{itemize}
\item $\varphi_\alpha(x_\alpha(y)) = \omega(y)$ if $\alpha$ is a non-multipliable and non-divisible root, and if $y \in \mathbb{L}_\alpha$;
\item $\varphi_\alpha(x_\alpha(y,y')) = \frac{1}{2} \omega(y')$ if $\alpha$ is a multipliable root and if $(y,y') \in H(\mathbb{L}_\alpha,\mathbb{L}_{2\alpha})$;
\item $\varphi_{2\alpha}(x_\alpha(0,y')) = \omega(y')$ if $\alpha$ is a multipliable root and if $y' \in \mathbb{L}_\alpha^0:=\{ v \in \mathbb{L}_\alpha,\  v + {^\tau}\!v=0 \}$\index[notation]{L@$\mathbb{L}_\alpha^0$}.
\end{itemize}
\end{Not}

Note that, by convention, we set $\omega(0) = \infty = \frac{1}{2} \omega(0)$.

\begin{Rq}
Despite the fact that the parametrization $x_\alpha$ depends on the choice of $\widetilde{\alpha} \in \widetilde{\Phi}$ such that $\widetilde{\alpha}|_\mathbf{S} = \alpha$, the value $\varphi_\alpha(u) \in \RF^S \cup \{\infty\}$ for $u \in U_\alpha$ does not depend on this choice.
Indeed, assume for instance that $\alpha \in \Phi_{\mathrm{nd}}$ is non-multipliable.
For any $\sigma \in \operatorname{Gal}\left(\widetilde{\mathbb{K}} / \mathbb{K}\right)$, the isomorphism $\sigma^{-1}: \mathbb{L}_{\sigma(\widetilde{\alpha})} \to \mathbb{L}_\alpha$ induces a $\mathbb{K}$-isomorphism of the Weil restrictions: $j: R_{\mathbb{L}_{\sigma(\widetilde{\alpha})} / \mathbb{K}}\left( \mathbb{G}_{a,\mathbb{L}_{\sigma(\widetilde{\alpha})}} \right) \to R_{\mathbb{L}_{\widetilde{\alpha}} / \mathbb{K}}\left( \mathbb{G}_{a,\mathbb{L}_{\widetilde{\alpha}}} \right)$.
Thus the parametrization of $U_\alpha$ defined by $\widetilde{\alpha}$ instead of $\alpha$ would be precisely $x_\alpha \circ j: R_{\mathbb{L}_{\sigma(\widetilde{\alpha})} / \mathbb{K}}\left( \mathbb{G}_{a,\mathbb{L}_{\sigma(\widetilde{\alpha})}} \right) \to U_\alpha$.
Since $\omega = \omega \circ j$, we deduce that the value of $\varphi_\alpha$ does not depend on $\sigma$.
If $\alpha$ is a multipliable root, we can make a similar observation.
\end{Rq}

\begin{Prop}\label{Vtotal}
The datum $\left( \varphi_\alpha \right)_{\alpha\in \Phi}$ is an $\RF^S$-valuation of the root group datum $\Big(T,\big(U_\alpha,M_\alpha\big)_{\alpha\in\Phi}\Big)$, i.e. it satisfies axioms~\ref{axiomV0} to~\ref{axiomV5}.
\end{Prop}

According to Bruhat-Tits~\cite[4.2.11]{BruhatTits2}, it is easy to check it.
Such a verification is carried out by Landvogt \cite[7.4]{Landvogt}.

\begin{proof}
Axiom~\ref{axiomV0} is immediate since $\varphi_\alpha(\mathbf{U}_\alpha(\mathbb{K}))$ contains $\Lambda \cup \{\infty\}$ and the totally ordered group $\Lambda$ is not trivial by assumption.
Axiom~\ref{axiomV4} is immediate by definition.

\paragraph{Axioms~\ref{axiomV1},~\ref{axiomV2} and~\ref{axiomV5} for $\alpha$ non-multipliable:}

Let $\lambda \in \RF^S$ and let $g_1 = x_\alpha(u_1)$, $g_2 = x_\alpha(u_2)$ be elements in $U_{\alpha,\lambda}$ for some parameters $u_1, u_2 \in \mathbb{L}_\alpha$. Then $g_1 g_2^{-1} = x_\alpha(u_1) x_\alpha(u_2)^{-1} = x_\alpha(u_1) x_\alpha(-u_2) = x_\alpha(u_1-u_2)$.
Thus $\varphi_\alpha(g_1 g_2^{-1}) = \omega(u_1-u_2) \geqslant \min (\omega(u_1),\omega(u_2)) \geqslant \lambda$.
Moreover, $x_\alpha(0)$ is the only element with valuation $\infty$ which gives~\ref{axiomV1}.

Let $x_\alpha(u) \in U_\alpha$ with $u \in \mathbb{L}^*_\alpha$ and $ m = m_\alpha t \in M_\alpha = M_{-\alpha}$ with $t \in T$.
By formulas in~\ref{FactCommutationTXaNonMult} and~\ref{FactMaNonMult}, we have $m x_\alpha(u) m^{-1} = m_\alpha t x_\alpha(u) t^{-1} m_\alpha^{-1} = m_\alpha x_\alpha(\alpha(t)u) m_\alpha^{-1} = x_{-\alpha}(\alpha(t)u)$.
Hence $\varphi_\alpha(x_\alpha(u)) - \varphi_\alpha(m x_\alpha(u) m^{-1}) = \omega(u) - \omega(\alpha(t)u) = -\omega(\alpha(t))$ which does not depend on $u$.
This proves~\ref{axiomV2}.

Let $x_\alpha(u) \in U_\alpha$ and $x_{-\alpha}(u'),x_{-\alpha}(u'') \in U_{-\alpha}$ such that $x_{-\alpha}(u') x_\alpha(u) x_{-\alpha}(u'') \in M_\alpha$.
By the uniqueness in \cite[6.1.2(2)]{BruhatTits1} and the formula defining $m_{-\alpha}(u^{-1})$ in~\ref{NotMaNonMult}, we get $u' = u'' = u^{-1}$.
Thus $\varphi_{-\alpha}(x_{-\alpha}(u'))= \omega(u^{-1}) = -\varphi_\alpha(x_\alpha(u))$ which gives~\ref{axiomV5}.

\paragraph{Axioms~\ref{axiomV1},~\ref{axiomV2} and~\ref{axiomV5} for $\alpha$ multipliable:}
Let $\lambda \in \RF^S$ and let $g_1 = x_\alpha(u_1,v_1)$, $g_2 = x_\alpha(u_2,v_2)\in U_{\alpha,\lambda}$ for some parameters $(u_1,v_1), (u_2,v_2) \in H(\mathbb{L}_\alpha,\mathbb{L}_{2\alpha})$.
Then $u_i {^\tau\!}u_i = v_i + {^\tau\!}v_i$ gives $\omega(u_i) \geqslant \frac{1}{2} \omega(v_i) \geqslant \lambda$.
Thus \[g_1 g_2^{-1} = x_\alpha(u_1,v_1) x_\alpha(u_2,v_2)^{-1} = x_\alpha(u_1,v_1) x_\alpha(-u_2,{^\tau\!}v_2) = x_\alpha(u_1-u_2, -{^\tau\!}u_1 u_2+v_1+{^\tau\!}v_2).\]
Hence $\varphi_\alpha(g_1 g_2^{-1}) = \frac{1}{2}\omega(-{^\tau\!}u_1 u_2+v_1+{^\tau\!}v_2) \geqslant \frac{1}{2}\min (\omega(u_1)+\omega(u_2),\omega(v_1),\omega(v_2)) \geqslant \lambda$.
Since, for $(u,v) \in H(\mathbb{L}_\alpha,\mathbb{L}_{2\alpha})$, we have $u=0$ whenever $v=0$, we get that $x_\alpha(0,0)$ is the only element with valuation $\infty$.
This gives~\ref{axiomV1}.

Let $x_\alpha(u,v) \in U_\alpha$ with $(u,v) \in H(\mathbb{L}_\alpha,\mathbb{L}_{2\alpha}) \setminus \{(0,0)\}$ and $ m = m_\alpha t \in M_\alpha = M_{-\alpha}$ with $t \in T$.
By formulas in~\ref{FactCommutationTXaMult} and~\ref{FactMaMult}, we have $m x_\alpha(u,v) m^{-1} = m_\alpha t x_ \alpha(u,v) t^{-1} m_\alpha^{-1} = m_\alpha x_\alpha(\widetilde{\alpha}(t)u,\widetilde{\alpha}(t)\ {^\tau\!}\widetilde{\alpha}(t)v) m_\alpha^{-1} = x_{-\alpha}(-\widetilde{\alpha}(t)u, \widetilde{\alpha}(t)\ {^\tau\!}\widetilde{\alpha}(t) {^\tau\!}v).$
Hence \[\varphi_\alpha(x_\alpha(u)) - \varphi_{-\alpha}(m x_\alpha(u) m^{-1}) = \frac{1}{2}\omega(v) - \frac{1}{2}\omega(\widetilde{\alpha}(t)\ {^\tau\!}\widetilde{\alpha}(t)v) = -\omega(\alpha(t))\] which does not depend on $(u,v)$.
This proves~\ref{axiomV2}.

Let $x_\alpha(u,v) \in U_\alpha$ and $x_{-\alpha}(u',v'),x_{-\alpha}(u'',v'') \in U_{-\alpha}$ such that $x_{-\alpha}(u',v') x_\alpha(u,v) x_{-\alpha}(u'',v'') \in M_\alpha$.
By the uniqueness in \cite[6.1.2(2)]{BruhatTits1} and the formula defining $m_{-\alpha}(u,v)$ in~\ref{NotMaMult}, we get $(u',v') = (uv^{-1},{^\tau\!}v^{-1})$.
Thus $\varphi_{-\alpha}(x_{-\alpha}(u',v'))= \frac{1}{2}\omega({^\tau\!}v^{-1}) = -\varphi_\alpha(x_\alpha(u,v))$ which gives~\ref{axiomV5}.

\paragraph{Axiom~\ref{axiomV3} of commutation:}

It is a consequence of \cite[Annexe A]{BruhatTits2}.

\end{proof}

\subsection{Action of \texorpdfstring{$N$}{N} on an \texorpdfstring{$\RF^S$}{RS}-aff space}\label{subsecAction_N_affine_space}

In the sequel, we adopt the following notations:
\begin{align*}
\mathbf{G}: &\;\text{a quasi-split reductive group over $\mathbb{K}$,}\\
\widetilde{\mathbb{K}}: &\; \text{a finite Galois extenion of $\mathbb{K}$ on which $\mathbf{G}$ splits,}\\
\mathbf{S}: &\;\text{a maximal $\mathbb{K}$-split torus of $\mathbf{G}$,}\\
\mathbf{T}: &\;\text{the centralizer of $\mathbf{S}$ in $\mathbf{G}$,}\\
\mathbf{N}: &\;\text{the normalizer of $\mathbf{S}$ in $\mathbf{G}$,}
\end{align*}
and for any algebraic $\mathbb{K}$-group $\mathbf{H}$:
\begin{align*}
\widetilde{\mathbf{H}}: &\;\text{the scalar extension $\mathbf{H}_{\widetilde{\mathbb{K}}}$ of $\mathbf{H}$ to $\widetilde{\mathbb{K}}$,}\\
X_*(\mathbf{H}):&\;\text{the group of cocharacters of $\mathbf{H}$,}\\
X^*(\mathbf{H}):&\;\text{the group of characters of $\mathbf{H}$,}\\
X^*_{\mathbb{K}}(\mathbf{H})
:&\;\text{the group of rational characters of $\mathbf{H}$ over $\mathbb{K}$.}\\
\end{align*}

The natural pairing of abelian groups:
$$X_*(\mathbf{S}) \otimes X^*(\mathbf{S}) \rightarrow \mathbb{Z}$$
is perfect and therefore induces an isomorphism:
$$X_*(\mathbf{S}) \otimes_{\mathbb{Z}} \mathbb{R} \cong \mathrm{Hom}_{\mathbb{R}}(X^*(\mathbf{S}) \otimes_{\mathbb{Z}} \mathbb{R},\mathbb{R}).$$
By tensorization by the abelian group $\mathfrak{R}^S$, we obtain an isomorphism of $\mathfrak{R}^S$-modules:
$$V_1:=X_*(\mathbf{S}) \otimes \mathfrak{R}^S \cong \mathrm{Hom}_{\mathbb{R}}(X^*(\mathbf{S}) \otimes_{\mathbb{Z}} \mathbb{R},\mathbb{R})\otimes \mathfrak{R}^S \cong \mathrm{Hom}_{\mathbb{R}}(X^*(\mathbf{S}) \otimes_{\mathbb{Z}} \mathbb{R},\mathfrak{R}^S).$$
The Weyl group $W:= \mathbf{N}(\mathbb{K})/\mathbf{T}(\mathbb{K})$\index[notation]{w@$W$}
acts  $\mathbb{R}$-linearly on $X_*(S) \otimes \mathbb{R}$ and hence $\RF^S$-linearly on $V_1$.
Since $X^*_\mathbb{K}(\mathbf{T})$ is a finite index subgroup of $X^*(\mathbf{S})$, we have:
$$X^*(\mathbf{S}) \otimes_{\mathbb{Z}} \mathbb{R} = X^*_{\mathbb{K}}(\mathbf{T}) \otimes_{\mathbb{Z}} \mathbb{R}.$$
Moreover, for each $t \in \mathbf{T}(\mathbb{K})$, the map:
\begin{align*}
\rho(t): X_{\mathbb{K}}^*(\mathbf{T}) \otimes \mathbb{R} & \rightarrow \mathfrak{R}^S\\
\chi \otimes \lambda & \mapsto -\lambda \omega(\chi(t))
\end{align*}
is well-defined and is $\mathbb{R}$-linear.

We can therefore see $\rho(t)$ as an element in $V_1$ and we get a group homomorphism: 
\begin{align*}
\rho: \mathbf{T}(\mathbb{K}) &\rightarrow V_1\\
t& \mapsto \rho(t).
\end{align*}
Let $\mathbf{T}_b(\mathbb{K})$ be the kernel of $\rho$ and let $V_0$ (resp. $V_{0,\Z}$) be the subspace of $V_1$ (resp. $X_*(\mathbf{S})$) given by the vectors $v$ such that $\alpha(v)=0$ for every root $\alpha\in \Phi$.
The quotient $V'_\Z:=X_*(\mathbf{S})/V_{0,\Z}$ is then a $\Z$-module.
By flatness of $\RF^S$ over $\Z$, the quotient $V:=V_1/V_0= V'_\Z \otimes \RF^S$ is then an $\RF^S$-module that is endowed with the following structures:
\begin{itemize}
\item[-] a morphism:
$$\overline{\rho}: \mathbf{T}(\mathbb{K})/\mathbf{T}_b(\mathbb{K}) \rightarrow V,$$
induced by $\rho$;
\item[-] a morphism:
$$j: W \rightarrow \mathrm{GL}(V'_\Z),$$
induced by the action of $W $ on $X_*(\mathbf{S})$.
\end{itemize}

Let $W'$ be the push-out of the morphism $\overline{\rho}: \mathbf{T}(\mathbb{K})/\mathbf{T}_b(\mathbb{K}) \rightarrow V$ and the inclusion $\mathbf{T}(\mathbb{K})/\mathbf{T}_b(\mathbb{K}) \subseteq \mathbf{N}(\mathbb{K})/\mathbf{T}_b(\mathbb{K})$. The group $W'$ is then an extension of $W$ by $V$:
\begin{equation}\label{exten}
1 \rightarrow V \rightarrow W'\rightarrow W \rightarrow 1.
\end{equation}
If $\kappa: W \rightarrow \mathrm{Aut}_{\mathrm{group}}(V)$ is the induced action by conjugation of $W$ on $V$, we can see the previous exact sequence as a class in the cohomology group $H^2(W,V)$. But this group is trivial since $W$ is finite and $V$ is uniquely divisible. Hence exact sequence (\ref{exten}) splits and $W' = V\rtimes_{\kappa} W$. The action $\kappa$ is computed as follows: 
$$\forall v\in V, \forall w \in W, \kappa(w)(v)=j(w)(v),$$
Hence $j$ induces a morphism:
$$j':W' = V\rtimes_{\kappa} W \rightarrow V\rtimes \mathrm{GL}(V'_\Z) =\mathrm{Aff}_{\RF^S}(V).$$
By composing the projection $\mathbf{N}(\mathbb{K})\rightarrow \mathbf{N}(\mathbb{K})/\mathbf{T}_b(\mathbb{K})$, the natural morphism $\mathbf{N}(\mathbb{K})/\mathbf{T}_b(\mathbb{K}) \rightarrow W'$ and $j'$, we get a morphism:
$$\nu: \mathbf{N}(\mathbb{K})\rightarrow \mathrm{Aff}_{\RF^S}(V)$$
  such that the following diagram commutes:
 \begin{equation}\label{diagramAffineAction}
 \xymatrix{
 1 \ar[r] & \mathbf{T}(\mathbb{K})\ar[d]^{\rho_0} \ar[r] & \mathbf{N}(\mathbb{K})\ar[d]^{\nu}\ar[r] & W\ar[d]^j \ar[r] & 1\\
 1 \ar[r] & V \ar[r] & \mathrm{Aff}_{\RF^S}(V)\ar[r] & \mathrm{GL}(V'_\Z) \ar[r] & 1,
 }
 \end{equation}
 where $\rho_0$ is the composite of the map $\rho:\mathbf{T}(\mathbb{K}) \rightarrow V_1$ and the projection $V_1 \rightarrow V = V_1 / V_0$.

\begin{Rq}\label{RkTbK}
Denote by $\mathbf{T}(\mathbb{K})_b$ the kernel of $\nu$.
By construction, it is the kernel of the composition of the natural projection $\mathbf{T}(\mathbb{K}) \to \mathbf{T}(\mathbb{K}) /  \mathbf{T}_b(\mathbb{K})$ and $\overline{\rho}$.
Hence $\ker \rho = \mathbf{T}_b(\mathbb{K})$ is a subgroup of $\ker \nu = \mathbf{T}(\mathbb{K})_b$ but this inclusion is strict in general (see Example~\ref{ExTbGLn} for $\mathbf{G} = \mathrm{GL}_n$).

If $\mathbf{G}$ is semi-simple, then $V_{0,\mathbb{Z}} = 0$ by \cite[8.1.8(ii)]{Springer} and therefore $\overline{\rho}$ is injective. Thus $\ker \nu = \mathbf{T}(\mathbb{K})_b = \mathbf{T}_b(\mathbb{K}) = \ker \rho$.
\end{Rq}
 
\begin{Lem}[see {\cite[4.2.5, 4.2.6 and 4.2.7]{BruhatTits2}}]\label{LemActionOfT}
For any relative root $\alpha \in \Phi$, any absolute root $\widetilde{\alpha} \in \widetilde{\Phi}$ such that $\widetilde{\alpha}|_\mathbf{S} = \alpha$ and any $t \in \mathbf{T}(\mathbb{K})$, we have:
\[ \alpha\big( \nu(t) \big) = - \omega\big(\widetilde{\alpha}(t) \big). \]
\end{Lem}

\begin{proof}
For $\chi \in X^*_{\mathbb{K}}(\mathbf{T})$, we have by definition of the action that
\begin{equation} \chi(\nu(t)) = \chi(\rho(t)) = - \omega(\chi(t)).\label{eqChiRhoT} \end{equation}
Inside the $\mathbb{R}$-module $X^*(\mathbf{S}) \otimes \mathbb{R} = X^*_\mathbb{K}(\mathbf{T}) \otimes \mathbb{R}$, we have the identification
\[\alpha \otimes 1 = \sum_{\sigma \in \operatorname{Gal}(\widetilde{\mathbb{K}} / \mathbb{K})} \sigma(\widetilde{\alpha}) \otimes \frac{1}{[\widetilde{\mathbb{K}}:\mathbb{K}]}.\]
Thus, applying the formula (\ref{eqChiRhoT}) to $\chi = \sigma(\widetilde{\alpha})$ for $\sigma \in \operatorname{Gal}(\widetilde{\mathbb{K}} / \mathbb{K})$, we get
\[
 \alpha\big( \nu(t) \big) = - \frac{1}{[\widetilde{\mathbb{K}}:\mathbb{K}]} \sum_{\sigma \in \operatorname{Gal}(\widetilde{\mathbb{K}} / \mathbb{K})} \omega\Big(\sigma\big(\widetilde{\alpha}\big)(t)\Big)
= - \omega\big( \widetilde{\alpha}(t) \big).\]
\end{proof}

\begin{Lem}\label{LemN1finite}
The subgroup of $\mathbf{N}(\widetilde{\mathbb{K}})$ generated by the $m_{\widetilde{\alpha}}$ for $\widetilde{\alpha} \in \widetilde{\Phi}$ is finite.
\end{Lem}

\begin{proof}
As a split group is, in particular, a quasi-split group with $\widetilde{\mathbb{K}} = \mathbb{K}$, we keep the notations $x_{\widetilde{\alpha}}$, $\widehat{\widetilde{\alpha}}$, $m_{\widetilde{\alpha}}$ of section~\ref{subsecParametrization} for $\widetilde{\alpha} \in \widetilde{\Phi} = \Phi(\widetilde{\mathbf{G}},\widetilde{\mathbf{T}})$.
Denote by $\widetilde{T} = \widetilde{\mathbf{T}}(\widetilde{\mathbb{K}})$ and by $\widetilde{N} = \widetilde{\mathbf{N}}(\widetilde{\mathbb{K}})$.
Let $\widetilde{N}_1$ be the subgroup of $\widetilde{N}$ generated by the $m_{\widetilde{\alpha}}$ for $\widetilde{\alpha} \in \widetilde{\Phi}$.
Let $\widetilde{T}_1$ be the subgroup of $\widetilde{T}$ generated by the $\widehat{\widetilde{\alpha}}(-1) \in \widetilde{T}$ for $\widetilde{\alpha} \in \widetilde{\Phi}$. It is a subgroup of $\widetilde{N}_1$ since $m_{\widetilde{\alpha}}^2=\widehat{\widetilde{\alpha}}(-1)$ for each $\widetilde{\alpha} \in \widetilde{\Phi}$.

The group $\widetilde{T}_1$ is finite of exponent $2$ since it is a subgroup of a commutative group $\widetilde{T}$ generated by finitely many elements of order $2$.
Moreover, for any $\widetilde{\alpha}, \widetilde{\beta} \in \widetilde{\Phi}$, we have
\begin{align*}
\widehat{\widetilde{\alpha}}(-1) m_{\widetilde{\beta}} \widehat{\widetilde{\alpha}}(-1)^{-1}
&=\widehat{\widetilde{\alpha}}(-1) x_{\widetilde{\beta}}(1) x_{-\widetilde{\beta}}(1) x_{\widetilde{\beta}}(1) \widehat{\widetilde{\alpha}}(-1)^{-1}\\
&= x_{\widetilde{\beta}}\left( \widetilde{\beta}\left( \widehat{\widetilde{\alpha}}(-1) \right) \right) 
x_{-\widetilde{\beta}}\left( (-\widetilde{\beta})\left( \widehat{\widetilde{\alpha}}(-1) \right) \right)
x_{\widetilde{\beta}}\left( \widetilde{\beta}\left( \widehat{\widetilde{\alpha}}(-1) \right) \right)\\
&= m_{\widetilde{\beta}} \widehat{\widetilde{\beta}}\left((-1)^{\langle \widetilde{\beta},\widetilde{\alpha}^\vee \rangle}\right)
\end{align*}
Thus $\widetilde{N}_1$ normalizes $\widetilde{T}_1$.
Moreover for every $\widetilde{\alpha},\widetilde{\beta} \in \widetilde{\Phi}$, we have
\begin{align*}
m_{\widetilde{\alpha}} m_{\widetilde{\beta}} m_{\widetilde{\alpha}}
& = m_{\widetilde{\alpha}} x_{\widetilde{\beta}}(1) x_{-\widetilde{\beta}}(1) x_{\widetilde{\beta}}(1) m_{\widetilde{\alpha}}^{-1}\\
& = x_{r_{\widetilde{\alpha}}(\widetilde{\beta})}(\varepsilon_1)
x_{r_{\widetilde{\alpha}}(-\widetilde{\beta})}(\varepsilon_2)
x_{r_{\widetilde{\alpha}}(\widetilde{\beta})}(\varepsilon_1)
\end{align*}
for two signs $\varepsilon_1, \varepsilon_2 \in \{\pm 1\}$, according to  axioms of Chevalley systems.

By contradiction, suppose that $\varepsilon_2 = -\varepsilon_1$ and $\operatorname{char}(\widetilde{\mathbb{K}}) \neq 2$.
Denote by $\widetilde{\gamma} = r_{\widetilde{\alpha}}(\widetilde{\beta})$.

Consider the element $n:=m_{\widetilde{\alpha}} m_{\widetilde{\beta}} m_{\widetilde{\alpha}} \in \widetilde{N_1}$.
Since $\widetilde{N_1}$ normalizes $\widetilde{T_1}$, we have 
\[n^2 \widetilde{T_1} = \left( m_{\widetilde{\alpha}} m_{\widetilde{\beta}} m_{\widetilde{\alpha}} \right)^2 \widetilde{T_1} = m_{\widetilde{\alpha}} m_{\widetilde{\beta}} \widehat{\widetilde{\alpha}}(-1)m_{\widetilde{\beta}}m_{\widetilde{\alpha}} \widetilde{T_1} = m_{\widetilde{\alpha}} \widehat{\widetilde{\beta}}(-1) m_{\widetilde{\alpha}} \widetilde{T_1} = \widehat{\widetilde{\alpha}}(-1) \widetilde{T_1} = \widetilde{T_1}.\]
Then $n \widetilde{T_1}$ has order at most $2$ in $\widetilde{N_1}/\widetilde{T_1}$ and therefore, since $\widetilde{T_1}$ has exponent $2$, we deduce that $n$ has order at most $4$.

Then the matrix realisation of $m_{\widetilde{\alpha}} m_{\widetilde{\beta}} m_{\widetilde{\alpha}}$ in the universal covering $\widetilde{G}^{\widetilde{\gamma}}$ of the subgroup spanned by $U_{-\widetilde{\gamma}}(\widetilde{\mathbb{K}})$ and $U_{\widetilde{\gamma}}(\widetilde{\mathbb{K}})$ given in Notation~\ref{NotParamNonMult} is
\[\begin{pmatrix}1 & \varepsilon_1\\0&1\end{pmatrix}
\begin{pmatrix}1 & 0\\-\varepsilon_2&1\end{pmatrix}
\begin{pmatrix}1 & \varepsilon_1\\0&1\end{pmatrix}=
\begin{pmatrix}2 & 3\varepsilon_1\\\varepsilon_1&2\end{pmatrix}.\]

If $\operatorname{char}(\widetilde{\mathbb{K}}) = 7$,
since $\begin{pmatrix}2 & \varepsilon_1\\3\varepsilon_1&2\end{pmatrix}^2 = \begin{pmatrix}7 & 12\varepsilon_1\\4\varepsilon_1&7\end{pmatrix}$
does not centralizes $\begin{pmatrix} 2 & 0\\ 0 & 4\end{pmatrix}$ which is the matrix realization of $\widehat{\widetilde{\gamma}}(2) \in \widetilde{T_1}$ in $G^{\widetilde{\gamma}}$, it contradicts the fact that $n^2$ is in the commutative group $\widetilde{T_1}$.

If $\operatorname{char}(\widetilde{\mathbb{K}}) \neq 7$, since $\begin{pmatrix}2 & \varepsilon_1\\3\varepsilon_1&2\end{pmatrix}^4 = \begin{pmatrix}97 & 168\varepsilon_1\\56\varepsilon_1&97\end{pmatrix}$ is not an homothety, we get a contradiction with the order of $n$ that is at most $4$ since $\pi : G^{\widetilde{\gamma}} \simeq \mathrm{SL}_2(\widetilde{K}) \to \langle U_{-\widetilde{\gamma}}(\widetilde{\mathbb{K}}), U_{\widetilde{\gamma}}(\widetilde{\mathbb{K}}) \rangle$ is a central isogeny.

As a consequence, in any case, we have that $\varepsilon_2 = \varepsilon_1$, whence
 $m_{\widetilde{\alpha}} m_{\widetilde{\beta}} m_{\widetilde{\alpha}} = m_{r_{\widetilde{\alpha}}(\widetilde{\beta})} \widehat{r_{\widetilde{\alpha}}(\widetilde{\beta})} (\varepsilon_1)$.

Let $w_{\widetilde{\alpha}}$ be the image of $m_{\widetilde{\alpha}}$ in $\widetilde{W}_1:=\widetilde{N_1} / \widetilde{T_1}$ for any $\widetilde{\alpha} \in \widetilde{\Phi}$.
We have seen that $w_{\widetilde{\alpha}} w_{\widetilde{\beta}} w_{\widetilde{\alpha}} = w_{r_{\widetilde{\alpha}}(\widetilde{\beta})}$ for any $\widetilde{\alpha},\widetilde{\beta} \in \widetilde{\Phi}$. Denote by $\widetilde{\Delta}$ a basis of the root system $\widetilde{\Phi}$.
Note that the $w_{\widetilde{\alpha}}$ for $\widetilde{\alpha} \in \widetilde{\Delta}$ generate $\widetilde{W}_1$.
Indeed, let $\widetilde{\gamma} \in \widetilde{\Phi}$ and write it $\widetilde{\gamma} =w(\widetilde{\alpha})$ for $\widetilde{\alpha} \in \widetilde{\Delta}$ and $w \in W(\widetilde{\Phi})$ since $W(\widetilde{\Phi}) \cdot \widetilde{\Delta} = \widetilde{\Phi}$.
Since the $r_{\widetilde{\beta}}$ generate $W(\widetilde{\Phi})$, one can write $w = r_{\widetilde{\beta}_1} \circ \cdots \circ r_{\widetilde{\beta}_n}$ for some $\widetilde{\beta}_1,\dots,\widetilde{\beta}_n \in \widetilde{\Delta}$.
Thus, we have $w_{\widetilde{\gamma}} = w_{\widetilde{\beta}_1} \cdots w_{\widetilde{\beta}_n} w_{\widetilde{\alpha}} w_{\widetilde{\beta}_n} \cdots w_{\widetilde{\beta}_1}$.

\textbf{Claim: $\widetilde{W}_1$ is isomorphic to $W(\widetilde{\Phi})$.}

It suffices to observe that the relations characterizing the spherical Coxeter group $W(\widetilde{\Phi})$ are satisfied in $\widetilde{W}_1$ by the family of generators $(w_{\widetilde{\alpha}})_{\widetilde{\alpha}\in\widetilde{\Delta}}$.
Let $\widetilde{\alpha}\neq \widetilde{\beta} \in \widetilde{\Delta}$ and $n(\widetilde{\alpha},\widetilde{\beta})\in \Z_{\geqslant 2}$ be the order of $r_{\widetilde{\alpha}} \circ r_{\widetilde{\beta}} \in W(\widetilde{\Phi})$.
We claim that $(m_{\widetilde{\alpha}} m_{\widetilde{\beta}})^{n(\widetilde{\alpha},\widetilde{\beta})} \in \widetilde{T}_1$.
Indeed, $(m_{\widetilde{\alpha}} m_{\widetilde{\beta}})^{n(\widetilde{\alpha},\widetilde{\beta})-1} m_{\widetilde{\alpha}} = m_{\widetilde{\alpha}} m_{\widetilde{\beta}} \cdots m_{\widetilde{\beta}} m_{\widetilde{\alpha}}= m_{\widetilde{\gamma}} \mod \widetilde{T}_1$ for some $\widetilde{\gamma} \in \widetilde{\Phi}$.
Multiplying by $m_{\widetilde{\beta}}$ on the right and applying ${}^{v\!}\nu$, we get that ${}^{v\!}\nu\left( m_{\widetilde{\gamma}} m_{\widetilde{\beta}} \right) = {}^{v\!}\nu\left( ( m_{\widetilde{\alpha}} m_{\widetilde{\beta}} )^{n(\widetilde{\alpha},\widetilde{\beta})} \right)$.
Whence $m_{\widetilde{\alpha}} T = m_{\widetilde{\beta}} T$.
Because the classes $M_{\widetilde{\alpha}}$ are pairwise disjoint and we took exactly one $m_{\widetilde{\alpha}}$ in each, we deduce that $m_{\widetilde{\gamma}} = m_{\widetilde{\beta}}$.
Hence $(m_{\widetilde{\alpha}} m_{\widetilde{\beta}})^{n(\widetilde{\alpha},\widetilde{\beta})}\widetilde{T}_1 = m_{\widetilde{\beta}}^2 \widetilde{T}_1 = \widetilde{T}_1$.
Thus $\widetilde{W}_1$ is a quotient of $W(\widetilde{\Phi})$.

Conversely, applying $^{v\!}\nu$ to the $m_{\widetilde{\alpha}} \in \widetilde{N}_1 \subset N$, we obtain a surjective group homomorphism $\widetilde{W}_1 \to W(\widetilde{\Phi})$, whence $\widetilde{W}_1 = W(\widetilde{\Phi})$.
In particular, $\widetilde{N}_1$ is finite as an extension of finite groups.
\end{proof}

\begin{Lem}[see {\cite[4.2.9]{BruhatTits2}}]\label{LemChoiceOfOrigin}
Let $m_\alpha$ for $\alpha \in \Phi$ be defined as in Notations~\ref{NotMaNonMult} and~\ref{NotMaMult}.
There is a point $o \in V$ such that $\nu(m_\alpha)(o) = o$ for every $\alpha \in \Phi$.
\end{Lem}

\begin{proof}
According to Facts~\ref{FactMaNonMult} and~\ref{FactMaMult}, the subgroup $N_1$ of $\mathbf{N}(\mathbb{K})$ generated by the $m_\alpha$ for $\alpha \in \Phi$ is contained in the subgroup of $\mathbf{N}(\widetilde{\mathbb{K}})$ generated by the $m_{\widetilde{\alpha}}$ for $\widetilde{\alpha} \in \widetilde{\Phi}$.
Thus, according to Lemma~\ref{LemN1finite}, the group $N_1$ is a finite subgroup of $N = \mathbf{N}(\mathbb{K})$.
Since $V$ is, in particular, an $\mathbb{R}$-vector space on which $N$ acts by affine transformation, the group $N_1$ has to fix a point $o \in V$.
In particular, this point is fixed by the $m_\alpha$ for $\alpha \in \Phi$.
\end{proof}

\begin{Not}\label{NotAffineApartment}
We denote by $\mathbb{A}$\index[notation]{a@$\mathbb{A}$} the $\RF^S$-aff space with underlying free $\Z$-module of finite type $V_\Z$ and affine space $V$ with origin $o$\index[notation]{o@$o$} chosen as in Lemma~\ref{LemChoiceOfOrigin}.

Thus, we have by definition that $\nu(m_\alpha) = r_{\alpha,0} =o - \alpha( \cdot - o) \alpha^\vee$
for any $\alpha \in \Phi$.
\end{Not}

\begin{Prop}\label{PropCompatibleAction}
The action $\nu: \mathbf{N}(\mathbb{K}) \to \operatorname{Aff}_{\RF^S}(\mathbb{A})$ satisfies~\ref{axiomCA1}, \ref{axiomCA2} and~\ref{axiomCA3}.\index{compatible action}
\end{Prop}

\begin{proof}
The condition~\ref{axiomCA1} is a consequence of the commutation of the diagram (\ref{diagramAffineAction}) since it is well-known in reductive groups that $W \simeq N/T$ naturally identifies with $W(\Phi)$.

We use the notations of sections~\ref{subsecParametrization} and~\ref{raff}.b).
Let $\alpha \in \Phi$ and $u \in U_\alpha \setminus \{1\}$.
Consider the unique element $m(u) \in M_\alpha$ given by \cite[6.1.2(2)]{BruhatTits1}.
If $\alpha$ is non-multipliable, we deduce from Fact~\ref{FactMaNonMult} that $m(u)^2 = \widehat{\alpha}(-1) \in \mathbf{T}(\mathbb{K})_b \subset \ker \nu$, whence axiom~\ref{axiomCA3} is satisfied.
If $\alpha$ is multipliable, we deduce from Fact~\ref{FactMaMult} that $m(u)^2 = \widehat{\alpha}({^\tau\!}v^{-1} v) \in \mathbf{T}(\mathbb{K})_b \subset \ker \nu$ for some couple $(u',v) \in H(\mathbb{L},\mathbb{L}_2) \setminus \{(0,0)\}$, whence axiom~\ref{axiomCA3} is satisfied.

Consider $t = m_\alpha m(u) \in T$.
Let $v \in V_\Rtot$ be such that $\nu\big(m(u)\big) = \big(r_\alpha, v \big)$.
We have that $\nu(m(u))(o) - o = v$, whence $\alpha\Big(\nu\big(m(u)\big)(o) - o\Big) = \alpha(v)$.
On the one hand, $\nu(t) = \nu\big(m_\alpha m(u)\big) = r_{\alpha,0} \circ \big(r_\alpha, v \big)$ is the translation by $r_\alpha(v) \in V_\Rtot$.
Hence $\alpha(\nu(t)) = \alpha(r_\alpha(v)) = - \alpha(v)$.
On the other hand, according to Lemma~\ref{LemActionOfT}, we have $\alpha( \nu(t) ) = - \omega(\widetilde{\alpha}(t))$ for any $\widetilde{\alpha} \in \widetilde{\Phi}$ such that $\widetilde{\alpha}|_{\mathbf{S}} = \alpha$.

If we show that $\omega(\widetilde{\alpha}(t)) = - 2 \varphi_\alpha(u)$, then condition~\ref{axiomCA2} will be proven.

If $\alpha$ is non-multipliable and non-divisible, then one can write $u = x_\alpha(z)$ with $z \in \mathbb{L}_\alpha^*$.
Then $m(u) = m_{-\alpha}(z^{-1}) =m_{-\alpha} \widehat{-\alpha}(z)$ according to Notation~\ref{NotMaNonMult} and Fact~\ref{FactMaNonMult}.
Thus $m_\alpha m(u) = - \widehat{-\alpha}(z)$ and therefore $\widetilde{\alpha}(t) = (-\widetilde{\alpha})(t)^{-1} = (-\widetilde{\alpha})( - \widehat{-\alpha}(z))^{-1} = -z^{-2}$.
Thus $\omega( \widetilde{\alpha}(t)) = -2 \omega(z) = -2 \varphi_\alpha(u)$.

If $\alpha$ is multipliable, then one can write $u = x_\alpha(y,z)$ with $(y,z) \in H(\mathbb{L}_\alpha,\mathbb{L}_{2\alpha}) \setminus \{(0,0)\}$.
Then $m(u) = m_{-\alpha}(y,z) = m_{-\alpha} \widehat{-\alpha}(z)$ according to Notation~\ref{NotMaMult}.
Thus $m_\alpha m(u) = \widehat{-\alpha}(z)$ and therefore $\widetilde{\alpha}(t) = (-\widetilde{\alpha})(t)^{-1}= (-\widetilde{\alpha})( \widehat{-\alpha}(z))^{-1} = z {^\tau\!}z^{-2}$.
Thus $\omega( \widetilde{\alpha}(t)) = - \omega(z) = -2 \varphi_\alpha(u)$.

If $\alpha$ is divisible, then there are $\beta \in \Phi$, $\widetilde{\beta} \in \widetilde{\Phi}$ and $\tau \in \operatorname{Gal}(\mathbb{L}_\beta/ \mathbb{L}_{2\beta})$ such that $\alpha = 2 \beta$, $\widetilde{\beta}|_{\mathbf{S}} = \beta$ and $\widetilde{\alpha} = \widetilde{\beta} + \tau(\widetilde{\beta})$.
By definition $m_\beta = m_\alpha$.
In particular, $u \in U_\beta$ and $t = m_\beta m(u)$.
Thus, we have shown that $\omega(\widetilde{\beta}(t)) = - 2 \varphi_\beta(u)$.
Hence, using~\ref{axiomV4}, we get $\omega( \widetilde{\alpha}(t)) = \omega(\widetilde{\beta}(t)) + \omega(\tau(\widetilde{\beta})(t)) = - 4 \varphi_\beta(u) = - 2 \varphi_\alpha(u)$.
\end{proof}

\subsection{The \texorpdfstring{$\RF^S$}{RS}-building of a quasi-split reductive group}\label{buildredgp}

In this section, we have defined a generating root group datum $\Big(T,\big(U_\alpha,M_\alpha\big)_{\alpha\in\Phi}\Big)$ (in the sense of Definition~\ref{DefRGD}) together with an $\RF^S$ valuation $\left( \varphi_\alpha \right)_{\alpha\in \Phi}$ of this root group datum (in the sense of Definition~\ref{DefValuation}, see Proposition~\ref{Vtotal}).
Moreover, in Notation~\ref{NotAffineApartment} we defined an $\RF^S$-aff space $\mathbb{A}$ together with an action $\nu$ of $N$ by $\RF^S$-aff transformations given by commutative diagram (\ref{diagramAffineAction}).
According to Proposition~\ref{PropCompatibleAction}, this action is compatible with the valuation in the sense of Definition~\ref{DefCompatibleAction}.
Thus assumption~\ref{HypCAVRGD} is satisfied and by construction $\RF^S = \RF^S_\Q$.
Thus, we provided a datum as in assumptions of section~\ref{sectionLambdaBuildingFromVRGD} so that we can define, as in Definition~\ref{DefLambdaBuildingFromDatum}, the following space:

\begin{definition}\label{defBuilding_quasi_split_group}
The \textbf{$\RF^S$-building associated to the quasi-split reductive $\mathbb{K}$-group $\mathbf{G}$}\index{building!associated to a quasi-split reductive group} is:
$$\mathcal{I}(\mathbf{G})=\mathcal{I}(\mathbb{K},\omega,\mathbf{G}):=\mathcal{I}\left(G,T,(U_{\alpha})_{\alpha\in\Phi},(M_{\alpha})_{\alpha\in \Phi},(\varphi_{\alpha})_{\alpha\in\Phi},\nu\right).$$
\index[notation]{i@$\mathcal{I}(\mathbb{K},\omega,\mathbf{G})$}
\end{definition}

\subsection{Extended Affine Weyl group}

This section is dedicated to describing more precisely the extended affine Weyl group $\Wext= \nu(N) \simeq N/T_b$.\index{Weyl group!extended affine}

\subsubsection{Quadratic Galois extensions}

This first part is devoted to the proof of some useful lemmas on quadratic Galois extensions of valued fields.
We consider a quadratic Galois extension $L/K$ of valued fields, and we assume that $L/K$ is univalent (see~\ref{assumpfield}).
We denote by $\tau$ the non-trivial element of $\operatorname{Gal}(L/K)$, by $N : x \in L \mapsto x \tau(x) \in K$ the norm map and by $T : x \in L \mapsto x + \tau(x) \in K$ the trace map.

\begin{Not}
For $u \in K$, we denote by $L^u = \{y \in L,\ T(y) = u\}$ and by
\[L^u_{\max} = \left\{y \in L^u,\ \forall z \in L^0,\ \omega(y) \geqslant \omega(y+z)\right\}\]
the subset, possibly empty, of elements $y \in L^u$ such that the map $z \in L^0 \mapsto \omega(y+z)$ admits a maximum reached at $z=0$.

Note that $0 \in L^u \Leftrightarrow u = 0$ and that the $L^u$ are $1$-dimensional $K$-affine spaces forming a partition of $L$.
Indeed, $L^u \neq \emptyset$ because the trace map is surjective when $L/K$ is separable.
\end{Not}

\begin{Lem}[{see \cite[4.2.20(i)]{BruhatTits2}}]\label{LemLmaxNonempty}
The following are equivalent:
\begin{enumerate}[label={(\roman*)}]
\item\label{LemLmaxNonempty1} $\forall u \in K^\times,\ L^u_{\max} \neq \emptyset$;
\item\label{LemLmaxNonempty2} $L^1_{\max} \neq \emptyset$.
\end{enumerate}
If these conditions are satisfied, then $\omega(L^u_{\max}) = \omega(L^1_{\max}) + \omega(u)$.
\end{Lem}

\begin{proof}
Obviously, \ref{LemLmaxNonempty1} implies \ref{LemLmaxNonempty2}.
Conversely, let $u \in K^\times$ and suppose that $y \in L^1_{\max} \neq \emptyset$.
Set $x = uy \in L^u$.
For any $z \in L^0$, we have
\[\omega(x) = \omega(y) + \omega(u) \geqslant \omega(y+z) + \omega(u) = \omega(x + uz).\]
Since the map $L^0 \to L^0$ given by $z \mapsto uz$ is a bijection, we have proven that $x \in L^u_{\max}$ and that $\omega(u) + \omega(L^1_{\max}) \subset \omega(L^u_{\max})$.
Conversely, any $x \in L^u_{\max}$ can be written $x = uy$ with $y = \frac{x}{u} \in L^1_{\max}$.
Thus $\omega(L^u_{\max}) = \omega(L^1_{\max}) + \omega(u)$.
\end{proof}

\begin{Lem}\label{Lem1/2inL1max}
In any case, $\omega(L^1) \subset \Lambda'_{\leqslant 0}$.

If $\operatorname{char}(K) = 2$ or $\omega(L^\times) = \omega(K^\times)$, then $\omega\big(L^0 \setminus \{0\}\big) = \omega\big(K^\times\big)$.

If $\operatorname{char}(K) \neq 2$ and $\omega(L^\times) \neq \omega(K^\times)$, then $\frac{1}{2} \in L^1_{\max}$ and $\omega(L^\times) = \omega(L^0 \setminus \{0\}) \sqcup \omega(K^\times)$.
\end{Lem}

\begin{proof}
Let $x \in L^1$.
By contradiction, suppose that $\omega(x) > 0$.
Then $\omega(N(x)) = 2 \omega(x) = \omega(x(1-x)) = \omega(x) + \omega(1-x)$.
But $\omega(1-x) = \omega(1) = 0$ since $\omega(x) > 0$.
Hence $\omega(x) = 2 \omega(x)$ which is a contradiction.
Thus $\omega(L^1) \subset \Lambda'_{\leqslant 0}$.

If $\operatorname{char}(K) = 2$, then $ L^0 = \ker (\tau - \operatorname{id}) = K$.

If $\omega(L^\times) = \omega(K^\times)$ and $x_0 \in L^0$, then $L^0 = x_0 K$ and $\omega\big(L^0 \setminus \{0\} \big) = \omega(x_0) + \omega\big(K^\times\big) = \omega\big(K^\times\big)$ since $\omega(x_0) \in \omega\big(K^\times\big)$.

If $\operatorname{char}(K) \neq 2$, then there exists $x \in K$ such that $L \simeq K[t] / (t^2-x)$.
If, moreover, $\omega(K^\times) \neq \omega(L^\times)$, we have that $\omega(t) = \frac{1}{2} \omega(x) \not\in \omega(K^\times)$.
We have that $T(t) = 0$ so that $L^0 = t K$ and $L^1 = \frac{1}{2} + tK$.
In particular, we have $\omega(L^0 \setminus \{0\}) = \omega(t) + \omega(K^\times)$ with $\omega(t) \not\in \omega(K^\times)$ so that $\omega(L^\times) = \Big(  \omega(t) + \omega(K^\times) \Big) \sqcup \omega(K^\times)$.
Moreover $\omega(\frac{1}{2}) \in \omega(K^\times)$.
Hence $\omega(\frac{1}{2} + ty) = \min\big( \omega(\frac{1}{2}),\omega(t) + \omega(y)\big)$ for any $y \in K$ since $\omega(\frac{1}{2}) \neq \omega(ty)$.
Hence $\frac{1}{2} \in L^1_{\max}$.
\end{proof}

\begin{Ex}\label{ExGaloisRamified}
Consider $K = \Q_2(\!( x )\!)$ and the rupture field $L = K[t] / (t^2-f)$ for some $f \in K^\times \setminus {K^\times}^2$.
Consider the canonical valuation $\omega: K \to \mathbb{Z} \times \mathbb{Z} \cup \{\infty\}$ given by $\omega(2) = (0,1)$ and $\omega(x) = (1,0)$.
Assume that $\omega(f) \not \in 2 \omega(K^\times)$ and write it $\omega(f) \in (a,b) + 2 \omega(K^\times)$ with $(a,b) \in \{(0,1), (1,0), (1,1)\}$.
The valuation uniquely extends to $L$.
Then
\[L^0 = \{y+tz,\ y,z \in K,\ T(y+tz) = 2y = 0\} = tK\]
Thus $\omega(L^0 \setminus \{0\}) = \omega(t) + \omega(K^\times) = \left( \frac{a}{2}, \frac{b}{2} \right) + \mathbb{Z} \times \mathbb{Z}$.
\[L^1 = \{y+tz,\ y,z \in K,\ T(y+tz) = 2y = 1\} = \frac{1}{2} + tK\]
One can check that $\frac{1}{2} \in L^1_{\max}$.

In the context of Lemma~\ref{LemSetOfValues}.\ref{LemSetOfValuesB}, the set $2\Gamma_{\alpha}$ will then be :
\begin{align*}2\Gamma_\alpha
&= \omega(L^0 \setminus \{0\}) \cup \left( \omega(\frac{1}{2}) + \omega(N(L^\times)) \right)\\
&=\left(\frac{1}{2} \omega(f) + \omega(K^\times) \right)
\cup \left( \omega(\frac{1}{2}) + 2 \omega(K^\times) \right)
\cup \left( \omega(\frac{1}{2}) + \omega(f) + 2 \omega(K^\times) \right)\\
&=\bigg(\Big(\frac{a}{2}+ \mathbb{Z}\Big) \times \Big( \frac{b}{2} + \mathbb{Z} \Big) \bigg)
\cup \bigg( 2 \mathbb{Z} \times (1 + 2\mathbb{Z}) \bigg)
\cup \bigg( \Big( a + 2 \mathbb{Z}\Big) \times \Big( (b-1) + 2 \mathbb{Z}\Big) \bigg).
\end{align*}
One can easily check that $\frac{1}{2} \omega(L^\times) = \widetilde{\Gamma}_\alpha = \langle \Gamma_\alpha - \Gamma_\alpha \rangle \supset \Gamma_\alpha$ and that $0 \in \Gamma_\alpha$ if, and only if, $(a,b) = (0,1)$.
\end{Ex}

\subsubsection{Sets of values}

\begin{Lem}[{see \cite[4.2.21]{BruhatTits2}}]\label{LemSetOfValues}
Let $\alpha \in \Phi_{\mathrm{nd}}$.
\begin{enumerate}
\item\label{LemSetOfValues1} Suppose that $\alpha$ is non-multipliable. Then $\Gamma_\alpha = \Gamma'_\alpha = \omega(\mathbb{L}_\alpha^\times)$.
\item\label{LemSetOfValues2} Suppose that $\alpha$ is multipliable.
\begin{enumerate}
\item\label{LemSetOfValuesA} If $\left(\mathbb{L}_\alpha\right)^1_{\max} = \emptyset$, then
$\Gamma'_\alpha = \emptyset$ and $\Gamma_\alpha= \frac{1}{2} \omega\left( \left(\mathbb{L}_{\alpha}\right)^0 \setminus \{0\}\right)$.
\item\label{LemSetOfValuesB} If $x \in \left(\mathbb{L}_\alpha\right)^1_{\max} \neq \emptyset$, then
$\Gamma'_\alpha = \frac{1}{2} \omega(x) + \frac{1}{2} \omega\left(N \left(\mathbb{L}_{\alpha}^\times\right)\right)$ 
and
\[\Gamma_\alpha = \frac{1}{2} \omega\Big(\mathbb{L}_\alpha^0 \setminus \{0\}\Big) \cup \bigg( \frac{1}{2} \omega(x) + \frac{1}{2} \omega\Big( N\big( \mathbb{L}_\alpha^\times \big) \Big) \bigg).\]
\item\label{LemSetOfValuesC} In both cases $ \Gamma_{2\alpha} = \Gamma'_{2\alpha} = \omega\left( \left(\mathbb{L}_{\alpha}\right)^0 \setminus \{0\}\right)$.
\end{enumerate}
\end{enumerate}
\end{Lem}

\begin{proof}
(\ref{LemSetOfValues1})
By Fact~\ref{FactDecompositionSetOfValue}, $\Gamma_{\alpha} = \Gamma'_\alpha$ since $\Gamma_{2\alpha} = \emptyset$.
From definitions and notations, we have $\Gamma_\alpha = \varphi_\alpha(U_\alpha \setminus \{1\})  = \{\omega(y) = \varphi_\alpha(x_\alpha(y)),\ y \in \mathbb{L}_\alpha\} = \omega(\mathbb{L}_\alpha^\times)$.

(\ref{LemSetOfValues2})
By definition,
\begin{align*}
\Gamma_{2\alpha} = \Gamma'_{2\alpha} 
&= \bigg\{\varphi_{2\alpha}(u),\ u \in U_{2\alpha} \setminus \{1\} \bigg\}\\
&= \bigg\{ \omega(y),\ (0,y) \in H(\mathbb{L}_\alpha,\mathbb{L}_{2\alpha}) \setminus \{(0,0)\} \bigg\}\\
& = \bigg\{ \omega(y),\ y \in \mathbb{L}_\alpha^0 \setminus \{0\} \bigg\}
\end{align*}
Thus, we get \ref{LemSetOfValuesC}.

By definitions and notations,
\begin{align*}
\Gamma'_{\alpha}  =& \bigg\{ \varphi_\alpha(u),\ u \in U_\alpha \setminus \{1\} \text{ and } U_{\alpha,\varphi_\alpha(u)} = \bigcap_{v \in U_{2\alpha}} U_{\alpha,\varphi_\alpha(uv)}\bigg\}\\
=& \bigg\{ \frac{1}{2} \omega(y),\ (x,y) \in H(\mathbb{L}_\alpha,\mathbb{L}_{2\alpha}) \setminus \{(0,0)\} \text{ and } \forall (x',y') \in H(\mathbb{L}_\alpha,\mathbb{L}_{2\alpha}),\\
& \quad \Big(\frac{1}{2} \omega(y') \geqslant \frac{1}{2} \omega(y) \Leftrightarrow \forall (0,z) \in H(\mathbb{L}_\alpha,\mathbb{L}_{2\alpha}),\ \frac{1}{2} \omega(y') \geqslant \frac{1}{2} \omega(y+z) \Big) \bigg\}\\
 =& \bigg\{ \frac{1}{2}\omega(y),\ (x,y) \in H(\mathbb{L}_\alpha,\mathbb{L}_{2\alpha}) \setminus \{(0,0)\} \text{ and } \forall z \in \mathbb{L}_{\alpha}^0,\ \omega(y) \geqslant \omega(y+z) \bigg\}\\
=& \bigcup_{u \in N\left( \mathbb{L}_{\alpha}^\times \right)} \frac{1}{2} \omega\Big(\big(\mathbb{L}_\alpha\big)^u_{\max}\Big).
\end{align*}
Hence, according to Lemma~\ref{LemLmaxNonempty} and Fact~\ref{FactDecompositionSetOfValue}, we get \ref{LemSetOfValuesA}.

From now on, assume that $x \in \left(\mathbb{L}_\alpha\right)^1_{\max} \neq \emptyset$.
From Lemma~\ref{LemLmaxNonempty}, we deduce that
\[\Gamma'_\alpha = \frac{1}{2} \omega(x) + \frac{1}{2} \omega\Big(N\big(\mathbb{L}_{\alpha}^\times\big)\Big)\]
so that
\[\Gamma_{\alpha} = \frac{1}{2} \Gamma'_{2\alpha} \cup \Gamma'_\alpha = \bigg( \frac{1}{2} \omega\Big(\mathbb{L}_\alpha^0 \setminus \{0\}\Big)\bigg) \cup \bigg(\frac{1}{2} \omega(x) + \frac{1}{2} \omega\Big( N\big(\mathbb{L}_{\alpha}^\times\big)\Big) \bigg).\]
Hence we get \ref{LemSetOfValuesB}.
\end{proof}

\begin{Prop}\label{PropSemiSpecialValuation}
Let $\alpha\in \Phi_{\mathrm{nd}}$.
\begin{itemize}
\item If $\alpha$ is non-multipliable, then $\widetilde{\Gamma}_\alpha = \omega\left(\mathbb{L}_\alpha^\times\right) = \Gamma_\alpha$.
\item If $\alpha$ is multipliable and $\left(\mathbb{L}_\alpha\right)^1_{\max} = \emptyset$, then
$\widetilde{\Gamma}_\alpha = \frac{1}{2} \omega\left(\mathbb{L}_{2\alpha}^\times\right) = \Gamma_\alpha$.
\item If $\alpha$ is multipliable and $\left(\mathbb{L}_\alpha\right)^1_{\max} \neq \emptyset$, then $\widetilde{\Gamma}_\alpha = \frac{1}{2} \omega\left(\mathbb{L}_{\alpha}^\times\right) \supset \Gamma_\alpha$.
\end{itemize}
Moreover, if $0 \not\in \Gamma_\alpha$, then $\alpha$ is multipliable, $\omega(2) \not\in \{0,\infty\}$ and $\omega(\mathbb{L}_\alpha^\times) \neq \omega(\mathbb{L}_{2\alpha}^\times)$.
\end{Prop}

In particular, the valuation  $\left( \varphi_\alpha \right)_{\alpha \in \Phi}$ is semi-special, and this corollary provides a sufficient condition for the valuation to be special.
Example~\ref{ExGaloisRamified} shows that this sufficient condition is not necessary.

\begin{proof}
The non-multipliable case is treated by Lemma~\ref{LemSetOfValues}.
Assume that $\alpha$ is multipliable and observe that $\Gamma_\alpha \subset \frac{1}{2}  \omega(\mathbb{L}_{\alpha}^\times)$ in any case.

If $\omega(\mathbb{L}_\alpha^\times) = \omega(\mathbb{L}_{2\alpha}^\times)$, then $\omega(\mathbb{L}_\alpha^0 \setminus\{0\}) = \omega(\mathbb{L}_{2\alpha}^\times) = \Gamma_{2\alpha}$.
If $\left(\mathbb{L}_\alpha\right)^1_{\max} = \emptyset$, then $\Gamma_\alpha = \frac{1}{2}  \omega(\mathbb{L}_{2\alpha}^\times)$.
If $\left(\mathbb{L}_\alpha\right)^1_{\max} \neq \emptyset$, then $\frac{1}{2}  \omega(\mathbb{L}_{2\alpha}^\times) = \frac{1}{2}  \omega(\mathbb{L}_{\alpha}^\times) \subset \Gamma_\alpha \subset \frac{1}{2} \omega(\mathbb{L}_{\alpha}^\times)$ so that $\Gamma_\alpha = \frac{1}{2}  \omega(\mathbb{L}_{\alpha}^\times)$.
In both cases, $\Gamma_\alpha = \widetilde{\Gamma}_\alpha$ since it is a group.

If $\omega(2) = \infty$, then $\mathbb{L}_{\alpha}^0 = \mathbb{L}_{2\alpha}$.
Hence $\Gamma_{2\alpha} = \omega(\mathbb{L}_{2\alpha}^\times)$ and $0 \in \Gamma_{\alpha}$.
Thus $\Gamma_\alpha \subset \widetilde{\Gamma}_\alpha$ and we conclude distinguishing the cases on $\left(\mathbb{L}_\alpha\right)^1_{\max}$.

Otherwise, $\operatorname{char}\left(\mathbb{L}_{2\alpha}\right) \neq 2$ and $\omega(\mathbb{L}_\alpha^\times) \neq \omega(\mathbb{L}_{2\alpha}^\times)$.
Hence, by Lemma~\ref{Lem1/2inL1max}, we have $\frac{1}{2} \in \left(\mathbb{L}_\alpha\right)^1_{\max} \neq \emptyset$.
Let $t,x$ as in proof of Lemma~\ref{Lem1/2inL1max} such that $\mathbb{L}_{\alpha} \simeq \mathbb{L}_{2\alpha}[t] / (t^2 - x)$.
On the one side, we have $\frac{1}{2} \omega(\frac{t}{2} ) \in \frac{1}{2} \omega(L^0 \setminus \{0\}) = \frac{1}{2} \Gamma_{2\alpha}$ and $\frac{1}{2} \omega( \frac{1}{2}) \in \Gamma'_\alpha$.
Hence $\frac{1}{2} \omega(\frac{t}{2} ) - \frac{1}{2} \omega( \frac{1}{2}) = \frac{1}{2} \omega( t ) \in \widetilde{\Gamma}_\alpha$.
On the other side, for any $y \in \mathbb{L}_{2\alpha}^\times$, we have that 
$\frac{1}{2} \omega(ty ) \in \frac{1}{2} \Gamma_{2\alpha}$ and $\frac{1}{2} \omega( t ) \in \frac{1}{2} \Gamma_{2\alpha}$.
Hence $\frac{1}{2} \omega(ty ) - \frac{1}{2} \omega( t ) = \frac{1}{2} \omega( y ) \in \widetilde{\Gamma}_\alpha$.
As a consequence, we have the chain of inclusions
\[\frac{1}{2} \omega(\mathbb{L}_{\alpha}^\times) = \frac{1}{2} \omega(\mathbb{L}_{2\alpha}^\times) \sqcup \frac{1}{2} \omega(t) + \frac{1}{2} \omega(\mathbb{L}_{2\alpha}^\times)  \subset \widetilde{\Gamma}_\alpha \subset \langle \Gamma_\alpha \rangle \subset \frac{1}{2} \omega(\mathbb{L}_{\alpha}^\times)\]
that are, in fact equalities.
We deduce that $\Gamma_\alpha \subset \widetilde{\Gamma}_\alpha = \frac{1}{2} \omega(\mathbb{L}_{\alpha}^\times)$.
\end{proof}

\subsubsection{Consequences for the extended affine Weyl group}

\begin{Cor}\label{corAffine_weyl_group}
We have $\mathbf{N}(\mathbb{K}) / T_b \simeq \Wext = \nu(\mathbf{T}(\mathbb{K})) \rtimes W^v$ with $\nu(\mathbf{T}(\mathbb{K})) \subset \left( \bigoplus_{\alpha \in \Delta} \widetilde{\Gamma}_\alpha \varpi_\alpha^\vee \right)$.
\end{Cor}

\begin{proof}
We know that $\mathbf{T}(\mathbb{K}) = \ker {^v\!}\nu$ by \cite[6.1.11(ii)]{BruhatTits1} so that we deduce the equality.
By Proposition~\ref{PropSemiSpecialValuation}, the valuation is semi-special.
Thus we deduce the inclusion from Proposition~\ref{PropAffineWeylGroup}.
\end{proof}

\begin{Cor}\label{CorWeylsssc}
Suppose furthermore that the quasi-split reductive group $\mathbf{G}$ is simply connected semi-simple.
Let $\Delta$ be a basis of $\Phi$ and suppose that for any $\alpha \in \Delta$ multipliable, we have that $\left(\mathbb{L}_\alpha\right)^1_{\max} \neq \emptyset$.
Then 
$\Waff = \Wext = \left( \bigoplus_{\alpha \in \Delta} \omega\left( \mathbb{L}_\alpha^\times \right) \widehat{\alpha} \right) \rtimes W^v$ where $\widehat{\alpha} = \alpha^\vee$ when $\alpha$ is non-multipliable and $\widehat{\alpha} = \frac{1}{2} \alpha^\vee = (2\alpha)^\vee$ when $\alpha$ is multipliable.
\end{Cor}

Note that $\widehat{\alpha}$ is the basis of $\Phi^\vee$ canonically associated to $\Delta$.
The assumption is satisfied for instance when $\operatorname{char}(\mathbb{K}) \neq 2$ and the minimal extension $\widetilde{\mathbb{K}} / \mathbb{K}$ splitting $\mathbf{G}$ is totally ramified.

\begin{proof}
Let $\widetilde{\Delta}$ be the basis of $\widetilde{\Phi}$ that restricts in $\Delta$ and denote $\widehat{\widetilde{\alpha}}$ the associated coroots as in statement.
Then, the $\widehat{\widetilde{\alpha}}$ form a basis of $\widetilde{\Phi}^\vee$ and therefore of $X_*(\mathbf{T}_{\widetilde{\mathbb{K}}})$.
Hence, we have that
\[\mathbf{T}_{\widetilde{\mathbb{K}}} = \prod_{\widetilde{\alpha} \in \widetilde{\Delta}} \widehat{\widetilde{\alpha}}(\mathbb{G}_{m,\widetilde{\mathbb{K}}}) = \prod_{\alpha \in \Delta} \prod_{\widetilde{\alpha} \in \alpha}  \widehat{\widetilde{\alpha}}(\mathbb{G}_{m,\widetilde{\mathbb{K}}})\]
since $\mathbf{G}_{\widetilde{\mathbb{K}}}$ is simply connected semi-simple and split.
Consider the Galois group $\Sigma = \operatorname{Gal}(\widetilde{\mathbb{K}} / \mathbb{K})$ and its $*$-action on $\Phi$ as recalled in Notation~\ref{NotStarAction}.
Let $\alpha \in \Delta$ and fix an absolute root $\widetilde{\alpha} \in \alpha$.
Denote by $\Sigma_\alpha = \Sigma / \Sigma_{\widetilde{\alpha}}$ the quotient where $\Sigma_{\widetilde{\alpha}}$ is the stabilizer of $\widetilde{\alpha}$ as in Definition~\ref{DefSplittingField}.
Then $\alpha = \Sigma_\alpha \cdot \widetilde{\alpha}$.
Then, we have
\[\prod_{\widetilde{\alpha} \in \alpha}  \widehat{\widetilde{\alpha}}(\mathbb{G}_{m,\widetilde{\mathbb{K}}})
=\prod_{\sigma \in \Sigma_\alpha}  \sigma \cdot \widehat{\widetilde{\alpha}}(\mathbb{G}_{m,\widetilde{\mathbb{K}}})
=\widehat{\alpha}\left( R_{\mathbb{L}_\alpha / \mathbb{K}}(\mathbb{G}_{m,\mathbb{L}_\alpha}) \right)_{\widetilde{\mathbb{K}}}.\]
Hence
\[\mathbf{T} = \mathbf{T}_{\widetilde{\mathbb{K}}}^\Sigma = \prod_{\alpha \in \Delta} \widehat{\alpha}\left( R_{\mathbb{L}_\alpha / \mathbb{K}}(\mathbb{G}_{m,\mathbb{L}_\alpha}) \right)\]
and therefore
\[\mathbf{T}(\mathbb{K}) = \prod_{\alpha \in \Delta}  \widehat{\alpha}(\mathbb{L}_\alpha^\times).\]

Let $t = \widehat{\alpha}(x) \in \widehat{\alpha}(\mathbb{L}_\alpha^\times)$
For any $\beta \in \Delta$, we have that $\beta(\nu(t)) = - \omega \circ \beta(t) = - \beta(\widehat{\alpha}) \omega(x)$.
Hence $\nu(t) = -\omega(x) \widehat{\alpha} \in \omega(\mathbb{L}_\alpha^\times) \widehat{\alpha}$.
Therefore, $\nu(\mathbf{T}(\mathbb{K})) = \bigoplus_{\alpha \in \Delta} \omega(\mathbb{L}_\alpha^\times) \widehat{\alpha}$.
But for $\alpha \in \Delta$, we apply Proposition~\ref{PropSemiSpecialValuation} under the assumption $\left(\mathbb{L}_\alpha\right)^1_{\max} \neq \emptyset$:
\begin{itemize}
\item if $\alpha$ is non-multipliable we have $\widehat{\alpha} = \alpha^\vee$ and $\widetilde{\Gamma}_\alpha = \omega(\mathbb{L}_\alpha^\times)$;
\item if $\alpha$ is multipliable we have $\widehat{\alpha} = \frac{1}{2} \alpha^\vee$ and $\widetilde{\Gamma}_\alpha = \frac{1}{2} \omega(\mathbb{L}_\alpha^\times)$.
\end{itemize}
In both cases, we have that $\widetilde{\Gamma}_\alpha \alpha^\vee = \omega(\mathbb{L}_\alpha^\times) \widehat{\alpha}$.
Using Lemma \ref{PropAffineWeylGroup}, we get $\Waff \cap V_{\mathfrak{R}^S} = \bigoplus_{\alpha \in \Delta} \widetilde{\Gamma}_\alpha \alpha^\vee = \bigoplus_{\alpha \in \Delta} \omega(\mathbb{L}_\alpha^\times) \widehat{\alpha} = \nu(\mathbf{T}(\mathbb{K}))$.
Hence $\Wext = \Waff$.
\end{proof}

\begin{Ex}\label{ExSUhGalois}
Let $\mathbb{K} = K = \mathbb{Q}_2(\!( x )\!)$ and $\widetilde{\mathbb{K}} =L = K[t]/(t^2-f)$ a quadratic Galois extension with $f \in K^{\times} \smallsetminus K^{\times 2}$ (as in Example~\ref{ExGaloisRamified}).
Denote by $\Lambda = \omega(K^\times) = \mathbb{Z}^2$ the totally ordered abelian group, ordered by lexicographic order and pick, for instance $\Rtot = \mathfrak{R}^S = \mathbb{R}[t]/(t^2)$ ordered by lexicographic order.
Denote by $(a,b) = 2 \omega(f) \in \mathbb{Z}^2 \smallsetminus 2 \cdot \mathbb{Z}^2$. Up to a multiplying $f$ by a uniformizer, one can assume that $(a,b) \in \{(1,0),(0,1),(1,1)\}$.

Consider the simply connected semi-simple quasi-split group $\mathbf{G} = \mathrm{SU}(h)$ defined over the extension $L/K$ as before and the natural faithful linear representation $G=\mathrm{SU}(h)(K) \to \mathrm{SL}_3(L)$.
Let $\mathbf{T}$ be the maximal torus consisting in diagonal matrices parametrized by $\widehat{\alpha} : R_{L/K}(\mathbb{G}_{m,L}) \to \mathbf{T}$ as in Notation~\ref{NotParamMult}.
Then $\mathbf{T}$ is the centralizer of the maximal split rank-$1$ torus $\mathbf{S} = \widehat{\alpha}(\mathbb{G}_{m,K})$.
The cocharacter module $X_*(\mathbf{S})$ is the rank-$1$ $\mathbb{Z}$-module spanned by $\widehat{\alpha}_{|\mathbb{G}_{m,K}} = \varpi_\alpha$, whence $\mathbb{A} =X_*(\mathbf{S}) \otimes_\mathbb{Z} \Rtot$ is the $\Rtot$-module $\Rtot \varpi_\alpha$.
The set of walls, that also are the vertices, of $\mathbb{A}$ is given by $\Gamma_\alpha \varpi_\alpha$. Explicitely, it is the set:
\[\Gamma_\alpha \varpi_\alpha =\bigg(\Big(\big(\frac{a}{2},\frac{b}{2}\big) + \mathbb{Z}^2\Big) \cdot \frac{\varpi_\alpha}{2} \bigg)
\cup \bigg( \Big( \big(0,1) + 2 \cdot \mathbb{Z}^2 \Big) \cdot \frac{\varpi_\alpha}{2} \bigg)
\cup \bigg( \Big( \big( a,b \big) + \big( 0,1 \big) + 2 \cdot \mathbb{Z}^2 \Big) \cdot \frac{\varpi_\alpha}{2} \bigg)\]
according to Example~\ref{ExGaloisRamified}.

The root system of $\mathrm{SU}(h)$ with respect to $\mathbf{S}$ is $\{\pm \alpha,\pm 2\alpha\}$, we that $W^\vee = \{1,r_\alpha\}$
Since $\mathrm{SU}(h)$ is simply connected, we know by Corollary~\ref{CorWeylsssc} that $\Wext = \Waff = \bigoplus \omega(L^\times) \varpi_\alpha \rtimes W^\vee$.
Note that $\omega(L^\times) \varpi_\alpha = 2 \widetilde{\Gamma}_\alpha \varpi_\alpha = \widetilde{\Gamma}_\alpha \alpha^\vee$. It is the subgroup of translations of $\Wext$. Explicitely, it is the free $\mathbb{Z}$-module of rank $2$:
\[\widetilde{\Gamma}_\alpha \alpha^\vee = (2,0) \mathbb{Z} \cdot \varpi_\alpha + (0,2) \mathbb{Z} \cdot \varpi_\alpha + (a,b) \mathbb{Z} \cdot \varpi_\alpha\]
Thus, one can observe that the action of $\Wext = \Waff$ on $\Gamma_\alpha \varpi_\alpha$ defines 3 orbits that correspond to conjugacy classes of maximal parahoric subgroups (i.e. maximal compact subgroups here). These orbits are:
\begin{itemize}
    \item $\bigg( \Big\{ \big(\frac{a}{2},\frac{b}{2}\big), \big(\frac{3a}{2},\frac{3b}{2}\big) \Big\} + 2 \cdot \mathbb{Z}^2 \bigg) \cdot \frac{\varpi_\alpha}{2}$
    \item $\bigg( \Big\{ \big(\frac{3a}{2},\frac{b}{2}\big), \big(\frac{a}{2},\frac{3b}{2}\big) \Big\} + 2 \cdot \mathbb{Z}^2 \bigg) \cdot \frac{\varpi_\alpha}{2}$
    \item $\bigg( \Big\{ \big(0,1\big), \big(a,b+1) \Big\} + 2 \cdot \mathbb{Z}^2 \bigg) \cdot \frac{\varpi_\alpha}{2}$
\end{itemize}
We have $T_b = \mathbf{T}(K) \cap \mathrm{SL}_3(\mathbb{O}_L) = \widehat{\alpha}(\mathbb{O}_L^\times)$.
The fact that $0 \not\in \Gamma_\alpha$ when $a=1$ means that the subgroup $\mathrm{SU}(h) \cap \mathrm{SL}_3(\mathbb{O}_L)$ is not a maximal parahoric subgroup.
In fact, it is contained in the two maximal parahoric subgroups $P_{x_+}$ and $P_{x_-}$ with $x_+ = \min (\Gamma_\alpha \cap \Lambda_{>0}) \cdot \varpi_\alpha$ and $x_- = \max (\Gamma_\alpha \cap \Lambda_{<0}) \cdot \varpi_\alpha$.
\end{Ex}

\subsection{The case of a split reductive group}

In this paragraph, we focus on the case of a split reductive group $\mathbf{G}$ over a field $\mathbb{K}$ equipped with a non-trivial valuation $\omega$. Let $\Lambda = \omega(\mathbb{K}^\times)$ be the non-trivial totally ordered abelian group of valuation of $\mathbb{K}$.

Since $\mathbf{G}$ is split, $\widetilde{\mathbb{K}} = \mathbb{K}$ by definition and we know that the root system $\widetilde{\Phi} = \Phi$ is reduced \cite[7.4.4]{Springer}.
In fact, by \cite[5.1.5]{Demazure}, $\mathbf{G}$ can be realized as the scalar extension to $\mathbb{K}$ of a Chevalley-Demazure group scheme over $\mathbb{Z}$, still denoted by $\mathbf{G}$.
One can pick a (split) maximal torus $\mathbf{T} = \mathbf{S}$ defined over $\mathbb{Z}$ together with root groups $\mathbf{U}_\alpha$ for $\alpha\in \Phi$.
One can pick a Chevalley system $\left(x_\alpha\right)_{\alpha\in \Phi}$ parametrizing the root groups over $\mathbb{Z}$ \cite[5.4.2]{Demazure}, so that $x_\alpha: \mathbb{G}_{a,\mathbb{Z}} \stackrel{\simeq}{\longrightarrow} \mathbf{U}_\alpha$ for any $\alpha \in \Phi$.

Recall that for any sub-ring $R$ of $\mathbb{K}$, there are natural embeddings of groups $\mathbf{T}(R) \subset \mathbf{G}(R)$, $\mathbf{U}_\alpha(R) \subset \mathbf{G}(R)$, $\mathbf{T}(R) \subset \mathbf{T}(\mathbb{K})$, $\mathbf{U}_\alpha(R) \subset \mathbf{U_\alpha}(\mathbb{K})$ and $\mathbf{G}(R) \subset \mathbf{G}(\mathbb{K})$.
Here, we will focus on the valuation ring $\mathbb{O} := \omega^{-1}(\RF^S_{\geqslant 0} \cup \{\infty\})$ of $\mathbb{K}$.

Denote $G = \mathbf{G}(\mathbb{K})$, $T = \mathbf{T}(\mathbb{K})$, $N = \mathcal{N}_\mathbf{G}(\mathbf{T})(\mathbb{K})$ and $U_\alpha = \mathbf{U}_\alpha(\mathbb{K})$ as done in section~\ref{subsecRvaluation}.
Note that the elements $m_\alpha = x_\alpha(1) x_{-\alpha}(-1) x_\alpha(1)$ defined in Notation~\ref{NotMalphaInSTsystems} belongs to $N \cap \mathbf{G}(\mathbb{O})$.
Denoting $M_\alpha = T m_\alpha$, the generating root group datum of $G$ is given by $\Big(T, \big( U_\alpha, M_\alpha \big)_{\alpha \in \Phi}\Big) $.

The parametrization $x_\alpha : \mathbb{K} \to U_\alpha$ identifies $U_{\alpha,\lambda}$ with $\{x \in \mathbb{K},\ \omega(x) \geqslant \lambda\}$ for any $\lambda \in \RF^S$ and $\alpha \in \Phi$.
The valuation maps $\varphi_\alpha : x_\alpha(x) \in U_\alpha \mapsto \omega(x)$, defined in Notation~\ref{NotRSvaluation}, take values in $\Lambda \cup \{\infty\}$.
As a consequence:
\begin{Fact}\label{FactGamma_split}
For any $\alpha \in \Phi$, we have $\Gamma_\alpha= \Gamma'_\alpha = \Lambda$.

In particular, we note that $0 \in \Lambda = \Gamma_\alpha$ for any $\alpha \in \Phi$ and therefore the $\RF^S$-valuation $\left(\varphi_\alpha\right)_{\alpha \in \Phi}$ is special.

If $\mathbf{G}$ is a semi-simple simply connected split $\mathbb{K}$-group, then the associated  extended affine Weyl group is $\Wext = \left( \bigoplus_{\alpha^\vee \in \Delta^\vee} \Lambda \alpha^\vee \right) \rtimes W(\Phi)$ where $\Delta^\vee$ is a basis of coroots (see Corollary~\ref{CorWeylsssc}).
\end{Fact}

Let us interpret the natural subgroups $\mathbf{G}(\mathbb{O})$, $\mathbf{T}(\mathbb{O})$ and $\mathbf{U}_\alpha(\mathbb{O})$ in terms of parahoric subgroups.

\begin{Lem}\label{LemTUoverO}
We have $\mathbf{U}_\alpha(\mathbb{O}) = U_{\alpha,0}$ for any $\alpha \in \Phi$ and $\mathbf{T}(\mathbb{O}) = \mathbf{T}_b(\mathbb{K}) \subseteq T_b$, with equality whenever the split $\mathbb{K}$-group $\mathbf{G}$ is semi-simple.
\end{Lem}

\begin{proof}
The first equality has been explained. Let us explain the second one.
Recall that $\mathbf{T}_b(\mathbb{K}) = \{t \in \mathbf{T}(\mathbb{K}):\ \forall \chi \in X_{\mathbb{K}}^*(\mathbf{T}),\, \omega(\chi(t)) = 0\}$.
Since $\mathbf{T}$ is split, it can be parametrized by a basis of cocharacters $(\lambda_i)$.
Thus, for any $t \in \mathbf{T}(\mathbb{O})$ and any $\chi \in X^*(\mathbf{T})$, we have that $\chi(t) \in \mathbb{O}^\times$ and therefore $\omega(\chi(t)) = 0$.
Hence $\mathbf{T}(\mathbb{O}) \subseteq \mathbf{T}_b(\mathbb{K})$.

Conversely, let $t \in \mathbf{T}_b(\mathbb{K})$ and write it $t = \prod_i \lambda_i(x_i)$ for some parameters in $\mathbb{K}^\times = \mathbb{G}_{m}(\mathbb{K})$.
Since the dual pairing is perfect, consider the antedual basis $(\chi_i)$ of $(\lambda_i)$.
Then $0 = \omega(\chi_i(t)) = \omega(x_i)$ and therefore $x_i \in \mathbb{O}^\times$ which means that $t \in \mathbf{T}(\mathbb{O})$.
Hence $\mathbf{T}_b(\mathbb{K}) \subseteq \mathbf{T}(\mathbb{O})$.
The comparison between $\mathbf{T}_b(\mathbb{K})$ and $T_b = \mathbf{T}(\mathbb{K})_b$ is given by Remark~\ref{RkTbK}.
\end{proof}

\begin{Ex}\label{ExTbGLn}
Take $\mathbf{G} = \mathrm{GL}_n$ and $\mathbf{T}$ the split maximal torus of diagonal matrices.
Then the $\Big(\chi_i: \operatorname{diag}(x_1,\dots,x_n) \mapsto x_i\Big)$ form a basis of $X^*(\mathbf{T})$ and roots are given by $\alpha_{i,j} = \chi_i - \chi_{j}$ for $i \neq j$.
Here, $\ker \nu = \mathbf{T}(\mathbb{K})_b = \mathcal{Z}_{\mathbf{G}}(\mathbb{K}) \cdot \mathbf{T}(\mathbb{O}) \supsetneq \mathbf{T}_b(\mathbb{K}) = \mathbf{T}(\mathbb{O})$.
\end{Ex}

Let $\mathbb{A}$ be the $\RF^S$-aff space defined as in Notation~\ref{NotAffineApartment} and pick $o := 0 \in V_{\RF^S}$ as origin of the standard apartment $\mathbb{A}$.

\begin{Prop}
The point $o$ is a special vertex and $P_o = \widehat{P}_o$.

If the split group $\mathbf{G}$ is semi-simple, then $\mathbf{G}(\mathbb{O}) = P_o$ is a parahoric subgroup.
\end{Prop}

\begin{proof}
For any $\alpha \in \Phi$, we have $\alpha(o) = 0$ and therefore $U_{\alpha,o} = U_{\alpha,0} = \mathbf{U}_\alpha(\mathbb{O})$ by Lemma~\ref{LemTUoverO}.
By definition, $m_\alpha = x_\alpha(1)x_{-\alpha}(-1) x_\alpha(1) \in U_o$ for any $\alpha \in \Phi$.
Therefore $W_o = {^v}\nu(\widetilde{N}_o) = W^v$.
Hence $o$ is a special vertex and $\widehat{P}_o = P_o$.

Since $P_o$ is generated by $T_b = \mathbf{T}(\mathbb{O})$ and the $U_{\alpha,o} = \mathbf{U}_\alpha(\mathbb{O})$, it is contained in $\mathbf{G}(\mathbb{O})$.
Conversely, by \cite[1.6]{Abe}, we know that $\mathbf{G}(\mathbb{O})$ is generated by $\mathbf{T}(\mathbb{O})$ and the $\mathbf{U}_\alpha(\mathbb{O})$ for $\alpha \in \Phi$ since $\mathbb{O}$ is a local ring (being the valuation ring of a valued field).
Hence $\mathbf{G}(\mathbb{O}) = P_o$.
\end{proof}

\section{Projection maps}\label{projectionsec}

\subsection{A remark on Hahn's embedding theorem}\label{subRemark_Hahn}

By Hahn's embedding theorem (see for instance \cite{Gravett}), there exists an increasing embedding $\iota: \Lambda \rightarrow \mathfrak{R}^{\mathrm{rk}(\Lambda)}$ of the totally ordered abelian group $\Lambda$ into the totally ordered real vector space $\mathfrak{R}^{\mathrm{rk}(\Lambda)}$. Set:
$$S:=\{ \min \mathrm{Supp}(\iota(\lambda)) \, | \,   \lambda \in \Lambda \setminus \{0\}  \} \subseteq \mathrm{rk}(\Lambda),$$
and let $\iota': \Lambda \rightarrow \mathfrak{R}^S$ be the composite of $\iota:\Lambda \rightarrow \mathfrak{R}^{\mathrm{rk}(\Lambda)}$ with the projection $\mathfrak{R}^{\mathrm{rk}(\Lambda)} \rightarrow \mathfrak{R}^S$. The map $\iota'$ is a non-decreasing group homomorphism with trivial kernel. It is therefore an increasing embedding of $\Lambda$ into $\mathfrak{R}^S$. Moreover, $\iota'$ induces an increasing bijection
  $\iota'_{\mathrm{rk}}: \mathrm{rk}(\Lambda) \rightarrow S$ such that, for any $x \in \Lambda \setminus \{0\}$, if $[x]$ stands for the Archimedian class of $x$ in $\Lambda$:
  $$\min \mathrm{Supp}(\iota'(\lambda)) = \iota'_{\mathrm{rk}}([\lambda]).$$

Therefore, up to replacing $\iota$ by $\iota'$, we may and do assume that $\iota: \Lambda \rightarrow \mathfrak{R}^{\mathrm{rk}(\Lambda)}$ is an increasing embedding such that, for any $\lambda \in \Lambda \setminus \{0\}$, the minimum of the support of $ \iota(\lambda)$ is the Archimedian class of $\lambda$. 

\subsection{Construction and explicit description of the fibers}\label{SubsecContructionFibers}

Let $\Lambda_0$ be a convex subgroup of $\Lambda$ such that $\omega(\mathbb{K}^\times) \cap \Lambda_0 \neq 0$ and $\omega(\mathbb{K}^\times)$ is not contained in $\Lambda_0$. Set $\Lambda_1 := \Lambda/\Lambda_0$. The group $\Lambda_1$ is then naturally endowed with the structure of a totally ordered abelian group. The rank $\mathrm{rk}(\Lambda)$ is isomorphic to the set $\mathrm{rk}(\Lambda_1) \amalg \mathrm{rk}(\Lambda_0)$ endowed with the total order such that $s_1 < s_0$ for any $s_0 \in \mathrm{rk}(\Lambda_0)$ and $s_1\in \mathrm{rk}(\Lambda_1)$. Hence $\mathfrak{R}^{\mathrm{rk}(\Lambda_0)}$ is a convex vector subspace of $\mathfrak{R}^{\mathrm{rk}(\Lambda)}$ such that: $$\mathfrak{R}^{\mathrm{rk}(\Lambda)}/\mathfrak{R}^{\mathrm{rk}(\Lambda_0)}\cong \mathfrak{R}^{\mathrm{rk}(\Lambda_1)}$$ as ordered real vector spaces.  We denote by $\pi: \mathfrak{R}^{\mathrm{rk}(\Lambda)}\rightarrow \mathfrak{R}^{\mathrm{rk}(\Lambda_1)}$ the projection.\\

Since $\min \mathrm{Supp}(\iota(\lambda))$ is the Archimedian class of $\lambda$ in $\Lambda$ for any $\lambda \in \Lambda \setminus \{0\}$, we have $\iota(\Lambda_0) \subseteq \mathfrak{R}^{\mathrm{rk}(\Lambda_0)}$. We deduce that $\iota$ induces two increasing embeddings $\iota_0: \Lambda_0 \rightarrow \mathfrak{R}^{\mathrm{rk}(\Lambda_0)}$ and $\iota_1: \Lambda_0 \rightarrow \mathfrak{R}^{\mathrm{rk}(\Lambda_0)}$ such that the following diagram with exact lines commutes:
\begin{equation*}
\xymatrix{
0 \ar[r] & \Lambda_0 \ar[r]\ar[d]^{\iota_0} & \Lambda \ar[r]\ar[d]^{\iota}&\Lambda_1 \ar[r]\ar[d]^{\iota_1}&0\\
0 \ar[r] & \mathfrak{R}^{\mathrm{rk}(\Lambda_0)} \ar[r] & \mathfrak{R}^{\mathrm{rk}(\Lambda)} \ar[r]&\mathfrak{R}^{\mathrm{rk}(\Lambda_1)} \ar[r]&0
}
\end{equation*}

Let $\mathbb{A}$ be an $\mathfrak{R}^{\mathrm{rk}(\Lambda)}$-aff space with underlying $\mathbb{Z}$-module $V$ and let $o$ be a fixed origin of $\mathbb{A}$. 
Similarly, let $\mathbb{A}_1$ be an $\mathfrak{R}^{\mathrm{rk}(\Lambda_1)}$-aff space with underlying $\mathbb{Z}$-module $V$ and let $o_1$ be a fixed origin of $\mathbb{A}_1$. The projection $\pi: \mathfrak{R}^{\mathrm{rk}(\Lambda)}\rightarrow \mathfrak{R}^{\mathrm{rk}(\Lambda_1)}$ induces by tensorization a linear map:
$$\pi_{\mathrm{vect}}: V \otimes \mathfrak{R}^{\mathrm{rk}(\Lambda)}\rightarrow V \otimes \mathfrak{R}^{\mathrm{rk}(\Lambda_1)},$$
and hence an epimorphism:
\begin{align*}
\pi: \mathbb{A}&\rightarrow \mathbb{A}_1\\
x=o+v & \mapsto \pi(x) = o'+\pi_{\mathrm{vect}}(v).
\end{align*}
By abuse of notations, we will denote $\pi$ instead of $\pi_{\mathrm{vect}}$ and $\pi$ in the sequel.

Let now $\omega_1: \mathbb{K} \rightarrow \Lambda_1 \cup \{\infty\}$ be the composite of the valuation $\omega$ followed by the projection $\Lambda \rightarrow \Lambda_1$. It is a non-trivial valuation on $\mathbb{K}$, and hence, all the work that was done in the previous sections for the quasi-split reductive group $\mathbf{G}$ over the valued field $(\mathbb{K},\omega)$ can be done over the valued field $(\mathbb{K},\omega_1)$. In particular, the root group datum:
$$(G,T,(U_{\alpha})_{\alpha \in \Phi}, (M_{\alpha})_{\alpha \in \Phi})$$
associated to $\mathbf{G}$ can be endowed with an $\mathfrak{R}^{\mathrm{rk}(\Lambda)}$-valuation $(\varphi_{\alpha})_{\alpha \in \Phi}$ induced by $\omega$ and with an $\mathfrak{R}^{\mathrm{rk}(\Lambda_1)}$-valuation $(\varphi^1_{\alpha})_{\alpha \in \Phi}$ induced by $\omega_1$. This second valuation can and will be chosen in a way that $\varphi^1_{\alpha}=\pi \circ \varphi_{\alpha}$ for $\alpha \in \Phi$. 

Similarly, the aff spaces $\mathbb{A}$ and $\mathbb{A}_1$ can both be naturally endowed with an action of $N$. More precisely, the $\mathfrak{R}^{\mathrm{rk}(\Lambda)}$-aff space $\mathbb{A}$ can be endowed with an action $\nu$ of $N$ by $\mathfrak{R}^{\mathrm{rk}(\Lambda)}$-aff maps that is compatible with 
the valuation $(\varphi_{\alpha})_{\alpha \in \Phi}$, and the $\mathfrak{R}^{\mathrm{rk}(\Lambda_1)}$-aff space $\mathbb{A}_1$ can be endowed with an action $\nu_1$ of $N$ by $\mathfrak{R}^{\mathrm{rk}(\Lambda_1)}$-aff maps that is compatible with 
the valuation $(\varphi^1_{\alpha})_{\alpha \in \Phi}$. This second action can and will be chosen so that $\nu_1 (n)(\pi(x)) = \pi(\nu(n)(x))$ for $n \in N$ and $x \in \mathbb{A}$.

By gathering all the  previous observations, we can then construct the $\mathfrak{R}^{\mathrm{rk}(\Lambda)}$-building $\mathcal{I}(\mathbb{K},\omega,\mathbf{G})$ associated to $\mathbf{G}$ over $(\mathbb{K},\omega)$ and the $\mathfrak{R}^{\mathrm{rk}(\Lambda_1)}$-building $\mathcal{I}(\mathbb{K},\omega_1,\mathbf{G})$ associated to $\mathbf{G}$ over $(\mathbb{K},\omega_1)$. Moreover, the epimorphism $\pi = \pi: \mathbb{A} \rightarrow \mathbb{A}_1$ induces a surjective map:
$$\pi: \mathcal{I}(\mathbb{K},\omega,\mathbf{G}) \rightarrow \mathcal{I}(\mathbb{K},\omega_1,\mathbf{G})$$
which is compatible with the $G$-action.

\begin{Not} In the sequel, any symbol $\mathfrak{X}$ that has been introduced in the previous sections for $\mathcal{I}(\mathbb{K},\omega,\mathbf{G})$ will be denoted $\mathfrak{X}_1$ when it refers to $\mathcal{I}(\mathbb{K},\omega_1,\mathbf{G})$. For instance, the notation $\Gamma_{\alpha,1}'$ will stand for the set that plays the same role as $\Gamma_{\alpha}$ for the valuation $\omega_1$.
\end{Not}

Take now $g_1\in G$ and $x_1\in \mathbb{A}_1$, 
and consider the point $X_1:=[g_1,x_1]$
in $\mathcal{I}(\mathbb{K},\omega_1,\mathbf{G})$. 
By setting $U_{X_1}:=g_1U_{x_1}g_1^{-1}$, we can define the map:
\begin{align*}
\varphi_{X_1}: U_{X_1} \times \pi^{-1}(\{x_1\}) &\rightarrow \mathcal{I}(\mathbb{K},\omega,\mathbf{G})\\
(u,z) & \mapsto [ug_1,z].
\end{align*}

\begin{proposition}
The image of $\varphi_{X_1}$ is the fiber $\pi^{-1}(\{X_1\})$.
\end{proposition}

\begin{proof}
For any $(u,z) \in U_{X_1} \times \pi^{-1}(\{x_1\})$, we have:
$$\pi(\varphi_{X_1}(u,z))= [ug_1,\pi(z)]=[ug_1,x_1]=u\cdot X_1=X_1.$$
Hence $\mathrm{Im}(\varphi_{X_1}) \subseteq \pi^{-1}(\{X_1\})$. 
Conversely, choose an element $X:=[g,x]$ in the fiber $\pi^{-1}(\{X_1\})$. 
We then have $[g,\pi(x)] = [g_1,x_1] \in \mathcal{I}(\mathbb{K},\omega_1,\mathbf{G})$ 
and hence we can find $n\in N$ such that 
 $\nu(n)(x_1)=\pi(x)$ and $g_1^{-1}gn \in U_{x_1}=g_1^{-1}U_{X_1}g_1$. 
Set $u:=gng_1^{-1}$ and $z:=\nu(n^{-1})(x)$. Then $u\in U_{X_1} $, $z\in \pi^{-1}(\{x_1\})$, and  :
$$\pi(\varphi_{X_1}(u,z))=[gn,\nu(n^{-1})(x)]=gnn^{-1}\cdot [1,x] = g \cdot [1,x] = [g,x]=X.$$
Hence $X \in \mathrm{Im}(\varphi_{X_1})$, and $\mathrm{Im}(\varphi_{X_1}) = \pi^{-1}(\{X_1\})$.
\end{proof}

As a consequence, if $U_{\pi^{-1}(X_1)}$ stands for the pointwise stabilizer of 
$\pi^{-1}(X_1)$ in $U_{X_1}$, then 
for any $u_0 \in U_{\pi^{-1}(X_1)}$, $u \in U_{X_1}$ and 
$z \in \pi^{-1}(\{x_1\})$, we have 
$[ug_1,z] \in \pi^{-1}(X_1)$, so that:
$$\varphi_{X_1}(u_0u,z)=[u_0ug_1,z]=u_0\cdot [ug_1,z] = [ug_1,z] = 
\varphi_{X_1}(u,z).$$
Hence $\varphi_{X_1}$ induces a surjective map:
$$\overline{\varphi}_{X_1}: (U_{X_1}/U_{\pi^{-1}(\{X_1\})}) \times \pi^{-1}(\{x_1\}) \rightarrow \pi^{-1}(\{X_1\}).$$

Consider now two elements $(u,z)$ and $(u',z')$ in $U_{X_1}\times \pi^{-1}(\{x_1\})$
 such that 
$\varphi_{X_1}(u,z)=
\varphi_{X_1}(u',z')$. Then we have 
$[ug_1,z]=[u'g_1,z']$, and hence we can find $n \in N$ such that 
\begin{equation}\label{calglo}
\begin{cases}
z'=\nu(n)(z),\\
g_1^{-1}u^{-1}u'g_1n \in U_{z} \subseteq U_{x_1}.
\end{cases}
\end{equation}
By setting $m:=g_1ng_1^{-1}$, we get 
$m \in U_{X_1}\cap g_1Ng_1^{-1}$ and :
$$\begin{cases}
  z'=\nu(g_1^{-1}mg_1)(z),\\
  u^{-1}u'm \in g_1U_{z}g_1^{-1}.
\end{cases}$$
In other words, if we introduce the groups:
\begin{gather*}
  N_{0,X_1}:= U_{X_1}\cap g_1Ng_1^{-1},\\
  U_{0,z}:=  g_1 U_{z} g_1^{-1},
\end{gather*}
and the group homomorphism:
\begin{align*}
\nu_{0,X_1}: N_{0,X_1} &\rightarrow \mathrm{Aff}_{\mathfrak{R}^{\mathrm{rk}(\Lambda_0)}}\left(\pi^{-1}(\{x_1\})\right)\\
\eta & \mapsto \nu(g_1^{-1}\eta g_1)|_{\pi^{-1}(\{x_1\})}
\end{align*}
where $\mathrm{Aff}_{\mathfrak{R}^{\mathrm{rk}(\Lambda_0)}}\left(\pi^{-1}(\{x_1\})\right)$ has been defined in section \ref{raff}, 
 then $m \in N_{0,X_1}$ and:
\begin{equation} \label{calinfi}
\begin{cases}
    z'=\nu_{0,X_1}(m)(z),\\
  u^{-1}u'm \in U_{0,z}.
\end{cases}
\end{equation}

  Conversely, if we can find $m \in N_{0,X_1}$ satisfying (\ref{calinfi}),
  then we can write $u^{-1}u'm=g_1vg_1^{-1}$ with $v \in U_{z}$, and hence:
  \begin{align*}
    \varphi_{X_1}(u',z') &= [u'g_1,z'] = u'g_1[1,z']\\
    &= ug_1vg_1^{-1}m^{-1}g_1 [1,z'] =ug_1v([1,\nu_{0,X_1} (m^{-1})z']) \\
    &= ug_1v([1,z]) = ug_1[1,z]=u[g_1,z]= \varphi_{X_1}(u,z).
  \end{align*}

We have thus proved the following proposition:

 \begin{proposition} 
Consider the group $\overline{N}_{0,X_1}:= N_{0,X_1}/\left(U_{\pi^{-1}(X_1)} \cap g_1Ng_1^{-1}\right)$ and the group
homomorphism:
$$  \overline{\nu}_{0,X_1}: \overline{N}_{0,X_1} \rightarrow \mathrm{Aff}\left(\pi^{-1}(\{x_1\})\right)$$
induced by $\nu_{0,X_1}$. For each $z\in \pi^{-1}(\{x_1\})$,
 set $\overline{U}_{0,z}:=U_{0,z}/\left( U_{\pi^{-1}(X_1)}\cap g_1U_zg_1^{-1}\right)$. Endow the set $(U_{X_1}/U_{\pi^{-1}(\{X_1\})}) \times \pi^{-1}(\{x_1\})$ 
with the equivalence relation defined in the following way: $(p,z)\sim (p',z')$ if, and only if, 
there exists $n \in \overline{N}_{0,X_1}$ satisfying equations:
\begin{equation} 
  \begin{cases}
      z'=\overline{\nu}_{0,X_1}(n)(z),\\
    p^{-1}p'n \in \overline{U}_{0,z}.
  \end{cases}
  \end{equation}
Then the map $\overline{\varphi}_{X_1}$ induces a bijection:
   $$ \frac{(U_{X_1}/U_{\pi^{-1}(\{X_1\})}) \times \pi^{-1}(\{x_1\})}{\sim} \rightarrow \pi^{-1}(\{X_1\}).$$
 which is compatible with the action of $U_{X_1}/U_{\pi^{-1}(\{X_1\})}$ and which will still be denoted $\overline{\varphi}_{X_1}$.
  \end{proposition}

  We set $\widetilde{W}_{X_1}:=\mathrm{im}(\overline{\nu}_{0,X_1})$.

  \subsection{The root group data axioms for the fibers}
  
Recall that $\Phi_{x_1} = \{\alpha \in \Phi \, | \, -\alpha (x_1) \in \Gamma'_{\alpha,1}\}$. For $\alpha \in \Phi_{x_1}$, set:
\begin{gather*}
U_{0,\alpha}:=U_{X_1}\cap g_1U_{\alpha}g_1^{-1},\\
\overline{U}_{0,\alpha}:=U_{0,\alpha}/\left( U_{\pi^{-1}(\{X_1\})} \cap g_1U_{\alpha}g_1^{-1} \right),\\
S_{0,X_1}:= U_{X_1}\cap g_1Sg_1^{-1},\\
\overline{S}_{0,X_1}:=S_{0,X_1}/\left( U_{\pi^{-1}(\{X_1\})} \cap g_1Sg_1^{-1} \right),\\
T_{0,X_1}:= U_{X_1}\cap g_1Tg_1^{-1},\\
\overline{T}_{0,X_1}:=T_{0,X_1}/\left( U_{\pi^{-1}(\{X_1\})} \cap g_1Tg_1^{-1} \right),\\
M_{0,\alpha}:= U_{X_1}\cap g_1M_{\alpha}g_1^{-1}\subset N_{0,X_1},\\
\overline{M}_{0,\alpha}:=\mathrm{Im}\left( M_{0,\alpha} \hookrightarrow N_{0,X_1} \twoheadrightarrow \overline{N}_{0,X_1}   \right).
\end{gather*}

Note that here $S$ does not stand for a subset of $\mathrm{rk}(\Lambda)$, but for the group of $\mathbb{K}$-rational points of a maximal split torus of $\mathbf{G}$.

\begin{Lem}\label{LemDiegoValuation}
Let $L/K$ be any finite extension of $K$ and consider the extension of the valuations $\omega_1$ and $\omega$ to $L$.
For any $x \in L$, if $\omega_1(x) > 0$ then $\omega(x) >0$.
\end{Lem}

\begin{proof}
We have $\omega(L\setminus \{0\})\subseteq \mathfrak{R}^{\mathrm{rk}(\Lambda)}$ and $\omega_1 = \pi\circ \omega$ where $\pi: \mathfrak{R}^{\mathrm{rk}(\Lambda)} \rightarrow \mathfrak{R}^{\mathrm{rk}(\Lambda_1)}$ is the natural projection. The result then follows from the fact that $\pi$ is non-decreasing.
\end{proof}

\begin{Prop}\label{PropTstar1InTb}
Let $\Omega_1$ be a non-empty subset of $\mathbb{A}(\mathbb{K},\omega_1,\mathbf{G})$. Let $T_{\Omega_1,1}$ be the subgroup of $T$ defined as in notation~\ref{T*not} and associated to $\Omega_1$. We then have $T^*_{\Omega_1,1} \subset T_{b}$.
\end{Prop}

\begin{remark}
Note that, in the previous proposition, the group $T^*_{\Omega_1,1}$ is defined thanks to the apartment $\mathbb{A}(\mathbb{K},\omega_1,\mathbf{G})$ while the group $T_b$ is defined thanks to the apartment $\mathbb{A}(\mathbb{K},\omega,\mathbf{G})$
\end{remark}

\begin{proof}
Let $x_1 \in \Omega_1$.
Since $T^*_{\Omega_1,1} \subset T^*_{x_1,1}$, it suffices to prove the proposition for $\Omega_1 = \{x_1\}$.
Let $a \in \Phi$.
Consider any $u^+ \in U_{a,x_1}$ and any $u^- \in U'_{-a,x_1}$ so that $\varphi_a^1(u^+) + \varphi_{-a}^1(u^-) > 0$.
Let $t(u^+,u^-)$ be the unique element in $T$ given by Lemma~\ref{LemTripleLeviCommutation}\ref{LemTripleLeviCommutation:1} so that $u^- (u^+)^{-1} \in U_{a} t(u^+,u^-) U_{-a}$.

We firstly prove that $\varphi_a(u^+) + \varphi_{-a}(u^-) > 0$.

\paragraph{Case $a$ non-multipliable root:}
Write $u^+ = x_{a}(x)$ and $u^- = x_{-a}(y)$ with $x,y \in L_a$.
By definition $\varphi_a^1(u^+) = \omega_1(x)$ and $\varphi_a^1(u^-) = \omega_1(y)$.
We have $\varphi_a^1(u^+) + \varphi_{-a}^1(u^-) = \omega_1(x) + \omega_1(y) = \omega_1(xy)>0$.
Hence $\varphi_a(u^+) + \varphi_{-a}(u^-) = \omega(x) + \omega(y) = \omega(xy)>0$ by Lemma~\ref{LemDiegoValuation}.

\paragraph{Case $a$ multipliable root:}
Write $u^+ = x_{a}(u,x)$ and $u^- = x_{-a}(v,y)$ with $(u,x),(v,y) \in H(L_a,L_{2a})$.
By definition $\varphi_a^1(u^+) = \frac{1}{2} \omega_1(x)$ and $\varphi_a^1(u^-) = \frac{1}{2} \omega_1(y)$.
We have $\varphi_a^1(u^+) + \varphi_{-a}^1(u^-) = \frac{1}{2} \omega_1(x) + \frac{1}{2} \omega_1(y) = \frac{1}{2} \omega_1(xy)>0$.
Hence $\varphi_a(u^+) + \varphi_{-a}(u^-) = \frac{1}{2} \omega(x) + \frac{1}{2} \omega(y) = \frac{1}{2} \omega(xy)>0$ by Lemma~\ref{LemDiegoValuation}.

Thus, in both cases, we get that $\varphi_a(u^+) + \varphi_{-a}(u^-) > 0$ which implies that $t(u^+,u^-) \in T_b$ according to Lemma~\ref{LemTripleLeviCommutation}\ref{LemTripleLeviCommutation:2}.

According to definition of $T'_{a,x_1}$ and Proposition~\ref{PropTcomponent}, we know that the group $T'_{a,x_1,1}$ is generated by the $t(u^+, u^-)$ for $u^+ \in U_{a,x_1}$ and $u^- \in U'_{-a,x_1}$.
But we have shown that the $t(u^+,u^-)$ all are contained in the group $T_b$, hence we get that $T'_{a,x_1,1} \subset T_b$.
Since, by definition, the group $T^*_{x_1,1}$ is generated by the $T'_{a,x_1,1}$ for $a \in \Phi$, then it is contained in $T_b$.
\end{proof}

\begin{Cor}\label{conjpchap}
For any non-empty subset $\Omega_1 \subset \mathbb{A}_1$, any $p \in \widehat{P}_{\Omega_1}$ and any $u \in U'_{\Omega_1}$, we have $pup^{-1} \in U'_{\Omega_1} T_b$.
\end{Cor}

\begin{proof}
It is an immediate consequence of Proposition~\ref{PropPchapeauDistingueUprimeTstar} and Proposition~\ref{PropTstar1InTb}.
\end{proof}

\begin{lemma}\label{uprime}
For $\alpha \in \Phi$, we have:
$$ \bigcap_{\substack{x \in \mathbb{A}(\mathbb{K},\omega,\mathbf{G})\\ \pi(x)=x_1}} U_{\alpha,x}= U'_{\alpha,x_1}.$$
\end{lemma}

\begin{proof}
Observe that, for any $\lambda_1 \in \mathfrak{R}^{\mathrm{rk}(\Lambda_1)}$:
$$\pi^{-1}\left(]\lambda_1,+\infty]\right)  = \bigcap_{\lambda \in \pi^{-1}(\lambda_1)} [\lambda,+\infty].$$
Hence:
\begin{align*}
U'_{\alpha,x_1} & =(\varphi_{\alpha}^1)^{-1}\left( ]-\alpha(x_1),+\infty] \right)
 = \varphi_{\alpha}^{-1} \left( \pi^{-1} \left( ]-\alpha(x_1),+\infty] \right) \right)
\\&= \varphi_{\alpha}^{-1} \left( \bigcap_{\lambda \in \pi^{-1}(-\alpha(x_1))} [\lambda,+\infty] \right)
= \bigcap_{\lambda \in \pi^{-1}(-\alpha(x_1))} \varphi_{\alpha}^{-1}( [\lambda,+\infty] )
\\&=\bigcap_{\lambda \in \pi^{-1}(-\alpha(x_1))} U_{\alpha,\lambda} .
 \end{align*}
 But:
 $$\pi^{-1}(-\alpha(x_1)) = \{ -\alpha(x) \, | \, x \in \pi^{-1}(x_1)\},$$
 and therefore:
 $$U'_{\alpha,x_1} =\bigcap_{\lambda \in \pi^{-1}(-\alpha(x_1))} U_{\alpha,\lambda} 
 =\bigcap_{\substack{x \in \mathbb{A}(\mathbb{K},\omega,\mathbf{G})\\ \pi(x)=x_1}} U_{\alpha,x}.$$
\end{proof}

\begin{corollary} \label{egal}
 For $\alpha\in \Phi$, we have:
  \begin{gather*}
    U_{X_1} \cap g_1U_{\alpha}g_1^{-1} = g_1U_{\alpha,x_1}g_1^{-1}, \\
    U_{\pi^{-1}(X_1)}\cap g_1U_{\alpha}g_1^{-1} = g_1U'_{\alpha,x_1}g_1^{-1}.
  \end{gather*}
  \end{corollary}
  
  \begin{proof}
   The first equality immediately follows from:
  $$U_{X_1}\cap g_1U_{\alpha}g_1^{-1} = g_1\left( U_{x_1}\cap U_{\alpha} \right) g_1^{-1}
   = g_1 U_{\alpha,x_1} g_1^{-1}.$$
  The second equality is a bit more delicate. If we choose $x \in \mathbb{A}(\mathbb{K},\omega,\mathbf{G})$
  such that $\pi(x) = x_1$ and we set $X:=[g_1,x]\in \pi^{-1}(X_1)$,
  then we have:
  \begin{align*}
    U_{\pi^{-1}(X_1)}\cap g_1U_{\alpha}g_1^{-1} 
    & \subseteq \widehat{P}_{X}\cap g_1U_{\alpha}g_1^{-1}\\
    & \subseteq g_1 \left( \widehat{P}_x \cap U_{\alpha}\right) g_1^{-1} \\
    & \subseteq g_1 U_{\alpha,x} g_1^{-1}.
  \end{align*}
  Hence, by lemma \ref{uprime}:
  $$ U_{\pi^{-1}(X_1)}\cap g_1U_{\alpha}g_1^{-1} 
     \subseteq g_1 \left( \bigcap_{\substack{x\in \mathbb{A}(\mathbb{K},\omega,\mathbf{G})\\ \pi(x)=x_1}} U_{\alpha,x} \right) g_1^{-1}=g_1U'_{\alpha,x_1}g_1^{-1}.$$
  Conversely, let us take $u \in U'_{\alpha,x_1}$ and let us prove that 
  $g_1ug_1^{-1} \in U_{\pi^{-1}(X_1)}$. In other words, we have to check that $g_1 u g_1^{-1}$ 
  fixes $\pi^{-1}(X_1)$. To do so, take $X:= [vg_1,x] \in \pi^{-1}(X_1)$ with $v\in U_{X_1}$ and $x\in \pi^{-1}_{\mathrm{aff}}(x_1)$. By corollary~\ref{conjpchap}, we have $(g_1^{-1}vg_1)u(g_1^{-1}vg_1)^{-1} \in U'_{x_1} T_b$. But the groups $U'_{x_1} $ and $T_b$ both fix $x$. Hence $v(g_1ug_1^{-1})v^{-1} $ fixes $[g_1,x]$, so that $g_1ug_1^{-1}$ fixes $v\cdot[g_1,x]=X$, as wished. 
  \end{proof}

\begin{lemma}
For $\alpha\in \Phi_{x_1}$, the subset $\overline{M}_{0,\alpha}$ of $\overline{N}_{0,X_1}$
 is a right coset of $\overline{T}_{0,X_1}$ in $\overline{N}_{0,X_1}$.
\end{lemma}

\begin{proof}
For any $m,m' \in M_{0,\alpha}$ and $t \in T_{0,X_1}$, 
we have $m^{-1}m' \in T_{0,X_1}$ and 
 $mt \in M_{0,\alpha}$. It is therefore enough to check that $M_{0,\alpha}$ 
 is not empty. To do so, we distinguish two cases.\\

Assume first that $\alpha$ is non-multipliable. 
Since $\alpha \in \Phi_{x_1}$, we have $-\alpha(x_1) \in \Gamma'_{\alpha,1}$, and hence we can find $y \in \mathbb{L}_\alpha$ such that 
$\omega_1(y)=-\alpha(x_1)$. We then have:
$$\begin{cases}
x_{\alpha}(y) \in U_{\alpha,-\alpha(x_1)} = U_{\alpha}\cap U_{x_1}\\
x_{-\alpha}(y^{-1}) \in U_{-\alpha,\alpha(x_1)} = U_{-\alpha}\cap U_{x_1},
\end{cases}$$
so that: 
$$\begin{cases}
g_1x_{\alpha}(y)g_1^{-1} \in U_{0,\alpha}\\
g_1x_{-\alpha}(y^{-1})g_1^{-1} \in U_{0,-\alpha}.
\end{cases}$$
Hence: $$m_{0,\alpha}(y):=g_1 m_{\alpha}(y) g_1^{-1} = g_1x_{\alpha}(y)x_{-\alpha}(y^{-1})x_{\alpha}(y)g_1^{-1} \in M_{0,\alpha}.$$

Assume now that $\alpha$ is multipliable.
Since $\alpha \in \Phi_{x_1}$, we have $\alpha(x_1) \in \Gamma'_{\alpha,1}$ and hence we can find 
 $(y,y') \in H(\mathbb{L}_{\alpha},\mathbb{L}_{2\alpha})$ such that $2\varphi^1_\alpha(x_{\alpha}(y,y'))=\omega_1(y')=2\alpha(x_1)$. We then have:
$$\begin{cases}
x_{\alpha}(yy'^{-1},({^{\tau}}y')^{-1}) \in U_{\alpha,-\alpha(x_1)} = U_{\alpha}\cap U_{x_1}\\
x_{\alpha}(y({^{\tau}}y')^{-1},({^{\tau}}y')^{-1}) \in U_{\alpha,-\alpha(x_1)} = U_{\alpha}\cap U_{x_1}\\
x_{-\alpha}(y,y') \in U_{-\alpha,\alpha(x_1)} = U_{-\alpha}\cap U_{x_1},
\end{cases}$$
so that: 
$$\begin{cases}
g_1x_{\alpha}(yy'^{-1},({^{\tau}}y')^{-1})g_1^{-1} \in U_{0,\alpha}\\
g_1x_{\alpha}(y({^{\tau}}y')^{-1},({^{\tau}}y')^{-1}) \in U_{0,\alpha}\\
g_1x_{-\alpha}(y,y')g_1^{-1} \in U_{0,-\alpha}.
\end{cases}$$
Hence: 
$$m_{0,\alpha}(y,y'):=g_1 m_{\alpha}(y,y') g_1^{-1}=g_1x_{\alpha}(yy'^{-1},({^{\tau}}y')^{-1})x_{-\alpha}(y,y')x_{\alpha}(y({^{\tau}}y')^{-1},({^{\tau}}y')^{-1})g_1^{-1} \in M_{0,\alpha}.$$
\end{proof}

In the sequel, we will keep the notations $m_{0,\alpha}(y)$ and $m_{0,\alpha}(y,y')$ 
that have been used in the previous proof.

\begin{proposition}
The system $\left( \overline{T}_{0,X_1}, (\overline{U}_{0,\alpha})_{\alpha\in\Phi_{x_1}}, (\overline{M}_{0,\alpha})_{\alpha\in\Phi_{x_1}}\right)$ is a generating root group datum in $U_{X_1}/U_{\pi^{-1}(X_1)}$.
\end{proposition}

\begin{proof}~\\

\textbf{Axiom \ref{axiomRGD1}.} Since $\alpha \in \Phi_{x_1}$, we can find 
$u \in U_{\alpha}$ such that 
 $\varphi^1_\alpha(u) = \pi(\varphi_{\alpha}(u))=-\alpha(x_1)$. By lemma~\ref{egal}, we have:
$$g_1ug_1^{-1} \in g_1U_{\alpha}g_1^{-1} \cap \left( U_{X_1} \setminus U_{\pi^{-1}(X_1)} \right).$$ 
Hence $\overline{U}_{0,\alpha}$ is not trivial.\\

\textbf{Axiom \ref{axiomRGD2}.} Consider two roots $\alpha,\beta \in \Phi_{x_1}$ such that $\alpha \not\in -\mathbb{R}_+\beta$ and take an element $\overline{u} \in [\overline{U}_{0,\alpha},\overline{U}_{0,\beta}]$. Fix a lifting $u$ of $\overline{u}$ in $[U_{0,\alpha},U_{0,\beta}]$. By proposition~\ref{Vtotal}, the group $[U_{0,\alpha},U_{0,\beta}]=g_1[U_{\alpha,x_1},U_{\beta,x_1}]g_1^{-1}$ is contained in 
the subgroup of $G$ spanned by the groups $U_{0,\gamma}=g_1U_{\gamma,x_1}g_1^{-1}$ with $\gamma \in \Phi \cap (\mathbb{Z}_{>0}\alpha + \mathbb{Z}_{>0}\beta)$. Hence we can find $\gamma_1,...,\gamma_m \in \Phi \cap (\mathbb{Z}_{>0}\alpha + \mathbb{Z}_{>0}\beta)$ and $u_1 \in U_{0,\gamma_1}$, ..., $u_m \in U_{0,\gamma_m}$ such that $u = u_1...u_m$.  \\

Fix now $i \in \{1,...,m\}$ and let $\overline{u}_i$ be the image of $u_i$ in $U_{X_1}/U_{\pi^{-1}(X_1)}$. Three cases arise:
\begin{itemize}
\item[(i)] \textit{First case:} $\gamma_i \in \Phi_{x_1}$. We then have $\overline{u}_i \in \overline{U}_{0,\gamma_i}$ and $\gamma_i \in \Phi_{x_1} \cap (\mathbb{Z}_{>0}\alpha + \mathbb{Z}_{>0}\beta)$.
\item[(ii)] \textit{Second case:} $\gamma_i \in \Phi \setminus \Phi_{x_1}$ and $u_i \in U_{\pi^{-1}(X_1)}$. We then have $\overline{u}_i=1$.
\item[(iii)] \textit{Third case:} $\gamma_i \in \Phi \setminus \Phi_{x_1}$ and $u_i \in U_{0,\gamma_i} \setminus (U_{\pi^{-1}(X_1)}\cap g_1U_{\gamma_i}g_1^{-1})$. Let's then prove that $2\gamma_i \in \Phi_{x_1}$ and that $\overline{u}_i \in \overline{U}_{0,2\gamma_i}$. By corollary \ref{egal}, we have $u_i \in g_1(U_{\gamma_i,x_1} \setminus U'_{\gamma_i,x_1})g_1^{-1}$, so that $\pi(\varphi_{\gamma_i}(g_1^{-1}u_ig_1))=-\gamma_i(x_1)$. But $\gamma_i \not\in \Phi_{x_1}$, and hence by Lemma~\ref{LemEquivalenceDefPhiOmega} there exists $u'_i \in U_{2\gamma_i}$ such that:
\begin{equation}\label{eqqqq}
-\gamma_i(x_1)=\pi(\varphi_{\gamma_i}(g_1^{-1}u_ig_1))<\pi(\varphi_{\gamma_i}(g_1^{-1}u_ig_1u'_i)).
\end{equation}
 In particular: 
 \begin{equation}\label{eqqqq2}
 \varphi^1_{2\gamma_i}(u'_i)=\pi(\varphi_{2\gamma_i}(u'_i))=2\pi(\varphi_{\gamma_i}(u'_i))=2\pi(\varphi_{\gamma_i}(g_1^{-1}u_ig_1))=-2\gamma_i(x_1).
 \end{equation}
 We deduce that $2\gamma_i \in \Phi_{x_1}$. Moreover $u_i$ can be factored as: 
 $$u_i=(u_ig_1u'_ig_1^{-1}) \cdot (g_1 u_i'^{-1}g_1^{-1})$$
and we have:
$$u_ig_1u'_i g_1^{-1} =g_1(g_1^{-1}u_ig_1u'_i)g_1^{-1}\in g_1 U'_{\gamma_i,x_1} g_1^{-1} = U_{\pi^{-1}(X_1)}\cap g_1U_{\gamma_i}g_1^{-1}$$
by (\ref{eqqqq}) and by corollary \ref{egal}, and:
$$g_1u_i'^{-1}g_1^{-1} \in g_1 U_{2\gamma_i,x_1} g_1^{-1} = U_{0,2\gamma_i}$$
by (\ref{eqqqq2}). We deduce that $\overline{u}_i \in \overline{U}_{0,2\gamma_i}$.
\end{itemize}
The previous three cases show that $\overline{u} = \overline{u}_1...\overline{u}_m$ is contained in 
the subgroup of $U_{X_1}/U_{\pi^{-1}(X_1)}$ spanned by the groups $\overline{U}_{0,\gamma}$ with $\gamma \in \Phi_{x_1} \cap (\mathbb{Z}_{>0}\alpha + \mathbb{Z}_{>0}\beta)$. That is exactly what we wanted to prove.\\

\textbf{Axiom \ref{axiomRGD3}.} If $\alpha$ and $2\alpha$ belong to $\Phi_{x_1}$, 
then $U_{2\alpha} \subseteq U_{\alpha}$, and hence
$\overline{U}_{0,2\alpha} \subseteq \overline{U}_{0,\alpha}$. 
Let's check that this inclusion is strict. The condition that $\alpha \in \Phi_{x_1}$ implies that we can find $(y,y') \in 
H(\mathbb{L}_{\alpha},\mathbb{L}_{2\alpha})$ such that $\mathrm{Tr}_{L_{\alpha}/L_{2\alpha}}(y')=N_{L_{\alpha}/L_{2\alpha}}(y)$ 
and $\omega_1(y')=-2\alpha(x_1)$. 

If $y \neq 0$, then: $$g_1x_{\alpha}(y,y')g_1^{-1} 
\in \left( U_{X_1}\cap g_1U_{\alpha}g_1^{-1}\right) \setminus 
\left(U_{\pi^{-1}(X_1)}\cap g_1U_{\alpha}g_1^{-1}\right)$$
and:
$$g_1x_{\alpha}(y,y')g_1^{-1} 
\not\in \left( U_{X_1}\cap g_1U_{2\alpha}g_1^{-1}\right).$$
Hence the class of $g_1x_{\alpha}(y,y')g_1^{-1} $ in $\overline{U}_{0,\alpha}$ is not 
in $\overline{U}_{0,2\alpha}$.

Now assume that $y=0$. Let $(z,z')$ be any element of $H(\mathbb{L}_{\alpha},\mathbb{L}_{2\alpha})$
with $z \neq 0$. Let $\lambda \in \mathbb{L}_{2\alpha}^{\times}$ such that $\omega_1(\lambda^2z') > \omega_1(y')$. 
We then have:
\begin{gather*}
(\lambda z , \lambda^2z'+y')\in H(\mathbb{L}_{\alpha},\mathbb{L}_{2\alpha})\\
\lambda z \neq 0,\\
\omega_1(\lambda^2z'+y')=-2\alpha(x_1).
\end{gather*} 
Hence the class of $g_1x_{\alpha}(\lambda z , \lambda^2z'+y')g_1^{-1}$ in $\overline{U}_{0,\alpha}$ is not 
in $\overline{U}_{0,2\alpha}$.\\

\textbf{Axiom \ref{axiomRGD4}.} Take any element $\overline{u} \in \overline{U}_{0,-\alpha} \setminus \{1\}$ and 
fix a lifting $$u\in \left(U_{X_1}\cap g_1U_{-\alpha}g_1^{-1}\right) \setminus \left( U_{\pi^{-1}(X_1)}\cap g_1U_{-\alpha}g_1^{-1}\right)$$ 
of $\overline{u}$. \\

Assume first that $\alpha$ is non-multipliable. We can then 
find $y\in L_{-\alpha}^{\times}$ such that: 
$$u=g_1x_{-\alpha}(y^{-1})g_1^{-1}.$$ By 
Corollary~\ref{egal}, we have $\omega_1(y)=-\alpha(x_1)$. Hence:
$$u=(g_1x_{\alpha}(y)^{-1}g_1^{-1})m_{\alpha}(y)(g_1x_{\alpha}(y)^{-1}g_1^{-1}) \in U_{0,\alpha}m_{\alpha}(y)U_{0,\alpha} \subseteq U_{0,\alpha}M_{0,\alpha}U_{0,\alpha}.$$

Now assume that $\alpha$ is multipliable. We can then find $(y,y') \in H(\mathbb{L}_{\alpha},\mathbb{L}_{2\alpha})$
such that:
$$u=g_1x_{-\alpha}(y,y')g_1^{-1}.$$
By Corollary~\ref{egal}, we have $\omega_1(y')=2\alpha(x_1)$. Hence:
\begin{align*}
u&=\Big(g_1x_{\alpha}\big(yy'^{-1},({^{\tau}}y')^{-1}\big)^{-1}g_1^{-1}\Big)\cdot  m_{\alpha}(y,y') \cdot \Big(g_1x_{\alpha}\big(y({^{\tau}}y')^{-1},({^{\tau}}y')^{-1}\big)^{-1}g_1^{-1}\Big)
\\ &\in U_{0,\alpha}m_{\alpha}(y,y')U_{0,\alpha} \subseteq U_{0,\alpha}M_{0,\alpha}U_{0,\alpha}.
\end{align*}

\textbf{Axiom \ref{axiomRGD5}.} Take $\alpha, \beta \in \Phi_{x_1}$ and $\overline{m} \in \overline{M}_{0,\alpha}$. 
 Let $m \in M_{0,\alpha}$ be a lifting of $\overline{m}$. Since 
 $g_1^{-1}mg_1 \in M_{\alpha}$, Proposition~\ref{Vtotal} implies that:
 $$(g_1^{-1}mg_1)U_{\beta}(g_1^{-1}mg_1)^{-1} = U_{r_{\alpha}(\beta)}.$$
Since $m \in U_{X_1}$, we deduce that $mU_{0,\beta}m^{-1} = 
U_{0,r_{\alpha}(\beta)}$, and hence: $$\overline{m}\overline{U}_{0,\beta}\overline{m}^{-1} = 
\overline{U}_{0,r_{\alpha}(\beta)}.$$

\textbf{Axiom \ref{axiomRGD6}.} Let $\overline{U}^{\pm}_{0,X_1}$ be the subgroup of $U_{X_1}/U_{\pi^{-1}(X_1)}$ 
generated by the $\overline{U}_{0,\alpha}$ for $\alpha \in \Phi_{x_1}^{\pm}$. Take $\overline{u}_+ \in \overline{U}^{+}_{0,X_1}$, $\overline{u}_- \in \overline{U}^{-}_{0,X_1}$ and $\overline{t} \in \overline{T}_{0,X_1}$ such that $\overline{t}\overline{u}_+ = \overline{u}_-$. By setting $U_{0,X_1}^{\pm} := \langle U_{0,\alpha}, \alpha \in \Phi_{x_1}^{\pm} \rangle$ , we can find a lifting $u_+$ of $\overline{u}_+$ in $U_{0,X_1}^{+}$, a lifting $u_-$ of $\overline{u}_-$ in $U_{0,X_1}^{-}$, a lifting $t$ of $\overline{t}$ in $T_{0,X_1}$ and an element $u \in U_{\pi^{-1}(X_1)}$ such that $tu_+=uu_-$. Since $\widehat{P}_{g_1\mathbb{A}\cap\pi^{-1}(X_1)} = g_1\widehat{N}_{\mathbb{A}\cap\pi^{-1}(x_1)} U_{\mathbb{A}\cap\pi^{-1}(x_1)}^+ U_{\mathbb{A}\cap\pi^{-1}(x_1)}^-g_1^{-1}$ by Corollary~\ref{CorPchapeauOmegaQC}, we can find $u'_+ \in U_{\mathbb{A}\cap\pi^{-1}(x_1)}^+$, $u'_- \in U_{\mathbb{A}\cap\pi^{-1}(x_1)}^-$ and $n' \in \widehat{N}_{\mathbb{A}\cap\pi^{-1}(x_1)}$ such that $u=g_1n'u'_+u'_-g_1^{-1}$. We therefore have:
$$g_1^{-1}tu_+g_1=n'u'_+u'_-(g_1^{-1}u_-g_1),$$
and hence:
$$ U^+g_1^{-1}tu_+g_1U^-= U^+n'u'_+U^-.$$
But $g_1^{-1}tg_1\in T$, $n' \in N_{\mathbb{A}\cap\pi^{-1}(x_1)} \subseteq T^1_b \subseteq T$ and $T$ normalizes $U^+$. Hence:
$$U^+(g_1^{-1}tg_1)U^-=U^+n'U^-.$$
We deduce from the spherical Bruhat decomposition \cite[6.1.15(c)]{BruhatTits1} that: 
\begin{gather*}
g_1^{-1}tg_1=n',\\
{u'_+}^{-1} (g_1^{-1}u_+g_1) = u'_-(g_1^{-1}u_-g_1) \in U^+\cap U^-.
\end{gather*}
But $U^+\cap U^- = \{1\}$. 
Hence, by using corollary~\ref{egal} and the inclusion of groups $U_{\mathbb{A} \cap \pi^{-1}(x_1)} \subseteq U'_{x_1}$, we get:
\begin{gather*}
u_+=g_1u'_+g_1^{-1} \in g_1U_{\mathbb{A}\cap\pi^{-1}(x_1)}g_1^{-1}\subseteq U_{\pi^{-1}(X_1)},\\
u_-=g_1{u'_-}^{-1}g_1^{-1} \in g_1U_{\mathbb{A}\cap\pi^{-1}(x_1)}g_1^{-1}\subseteq U_{\pi^{-1}(X_1)},\\
t=g_1n'g_1^{-1}=uu_-^{-1}u_+^{-1} \in U_{\pi^{-1}(X_1)},
\end{gather*} 
so that $\overline{u}_+=\overline{u}_-=\overline{t}=1$.\\

\textbf{The root group datum is generating.} Indeed, $U_{X_1}$ is spanned by the $g_1U_{\alpha,x_1}g_1^{-1}$ for $\alpha \in \Phi$. So it suffices to check that:
$$g_1U_{\alpha,x_1}g_1^{-1}\subseteq \langle U_{\pi^{-1}(X_1)},T_{0,X_1},U_{0,\beta} | \beta \in \Phi_{x_1}\rangle$$
for each $\alpha\in \Phi$. To do so, fix a root $\alpha\in \Phi$ and an element $u \in g_1U_{\alpha,x_1}g_1^{-1}$. If $u \in g_1U'_{\alpha,x_1}g_1^{-1}$, then $u \in U_{\pi^{-1}(X_1)}$ by corollary~\ref{egal}. Otherwise, two cases arise:
\begin{itemize}
\item[$\bullet$] if $\alpha \in \Phi_{x_1}$, then corollary~\ref{egal} implies that $U_{0,\alpha}=g_1U_{\alpha,x_1}g_1^{-1}$, and hence $u\in U_{0,\alpha}$.
\item[$\bullet$] if $\alpha \not\in \Phi_{x_1}$, then, by proceeding as in the proof of axiom~\ref{axiomRGD2}, we have $2\alpha \in \Phi_{x_1}$ and $u \in U_{\pi^{-1}(X_1)}U_{0,2\alpha}$.
\end{itemize}
\end{proof}

\subsection{The valuation axioms for the fibers}

Fix $\alpha \in \Phi_{x_1}$. Observe that, according to Corollary~\ref{egal}, given $u\in \left( U_{X_1}\setminus U_{\pi^{-1}(X_1)} \right) \cap g_1U_{\alpha}g_1^{-1}  $ and 
$u'\in U_{\pi^{-1}(X_1)} \cap g_1U_{\alpha}g_1^{-1}$, we have $\pi(\varphi_{\alpha}(g_1^{-1}ug_1)) = -\alpha(x_1)$ 
and $\pi(\varphi_{\alpha}(g_1^{-1}u'g_1)) > -\alpha(x_1)$. Hence: 
$$\varphi_{\alpha}(g_1^{-1}ug_1) < \varphi_{\alpha}(g_1^{-1}u'g_1),$$
so that:
$$\varphi_{\alpha}(g_1^{-1}uu'g_1) = \varphi_{\alpha}(g_1^{-1}ug_1).$$
By choosing an element $\tilde{x}_1 \in \pi^{-1}(\{x_1\})$, we can therefore define the map:
 $$\overline{\varphi}_{\alpha}: \overline{U}_{0,\alpha} \rightarrow \mathfrak{R}^{\mathrm{rk}(\Lambda_0)} \cup \{\infty\}$$
 that sends the class in $\overline{U}_{0,\alpha}$ of an element $u \in U_{X_1} \cap g_1U_{\alpha}g_1^{-1}$
 to $\varphi_{\alpha}(g_1^{-1}ug_1)+\alpha (\tilde{x}_1)$ if $u\not\in U_{\pi^{-1}(X_1)}$ and to $\infty$ otherwise.
 
 \begin{proposition}
The system $\left( \overline{T}_{0,X_1}, (\overline{U}_{0,\alpha})_{\alpha\in\Phi_{x_1}}, (\overline{M}_{0,\alpha})_{\alpha\in\Phi_{x_1}},(\overline{\varphi}_{\alpha})_{\alpha\in\Phi_{x_1}}\right)$ is a valued generating root group datum in $U_{X_1}/U_{\pi^{-1}(X_1)}$.
 \end{proposition}
 
\begin{proof} Observe that axiom \ref{axiomV4} is obviously satisfied. We will therefore use it freely to prove the other axioms.\\

\textbf{Axiom \ref{axiomV0}.} Assume first that $\alpha$ is non-multipliable. By Corollary~\ref{egal}, the image of $\overline{\varphi}_{\alpha}$ contains $\infty$ as well as 
   $\{\lambda \in \Gamma_{\alpha} + \alpha(\tilde{x}_1) | \pi(\lambda)= 0\}$. 
   Since  $\alpha \in \Phi_{x_1}$, we can find $\gamma \in \Gamma_{\alpha}$ such that 
    $-\alpha(x_1) = -\pi(\alpha(\tilde{x}_1))  = \pi(\gamma)$, and hence:
   $$\{\lambda \in \Gamma_{\alpha} + \alpha(\tilde{x}_1) | \pi(\lambda)= 0\}=\{\mu + \gamma + \alpha(\tilde{x}_1) |\ \mu\in (\Gamma_{\alpha} - \gamma) \cap \mathfrak{R}^{\mathrm{rk}(\Lambda_0)}\}.$$
   We have that $\gamma \in \Gamma_{\alpha}$, whence $\Gamma_\alpha - \gamma = \Gamma_\alpha = \omega(\mathbb{L}_\alpha^\times)$ by Proposition~\ref{PropSemiSpecialValuation}. Thus 
   the set $(\Gamma_{\alpha} - \gamma) \cap \mathfrak{R}^{\mathrm{rk}(\Lambda_0)}=\Gamma_{\alpha} \cap \mathfrak{R}^{\mathrm{rk}(\Lambda_0)}$ contains $\omega(\mathbb{K}^\times)\cap \Lambda_0$ and hence it is infinite by assumption on $\Lambda_0$ at beginning of section~\ref{SubsecContructionFibers}.

Now, if $\alpha$ is multipliable, then, by axiom~\ref{axiomV4}, the set $\mathrm{Im}(\overline{\varphi}_{\alpha})$ contains $\frac{1}{2}\mathrm{Im}(\overline{\varphi}_{2\alpha})$, and is hence also infinite.\\

\textbf{Axiom \ref{axiomV1}.}  Fix $\lambda \in \mathfrak{R}^{\mathrm{rk}(\Lambda_0)} \subseteq 
  \mathfrak{R}^{\mathrm{rk}(\Lambda)}$, and take $u,v \in U_{X_1} \cap g_1U_{\alpha}g_1^{-1}$ 
  such that: 
  \begin{gather*}
  \varphi_{\alpha}(g_1^{-1}ug_1)+\alpha(\tilde{x}_1) \geq \lambda,\\
  \varphi_{\alpha}(g_1^{-1}vg_1)+\alpha(\tilde{x}_1) \geq \lambda.
  \end{gather*}
  Then:
  \begin{gather*}
    \varphi_{\alpha}(g_1^{-1}uvg_1)+\alpha(\tilde{x}_1) \geq \min \{ \varphi_{\alpha}(g_1^{-1}ug_1), \varphi_{\alpha}(g_1^{-1}vg_1) \} +\alpha(\tilde{x}_1) \geq 
    \lambda,\\
    \varphi_{\alpha}(g_1^{-1}u^{-1}g_1)+\alpha(\tilde{x}_1)=\varphi_{\alpha}(g_1^{-1}ug_1)+\alpha(\tilde{x}_1) \geq \lambda.
    \end{gather*}
    Hence $\overline{U}_{0,\alpha,\lambda}:= \overline{\varphi}_{\alpha}^{-1}([\lambda,+\infty])$ is a subgroup 
    of $\overline{U}_{0,\alpha}$. Moreover, 
    $\overline{\varphi}_{\alpha}^{-1}(\{\infty\})$ is the trivial
    subgroup of $\overline{U}_{0,\alpha}$ by definition.\\
    
  \textbf{Axiom \ref{axiomV2}.}  Let $\alpha \in \Phi_{x_1}$, $\overline{n} \in \overline{M}_{0,\alpha}$ 
  and $\overline{u} \in \overline{U}_{0,-\alpha}\setminus \{1\}$. Denote by $n$ (resp. 
  $u$) a lifting of $\overline{n}$ to $M_{0,\alpha}$ (resp. of $\overline{u}$
   to $U_{0,-\alpha}$). Since $\overline{u} \neq 1$, 
   we have $u \not\in U_{\pi^{-1}(\{X_1\})}$ and $nun^{-1} \not\in U_{\pi^{-1}(\{X_1\})}$.
   Hence:
   $$\overline{\varphi}_{-\alpha}(\overline{u}) 
   - \overline{\varphi}_{\alpha}(\overline{n}\overline{u}\overline{n}^{-1})
   =\varphi_{-\alpha}(g_1^{-1}ug_1) 
   - \varphi_{\alpha}(g_1^{-1}nun^{-1}g_1).$$
   But the function:
   $$u \in U_{0,\alpha} \mapsto \varphi_{-\alpha}(g_1^{-1}ug_1) 
   - \varphi_{\alpha}(g_1^{-1}nun^{-1}g_1)$$ is constant by proposition~\ref{Vtotal}.\\
   
   \textbf{Axiom \ref{axiomV3}.}  Fix $\alpha, \beta \in \Phi_{x_1}$ and $\lambda, \mu \in \mathfrak{R}^{\mathrm{rk}(\Lambda_0)}$
  such that $\beta \not\in -\mathbb{R}_+ \alpha$. By setting $U_{0,\alpha,\lambda}:= g_1U_{\alpha,\lambda-\alpha(\tilde{x}_1)}g_1^{-1}$ and $\overline{U}_{0,\alpha,\lambda}:= \overline{\varphi}_{\alpha}^{-1}([\lambda,\infty])$, corollary~\ref{egal} implies that: 
  $$\overline{U}_{0,\alpha,\lambda} = \left(U_{0,\alpha,\lambda}\right)/\left( g_1U_{\alpha}g_1^{-1} \cap U_{\pi^{-1}(X_1)}\right) = U_{0,\alpha,\lambda}/\left( g_1U'_{\alpha,x_1}g_1^{-1} \right).$$
 By proposition~\ref{Vtotal}, the group $[U_{0,\alpha,\lambda},U_{0,\beta,\mu}]$ is contained in: 
$$\langle U_{0,p\alpha+q\beta,p\lambda+q\mu}\;|\; p,q\in \mathbb{Z}_{>0},\; p\alpha+q\beta \in \Phi\rangle.$$
Now take $p,q\in \mathbb{Z}_{>0}$ such that $\gamma:=p\alpha+q\beta \in \Phi \setminus \Phi_{x_1}$ and fix an element: $$u \in U_{0,\gamma,p\lambda+q\mu} \setminus \left( g_1U'_{\gamma,x_1}g_1^{-1}\right) \subseteq g_1 \left( U_{\gamma,x_1} \setminus U'_{\gamma,x_1}\right)g_1^{-1}.$$
Since $\pi(\varphi_{\gamma}(g_1^{-1}ug_1)) = -\gamma(x_1)$ and $\gamma\not\in\Phi_{x_1}$, Lemma~\ref{LemEquivalenceDefPhiOmega} implies that there exists $u' \in U_{2\gamma}$ such that:
$$-\gamma(x_1) = \pi(\varphi_{\gamma}(g_1^{-1}ug_1)) < \pi(\varphi_{\gamma}(g_1^{-1}ug_1u')),$$
and we necessarily have: 
\begin{gather*}
\varphi_{2\gamma}(u')=2\varphi_{\gamma}(u')=2\varphi_{\gamma}(g_1^{-1}ug_1) \geq 2\left( p\lambda +q\mu - \gamma(\tilde{x_1})\right),\\
\pi(\varphi_{2\gamma}(u'))=2\pi(\varphi_{\gamma}(g_1^{-1}ug_1)) =-2\gamma(x_1).
\end{gather*}
We deduce that $2\gamma \in \Phi_{x_1}$ and that:
$$g_1^{-1}ug_1=(g_1^{-1}ug_1u') \cdot u'^{-1}$$
with:
\begin{gather*}
ug_1u' g_1^{-1} \in g_1 U'_{\gamma,x_1} g_1^{-1} = U_{\pi^{-1}(X_1)}\cap g_1U_{\gamma}g_1^{-1},\\
g_1u'^{-1}g_1^{-1} \in U_{0,2\gamma,2p\lambda+2q\mu} .
\end{gather*}
The last equalities show that the image of $u$ in $\overline{U}_{0,\gamma,p\lambda+q\mu}$ belongs to $\overline{U}_{0,2\gamma,2p\lambda+2q\mu}$, and hence the group $[\overline{U}_{0,\alpha,\lambda},\overline{U}_{0,\beta,\mu}]$ is necessarily contained in:
$$\langle \overline{U}_{0,p\alpha+q\beta,p\lambda+q\mu}\;|\; p,q\in \mathbb{Z}_{>0},\; p\alpha+q\beta \in \Phi_{x_1}\rangle.$$

\textbf{Axiom \ref{axiomV5}.}  Let $\alpha \in \Phi_{x_1}$, $\overline{u} \in \overline{U}_{0,\alpha}$ and 
  $\overline{u}',\overline{u}'' \in  \overline{U}_{0,-\alpha}$ such that $\overline{u}'\overline{u}\overline{u}'' \in \overline{M}_{0,\alpha}$.   Let $u$ be a lifting of $\overline{u}$ in $U_{0,-\alpha}$. Note that $\overline{u}\neq 1$.

Assume first that $\alpha$ is non-multipliable. By proceeding as in the proof of the axiom~\ref{axiomRGD4},  there exists  $y\in \mathbb{L}_{\alpha}^{\times}$ such that: 
\begin{gather*}
\omega_1(y)=-\alpha(x_1),\\
(g_1x_{\alpha}(y)g_1^{-1})u(g_1x_{\alpha}(y)g_1^{-1})=m_{0,\alpha}(y).
\end{gather*}
 Since $g_1x_{\alpha}(y)g_1^{-1} \in U_{0,\alpha}$ and $m_{0,\alpha}(y) \in M_{0,\alpha}$, uniqueness in paragraph 6.1.2.(2) of {\color{red}\cite{BruhatTits1}} implies that $g_1x_{\alpha}(y)g_1^{-1}$ is a lifting of both $\overline{u}'$ and $\overline{u}''$.
Hence:
$$\overline{\varphi}_{-\alpha}(\overline{u}) = \varphi_{-\alpha}(g_1^{-1}ug_1) - \alpha(\tilde{x}_1) = -\varphi_{\alpha}(x_{\alpha}(y)) - \alpha(\tilde{x}_1) = -\overline{\varphi}_{\alpha}(\overline{u}').$$

Assume now that $\alpha$ is multipliable. By proceeding as in the proof of the axiom~\ref{axiomRGD4}, there exists $(y,y') \in H(\mathbb{L}_\alpha,\mathbb{L}_{2\alpha})$ such that: 
\begin{gather*}
\omega_1(y')=-2\alpha(x_1),\\
(g_1x_{\alpha}(yy'^{-1},({^{\tau}}y')^{-1})g_1^{-1})\cdot u \cdot (g_1x_{\alpha}(y({^{\tau}}y')^{-1},({^{\tau}}y')^{-1})g_1^{-1})= m_{0,\alpha}(y,y').
\end{gather*}
 Since $g_1x_{\alpha}(yy'^{-1},({^{\tau}}y')^{-1})g_1^{-1}$ and $ g_1x_{\alpha}(y({^{\tau}}y')^{-1},({^{\tau}}y')^{-1})g_1^{-1}$ are both in $ U_{0,\alpha}$ and $m_{0,\alpha}(y,y') \in M_{0,\alpha}$, uniqueness in paragraph 6.1.2.(2) of {\color{red}\cite{BruhatTits1}} implies that $g_1x_{\alpha}(yy'^{-1},({^{\tau}}y')^{-1})g_1^{-1}$ is a lifting of $\overline{u}'$.
Hence:
$$\overline{\varphi}_{-\alpha}(\overline{u}) = \varphi_{-\alpha}(g_1^{-1}ug_1) - \alpha(\tilde{x}_1) = -\varphi_{\alpha}(x_{\alpha}(yy'^{-1},({^{\tau}}y')^{-1})) - \alpha(\tilde{x}_1) = -\overline{\varphi}_{\alpha}(\overline{u}').$$
  \end{proof}

\subsection{Compatibility axioms for the \texorpdfstring{$N$}{N}-action}

\begin{lemma}
The group $\overline{N}_{0,X_1}$ is the subgroup of $U_{X_1}/U_{\pi^{-1}(X_1)}$ spanned by the $\overline{M}_{0,\alpha}$ for $\alpha \in \Phi_{x_1}$.
\end{lemma}

\begin{proof}
 Example~\ref{ExQC} shows that the group $N_{0,X_1}$ is spanned by the $g_1(N\cap L_{\alpha,x_1})g_1^{-1}$ for $\alpha \in \Phi$. We distinguish 3 cases:
\begin{itemize}
\item[$\bullet$] if $-\alpha(x_1)\not\in \Gamma_{\alpha}$, then $g_1 L_{\alpha,x_1}g_1^{-1} \subseteq U_{\pi^{-1}(X_1)}$. Hence $g_1(N\cap L_{\alpha,x_1})g_1^{-1}\subseteq U_{\pi^{-1}(X_1)} \cap g_1Ng_1^{-1}$.
\item[$\bullet$] if $-\alpha(x_1)\in \Gamma_{\alpha} \setminus \Gamma'_{\alpha}$, then, by Fact~\ref{FactDecompositionSetOfValue}, one has $2\alpha \in \Phi_{x_1}$ and, by \cite[6.1.2]{BruhatTits1}: 
\begin{align*}
g_1(N\cap L_{\alpha,x_1})g_1^{-1}&\subseteq g_1(N\cap L_{\alpha} \cap U_{x_1})g_1^{-1} \subseteq g_1(T \cup M_{\alpha})g_1^{-1} \cap U_{X_1}\\ & = g_1(T \cup M_{2\alpha})g_1^{-1} \cap U_{X_1}\subseteq T_{0,X_1} \cup M_{0,2\alpha}.
\end{align*}
\item[$\bullet$] if $-\alpha(x_1)\in \Gamma'_{\alpha}$, then $\alpha\in\Phi_{x_1}$ and: $$g_1(N\cap L_{\alpha,x_1})g_1^{-1}\subseteq g_1(N\cap L_{\alpha} \cap U_{x_1})g_1^{-1} \subseteq g_1(T \cup M_{\alpha})g_1^{-1} \cap U_{X_1} = T_{0,X_1} \cup M_{0,\alpha}.$$
\end{itemize} 
We deduce that $N_{0,X_1}$ is the subgroup of $U_{X_1}$ spanned by $U_{\pi^{-1}(X_1)}\cap g_1Ng_1^{-1}$ and the $M_{0,\alpha}$ for $\alpha \in \Phi_{x_1}$. Hence $\overline{N}_{0,X_1}$ is the subgroup of $U_{X_1}/U_{\pi^{-1}(X_1)}$ spanned by the $\overline{M}_{0,\alpha}$ for $\alpha \in \Phi_{x_1}$.
\end{proof}

Consider the $\mathfrak{R}^{\mathrm{rk}(\Lambda_0)}$-aff space:
$$\mathbb{A}_{x_1}:=\pi^{-1}(x_1)/\left(\langle \Phi_{x_1} \rangle^{\perp}\cap \ker(\pi)\right)$$
with underlying real vector space $V_{x_1}=\ker (\pi)/\left(\langle \Phi_{x_1} \rangle^{\perp}\cap \ker (\pi)\right)$.
By the previous lemma, the image of $\overline{\nu}_{0,X_1}(\overline{N}_{0,X_1})$ in $W(\Phi)$ is $\overline{N}_{0,X_1}/\overline{T}_{0,X_1} \cong W(\Phi_{x_1}) \subseteq \mathrm{Fix}(\langle \Phi_{x_1} \rangle^{\perp})$. Hence $\overline{\nu}_{0,X_1}$ induces a morphism:
$$\overline{\nu}_{X_1}:  \overline{N}_{0,X_1} \rightarrow \mathrm{Aff}_{\mathfrak{R}^{\mathrm{rk}(\Lambda_0)}}(\mathbb{A}_{x_1}).$$

\begin{proposition}
The action $\overline{\nu}_{0,X_1}$ is compatible with the valuation $(\overline{\varphi}_{\alpha})_{\alpha\in\Phi_{x_1}}$.
\end{proposition}

\begin{proof} 
Axiom~\ref{axiomCA1} is obvious. Let's prove axiom~\ref{axiomCA2}. To do so, let $\alpha$ be a root in $ \Phi_{x_1}$ and take $\overline{u} \in \overline{U}_{0,\alpha} \setminus \{1\}$. We can find $y \in L_{\alpha}$ such that $u:=g_1x_{\alpha}(y)g_1^{-1}$ is a lifting of $\overline{u}$ in $U_{0,\alpha}$ and $v_1(y)=-\alpha(x_1)$. We then have:
\begin{align*}
\alpha \Big(\nu_{0,X_1}\big(m_{0,\alpha}(y)\big)(\tilde{x}_1) - \tilde{x}_1\Big) &=
\alpha \Big(\nu_{0,X_1}\big(m_{0,\alpha}(y)\big)(o) - o\Big) + \alpha \Big(\nu_{0,X_1}^v\big(m_{0,\alpha}(y)\big)(\tilde{x}_1-o) - (\tilde{x}_1-o)\Big) \\&=  -2\varphi_{\alpha}\big(x_{\alpha}(y)\big)-2\alpha(\tilde{x}_1)\\
&=-2\overline{\varphi}_{\alpha}(\overline{u}).
\end{align*}
The multipliable case is analoguous. Finally, for $m\in M_{0,\alpha}$:
$$\nu_{0,X_1}(m)^2=\nu(g_1^{-1}m g_1)^2|_{\pi^{-1}(\{x_1\})}=1,$$
and hence axiom~\ref{axiomCA3} holds.
\end{proof}

\subsection{Conclusion}

\begin{theorem}\label{thmFibers_buildings}
The fiber $\pi^{-1}(X_1)$ is isomorphic to:
$$\mathcal{I}\left( U_{X_1}/U_{\pi^{-1}(X_1)}, (\overline{U}_{0,\alpha})_{\alpha\in\Phi_{x_1}},(\overline{M}_{0,\alpha})_{\alpha\in\Phi_{x_1}},(\overline{\varphi}_{0,\alpha})_{\alpha\in\Phi_{x_1}},\overline{\nu}_{X_1}) \right) \times \left\langle \Phi_{x_1} \right\rangle^{\perp}.$$
\end{theorem}

\begin{proof}
  All the previous considerations show that the $5$-tuple: 
  $$\left( U_{X_1}/U_{\pi^{-1}(X_1)}, (\overline{U}_{0,\alpha})_{\alpha\in\Phi_{x_1}},(\overline{M}_{0,\alpha})_{\alpha\in\Phi_{x_1}},(\overline{\varphi}_{0,\alpha})_{\alpha\in\Phi_{x_1}},\overline{\nu}_{X_1}) \right)$$ 
is a generating and valued root group datum together with a compatible action. We can therefore consider the building:
$$\mathcal{I}\left( U_{X_1}/U_{\pi^{-1}(X_1)}, (\overline{U}_{0,\alpha})_{\alpha\in\Phi_{x_1}},(\overline{M}_{0,\alpha})_{\alpha\in\Phi_{x_1}},(\overline{\varphi}_{0,\alpha})_{\alpha\in\Phi_{x_1}},\overline{\nu}_{X_1}) \right)$$
 as defined in Definition~\ref{DefLambdaBuildingFromDatum}. It can be described as:
$$\left( U_{X_1}/U_{\pi^{-1}(X_1)} \times \mathbb{A}_{x_1} \right)/\sim,$$
where:
$$(u,x) \sim (v,y) \Leftrightarrow \exists n \in \overline{N}_{0,X_1}, \begin{cases}
y=\overline{\nu}_{X_1}(n)(x),\\
u^{-1}vn \in U_x.
\end{cases}$$

 Consider now the surjective map:
$$\mathrm{pr}: \pi^{-1}(X_1) \rightarrow \mathcal{I}\left( U_{X_1}/U_{\pi^{-1}(X_1)}, (\overline{U}_{0,\alpha})_{\alpha\in\Phi_{x_1}},(\overline{M}_{0,\alpha})_{\alpha\in\Phi_{x_1}},(\overline{\varphi}_{0,\alpha})_{\alpha\in\Phi_{x_1}},\overline{\nu}_{X_1}) \right)$$
induced by the projection $\mathrm{pr}:\pi^{-1}(x_1) \rightarrow \mathbb{A}_{x_1} $. 
It is compatible with the $U_{X_1}$-action. Take a point: 
$$X_0:=[u_0,x_0] \in \mathcal{I}\left( U_{X_1}/U_{\pi^{-1}(X_1)}, (\overline{U}_{0,\alpha})_{\alpha\in\Phi_{x_1}},(\overline{M}_{0,\alpha})_{\alpha\in\Phi_{x_1}},(\overline{\varphi}_{0,\alpha})_{\alpha\in\Phi_{x_1}},\overline{\nu}_{X_1}) \right),$$
and fix a lifting $\tilde{x}_0$ of $x_0$ in $\pi^{-1}(x_1)$. Observe that, if $X:=[u,x] \in \mathrm{pr}^{-1}(X_0)$, then there exists $n\in 
\overline{N}_{0,X_1}$ such that:
$$\begin{cases}
\mathrm{pr}(x)=\overline{\nu}_{X_1}(n)(x_0)\\
u_0^{-1}un \in \overline{U}_{0,\tilde{x}_0}.
\end{cases}$$
Hence $x-\overline{\nu}_{0,X_1}(n)(\tilde{x}_0) 
\in \langle \Phi_{x_1}\rangle^{\perp}$, and we can find $y \in \langle \Phi_{x_1}\rangle^{\perp}$ 
such that $x-\overline{\nu}_{0,X_1}(n)(\tilde{x}_0) =\overline{n} \cdot y$ where $\overline{n} $ 
is the image of $n$ in $W(\Phi_{x_1})$. Since $\alpha(\tilde{x}_0+y) = 
\alpha(\tilde{x}_0) = \alpha(x_0)$ for each $\alpha\in \Phi_{x_1}$, we deduce that:
$$\begin{cases}
  x=\overline{\nu}_{0,X_1}(n)(\tilde{x}_0+y)\\
  u_0^{-1}un \in \overline{U}_{0,\tilde{x}_0}=\langle \overline{U}_{\alpha,\tilde{x}_0} \, | \, \alpha \in \Phi_{x_1}\rangle 
  =\langle \overline{U}_{\alpha, \tilde{x}_0+y} \, | \, \alpha \in \Phi_{x_1}\rangle 
  =\overline{U}_{\tilde{x}_0+y}
\end{cases}$$
Hence $X=[u_0,\tilde{x}_0+y]$, so that the map:
\begin{align*}
\psi_{X_0}: (\mathrm{pr})^{-1}(x_0) &\rightarrow \mathrm{pr}^{-1}(X_0)\\
x & \mapsto [u_0,x]
\end{align*}
is surjective. Now take $x,x' \in (\mathrm{pr})^{-1}(x_0)$ such that $\psi_{X_0}(x)=\psi_{X_0}(x')$. 
We can then find $n \in\overline{N}_{0,X_1}\cap \overline{U}_x$ such that $x' = \overline{\nu}_{0,X_1}(n)(x)$. 
Hence $x=x'$, and $\psi_{X_0}$ is a bijection.
\end{proof}

\begin{remark}
\begin{itemize}
\item[(i)] A subset of $\pi^{-1}(X_1)$ is the intersection of an apartment of $\mathcal{I}(\mathbb{K},\omega,\mathbf{G})$ with $\pi^{-1}(X_1)$ if, and only if, it is of the form $A_{x_1} \times \langle \Phi_{x_1} \rangle^{\perp}$ with $A_{x_1}$ an apartment of $\mathcal{I}\left( U_{X_1}/U_{\pi^{-1}(X_1)}, (\overline{U}_{0,\alpha})_{\alpha\in\Phi_{x_1}},(\overline{M}_{0,\alpha})_{\alpha\in\Phi_{x_1}},(\overline{\varphi}_{0,\alpha})_{\alpha\in\Phi_{x_1}},\overline{\nu}_{X_1} \right)$.
\item[(ii)] By considering all the intersections of an apartment of $\mathcal{I}(\mathbb{K},\omega,\mathbf{G})$ with $\pi^{-1}(X_1)$, one endows the fiber $\pi^{-1}(X_1)$ with a system of apartments of type: $$(\pi^{-1}(X_1),\Phi,(\tilde{\Gamma}_{\alpha})_{\alpha\in\Phi})$$ where $\tilde{\Gamma}_{\alpha} = \Gamma_{\alpha}$ if $\alpha\in\Phi_{x_1}$ and $\tilde{\Gamma}_{\alpha}=\emptyset$ otherwise. The fiber $\pi^{-1}(X_1)$ then automatically satisfies axioms~\ref{axiomA2},~\ref{axiomA3},~\ref{axiomA4} and \ref{axiomGG}. 

\item[(iii)] If we already knew that $\mathcal{I}(\mathbb{K},\omega,\mathbf{G})$ is an $\RF^{\mathrm{rk}(\Lambda)}$-building, then it would follow from \cite[Main result 2]{schwer2012lambda} that $\pi^{-1}(X_1)$ is an $\RF^{\mathrm{rk}(\Lambda_0)}$-building. However we cannot apply this result yet, and actually, we will use  Theorem~\ref{thmFibers_buildings} in our proof of the fact that $\mathcal{I}(\mathbb{K},\omega,\mathbf{G})$ is indeed an $\RF^{\mathrm{rk}(\Lambda)}$-building. 
\end{itemize}
\end{remark}

\begin{example}\label{constantex}
Let $\Lambda_0$ be a non-zero totally ordered abelian group and let $k$ be a field endowed with a valuation $\omega_0: k \rightarrow 
\Lambda_0 \cup \{\infty\}$. Set 
$\mathbb{K}:=k((t))$, let $\omega_1$ be the 
$t$-adic valuation on $\mathbb{K}$, and let 
$\omega: \mathbb{K}\rightarrow \mathbb{Z} \times \Lambda_0$ be the valuation 
defined by $\omega (x)=(\omega_1(x),\omega_0(xt^{-\omega_1(x)}))$.    

Consider a quasi-split reductive $k$-group $\mathbf{G}_k$ and set $ \mathbf{G}:= \mathbf{G}_k\times_k \mathbb{K}$. For $\alpha\in \Phi$, we denote $\mathbf{U}_{\alpha,k}$ the root group of $\mathbf{G}_k$ associated to $\alpha$. We may and do assume that the valuations on $U_{\alpha,k}$ and on $U_{\alpha}$ are compatible. 

The group $\mathbf{G}_k$ splits over a finite separable extension $k'$ of $k$, and hence one can find a finite extension $\mathbb{K}'$ of $\mathbb{K}$ over which 
$\omega_1$ does not ramify and $\mathbf{G}$ splits. Thus $0 \in \Gamma_{\alpha,1}'$ for every $\alpha \in \Phi$, and hence the point 
$X_1:= [1,0] \in \mathcal{I}(\mathbb{K},\omega_1,\mathbf{G})$ is a hyperspecial point (see Section 1.10.2 of \cite{Tits} for the definition). 

Now take $u \in U_{X_1}$ and $x \in \pi^{-1}(\{0\})$ such that $[u,x] \in \pi^{-1} (X_1)$. Since $u \in U_{X_1}$, we can find  $\alpha_1,...,\alpha_n \in \Phi$ so that $u$ can be written as $u_1...u_n$ with $u_i \in U_{\alpha_i,0}$ for each $i \in \{1,...,n\}$. Each $u_i$ can then be written as $v_iu'_i$ with $v_i \in U_{\alpha_i,k}$ and $u'_i \in U'_{\alpha_i,0}$. As a consequence, the product $u_1...u_n$ can be written as $u'v_1...v_n$ for some $u' \in U_{X_1}$ that fixes $\pi^{-1}(X_1)$. Hence:
$$[u,x] = [u'v_1...v_n,x] = u' \cdot [v_1...v_n,x] = [v_1...v_n,x],$$
and the map:
\begin{align*}
\psi_{X_1}: G_k \times \pi^{-1}(\{0\}) &\rightarrow \pi^{-1}(X_1)\\
(g,x) & \mapsto [g,x]
\end{align*}
is surjective. Now, if $\psi_{X_1}(g,x) = \psi_{X_1}(g',x')$ for some $g, g', x, x'$, then we can find $n\in N$ such that:
$$\begin{cases}
x'=\nu(n)(x)\\
g^{-1}g'n \in U_x.
\end{cases}$$
In particular: $$n\in g'^{-1}gU_x \cap N \subseteq \mathbf{G}(k[[t]]) \cap N = N_k.$$
Hence, $\psi_{X_1}$ induces a bijection:
$$\overline{\psi}_{X_1}:\mathcal{I}(k, \omega_0, \mathbf{G}_k) \rightarrow \pi^{-1}(X_1).$$
\end{example}

\subsection{Further notations related to the projection maps}\label{subFurther_notation}

In this section, we fix some notations that will be used throughout the paper.
 Let's fix some $s_0 \in \mathrm{rk}(\Lambda)$. 
 Given a positive element $\lambda_{s_0}$ 
in the archimedean class of $\Lambda$ corresponding to $s_0$, we define:
\begin{gather*}
\Lambda_{\geq s_0} := \{ \lambda \in \Lambda | \exists n >0, \lambda < n \lambda_{s_0} \},\\
\Lambda_{> s_0} := \{ \lambda \in \Lambda | \forall n >0, n\lambda < \lambda_{s_0} \}.
\end{gather*}
In terms of Hahn's embedding $\Lambda \subset \RF^{\mathrm{rk}(\Lambda)}$:
\begin{gather*}
\Lambda_{\geq s_0} = \Lambda \cap \RF^{\geq s_0} , \\ \Lambda_{> s_0} = \Lambda \cap \RF^{> s_0} .
\end{gather*}
The sets $\Lambda_{\geq s_0}$ and $\Lambda_{> s_0}$ are both convex subgroups 
of $\Lambda$ and they do not depend on the choice of $\lambda_{s_0}$. 
We may therefore introduce the quotients:
\begin{gather*}
  \Lambda_{< s_0} := \Lambda/\Lambda_{\geq s_0}, \\
  \Lambda_{\leq  s_0} := \Lambda/\Lambda_{> s_0}.
  \end{gather*}
  Again, in terms of Hahn's embedding $\Lambda \subset \RF^{\mathrm{rk}(\Lambda)}$, the quotients $\Lambda_{< s_0}$ and $\Lambda_{\leq  s_0}$ are the images of $\Lambda$ under the projections $\RF^{\mathrm{rk}(\Lambda)} \rightarrow \RF^{< s_0}$ and $\RF^{\mathrm{rk}(\Lambda)} \rightarrow \RF^{\leq s_0}$ respectively.

  \begin{example}
      If $\Lambda=\RF^{\mathrm{rk}(\Lambda)}$, we have:
      \begin{gather*}
      \Lambda_{\geq s_0} = \RF^{\geq s_0} , \;\;\; \Lambda_{> s_0} = \RF^{> s_0},\\
  \Lambda_{< s_0} := \RF^{< s_0}, \;\;\;
  \Lambda_{\leq  s_0} := \RF^{\leq s_0}.
  \end{gather*}
  Note that, in this case, we have isomorphisms of totally ordered abelian groups:
  $$\Lambda \cong \RF^{< s_0} \times \RF^{\geq s_0}\cong \RF^{\leq s_0} \times \RF^{> s_0}\cong \RF^{< s_0}\times \mathbb{R} \times \RF^{> s_0}.$$
  \end{example}
  
Denote by $\omega_{<s_0}: \mathbb{K} \rightarrow \Lambda_{< s_0} \cup \{\infty\}$ 
(resp. $\omega_{\leq  s_0} : \mathbb{K} \rightarrow \Lambda_{\leq s_0} \cup \{\infty\}$ ) 
the composite of the valuation $\omega$ followed by the projection 
$\Lambda \rightarrow \Lambda_{< s_0}$ (resp. $\Lambda \rightarrow \Lambda_{\leq s_0}$). 
According to the previous section, when $\Lambda_{\geq s_0}$ and $\Lambda_{> s_0}$ are strict subgroups of $\Lambda$, we have three projection maps:
\begin{gather*}
\pi_{\leq s_0} : \mathcal{I}(\mathbb{K},\omega,\mathbf{G}) \rightarrow \mathcal{I}(\mathbb{K},\omega_{\leq s_0},\mathbf{G}),\\
\pi_{< s_0}^{\leq s_0} : \mathcal{I}(\mathbb{K},\omega_{\leq s_0},\mathbf{G}) \rightarrow \mathcal{I}(\mathbb{K},\omega_{< s_0},\mathbf{G}),\\
\pi_{< s_0} = \pi_{< s_0}^{\leq s_0} \circ \pi_{\leq s_0}: \mathcal{I}(\mathbb{K},\omega,\mathbf{G}) \rightarrow \mathcal{I}(\mathbb{K},\omega_{< s_0},\mathbf{G}).
\end{gather*}
The map 
$\pi_{\leq s_0} : \mathcal{I}(\mathbb{K},\omega,\mathbf{G}) \rightarrow \mathcal{I}(\mathbb{K},\omega_{\leq s_0},\mathbf{G})$ 
then induces a surjection:
$$\pi_{=s_0,X} : \pi_{<s_0}^{-1}(X) \rightarrow \left( \pi_{< s_0}^{\leq s_0} \right)^{-1}(X)$$
for each $X \in \mathcal{I}(\mathbb{K},\omega_{< s_0},\mathbf{G})$.

\section{Axiom~\ref{axiomCO}}
\label{sectionCO}

Let $\mathbf{G}$ be a quasi-split reductive group over a field $\mathbb{K}$ equipped with a 
non-trivial  valuation $\omega:\mathbb{K}\rightarrow \RF^S\cup\{\infty\}$,  where $S$ is a totally ordered set. We assume that $\mathbb{K}$ satisfies Assumption~\ref{assumpfield}.  Let $\I=\I(\Kb,\omega,\GB)$ be the building defined in Definition~\ref{defBuilding_quasi_split_group}.

In this section, we prove that axiom~\ref{axiomCO} is satisfied by the building  $\I$ that we defined in subsection~\ref{buildredgp} and by its fibers for the projections, which completes the proof of Theorem~\ref{thmMain}.
 We keep the notation of Subsection~\ref{subsubRSMetrics}.
Let us sketch the ideas of our proof. We  proceed as follows. \begin{enumerate}
\item\label{itStep_R_bulding} (see Subsection~\ref{subSufficient_condition_CO}) We begin by giving a sufficient condition for an $\R$-building to satisfy~\ref{axiomCO}. 

\item\label{itStep_preservation_opposition} (see Subsection~\ref{subPreservation_opposition}) We prove that when $S$ admits a minimum $0_S$ (and when $\pi_{\leq 0_S}\left(\omega(\Kb^*)\right)$ is non-trivial), if $C$ and $\tilde{C}$ are two sectors opposed at $x\in \I$, then the corresponding sectors of $\pi_{\leq s}(\I)$ are opposed at $\pi_{\leq s}(x)$ for every $s\in S$.

\item (see Subsection~\ref{subProof_CO_when_S_admits_minimum}) We still assume that $S$ admits a minimum and we prove~\ref{axiomCO}. Let $x\in \I$ and $C$, $\tilde{C}$ be two sectors opposed at $x$. Let $A$ be the apartment containing the germs at infinity $C_\infty$, $\tilde{C}_\infty$ of $C$, $\tilde{C}$. We want to prove that $x\in A$. We consider an apartment $A_C$ containing $C$ and an isomorphism $\phi:A_C\rightarrow A$ fixing $A_C\cap A$. We then prove that $\phi(x)=x$. For this, we act by contradiction and assume that $\phi(x)\neq x$.  Then we prove that there exists a minimal $s\in S$ such that $\pi_{\leq s}\big(\phi(x)\big)\neq  \pi_{\leq s}(x)$. By working in the  $\R$-building  $\pi_{=s}\big(\pi_{\leq s}^{-1}\big(\pi_{\leq s}(x)\big)\big)$ and by using steps~(\ref{itStep_R_bulding}) and~(\ref{itStep_preservation_opposition}), we reach a contradiction.

\item (see subsection~\ref{subProof_CO_general_case} and subsection~\ref{subsecCO_fibers}) We then prove that $\I$ satisfies~\ref{axiomCO} in the general case, by considering $\GB\big(\Kb(\!(t)\!)\big)$. We then deduce that the fibers for the projections also satisfy~\ref{axiomCO}.
\end{enumerate}

\subsection{A sufficient condition for~\ref{axiomCO} for \texorpdfstring{$\R$}{R}-buildings}\label{subSufficient_condition_CO}

In this subsection, we assume that $S$ is reduced to a single element. Let $\I_\R$ be a set covered with apartments and satisfying~\ref{axiomA1},~\ref{axiomA2}, \ref{axiomGG},~\ref{axiomA4} and 

(Iwa) : for all local face $F$, for all sector-germ at infinity $C_\infty$, there exists an apartment containing $F$ and $C_\infty$.

Let $x,y\in \I_\R$, and $A$ be an apartment containing $x$ and $y$, which exists by \ref{axiomGG} and~\ref{axiomA2}. The \textbf{line segment}\index{line segment} $[x,y]$ is the line segment in $A$ between $x$ and $y$. One also defines the \textbf{germ at $x$ of $[x,y]$}\index{germ} as the filter $\germ_x([x,y])$ of subsets of $A$ containing  $\Omega\cap [x,y]$, for some neighborhood $\Omega$ of $x$ in $A$. By~\ref{axiomA2}, all these notions are well defined independently of the choice of $A$.

The \textbf{ray}\index{ray} based at $x$ and containing $y$ is the closed half-line of $A$ based at $x$ and containing $y$.

\begin{lemma}\label{lemSpliting_segments_sector_friendly}
Let $x,y\in \I_\R$ and $C_\infty$ be a sector-germ at infinity of $\I_\R$. Then there exist $n\in \Z_{>0}$ , $x_1,\ldots,x_n\in [x,y]$ such that $[x,y]=\bigcup_{i=1}^{n-1} [x_i,x_{i+1}]$ and for all $i\in \llbracket 1,n-1\rrbracket$, there exists an apartment $A_i$ containing $[x_i,x_{i+1}]$ and $C_\infty$.
\end{lemma}

\begin{proof}
This is a standard result. Let $A$ be an apartment containing $[x,y]$. For $a\in [x,y)$ (resp. $a\in (x,y]$) we choose a local chamber $C_a^+$ (resp. $C_a^-$) of $A$ such that $C_a^+\cap [x,y]\Supset \germ_a\big([a,y)\big)$ (resp. $C_a^-\cap [x,a)\Supset \germ_a\big([a,x]\big)$). For $a\in (x,y]$ (resp. $a\in [x,y)$) we choose an apartment $A_a^+$ (resp. $A_a^-$) containing $C_a^+$ and $C_\infty$ (resp. $C_a^-$ and $C_\infty$), which is possible by (Iwa). We choose a neighborhood $V_a^+$ (resp. $V_a^-$) of $a$ in $[a,y]$ (resp. in $[x,a]$) such that $V_a^+\subset A_a^+$ (resp. $V_a^-\subset A_a^-$). Set $V_x=V_x^+$, $V_y=V_y^-$ and $V_a=V_a^+\cup V_a^-$, for $a\in (x,y)$. Then by compactness of $[x,y]$, there exists a finite subset $\{a_1,\ldots,a_k\}$ of $[x,y]$ such that $[x,y]=\bigcup_{i=1}^k V_{a_i}$, and the result follows.
\end{proof}

\begin{lemma}\label{lemCO_for_R_buildings}
Let $C_\infty,\tilde{C}_\infty$ be two sector-germs at infinity of $\I_\R$. Let $x\in \I_\R$. We assume that $\germ_x(x+C_\infty)$ and $\germ_x(x+\tilde{C}_\infty)$ are opposite. 
 Then:\begin{enumerate}
\item $C_\infty$ and $\tilde{C}_\infty$ are opposite,

\item there exists  a unique apartment $A_{C_\infty,\tilde{C}_\infty}$ containing $C_\infty$ and $\tilde{C}_\infty$,

\item the point  $x$ belongs to $A_{C_\infty,\tilde{C}_\infty}$.
\end{enumerate}

\end{lemma}

\begin{proof}
We follow \cite[Proposition 1.12]{parreau2000immeubles} and \cite[Proposition 5.4 2)]{rousseau2011masures}. Let $A_{C_\infty,\tilde{C}_\infty}$ be an apartment containing $C_\infty$ and $\tilde{C}_\infty$. Let $C=x+C_\infty$, $\tilde{C}=x+\tilde{C}_\infty$, $C_x=\germ_x(C)$ and $\tilde{C}_x=\germ_x(\tilde{C})$. Let $A_C$ and $A_{\tilde{C}}$ be apartments containing $C$ and $\tilde{C}$ respectively. Let $A_{C_x,\tilde{C}_x}$ be an apartment containing $C_x,\tilde{C}_x$. Let $y,\tilde{y}\in A_{C_x,\tilde{C}_x}$ be such that $y\in C$, $\tilde{y}\in \tilde{C}$ and $x\in [\tilde{y},y]$. Let $\delta$ (resp. $\tilde{\delta}$) be the ray of $A_C$ (resp. $A_{\tilde{C}}$) based at $x$ and  containing $y$ (resp.  $\tilde{y}$). 

The ray $\delta$ meets $A_{C_\infty,\tilde{C}_\infty}$. We choose $y'$ in $\delta\cap A_{C_\infty,\tilde{C}_\infty}$. Let $r:[0,\infty[\rightarrow \delta$ be the affine parametrization of $\delta$ such that $r(0)=x$ and $r(1)=y'$. Using Lemma~\ref{lemSpliting_segments_sector_friendly}, we choose  $n\in \N$, $t_0=0<t_1<\ldots <t_n=1$ such that for all $i\in \llbracket 0,n-1\rrbracket$, $[r(t_i),r(t_{i+1})]$ and $\tilde{C}_\infty$ are contained in  an apartment $A_i$.

 Let $r:(-\infty,0]\rightarrow \tilde{\delta}$ be the affine parametrization of   $\tilde{\delta}$ such that $r(0)=x$ and $r(-1)=\tilde{y}$. We set $t_{-1}=-1$. Then $\tilde{y}=r(t_{-1})\in r(t_0)+\tilde{C}_\infty=x+\tilde{C}_\infty$.  Let $i\in \llbracket 0,n-1\rrbracket$.  We assume that $r(t_i)+\tilde{C}_\infty$ contains $r(t_{i-1})$. Let $B_i$ be an apartment containing $[r(t_i-\epsilon),r(t_i)]\cup [r(t_i),r(t_i+\epsilon)]$, for $\epsilon>0$ small enough (one may take $B_0=A_{C_x,\tilde{C}_x}$ and $B_i=A_C$, for $i>0$). By assumption, $A_i$ contains $r(t_{i-1})$ and as $[r(t_i-\epsilon),r(t_i+\epsilon)]$ is a line segment in $B_i$, $[r(t_{i-1}),r(t_{i+1})]$ is a line segment in $A_i$. In the apartment $A_i$, $r(t_i)+\tilde{C}_\infty$ is a sector parallel to $r(t_{i+1})+\tilde{C}_\infty$. As $r(t_i)+\tilde{C}_\infty$ contains $r(t_{i-1})$, we deduce that $r(t_{i+1})+\tilde{C}_\infty$ contains $r(t_i)$. By induction, we deduce that $r(t_0)=x\in r(t_n)+\tilde{C}_\infty=y'+\tilde{C}_\infty$. Therefore $x\in A_{C_\infty,\tilde{C}_\infty}$. Consequently, $x+C_\infty$ and $x+\tilde{C}_\infty$ are two opposite sectors of $A_{C_\infty,\tilde{C}_\infty}$. Therefore, $C_\infty$ and $\tilde{C}_\infty$ are opposite and by~\ref{axiomA2}, there exists at most one apartment containing $C_\infty$ and $\tilde{C}_\infty$, which proves the lemma.

\end{proof}

\subsection{Preservation of the opposition}\label{subPreservation_opposition}

The aim of this subsection is to prove that if $S$ admits a minimum and if  $C$ and $\tilde{C}$ are two sectors opposite at some point $x\in \I$, then for every $s\in S$, $C_{\leq s}$ and $\tilde{C}_{\leq s}$ are opposite at $\pi_{\leq s}(x)$. For this, our idea is to prove that if $C_{1,\infty}, \ldots, C_{n,\infty}$ is a gallery of sector-germs at infinity such that $\germ_{\pi_{\leq s}(x)}(\pi_{\leq s}(x)+C_{1,\infty,\leq s}),\ldots, \germ_{\pi_{\leq s}(x)}(\pi_{\leq s}(x)+C_{n,\infty,\leq s})$ is a gallery from $\germ_{\pi_{\leq s}(x)}(C)$ to $\germ_{\pi_{\leq s}(x)}(\tilde{C})$, then $\germ_x(x+C_{1,\infty}),\ldots,\germ_{x}(x+C_{n,\infty})$ is a gallery from $\germ_x(C)$ to $\germ_x(\tilde{C})$. 

\subsubsection{Preleminaries on enclosed sets}\label{subsubPreliminaries_enclosed_sets}

\begin{Not}\label{notProjection_sector}

For the moment, we do not assume that $S$ admits a minimum. We will make this assumption from Lemma~\ref{l_projection_0_sector} to Lemma~\ref{lemCO_when_S_admits_minimum}. Let $C$ be a sector of $\A_S$ and $s\in S$. Write $C=x+w.C^v_{f,\A_S}$, where $x\in \A_S$ and $w\in W^v$. One sets $C_{\leq s}=x+w.\{y_{\leq s}\in \A_{\leq s} |\alpha(y_{\leq s})> 0,\ \forall \alpha\in \Delta_f\}$\index[notation]{c@$C_{\leq s}$}. This is the sector of $\A_{\leq s}$ corresponding to $C$. One has $C_{\leq s}\subset \pi_{\leq s}(C)$ but this containment is strict, because $\pi_{\leq s}$ preserves large inequalities but not strict inequalities.

Let $C$ be a sector of $\I$. Write $C=g.\tilde{C}$, with $g\in \mathbf{G}(\mathbb{K})$ and $\tilde{C}$ a subsector of $\A_S$. We set $C_{\leq s}=g.\tilde{C}_{\leq s}$. This does not depend on the choice of $g$ and $\tilde{C}$ by the lemma below.
\end{Not}

\begin{lemma}
Let $g_1,g_2\in \mathbf{G}(\mathbb{K})$, $C^{(1)},C^{(2)}$ be sectors of $\A_S$ and $s\in S$. Suppose that $g_1.C^{(1)}=g_2.C^{(2)}$. Then $g_1.C^{(1)}_{\leq s}=g_2.C^{(2)}_{\leq s}$.
\end{lemma}

\begin{proof}
Let $C$ be a sector of $\A_S$ and $g\in \mathbf{G}(\Kb)$. Let $h\in \mathbf{G}(\Kb)$ be such that $h$ fixes $g.C$ pointwise. Then as $\pi_{\leq s}$ is $\mathbf{G}(\Kb)$-equivariant, $h$ fixes $\pi_{\leq s}(g.C)=g.\pi_{\leq s}(C)$. Moreover $\pi_{\leq s}(C)\supset C_{\leq s}$ and thus $h$ fixes $g.C_{\leq s}$ pointwise.

For $i\in \{1,2\}$, set $A_i=g_i.\A_S$. By Proposition~\ref{propAxiom(A2)}, there exists $h\in \mathbf{G}(\mathbb{K})$ inducing an isomorphism $\phi:A_1\rightarrow A_2$ fixing $A_1\cap A_2$. Let $h\in \mathbf{G}(\Kb)$ be such that $h$ fixes $A_1\cap A_2$ and sends $A_1$ on $A_2$, which exists by Proposition~\ref{propAxiom(A2)}.
Let $n=g_2^{-1} hg_1$. Then $n$ stabilizes $\A_S$ and thus  by Corollary~\ref{CorStabilizerA},  $n \in N$.
Then $n.C^{(1)}=C^{(2)}$. By Corollary~\ref{corAffine_weyl_group}, the  restriction of $n$ to $\A_S$ is induced by an element of $\Wext \subset W^v \ltimes \A_S$
  and we deduce that $n.C^{(1)}_{\leq s}=C^{(2)}_{\leq s}$. Therefore $hg_1.C^{(1)}_{\leq s}=g_2.C^{(2)}_{\leq s}$. As $h$ fixes $g_1.C^{(1)}$ pointwise, we deduce that $hg_1.C^{(1)}_{\leq s}=g_1.C^{(1)}_{\leq s}=g_2.C^{(2)}_{\leq s}$  which proves the lemma.
\end{proof}

If $\top$ is a binary relation on $S\times S$ (for example $\top= ``<'', ``='', \ldots$) and $s\in S$ we define the projection $\pi_{\top s}:\RF^S\rightarrow \RF^{\top s}$  by $\pi_{\top s}\big((\lambda)_{t\in S}\big)=(\lambda_t)_{t\top s}$, for $(\lambda_t)\in \RF^S$.

\begin{lemma}\label{lemProjection_nonempty_polyhedron}
Let $(\lambda_\alpha)\in (\RF^{S}\cup\{\infty\})^\Phi$ and $\Omega=\bigcap_{\alpha\in \Phi}D_{\alpha,\lambda_\alpha}$. We assume that there exists a vector chamber $C^v_S$ of $\A_{S}$ and $x_{\leq s}\in \A_{\leq s}$ such that $\germ_{x_{\leq s}}(x_{\leq s}+C^v_{\leq s})\Subset \bigcap_{\alpha\in \Phi} D_{\alpha,\pi_{\leq s}(\lambda_\alpha)}$. Then $\pi_{\leq s}(\Omega)=\bigcap_{\alpha\in \Phi} D_{\alpha, \pi_{\leq s}(\lambda_\alpha)}$.
\end{lemma}

\begin{proof}
Let $\Omega'=\bigcap_{\alpha\in  \Phi} D_{\alpha, \pi_{\leq s}(\lambda_\alpha)}$. The inclusion $\pi_{\leq s}(\Omega)\subset \Omega'$ is clear. Let us prove the reverse inclusion.   Let   $\Phi^+=\Phi^+_{C^v_\R}$. 

Let $z'_{\leq s}\in \Omega'\cap (x_{\leq s}+C^v_{\leq s})$ and $z_{\leq s}\in (x_{\leq s}+C^v_{\leq s})\cap (z'_{\leq s}-C^v_{\leq s})$. Then for all $\alpha\in \Phi^+$, one has $-\pi_{\leq s}(\lambda_\alpha)\leq \alpha(x_{\leq s})<\alpha (z_{\leq s})<\alpha(z'_{\leq s})$  and for all $\alpha\in \Phi^-$, $-\pi_{\leq s}(\lambda_\alpha)\leq \alpha(z'_{\leq s})<\alpha(z_{\leq s})$. In particular, \begin{equation}\label{e_inequality_zleqs}\alpha(z_{\leq s})>-\pi_{\leq s}(\lambda_\alpha),\end{equation} for every $\alpha\in \Phi$. 

Let $y_{\leq s}\in \Omega'$. Set $\Phi(y_{\leq s})=\{\alpha\in \Phi\mid \alpha(y_{\leq s})=-\pi_{\leq s}(\lambda_\alpha)\}$. Then by definition, 
\[\Omega\cap \big(\{y_{\leq s}\}\times \A_{>s}\big)=\{y_{\leq s}\}\times \bigcap_{\alpha\in \Phi(y_{\leq s})}D_{\alpha,\pi_{>s}(\lambda_\alpha) }.\] If $\Phi(y_{\leq s})$ is empty, then $\{y_{\leq s}\}\times \A_{>s}\subset \Omega$ and thus $y_{\leq s}\in \pi_{\leq s}(\Omega)$. We now assume that $\Phi(y_{\leq s})$ is nonempty. By \eqref{e_inequality_zleqs},   one has $\alpha(z_{\leq s}-y_{\leq s})>0$, for all $\alpha\in \Phi(y_{\leq s})$. Write $y_{\leq s}=(y_t)_{t\leq s}$ and $z_{\leq s}=(z_t)_{t\leq s}$. For $\alpha\in  \Phi(y_{\leq s})$, set \[s_\alpha=\min \{t\in \supp(z_{\leq s}-y_{\leq s})\mid \alpha(z_t)>\alpha(y_t)\}.\]

Let $k=|\Phi(y_{\leq s})| $. We write $\Phi(y_{\leq s})=\{\alpha_1,\ldots,\alpha_k\}$ in such a way that $(s_{\alpha_i})_{i\in \llbracket 1,k\rrbracket}$ is  non-decreasing.  For $i\in \llbracket 1,k\rrbracket$, one sets $u_i=z_{s_{\alpha_i}}-y_{s_{\alpha_i}}\in \A_\R$. Then  for all $j\in \llbracket 1,k\rrbracket$, one has $\alpha_j(u_i)>0$ if $s_{\alpha_i}=s_{\alpha_j}$ and $\alpha_j(u_i)=0$ if $s_{\alpha_j}>s_{\alpha_i}$. In particular, for all $i,j\in \llbracket 1,k\rrbracket$, one has: \begin{equation}\label{eqPositivity_relation}
\alpha_i(u_i)>0\text{ and }j\geq i\implies \alpha_j(u_i)\geq 0.
\end{equation}

If $(y_{\leq s},0)\in \Omega$, then we are done. Suppose that $(y_{\leq s},0)\notin \Omega$. Let \[\tilde{s}=\min \{t\in \supp (\lambda_\alpha)\mid \exists \alpha\in \Phi(y_{\leq s})\mid -\pi_{=t}(\lambda_\alpha)>0\}.\] One chooses $t_k\in\R_{\geq 0}$ such that $\alpha_k(t_ku_k)>-\pi_{=\tilde{s}}(\lambda_{\alpha_k})$. Let $i\in \llbracket 2,k\rrbracket$. Suppose we have constructed $(t_j)_{j\in \llbracket i,k\rrbracket}\in  (\R_{\geq 0})^{\llbracket i,k\rrbracket}$ such that \[\sum_{j=i}^k t_j \alpha_{\ell}(u_j)>-\pi_{=\tilde{s}}(\lambda_{\alpha _\ell})\] 
for all $\ell\in \llbracket i,k\rrbracket$. Then one chooses $t_{i-1}\in \R_{\geq 0}$ such that  $\sum_{j=i-1}^k t_j\alpha_{i-1}(u_j)>-\pi_{=\tilde{s}}(\lambda_{\alpha_{i-1}})$. 
Then by~(\ref{eqPositivity_relation}), one has $\sum_{j=i-1}^k \alpha_\ell(t_j u_j)>-\pi_{=\tilde{s}}(\lambda_{\alpha_\ell})$
for all $\ell\in \llbracket i-1,k\rrbracket$.
By induction, we thus find $t_1,\ldots,t_k\in\R_{\geq 0}$ such that $\sum_{j=1}^k \alpha_\ell(t_j u_j)>-\pi_{=\tilde s}(\lambda_{\alpha_\ell})$ for all $\ell\in \llbracket 1,k\rrbracket$. 

Let $y_{<\tilde{s}}=(a_t y_t)\in\A_{<\tilde{s}}$, where $a_t=1$ if $t\leq s$ and $a_t=0$ if $t\geq s$. Let \[y=(y_{<\tilde{s}},\sum_{i=1}^k t_i u_i,0)\in \A_{S}.\] Let $\alpha\in \Phi(y_{\leq s})$. Then:

\[\alpha(y)=\big(-\pi_{<\tilde{s}}(\lambda_{\alpha}),\sum_{j=i}^k t_j \alpha(u_j),0\big)>-\lambda_{\alpha _\ell}.\]

Therefore, $y\in \Omega$. Moreover $\pi_{\leq s}(y)=y_{\leq s}$ and thus $\pi_{\leq s}(\Omega)=\Omega'$.

\end{proof}

\begin{remark}
The above lemma is not true when  $\Omega$ is an arbitrary enclosed subset of $\A_{S}$. For example if $S=\{1,2\}$, $\A_\R=\R$ and $\alpha=\Id$. Set $s=0$. Then we have $\A_S=\R^2$ and $\pi:=\pi_{\leq s}$ is the projection on the first coordinate. Set $\lambda_\alpha=(0,0)$ and $\lambda_{-\alpha}=(0,-1)$. Then $D_{\alpha,\lambda_\alpha}\cap D_{-\alpha,\lambda_{-\alpha}}=\emptyset$ and $D_{\alpha,\pi(\lambda_\alpha)}\cap D_{-\alpha,\pi(\lambda_{-\alpha})}=\{0\}$.
\end{remark}

\subsubsection{Projection of an intersection of apartments}

\begin{lemma}\label{lemInjectivity_projection_apartments}
Let $s\in S$. The map $\pi_{\leq s}$ induces a bijection $A\mapsto\pi_{\leq s}(A)$ from the set of apartments of $\I$ to the set of apartments of $\pi_{\leq s}(\I)$.
\end{lemma}

\begin{proof}

By definition (see Subsection~\ref{SubsecContructionFibers}), $\pi_{\leq s}(\mathbb{A}_S) = \mathbb{A}_{\leq s}$ is an apartment of $\pi_{\leq s}(\mathcal{I})$ and the map $\pi_{\leq s} : \mathcal{I} \to \pi_{\leq s}(\mathcal{I})$ is $\mathbf{G}(\mathbb{K})$-equivariant.
Since by Definition~\ref{DefCombinatorialStructureBuilding} $\mathbf{G}(\mathbb{K})$ acts transitively on the set of apartments of $\mathcal{I}$ (resp. $\pi_{\leq s}(\mathcal{I})$), the map $A\mapsto\pi_{\leq s}(A)$ from the set of apartments of $\I$ to the set of apartments of $\pi_{\leq s}(\I)$ is well-defined and surjective.
Moreover, given two apartments $A$ and $B$ such that $\pi_{\leq s}(A)=\pi_{\leq s}(B)$, we can assume that $A = \mathbb{A}_S$ by $\mathbf{G}(\mathbb{K})$-equivariance of $\pi_{\leq s}$. Let $g\in \mathbf{G}(\mathbb{K})$ be such that $B=g \cdot \A_S$.
Then $g \cdot \pi_{\leq s}(\A_S) = \pi_{\leq s}(g \cdot \A_S)=\pi_{\leq s}(B)=\pi_{\leq s}(\A_S)$.

By Proposition~\ref{propAxiom(A2)} applied in $\pi_{\leq s}(\mathcal{I})$, there exists $n\in N$  such that $g.x=n.x$ for all $x\in \A_{\leq s}$. Then by Lemma~\ref{lemParahoric_fixator} $g^{-1}.n\in \widehat{P}_{\A_{\leq s}}$.

  By Corollary~\ref{CorPchapeauOmegaQC}, if $\Delta$ is a basis of $\Phi$, then: \[\widehat{P}_{\A_{\leq s}}=(U_{\A_{\leq s}}\cap U_{\Delta}^+)(U_{\A_{\leq s}}\cap U_{\Delta}^-) \widehat{N}_{\A_{\leq s}}.\] By Example~\ref{ExQC}, $U_{\A_{\leq s}}\cap U_{\Delta}^+=U_{\A_{\leq s}}\cap U_{\Delta}^-=\{1\}$,  thus $g^{-1}n\in N$, hence $g\in N$ and $B=\A_S$, which proves the lemma.
\end{proof}

\begin{lemma}\label{lemProjection_intersection_apartments}
Let $A$, $B$ be two apartments and $s\in S$. We assume that there exists a sector $C$ based at some $x\in \I$ such that  $\pi_{\leq s}(A)\cap \pi_{\leq s}(B)\Supset \germ_{\pi_{\leq s}(x)}(C_{\leq s})$. Then $\pi_{\leq s}(A\cap B)=\pi_{\leq s}(A)\cap \pi_{\leq s} (B)$.  
\end{lemma}

\begin{proof}
Using some $g\in \mathbf{G}(\mathbb{K})$, we may assume that $A=\A_S$. Let $\Omega_{\leq s}=\pi_{\leq s}(\A_S)\cap \pi_{\leq s}(B)$. One has $\Omega_{\leq s}\supset \pi_{\leq s}(\A_S\cap B)$. Let us prove the reverse inclusion. 

By Corollary~\ref{corLandvogt9.7}, there exists $u\in U_{\Omega_{\leq s}}$ such that $u.\pi_{\leq s}(\A_S)=\pi_{\leq s} (B)$.
By Lemma~\ref{lemInjectivity_projection_apartments}, $u.\A_S=B$. By definition of $U_\Omega$, there 
exist $k\in \N$, $\alpha_1,\ldots,\alpha_k\in \Phi$ and elements $u_i\in U_{\alpha_i,\Omega_{\leq s}}$ 
such that $u=\prod_{i=1}^k u_i$. For $\alpha\in \Phi_{\mathrm{nd}}$, set $\lambda_\alpha=\min\{\varphi_{\alpha_i}(u_i)|i\in \llbracket 1,k\rrbracket \mathrm{\ and\  }\alpha_i=\alpha\}$
(one may have $\lambda_\alpha=\infty$). Then $\Omega_{\leq s}\subset \bigcap_{\alpha\in \Phi_{\mathrm{nd}}} D_{\alpha,\pi_{\leq s}(\lambda_\alpha)}$. As $\Omega_{\leq s}$ contains a local chamber, we can apply Lemma~\ref{lemProjection_nonempty_polyhedron} and one has 
$\pi_{\leq s}\big(\bigcap_{\alpha\in \Phi_{\mathrm{nd}}} D_{\alpha,\lambda_\alpha}\big)=\bigcap_{\alpha\in\Phi_{\mathrm{nd}}} D_{\alpha, \pi_{\leq s}(\lambda_\alpha)}$.    Moreover $u$ fixes $\bigcap_{\alpha\in \Phi_{\mathrm{nd}}}D_{\alpha,\lambda_\alpha}$ and thus $\A_{S}\cap B$ contains $\bigcap_{\alpha\in \Phi_{\mathrm{nd}}}D_{\alpha,\lambda_\alpha}$. Thus \[\pi_{\leq s}(\A_{S}\cap B)\supset \bigcap_{\alpha\in\Phi_{\mathrm{nd}}} D_{\alpha, \pi_{\leq s}(\lambda_\alpha)}\supset \Omega_{\leq s}.\] Therefore $\Omega_{\leq s}=\pi_{\leq s}(\A_S\cap B)$, which proves the lemma.  
\end{proof}

\begin{remark}
If we already knew that $\I$ is a building, we could use \cite[Lemma 3.7 and 3.10] {schwer2012lambda}, but their proof uses the fact that retractions are $1$-Lipschitz continuous, which is what we want to prove.
\end{remark}

\subsubsection{The exchange condition}

We now prove that $\I$ satisfies the exchange condition (EC) (see Lemma~\ref{lemExchange_condition} or \cite[Section 2]{bennett2014axiomatic} for the definition of (EC)). We will use it to prove that $\I$ satisfies a property called the sundial configuration (SC) in \cite[Section 2]{bennett2014axiomatic} (see Lemma~\ref{lemHeb3.6}).

\begin{lemma}\label{lemIntersection_apartments_containing_half-apartment}
Let $A$ be an apartment of $\I$. Let $D$ be a half-apartment of $\A_S$. Suppose that $A\cap \A_S$ contains $D$. Then either $A=\A_S$ or there exists $\alpha\in \Phi$ and  $\lambda\in \Gamma_\alpha$ such that $A\cap \A_S=D_{\alpha,\lambda}$. 
\end{lemma}

\begin{proof}
Suppose $A\neq \A_S$. Let $\alpha\in \Phi$ be such that $A\cap \A_S\supset D_{\alpha,\lambda}$, for some $\lambda\in \RF^S$. Using~\ref{axiomA2} we write $A\cap \A_S=\bigcap_{\beta\in \Phi} D_{\beta,\ell_\beta}$, where $(\ell_\beta)\in \prod_{\beta\in \Phi} (\Gamma_\beta\cup \{\infty\})$. Let $\beta\in \Phi\setminus \R_{>0}\alpha$. Let $\mu\in \RF^S$. Let us prove that $\ell_\beta\geq \mu$. Let  $s=\min \big( \supp(\lambda)\cup \supp(\mu)\big)$. Write $\lambda=(\lambda_t)_{t\in S}$. Suppose $\beta\notin \R \alpha$. Then for every $n\in \N$, there exists $x_n\in \A_\R$ such that $\alpha(x_n)=-\lambda_s+1$ and $\beta(x_n)=-n$. Let $(y_n)=(\delta_{s,t}x_n)_{t\in S}\in \A_S$. Then $\alpha(y_n)\geq \lambda$  and thus $y_n\in D_{\alpha,\lambda}$, for $n\in \N$. If $s=\min\big(\supp(\mu)\big)$, then $-\beta(y_n)\geq \mu$, for $n\gg 0$. If $s<\min \big(\supp(\mu)\big)$, then $-\beta(y_n)\geq \mu$, for $n> 0$. Moreover, $\ell_\beta\geq -\beta(y_n)$, for all $n\in \N$ and thus $\ell_\beta\geq \mu$.

Suppose $\beta\in \R_{<0} \alpha$. For $n\in \N$, choose $x_n\in \A_\R$ such that $\alpha(x_n)=n$ and set $y_n=(\delta_{s,t}x_n)_{t\in S}\in \A_S$. Then for $n\gg 0$, $x_n\in D_{\alpha,\lambda}$ and $-\beta(x_n)\geq \mu$ and thus $\ell_\beta\geq \mu$. Therefore in both cases, $\ell_\beta\in \bigcap_{\mu\in \RF^S}[\mu,\infty]=\{\infty\}$. Consequently, $A\cap \A_S=\bigcap_{\beta\in \R_{>0}\alpha}D_{\alpha,\ell_\alpha}$. For $m\in \{\frac{1}{2},2\}$, set $\ell_{m\alpha}=\infty$ if $m\alpha\notin \Phi$.
Then $A\cap \A_S=D_{\alpha,\ell'}$, where $\ell'=\min_{m\in \{\frac{1}{2},1,2\}}\frac{1}{m}\ell_{m\alpha}$.  Let $m\in \{\frac{1}{2},1,2\}$ be such that $\ell'=\frac{1}{m}\ell_{m\alpha}$. Then $A\cap \A_S=D_{m\alpha,\ell_{m\alpha}}$, which proves the lemma.
\end{proof}

\begin{lemma}\label{lemExchange_condition}
The set $\I$ satisfies the exchange condition (EC): if $A$ and $B$ are apartments of $\I$ such that $A\cap B$ is a half-apartment, then $(A\cup B)\setminus (A\cap B)\cup M$ is an apartment of $\I$, where $M$ is the wall of $A\cap B$.
\end{lemma}

\begin{proof}
Using isomorphisms of apartments, we may assume that $A=\A_S$. Then by Lemma~\ref{lemIntersection_apartments_containing_half-apartment}, there exists $\alpha\in \Phi$ and $\lambda\in \Gamma_\alpha$ such that $D:=\A_S\cap B=D_{\alpha, \lambda}$. By Corollary~\ref{corLandvogt9.7}, one has $B=u.\A_S$, where $u\in U_D$. By Example~\ref{ExQC}, we can write $u=u_D^+ u_D^-n_D$, where $u_D^+\in U_D^+$,  $u_D^-\in U_D^-$ and $n_D\in N$. Moreover, either  $U_{D}^+=\{1\}$ or $U_{D}^-=\{1\}$ (depending on whether $\alpha\in \Phi_+$ or $\alpha\in \Phi_-$). By symmetry, we may assume that $U_{D}^+=\{1\}$. Then  $U_{D}^-=U_{\alpha,D}$. 

One has $\lambda=\varphi_{\alpha}(u_D^-)$.  Using~\ref{axiomRGD4}, we write $u_D^-=v_+m u_+^{-1}$, with $u_+,v_+ \in U_D^+$  and $m\in M_\alpha$. By~\ref{axiomV5} and Lemma~\ref{LemAxiomV5}, we  have $\lambda=-\varphi_{-\alpha}(u_+)=-\varphi_{-\alpha}(v_+)=\varphi_{\alpha}(u_D^-)$.
Let   $D_-=D_{\alpha,\lambda}=D$ $D_+=D_{-\alpha,-\lambda}$ and $M=H_{\alpha,\lambda}=D_-\cap D_+$. Then one has: : \[u_D^-u_+.D_-=v_+m.D_-=v_+.D_+=D_+\text{ and }u_D^-.D_+ = u_D^-u_+.D_+=v_+m.D_+=v_+.D_-.\]
Therefore:
 since $u_D^-.\A_S\cup \A_S = D_- \cup u_D^-.D_+ \cup D_+$ and $\A_S\cap u_D^-.\A_S = D_-$, we have:

\[(u_D^-.\A_S\cup \A_S)\setminus(\A_S\cap u_D^-.\A_S)\cup M=u_D^-.D_+\cup D_+=v_+.D_+\cup D_+=v_+.\A_S, \]
which is an apartment of $\mathcal{I}$, which proves the lemma.
\end{proof}

\subsubsection{Germ of a gallery of local chambers and conclusion}

If $C$ is a sector and $s\in S$, recall the definition of $C_{\leq s}$ in Notation~\ref{notProjection_sector}.

\begin{lemma}\label{lemGerm_sectors_coinciding_before_s}
Let $C_\infty,\tilde{C}_\infty$ be two sector-germs at infinity and $x\in \I$. Let $C=x+C_\infty$ and $\tilde{C}=x+\tilde{C}_\infty$.  We assume that there exists $s\in S$ such that  $\germ_{\pi_{\leq s}(x)}(C_{\leq s})=\germ_{\pi_{\leq s}(x)}(\tilde{C}_{\leq s})$. Then $\germ_x(C)=\germ_x(\tilde{C})$. 
\end{lemma}

\begin{proof}
Let $A$ be an apartment containing $C$.  Then $\pi_{\leq s}(A)\Supset \germ_{\pi_{\leq s}(x)}(C_{\leq s})=\germ_{\pi_{\leq s}(x)}(\tilde{C}_{\leq s})$.
Let $\tilde{A}$ be an apartment containing $\tilde{C}$. Let $\Omega_{\leq s}$ be a neighborhood of $\pi_{\leq s}(x)$ in $\pi_{\leq s}(A)$ such that $\Omega_{\leq s}\cap C_{\leq s}=\Omega_{\leq s}\cap \tilde{C}_{\leq s}$. Let $\tilde{y}_{\leq s}\in \Omega_{\leq s}\cap C_{\leq s}=\Omega_{\leq s}\cap \tilde{C}_{\leq s}$. 
By Lemma~\ref{lemProjection_intersection_apartments}, $\pi_{\leq s}(A\cap \tilde{A})=\pi_{\leq s}(A)\cap \pi_{\leq s}(\tilde{A})$. Thus  there exists $\tilde{y}\in A\cap \tilde{A}$ such that $\pi_{\leq s}(\tilde{y})=\tilde{y}_{\leq s}$.

Let us prove that $\tilde{y}\in C$. Identify   $A$ with $\A_S$. Write  $C=\{x\in \A_S \mid \alpha(x)>\lambda_\alpha,\forall \alpha\in \Delta_C\}$, where $\Delta_C$ is some subset of $\Phi$ and $(\lambda_\alpha)\in (\RF^S)^{\Delta_C}$. By definition, $C_{\leq s}=\{x\in \A_{\leq s}\mid \alpha\big(\pi_{\leq s}(x)\big)>\pi_{\leq s}(\lambda),\forall \alpha\in \Delta_C\}$. Then $\tilde{y}_{\leq s}=\pi_{\leq s}(\tilde{y})\in C_{\leq s}$ and thus $\tilde{y}\in C$. Similarly, as  $\tilde{y}\in \tilde{A}$, we have $\tilde{y}\in \tilde{C}$. Using Lemma~\ref{lemEnclosure_point_in_sector}  and Proposition~\ref{propAxiom(A2)}, we deduce that  $A\cap \tilde{A}\Supset \cl(\{x,\tilde{y}\})\Supset \germ_x(C)$. In other words, there exists a neighborhood $V$ of $x$ in $A$ such that $\Omega:=V\cap C\subset A\cap \tilde{A}$. Let $\phi:A\rightarrow \tilde{A}$ be the apartment isomorphism fixing $A\cap \tilde{A}$. Then $\phi(\tilde{y})=\tilde{y}$. The sector $C$ (resp. $\tilde{C}$) is the unique sector of $A$ (resp. $\tilde{A}$) based at $x$ and containing $\tilde{y}$. Therefore $\phi(C)=\tilde{C}$. As $\Omega\in \germ_x(C)$, we have   $\phi(\Omega)=\Omega\in \germ_{\phi(x)}(\phi(C))=\germ_x(\tilde{C})$ and thus $\germ_x(C)=\germ_x(\tilde{C})$.

\end{proof}

Let $C_1^v$ and $C^v_2$ be two vector chambers of $\A_S$. Let $\Delta_1$ be the  basis of $\Phi$ associated with $C_1^v$. We say that $C_1^v$ and $C_2^v$ are \textbf{adjacent}\index{adjacent}  if $C_1^v=C_2^v$ or if there exists $\alpha\in \Delta_1$ such that $C^v_2=r_\alpha.C^v_1$.  Let $C_1$ and $C_2$ be two sectors of $\A_S$ based at the same point. We say that $C_1$ and $C_2$ are \textbf{adjacent} if their directions are. We extend this definition to the pairs of sectors of $\I$ which are contained in a common apartment and which have the same base point, which is possible by~\ref{axiomA2}.

Two sector-germs $Q_{1,\infty}$ and $Q_{2,\infty}$ are said to be \textbf{adjacent} if there exists an  apartment $A$ containing $Q_{1,\infty}$ and $Q_{2,\infty}$ and such that $x+Q_{1,\infty}$ and $x+Q_{2,\infty}$ are adjacent, for any $x\in A$. If such an $A$ exists, then any apartment containing $Q_{1,\infty}$ and $Q_{2,\infty}$ satisfies this property.

\medskip

\textbf{From now on and until Lemma~\ref{lemCO_when_S_admits_minimum}, we assume that $S$ admits a minimum.} We denote it $0_S$. We denote $\pi_{=0_S}$ instead of $\pi_{\leq 0_S}$ and  if $C$ is a sector of $\I$, we denote $C_{=0_S}$ instead of $C_{\leq 0_S}$. We assume moreover that $\pi_{=0_S}\left(\omega(\Kb^*)\right)\neq \{0\}$.

\begin{lemma}\label{l_projection_0_sector}
Let $C$ be a sector of $\I$. Then $\germ_{\infty}(C_{=0_S})=\pi_{=0_S}\big(\germ_\infty(C)\big)$. In particular the sector-germs at infinity of $\pi_{=0_S}(\I)$ are exactly the $\pi_{=0_S}(C_\infty)$ such that $C_\infty$ is a sector-germ at infinity of $\I$.
\end{lemma}

\begin{proof}
As $\pi_{=0_S}$, $\germ_\infty$ and $C\mapsto C_{=0_S}$ are $\mathbf{G}(\mathbb{K})$-equivariant, it suffices to check it when $C$ is a subsector of $C^v_{f,S}$, which is straightforward.
\end{proof}

\begin{lemma}\label{lemHeb3.6}
Let $A$ be an apartment of $\I$ and $C_\infty$ be a sector-germ of $A$. Let $\tilde{C}_\infty$ be a sector-germ of $\I$ adjacent to $C_\infty$ and different from it. Then one can write $\pi_{=0_S}(A)=D_{1,0_S}\cup D_{2,0_S}$, where for both $i\in \{1,2\}$, $D_{i,0_S}$ is a half-apartment of $\pi_{=0_S}(A)$ and there exists an apartment $\pi_{=0_S}(A_i)$ containing $D_{i,0_S}$ and $\pi_{=0_S}(\tilde{C}_\infty)$.
\end{lemma}

\begin{proof}
By Lemma~\ref{lemDecompositions_axioms}, Proposition~\ref{propAxiom(A2)},  Lemma~\ref{lemCO_for_R_buildings}  by \cite[Theorem 3.3]{bennett2014axiomatic} and as $\pi_{=0_S}(\omega(\Kb^*))\neq \{0\}$,  $\pi_{=0_S}(\I)$ is an $\R$-building and in particular, it satisfies the axiom (SC) of \cite{bennett2014axiomatic}. Moreover $\pi_{=0_S}(C_\infty)$  and $\pi_{=0_S}(\tilde{C}_\infty)$ are adjacent and thus  the lemma is a consequence of the paragraph after (SC) in \cite[page 385]{bennett2014axiomatic}.

\end{proof}

\paragraph{Galleries of local chambers}

Let $x\in \I$. Two local chambers $C_{1,x}$, $C_{2,x}$ are called \textbf{adjacent} if there exist an apartment $A$ containing $C_{1,x},C_{2,x}$ and two adjacent sectors $Q_1,Q_2$ of $A$ based at $x$ such that  $\germ_x(Q_1)=C_{1,x}$ and $\germ_x(Q_2)=C_{2,x}$. By~\ref{axiomA2}, this does not depend on the choice of apartment $A$. A \textbf{gallery of local chambers based at $x$}\index{gallery} is a finite sequence $\Gamma_x=(C_{1,x},\ldots,C_{k,x})$ such that for all $i\in \llbracket 1,k-1\rrbracket$, $C_{i,x}$ and $C_{i+1,x}$ are adjacent local chambers based at $x$. The \textbf{length} of $\Gamma_x$ is then $k$. The gallery is called  \textbf{minimal} if $k$ is the minimal possible length for a gallery joining $C_{1,x}$ and $C_{k,x}$. 

If $C_{x}$, $\tilde{C}_{x}$ are two local chambers based at $x$, there exists a gallery $\Gamma_x$ joining $C_{x}$ to $\tilde{C}_{x}$. Indeed, by \ref{axiomGG}, there exists an apartment $A$ containing $C_x$ and $\tilde{C}_x$. Let $Q$ and $\tilde{Q}$ be the sectors of $A$ corresponding to $C_x$ and $\tilde{C}_x$. Then if $\Gamma$ is a gallery of sectors from $Q$ to $\tilde{Q}$, then the germ $\Gamma_x$ of $\Gamma$ at $x$ is a gallery joining $C_x$ to $\tilde{C}_x$. The minimal length of a  gallery joining $C_x$ to $\tilde{C}_x$ is called  the \textbf{distance between $C_x$ and $\tilde{C}_x$} and we denote it $d(C_x,\tilde{C}_x)$.

\begin{lemma}\label{lemApartment_projection_contains_germ}
Let $A$ be an apartment and $C_\infty$ be a sector-germ at infinity of $\I$. We assume that $\pi_{=0_S}(A)\Supset \pi_{=0_S}(C_\infty)$. Then $A\Supset C_\infty$.
\end{lemma}

\begin{proof}
Let $\Omega\in C_\infty$ be such that $\pi_{=0_S}(A)\supset \pi_{=0_S}(\Omega)$. Let $B$ be an apartment containing $C_\infty$. Then $\Omega\cap B\in C_\infty$. Moreover $\pi_{=0_S}(A)\cap \pi_{=0_S}(B)\supset \pi_{=0_S}(\Omega\cap B)$ and by Lemma~\ref{lemProjection_intersection_apartments}, $\pi_{=0_S}(A\cap B)=\pi_{=0_S}(A)\cap \pi_{=0_S}(B)$.  By $\mathbf{G}(\mathbb{K})$-equivariance, we can identify $B$ and $\A_S$ and we can assume that $C_\infty$ is the germ of $C^v_{f,S}$.  We choose $(y_{n,=0_S})\in \pi_{=0_S}(A\cap \A_S)^{\N}$ such that $\beta(y_{n,=0_S})\rightarrow +\infty$ for all $\beta\in \Delta_f$. Let $(y_n)\in (A\cap \A_S)^{\N}$ be such that $\pi_{=0_S}(y_n)=y_{n,0_S}$ for all $n\in \N$. Then $A \cap \A_S\Supset \cl(\{y_n|n\in \N\})\Supset C_\infty$, which proves the lemma.

\end{proof}

 Recall that a vector panel of $\A_S$ is a set of the form $P^v=H_{\beta,0}\cap \bigcap_{\alpha\in \Delta\setminus\{\beta\}} D_{\alpha,0}$, where $\Delta$ is a basis of $\Phi$ and $\beta\in \Delta$. A panel of $\A_S$  is a set of the form $x+P^v$, for some vector panel $P^v$. 

Let $C^v$ be a vector chamber of $\A_S$ and $\Delta$ be the associated basis of $\Phi$. We set $\overline{C^v}=\bigcap_{\alpha\in \Delta}D_{\alpha,0}$. If $Q$ is a sector of $\A_S$ with direction $C^v$ and base point $x$, we set $\overline{Q}=x+\overline{C^v}=\{x+u\mid u\in \overline{C^v}\}$. Let $Q$ be a sector of $\I$ and  $g\in \mathbf{G}(\mathbb{K})$ be such that $g.Q\subset \A_S$. We set $\overline{Q}=g^{-1}.(g.\overline{Q})$. This is well defined by~\ref{axiomA2}. 

Let $x\in \I$, $Q$ be a sector of $\I$ based at $x$ and $F$ be a sector-face based at $x$. We say that $Q$ \textbf{dominates}\index{domination} $F$ if $F\subset \overline{Q}$.

\begin{lemma}\label{lemApartment_thin_building}
Let $P$ be a panel of $\A_S$. Then there exist exactly two sectors of $\A_S$ dominating $P$.
\end{lemma}

\begin{proof}
Using translations, we may assume that $P$ is based at $0$. Let $\Delta$ be a basis of $\Phi$ and $\beta\in \Delta$ be such that $P=H_{\beta,0}\cap \bigcap_{\alpha\in \Delta\setminus\{\beta\}} D_{\alpha,0}$. Let $C^v=\bigcap_{\alpha\in \Delta} \mathring{D}_{\alpha,0}$.  Then $C^v$ and  $r_\beta.C^v$ dominate $P$. Let $P_\R=H_{\beta,0_\R}\cap \bigcap_{\alpha\in \Delta\setminus\{\beta\}} \mathring{D}_{\alpha,0_\R}\subset \A_\R$ and $C^v_\R=\bigcap_{\alpha\in \Delta} D_{\alpha,0_\R}\subset \A_\R$. Let $C^v_1$ be a vector chamber of $\A_S$ dominating $P$.  Let $w\in W^v$ be such that $w.C^v_1=C^v$. Let $s\in S$ and $x_\R\in P_\R$. Let $x=(x_s)_{t\in S}$ be defined by $x_t=\delta_{s,t}x_\R$, for $t\in S$. Then $x\in P$ and thus $w.x\in \overline{C^v}$. Therefore $w.x_\R\in \overline{C^v_\R}$. By \cite[V. 3.3 Théorème 2]{bourbaki1981elements} we deduce that $w.x_\R=x_\R$. By \cite[V. 3.3 Proposition 1 \&2]{bourbaki1981elements}, $w\in \langle r_\beta\rangle$. Therefore $C^v_1\in \langle r_\beta\rangle C^v$, which proves the lemma.
\end{proof}

Let $H$ be a wall of $\A_S$ and $Q_1,Q_2$ be two sectors of $\A_S$. We say that $H$ \textbf{separates}\index{separation} $Q_1$ and $Q_2$ if the two half-apartments $D_1$ and $D_2$ delimited by $H$ satisfy either $Q_1\subset D_1$ and $Q_2\subset D_2$ or $Q_1\subset D_2$ and $Q_2\subset D_1$. We extend these notions to $\I$ using isomorphisms of apartments.

\begin{lemma}\label{lemSundial_configuration}
\begin{enumerate}
\item\label{itSundial_configuration} Let $C$ and $\tilde{C}$ be two sectors of $\I$ and $C_{\infty}$, $\tilde{C}_{\infty}$ be their germs. We assume that $C_{\infty}$ and $\tilde{C}_{\infty}$ are adjacent. Let  $A$ be an apartment containing $C_\infty$. Suppose that  $A$ does not  contain $\tilde{C}_\infty$. Then  one can write $A=D_1\cup D_2$, where $D_1,D_2$ are half-apartments of $A$ which have the same wall and such that for both $i\in \{1,2\}$, there exists an apartment containing $D_i$ and $\tilde{C}_\infty$. 

\item\label{itGerme_gallery} Let $y\in \I$.  Then $\germ_y(y+C_\infty)$ and $\germ_y(y+\tilde{C}_\infty)$ are adjacent.
\end{enumerate}
\end{lemma}

\begin{proof}
Let $\pi=\pi_{=0_S}$. By Lemma~\ref{lemHeb3.6}, one can write $\pi(A)=D_{1,\pi}\cup D_{2,\pi}$, where  $D_{1,\pi}$ and $\pi(\tilde{C}_\infty)$ are contained in an apartment $\pi(A_1)$. Then $\pi(A)\cap \pi(A_1)$ is a half-apartment and thus by Lemma~\ref{lemProjection_intersection_apartments} and~\ref{axiomA2}, $A\cap A_1$ is a half-apartment. Let $H$ be the wall of $A\cap A_1$. Then by (EC) (Lemma~\ref{lemExchange_condition}), $\big(A\cup A_1)\setminus (A\cap A_1)\big)\sqcup H$  is an apartment $A_2$ of $\I$. By Lemma~\ref{lemApartment_projection_contains_germ}, $A_1\Supset \tilde{C}_\infty$. Let $D_1=A\cap A_1$ and $D_2=A_2\cap A$. Then $(A,D_1,D_2)$  satisfies the condition of~(\ref{itSundial_configuration}).

Let now $y\in \I$. We want to prove that $\germ_y(y+C_\infty)$ and $\germ_y(y+\tilde{C}_\infty)$ are adjacent. As there exists an apartment containing $y$ and $C_\infty$,  there is no loss of generality in assuming that $y\in A$.

Maybe exchanging the roles of $D_1$ and $D_2$, we may assume that $D_1\Supset C_\infty$. Suppose $y\in D_1$. Then $A_1$ contains $y+C_\infty$ and $y+\tilde{C}_\infty$ and thus $\germ_y(y+C_\infty)$ and $\germ_y(y+\tilde{C}_\infty)$ are adjacent (and distinct).

  By construction, $H$ is also the wall of $A\cap A_2$. Let $x\in H$,  $C=x+C_\infty$ and $\tilde{C}=x+\tilde{C}_\infty$. Then $x+C_\infty\subset D_1$ and $x+\tilde{C}_\infty\nsubseteq A$. Therefore $H$ separates  $x+C_\infty$ and $x+\tilde{C}_\infty$ in $A_1$.
  
  Let $\phi:A\rightarrow A_2$ be the isomorphism of apartments fixing $A\cap A_2$. Then $\phi$ induces a map (still denoted $\phi$) from the set of sector-germs at infinity of $A$ to the set of sector-germs at infinity of $A_2$. If $x\in H$,  then  $x+C_\infty$ and $x+\tilde{C}_\infty$  dominate a panel  $P$ of  $H$. Let $C'_\infty$ be the sector-germ at infinity of $A$ different from $C_\infty$ such that $x+C'_\infty$ dominates $P$ (the uniqueness is a consequence of Lemma~\ref{lemApartment_thin_building}). Then $C'_\infty$ is not contained in $D_1$, hence $C'_\infty\Subset D_2$ and thus $\phi(C'_\infty)=C'_\infty$. Moreover $P\subset A\cap A_2$ and $P\subset \overline{x+C_\infty}$.   Therefore $P\subset \phi(\overline{x+C_\infty})=\overline{x+\phi(C_\infty)}$. Thus $x+\phi(C_\infty)$ dominates $P$.  As $\phi(C_\infty)\neq \phi(C_\infty')$, Lemma~\ref{lemApartment_thin_building} implies that  $\phi(C_\infty)=\tilde{C}_\infty$ and thus for all $y\in A$, $\phi(y+C_\infty)=\phi(y)+\tilde{C}_\infty$. 
  
  Let now $y\in A\setminus D_1$. Let $\epsilon\in \RF^S$ be such that $B(y,\epsilon)\subset D_2$, which exists by Lemma~\ref{lemDistance}~(\ref{itBalls_bounded}). Then $\phi\big(B(y,\epsilon)\cap (y+C_\infty)\big)= B(y,\epsilon)\cap \phi(y+C_\infty)=B(y,\epsilon)\cap (y+\tilde{C}_\infty)$.  Therefore $\germ_y(y+C_\infty)=\germ_y(y+\tilde{C}_\infty)$ and in particular $\germ_y(y+C_\infty)$ and $\germ_y(y+\tilde{C}_\infty)$ are adjacent. Lemma follows. 

\end{proof}

Using Lemma~\ref{lemSundial_configuration}~(\ref{itGerme_gallery}) we deduce the following proposition.

\begin{proposition}\label{propGerm_gallery_sector_germs}
Let $\Gamma=(C_{1,\infty},\ldots,C_{n,\infty})$ be a gallery of sector-germs at infinity. Then for all $x\in \I$, $\Gamma_x=\big(\germ_x(x+C_{1,\infty}),\ldots,\germ_x(x+C_{2,\infty})\big)$ is a gallery.
\end{proposition}

\begin{lemma}\label{lemGalleries_maximal_length}
Let $x\in \I$. Let  $C_x,\tilde{C}_x$ be two local chambers based at $x$.  Let $w_0$ be the longest element of $W^v$. Then:\begin{enumerate}
\item\label{itMajoration_combinatorial_distance}  $d(C_x,\tilde{C}_x)\leq \ell(w_0)$ and

\item\label{itChamber_maximal_combinatorial_distance} $d(C_x,\tilde{C}_x)=\ell(w_0)$ if and only if $C_x$ and $\tilde{C}_x$ are opposite.
\end{enumerate} 
\end{lemma}

\begin{proof}
Let $A$ be an apartment such that $\pi_{\leq s}(A)$ contains $C_x$ and $\tilde{C}_x$. Using an isomorphism of apartments, we identify $A$ and $\A_S$.   Let $C^v_S$ and $\tilde{C}^v_S$ be the vector chambers of $\A_S$ such that $C_x=\germ_x(x+C^v_S)$ and $\tilde{C}_x=\germ_x(x+\tilde{C}^v_S)$. Write $\tilde{C}^v_S=w.C^v_S$, with $w\in W^v$. By \cite[I Proposition 4]{brown1989buildings}, one has $\ell(w)=d(C^v_S,\tilde{C}^v_S)=d(C_x,\tilde{C}_x)$. By definition of $w_0$ we deduce that $d(C_x,\tilde{C}_x)\leq \ell(w_0)$. Point~\ref{itChamber_maximal_combinatorial_distance} is a consequence of the uniqueness of $w_0$, see \cite[Proposition 2.2.9]{bjorner2005combinatorics} for example.
\end{proof}

\begin{lemma}\label{lemProjection_opposite_sectors}
Let $x\in \I$ and $C$, $\tilde{C}$ be two sectors opposite at $x$. Let $s\in S$. Then $C_{\leq s}$ and $\tilde{C}_{\leq s}$ are opposite at $\pi_{\leq s}(x)$.
\end{lemma}

\begin{proof}
Let $A$ be an apartment containing $\germ_{\pi_{\leq s}(x)}(C_{\leq s})$ and $\germ_{\pi_{\leq s}(x)}(\tilde{C}_{\leq s})$. Let $Q$ and $\tilde{Q}$ be the sectors of $A$ based at $\pi_{\leq s}(x)$ and such that $Q_{\leq s}$ and $\tilde{Q}_{\leq s}$ contain $\germ_{\pi_{\leq s}(x)}(C_{\leq s})$  and $\germ_{\pi_{\leq s}(x)}(\tilde{C}_{\leq s})$. Let $\Gamma=(Q_{1,\infty},Q_{2,\infty},\ldots,Q_{k,\infty})$ be a minimal gallery of sector-germs at infinity of $A$ from $\germ_\infty(Q)$ to $\germ_{\infty}(\tilde{Q})$. Then by Proposition~\ref{propGerm_gallery_sector_germs}, $\Gamma_x=\big(\germ_x(x+Q_{1,\infty}),\ldots,\germ_x(x+Q_{k,\infty})\big)$ is a gallery. By Lemma~\ref{lemGerm_sectors_coinciding_before_s}, $\germ_x(x+Q_{1,\infty})=\germ_x(C)$ and $\germ_{x}(x+Q_{k,\infty})=\germ_x(\tilde{C})$. As $\germ_x(C)$ and $\germ_x(\tilde{C})$ are opposite, we have $k\geq \ell(w_0)$. As $\Gamma$ is minimal, one has $k\leq \ell(w_0)$. Consequently, $k=\ell(w_0)$ and thus by Lemma~\ref{lemGalleries_maximal_length}, $\germ_{\pi_{\leq s}(x)}(C_{\leq s})$ and $\germ_{\pi_{\leq s}(x)}(\tilde{C}_{\leq s})$ are opposite, which proves the lemma.
\end{proof}

\subsection{Proof of~\ref{axiomCO} when \texorpdfstring{$S$}{S} admits a minimum}\label{subProof_CO_when_S_admits_minimum}
 
\begin{lemma}\label{l_fibers_opposite_sectors}
Let $T$ be a totally ordered set. Let $T_1,T_0\subset T$ be such that $T=T_1\sqcup T_0$ and $T_1<T_0$.   
Let $\A_{T_i}=\A_\R\otimes \RF^{T_i}$ for $i\in \{1,2\}$ and $\A_T=\A_{T_1}\times \A_{T_0}=\A_\R\otimes \RF^T$. Let $\pi:\RF^T\twoheadrightarrow \RF^{T_1}$ be the canonical projection. Let $x\in \A_T$ and  $C$ and $C'$ be two sectors of $\A_T$ opposite at $x$. Let $X_1=\pi(x)\in \A_{T_1}$. Write  $\pi^{-1}(X_1)\cap C=\{\pi(x)\}\times C_0$ and  $\pi^{-1}(X_1)\cap C'=\{\pi(x)\}\times C_0'$, where $C_0,C_0'\subset \A_{T_0}$. Then $C_0$ and $C_0'$ are sectors of $\A_{T_0}$ which are opposite at their common basis.
\end{lemma}

\begin{proof}
    Write $C=x+w.C^v_{f,T}$ and $C'=x+w'.C^v_{f,T}$, where $w,w'\in W^v$ and $C^v_{f,T}=\{a\in \A_T\mid \alpha(a),\forall \alpha\in \Phi^+\}$. Write $x=(x_t)_{t\in T}$ and set $x_{T_0}=(x_t)_{t\in T_0}$. Then: \[\begin{aligned} \pi^{-1}(X_1)\cap C&=\{a\in \{\pi(x)\}\times \A_{T_0}\mid \alpha(a)>\alpha(x),\forall\alpha\in w.\Phi^+\} \\
     &=\{(\pi(x),a')\in \{\pi(x)\}\times \A_{T_0}\mid \alpha(a')>\alpha(x_{T_0}), \forall\alpha\in w.\Phi^+\}\\
     &=\{\pi(x)\}\times (x_{T_0}+w.C^v_{f,T_0}).\end{aligned}\] Similarly,  $\pi^{-1}(X_1)\cap C'=\{\pi(x)\}\times (x_{T_0}+w'.C^{v}_{f,T_0}),$ and the lemma follows.
     \end{proof}

\begin{lemma}\label{lemCO_when_S_admits_minimum}
We assume that $S$ admits a minimum $0_S$ and that $\pi_{\leq 0_S}(\omega(\Kb^*))\neq \{0\}$). Let $x\in \mathcal{I}$ and let $C$ and $\tilde{C}$ be two sectors opposite at $x$. Then there exists an apartment $A$ containing $C$ and $\tilde{C}$.
\end{lemma}

\begin{proof}
Let $A_C$ and $A_{\tilde{C}}$ be apartments containing $C$ and $\tilde{C}$ respectively. By~\ref{axiomA4} (Lemma~\ref{lemDecompositions_axioms}) we can find an apartment $A_{C',\tilde{C}'}$ containing subsectors $C'$ and $\tilde{C}'$ of $C$ and $\tilde{C}$. Let us prove that $x\in A_{C',\tilde{C}'}$. \\

Let $\phi: A_C \rightarrow A_{C',\tilde{C}'}$ be an apartment isomorphism such that $\phi|_{A_C \cap A_{C',\tilde{C}'}} = \mathrm{Id}|_{A_C \cap A_{C',\tilde{C}'}}.$ Set $y:=\phi(x)$. We want to prove that $x=y$. By contradiction, assume that $x\neq y$.\\

Let $A_{x,y}$ be an apartment containing both $x$ and $y$ and $\psi: A_{x,y} \rightarrow \mathbb{A}_S$. Let $\supp(y-x)\subset S$ be the support  of $\psi(y)-\psi(x)\in \A_S$. Then $\supp(y-x)$ depends  neither on the choice of $A_{x,y}$ nor on the choice of $\psi$. By assumption, $\supp(y-x)$ is nonempty. Let $s_0$ be the minimum of $\supp(y-x)$.  Let $g \in \mathbf{G}(\mathbb{K})$ inducing the isomorphism $\psi$, whose existence is provided by Proposition~\ref{propAxiom(A2)}. Then we have:
$$g \cdot\pi_{<s_0}( x)=\pi_{<s_0}(g \cdot x) = \pi_{<s_0}(\psi( x)) =\pi_{<s_0}(\psi( y)) = \pi_{<s_0}(g \cdot y)=g \cdot\pi_{<s_0}( y)$$
and:
$$g \cdot\pi_{\leq s_0}( x)=\pi_{\leq s_0}(g \cdot x) = \pi_{\leq s_0}(\psi( x))\neq \pi_{\leq s_0}(\psi( y)) = \pi_{\leq s_0}(g \cdot y)=g \cdot\pi_{\leq s_0}( y)$$
and hence $\pi_{\leq s_0}( x)\neq \pi_{\leq s_0}( y)$.

Set $X=\pi_{<s_0}(x)=\pi_{<s_0}(y)$ and $\I_X=\pi_{<s_0}^{-1}(X)\subset \I$. Set $$\pi_{=s_0}=\pi_{=s_0,X}:\I_X\twoheadrightarrow (\pi_{<s_0}^{\leq s_0})^{-1}(X)\subset \I(\mathbb{K},\omega_{\leq s_0},G),$$ with the notation of~\ref{subFurther_notation}. Then $\pi_{=s_0}(x)\neq\pi_{=s_0}(y)$ and $(\pi_{<s_0}^{\leq s_0})^{-1}(X)$ is an $\R$-building.

Since $x\in A_C \cap \I_{X}$ and $y\in A_{C',\tilde{C}'} \cap \I_{X}$, the sets $A_C \cap \I_{X}$ and $A_{C',\tilde{C}'} \cap \I_{X}$ are nonempty. Moreover, let $h \in \mathbf{G}(\mathbb{K})$ inducing $\phi$. Then:
\begin{align*}
\phi(A_C \cap \I_{X}) \subseteq A_{C',\tilde{C}'} \cap (h\cdot \I_{X})& = A_{C',\tilde{C}'} \cap \pi_{<s_0}^{-1}(h\cdot X)= A_{C',\tilde{C}'} \cap \pi_{<s_0}^{-1}(h\cdot X)\\& =
A_{C',\tilde{C}'} \cap \pi_{<s_0}^{-1}\big(\phi( X)\big)=A_{C',\tilde{C}'} \cap \pi_{<s_0}^{-1}\big(\pi_{<s_0}(y)\big)\\&=A_{C',\tilde{C}'} \cap \I_{X}.
\end{align*}
Hence $\phi$ induces a map $\phi_{\geq s_0} : A_C \cap \I_{X}\rightarrow A_{C',\tilde{C}'} \cap \I_{X}$, which induces itself a map:
$$\phi_{=s_0}: A_{C,=s_0} \rightarrow A_{C',\tilde{C}',=s_0}$$
where $A_{C,=s_0}:= \pi_{=s_0}(A_C) \cap \I_{X}$ and
$A_{C',\tilde{C}',=s_0}=\pi_{=s_0}(A_{C',\tilde{C}'} )\cap \I_{X}$.  

We know that $\phi_{=s_0}\big(\pi_{=s_0}(x)\big) = \pi_{=s_0}(y) \neq \pi_{=s_0}(x)$. By Lemma~\ref{lemProjection_opposite_sectors}, $C_{\leq s_0}$ and $\tilde{C}_{\leq s_0}$ are opposite at $\pi_{\leq s_0}(x)$. By  Lemma~\ref{l_fibers_opposite_sectors} applied with $T=\{s\in S\mid s\leq s_0\}$, $T_1=T\setminus\{s_0\}$ and $T_0=\{s_0\}$, we deduce that  the sectors-germs $C_{=s_0}:=\pi_{=s_0}(C_{\leq s_0}\cap \I_{\pi_{<s_0}(x)})$ and $\tilde{C}_{=s_0}:=\pi_{=s_0}(\tilde{C}_{\leq s_0}\cap \I_{\pi_{<s_0}(x)})$ are opposite in $\pi_{=s_0}(x)$ (in the $\R$-building $(\pi_{<s_0}^{\leq s_0})^{-1}(X)$). 

 Hence, by Lemma~\ref{lemCO_for_R_buildings}, $\pi_{=s_0}(x)$ belongs to the unique appartment of $\pi_{=s_0}(\I_{\pi_{<s_0}(x)})$ containing $\mathrm{germ}_{\infty}(\pi_{=s_0}(C))$ and $\mathrm{germ}_{\infty}(\pi_{=s_0}(C))$. In other words, $\pi_{=s_0}(\I_{\pi_{<s_0}(x)}) \in A_{C,=s_0} \cap A_{C',\tilde{C}',=s_0}$. Since $\phi$ fixes $A_C \cap A_{C',\tilde{C}'}$, we deduce that $\phi_{=s_0}(\pi_{=s_0}(x)) = \pi_{=s_0}(y) \neq \pi_{=s_0}(x)$: contradiction! Hence $x=y$. 
\end{proof}

\subsection{Proof of~\ref{axiomCO} in the general case}\label{subProof_CO_general_case}

Let $\hat{\mathbb{K}}=\mathbb{K}(\!(t)\!)$. Let $\hat{\omega}:\hat{\mathbb{K}}\rightarrow \Z\times \Lambda$ be defined by
$\hat{\omega}(\sum_{k=n}^\infty a_k t^k)=\big(n,\omega(a_n)\big)$, 
if $a_n\neq 0$. Let $\I=\I(\Kb,\omega,\GB)$
and  $\hat{\I}=\I(\hat{\mathbb{K}},\hat{\omega},\mathbf{G})$. Note that the rank of the ordered abelian group $\Z\times \Lambda$ is $\hat{S}:=\{0_{\hat{S}}\}\sqcup S$, where $0_{\hat{S}}$ is an element of $\hat{S}$ such that $s>0_{\hat{S}}$ for all $s\in S$. In particular, we have an isomorphism of ordered $\R$-algebras $\R\times \RF^S \cong \RF^{\hat{S}}$.  \\
\begin{remark}\label{r_univalent}
    The extension $\widetilde{\mathbb{K}}(\!(t)\!)/\mathbb{K}(\!(t)\!)$ is univalent in the sense of Assumption \ref{assumpfield}. Let us briefly explain why. To do so, let $\omega': \widetilde{\mathbb{K}}^\times \rightarrow \Lambda '$ be the unique valuation given by that assumption, and let 
    $\hat{\omega}':\widetilde{\mathbb{K}}(\!(t)\!)^\times\rightarrow \hat{\Lambda} '$ be a valuation satisfying the conditions of that assumption for the extension $\widetilde{\mathbb{K}}(\!(t)\!)/\mathbb{K}(\!(t)\!)$. By uniqueness of $\omega'$, we have an identification 
    $\hat{\omega}'(\widetilde{\mathbb{K}}^\times)=\Lambda'$ such that $\hat{\omega}'|_{\widetilde{\mathbb{K}}^\times}=\omega'$. Since $\hat{\omega}'$ extends $\hat{\omega}$, we deduce an identification $\hat{\omega}'(\widetilde{\mathbb{K}}^\times t^\mathbb{Z})=\mathbb{Z}\times \Lambda'$ such that $\hat{\omega}'(\widetilde{a} t^m)=(m,\omega'(\widetilde{a}))$ for $m \in \mathbb{Z}$ and 
    $\widetilde{a}\in \widetilde{\mathbb{K}}^\times $. Now, given an element $f \in \widetilde{\mathbb{K}}(\!(t)\!)^\times$, one can always write $f=\widetilde{a}t^mf_0$ with $m \in \mathbb{Z}$, $\widetilde{a}\in \widetilde{\mathbb{K}}^\times $ and $f_0 \in 1+t\widetilde{\mathbb{K}}[\![t]\!]$. By observing that $\hat{\omega}'(1)=0$, $\hat{\omega}'(t)=(1,0) \in \mathbb{Z}\times \Lambda'\subset \hat{\Lambda} '$ and $\hat{\omega}'(f_1)>(-1,0)$ for every $f_1 \in \widetilde{\mathbb{K}}[\![t]\!]$, we deduce that $\hat{\omega}'(f_0)=0$, so that:
    $$\hat{\omega}'(f)=(m,\omega'(\widetilde{a}))\in \mathbb{Z}\times \Lambda'= \hat{\Lambda} '.$$
\end{remark}

Let $x\in \mathcal{I}$ and consider $C$ and $\tilde{C}$ two sectors in $\mathcal{I}$ opposite at $x$. According to example~\ref{constantex}, if $\omega_t$ stands for the $t$-adic valuation on $\hat{\mathbb{K}}$, we have a projection:
$$\pi: \hat{\mathcal{I}} \rightarrow \mathcal{I}(\hat{\mathbb{K}},\omega_t,\mathbf{G})$$
such that the fiber of the point $X_1:=[(1,0)]$ is $\mathcal{I}$. Let $\hat{C}$ and $\hat{\tilde{C}}$ be sectors of $\hat{\mathcal{I}}$ such that $\hat{C} \cap \pi^{-1}(X_1) = C$ and $\hat{\tilde{C}} \cap \pi^{-1}(X_1) = \tilde{C}$. The sectors $\hat{C}$ and $\hat{\tilde{C}}$ being opposite  (by  Lemma~\ref{l_opposite_local_chambers} below applied with $\mathbb{F}=\hat{\Kb}$, $T=\hat{S}$ and $T_1=\{0_{\hat{S}}\}$), we deduce that there is an apartment $A$ of $\hat{\mathcal{I}}$ that contains both $\hat{C}$ and $\hat{\tilde{C}}$. Hence $A\cap \pi^{-1}(X_1)$ is an apartment of $\mathcal{I}=\pi^{-1}(X_1)$ that contains both $C$ and $\tilde{C}$.

\begin{lemma}\label{l_opposite_local_chambers}
 Let $\mathbb{F}$ be a field, equipped with a valuation $\omega:\mathbb{F}\rightarrow \RF^T\cup\{\infty\}$, where $T$ is a totally ordered set. We assume that $\mathbb{F}$ satisfies Assumption~\ref{assumpfield}. Assume that $T=T_1\sqcup T_0$, where $T_1,T_0$ are subsets of $T$ such that  $T_1<T_0$. We assume that $\omega(\mathbb{F}^\times)\cap \left(\{0_{T_1}\}\times \RF^{T_0}\right)\neq \{0\}$ and $\omega(\mathbb{F}^\times)\not \subset \{0_{T_1}\}\times \RF^{T_0}$. Let $\pi:\RF^{T}\twoheadrightarrow \RF^{T_1}$ be the canonical projection. Assume that $\mathbf{G}$ is defined over $\mathbb{F}$  Let $\I_T=\I(\mathbb{F},\omega,\mathbf{G})$ and $\I_{T_1}=\I(\mathbb{F},\omega_1,\mathbf{G})$, where $\omega_1=\pi\circ \omega$. Denote by $\pi$ the canonical projection $\I_{T}\twoheadrightarrow \I_{T_1}$. Let $X_1\in \I_{T_1}$ and $\I=\pi^{-1}(X_1)$. Let $x\in \I$ and $C, \tilde{C}$ be two sectors of $\I$, opposite at $x$. Let $C_T$ and $\tilde{C}_T$ be two sectors of $\I_{T}$ such that $C_T\cap \I=C$ and $\tilde{C}_T\cap \I=\tilde{C}$. Then $C_T$ and $\tilde{C}_T$ are opposite at $x$.
\end{lemma}

\begin{proof}
Denote by $\A_T=\A_\R \otimes \RF^{T}$ and $\A_{T_1}=\A_{\R}\otimes \RF^{T_1}$ the standard apartments of $\I_T$ and $\I_{T_1}$.
 Denote by $d_T:\I_T\times \I_T\rightarrow (\RF^{T})_{\geq 0}$ (resp. $d_{T_1}:\I_{T_1}\times \I_{T_1}\rightarrow (\RF^{T_1})_{\geq 0}$)  the distance on $\I_T$ (resp. $\I_{T_1}$) for which $\mathbf{G}(\mathbb{F})$ acts by isometry on $\I_T$ (resp. $\I_{T_1}$) and such that $d_T(a,b)=\sum_{\alpha\in \Phi^+}|\alpha(b-a)|$ (resp. $d_{T_1}(a,b)=\sum_{\alpha\in \Phi^+}|\alpha(b-a)|$) for $a,b\in \A_{T}$ (resp. $a,b\in \A_{T_1}$).  Then as $\pi$ is $\mathbf{G}(\mathbb{F})$-equivariant and as $\pi(d_T(a,b))=d_{T_1}(\pi(a),\pi(b))$ for $a,b\in \A_{T}$, we have: \[\pi(d_T(a,b))=d_{T_1}(\pi(a),\pi(b)), \forall a,b\in \I_T.\]

 Therefore $\I=\{y\in \I_T\mid d_T(x,y)\in \{0\}\times \RF^{T_0}\}$ is open in $\I_T$. 

 Consequently, \begin{equation}\label{e_equality_germs} \forall \Omega\in \germ_x(C_T), \Omega\cap \I\in \germ_x(C)\cap \germ_x(C_T),
 \end{equation} and the same holds for $\tilde{C}_T$.

 Let $A_T$ be an apartment of $\I_T$ containing $\germ_x(C_T)$ and $\germ_x(\tilde{C}_T)$, which exists by Lemma~\ref{lemDecompositions_axioms}. We identify $A_T$ and $\A_{T}$. Then $\A_T\cap  \I=\{X_1\}\times \A_{T_0}$. Let $Q$ and $\tilde{Q}$ be the sectors of $\A_{T_0}$, based at $x$ and such that $\germ_x(Q)=\germ_x(C)$ and $\germ_x(\tilde{Q})=\germ_x(\tilde{C})$. Let $C^v_{T_0}$ and $\tilde{C}^v_{T_0}$ be the vector chambers of $\A_{T_0}$ such that $Q=x+C^v_{T_0}$ and $\tilde{Q}=x+\tilde{C}^v_{T_0}$. We may assume that $C^v_{T_0}$ is the fundamental vector chamber $C^v_{f,T_0}=\{a\in \A_{T_0}\mid \alpha(a)>0, \forall \alpha\in \Phi^+\}$. As $Q$ and $\tilde{Q}$ are opposite at $x$, we have $\tilde{C}^v_{T_0}=w_0.C^v_{f,T_0}$. Let $w\in W^v$ be such that $\germ_x(x+w.C^v_{f,T})=\germ_x(C_T)$. Then by \eqref{e_equality_germs}, \[(x+w.C^v_{f,T})\cap \I=x+(\{0_{T_1}\}\times w.C^v_{f,T_0})\in \germ_x(C)\] and thus there exists a neighborhood $\Omega$ of $x$ in $\A_{T_0}$ such that \[x+\{0_{T_1}\}\times w.C^v_{f,T_0}\supset \Omega \cap \left(x+\{0_{T_1}\}\times C^v_{f,T_0}\right).\]  Therefore there exists $\epsilon\in \RF^{T_0}$ such that \[B(0_{T_0},\epsilon) \cap C^v_{f,T_0}\subset w.C^v_{f,T_0}.\]  Let $s=\min \left(\supp(\epsilon)\right)$. Then if $a=(a_t)_{t\in T_0}$  is such that $a_s\in  C^v_{f,\R}\subset \A_\R$ is close  enough to $0$ and $a_t=0$ for all $t\in T_0$ such that $t<s$, we have: \[a\in B(0_{T_0},\epsilon)\cap C^v_{f,T_0}\cap w.C^v_{f,T_0}.\]  Let $\alpha\in \Phi^+$. Then $\alpha(a)>0$ and $\alpha(w^{-1}.a)>0$ and thus $w.\alpha\in \Phi^+$. Therefore $w.\Phi^+\subset \Phi^+$. Using \cite[Chap VI, 1.6 Corollaire 2]{Bourbaki-Lie456}, we deduce that $w=1$ and $\germ_x(x+C^v_{f,T})=\germ_x(C_T)$. Similarly, $\germ_x(x+w_0.C^v_{f,T})=\germ_{x}(\tilde{C}_T)$, which proves that $C_T$ and $\tilde{C}_T$ are opposite at $x$.
\end{proof}

\subsection{Proof of~\ref{axiomCO} for the fibers of the projection maps}\label{subsecCO_fibers}

Adopt the notations of section~\ref{projectionsec}. In particular, we have a projection map:
$$\pi: \mathcal{I}(\mathbb{K},\omega,\mathbf{G}) \rightarrow \mathcal{I}(\mathbb{K},\omega_1,\mathbf{G}).$$
Fix a point $X_1 \in \mathcal{I}(\mathbb{K},\omega_1,\mathbf{G})$, let $x\in \pi^{-1}(X_1) \subseteq \mathcal{I}(\mathbb{K},\omega,\mathbf{G})$ and consider $C$ and $\tilde{C}$ two sectors of $\pi^{-1}(X_1)$ opposite at $x$. Let $\hat{C}$ and $\hat{\tilde{C}}$ be sectors of $\mathcal{I}(\mathbb{K},\omega,\mathbf{G})$ such that $\hat{C} \cap \pi^{-1}(X_1) = C$ and $\hat{\tilde{C}} \cap \pi^{-1}(X_1) = \tilde{C}$. The sectors $\hat{C}$ and $\hat{\tilde{C}}$ being opposite (by Lemma~\ref{l_opposite_local_chambers} applied with $T=\mathrm{rk}(\Lambda)$, $T_1=\mathrm{rk}(\Lambda_1)$ and $T_0=\mathrm{rk}(\Lambda_0)$), we deduce that there is an apartment $A$ of $\mathcal{I}(\mathbb{K},\omega,\mathbf{G})$ that contains both $\hat{C}$ and $\hat{\tilde{C}}$. Hence $A\cap \pi^{-1}(X_1)$ is an apartment of $\pi^{-1}(X_1)$ that contains both $C$ and $\tilde{C}$.

\section{Replacing \texorpdfstring{$\mathfrak{R}^S$}{RS} by smaller ordered abelian groups}\label{section_replacingRS}

In \cite{bennett1994affine}, Bennett defines $\Gamma$-buildings for a totally ordered abelian group $\Gamma$ other than $\mathfrak{R}^S$, for instance for $\mathbb{Z}\times \mathbb{Z}$. The definition is similar to the one we gave in section~\ref{subBennett_definition}. One has to replace axiom~\ref{axiomA1} by the following:
\begin{enumerate}[label={(A1')}]
\item\label{axiomA1'} Each $A \in \ACC$ is equipped with the structure of an apartment over $\Gamma$ of type $\underline{\mathbb{A}_{\Gamma}} = (\mathbb{A}_{\Gamma},V_{\mathbb{R}},\Phi,(\Gamma_{\alpha})_{\alpha \in \Phi})$,\axiom{A1'@\ref{axiomA1'}}
\end{enumerate}
and then one keeps word by word the axioms \ref{axiomA2}-\ref{axiomA6}. The main goal of this section consists in proving the following statement:

\begin{proposition}\label{propLambda_building}
Let $\mathbf{G}$ be a quasi-split reductive group over a $\Lambda$-valued field $\mathbb{K}$ and let $\mathbf{S}$ be a maximal split torus in $\mathbf{G}$. We keep the notation $\Gamma_\alpha$ for $\alpha\in\Phi$ as defined in section \ref{subsection_Rvaluation}. Let $S$ be the rank of $\Lambda$ and let $\mathcal{I}$ be the $\mathfrak{R}^S$-building associated with $\mathbf{G}$. Let $\Gamma$ be a totally ordered subgroup of $\mathfrak{R}^S$ containing $\Gamma_\alpha$ for each $\alpha \in \Phi$. Then the $\mathfrak{R}^S$-building structure of $\mathcal{I}$ induces a $\Gamma$-building structure on the set:
$$G \cdot p\left( Y \otimes \Gamma\right), $$
where $Y \otimes \Gamma$ is seen inside $X_*(\mathbf{S}) \otimes \mathfrak{R}^S$ and $Y$ is the coweight lattice \[Y = \left\{ x \in X_*(\mathbf{S}) \otimes \mathbb{R},\ \forall \alpha \in \Phi,\ \alpha(x) \in \mathbb{Z} \right\},\] and $p$ is the projection from $X_*(\mathbf{S})\otimes \mathfrak{R}^S$ to the standard apartment $\mathbb{A}$.

\end{proposition}

Note that if $\mathbf{G}$ is semi-simple and adjoint, we have that $Y = X_*(\mathbf{S})$.
In particular, up to considering the adjoint quotient $\overline{\mathbf{G}}$ of the reductive group $\mathbf{G}$, one can always work with $X_*(\overline{\mathbf{S}}) \otimes \Gamma$ where $\overline{\mathbf{S}}$ denotes the image of $\mathbf{S}$ in $\overline{\mathbf{G}}$.

We say that two half-apartments $D_1,D_2$ of $\A$ have opposite directions  in $\A$ (resp. have the same direction) if there exists $\alpha\in \Phi$ and $\lambda,\lambda'\in \RF^S$ such that $D_1=D_{\alpha,\lambda}$ and $D_2=D_{-\alpha,\lambda'}$ (resp. $D_{\alpha,\lambda'}$). We extend these notions to any apartment using isomorphisms of apartments.

\begin{proof}(of Proposition~\ref{propLambda_building})
Set $\A_\Gamma = p\left( Y \otimes \Gamma\right) $ 
and  $\I_\Gamma=G . \A_\Gamma$.  We define the set of apartments of $\I_\Gamma$ to be the set of sets of the form $A\cap \I_\Gamma$ such that $A$ is an apartment of $\I$.

We first prove that $\A\cap \I_\Gamma=\A_\Gamma$. 
Let $y\in \I_\Gamma\cap \A$. There exists $g\in G$ 
and $x\in \A_\Gamma$ such that $y=g.x$. 
Let $h\in G$ be such that $hg.\A=\A$ and such that $h$ fixes $y$, which exists by (A2) (for $\I$). 
Then by Proposition~\ref{propAxiom(A2)}, there exists $n\in N$ such that $hg.z=n.z$ for all $z\in \A$. Thus $w = \nu(n) \in \Wext$ satisfies $y= h.y = hg.x = w(x)$
and thus $y\in \Wext.\A_\Gamma$.

According to Proposition~\ref{PropSemiSpecialValuation}, the valuation is semi-special so that the group $\widetilde{\Gamma}_\alpha$ is generated by $\Gamma_\alpha$ for any $\alpha \in \Phi$.
Hence, $\Gamma \supset \widetilde{\Gamma}_\alpha$ for any $\alpha \in \Phi$.
We have that $\mathbb{A}_\Gamma = \bigoplus_{\alpha \in \Delta} \Gamma \varpi_\alpha^\vee \supset \bigoplus_{\alpha \in \Delta} \widetilde{\Gamma}_\alpha \varpi_\alpha^\vee \supset \Wext \cap V_{\mathfrak{R}^S}$.
Since $W^v$ stabilizes the coweight lattice $Y = \bigoplus_{\alpha \in \Delta} \Z \varpi_\alpha^\vee$, it stabilizes $Y \otimes \Gamma = \mathbb{A}_\Gamma$.
Thus $\Wext$ stabilizes $\mathbb{A}_\Gamma$.

 Consequently \begin{equation}\label{eqIntersection_appartment_A_gamma}
\A_\Gamma=\A\cap \I_\Gamma. \end{equation} Let now $A$ be an apartment of $\I$. Write $A=g.\A$, $g\in G$. Then by \eqref{eqIntersection_appartment_A_gamma}, $A\cap \I_\Gamma=g.\A_\Gamma$. This enables us to equip each apartment with the structure of an apartment of type $\A_\Gamma$, which proves \ref{axiomA1'}. We also deduce that \ref{axiomA2}, \ref{axiomA3}, \ref{axiomA4} and \ref{axiomA5} are satisfied by $\I_\Gamma$, using the fact that they are satisfied by $\I$.

Let us prove that $\I_\Gamma$ satisfies \ref{axiomA6}. Let $A_1^\Gamma, A_2^\Gamma$ and $A_3^\Gamma$ be three apartments of $\I_\Gamma$ such that $A_1^\Gamma\cap A_2^\Gamma$, $A_1^\Gamma\cap A_3^\Gamma$ and $A_2^\Gamma\cap A_3^\Gamma$ are half-apartments.
Let $A_1,A_2$ and $A_3$ be the corresponding apartments of $\I$. Using isomorphisms of apartments, we may assume that $A_1=\A$.
Let $D_3=\A\cap A_2$ and  $D_2=\A\cap A_3$. By definition of half-apartments, there are $\alpha_2,\alpha_3\in \Phi$ and $\lambda_2,\lambda_3\in \Gamma$ such that $A_1^\Gamma \cap A_2^\Gamma = D_{\alpha_3,\lambda_3}^\Gamma \subset \mathbb{A}_\Gamma$ and $A_1^\Gamma \cap A_3^\Gamma = D_{\alpha_2,\lambda_2}^\Gamma \subset \mathbb{A}_\Gamma$.
Since $D_2$ and $D_3$ are enclosed by \ref{axiomA2}, we deduce that $D_2 = D_{\alpha_2,\lambda_2}$ and $D_3 = D_{\alpha_3,\lambda_3}$.
Suppose that $D_2$ and $D_3$ are not parallel (that is $\R\alpha_2\neq \R\alpha_3$).
Let $\Delta$ be a basis such that $C^v_\Delta \subset D_{\alpha_2} \cap D_{\alpha_3}$ (i.e. a basis such that $\alpha_2$ and $\alpha_3$ both are positive roots).
Let $x_\Z$ be any element of $C^v_{\Delta,\Z}$.
Then $\alpha_2(x_\Z), \alpha_3(x_\Z) \in \mathbb{N}$.
Consider any $\gamma > \max( |\lambda_2|,|\lambda_3| ) \in \Gamma_{> 0}$.
Then $\alpha_2( \gamma x_\Z ) > \alpha_2(x_\Z) |\lambda_2| \geqslant \lambda_2$ and $\alpha_3( \gamma x_\Z ) > \alpha_3(x_\Z) |\lambda_3| \geqslant \lambda_3$.
Hence $\gamma x_\Z \in D_2 \cap D_3 = A_1 \cap A_2 \cap A_3$.

If $D_2$ and $D_3$ have the same direction, it is  clear that $D_2\cap D_3$ is non empty. We now assume that for all $i\in \llbracket 1,3\rrbracket$, $A_i\cap A_j$ and $A_i\cap A_k$ have opposite directions in  $A_i$, where $\{j,k\}=\llbracket 1,3\rrbracket\setminus \{i\}$.  Let $M_3$ be the wall of $D_3$. By Lemma~\ref{lemExchange_condition}, $(\A\cup A_2)\setminus D_3\cup M_3$ is an apartment $A_3'$ of $\I$. Moreover, $A_3\cap A_3'\cap \A$ contains a half-apartment parallel to $D_2$ and $A_3\cap A_3'\cap A_2$ contains a half-apartment parallel to $A_2\cap A_3$. Therefore $A_3\cap A_3'$ contains two half-apartments which have opposite directions in  $A_3$. By \ref{axiomA2} we deduce that $A_3=A_3'$. In particular, $M_3\subset \A\cap A_2 \cap A_3$. Write $M_3=H_{\alpha, \lambda}$, for $\alpha\in \Phi$ and $\lambda\in \RF^S$. Then by the proof of Lemma~\ref{lemExchange_condition}, $\lambda\in \Gamma_\alpha\subset \Gamma$, which proves that $A_1^\Gamma\cap A_2^\Gamma\cap A_3^\Gamma\supset M_3\cap \A_\Gamma\neq \emptyset$. This proves that $\I_\Gamma$ satisfies \ref{axiomA6} and completes the proof of the proposition.
\end{proof}

As a particular case, we get the following corollary:

\begin{corollary}
Keep the previous notation and assume that $\mathbf{G}$ splits over $\mathbb{K}$. Then the $\mathfrak{R}^S$-building structure of $\mathcal{I}(\mathbb{K},\omega,\mathbf{G})$ induces a $\Lambda$-building structure on the set:
$$G \cdot p\left( Y \otimes \Lambda \right) ,$$
where $Y \otimes \Lambda$ is seen inside $X_*(\mathbf{S}) \otimes \mathfrak{R}^S$ and $p$ is the projection from $X_*(\mathbf{S})\otimes \mathfrak{R}^S$ to the standard apartment $\mathbb{A}$.
\end{corollary}

\begin{proof}
For a split reductive group $\mathbf{G}$, since $\Gamma_\alpha = \Lambda$ for any $\alpha \in \Phi$ according to Fact~\ref{FactGamma_split}, one can take $\Gamma = \Lambda$ in Proposition~\ref{propLambda_building}.
\end{proof}

\section{The building of the group \texorpdfstring{$\mathrm{SL}_{\ell+1}$}{SL(l+1)}}\label{secBuilding_sld}

Let $\mathbb{K}$ be  a field equipped with a valuation $\omega: \mathbb{K}\rightarrow \Lambda\subset \RF^\cup \{\infty\}$. Let $\ell\in \Z_{\geq 1}$. In this section, we detail the construction of the building $\I$ associated to $(\mathrm{SL}_{\ell+1}(\Kb),\omega)$ by Definition~\ref{DefLambdaBuildingFromDatum}. We then define an embedding of the building $\I^{\LC}=\I^{\LC}(\mathrm{SL}_{\ell+1}(\Kb),\omega)$ constructed by Bennett using lattices in \cite[Example 3.2]{bennett1994affine} in  $\I$. Using this embedding, we then revisit - in the particular case of $\mathrm{SL}_2$ and of $\Lambda\subset \Z^2$ - the  construction of the boundary of $\I$ developped by Parshin in \cite{parshin1994higher}.

\subsection{The parahoric and the lattice buildings of \texorpdfstring{$\mathrm{SL}_{\ell+1}$}{SL(l+1)}}

Let $\ell\in \Z_{\geq 1}$ and  $\mathbb{K}$ be a field equipped with a non trivial valuation $\omega:\mathbb{K}\rightarrow \Lambda$, where $\Lambda$ is a subgroup of $\RF^S$, for some totally ordered set $S$. Let $\mathbf{G}=\mathbf{SL}_{\ell+1}$ and $G=\mathbf{G}(\Kb)$. We explicitly describe a root generating datum for $G$ and we then recall Bennett's lattice construction of the building of $G$.

\paragraph{The  parahoric building $\I$} Let $\mathbf{T}_{\mathrm{SL}_{\ell+1}}$ (resp. $\mathbf{T}_{\mathrm{GL}_{\ell+1}}$) be the functor which associates to each field $\mathcal{F}$ the subgroup of $\mathrm{SL}_{\ell+1}(\mathcal{F})$ (resp. $\mathrm{GL}_{\ell+1}(\mathcal{F})$)  composed with diagonal matrices. Let  $\TS=\mathbf{T}_{\mathrm{SL}_{\ell+1}}(\mathbb{K})$ and $\TG=\mathbf{T}_{\mathrm{GL}_{\ell+1}}(\mathbb{K})$.

For $i,j\in \llbracket 1,\ell+1\rrbracket$, we denote by $E_{i,j}$ the matrix $(\delta_{i,k}\delta_{j,k'})_{k,k'\in \llbracket 1,\ell+1\rrbracket}$. For $t=(t_1,\ldots,t_{\ell+1})\in (\mathbb{K}^*)^{\ell+1}$, we set $D(t)=\sum_{i=1}^{\ell+1} t_i E_{i,i}$.  For $i,j\in \llbracket 1,\ell+1\rrbracket$ such that $i\neq j$, we define $\alpha_{i,j}\in X^*(\TS)$  by $\alpha_{i,j}\big(D(t_1,\ldots,t_{\ell+1})\big)=t_it_j^{-1}$, for $D(t_1,\ldots,t_{\ell+1})\in \TS$ and we define $\alpha_{i,j}^\vee\in X_*(\TS)$ by  $\alpha_{i,j}^\vee(u)=\sum_{k\in \llbracket 1,\ell+1\rrbracket \setminus \{i,j\}} E_{k,k}+uE_{i,i}+u^{-1}E_{j,j}\in \TS$, for $u\in \mathbb{K}^*$. The root system $\Phi=\Phi(\mathbf{SL}_{\ell+1},\mathbf{T}_{\mathrm{SL}_{\ell+1}})$ is $\{\alpha_{i,j}\mid i,j\in \llbracket 1,\ell+1\rrbracket|i\neq j\}$ and $\Delta=\{\alpha_{i,i+1}\mid i\in \llbracket 1,\ell\rrbracket\}$ is a basis of $\Phi$. For $i\in \llbracket 1,\ell+1\rrbracket$, define $\chi_i\in X^*(T_{\mathrm{SL}_{\ell+1}})$ by $\chi_i\left(D(t_1,\ldots,t_{\ell+1})\right)=t_i$ for $(t_1,\ldots,t_{\ell+1})\in (\mathbb{K}^*)^{\ell+1}$. One has $X_*(\TS)=\bigoplus_{i=1}^{\ell} \Z \alpha_{i,i+1}^\vee$ and  $X^*(\TS)=\bigoplus_{i=1}^{\ell}\Z\chi_i\supset \bigoplus_{i=1}^{\ell} \Z\alpha_{i,i+1}$. Let $\A_\R=X_*(\TS)\otimes_{\Z} \R$ and $\A_S=\A_\R\otimes_\R \RF^S$.  For $i\in \llbracket 1, \ell+1\rrbracket$, we denote by $\varpi_{i,i+1}^\vee$ the \textbf{fundamental weight} associated with $\alpha_{i,i+1}$, that is the unique element $\varpi_{i,i+1}^\vee$ of $\A_\R$ such that $\alpha_{j,j+1}(\varpi_{i,i+1}^\vee)=\delta_{i,j}$, for $j\in \llbracket 1,\ell\rrbracket$. 

For $\alpha=\alpha_{i,j}\in \Phi$, one defines $x_\alpha:\mathbb{K}^*\rightarrow G$ by $x_\alpha(u)=1+uE_{i,j}$ if $i< j$ and $x_\alpha(u)=1-uE_{i,j}$ is $i>j$. One sets $U_\alpha=x_\alpha(\mathbb{K})$ and one defines $\varphi_\alpha:U_\alpha\rightarrow \Lambda$ by $\varphi_\alpha\big(x_\alpha(u)\big)=\omega(u)$. One defines $m_\alpha:\mathbb{K}^*\rightarrow G$ by $m_\alpha(u)=x_\alpha(u)x_{-\alpha}(u^{-1})x_\alpha(u)$. One has $m_\alpha(u)=1-E_{i,i}-E_{j,j}+uE_{i,j}-u^{-1}E_{j,i}$ if $i<j$ and $m_\alpha(u)=1-E_{i,i}-E_{j,j}-uE_{i,j}+u^{-1}E_{j,i}$ if $i>j$. One sets $M_\alpha=m_\alpha(1)\TS$. 
Let $\mathfrak{S}_{\ell+1}$ be the set of permutations of $\llbracket 1,\ell+1\rrbracket$. For $\sigma\in \mathfrak{S}_{\ell+1}$, one chooses $P_\sigma=(P_{\sigma,i,j})\in G$ such that $P_{\sigma,i,j}\neq 0$ if and only if $j=\sigma(i)$ and such that $\omega(P_{\sigma,i,\sigma(i)})=0$, for $i,j\in \llbracket 1,\ell+1\rrbracket$ .  Let $\NS=\langle M_\alpha\mid\alpha\in \Phi\rangle\subset G$. Then $\NS=\bigcup_{\sigma\in\mathfrak{S}_{\ell+1}} P_\sigma \TS$ and $\NS$ is the normalizer of $\TS$ in $G$. Then $\big(\TS,(U_\alpha,M_{\alpha})_{\alpha\in\Phi}\big)$ is a root group datum and $(\varphi_{\alpha})_{\alpha\in \Phi}$ is a valuation of this datum in the sense of Definitions~\ref{DefRGD} and \ref{DefValuation}. We denote by $\I$ the building associated with this datum in Definition~\ref{DefLambdaBuildingFromDatum}.

Using the notation of subsection~\ref{subsecAction_N_affine_space}, one has $\rho(t)(\alpha_{i,j}\otimes \lambda)=\lambda\big(\omega(t_j)-\omega(t_i)\big)$ for $D(t_1,\ldots,t_{\ell+1})\in \TS$ and $\lambda\in \RF^S$. Thus $\rho(t)=\sum_{i=1}^{\ell} (\omega(t_{i+1})-\omega(t_i))\varpi_{i,i+1}^\vee\in \A_S$. We extend $\rho$ to a map $\rho:\TG\rightarrow \A_S$ by setting $\rho(t)=\sum_{i=1}^{\ell} (\omega(t_{i+1})-\omega(t_i))\varpi_{i,i+1}^\vee$, for $t\in \TG$. Then $\TG$ acts on $\A_S$ by $t.x=x+\rho(t)$ for $x\in \A_S$ and $t\in \TG$. 

\paragraph{The lattice building $\I^\LC$}

Let $\mathbb{O}:=\omega^{-1}(\Lambda_{\geq 0})$ be the ring of integers of $\mathbb{K}$.

 An \textbf{$\mathbb{O}$-lattice} in $\mathbb{K}^{\ell+1}$ is an $\mathbb{O}$-submodule of $\Kb^{\ell+1}$ of the form $\mathbb{O} b_1 \oplus ... \oplus \mathbb{O} b_{\ell+1}$ for some $\Kb$-basis $(b_1,...,b_{\ell+1})$ of $\mathbb{K}^{\ell+1}$. If $L_1$ and $L_2$ are two $\mathbb{O}$-lattices of $\mathbb{K}^{\ell+1}$, we say that they are \textbf{homothetic} if there exists $a\in \mathbb{K}^{\times}$ such that $L_2=aL_1$. In that case, we denote $L_1 \sim L_2$. Let $\I^{\LC}=\I^{\LC}(\mathbb{K},\omega)$ be the set of $\mathbb{O}$-lattices of $\mathbb{K}^{\ell+1}$ modulo the homothety relation. We say that $\I^{\LC}$ is the \textbf{lattice building} of $(G,\omega)$. The image of an $\mathbb{O}$-lattice $L$ in $\I^{\LC}(\mathbb{K},\omega)$ is denoted $[L]$.   If $t\in \TG$ and $[L]\in \A^{\LC}$, then $t.[L]=[t.L]\in \A^{\LC}$. If $(t_i),(x_i)\in (\mathbb{K}^*)^{\ell+1}$, then $D(t_1,\ldots,t_{\ell+1}).[\bigoplus_{i=1}^{\ell+1} \mathbb{O} x_ie_i]=[\bigoplus_{i=1}^{\ell+1}  \mathbb{O} t_i x_i e_i]=[\bigoplus_{i=1}^{\ell+1} \mathbb{O} t_1^{-1}t_i x_i e_i]$. 
 
 Choose a basis $(e_1,\ldots,e_{\ell+1})$ of $\mathbb{K}^{\ell+1}$. Let $\A^{\LC}=\{[\bigoplus_{i=1}^{\ell+1} \mathbb{O} x_ie_i]|(x_i)\in (\mathbb{K}^*)^{\ell+1}\}$ be the \textbf{standard lattice apartment}.  An apartment of $\I^\LC$ is a set of the form $g.\A^\LC$, for some $g\in G$. Set $\mathbb{O}^{\ell+1}=\bigoplus_{i=1}^{\ell+1} \mathbb{O} e_i$. For each $\lambda\in \Lambda$, choose $x_\lambda\in\Kb$ such that $\omega(x_\lambda)=\lambda$. The map $\iota:\Lambda^{\ell+1}/\Lambda (1,\ldots,1)\rightarrow \A^{\LC}$ defined by 
 $\iota(\overline{(\lambda_1,\ldots,\lambda_{\ell+1})})=[\mathbb{O} x_{\lambda_1}e_1\oplus \ldots \oplus \mathbb{O} x_{\lambda_{\ell+1}} e_{\ell+1}]$ for $\overline{(\lambda_1,\ldots,\lambda_{\ell+1})}\in \Lambda^{\ell+1}/\Lambda(1,\ldots,1)$ is a bijection which enables us to identify $\Lambda^{\ell+1}/\Lambda(1,\ldots,1)$ and $\A^\LC$ when convenient. In particular this equips $\A^\LC$ with the structure of a $\Z$-module. For $i,j\in \llbracket 1,\ell+1\rrbracket$ such that $i\neq j$, we define $\alpha_{i,j}^{\LC}:\A^{\LC} \rightarrow \Lambda$ by  $\alpha_{i,j}^{\LC}\big(\overline{(\lambda_1,\ldots,\lambda_{\ell+1})}\big)=\lambda_i-\lambda_j$, for $\overline{(\lambda_1,\ldots,\lambda_{\ell+1}})\in \Lambda^{\ell+1}/\Lambda(1,\ldots,1)$. For $[L],[L']\in \A^{\LC}$, we set $d^{\LC}([L],[L'])=\sum_{i,j\in \llbracket 1,\ell+1\rrbracket |i<j} |\alpha_{i,j}^{\LC}([L]-[L'])|\in \Lambda_{\geq 0}$. By \cite[Example 3.2 \& Remark 3.2]{bennett1994affine}, $d^{\LC}$ extends uniquely to a $G$-invariant distance $d^{\LC}:\I^\LC \rightarrow \Lambda$. 
 Note that although Bennett assumes that $\Lambda$ is a $\Q$-module in \cite{bennett1994affine}, it is useless for \cite[Example 3.2]{bennett1994affine}.

\subsection{Embedding of the lattice building in the parahoric building}\label{ss_Embedding_lattice_building}
In this subsection, we define a $G$-equivariant embedding $\psi:\I^{\LC}\rightarrow \I$ and  describe its image.

If $[L]=[\bigoplus_{i=1}^{\ell+1} \mathbb{O} x_i e_i]\in \A^{\LC}$, we set $\psi([L])=\sum_{i=1}^{\ell} \big(\omega(x_{i+1})-\omega(x_{i})\big)\varpi_{i,i+1}^\vee\in \A_S$.

\begin{lemma}\label{lemN_equivariance_psi}
The map $\psi:\A^{\LC}\rightarrow \A_S$ is $\NGL$-equivariant.
\end{lemma}

\begin{proof}
The fact that $\psi$ is $\TG$-equivariant is straightforward. Let $i,j\in \llbracket 1,\ell\rrbracket$. Then: \[\alpha_{i,i+1}(\alpha_{j,j+1}^\vee)=\left\{\begin{aligned}& 2\text{ if }i=j\\
& -1\text{ if }|j-i|=1\\
& 0 \text{ otherwise }\end{aligned}\right..\] Let $i\in \llbracket 1,\ell\rrbracket$. Let $[L]=[\bigoplus_{i=1}^{\ell+1} \mathbb{O} x_i e_i]\in \A^{\LC}$, where $(x_i)\in (\mathbb{K}^*)^{\ell+1}$. Suppose that $i\in \llbracket 2,\ell-2\rrbracket$.  Then: \[\psi([L])=\ldots+\big(\omega(x_i)-\omega(x_{i-1})\big)\varpi_{i-1,i}^\vee+\big(\omega(x_{i+1})-\omega(x_{i})\big)\varpi_{i,i+1}^\vee+\big(\omega(x_{i+2})-\omega(x_{i+1})\big)\varpi_{i+1,i+2}^\vee+\ldots\] and $m_{\alpha_{i,i+1}}(1).[L]=[\bigoplus_{j\in \llbracket 1,\ell\rrbracket \setminus\{i,i+1\}}\mathbb{O} x_je_j\oplus \mathbb{O} x_{i+1}e_i\oplus \mathbb{O} x_i e_{i+1}]$. Therefore: \[\begin{aligned} &\psi\big(m_{\alpha_{i,i+1}}(1).[L]\big)\\ = & \ldots +\big(\omega(x_{i+1})-\omega(x_{i-1})\big)\varpi_{i-1,i}^\vee +\big(\omega(x_i)-\omega(x_{i+1})\big)\varpi_{i,i+1}^\vee+\big(\omega(x_{i+2})-\omega(x_i)\big)\varpi_{i+1,i+2}^\vee+\ldots\\ = & \psi([L])+\big(\omega(x_{i+1})-\omega(x_i)\big)\varpi_{i-1,i}^\vee+2\big(\omega(x_{i})-\omega(x_{i+1})\big)\varpi_{i,i+1}^\vee+\big(\omega(x_{i+1})-\omega(x_{i})\big)\varpi_{i+1,i+2}^\vee+\ldots\\ = & \psi([L])+ \big(\omega(x_{i})-\omega(x_{i+1})\big)(-\varpi_{i-1,i}^\vee+2\varpi_{i,i+1}^\vee-\varpi_{i+1,i+2}^\vee ).\end{aligned}\]

 On the other side, one has: \[\begin{aligned}  r_{\alpha_{i,i+1}}\big(\psi([L])\big) = & \psi([L])-\alpha_{i,i+1}\big(\psi([L])\big)\alpha_{i,i+1}^\vee\\ = & \psi([L])+\big(\omega(x_i)-\omega(x_{i+1})\big)\alpha_{i,i+1}^\vee\\
 = & \psi([L])+\big(\omega(x_i)-\omega(x_{i+1})\big)\sum_{j=1}^{\ell}\alpha_{j,j+1}(\alpha_{i,i+1}^\vee)\varpi_{j,j+1}^\vee\\ 
 = & \psi([L])+ \big(\omega(x_{i})-\omega(x_{i+1})\big)(-\varpi_{i-1,i}^\vee+2\varpi_{i,i+1}^\vee-\varpi_{i+1,i+2}^\vee )=\psi\big(m_{\alpha_{i,i+1}}(1).[L]\big) .\end{aligned}\] Similar computations for $i=1$ and $i=\ell$ prove that $r_{\alpha_{i,i+1}}\big(\psi([L])\big) =\psi\big(m_{\alpha_{i,i+1}}(1).[L]\big)$ for all $i\in \llbracket 1,\ell\rrbracket$.
 
By Lemma~\ref{LemMuNu},  the image of $m_{\alpha_{i,i+1}}(1)$ in $W^v$ is $r_{\alpha_{i,i+1}}$.  We deduce that $\psi$ is $W^v$-equivariant and thus  $\NGL$-equivariant.

\end{proof}

Recall the definition of $\widehat{P}_0$ from Notation~\ref{notParahoric_subgroup}.

\begin{lemma}\label{lemComputation_P_0}
\begin{enumerate}
\item The fixator of $[\mathbb{O} ^{\ell+1}]$ in $G$ is $\mathrm{SL}_{\ell+1}(\mathbb{O} )$.

\item One has $\widehat{P}_{0}=\mathrm{SL}_{\ell+1}(\mathbb{O})$. In particular, the fixator of $0=\psi([\mathbb{O} ^{\ell+1}])$ in $\I$ is $\mathrm{SL}_{\ell+1}(\mathbb{O})$.
\end{enumerate}

\end{lemma}

\begin{proof}
(1) Let $g\in \mathrm{SL}_{\ell+1}(\mathbb{O})$. Then $g.\mathbb{O}^{\ell+1}\subset \mathbb{O}^{\ell+1}$ and $g^{-1}.\mathbb{O}^{\ell+1}\subset \mathbb{O}^{\ell+1}$. Thus $g.\mathbb{O}^{\ell+1}=\mathbb{O}^{\ell+1}$ and $\mathrm{SL}_{\ell+1}(\mathbb{O})$ fixes $[\mathbb{O}^{\ell+1}]$. Reciprocally,  let $g\in G$ be such that $g.[\mathbb{O}^{\ell+1}]=[\mathbb{O}^{\ell+1}]$. Then there exists $\lambda\in \Kb^*$ such that $g.\mathbb{O}^{\ell+1}=\lambda \mathbb{O}^{\ell+1}$. Then $g\in \mathrm{SL}_{\ell+1}(\lambda \mathbb{O})$. Therefore $1=\det(g)\in \lambda^{\ell+1} \mathbb{O}$ and hence $\omega(\lambda^{\ell+1})=(\ell+1)\omega(\lambda)\leq 0$. As $g^{-1}$ fixes $[\mathbb{O}^{\ell+1}]$, we also deduce that $\omega(\lambda)\geq 0$. Therefore $\omega(\lambda)=0$ and thus $g\in \mathrm{SL}_{\ell+1}(\mathbb{O})$, which proves (1). 

(2) By definition, $\widehat{P}_0=\langle \widehat{N}_0,U_0\rangle$. By definition and by Lemma~\ref{lemN_equivariance_psi}, one has, \[ \begin{aligned}\widehat{N}_0=\{n\in \NS |n.0=0=n.\psi([\mathbb{O} ^{\ell+1}])=\psi(n.[\mathbb{O} ^{\ell+1}])&=\psi([\mathbb{O} ^{\ell+1}])\}\\
& =\{n\in \NS |n.[\mathbb{O} ^{\ell+1}] = [\mathbb{O} ^{\ell+1}]\} \\ &= \NS\cap \mathrm{SL}_{\ell+1}(\mathbb{O} )\end{aligned}\] by (1). Moreover, $U_{0}=\langle U_{\alpha,0}|\alpha\in \Phi\rangle\subset \mathrm{SL}_{\ell+1}(\mathbb{O})$ by definition. Consequently, $\widehat{P}_0\subset \mathrm{SL}_{\ell+1}(\mathbb{O})$. Reciprocally, let $g\in \mathrm{SL}_{\ell+1}(\mathbb{O})$. Using Gaussian elimination, we can write $g=u_- t u_+$, where $u_-\in \mathrm{SL}_{\ell+1}(\mathbb{O})$ is lower triangular, $t \in \TS\cap \mathrm{SL}_{\ell+1}(\mathbb{O})$ and $u_+\in \mathrm{SL}_{\ell+1}(\mathbb{O})$ is upper triangular. Then $u_-\in U_0$, $t\in \widehat{N}_0$ and $u_+\in U_0$, which proves that $g\in \widehat{P}_0$. Therefore $\widehat{P}_0=\mathrm{SL}_{\ell+1}(\mathbb{O})$. 
\end{proof}

\begin{Lem}\label{lemComputation_Pt0}
Let $t\in \TG$. Then the fixator of $t.[\mathbb{O}^{\ell+1}]$ in $G$ and the fixator of $t.0$ in $G$ are equal to $t\mathrm{SL}_{\ell+1}(\mathbb{O})t^{-1}$.
\end{Lem}

\begin{proof}
By Lemma~\ref{lemComputation_P_0}, the fixator of $t.[\mathbb{O}^{\ell+1}]$ in $G$ is $t\mathrm{SL}_{\ell+1}(\mathbb{O})t^{-1}$. By definition, the fixator of $t.0$ in $G$ is $\widehat{P}_{t.0}=\langle \widehat{N}_{t.0},U_{t.0}\rangle$. One has $\widehat{N}_{t.0}=t \widehat{N}_0 t^{-1}$ because $\widehat{N}_{t.0}$ is the fixator of $t.0$ in $\NS$. Let $\alpha\in \Phi$. Write $\alpha=\alpha_{i,j}$, with $i,j\in \llbracket 1,\ell+1\rrbracket $ such that $i\neq j$. Write $t=D(t_1,\ldots,t_{\ell+1})$, with $t_1,\ldots,t_{\ell+1}\in \Kb^{\ell+1}$. Let $u\in U_\alpha$.  Then $\varphi_\alpha(tut^{-1})=\varphi_\alpha(u)+\omega(t_i)-\omega(t_j)$.  Thus $tU_{\alpha,0}t^{-1}=U_{\alpha,\omega(t_i)-\omega(t_j)}=U_{\alpha,-\alpha(t.0)}=U_{\alpha,t.0}$. Thus $\widehat{P}_{t.0}=\langle t \widehat{N}_0 t^{-1}, tU_{\alpha,0}t^{-1}|\alpha\in \Phi\rangle=t\widehat{P}_0 t^{-1}$ and we conclude with Lemma~\ref{lemComputation_P_0}.
\end{proof}

Let $\A_\Lambda=\{x\in \A\mid \alpha(x)\in \Lambda,\ \forall\alpha\in \Phi\}$. 

\begin{proposition}\label{propComparison_constructions_SL}
\begin{enumerate}
\item The map $\psi:\I^{\LC}\rightarrow \I$ defined by $\psi(g.[L])=g.\psi([L])$ for $g\in \mathrm{SL}_{\ell+1}(\Kb)$  and $[L]\in \A^{\LC}$ is well defined and is an isometry. 

\item One has $\psi^{-1}(\A)=\A^{\LC}$. 

\item The image of $\psi$ is the set $G.\A_\Lambda$.
\end{enumerate}
\end{proposition}

\begin{proof}
(1) Let $g_1,g_2\in  G$ and $[L_1],[L_2]\in \A^{\LC}$ be such that $g_1.[L_1]=g_2.[L_2]$. Let us prove that $g_1.\psi([L_1])=g_2.\psi([L_2])$. By assumption, there exists $g\in G$ such that $g.[L_1]=[L_2]\in \A^{\LC}$. By~\ref{axiomA2} of \cite[Example 3.2]{bennett1994affine}, there exists $n\in \NS$ such that $g.[L1]=[L_2]=n.[L_1]$. Write $[L_1]=t.[\mathbb{O}^{\ell+1}]$, with $t\in \TG$. Then  $g_1t.[\mathbb{O}^{\ell+1}]=g_2nt.[\mathbb{O}^{\ell+1}]$ and thus $n^{-1}g_2^{-1}g_1t .[\mathbb{O}^{\ell+1}]=t.[\mathbb{O}^{\ell+1}]$. By Lemma~\ref{lemComputation_Pt0} and Lemma~\ref{lemN_equivariance_psi}, we deduce that \[n^{-1}g_2^{-1}g_1t .\psi([\mathbb{O}^{\ell+1}])=t.\psi([\mathbb{O}^{\ell+1}])=n^{-1}g_2^{-1}g_1.\psi([L_1])=\psi([L_1]).\]  Thus $g_1.\psi([L_1])=g_2.\psi([L_2])$ which proves that $\psi$ is well defined. It is $G$-equivariant by definition.

Let us prove that $\psi$ is an isometry. Let $k\in \llbracket 1,\ell\rrbracket$ and $[L]=[\bigoplus_{i=1}^{\ell+1} \mathbb{O} x_i e_i]\in \A^{\LC}$. Then $\alpha_{k,k+1}\big(\psi([L])\big)=\omega(x_{k+1})-\omega(x_k)=\alpha_{k,k+1}^{\LC}([L])$. Therefore for all $i,j\in \llbracket 1,\ell+1\rrbracket $ such that $i\neq j$, $\alpha_{i,j}\big(\psi([L])\big)=\alpha_{i,j}^{\LC}([L])$. Therefore the restriction of $\psi$ to $\A^\LC$ is an isometry. Let $[L_1],[L_2]\in \I^{\LC}$. By \cite[Example 3.2]{bennett1994affine}, $\I^{\LC}$ satisfies~\ref{axiomA3} and thus there exists $g\in G$ such that $g.[L_1],g.[L_2]\in \A^{\LC}$. Then \[\begin{aligned} d([L_1],[L_2])=d(g.[L_1],g.[L_2])=d\big(\psi(g.[L_1]),\psi(g.[L_2])\big)&= d\big(g.\psi([L_1],g.\psi([L_2])\big)\\ &=d\big(\psi([L_1]),\psi([L_2])\big),\end{aligned}\] which proves that $\psi$ is an isometry.

(2) Let $[L]\in \I^{\LC}$ be such that $\psi([L])\in \A_S$. Write $[L ]=g.[L']$, with $g\in G$ and $[L']\in \A^{\LC}$. Then $\psi([L])=\psi(g.[L'])=g.\psi([L'])$. By Proposition~\ref{propAxiom(A2)}, we deduce that there exists $n\in \NS$ such that $g.\psi([L'])=n.\psi([L'])$. Therefore $\psi([L])=\psi(g.[L'])=\psi(n.[L'])$ and as $\psi$ is injective, $n.[L']=[L]$. As $\NS$ stabilizes $\A^{\LC}$, we deduce that $[L]\in \A^{\LC}$. Therefore $\psi^{-1}(\A)=\A^{\LC}$, which proves (2).

(3) It remains to determine the image of $\psi$. By (2) and by $G$-equivariance, it suffices to prove that $\psi(\A^{\LC})=\A_\Lambda$. Let $(x_i)\in \Lambda^{\ell+1}$. Then $\psi([\bigoplus_{i=1}^{\ell+1} \mathbb{O}  x_i e_i])=\sum_{i=1}^{\ell} \big(\omega(x_i)-\omega(x_{i+1})\big)\varpi_{i,i+1}^\vee$ and hence $\alpha_{j,j+1}\big(\psi([\bigoplus_{i=1}^{\ell+1} \mathbb{O}  x_i e_i])\big)\in \Lambda$ for all $j\in \llbracket 1,\ell\rrbracket$. Thus
$\alpha\big(\psi([\bigoplus_{i=1}^{\ell+1} \mathbb{O}  x_i e_i])\big)\in \Lambda$, for all $\alpha\in \Phi$.   Reciprocally, let $x\in \A_\Lambda$,  $(\lambda_i)=\big(\alpha_{i,i+1}(x)\big)_{i\in\llbracket 1,\ell\rrbracket}\in \Lambda^{\ell}$ and $(x_i)\in (\Kb^*)^{\ell}$ be such that $\omega(x_\ell)=\lambda_\ell$, $\omega(x_{\ell-1})=\lambda_{\ell-1}+\lambda_\ell$, $\ldots$, $\omega(x_1)=\lambda_1+\ldots+\lambda_\ell$. Set $x_{\ell+1}=1$. Then 
$\psi([\bigoplus_{i=1}^{\ell+1}\mathbb{O}  x_i e_i])=\sum_{i=1}^{\ell} \lambda_i\varpi_{i,i+1}^\vee$, which proves that $\psi(\A^{\LC})=\A_{\Lambda}$ and completes the proof of the proposition. 
\end{proof}

Let $\Lambda_0$ be a strict convex subgroup of $\Lambda$. Set $\Lambda_1=\Lambda/\Lambda_0$. Denote by $\omega_1:\mathbb{K}\rightarrow \Lambda_1\cup \{\infty\}$ the composition of $\omega$ with the projection $\Lambda\rightarrow \Lambda_1$. Let $\pi:\I(\Kb,\omega,\GB)\rightarrow \I(\mathbb{K},\omega_1,\GB)$ be the projection defined in subsection~\ref{SubsecContructionFibers}. Let $\mathbb{O} _1=\{x\in \mathbb{K}|\omega_1(x)\geq 0\}$. One has $\mathbb{O} _1\supset \mathbb{O} $. Let $\I^{\LC}_1$ be the set of $\mathbb{O} _1$-lattices of $\Kb^{\ell+1}$. Let $\pi^\LC:\I^\LC\rightarrow\I_1^{\LC}$ be defined by $\pi^\LC([L])=[\mathbb{O} _1. L]$, for $[L]\in \I^{\LC}$, where $\mathbb{O} _1.L$ is the $\mathbb{O} _1$-submodule of $\Kb^{\ell+1}$ generated by $L$. Let $\psi_1:\I_1^\LC\rightarrow \I^\LC$ be defined by $\psi_1(g.[L])= \sum_{i=1}^{\ell} \big(\omega_1(x_{i+1})-\omega_1(x_{i})\big)\varpi_{i,i+1}^\vee\in \pi(\A^\LC)$, for $[L]=[\bigoplus_{i=1}^{\ell+1} \mathbb{O} _1 x_i e_i]\in \A_1^{\LC}$ and $g\in G$. By the following corollary, the projection $\pi$ defined in subsection~\ref{SubsecContructionFibers} corresponds to  a tensorization by $\mathbb{O}_1$ in the particular case of $\mathrm{SL}_{\ell+1}$. This is the viewpoint of Parshin in \cite{parshin1994higher}.

\begin{corollary}
One has $\pi\circ \psi=\psi_1\circ\pi^{\LC}$. 
\end{corollary}

\begin{proof}
If $[L]\in \A^{\LC}$, then by the formulas defining $\psi$ and $\psi_1$, one has $\pi\circ\psi([L])=\psi_1\circ\pi^{\LC}([L])$. The proposition follows by $G$-equivariance of $\pi,\pi^\LC, \psi,\psi_1$.
\end{proof}

\subsection{The Parshin building  of \texorpdfstring{$\mathrm{SL}_2$}{SL(2)} of a \texorpdfstring{$2$}{2}-local field revisited}

We now assume for simplicity that $d=2$ and that $S=\{1,2\}$.   In \cite{parshin1994higher}, Parshin constructs a kind of boundary to $\I$ (he works with $\mathrm{PGL}_{\ell+1}$ instead of $\mathrm{SL}_2$). We now give a new construction of this boundary, using the axioms of $\Lambda$-buildings. This construction relies on the following intuition. Let $x=(x_1,x_2)\in \A_S$. Then $(x_1,x_2+n)$ and $(x_1,x_2-n)$ converge to ``boundary points'' $(x_1,+\infty)$ and $(x_1,-\infty)$ when $n$ diverges to $+\infty$. We thus construct a set $\partial^{(-,\infty)}  \A_S$ corresponding to $\{(x,\epsilon\infty)\mid x\in \A_\R,\epsilon\in \{-,+\}\}$. We then define a set $\partial^{(-,\infty)}  \I$ using the action of $G$.  If $(x_k)=(x_k^{(1)},x_k^{(2)})\in \A_S$ is such that $x_k^{(1)}$ diverges to $\pm \infty$, we can consider that $x_k\rightarrow \pm \infty$. This defines an other boundary $\partial^\infty \A_S$ of $\A_S$ and thus (using the -action of $G$) an other boundary $\partial^{\infty}\I$ of $\I$. Actually $\partial^\infty \I$ is the usual building at infinity of $\I$ constructed in \cite[Theorem 3.7]{bennett1994affine}. The boundary of $\I$ is then $\partial \I=\partial^{\infty} \I\sqcup \partial^{(-,\infty)}\I$. This subsection is devoted to the construction of $\partial^{(-,\infty)}\I$ and to its study.

By our assumption on $d$ and $S$, one has $\A_\R=\R$ and  $\A_S=\RF^2$.  The vectorial Weyl group is then $\{\pm \Id_{\A_\R}\}$. 

\begin{Not}\label{n_parshin building}
If $x=(x_1,x_2)\in \A_S$ and $\epsilon\in \{-,+\}$, we define $(x,\epsilon\infty)$ as the filter on $\A_S$ composed with the sets containing $(x_1,x_2)+\{0\}\times V$, where $V\subset \R$ is a neighborhood of $\epsilon\infty$. Then  $(x,\epsilon\infty)$ does not depend on $x_2$ and we will sometimes write $(x_1,\epsilon\infty)$. We define the boundary (in the first coordinate) of $\A_S$ as the set 
\[\partial^{(-,\infty)} \A_S=\{(x,\epsilon\infty)\mid x\in \A_\R,\epsilon\in \{-,+\}\}.\] If $w=(u,\lambda)\in W^v\ltimes \RF^2$, we set $\vec{w}=u\in W^v$.   Then $W^v\ltimes \RF^2$ acts on $\partial^{(-,\infty)}\A_S$ by $w.(x,\epsilon\infty)=(w.x,(-1)^{\ell(\vec{w})}\epsilon \infty)$ for $w\in W^v\ltimes \RF^2$ and $(x,\epsilon\infty)\in \partial^{(-,\infty)} \A_S$. Set $N=N_{\mathrm{\mathrm{SL}_2}}$. By Corollary~\ref{corAffine_weyl_group}, $\Wext=\nu(N)\subset W^v\ltimes \RF^2$. This enables us to define the boundary (in the first coordinate) of $\I$ as the set \[\partial^{(-,\infty)} \I=\{g.(x,\epsilon \infty)\mid g\in G, x\in \A_\R,\epsilon\in \{-,+\}\}.\]
\end{Not}

Let $\pi:\RF^2\twoheadrightarrow \R$ be the projection on the first coordinate. By subsection~\ref{SubsecContructionFibers}, $\pi$ extends to a $G$-equivariant surjective map $\pi:\I(\Kb,\omega,\mathrm{SL}_2)\twoheadrightarrow \I(\Kb,\pi\circ \omega,\mathrm{SL}_2):=\pi(\I)$. 

We write  $\Phi=\{\alpha,-\alpha\}=\{\Id,-\Id\}$.

\begin{definition/proposition}
Let $(x_k)=(x_k^{(1)},x_k^{(2)})\in (\A_S)^{\N}$. We say that $(x_k)$ converges in $\partial^{(-,\infty)} \A_S$ if $\big(\pi(x_k)\big)=(x_k^{(1)})$ is stationary and $x_k^{(2)}\rightarrow \epsilon \infty$, for some $\epsilon\in \{-,+\}$. We then set $\lim x_k= (\lim x_k^{(1)},\lim x_k^{(2)})\in \partial^{(-,\infty)} \A_S$. 

Let now $(x_k)\in \I^\N$. We say that $(x_k)$ converges in $\partial^{(-,\infty)} \I$ if there exists $g\in G$ such that $g.x_k\in \A_S$ for $k\gg 0$ and  $(g.(x_k))$ converges in $\partial^{(-,\infty)}\A_S$. We then set $\lim x_k=g^{-1}.\lim g.x_k$. This is well defined, independently of the choice of $g$. More precisely, if $h\in G$ is such that $h.x_k\in \A_S$ for $k\gg 0$, then $(h.x_k)$ converges and $g^{-1}.\lim g.x_k=h^{-1}.\lim h.x_k$.
\end{definition/proposition}

\begin{proof}
Let $g'=hg^{-1}$, $A=g'^{-1}.\A_S$ and $\Omega=A\cap \A_S$. By Proposition~\ref{propAxiom(A2)} and Corollary~\ref{corAffine_weyl_group}, there exists $w\in \Wext$ such that for all $k\in \N$ such that $g.x_k,h.x_k\in \A_S$, one has $g.x_k=w.h.x_k$. Let $n\in N$ be such that $n$ induces $-\Id$ on $\A_S$. Let $\pi_2:\A_S\rightarrow \R$ be defined by $\pi_2(y_1,y_2)=y_2$, for $(y_1,y_2)\in \A_S$. Maybe replacing $g$ or $h$ by $ng$ or $nh$, we may assume $\pi_2(g.x_k),\pi_2(h.x_k)\rightarrow +\infty$. Then we have $\vec{w}=\Id$ and hence  there exists $a\in \A_S$ such that $w=\Id+a$. 

 By Proposition~\ref{propAxiom(A2)}, $\Omega$ is enclosed. Thus there exists $\lambda_\alpha,\lambda_{-\alpha}\in \RF^2$ such that $\Omega=D_{\alpha,\lambda_\alpha}\cap D_{-\alpha,\lambda_{-\alpha}}$. As $\Omega\supset \{g.x_k\ \mid k\gg 0\}$, we deduce that $\lambda_{-\alpha}>\lim \pi(g.x_k)$. Thus there exists a neighborhood $V$ of $+\infty$ such that $\Omega\supset E:=\{(\lim \pi(g.x_k),v)\mid v\in V\}$. Then $E\in \lim g.x_k$. Then: \[\begin{aligned} g'.E &= \{hg^{-1}.(\lim \pi(g.x_k),v)\mid v\in V\}\\ &= \{\big(\lim \pi( h.x_k),v+\pi_2(a)\big)\mid v\in V\}, \end{aligned}\] by $G$-equivariance of $\pi$. Thus $g'.E\in \lim h.x_k$. Moreover, $g'.E$ is of the form $y+\{0\}\times V'$, where $y\in \A_S$ and $V'$ is a neighborhood of $+\infty$ and thus $E\notin (z,+\infty)$ for all $z\in\A_\R$ such that $\pi(z)\neq \lim \pi(x_k)$. As $g'.E\in g'.\lim g.x_k$, we deduce that $g'.\lim g.x_k=\lim h.x_k= h g^{-1}.\lim g.x_k$, which proves the proposition. 
\end{proof}

\begin{Lem}\label{lemFixator_end}
Let $x\in \A_\R$, $\epsilon\in \{-,+\}$ and $g\in G$ be such that $g.(x,\epsilon\infty)=(x,\epsilon\infty)$. Let $\Omega=g^{-1}.\A_S\cap \A_S$. By Proposition~\ref{propAxiom(A2)} and Corollary~\ref{corAffine_weyl_group}, there exists $w\in  \Wext$ such that for all $y\in \Omega$, $g.y=w.y$. Then $w$ is a translation of $\A_S$ whose  vector is in $\{0\}\times \A_\R$.
\end{Lem}

\begin{proof}
Let $V$ be a neighborhood of $+\infty$ such that $(x,\epsilon\infty)\ni E=x+\{0\}\times V$. Then $g^{-1}.E\in g^{-1}.(x,\epsilon\infty)=(x,\epsilon\infty)$. Thus there exists a neighborhood $V'$ of $+\infty$ such that $g^{-1}.E\supset x+\{0\}\times V'$. Then $g.(x+\{0\}\times V')\subset x+\{0\}\times V$ and hence $\Omega\supset x+\{0\}\times V'$. One has $w.\big(x+(0,v)\big)=w.x+\vec{w}.(0,v)\in x+\{0\}\times V$ for all $v\in V'$ and thus $\vec{w}\neq -\Id$. Therefore $\vec{w}= \Id$ and $w$ is a translation. Let $a\in \A_S$ be such that $w.y=y+a$ for all $y\in \A_S$. Then $w.(x+\{0\}\times V')=x+a+\{0\}\times V'\subset x+\{0\}\times V$ and hence $a\in\{0\}\times \A_\R$. 
\end{proof}

We define $\pi:\partial^{(-,\infty)}\I\rightarrow \pi(\I)$ by $\pi\big(g.(x,\epsilon\infty)\big)=\pi(g.x)$, for $(x,\epsilon\infty)\in \partial^{(-,\infty)}\A_S$ and $g\in G$. We denote by $\R^*_+$ and by $\R^*_-$ the vector chambers of $\pi(\A_S)=\R$. We use the notation of subsection~\ref{SubsecGerms}.

We denote by $\mathrm{Alc}\big(\pi(\A_S)\big)$ (resp. $\mathrm{Alc}\big(\pi(\I)\big)$)  the set of alcoves of $\pi(\A_S)$ that is the set of faces of the form $F(x,\epsilon\infty)$, for $x\in \pi(\A_S)$ (resp. $g.F(x,\R^*_\epsilon)$) and $\epsilon\in \{-,+\}$ (resp. and $g\in G$). By Lemma~\ref{LemSetOfValues}, $\Gamma_\alpha=\Gamma_{-\alpha}=\Lambda$: if $\lambda\in \RF^2$, $D_{\alpha,\lambda}$ is a half-apartment if and only if $\lambda\in \Lambda$. Moreover $\Lambda$ is a group and in particular, $\pi(\Lambda)$ is either dense in $\R$ or discrete. In the case where it is discrete, we assume that $\pi(\Lambda)=\Z$.

\begin{proposition}\label{propBouts_aretes}(see Figure~\ref{figBouts_aretes})
\begin{enumerate}
\item Let $\Upsilon :\partial^{(-,\infty)} \A_S\rightarrow  \mathrm{Alc}\big(\pi(\A_S)\big)$ be the map defined by $\Upsilon\big((x,\epsilon\infty)\big)=F\big(\pi(x),\R^*_\epsilon\big)$, for $(x,\epsilon\infty)\in \partial^{(-,\infty)}(\A_S)$. Then $\Upsilon$ is  surjective. It extends uniquely to a surjective $G$-equivariant map $\Upsilon :\partial^{(-,\infty)} \I\rightarrow  \mathrm{Alc}\big(\pi(\I)\big)$.

\item Suppose  $\pi(\Lambda)$ is dense in $\R$. Let $g\in G$ and  $\pi(x)\in \pi(\A_S)$. Then 
\[\Upsilon^{-1}\big(\{g.F(\pi(x),\R^*_+)\}\big)= \left\{\begin{aligned}& \{g.(x,+\infty),g.(x,-\infty)\}&\text{ if }\pi(x)\notin \pi(\Lambda) \\ & g.(x,+\infty) &\text{ if }\pi(x)\in\pi(\Lambda)\end{aligned}\right..\]

\item Suppose $\pi(\Lambda)=\Z$. Let $x\in \R$. Then: \[\begin{aligned} \Upsilon^{-1}\big(\{g.F(\pi(x),\R^*_+)\}\big)=& \{g.(y,+\infty)\ |\ y\in [\lfloor\pi(x)\rfloor, \lfloor\pi(x)\rfloor+1[\}\\ 
&\cup \{g.(y,-\infty)\ |\ y\in ]\lfloor\pi(x)\rfloor, \lfloor\pi(x)\rfloor+1]\}.\end{aligned}\]
\end{enumerate}

\end{proposition}

\begin{proof}
(1) By definition of  an alcove, $\Upsilon$ is surjective. Let $x\in\A_\R$, $\epsilon\in \{-,+\}$ and $g,h\in G$ be such $g.(x,\epsilon \infty)=h.(x,\epsilon\infty)$. By symmetry we can  assume that $\epsilon=+$.
 Let $g'=h^{-1}g$ and $A=g'^{-1}.\A_S$. Then by Proposition~\ref{propAxiom(A2)}, $\Omega:=A\cap \A_S$ is enclosed and contains $(x,+\infty)$. Write $\Omega=D_{\alpha,\lambda_\alpha}\cap D_{-\alpha,\lambda_{-\alpha}}$.
  Then $D_{-\alpha,\lambda_{-\alpha}}\Supset (x,+\infty)$ and thus $\alpha(\pi(x))<\pi(\lambda_{-\alpha})$.  Therefore $\pi(\Omega)\Supset F\big(\pi(x),\R^*_+\big)$. By Lemma~\ref{lemFixator_end}, $g'.F(\pi(x),\R^*_+)=F(\pi(x),\R^*_+)$ and thus $g.F(\pi(x),\R^*_+)=h.F(\pi(x),\R^*_+)$, which proves that $\Upsilon$ is well defined. It is surjective by definition of the alcoves.

(2), (3) As for all $x\in\A_\R$ and $g\in G$, one has $\Upsilon\big(g.(x,+\infty)\big)=g.F(x,+\infty)$, it suffices to determine for which $\pi(x),\pi(x)'\in \R$ and $\epsilon,\epsilon'\in \{-,+\}$, one has $F(\pi(x),\R^*_\epsilon)=F(\pi(x'),\R^*_{\epsilon'})$. Let $x,x'\in \pi(\A_S)$ and $\epsilon,\epsilon'\in \{-,+\}$ be such that $F(\pi(x),\R^*_\epsilon)=F(\pi(x'),\R^*_{\epsilon'})$. 

Suppose that $\pi(\Lambda)$ is dense in $\R$. Then if $\pi(x)\neq \pi(x')$, there exist $\lambda_\alpha,\lambda_{-\alpha},\lambda_{\alpha}',\lambda_{-\alpha}'\in \Lambda$ such that $D_{\alpha,\lambda_\alpha}\cap D_{-\alpha,\lambda_{-\alpha}}$ (resp. $D_{\alpha,\lambda_\alpha'}\cap D_{-\alpha,\lambda_{-\alpha}'}$) contains an open neighborhood of $x$ (resp. $x'$) and such that $D_{\alpha,\lambda_\alpha}\cap D_{-\alpha,\lambda_{-\alpha}}\cap D_{\alpha,\lambda_\alpha'}\cap D_{-\alpha,\lambda_{-\alpha}'}=\emptyset$. Consequently $F(\pi(x),\R^*_\epsilon)\neq F(\pi(x'),\R^*_{\epsilon'})$: a contradiction.  Therefore $\pi(x)=\pi(x')$. If $\pi(x)\notin\pi(\Lambda)$, then $F(\pi(x),\R^*_\epsilon)$  and $F(\pi(x),\R^*_{\epsilon'})$ are the filters of neighborhoods of $\pi(x)$ and thus  $F(\pi(x),\R^*_\epsilon)=F(\pi(x),\R^*_{\epsilon'})$. If $\pi(x)\in \pi(\Lambda)$, then $D_{\alpha,-\pi(x)}\in F(\pi(x),\R^*_+)\setminus F(\pi(x),\R^*_-)$ and thus $F(\pi(x),\R^*_+)\neq  F(\pi(x),\R^*_-)$, which proves (2). 

Suppose that $\pi(\Lambda)=\Z$. Suppose $\pi(x)\notin \Z$. Let $E=\mathring{D}_{\alpha,\lfloor \pi(x) +1\rfloor}\cap \mathring{D}_{-\alpha,\lfloor\pi(x) \rfloor}$.  Then $F(\pi(x),\R^*_\epsilon)=\FCC_{E,\pi(\A_S)}$,
with the notation of \ref{SubsecGerms}, for both $\epsilon\in \{-,+\}$. Suppose $\pi(x)\in \pi(\Lambda)$. Let $E_-=\mathring{D}_{\alpha,\pi(x)+1}\cap \mathring{D}_{-\alpha,-\pi(x)-1}$ and $E_+=\mathring{D}_{-\alpha,\pi(x)}\cap \mathring{D}_{\alpha,-\pi(x)+1}$. Then $F(\pi(x),\R^*_+)=\FCC_{E_+,\pi(\A_S)}$ and $F(\pi(x),\R^*_-)=\FCC_{E_-,\pi(\A_S)}$, and (3) follows.
\end{proof}

Suppose that $\pi(\Lambda)= \Z$.  Then $\mathrm{Alc}(\I)=\{g.]n,n+1[\ |\ n\in \Z\}$. Then if $n\in \Z$ and $g\in G$, $\Upsilon\big(g. (n,+\infty)\big)=g.]n,n+1[$ and $\Upsilon(g.(n,-\infty)\big)=g.]n-1,n[$.

\begin{corollary}\label{corBijection_ends_edges}
Suppose that $\pi(\Lambda)=\Z$. Then $\Upsilon$ induces a two to one map between $\partial^{(-,\infty)} \I_{\Z}:=\{g.(\pi(x),\epsilon\infty)\mid x\in \pi^{-1}(\Z)\subset  \A_S, \epsilon\in \{-,+\}\text{ and } g\in G\}$ and $\mathrm{Alc}\big(\pi(\I)\big)$. More precisely, $g.]n,n+1[\in \mathrm{Alc}\big(\pi(\I)\big)$, with $g\in G$ and $n\in \Z$, then $\Upsilon^{-1}(g.]n,n+1[)=\{g.(n,+\infty),g.(n+1,-\infty)\}$: $g.(n,+\infty)$ is an end of the tree $\pi^{-1}(\{g.n\})$ and $g.(n+1,-\infty)$ is an end of  the tree $\pi^{-1}(\{g.(n+1)\})$.
\end{corollary}

\begin{figure}
 \centering
\includegraphics[scale=0.3]{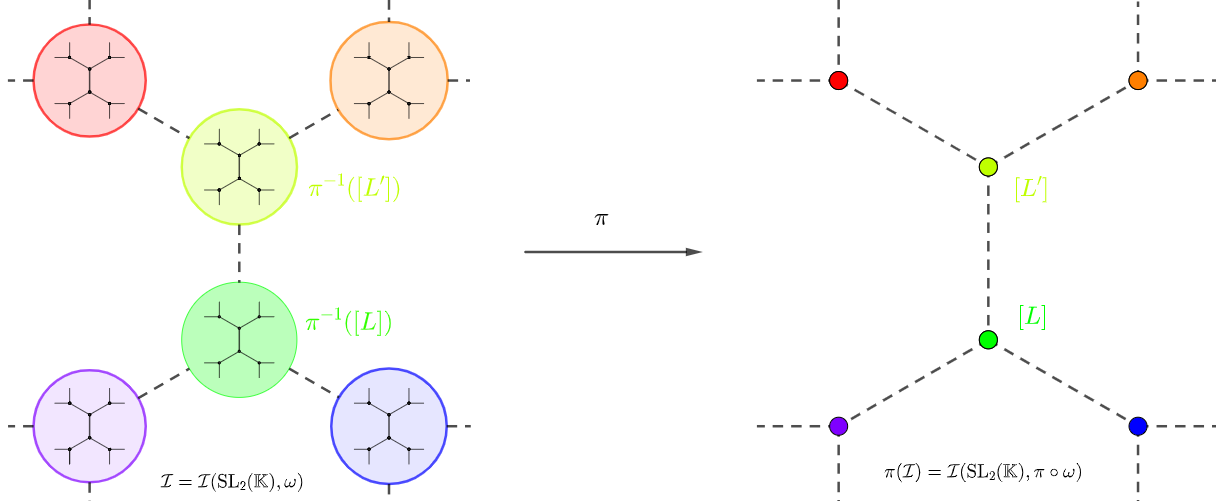}
\caption{Suppose for simplicity that $\pi(\Lambda)$ is discrete. Let $\A_\Lambda=\{x\in \A|\alpha(x)\in \Lambda\}$ and $\I_\Lambda=G.\A_\Lambda$. Let $[L],[L']\in \pi(\I_\Lambda)$. Then by Theorem~\ref{thmFibers_buildings}, $\I_{[L]}=\pi^{-1}([L])$ and $\I_{[L']}=\pi^{-1}([L'])$  are $\R$-trees.  By Proposition~\ref{propBouts_aretes} and  Corollary~\ref{corBijection_ends_edges},  the set of ends of $\I_{[L]}$ is in bijection with the set of edges of $\pi(\I)$ containing $\pi(x)$ and the edge $\pi([L])-\pi([L'])$ corresponds to a unique pair of ends $\mathcal{E}_{[L]}$, $\mathcal{E}_{[L']}$ of $\I_{[L]}$ and $\I_{[L']}$.}\label{figBouts_aretes}
\end{figure}

\paragraph{Interpretation in terms of lattices}(see Figure~\ref{figBouts_aretes_reseaux}) We assume that $\Lambda=\Z\times \Z$  and that $\mathbb{K}$ is a  field equipped with a surjective valuation $\omega_2:\mathbb{K}\rightarrow \Lambda\cup \{\infty\}$. Let $\omega_1=\pi\circ \omega_2:\mathbb{K}\rightarrow \Z$. Let $\mathbb{O}_2=\{x\in \mathbb{K}\mid  \omega_2(x)\geq 0\}$ and $\mathbb{O}_1=\{x\in \mathbb{K}\mid \omega_1(x)\geq 0\}$. For all $\lambda\in \Lambda$ choose $x_\lambda\in \mathbb{K}^*$ such that $\omega_2(x_\lambda)=\lambda$. Let $(e_1,e_2)$ be a  basis of $\mathbb{K}^2$ and $\A^{\LC}=\{[\mathbb{O}_2 x_1 e_1\oplus \mathbb{O}_2 x_2 e_2]\mid (x_1,x_2)\in \mathbb{K}^2\}=\{[\mathbb{O}_2 x_{\lambda_1} e_1  \oplus \mathbb{O}_2 x_{\lambda_2}e_2]\mid (\lambda_1,\lambda_2)\in \Lambda^2\}$. One has $\bigcap_{n\in \N} x_{(0,n)}\mathbb{O}_2 =x_{(1,0)}\mathbb{O}_1$ and $\bigcup_{n\in \N}x_{(0,-n)}\mathbb{O}_2 =\mathbb{O}_1$. Thus if $(b_1,b_2)$ is a basis of $\mathbb{K}^2$ and $\lambda_1\in \Z$, one has \[\bigcup_{n\in \N} \mathbb{O}_2 b_1\oplus \mathbb{O}_2 x_{(\lambda_1+1,-n)}=\mathbb{O}_2 b_1\oplus \mathbb{O}_1 x_{(\lambda_1+1,0)} b_2=\bigcap_{n\in \N}\mathbb{O}_2 b_1\oplus \mathbb{O}_2x_{(\lambda_1,n)} b_2.\] Thus we can go from $\pi^{-1}([\mathbb{O}_1b_1\oplus \mathbb{O}_1 x_{(\lambda_1,0)}b_2])$ to  $\pi^{-1}([\mathbb{O}_1b_1\oplus \mathbb{O}_1 x_{(\lambda_1+1,0)}b_2])$ via the ends $([\mathbb{O}_2 b_1\oplus \mathbb{O}_2x_{(\lambda_1,0)} b_2],+\infty)$ and $([\mathbb{O}_2 b_1\oplus \mathbb{O}_2x_{(\lambda_1+1,0)} b_2],-\infty)$: the edge $([\mathbb{O}_1b_1\oplus \mathbb{O}_1 x_{(\lambda_1,0)}b_2]),([\mathbb{O}_2 b_1\oplus \mathbb{O}_2x_{(\lambda_1+1,0)} b_2],+\infty)$ links $([\mathbb{O}_2 b_1\oplus \mathbb{O}_2x_{(\lambda_1,0)} b_2],+\infty)$ to $([\mathbb{O}_2 b_1\oplus \mathbb{O}_2x_{(\lambda_1+1,0)} b_2],-\infty)$.

\begin{figure}
 \centering
\includegraphics[scale=0.5]{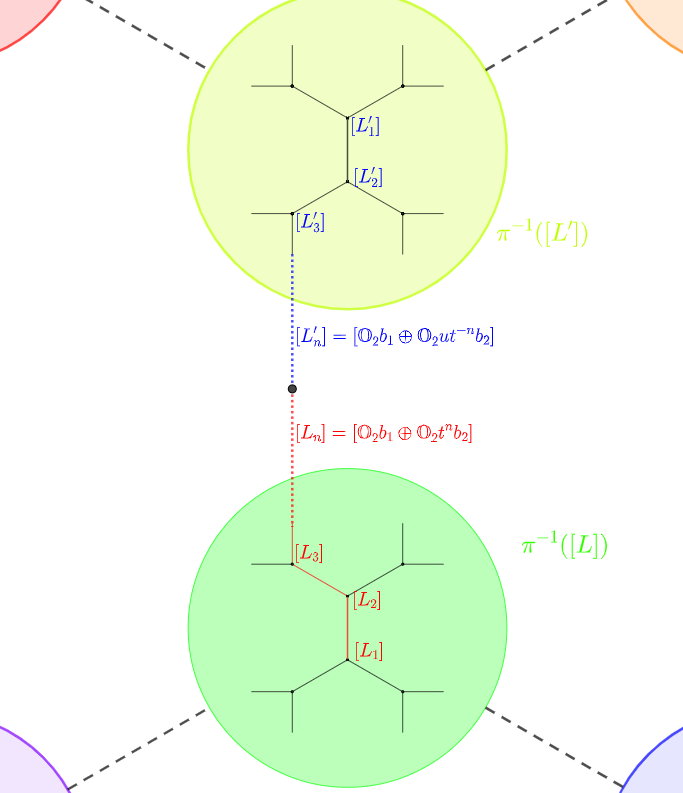}
\caption{We work with $\Kb=\F_q(\!(t)\!)(\!(u)\!)$. Let $(b_1,b_2)$ be a basis of $\Kb^2$. We can choose $x_{(1,0)}=u$ and $x_{(0,1)}=t$.} Then  $[L_n]=[\mathbb{O}_2 b_1\oplus \mathbb{O}_2 t^n b_2]$ and  $[L_n']=[\mathbb{O}_2 b_1 \oplus \mathbb{O}_2 ut^{-n} b_2]$ ``converge'' to $[\mathbb{O}_2 b_1\oplus \mathbb{O}_1 u b_2]$. \label{figBouts_aretes_reseaux}
\end{figure}

\addcontentsline{toc}{section}{Index}

\printindex

\addcontentsline{toc}{section}{Index of Axioms}

\printindex[axiom]

\addcontentsline{toc}{section}{Index of Symbols}

\printindex[notation]

\addcontentsline{toc}{section}{References}

\bibliography{biblio} 

\end{document}